\documentclass[reqno,10pt]{amsart} 

\usepackage[a4paper]{geometry}\usepackage{amssymb,amsmath,amsthm, stackengine}
\usepackage{amscd}
\usepackage{amsfonts}
\usepackage{amssymb,mathrsfs}
\usepackage{latexsym}
\usepackage{esint}
\usepackage{kotex}

\usepackage{hyperref}
\hypersetup{
    colorlinks=true,
    linkcolor=blue,
    filecolor=magenta,      
    urlcolor=cyan,
    pdftitle={Overleaf Example},
    pdfpagemode=FullScreen,
    }
    
\usepackage{comment}

\usepackage[makeroom]{cancel}

\usepackage[english]{babel}

\addto\captionsenglish{
	\renewcommand{\contentsname}%
	{Contents}%
}

\let\hide\iffalse

\newtheorem{theorem}{Theorem}

\newtheorem{corollary}[theorem]{Corollary}
\newtheorem{definition}[theorem]{Definition}
\newtheorem{lemma}[theorem]{Lemma}
\newtheorem{prop}[theorem]{Proposition}
\newtheorem{remark}[theorem]{Remark}

\let\e=\varepsilon

\let\p=\partial

\let\O=\Omega

\renewcommand{\b}{\mathbf{b}}

\newcommand{\be}{\begin{equation}}
	\newcommand{\bm}{\begin{multline}}
		\newcommand{\ee}{\end{equation}}
	\newcommand{\dd}{\mathrm{d}}

	\newcommand{\xb}{x_{\mathbf{b}}}

	\newcommand{\tb}{t_{\mathbf{b}}}
	\newcommand{\vb}{v_{\mathbf{b}}}
	\newcommand{\pb}{p_{\mathbf{b}}}
		\newcommand{\tpb}{\tilde{p}_{\mathbf{b}}}
		\newcommand{\tpB}{\tilde{p}_{\mathbf{B}}}
	
	\newcommand{\vbn}{v_{\mathbf{b}, 3}}
	\newcommand{\vbpn}{v_{\mathbf{b}, \pm, 3}}
	\newcommand{\pbn}{p_{\mathbf{b}, 3}}
	\newcommand{\pbpn}{p_{\mathbf{b}, \pm, 3}}

		\newcommand{\tpBn}{\tilde{p}_{\mathbf{B}, 3}}
	\newcommand{\tB}{t_{\mathbf{B}}}
	\newcommand{\tBf}{t_{\mathbf{B}, 1}}
	\newcommand{\xBf}{x_{\mathbf{B}, 1}}
	\newcommand{\pBf}{p_{\mathbf{B}, 1}}
	\newcommand{\tpBf}{\tilde{p}_{\mathbf{B}, 1}}
	\newcommand{\vB}{v_{\mathbf{B}}}
	\newcommand{\vBn}{v_{\mathbf{B}, 3}}
	\newcommand{\pB}{p_{\mathbf{B}}}
	\newcommand{\pBn}{p_{\mathbf{B}, 3}}
	
	\newcommand{\vBp}{v_{\mathbf{B}, \pm}}
	\newcommand{\xB}{x_{\mathbf{B}}}

	\newcommand{\tbp}{t_{\mathbf{b}, \pm}}
	\newcommand{\tfp}{t_{\mathbf{f}, \pm}}
	\newcommand{\tBp}{t_{\mathbf{B}, \pm}}
	\newcommand{\tFp}{t_{\mathbf{F}, \pm}}
	\newcommand{\xbp}{x_{\mathbf{b}, \pm}}
	\newcommand{\xbj}{x_{\mathbf{b}, j}}
	\newcommand{\pbj}{p_{\mathbf{b}, j}}

	\newcommand{\xBp}{x_{\mathbf{B}, \pm}}
	\newcommand{\pBp}{p_{\mathbf{B}, \pm}}
	\newcommand{\pBpn}{p_{\mathbf{B}, \pm, 3}}
	\newcommand{\vBpn}{v_{\mathbf{B}, \pm, 3}}
	\newcommand{\xfp}{x_{\mathbf{f}, \pm}}
	\newcommand{\vbp}{v_{\mathbf{b}, \pm}}
	\newcommand{\pbp}{p_{\mathbf{b}, \pm}}
		\newcommand{\tpbp}{\tilde{p}_{\mathbf{b}, \pm}}
		\newcommand{\tpBp}{\tilde{p}_{\mathbf{B}, \pm}}

	\newcommand{\pfp}{p_{\mathbf{f}, \pm}}
	
	\newcommand{\tF}{{t}_{\mathbf{F}}}

	\newcommand{\tf}{t_{\mathbf{f}}}

	\newcommand{\X}{\mathcal{X}}
	\newcommand{\V}{\mathcal{V}}
 	\newcommand{\Z}{\mathcal{Z}}
  
 	\renewcommand{\P}{\mathcal{P}}

		\newcommand{\w}{\mathfrak{w}}

	\newcommand{\Bes}{\begin{eqnarray*}}
		\newcommand{\Ees}{\end{eqnarray*}}
	\newcommand{\Be}{\begin{equation}}
		\newcommand{\Ee}{\end{equation}}

	\DeclareMathOperator{\sgn}{sgn}
	
	\numberwithin{equation}{section}

	\def\p{\partial}

	\def\O{\Omega}
	\def\R{\mathbb{R}}
	\def\N{\mathbb{N}}
	
	\def\B{\begin{equation}}
		\def\E{\end{equation}}
	\def\BN{\begin{eqnarray*}}
		\def\EN{\end{eqnarray*}}
	
	\usepackage{color}

	\renewcommand\contentsname{}

\newcommand{\px}{p_{\max}}

\newcommand{\tpx}{\tilde{p}_{\max}}

\begin{document}

\title[Stability of 3D Collisionless plasma states]{
Asymptotic Stability of 3D relativistic Collisionless plasma states in ambient magnetic fields with a boundary   
}

\author{Jiaxin Jin}
\address{Department of Mathematics, University of Louisiana at Lafayette, Lafayette, LA 70504}
\email{jiaxin.jin@louisiana.edu}

\author{Chanwoo Kim}
\address{Department of Mathematics, The University of Wisconsin-Madison, Madison, WI 53706}
\email{chanwookim.math@gmail.com; chanwoo.kim@wisc.edu}

\begin{abstract}
 
Motivated by the stellar wind ejected from the upper atmosphere (Corona) of a star, we explore a boundary problem of the two-species nonlinear relativistic Vlasov-Poisson systems in the 3D half space in the presence of a constant vertical magnetic field and strong background gravity. We allow species to have different mass and charge (as proton and electron, for example). As the main result, we construct stationary solutions and establish their nonlinear dynamical asymptotic stability in time and space. 
\end{abstract}

\maketitle

\tableofcontents

\section{Introduction}

The Vlasov equations serve as fundamental models in collisionless kinetic theory. One of their crucial applications lies in the solar wind model, which describes the dynamics and statics of plasma particles emitted from the solar corona under the effect of strong solar gravity (see \eqref{E/g2} and some explanation). Despite some significant progress in the field, several fundamental open problems persist within the solar wind model mathematically (\cite{Vidotto}) and scientifically (\cite{nature}). These challenges include collisionless shock phenomena, corona heating mechanisms, and the acceleration of solar wind particles (\cite{Morawetz, TK, Vidotto}). Recently, the Parker Solar Probe, launched in 2018, has provided accurate data on the solar atmosphere, drawing attention to these longstanding issues (see \cite{nature}, for example).

In this paper, we consider two species of relativistic Vlasov systems subjected to a vertical downward constant gravity $g > 0$, and a constant vertical ambient solar magnetic/electric fields $B = (0, 0, B_3)$ and $E = (0, 0, E_3)$ in a 3D half space 
$(x_1, x_2, x_3) \in \O : = \R^2 \times \R_+$: 
\begin{align}
	\p_t F_{\pm} + c \frac{p}{p^0_{\pm}} \cdot \nabla_x F_{\pm} +  \left(
	\frac{p}{p^0_{\pm}} \times e_{\pm} B
 + e_\pm ( E - 
  \nabla_x    \phi_F ) - m_{\pm} g  e_3
	\right) 
	\cdot \nabla_p F_{\pm} = 0,&
 \label{VP_F} \\
	\Delta_x \phi_F 
	=  - \int_{\R^3} ( e_+ F_+ + e_{-} F_{-} ) \dd p .
 \label{Poisson_F}
\end{align}  
Here, $p= (p_1, p_2, p_3) \in \R^3$  stands for the momentum, and $c$ is the speed of light. 
The notations with subscript symbols $``+''$ and $``-''$ correspond to protons and electrons, respectively. They are particle distribution functions in the phase space $F_\pm = F_\pm (t,x,p)$;  the magnitude of particle rest masses $m_\pm$; the particle electric charges $e_\pm$. We also define the particle energy and relativistic velocity for each species, respectively
\Be\label{totalE}
p_\pm^0 = \sqrt{ (m_\pm c)^2  + |p|^2 }, \ \ v_\pm =  c  {p} / {p^0_\pm}.
\Ee 
The model corresponds to the zero magnetic permittivity limit of the Vlasov-Maxwell system (see \cite{CK23}), which is also relevant when the ambient solar magnetic fields dominate the magnetic field generated by the plasma particles (see \cite{Vidotto}). 

The downward gravity of the Sun plays an important role in the solar wind. In \cite{Rosseland}, Rosseland investigated the ratio of gravitational force and electrical force in the atmosphere of a star consisting of electrons, each with a charge $e_-$ and mass $m_-$, and positive ions (e.g. protons), each with a charge $e_+$ and mass $m_+$. If the local balance of two net forces holds, electric field $E$ is upward ($E_3>0$) and the ratio of the ambient electric field $E$ and gravitational constant satisfies
\Be\label{E/g}
\frac{E_3}{g} = \frac{m_+ - m_-}{ e_+ - e_-}. 
\Ee
Note that a proton is much heavier than an electron $m_+ \gg m_-$ while they have the same magnitude of electric charge $e_+ = - e_-$. Therefore, the gravitation effect dominates the electrical effect:
\Be\label{E/g2}
m_+ g \sim 2 e_+ E_3.
\Ee
For the rest of the entire paper, we rightfully ignore the ambient electric field $E$. It is worth recalling that plasma in tokamak reactors is confined mainly by the magnetic field (while the solar wind is confined by the solar ambient gravity). 

Now we discuss the boundary and initial conditions. We consider the initial condition:
\Be \label{VP_0}
F_{\pm} (0,x,p) = F_{\pm, 0} (x,p) \ \ \text{in} \ \O \times \R^3. 
\Ee
Suppose that the lower half space is filled with the sea of plasma, which is already described (e.g. solutions of MHD). Since they would form a perfect conductor, we set the zero Dirichlet boundary condition for $\phi_F$ at the boundary (\cite{CK}):
\Be \label{Dbc:F}
\phi_F (t, x_{\|}, 0) = 0 \ \ \text{on} \ x \in \p \O.
\Ee
Here, we use the notation $x_{\|} := (x_1, x_2) \in \R^2$ and $x = (x_{\|}, x_3) \in \O$.
On the other hand, the particle distributions $F_\pm$ follow the given inflow boundary data $G_{\pm} (x, p) \geq 0$ on the incoming phase boundary $\gamma_- := \{(x,p): x \in \p \O \ \text{and } p_3 > 0\}$, which is described by the solar corona dynamics. Hence we set that, for given functions  $G_\pm \geq 0$ on $\gamma_-$, 
\Be \label{bdry:F}
F_{\pm}(x,p) = G_{\pm} (x,p) \ \ \text{on} \ (x, p) \in \gamma_-.
\Ee
(we let $\gamma_+ := \{(x,p): x \in \p \O \ \text{and } p_3 < 0\}$ denote the outgoing phase boundary.)
One of the typical inflow boundary data is the J\"uttner distribution of wall temperature $T_\pm$\footnote{We note that the notion of temperature in relativity has been a topic of debate since the birth of special relativity in 1907. It seems the J\"uttner distribution is one of the very rare cases, of which the temperature is defined  (\cite{Gavassino, Tolman}).}:
\Be  \label{def:juttner}
J_\pm (x_1, x_2, p) = 
\frac{ \exp \left( \frac{- c p^0_{\pm} }{ k_B T_{\pm} (x_1, x_2)} \right)}{4 \pi e_{\pm} m^2_{\pm} c k_B T_{\pm}(x_1, x_2 ) K_2 \left( \frac{m_{\pm} c^2}{   k_B T_{\pm}(x_1, x_2 )} \right) },
\Ee
where $K_2 (z) := \frac{z^2}{2} \int^{\infty}_{1} e^{-zt} (t^2-1)^{3/2} \dd t$ is the modified Bessel function of the second kind. We note that the wall temperate of protons $T_+$ and electrons $T_-$ can differ unless the whole system of corona reaches thermal equilibrium. Moreover their temperature asymptotic $ \lim_{x_1 \rightarrow \pm\infty}T_\pm (x_1,x_2) \rightarrow T_\pm^\pm $ could be different in general: 
$T^+_{+} \neq T^-_+$ and $T^+_{-} \neq T^-_-$.

\subsection{Informal statement of the Main theorem}

In this paper, we investigate the stabilizing effect of downward gravity and the boundary on the nonlinear relativistic Vlasov system in the presence of a vertical ambient magnetic field. 
We present an informal statement of the main result and explain the meaning/implication of results/assumptions/range of some important parameters.
Detailed technical statements of the main theorems can be found in Section \ref{sec:MR}.

\begin{theorem}[Rough version of the Main theorems]
\label{theo:inform}

Consider the problem \eqref{VP_F}-\eqref{bdry:F} with a constant ambient magnetic field $B= (0,0, B_3)$. Suppose that for some $0< \tilde \beta \leq \beta$ the inflow boundary data $G_\pm$ and their derivatives $\nabla_{x,p}G_{\pm}$ are bounded pointwisely by $e^{- \beta  p_\pm^0}$ and $e^{- \tilde \beta  p_\pm^0}$, respectively. We assume that the gravity is strong such that $g \beta \gg_{B_3} 1$ and $g \tilde \beta \gg_{B_3} 1$.
	
\smallskip 

(1) Steady problem. We construct a unique classical solution $(h_\pm, \Phi _h)$ to the steady problem. In particular $\nabla_p h_\pm$ is bounded pointwisely by $e^{- \frac{\tilde \beta}{2}  p_\pm^0}$.

\smallskip
	
(2) Dynamical Asymptotic Stability. We construct a unique classical solution around the steady solution: 
\Be \label{def:F}
F_{\pm} (t,x,p) = h_{\pm} (x,p) + f_{\pm} (t,x,p) , \ \ \ \phi_F (t,x) = \Phi_h(x)  + \Psi(t,x).
\Ee
Moreover, we show that the dynamical perturbation $(f_\pm, \nabla_x \Psi)$ decays exponentially as $e^{-\lambda t}$ where $\lambda \sim g \beta \min \{m_+, m_-\}$.
\end{theorem}

\begin{remark}
A physical implication of $e^{- \beta p^0_\pm}$ upper bound is the J\"uttner distribution of temperature $T_\pm\sim 1/\beta$. Note that the modified Bessel function of the second kind behaves like $K_2(\tau) \sim \frac{e^{- \tau}}{ \sqrt \tau}$ for $\tau\gg 1$ and $K_2 (\tau) \sim \frac{1}{|\tau|^2}$ for $\tau \ll 1$. Therefore the J\"uttner distribution of temperature $T_\pm \sim \frac{c}{k_B } \frac{1}{ \beta}$ behaves like
\Be
J_\pm \sim  \frac{1}{e_\pm m_\pm^3 c^3 } \frac{ c m_\pm  \beta }{ K_2 (c m_\pm \beta )} e^{- \beta p^0_\pm} \sim \begin{cases}
 \frac{1}{e_\pm m_\pm^3 c^3 }
 (c m_\pm \beta)^{3/2} e^{ c m_\pm \beta} e^{- \beta p^0_\pm}
  \ \ \text{for }  c m_\pm \beta \gg 1,
\\
 \frac{1}{e_\pm m_\pm^3 c^3 }  (c m_\pm \beta)^3 e^{- \beta p^0_\pm} 
\ \ \text{for  } c m_\pm \beta \ll 1.
 \end{cases}
\Ee
Therefore the condition $g \beta \gg 1$ roughly says that our stabilizing mechanism needs a stronger gravity for a higher boundary temperature. 
\end{remark}

\begin{remark}

In Section \ref{sec:asymptotic}, we study the steady-state problem \eqref{VP_h}-\eqref{eqtn:Dphi} under both isothermal and non-isothermal J\"uttner distributions.
Note that Theorem \ref{theo:CS} establishes the existence of two distinct solutions, $h_{I, \pm}$ and $h_{\pm}$, corresponding to these distributions.
Assuming that the non-isothermal wall temperature converges to a unit temperature as $|x_{\|}| \to \infty$, we investigate the asymptotic behavior of these solutions in Theorem \ref{thm:asymptotic_behavior}.
\end{remark}

\begin{remark}

Furthermore, in Section \ref{sec:specular} we extend the result to the scenario where the particle distributions $F_\pm$ adhere to a combination of inflow and specular boundary conditions defined as:
\Be \notag
F_{\pm} (t,x,p) = G_{\pm} (x,p) + \varepsilon F_{\pm} (t,x, \tilde{p})
\ \ \text{on} \ (x, p) \in \gamma_-,
\Ee
where $\varepsilon \ll 1$ and $\tilde{p} = (p_1, p_2, - p_3)$. Detailed technical statements of the theorems (Theorem \ref{theo:CS_spec} and Theorem \ref{theo:CD_spec}) and the corresponding proofs are provided in Section \ref{sec:specular}.
\end{remark}

\subsection{Novelties, Difficulties, and Ideas}

The construction of steady states that are \textit{dynamically asymptotically} stable in the presence of a nontrivial magnetic field is a challenging open problem in nonlinear collisionless plasma physics with important physical applications (\cite{BGK,  Guo_96, Guo_97, GL17, Morawetz, MS}). It is worth mentioning that stability fails not uncommonly \cite{GS_95}, global study of the Vlasov-Maxwell system is very challenging in 3D (\cite{GS85, GS86}). Our result provides an affirmative answer to this open question in the context of boundary problems. For some previous works on the boundary problem of Vlasov-Poisson equations, we refer to \cite{EH, EHS, HV,  Rein} and the reference therein. For the construction of steady solutions and proving their asymptotic stability we use the framework of the second author, recently established in \cite{Kim}. The first step is to achieve asymptotic stability of the vacuum state. The key ingredient is to control a particle's traveling time by the total energy of the particle. In the second step, we bootstrap the stability of the vacuum to the stability of nontrivial steady states. At this stage, the key ingredient is to establish an exponential decay of the velocity derivative of steady solutions. This framework is established for a single species Vlasov-Poisson equation of non-relativistic particles. 

There are several mathematical challenges to developing the framework for relativistic particles in the presence of a magnetic field as follows:

\smallskip

(a)
The {\it first difficulty} comes from the fact that the acceleration of a relativistic particle is not solely determined by the Lorentz force. Even if gravity is the dominant factor affecting the vertical component of the Lorentz force, the parallel momentum may still increase. This complication occurs because the relativistic velocity $v_i$ in \eqref{totalE} depends not only on its corresponding momentum $p_i$ but also on the total magnitude of the momentum $|p|$.

To accurately estimate $F_{\pm} (t,x,p)$ in \eqref{VP_F}, we need to trace $F_{\pm} (t,x,p)$ backward to either the boundary conditions $G_{\pm}$ or the initial data $F_{\pm, 0}$ along the characteristic trajectory. Thus, we need a thorough understanding of the backward exit time, which is the time it takes for the current particle to travel back to the boundary or initial data along the characteristic trajectory (see Definition \ref{def:tb} for details).

We overcome this difficulty by deeply leveraging the dominance of gravity. 
Consider any particle $(t,x,p) \in \R_+ \times \O \times \R^3$. Due to the gravitational dominance, we suppose this particle has experienced two phases (ascending and descending) during its characteristic trajectory from the boundary.
The period of descent can be determined by comparing $p_3$ with the acceleration of the vertical momentum. This is because the vertical momentum and vertical velocity reach zero simultaneously, and the vertical momentum is controlled by the vertical acceleration.
It is more challenging to bound the period of ascent. This is because the particle's vertical velocity is zero when it reaches the peak point along the characteristic trajectory. Hence we meticulously select a time frame based on the magnitude of the total momentum and demonstrate that at this chosen time, the vertical momentum is uniformly comparable to the total momentum. This critical comparison allows us to establish a lower bound for the vertical velocity, which in turn enables us to control the particle's travel time to the boundary. By doing so, we effectively manage the potential issues caused by the rapid growth of parallel momentum, ensuring the stabilization effect is maintained.

\smallskip

(b)
The {\it second difficulty} comes from the presence of a magnetic field. Specifically, some derivatives of the particle's trajectory can grow exponentially as a function of both the strength of the magnetic field and the particle's travel time, which may complicate the analysis of the particle's behavior.

We overcome this difficulty by a key observation: the magnetic field does not affect the vertical momentum. This insight is crucial because it allows us to focus on the third component of the particle's motion, which is not influenced by the magnetic field.
Then we leverage our control over the particle's travel time, which is derived from its total energy. By using this control, we can apply a strong weight function $e^{\beta p^0_\pm}$ to effectively manage the exponential growth term. This weight function helps to balance and control the potentially unbounded growth caused by the magnetic field.
Additionally, we employ elliptic estimates on the Poisson equations (see Lemma \ref{lem:rho_to_phi}) to further support our analysis. These estimates provide a way to control the regularity of the solutions in both steady and dynamic cases. By combining the strong weight function with these elliptic estimates, we can establish rigorous regularity estimates for the particle's behavior despite the complexities introduced by the magnetic field.

\smallskip

(c)
The {\it third difficulty} comes from the weaker weight in $|p|$ imposed by the J\"uttner upper bound $e^{-\beta p^0_\pm}$. In contrast, in the non-relativistic case we benefit from a stronger Gaussian bound derived from the conservation of energy, which simplifies the analysis.
In the relativistic case, although we have a modified conservation of energy (see Lemma \ref{lem:conservation_law}), the characteristic trajectory remains influenced by $|p|$. This weaker weight introduces challenges in ensuring the exponential decay of $f_{\pm} (t,x,p)$, which is crucial for stability, and in maintaining the regularity of $F_{\pm} (t,x,p)$.

We overcome this difficulty by carefully using the conservation laws and a precise control on $|\nabla_x \phi_F|$.
Beyond the influence of downward gravity across the field, this extra bound can evolve a better estimate of the backward momentum (see Proposition \ref{lem3:tB}). 
Specifically, we leverage the fact that energy is bounded by the backward momentum at the boundary. By using this relationship, we effectively manage the difficulties associated with the weaker weight and ensure both the stability of $f_{\pm}$ and the regularity of $F_{\pm}$.




\smallskip

(e) The fourth difficulty comes from analyzing the asymptotic behavior of the solution as $|x_\parallel | \rightarrow \infty$.
To address this, we consider both the isothermal and non-isothermal J\"uttner distributions, with the non-isothermal wall temperature $T_{\pm}(x)$ converging to a unit temperature as $|x_{\|}| \to \infty$.
Recall that Theorem\ref{theo:CS} establishes the existence and regularity of stationary solutions, $h_{I, \pm}$ and $h_{\pm}$, under the isothermal and non-isothermal J\"uttner distributions, respectively. However, it does not provide control over the solution in the $x_\parallel$-direction.

We overcome this difficulty by a careful analysis of the weight function $\langle x_\parallel \rangle^3 | \hat{w}_{\pm} (h_{I, \pm} - h_{\pm}) |$ defined in \eqref{def:hat_w_shock}.
By following the characteristic trajectories, we decompose the weight function into two parts as follows:
\Be \notag
\begin{split}
& \langle x_\parallel \rangle^3 | \hat{w} (h_{I} - h) (x,p) | 
\leq \langle x_\parallel \rangle^3 e^{ \beta' \sqrt{(m c)^2 + |\pb (x,p)|^2} } | h (\xb, \pb) - h_I (\xb, \pb) |
\\& \ \ + 
\langle x_\parallel \rangle^3 e^{ \beta' \left( \sqrt{(m c)^2 + |p|^2} + \frac{1}{c} ( e \Phi (x) + m g x_3 ) \right) } \int^0_{-\tb(x,p)} e \nabla \phi (X(s;x,p)) \cdot \nabla_p h_{I} (X(s;x,p), P(s;x,p)) \dd s.
\end{split}
\Ee
The first part relates to the boundary condition and can be controlled by the fact that $T_{\pm} (x) \to 1$ as $|x_{\|}| \to \infty$.
For the second part, we compute the relationship between $(X(s;x,p), P(s;x,p))$ and $(x, p)$ for all $-\tb(x,p) \leq s \leq 0$. By referring to the steps in the proof of Theorem \ref{theo:US} and utilizing the control on $\nabla_p h_{I}$ provided in Theorem \ref{theo:CS}, we obtain an estimate for the weight function.

\smallskip

(f)
The fifth difficulty arises when considering the combination of inflow and specular boundary conditions. 
The challenge stems from the fact that specular boundary conditions cause the characteristic trajectory to undergo continuous bouncing, whereas under pure inflow conditions, the trajectory terminates upon reaching the boundary.
As a result, under specular boundary conditions, the magnitude of the momentum $|p|$ can increase indefinitely after several bounces. This behavior can potentially lead to unbounded regularity issues for both stationary and dynamic solutions.

We overcome this difficulty by utilizing asymptotic stability and a priori estimates for the sequential functions $\{ f^{\ell+1}_{\pm} \}^{\infty}_{\ell = 0}$ and $\{ \nabla_x \Psi^{\ell} \}^{\infty}_{\ell = 0}$ derived from the construction process (see Section \ref{sec:DC_spec}). 
First, we impose precise restrictions on both the inflow and initial data to leverage asymptotic stability, demonstrating the overall stability of the solutions.
Next, we derive a priori estimates to manage the sequential functions. These estimates establish a framework that ensures each function in the sequence is appropriately related to its predecessor.
Subsequently, we use mathematical induction to show that the sequential functions $\{ f^{\ell+1}_{\pm} \}^{\infty}_{\ell = 0}$ have uniformly compact supports with respect to the momentum $p$, as detailed in Proposition \ref{prop:DC_spec}.
Finally, we apply similar methods under pure inflow boundary conditions to establish the existence and stability of the solutions, ensuring that the complexities introduced by the boundary conditions are effectively managed.

\section{Main Results} \label{sec:MR}

We record technical statements of main theorems (Theorem \ref{theo:CS} and Theorem \ref{theo:CD}) in this section.

\subsection{Stationary Results}

For the reader's convenience, we present the corresponding steady problem to \eqref{VP_F}-\eqref{bdry:F} for $F_\pm (x,p) = h_{\pm} (x,p)$ and $\phi_F (x) = \Phi_h (x)$:
\begin{align}
v_\pm \cdot \nabla_x h_{\pm} + \big( e_{\pm} ( \frac{v_\pm}{c} \times B - \nabla_x \Phi_h ) - \nabla_x (m_{\pm} g x_3) \big) \cdot \nabla_p h_{\pm} = 0 \ \ & \text{in} \ \O \times \R^3, \label{VP_h} \\
h_{\pm} (x,p) = G_{\pm} (x,p) \ \ & \text{in} \ \gamma_-, \label{bdry:h}
\end{align}
where $B= (0,0, B_3)$ and the relativistic velocity $v_\pm$ is defined in \eqref{totalE}. A steady electric potential is determined by solving:
\Be \label{eqtn:Dphi}
\begin{split}
- \Delta_x \Phi_h (x) = \rho_h (x) \ \ \text{in } \O,  
\ \ \text{and} \  \ 
\Phi_h  = 0 \ \ \text{on } \p\O,
\end{split}
\Ee
where the steady local charge density is given by
\Be \label{def:rho}
\rho_h (x) = \int_{\R^3} ( e_+ h_+ + e_{-} h_{-} ) \dd p. 
\Ee
For simplicity, we often let $(-\Delta_0)^{-1} \rho_h$ denote $\Phi_h$ solving \eqref{eqtn:Dphi}. 

\medskip

Next, we provide a precise definition of weak solutions to the steady problem. 

\begin{definition} \label{weak_sol}

(a) We say that $(h_{\pm}, \nabla_x  \Phi_h) \in \big( L^2_{loc}(\O \times \R^3) \; \cap \; L^2_{loc}(\p\O \times \R^3; \dd \gamma) \big) \times L^2_{loc}(\O \times \R^3)$ is a weak solution of \eqref{VP_h} and \eqref{bdry:h}, if all terms are bounded and the following condition holds for any test function $\psi \in C^\infty_c (\bar \O \times \R^3)$, 
\Be \label{weak_form}
\begin{split}
& \iint_{\O \times \R^3} h_{\pm} (x,p) v \cdot \nabla_x \psi(x, p) \dd p \dd x 
\\& \ \ \ \ - \iint_{\O \times \R^3} h_{\pm} (x, p) \big( e_{\pm} ( \frac{v_\pm}{c} \times B 
- \nabla_x \Phi_h ) - \nabla_x (m_{\pm} g x_3) \big) \cdot \nabla_p \psi (x,p)  \dd p \dd x 
\\& = \int_{ \p\O \times \{ p_3 <0 \}} h_{\pm} (x,p) \psi(x,p) \dd \gamma  - \int_{ \p\O \times \{ p_3 > 0 \}} G_{\pm} (x,p) \psi(x,p) \dd \gamma,
\end{split}
\Ee
where $\dd \gamma := |v_3| \dd S_x \dd p$ denotes the phase boundary measure.
	
(b) We say that $(h_{\pm}, \rho_h) \in L^2_{loc}(\O \times \R^3) \times L^2_{loc}(\O  )$ is a weak solution of \eqref{def:rho}, if all terms below are bounded and the following condition holds for any test function $\vartheta \in H^1_0 (\O) \cap C^\infty_c (\bar \O)$, 
\Be \label{weak_form_2}
\int_{\O} \vartheta (x) \rho_h (x) \dd x 
= \int_{\O} \vartheta (x) \int _{\R^3} ( e_+ h_{+} (x, p) + e_{-} h_{-} (x, p) ) \dd p \dd x.
\Ee
	
(c) We say that $(\rho_h, \Phi_h) \in L^2_{loc}(\O  ) \times W^{1,2}_{loc} (\O)$ is a weak solution of \eqref{eqtn:Dphi}, if all terms below are bounded and the following condition holds for any test function $\varphi \in H^1_0 (\O) \cap C^\infty_c (\bar \O)$, 
\Be \label{weak_form_3}
\int_{\O} \nabla_x \Phi_h \cdot \nabla_x \varphi \dd x 
= \int_{\O} \rho_h \varphi \dd x.
\Ee
\end{definition}

To state the main theorem for the steady problem, we introduce two weight functions. The first weight function is an exponential of the total particle energy, which remains invariant along the particle's trajectory.  

\begin{definition} 

Suppose $\Phi_h (x) $ is well-defined. For $\beta>0$, we define a weight function for the steady problem \eqref{VP_h}:
\Be \label{w^h}
w_{\pm} (x,p) =	w_{\pm, \beta} (x,p) = e^{ \beta \left( \sqrt{(m_{\pm} c)^2 + |p|^2} + \frac{1}{c} ( e_{\pm} \Phi_h (x) + m_{\pm} g x_3 ) \right) }.
\Ee
\end{definition}

The second weight function is used to measure the singularity of the derivative of the distribution functions, as discussed in \cite{Guo95}.

\begin{definition}

Suppose $\p_{x_3} \Phi_h (x_\parallel , 0)$ is well-defined for $x_\parallel = (x_1, x_2) \in \R^2$. Then we define a steady kinetic distance 
\Be \label{alpha}
\alpha_{\pm} (x, p) = \sqrt{ |x_3|^2  + |v_{\pm, 3} |^2 +  2 ( e_{\pm} \p_{x_3} \Phi_h (x_\parallel , 0) + m_{\pm} g ) \frac{x_3}{ \sqrt{ m_\pm^2 + |p|^2/ c^2} } }.
\Ee
\end{definition}

Now we state the main theorem of the steady problem.

\begin{theorem}[Main theorem of the Steady problem] 
\label{theo:CS}
	
Suppose $\| e^{ \beta \sqrt{(m_{\pm} c)^2 + |p|^2} }G_{\pm} \|_{L^\infty(\gamma_-)} < \infty$ and $g, \beta > 0$ with $g \beta \gg 1$ satisfy the following condition:
\Be \label{condition:beta}
\beta \geq
\frac{\mathfrak{C} }{\min \big( \frac{m_{+}}{e_{+}}, \frac{m_{-}}{e_{-}} \big) \times \  \frac{g}{2}} \big( e_{\pm} \| e^{ \beta \sqrt{(m_{\pm} c)^2 + |p|^2} } G_{\pm} \|_{L^\infty(\gamma_-)} \big) 
\times \big( 1 + \frac{2 c}{ \beta \hat{m} g } \big),
\Ee
where $\mathfrak{C}>0$ is defined in \eqref{est:nabla_phi} and $\hat{m} = \min\{m_+, m_- \}$.
Furthermore, let $\beta \geq \tilde \beta >0$ satisfy the following condition:
\Be \label{condition:tilde_beta}
\begin{split}
& \big( e_+ \| e^{ \beta \sqrt{(m_{+} c)^2 + |p|^2} } G_+ \|_{L^\infty(\gamma_-)} + e_{-} \| e^{ \beta \sqrt{(m_{-} c)^2 + |p|^2}} G_- \|_{L^\infty(\gamma_-)} \big)
\\& \ \ \ \ \times \log \Big( e+ \| e^{\tilde \beta \sqrt{(m_{+} c)^2 + |p|^2}} \nabla_{x_\parallel, p} G_+ \|_{L^\infty (\gamma_-)} + \| e^{\tilde \beta \sqrt{(m_{-} c)^2 + |p|^2}} \nabla_{x_\parallel, p} G_- \|_{L^\infty (\gamma_-)} \Big) 
\\& \leq \beta (\frac{m_+ + m_- }{8} g\tilde  \beta - (1 + B_3) ),
\end{split}	
\Ee
and
\be \label{condition:G_xv}
(1 + \frac{\tilde{\beta}}{\hat{m} g} ) \| e^{\tilde \beta \sqrt{(m_{\pm} c)^2 + |p|^2}} \nabla_{x_\parallel, p} G_{\pm} \|_{L^\infty (\gamma_-)} \leq \frac{1}{4}.
\ee
Then there exists a unique solution $(h_{\pm}, \rho_h, \Phi_h)$ to \eqref{VP_h}-\eqref{def:rho} in the sense of Definition \ref{weak_sol}. Moreover, the following estimates hold: 
\begin{align}
& \| w_{\pm, \beta} h_{\pm} \|_{L^\infty (\bar \O \times \R^3)} 
\leq \| e^{ \beta\sqrt{(m_{\pm} c)^2 + |p|^2} } G_{\pm} \|_{L^\infty (\gamma_-)},
\label{Uest:wh} 
\\& | \rho (x) |
\leq \frac{1}{\beta} \big( e_+ \| w_{+, \beta} G_+ \|_{L^\infty(\gamma_-)} e^{-  \beta \frac{ m_+ }{2 c} g x_3} + e_{-} \| w_{-, \beta} G_- \|_{L^\infty(\gamma_-)} e^{- \beta \frac{ m_{-} }{2 c} g x_3} \big), 
\label{Uest:rho} 
\\& \| \nabla_x \Phi_h \|_{L^\infty (\bar{\O})} 
\leq  \min \big( \frac{m_{+}}{e_{+}}, \frac{m_{-}}{e_{-}} \big) \times \  \frac{g}{2},
\label{Uest:DPhi}
\\& \frac{8}{\hat{m} g} (1 + B_3 + \| \nabla_x ^2 \Phi _h \|_\infty) \leq \tilde \beta \leq \beta.
\label{Uest:Phi_xx}
\end{align}
In addition, $\rho$ and $\nabla_x \Phi _h$ satisfy the following bounds:
\begin{align}
& e^{ \frac{\tilde \beta \hat{m} g}{4 c} x_3 } |\p_{x_i} \rho  (x)|   
\lesssim \| e^{\tilde \beta \sqrt{(m_+ c)^2 + |p|^2}} \nabla_{x_\parallel, p} G_+ \|_{L^\infty (\gamma_-)}  
\times \Big( 1 + \mathbf{1}_{|x_3| \leq 1} \frac{1}{\sqrt{ m_+ g x_3 }} \Big)
\notag \\
& \qquad \qquad \qquad \qquad + \| e^{\tilde \beta \sqrt{(m_- c)^2 + |p|^2}} \nabla_{x_\parallel, p} G_- \|_{L^\infty (\gamma_-)}  
\times \Big( 1 + \mathbf{1}_{|x_3| \leq 1} \frac{1}{\sqrt{ m_- g x_3 }} \Big),
\label{est_final:rho_x} \\
& \| \nabla_x^2 \Phi_h  \|_\infty \lesssim \frac{1}{\beta} \big( e_+ \| w_{+, \beta} G_+ \|_{L^\infty(\gamma_-)} + e_{-} \| w_{-, \beta} G_- \|_{L^\infty(\gamma_-)} \big)
\notag \\
& \ \ \times \log \big( e+ \| e^{\tilde \beta \sqrt{(m_+ c)^2 + |p|^2}} \nabla_{x_\parallel, p} G_+ \|_{L^\infty (\gamma_-)} + \| e^{\tilde \beta \sqrt{(m_- c)^2 + |p|^2}} \nabla_{x_\parallel, p} G_- \|_{L^\infty (\gamma_-)} \big).
\label{est_final:phi_C2}
\end{align}
Furthermore, given $\alpha_\pm (x, p)$ defined in $\eqref{alpha}$, then $h_{\pm}$ satisfies that
\begin{align}
& e^{ \frac{\tilde \beta}{2}|p^0_{\pm}|} e^{  \frac{\tilde \beta m_{\pm} g}{4 c} x_3} | \nabla_p h_{\pm} (x,p)| 
\notag \\
& \lesssim \big( 1 + \frac{ \| \nabla_x^2 \Phi _h \|_\infty + e B_3 + m_{\pm} g}{(m_{\pm} g)^2} \big) \| e^{\tilde \beta \sqrt{(m_{\pm} c)^2 + |p|^2}} \nabla_{x_\parallel, p} G_{\pm} \|_{L^\infty (\gamma_-)},
\label{est_final:hk_v} \\
& e^{ \frac{\tilde \beta}{2} |p^0_{\pm} |} e^{  \frac{\tilde \beta m_{\pm} g}{4 c} x_3} | \nabla_x h_{\pm} (x,p)|  
\notag \\
& \lesssim \big( \frac{\delta_{i3}}{\alpha_{\pm} (x,p)} + \frac{ \| \nabla_x^2 \Phi _h\|_\infty + e B_3 + m_{\pm} g}{(m_{\pm} g)^2} \big) \| e^{\tilde \beta \sqrt{(m_{\pm} c)^2 + |p|^2}} \nabla_{x_\parallel, p} G_{\pm} \|_{L^\infty (\gamma_-)}.
\label{est_final:hk_x}
\end{align}
\end{theorem}

A proof of this theorem is provided at the end of Section \ref{sec:EX_SS}.

\subsection{Dynamical Results}

We consider the dynamical problem \eqref{VP_F}-\eqref{bdry:F} for perturbation solutions around the steady solution $(h_\pm, \Phi_h)$ as described in Theorem \ref{theo:CS}. Specifically, the solutions are given by $F_\pm(t,x,p) = h_\pm(x,p) + f_\pm(t,x,p)$ in \eqref{def:F} and $\phi_F(t,x) = \Phi_h(x) + \Psi (t,x)$, where $(f_\pm, \Psi)$ solves the following equations:
\be \label{eqtn:f}
\begin{split}
	& \p_t f_{\pm}  + v_\pm \cdot \nabla_x f_{\pm} + \Big( e_{\pm} \big( \frac{v_\pm}{c} \times B - \nabla_x ( \Phi_h + \Psi ) \big) - \nabla_x ( m_\pm g x_3) \Big) \cdot \nabla_p f_{\pm} 
	\\& = e_{\pm} \nabla_x \Psi \cdot \nabla_p h_{\pm} 
	\ \ \text{in} \ \R_+ \times \O \times \R^3.
\end{split}
\ee

For the initial condition of \eqref{VP_0}, we let 
\be \label{def:F_0}
F_{\pm, 0} (x,p) = h_{\pm} (x,p) + f_{\pm, 0} (x,p) \ \ \text{in} \  \O \times \R^3,
\ee
which leads to the initial condition for $f_\pm$ given by
\be \label{f_0}
\begin{split}
f_{\pm} (0,x,p) = f_{\pm, 0} (x,p) 
\ \ \text{in} \ \{ 0 \} \times \O \times \R^3.
\end{split}
\ee
The boundary condition simplifies to an absorption condition:
\Be \label{bdry:f} 
f_{\pm} (t,x,p)= 0 \ \ \text{in} \ \R_+ \times \gamma_-. 
\Ee

An electric potential corresponding to the dynamical perturbation is determined by solving:
\be \label{Poisson_f}
- \Delta_x	\Psi (t,x)
= \varrho (t,x) \ \  \text{in} \ \R_+ \times \O, 
\ \text{and }
\Psi(t,x) = 0  \ \   \text{on} \ \R_+ \times \p\O,  
\ee  
where a local charge density of the dynamical fluctuation is given by
\Be \label{def:varrho}
\varrho(t,x) 
=  \int_{\R^3} ( e_+ f_+ + e_{-} f_{-} ) \dd p.
\Ee
For simplicity, we often let $(-\Delta_0)^{-1} \varrho$ denote $\Psi$ solving \eqref{Poisson_f}. 

\medskip

Next, we provide a precise definition of weak solutions to the dynamical problem. 

\begin{definition} \label{weak_sol_dy} 

Suppose $(h_{\pm}, \rho_h, \Phi_h)$ solves \eqref{VP_h}-\eqref{def:rho} in the sense of Definition \ref{weak_sol}.  
	
(a) We say that $(f_{\pm}, \nabla_x \Psi) \in \big( L^2_{loc}(\R_+ \times \O \times \R^3) \; \cap \; L^2_{loc}(\R_+ \times \p\O \times \R^3; \dd \gamma) \big) \times L^2_{loc}(\R_+ \times \O \times \R^3)$ is a weak solution of \eqref{eqtn:f}, \eqref{f_0}, and \eqref{bdry:f}, if all terms below are bounded and the following condition holds for any $t \in R_+$ and test function $\psi \in C^\infty_c (\R_+ \times \bar \O \times \R^3)$, 
\Be \label{weak_form_dy}
\begin{split}
& \iint_{\R_+ \times \O \times \R^3} f_{\pm} (t,x,p) \p_t \psi(t, x, p) \dd p \dd x \dd t + \iint_{\R_+ \times \O \times \R^3} f_{\pm} (t,x,p) v \cdot \nabla_x \psi(t, x, p) \dd p \dd x \dd t
\\& - \iint_{\R_+ \times \O \times \R^3} f_{\pm} (t, x, p) \big( e_{\pm} ( \frac{v_\pm}{c} \times B 
- \nabla_x (\Phi + \Psi) ) - \nabla_x (m_{\pm} g x_3) \big) \cdot \nabla_p \psi (t,x,p)  \dd p \dd x \dd t
\\& = \iint_{\R_+ \times \p\O \times \{ p_3 <0 \}} f_{\pm} (t,x,p) \psi(t,x,p) \dd \gamma \dd t 
- \iint_{ \O \times \R^3} f_{\pm, 0} (x,p) \psi(0, x,p) \dd p \dd x,
\end{split}
\Ee
where $\dd \gamma := |v_3| \dd S_x \dd p$ denotes the boundary measure.
	
(b) We say that $(f_{\pm}, \varrho) \in L^2_{loc}(R_+ \times \O \times \R^3) \times L^2_{loc} (R_+ \times \O)$ is a weak solution of \eqref{def:varrho}, if all terms below are bounded and the following condition holds for any $t \in R_+$ and test function $\vartheta \in H^1_0 (\O) \cap C^\infty_c (\bar \O)$, 
\Be \label{weak_form_2_dy}
\int_{\O} \vartheta (x) \varrho (t, x) \dd x 
= \int_{\O} \vartheta (x) \int _{\R^3} ( e_+ f_{+} (t, x, p) + e_{-} f_{-} (t, x, p) ) \dd p \dd x.
\Ee
	
(c) We say that $(\varrho, \Psi) \in L^2_{loc}(\O  ) \times W^{1,2}_{loc} (\O)$ is a weak solution of \eqref{Poisson_f}, if all terms below are bounded and the following condition holds for any $t \in R_+$ and test function $\varphi \in H^1_0 (\O) \cap C^\infty_c (\bar \O)$,
\Be \label{weak_form_3_dy}
\int_{\O} \nabla_x \Psi (t, x) \cdot \nabla_x \varphi (x) \dd x 
= \int_{\O} \varrho (t, x) \varphi (x) \dd x.
\Ee
\end{definition}

Analogous to the steady problem, we introduce two weight functions for the dynamic problem. 
Remark that unlike in the steady case, the first weight function is not invariant along the particle’s trajectory.

\begin{definition}

Suppose $\Phi_h (x) + \Psi(t,x)$ is well-defined. 
For $\beta>0$, we define a weight function for the dynamic problem \eqref{eqtn:f}:
\Be \label{w^F}
\begin{split}
\w_{\pm} (t,x,p) &= \w_{\pm, \beta} (t,x,p) = e^{ \beta \left( \sqrt{(m_{\pm} c)^2 + |p|^2} + \frac{1}{c} \big( e_{\pm} \big( \Phi_h (x) + \Psi(t,x) \big) + m_{\pm} g x_3 \big) \right) }, \\
\w_{\pm,0} (x,p) &= \w_{\pm, \beta,0} (x,p) = e^{ \beta \left( \sqrt{(m_{\pm} c)^2 + |p|^2} + \frac{1}{c} \big( e_{\pm} \big( \Phi_h (x) + \Psi(0,x) \big) + m_{\pm} g x_3 \big) \right) }.
\end{split}
\Ee
\end{definition}

\begin{definition} \label{def:alpha_F} 

Suppose $\p_{x_3} \phi_F (t,x_\parallel, 0) = \p_{x_3} \Phi_h(x_\parallel, 0) + \p_{x_3} \Psi (t,x_\parallel, 0)$ is well-defined for $x_\parallel = (x_1, x_2) \in \R^2$. Then we define a dynamic kinetic distance
\Be \label{alpha_F}
\alpha_{\pm, F} (t, x, p) = \sqrt{ |x_3|^2  + |v_{\pm, 3} |^2 +  2 ( e_{\pm} \p_{x_3} \phi_F (t, x_\parallel, 0) + m_{\pm} g ) \frac{x_3}{ \sqrt{ m_\pm^2 + |p|^2/ c^2}  } }.
\Ee
\end{definition}

To present the main theorem for the dynamical problem, we introduce additional notations that will be utilized in the theorem statement.
Given the initial data $F_{\pm, 0} (x,p)$ and the boundary data $G(x, p)$ and $\beta \geq \tilde \beta > 0$, we define:
\begin{align}
M & := 2 \sum\limits_{i = \pm} \big(
\| \w_{i, \beta, 0 } F_{i, 0} \|_{L^\infty (\O \times \R^3)}
+  \| e^{\beta \sqrt{(m_{i} c)^2 + |p|^2} } G_i \|_{L^\infty (\gamma_-)} \big),
\label{set:M} \\
L & := \sum\limits_{i = \pm} \big( \| \w_{i, \tilde \beta, 0}   \nabla_{x,v} F_{i, 0}  \|_{L^\infty (\O \times \R^3)} + \|  e^{\beta \sqrt{(m_{i} c)^2 + |p|^2} } \nabla_{x_\parallel,p} G_i \|_{L^\infty (\gamma_-)} \big).
\label{set:L}
\end{align}
Further, we define:
\Be \label{lambda}
\hat{m} = \min\{m_+, m_- \}
\ \text{ and } \
\lambda  = \frac{g \beta }{48} \hat{m}. 
\Ee

\medskip

Now we state the main theorem of the dynamic problem.

\begin{theorem}[Main theorem of the Dynamic problem]
\label{theo:CD}

We assume that all assumptions in Theorem \ref{theo:CS} hold with $g > 0, \beta \geq \tilde \beta >0$. Furthermore, suppose that
\Be \label{choice:g}
M \leq \beta e^{ - \frac{m_{\pm} g}{24} \beta}
\ \text{ and } \ 
L \leq \min \{ \tilde{\beta} e^{ - \frac{m_{\pm} g}{24} \tilde{\beta} }, \frac{1}{1024} \beta^2 e^{ - \frac{m_{\pm} g}{48} \beta } \}.
\Ee
Then there exists a unique solution $(f_\pm, \varrho, \Psi)$ to \eqref{eqtn:f}-\eqref{Poisson_f} in the sense of Definition \ref{weak_sol_dy}. Moreover, the following estimates hold:
\begin{align}
& \sup_{0 \leq t < \infty} \| e^{ \frac{\beta}{2} \sqrt{(m_{\pm} c)^2 + |p|^2} } e^{ \frac{m_{\pm} g}{4 c} \beta x_3} f_{\pm} (t,x,p)  \|_{L^\infty  (\O \times \R^3)} \leq M, 
\label{Uest:wh_dy} \\
& \sup_{0 \leq t < \infty} \|   \nabla_x  \Psi \|_{L^\infty (\bar{\O})} \leq \min \left\{\frac{m_+}{e_+}, \frac{m_-}{e_-} \right\} \times \frac{g}{48}.
\label{Uest:DxPsi}
\end{align}
In addition, $\phi_F =  \Phi_h + \Psi$ satisfies the following bounds:
\begin{align}
& \sup_{0 \leq t < \infty} \| \nabla_x \phi_F \|_{L^\infty (\bar{\O})} 
\leq \min \left\{\frac{m_+}{e_+}, \frac{m_-}{e_-} \right\} \times \frac{g}{2}, 
\label{Uest:Dxphi_F} \\ 
& (1 + B_3 + \| \nabla_x ^2 \phi_F  \|_\infty) + \| \p_t \p_{x_3} \phi_F (t, x_\parallel , 0) \|_{L^\infty(\p\O)} 
\leq \frac{\hat{m} g}{24} \tilde{\beta}.  
\label{Uest:D2xD3tphi_F}
\end{align}	
Furthermore, given $\alpha_{\pm, F} (t, x, p)$ defined in $\eqref{alpha_F}$, then $f_{\pm}$ satisfies that
\begin{align}
& \| e^{ \frac{\tilde{\beta}}{4} \sqrt{(m_{\pm} c)^2 + |p|^2} } e^{ \frac{m_{\pm} g}{8 c} \tilde{\beta} x_3} \nabla_p f_{\pm} (t,x,p) \|_{L^\infty(\O \times \R^3)} 
\lesssim 2 e^{ 2 } \tilde \beta + \frac{1}{24} {\tilde{\beta}}^2,
\label{Uest_final:F_v:dyn} \\
& e^{ \frac{\tilde{\beta}}{4} \sqrt{(m_{\pm} c)^2 + |p|^2} } e^{ \frac{m_{\pm} g}{8 c} \tilde{\beta} x_3} \big| \nabla_x f_{\pm} (t,x,p) \big|
\lesssim 2 e^{ 2 } \tilde{\beta} + ( \frac{ \mathbf{1}_{|x_3| \leq 1} }{\alpha_{\pm, F} (t,x,p)} + \frac{1}{24} \tilde{\beta} ) \tilde{\beta}.
\label{Uest_final:F_x:dyn}	
\end{align}
Finally, the solution $(f(t), \varrho(t), \nabla_x \Psi(t))$ exhibits exponential decay as follows:
\begin{align}
& \sup_{0 \leq t < \infty} e^{ \lambda t} \|  e^{\frac{\beta}{8} \sqrt{(m_{\pm} c)^2 + |p|^2} + \frac{m_{\pm} g}{16 c} \beta x_3} f_{\pm} (t ) \|_{L^\infty (\O \times \R^3)} \lesssim \beta, \label{Udecay:f} \\
& \sup_{0 \leq t < \infty} e^{ \lambda t} \| e^{  \frac{3 \lambda }{c}  x_3} \varrho (t)\|_{L^\infty(\O)} \lesssim \frac{e_+ + e_-}{\beta^2}, \label{Udecay:varrho}\\
& \sup_{0 \leq t < \infty} e^{ \lambda t} \|\nabla_x \Psi(t)\|_{L^\infty (\O)}
\lesssim \big( \frac{e_+ + e_-}{\beta^2} \big) \times (1 + \frac{16}{g \beta}). \label{Udecay:DxPsi}
\end{align} 
\end{theorem}

A proof of this theorem is provided at the end of Section \ref{sec:EX_DS}.

\section{Relativistic Particle Trajectory} \label{sec:char}

Consider the characteristics $Z_{\pm} (s;x,p) = (X_{\pm} (s;x,p), P_{\pm} (s;x,p))$ for the steady problem \eqref{VP_h}: 
\Be \label{ODE_h}
\begin{split} 
\frac{d X_{\pm} (s;x,p) }{d s} & = V_{\pm} (s;x,p) = \frac{P_{\pm} (s;x,p)}{\sqrt{m^2_{\pm} + |P_{\pm} (s;x,p)|^2 / c^2}}, \\
\frac{d P_{\pm} (s;x,p) }{d s} & =  {e_{\pm}} ( V_{\pm} (s;x,p) \times \frac{B}{c} - \nabla_x \Phi_h (X_{\pm} (s;x,p) )) - {m_{\pm}} g \mathbf{e}_3,
\end{split}
\Ee
where $\Phi_h $ solves \eqref{eqtn:Dphi} and $\mathbf{e}_3= (0,0,1)^{\intercal}$. At $s = 0$, we have
\[
Z_{\pm} (0; x, p) = (X_{\pm} (0; x, p), P_{\pm} (0; x,p)) = (x, p) = z_{\pm}.
\]
Note that $\frac{dX_\pm (s;x,p) }{ds}$ is the particle velocity; and $\frac{d P_\pm (s;x,p) }{ds} $ is the relativistic 3-force. 
An important quantity of the vertical particle velocity: 
\Be \label{sign:v3}
\sgn \big( V_{\pm, 3} (s;x,p) \big) 
= \sgn \left( \frac{P_{\pm, 3} (s;x,p)}{\sqrt{m^2_{\pm} + |P_{\pm} (s;x,p)|^2 / c^2}} \right)
= \sgn \big( P_{\pm, 3} (s;x,p) \big).
\Ee

We also define the characteristics $\Z_{\pm} (s;t,x,p) = (\X_{\pm} (s;t,x,p), \P_{\pm} (s;t,x,p))$ for the dynamical problem \eqref{eqtn:f} solving 
\Be \label{ODE_F}
\begin{split}
\frac{d \X_{\pm} (s;t,x,p) }{ d s} & = \V_{\pm} (s;t,x,p) = \frac{\P_{\pm} (s;t,x,p)}{\sqrt{m^2_{\pm} + |\P_{\pm} (s;t,x,p)|^2 / c^2}}, \\
\frac{d \P_{\pm} (s;t,x,p)}{d s} & = {e_{\pm}} \big( \V_{\pm} (s;t,x,p) \times \frac{B}{c} - \nabla_x \Psi (s, \X(s;t,x,p)) 
\\& \qquad \qquad - \nabla_x \Phi_h (\X(s;t,x,p)) \big) - {m_{\pm}}  g  \mathbf{e}_3,
\end{split}
\Ee 
and satisfying 
$\Z_{\pm} (t;t,x,p) = (\X_{\pm} (t;t,x,p), \P_{\pm} (t;t,x,p))  = (x, p) = z_{\pm}$.
Here, $\Psi$ and $\Phi_h$ solve \eqref{Poisson_f} and \eqref{eqtn:Dphi}, respectively.  
In the dynamical case, $\frac{d \X_\pm (s;t,x,p) }{ds}$ is the particle velocity; and $\frac{d \P_\pm (s;t,x,p) }{ds} $ is the relativistic 3-force. 
Similarly, the vertical particle velocity has an important quantity as follows:
\Be \label{sign:v3_dy}
\sgn \big( \V_{\pm, 3} (s;t,x,p) \big) 
= \sgn \left( \frac{ \P_{\pm, 3} (s;t,x,p)}{\sqrt{m^2_{\pm} + |\P_{\pm} (s;t,x,p)|^2 / c^2}} \right)
= \sgn \big( \P_{\pm, 3} (s;t,x,p) \big).
\Ee

Now we introduce the key concepts: the backward/forward exit time, and the backward/forward exit position and momentum.

\begin{definition} 
\label{def:tb}

(a)
Suppose $\nabla_x \Phi_h \in W^{1,p} (\O)$ for $p>2$. Then $Z_{\pm} (s;x,p)$ is well-defined as long as $X_{\pm} (s;x,p) \in \O$. There exists a backward/forward exit time 
\Be \label{tb^h}
\begin{split}
\tbp (x, p) & : = \sup \{ s \in [0,\infty) : X_{\pm, 3} (-\tau; x,p )>0  \ \text{for all }  \tau \in (0, s)\} \geq 0, \\
\tfp (x, p) & : = \sup \{ s \in [0,\infty) : X_{\pm, 3} ( +\tau; x,p )>0  \ \text{for all }  \tau \in (0, s)\}  \geq 0.
\end{split}	
\Ee
In particular, $X_{\pm, 3} ( -\tbp (x,p); x, p ) = X_{\pm, 3} (\tfp (x,p); x, p ) = 0$. Moreover, $Z_{\pm} (s;x,p)$ is continuously extended in a closed interval of $s \in [-\tbp (x, p), \tfp (x, p)]$.  

We also define backward/forward exit position and momentum:
\Be \label{def:zb^h}
\begin{split}
& \xbp (x,p) = X_{\pm} ( -\tbp (x, p); x, p) \in \p\O, \ \ 
\pbp (x,p) = P_{\pm} ( -\tbp (x, p); x, p),
\\& \xfp (x,p) = X_{\pm} ( \tfp (x, p); x, p) \in \p\O, \ \ 
\pfp (x,p) = P_{\pm} ( \tfp (x, p); x, p).
\end{split}
\Ee

(b)
Suppose $\nabla_x \Phi_h, \nabla_x \Psi(t, \cdot) \in W^{1,p} (\O)$ for $p>1$. Then $\Z_{\pm} (s;t,x,p)$ is well-defined as long as $\X_{\pm} (s;t,x,p) \in \O$. There exists a backward/forward exit time 
\Be \label{tb}
\begin{split}
\tBp (t,x,p)
&: = \sup \{
s \in [0,\infty): \X_{\pm, 3} (t - \tau; t,x,p) > 0 \ \text{for all } \tau \in (0,s)
\}
\geq 0, \\
\tFp (t,x,p)
&: = \sup \{
s \in [0,\infty): \X_{\pm, 3} (t + \tau; t,x,p) > 0 \ \text{for all } \tau \in (0,s)
\}
\geq 0,
\end{split}
\Ee
such that $\X_{\pm, 3} (t - \tBp (t,x,p); t,x,p) = 0$ and backward exit position and momentum are defined
\Be \label{def:zb}
\begin{split}
& \xBp (t,x,p) = \X_{\pm} ( t - \tBp (t,x,p); t,x,p ) \in \p\O, 
\\& \pBp (t,x,p) = \P_{\pm} (t - \tBp (t,x,p); t,x,p ).
\end{split}
\Ee
Then $\Z_{\pm} (s;t,x,p)$ is continuously extended in a closed interval of $s \in [t - \tBp (t,x,p), t]$. 
\end{definition}

For the characteristics \eqref{ODE_h}, the weak solutions to the steady problem \eqref{VP_h}-\eqref{eqtn:Dphi} satisfy:
\Be \label{Lform:h}
\begin{split}
h_{\pm} (x, p)  
& =  \mathbf{1}_{t \leq \tbp (x,p)} h_{\pm} (X_{\pm} (-t; x,p), P_{\pm} (-t; x,p) ) 
\\& \ \ \ \ +  \mathbf{1}_{t > \tbp (x,p)} G_{\pm} (\xbp (x,p), \pbp (x,p))).
\end{split}
\Ee
For the characteristics \eqref{ODE_F}, the weak solutions to the dynamic problem \eqref{eqtn:f}-\eqref{Poisson_f} satisfy:
\Be
\begin{split} \label{Lform:f}
& f_{\pm} (t,x,p) 
\\& = \mathbf{1}_{t \leq \tBp (t,x,p)} f_{\pm} (0, \X_{\pm} (0;t,x,p), \P_{\pm} (0;t,x,p))
\\& \ \ \ \ + \int^t_{ \max\{0, t - \tBp (t,x,p)\}} e \nabla_x \Psi (s, \X (s;t,x,p)) \cdot \nabla_p h_{\pm} ( \X_{\pm} (s;t,x,p), \P_{\pm} (s;t,x,p)) \dd s.
\end{split}
\Ee

In the remainder of this section, we establish several key estimates for $\tbp$ and $\tBp$ (see Propositions \ref{lem:tb}-\ref{lem3:tB}), which will be frequently used in subsequent sections. Before doing so, we provide a brief overview of the approach used to derive these estimates.

For the reader's convenience, we restate the coupled system of Vlasov-Poisson equations as follows:
\be \notag
\p_t F_{\pm} 
+ c \frac{p}{p^0_{\pm}} \cdot \nabla_x F_{\pm} 
+  \left(
\frac{p}{p^0_{\pm}} \times e_{\pm} B - \nabla_x (m_{\pm} g x_3) \right) \cdot \nabla_p F_{\pm} 
- \nabla_x ( e_{\pm} \phi_F)  \cdot \nabla_p F_{\pm}
= 0.
\ee
From Definition \ref{def:tb}, we have
\[
\X_{\pm} (t - \tBp (t,x,p); t,x,p) \in \p\O
\text{ if }
t \geq \tBp (t,x,p).
\]
This reveals the importance of $\tBp$, since it build the connection between the status $(t,x,p)$ and the boundary $(\X_{\pm} (t - \tBp (t,x,p); t,x,p), \P_{\pm} (t - \tBp (t,x,p); t,x,p)) \in \gamma_-$.
Together with $F_{\pm}$ being invariant along the characteristics, $F_{\pm} (t,x,p)$ can be controlled if we bound $\tB$ properly. 

From the assumption on $|\nabla_x \phi_F|$ (see \eqref{Uest:Dxphi_F} in Theorem \ref{theo:CD}), for any $(t,x,p) \in \R_+ \times  \O \times \R^3$,
\be \notag
\frac{d \P_{\pm, 3} (s;t,x,p) }{ d s} 
\leq - \frac{1}{2} m_{\pm} g,
\ee
Meanwhile, under direct computation, we also derive  for any $(t,x,p) \in \R_+ \times  \O \times \R^3$,
\be \notag
\Big| \frac{d}{ds} \sqrt{|\P_{\pm, 1} (s;t, x,p)|^2 + |\P_{\pm, 2} (s;t, x,p)|^2} \Big|
\leq \frac{m g}{2 \sqrt{2}}.
\ee
These two estimates show that
\be \notag
\P_{\pm, 3} (s;t, x,p) \geq \frac{m_{\pm} g}{2} |s-t| + p_3.
\ee
and
\be \notag
\sqrt{|\P_{\pm, 1} (s;t, x,p)|^2 + |\P_{\pm, 2} (s;t, x,p)|^2}
\leq \sqrt{p^2_1 + p^2_2} + \frac{m g}{2\sqrt{2}} (t-s).
\ee
If $\tBp (t,x,p) < \frac{4}{m_{\pm} g} |p| + 1$, then we already get the control. Otherwise, we deduce that for all $s \in [t- \tBp (t,x,p), t - (\frac{4}{m_{\pm} g} |p| + 1)]$,
\be \notag
\begin{split}
\V_{\pm, 3} (s;t, x,p) 
= \frac{ \P_{\pm, 3} (s;t, x,p) }{\sqrt{m^2 + |\P_{\pm} (s;t, x,p) |^2 / c^2 }}
\gtrsim 1,
\end{split}  
\ee
and thus $\tBp$ is controlled since $\V_{\pm, 3} \gtrsim 1$ has a lower bound.

\smallskip

We now present the conservation law in the relativistic setting, as stated in Lemma \ref{lem:conservation_law}.

\begin{lemma} \label{lem:conservation_law}

(a)
Suppose that $(x, p) \in \O \times \R^3$. Consider the characteristics trajectory $Z_{\pm} (s;x,p) = (X_{\pm} (s;x,p), P_{\pm} (s;x,p))$ in \eqref{ODE_h}. Then for all $s \in [-\tbp (x,p), 0]$,
\Be \label{eq:tb_conservation}
\begin{split}
& \ \ \ \ \sqrt{(m_{\pm} c)^2 + |P_{\pm} (s;x,p)|^2} + 
\frac{1}{c} \big( e_{\pm} \Phi_h (X_{\pm} (s;x,p)) + m_{\pm} g X_{\pm, 3} (s;x,p) \big) 
\\& = \sqrt{(m_{\pm} c)^2 + |\pbp (x,p)|^2}. 
\end{split}
\Ee 
For all $s \in [0, \tfp (x,p)]$,
\Be \label{eq:tf_conservation}
\begin{split}
& \ \ \ \ \sqrt{(m_{\pm} c)^2 + |P_{\pm} (s;x,p)|^2} + 
\frac{1}{c} \big( e_{\pm} \Phi_h (X_{\pm} (s;x,p)) + m_{\pm} g X_{\pm, 3} (s;x,p) \big) 
\\& = \sqrt{(m_{\pm} c)^2 + |\pfp (x,p)|^2}.  
\end{split}
\Ee 

(b) 
Further, suppose $(t, x, p) \in \R_+ \times \O \times \R^3$. Consider the characteristics 
$\Z_{\pm} (s;t,x,p) = (\X_{\pm} (s;t,x,p), \P_{\pm} (s;t,x,p))$ in \eqref{ODE_F}. Then for all $s - t \in [-\tBp (t,x,p), \tFp (t,x,p)]$,
\Be \label{eq:energy_conservation_dy}
\begin{split}
& \ \ \ \ \frac{d}{ds} \Big( \sqrt{(m_{\pm} c)^2 + |\P_{\pm} (s;t,x,p)|^2} + \frac{1}{c} \big( e_{\pm} \Phi_h (\X_{\pm} (s;t,x,p)) + m_{\pm} g \X_{\pm, 3} (s;t,x,p) \big) \Big) 
\\& = - \frac{e_{\pm}}{c} \nabla_x \Psi (s, \X_{\pm} (s;t,x,p)) \cdot \V_{\pm} (s;t,x,p),
\end{split}
\Ee 
and 
\Be \label{eq:energy_conservation_dy_2}
\begin{split}
& \frac{d}{ds} \Big( \sqrt{(m_{\pm} c)^2 + |\P_{\pm} (s;t,x,p)|^2} + \frac{1}{c} \big( e_{\pm} \big( \Phi_h (\X_{\pm} (s;t,x,p)) + \Psi (s, \X_{\pm} (s;t,x,p)) \big) 
\\& \qquad + m_{\pm} g \X_{\pm, 3} (s;t,x,p) \big) \Big) 
= \frac{e_{\pm}}{c} \p_t \Psi (s, \X_{\pm} (s;t,x,p)).
\end{split}
\Ee 
\end{lemma}

\begin{proof}

For part (a), both the backward part \eqref{eq:tb_conservation} and the forward part \eqref{eq:tf_conservation} can be proved by showing that 
\[
\sqrt{(m_{\pm} c)^2 + |P_{\pm} (s;x,p)|^2} + 
\frac{1}{c} \big( e_{\pm} \Phi_h (X_{\pm} (s;x,p)) + m_{\pm} g X_{\pm, 3} (s;x,p) \big),
\] is constant along the characteristics line \eqref{ODE_h}.
Thus we compute that
\be \label{eq:energy_conservation}
\begin{split}
& \ \ \ \ [ V_{\pm} \cdot \nabla_x + \big( e_{\pm} ( V_{\pm} \times \frac{B}{c} 
- \nabla_x \Phi_h ) - m_{\pm} g \mathbf{e}_3 \big) \cdot \nabla_p ]
\Big( \sqrt{(m_{\pm} c)^2 + |P_{\pm} (s;x,p)|^2} 
\\& \ \ \ \ + \frac{1}{c} \big( e_{\pm} \Phi_h (X_{\pm} (s;x,p)) + m_{\pm} g X_{\pm, 3} (s;x,p) \big) \Big) 
\\& = \big( e_{\pm} ( V_{\pm} \times \frac{B}{c} 
- \nabla_x \Phi_h ) - m_{\pm} g \mathbf{e}_3 \big) \cdot \frac{P_{\pm}}{\sqrt{(m_{\pm} c)^2 + |P_{\pm}|^2} } + \frac{V_{\pm}}{c} \cdot \big( e_{\pm} \nabla_x \Phi_h + m_{\pm} g \mathbf{e}_3 \big)
\\& = \big( - e_{\pm} \nabla_x \Phi_h - m_{\pm} g \mathbf{e}_3 \big) \cdot \frac{V_{\pm}}{c} + \frac{V_{\pm}}{c} \cdot \big( e_{\pm} \nabla_x \Phi_h + m_{\pm} g \mathbf{e}_3 \big) = 0,
\end{split}
\ee
where the last line follows from $V_{\pm} \times B = (B_3 V_{\pm, 2}, - B_3 V_{\pm, 1}, 0)$ and $\frac{P_{\pm}}{\sqrt{(m_{\pm} c)^2 + |P_{\pm}|^2} } = \frac{V_{\pm}}{c}$.

\smallskip

For part (b), we can directly compute
\be \notag
\begin{split}
& \ \ \ \ \frac{d}{ds} \Big( \sqrt{(m_{\pm} c)^2 + |\P_{\pm} (s;t,x,p)|^2} + \frac{1}{c} \big( e_{\pm} \Phi_h (\X_{\pm} (s;t,x,p)) + m_{\pm} g \X_{\pm, 3} (s;t,x,p) \big) \Big) 
\\& = [\p_t + \V_{\pm} \cdot \nabla_x + \big( e_{\pm} ( \V_{\pm} \times \frac{B}{c} - \nabla_x \Phi_h - \nabla_x \Psi ) - m_{\pm} g \mathbf{e}_3 \big) \cdot \nabla_p ]
\\& \qquad \ \ \Big( \sqrt{(m_{\pm} c)^2 + |\P_{\pm} (s;t,x,p)|^2}  + \frac{1}{c} \big( e_{\pm} \Phi_h (\X_{\pm} (s;t,x,p)) + m_{\pm} g \X_{\pm, 3} (s;t,x,p) \big) \Big)
\\& = \V_{\pm} \cdot \frac{1}{c} (e_{\pm} \nabla_x \Phi_h (\X) + m_{\pm} g \mathbf{e}_3 ) + \Big( e_{\pm} \big( \V \times \frac{B}{c} - \nabla_x \Phi_h - \nabla_x \Psi \big) - m_{\pm} g \mathbf{e}_3 \Big) \cdot  \frac{\V_{\pm}}{c}
\\& = \V_{\pm} \cdot \frac{1}{c} (e_{\pm} \nabla_x \Phi_h (\X) + m_{\pm} g \mathbf{e}_3 ) - \Big( e_{\pm} \big( \nabla_x \Phi_h + \nabla_x \Psi \big) + m_{\pm} g \mathbf{e}_3 \Big) \cdot  \frac{\V_{\pm}}{c} = - \frac{e_{\pm}}{c} \nabla_x \Psi \cdot  \V_{\pm},
\end{split}
\ee
and thus obtain \eqref{eq:energy_conservation_dy}.
Moreover, from the following equality:
\be \notag
\frac{d}{ds} \Psi (s, \X_{\pm} (s;t,x,p)) 
= \p_t \Psi (s, \X_{\pm} (s;t,x,p)) + \nabla_x \Psi (s, \X_{\pm} (s;t,x,p)) \cdot \V_{\pm} (s;t,x,p),
\ee
we conclude \eqref{eq:energy_conservation_dy_2}.
\end{proof}

Finally, we define \textit{the flux} as follows:
\Be \label{def:flux}
b (t,x) := \int_{\R^3} (v_+ e_+ f_+ + v_{-} e_{-} f_{-} ) \dd p,
\Ee
where $v_{\pm} = \frac{p_{\pm}}{\sqrt{m^2_{\pm} + |p_{\pm} |^2 / c^2}}$.
The equation \eqref{eqtn:f} yields a continuity equation 
\Be \label{cont_eqtn}
\p_t \varrho  + \nabla_x \cdot b  = 0 \ \ \text{in} \ \R_+ \times \O.
\Ee 
Note that \eqref{cont_eqtn} holds in the sense of distributions, and we can obtain 
\Be \label{identity:Psi_t}
\p_t \Psi (t,x)  = (-\Delta_0)^{-1} \p_t \varrho (t,x)  = - (-\Delta_0)^{-1} (\nabla_x \cdot b ) (t,x),
\Ee
where the first equality comes from \eqref{Poisson_f} and the second equality from \eqref{cont_eqtn}.

\smallskip

From Lemma \ref{lem:conservation_law}, we list some basic properties on weight functions $w_{\pm}$ and $\w_{\pm}$ as follows.

\begin{remark} \label{rmk:w_dy}	

(a) The steady weight $w_{\pm} (x,p)$ in \eqref{w^h} is invariant along the steady characteristics \eqref{ODE_h}, such that for all $s \in [-\tbp (x,p), \tfp (x,p)]$,
\Be \label{w:invar}
w_{\pm} ( X_{\pm} (s;x,p), V_{\pm} (s;x,p)) = w_{\pm} (x,p).
\Ee 

(b) The dynamic weight $\w_{\pm} (t,x,p)$ is not invariant along the dynamic characteristics. In particular, the dynamic total energy is not invariant along the dynamic characteristics, that is, for all $s - t \in [-\tB (t,x,p), \tF (t,x,p)]$,
\Be \label{dDTE}
\begin{split}
& \frac{d}{ds} \Big( \sqrt{(m_{\pm} c)^2 + |\P_{\pm} (s;t,x,p)|^2} + \frac{1}{c} \big( e_{\pm} \big( \Phi_h (\X_{\pm} (s;t,x,p)) 
\\& \qquad + \Psi (s, \X_{\pm} (s;t,x,p)) \big) + m_{\pm} g \X_{\pm, 3} (s;t,x,p) \big) \Big) 
\\& = \frac{e_{\pm}}{c} \p_t \Psi (s,\X_{\pm} (s;t,x,p)) 
= - \frac{e_{\pm}}{c} (-\Delta_0)^{-1}  (\nabla_x \cdot b ) (s, \X_{\pm} (s;t,x,p)).
\end{split}
\Ee

(c) At the boundary, the Dirichlet boundary conditions \eqref{eqtn:Dphi} and \eqref{Poisson_f} imply that 
\Be\label{w:bdry}
w_{\pm, \beta} (x,p) \equiv \w_{\pm, \beta}  (t,x,p) \equiv e^{\beta \sqrt{(m_{\pm} c)^2 + |p|^2} } \ \ \text{at}  \ \ x_3=0.
\Ee

(d) At the initial time $t=0$, 
\Be \label{W_t=0}
\begin{split}
\w_{\pm} (0,x,p) 
= \w_{\pm, \beta, 0} (x,p)
& = e^{ \beta \left( \sqrt{(m_{\pm} c)^2 + |p|^2} + \frac{1}{c} \big( e_{\pm} ( \Phi_h (x) + \Psi(0,x) ) + m_{\pm} g x_3 \big) \right) }
\\& 
= w_{\pm, \beta}(x,p) e^{ \beta \frac{e_{\pm}}{c} (-\Delta_0)^{-1} \varrho (0, x)}.
\end{split}
\Ee
\end{remark}

\begin{prop} \label{lem:tb}

Recall the steady characteristics \eqref{ODE_h} and its self-consistent potential $\Phi_h (x)$ in \eqref{eqtn:Dphi}. Suppose the condition \eqref{Uest:DPhi} holds. Then the backward exit time \eqref{tb^h} is bounded by 
\Be \label{est:tb^h}
\tbp (x,p) \leq \min \left\{ \frac{4}{m_{\pm} g} \big( \sqrt{(m_{\pm} c)^2 + |p|^2} + \frac{3}{2c} m_{\pm} g x_{3} \big), \ \frac{4}{m_{\pm} g} |\pbp (x, p)| \right\}. 
\Ee
Moreover, for all $s \in [-\tbp (x,p), \tfp (x,p)]$,
\Be \label{est:x3}
\frac{2}{3 m_{\pm} g}
\leq \frac{ X_{\pm, 3} (s,x,p) }{ \sqrt{(m_{\pm} c)^2 + |\pbp (s;x,p)|^2} - m_{\pm} c }
\leq \frac{2}{m_{\pm} g}. 
\Ee
\end{prop}

\begin{proof} 

For simplicity, we may omit the subscript $\pm$ when the statement/formula holds for both cases. We name this as
\Be \label{abuse}
\text{the abuse of notation about $\pm$.}
\Ee	
	
\textbf{Step 1.}
From the assumption \eqref{Uest:DPhi}, we get
\be \label{Uest2:DPhi}
\| e_{\pm} \nabla_x \Phi_h \|_{L^\infty (\bar{\O})} \leq \min \left\{m_+, m_- \right\} \times \frac{g}{2}.
\ee 
From $V \times B = (B_3 V_2, - B_3 V_1, 0)$ and \eqref{Uest2:DPhi}, the vertical 3-force is negative and  bounded from above as
\Be \label{upper_dotV_3}
\frac{d P_{\pm, 3} (s;x,p)}{ds} 
= - e_{\pm} \nabla_{x_3} \Phi_h (X_{\pm} (s;x,p) )) - {m_{\pm}} g 
\leq - \frac{m_{\pm} g }{2}.
\Ee
Using the zero Dirichlet boundary condition of $\Phi_h$ in \eqref{eqtn:Dphi} and \eqref{Uest2:DPhi}, we obtain for $x \in \O$,
\be \label{est:phih}
- \frac{m_{\pm} g x_3}{2} \leq e_{\pm} \Phi_h (x) \leq \frac{m_{\pm} g x_3}{2}.
\ee
From Lemma \ref{lem:conservation_law}, then for $(x, p) \in \O \times \R^3$,
\Be \notag
\begin{split}
& \ \ \ \ \sqrt{(m_{\pm} c)^2 + |P_{\pm} (s;x,p)|^2} + 
\frac{1}{c} \big( e_{\pm} \Phi_h (X_{\pm} (s;x,p)) + m_{\pm} g X_{\pm, 3} (s;x,p) \big) 
\\& = \sqrt{(m_{\pm} c)^2 + |\pbp (x,p)|^2} = \sqrt{(m_{\pm} c)^2 + |\pfp (x,p)|^2}. 
\end{split}
\Ee 
This implies that for $(x, p) \in \O \times \R^3$,
\be \label{est:vb=vf}
\begin{split}
& |\pbp (x, p)| = |\pfp (x, p)| 
\\& \leq \sqrt{(m_{\pm} c)^2 + |P_{\pm} (s;x,p)|^2} + 
\frac{1}{c} \big( e_{\pm} \Phi_h (X_{\pm} (s;x,p)) + m_{\pm} g X_{\pm, 3} (s;x,p) \big).
\end{split}
\ee

\smallskip

\textbf{Step 2.}
Now we claim that 
\be \label{est:vb+vf}
\tb (x,p) + \tf (x, p) \leq \frac{2}{m g} (|p_{\mathbf{b},3} (x, p)| + |p_{\mathbf{f},3} (x, p)|).
\ee

Recall from \eqref{upper_dotV_3} that $\frac{d}{ds} P_{3} (s;x,p) < 0$. This shows that $P_{3} (s;x , p)$ is strictly decreasing along characteristics \eqref{ODE_h}.
Moreover, from \eqref{sign:v3} the vertical particle velocity satisfies
\be \label{sign:v3p3}
\sgn \big( V_{3} (s;x,p) \big) = \sgn \big( P_{3} (s;x,p) \big).
\ee
Hence we can always choose the unique time $s^*$ to achieve a vertical peak such as  
\Be \label{def:s^*}
X_{3} (s^*;x,p) = \max_{s \in [-\tb (x, p), \tf (x,p)]} { X_{3} (s;x,p) }.
\Ee
Note that at $s=s^*$, the vertical velocities vanish so that 
\Be \label{def2:s^*}
0 = \frac{d}{ds} X_{3} (s, x, p)|_{s= s^*} = V_{3} (s^*, x, p) =  c \frac{P_{3} (s^*;x,p)}{P^0 (s^*;x,p) }.
\Ee
From \eqref{sign:v3p3}, it implies that $s^*$ can be identified alternatively and equivalently as
\Be \label{def1:s^*}
\begin{split}
s^* \in    [-\tb (x, p), \tf (x,p)  ] \ \ \text{if and only if} \ \  P_{3} (s^*;x,p) =0;&   \\
\text{and such a $s^*$ is unique for each $x$ and $p$.}
\end{split}
\Ee  
Now we first bound the traveling time from $s^*$ to $\tf (x,p)$. From \eqref{upper_dotV_3},
\Be \notag
P_{3} ( \tf (x,p);x,p) - P_{3} (s^*;x,p)
= \int^{\tf (x,p)}_{s^*} \frac{dP_3(\tau;x,p) }{d \tau } \dd \tau 
\leq - \frac{m g}{2} ( \tf (x,p)  -s ^*).
\Ee
Using \eqref{def1:s^*}, we get $P_3 (s_*;x,p) = 0$. Then we derive that 
\Be\label{tb+tf:1}
| p_{\mathbf{f}, 3}| (x, p) = - P_3 (\tf (x,p);x,p) 
\geq \frac{m g}{2} ( \tf (x,p)  -s ^*).
\Ee

Next, following the same argument, we claim that 
\Be\label{tb+tf:2}
|p_{\mathbf{b}, 3} (x, p)| = P_3 (-\tb (x,p);x,p)
\geq \frac{m g}{2} ( s^* - \tb(x,p) ).
\Ee
This is again due to \eqref{upper_dotV_3} and therefore 
\Be \notag 
P_3 (s^*; x,p) - P_3 (-\tb(x,p);x,p) 
= \int_{-\tb(x,p)}^{s^*} \frac{d P _3 (\tau; x,p)}{d \tau} \dd \tau  
\leq - \frac{m g}{2} ( s^* - \tb(x,p)). 
\Ee
Finally, combining \eqref{tb+tf:1} and \eqref{tb+tf:2}, we conclude \eqref{est:vb+vf}.

\smallskip

\textbf{Step 3.}
Picking $s = 0$ in \eqref{est:vb=vf} and using \eqref{est:vb+vf}, we can get 
\Be \notag
\tb (x,p) + \tf (x, p)
\leq \frac{4}{m g} |\pb (x, p)| 
\leq \frac{4}{m g} \left( \sqrt{(m c)^2 + |p|^2} + 
\frac{1}{c} \big( e \Phi_h (x) + m g x_{3} \big) \right).
\Ee
Together with \eqref{est:phih}, we obtain
\Be \label{est1:tb}
\tb (x,p) + \tf (x, p) 
\leq \frac{4}{m g} \big( \sqrt{(m c)^2 + |p|^2} + 
\frac{3}{2c} m g x_{3} \big).
\Ee
Therefore, we conclude \eqref{est:tb^h}.

\smallskip

On the other hand, \eqref{def:s^*} shows
that given $(x, p) \in \O \times \R^3$, there exists a unique peak in $x_3$-axis along the characteristic line.
Suppose when $s = s^*$, 
\[
X_3 (s^*;x,p) = \max_{s \in [-\tb(x, p), \tf (x,p)]} { X_3 (s;x,p) }.
\]
Recall that $P_3(s, x, p)$ is monotone in $s$ due to the fact that $\frac{d}{ds} P_{3} (s;x , p)<0$. From \eqref{sign:v3p3} and \eqref{def2:s^*}, it is clear that 
\begin{equation} \label{est:v3}
P_3 (s;x,p) =
\begin{cases}
> 0, & \text{if } -\tb(x, p) \leq s < s^*, \\[2pt]
= 0, & \text{if } s = s^*, \\[2pt]
< 0, & \text{if } s^* < s \leq \tf (x,p).
\end{cases}
\end{equation}
From \eqref{est:phih}, we get
\be \label{est:phih_max}
\frac{mg}{2} X_3(s^*;x,p) \leq 
e \Phi_h (X(s^*;x,p)) + m g X_3(s^*;x,p) 
\leq \frac{3mg}{2} X_3(s^*;x,p).
\ee
Using Lemma \ref{lem:conservation_law} and \eqref{est:v3}, we derive
\Be 
\begin{split}
\sqrt{(m c)^2 + |\pb (s;x,p)|^2}
& = \sqrt{(m c)^2 + |P (s^*;x,p)|^2} + 
\frac{1}{c} \big( e \Phi_h (X (s^*;x,p)) + m g X_{3} (s^*;x,p) \big) 
\\& = m c + \frac{1}{c} \big( e \Phi_h (X (s^*;x,p)) + m g X_{3} (s^*;x,p) \big).
\end{split} 
\Ee 
This implies that 
\Be\label{bound:vb}
\sqrt{(m c)^2 + |\pb (s;x,p)|^2} - mc
= \frac{1}{c} \big( e \Phi_h (X (s^*;x,p)) + m g X_{3} (s^*;x,p) \big).
\Ee
Using \eqref{bound:vb} and \eqref{est:phih_max}, we complete the proof of \eqref{est:x3} via
\Be 
\frac{mg}{2} X_3(s^*;x,p) \leq
\sqrt{(m c)^2 + |\pb (s;x,p)|^2} - mc
\leq \frac{3mg}{2} X_3(s^*;x,p).
\Ee 
\end{proof} 

In the next Lemma, we show another estimate on $\tb (x,p)$ in which $\tb (x,p)$ can be controlled by $\pb (x, p)$ and $\pbn (x,p)$. This is crucial to the regularity estimates in Section \ref{sec:RS}.

\begin{prop} \label{lem2:tb} 

Recall the steady characteristics \eqref{ODE_h} and its self-consistent potential $\Phi_h (x)$ in \eqref{eqtn:Dphi}. Suppose the condition \eqref{Uest:DPhi} holds. Then the backward exit time \eqref{tb^h} is bounded by  
\Be \label{est:tbpb3}
\tbp (x,p) \leq \Big( \frac{2}{m_{\pm} g} + \frac{\sqrt{8} \sqrt[4]{(m_{\pm} c)^2 + |\pbp (x,p)|^2}}{\sqrt{ c (m_{\pm})^3 g^2}} \Big) | \pbpn (x, p) |. 
\Ee
Moreover, for all $s \in [-\tbp (x,p), \tfp (x,p)]$,
\Be \label{est:x3pb3}
X_{\pm, 3} (s,x,p) \leq \frac{2}{(m_{\pm})^2 g} (\pbpn (x, p))^2.
\Ee
\end{prop}

\begin{proof}

In the proof, we adopt the abuse of notation about $\pm$ in \eqref{abuse}.
From the assumption \eqref{Uest:DPhi}, we get
\Be \notag
\| e_{\pm} \nabla_x \Phi_h \|_{L^\infty (\bar{\O})} \leq \min \left\{m_+, m_- \right\} \times \frac{g}{2}.
\Ee
This implies that
\be  \label{lower_upper_dotp_3}
- \frac{3 m g }{2} 
\leq \frac{d P_{3} (s;x,p)}{ds} 
= - e \nabla_{x_3} \Phi_h (X (s;x,p) ) - m g 
\leq - \frac{m g }{2}.
\ee
Using Lemma \ref{lem:conservation_law}, together with \eqref{est:phih}, then for $(x, p) \in \O \times \R^3$,
\be \label{est:pb>p}
\begin{split}
& \sqrt{(m c)^2 + |\pb (x,p)|^2} 
\\& = \sqrt{(m c)^2 + |P (s;x,p)|^2} + 
\frac{1}{c} \big( e \Phi_h (X (s;x,p)) + m g X_{3} (s;x,p) \big) 
\\& \geq \sqrt{(m c)^2 + |P (s;x,p)|^2} + \frac{m g}{2 c} X_{3} (s;x,p) \geq \sqrt{(m c)^2 + |P (s;x,p)|^2}.
\end{split}
\ee
Follow the proof in Proposition \ref{lem:tb}, $P_{3} (s;x, p)$ is strictly decreasing along characteristics \eqref{ODE_h}, and the vertical particle velocity satisfies
\be \notag
\sgn \big( V_{3} (s;x,p) \big) = \sgn \big( P_{3} (s;x,p) \big).
\ee
Hence there exists a unique time $s^*$ to achieve a vertical peak, such that 
\Be \notag
X_{3} (s^*;x,p) = \max_{s \in [-\tb (x, p), \tf (x,p)]} { X_{3} (s;x,p) }.
\Ee
Moreover, at $s=s^*$, the vertical velocities satisfies that 
\Be \label{eq:V3s^*}
0 = \frac{d}{ds} X_{3} (s; x, p)|_{s= s^*} = V_{3} (s^*; x, p) =  c \frac{P_{3} (s^*;x,p)}{P^0 (s^*;x,p) },
\Ee
where $P^0 = \sqrt{(m_\pm c)^2 + |P|^2}$.

Note that 
\be \label{eq:sign_V3s^*}
\sgn \big( V_{3} (s;x,p) \big) = \sgn \big( P_{3} (s;x,p) :=
\begin{cases}
& > 0, \ \text{ if } -\tb (x,p) \leq s < s^*, \\[5pt]
& < 0, \ \text{ if } s^* < s \leq  \tf (x,p).
\end{cases}
\ee

For $p_3 \geq 0$, using \eqref{eq:sign_V3s^*} and $P_{3} (0;x,p) = p_3 \geq 0$, we derive that
\be \label{est:s^*}
-\tb (x,p) \leq 0 \leq s^*, \text{ for } p_3 \geq 0.
\ee
Using \eqref{lower_upper_dotp_3} and \eqref{est:s^*}, we get
\be \label{est:pb3tb}
\begin{split}
P_{3} (s^*;x,p) - P_{3} (- \tb (x, p);x,p)
& = \int^{s^*}_{- \tb (x, p)} \big( - e \nabla_{x_3} \Phi_h (X (s;x,p) ) - m g \big) \dd s
\\& \leq - \frac{m g }{2} (s^* + \tb (x, p)) \leq - \frac{m g }{2} \tb (x, p).
\end{split}
\ee
Together with \eqref{est:pb3tb} and \eqref{eq:V3s^*}, we have
\be \notag
- \pbn (x, p) = P_{3} (s^*;x,p) - P_{3} (- \tb (x, p);x,p)
\leq - \frac{m g }{2} \tb (x, p).
\ee
This shows that
\be \label{est2:pb3tb}
\frac{2}{m g} \pbn (x, p) \geq \tb (x, p), \text{ for } p_3 \geq 0.
\ee

For $p_3 < 0$, using \eqref{eq:sign_V3s^*} and $P_{3} (0;x,p) = p_3 < 0$, we derive that
\be \label{est2:s^*}
-  \tb (x,p)  \leq s^* < 0 \leq  \tf (x,p), \text{ for } p_3< 0.
\ee

\smallskip

\textbf{Step 1. } 
First, we claim that \eqref{est:x3pb3} holds.
Similar from \eqref{est:pb3tb} and \eqref{est2:pb3tb}, we get
\be \notag
- \pbn (x, p) = P_{3} (s^*;x,p) - P_{3} (- \tb (x, p);x,p)
\leq - \frac{m g }{2} (s^* + \tb (x, p)),
\ee
and hence
\be \label{est3:pb3tb}
\frac{2}{m g} \pbn (x, p) \geq s^* + \tb (x, p).
\ee
Furthermore, we have
\Be \label{eq:x3s^*}
\begin{split} 
X_3 (s^*; x,p) 
&= X_3 (s^*; x,p) - X_3 (-\tb(x,p);x,p) 
\\& = \int_{-\tb(x,p)}^{s^*} V_{3} (s; x, p) \dd s
= \int_{-\tb(x,p)}^{s^*} c \frac{P_{3} (s;x,p)}{P^0 (s;x,p) } \dd s.
\end{split}
\Ee
Since $P_{3} (- \tb (x,p); x,p) = \pbn (x,p)$, from \eqref{lower_upper_dotp_3} we derive, for $- \tb (x, p) \leq s \leq s^*$,
\Be \notag
P_{3} ( s;x,p) - P_{3} (- \tb (x,p);x,p)
= \int^{s}_{- \tb (x,p)} \frac{dP_3(\tau;x,p) }{d \tau } \dd \tau 
\leq - \frac{m g}{2} ( s + \tb (x,p)).
\Ee
This, together with \eqref{eq:sign_V3s^*}, shows that 
\be \label{est:p3pb3}
0 \leq P_{3} ( s;x,p)
\leq \pbn (x,p) - \frac{m g}{2} ( s + \tb (x,p)),
\text{ for } - \tb (x, p) \leq s \leq s^*.
\ee
Inputting \eqref{est:p3pb3} into \eqref{eq:x3s^*}, we obtain
\Be \label{est:x3s^*}
\begin{split} 
X_3 (s^*; x,p) 
& = \int_{-\tb(x,p)}^{s^*} c \frac{P_{3} (s;x,p)}{P^0 (s;x,p) } \dd s
\leq c \int_{-\tb(x,p)}^{s^*} \frac{\pbn (x,p) - \frac{m g}{2} ( s + \tb (x,p))}{P^0 (s;x,p) } \dd s.
\end{split}
\Ee
Since $P^0 (s;x,p) = \sqrt{(mc)^2 + |P (s;x,p)|^2} \geq m c$, we bound \eqref{est:x3s^*} by
\Be \label{est2:x3s^*}
\begin{split} 
X_3 (s^*; x,p) 
& \leq \int_{-\tb(x,p)}^{s^*} \frac{\pbn (x,p) - \frac{m g}{2} ( s + \tb (x,p))}{m } \dd s
\\& = (s^* + \tb (x, p)) \frac{\pbn (x,p)}{m} -
\frac{g}{2} \int_{-\tb(x,p)}^{s^*} ( s + \tb (x,p) ) \dd s 
\\& =  (s^* + \tb (x, p)) \frac{\pbn (x,p)}{m} -
\frac{g}{4} ( s^* + \tb (x,p) )^2
\leq \frac{2}{m^2 g} (\pbn (x, p))^2.
\end{split}
\Ee
where the last inequality follows from \eqref{est3:pb3tb} and $( s^* + \tb (x,p) )^2 \geq 0$. Thus, we prove \eqref{est:x3pb3}.

\smallskip

\textbf{Step 2. }  
Second, we claim that \eqref{est:tbpb3} holds. The proof is already complete in the case of $p_3\geq 0$ due to \eqref{est2:pb3tb}. We only consider the other case $p_3<0$. Then, from \eqref{est2:s^*}, we note that  $s^*   \leq 0$. 
Hence we have that
\Be \label{eq:x3x3s^*}
X_3 (0; x,p) - X_3 (s^*;x,p) 
= \int^{0}_{s^*} V_{3} (s; x, p) \dd s
= \int^{0}_{s^*} c \frac{P_{3} (s;x,p)}{P^0 (s;x,p) } \dd s.
\Ee
Since $X_3 (0; x,p) = x_3 \geq 0$, \eqref{eq:x3x3s^*} implies
\Be \label{est:x3x3s^*}
- \int^{0}_{s^*} c \frac{P_{3} (s;x,p)}{P^0 (s;x,p) } \dd s \leq X_3 (0; x,p) - \int^{0}_{s^*} c \frac{P_{3} (s;x,p)}{P^0 (s;x,p) } \dd s = X_3 (s^*;x,p).
\Ee 
Recall that $P_{3} (s^*; x,p) = 0$, from \eqref{lower_upper_dotp_3} we derive, for $s^* \leq s \leq 0$,
\Be \label{est:p3p3s^*}
P_{3} ( s;x,p) = P_{3} ( s;x,p) - P_{3} (s^*; x,p)
= \int^{s}_{s^*} \frac{dP_3(\tau;x,p) }{d \tau } \dd \tau 
\leq - \frac{m g}{2} ( s - s^*).
\Ee
Using \eqref{est:p3p3s^*} and \eqref{est:pb>p}, we have
\be \label{est2:x3x3s^*}
\begin{split}
- \int^{0}_{s^*} c \frac{P_{3} (s;x,p)}{P^0 (s;x,p) } \dd s 
& \geq \int^{0}_{s^*} c \frac{ \frac{m g}{2} ( s - s^*) }{\sqrt{(m c)^2 + |\pb (x,p)|^2}} \dd s
\\& = \frac{c m g }{4 \sqrt{(m c)^2 + |\pb (x,p)|^2}} (s^*)^2.
\end{split}
\ee
Inputting \eqref{est2:x3x3s^*} into \eqref{est:x3x3s^*}, together with \eqref{est2:x3s^*}, we derive 
\Be \label{est3:s^*}
\frac{c m g }{4 \sqrt{(m c)^2 + |\pb (s;x,p)|^2}} (s^*)^2
\leq X_3 (s^*; x,p) \leq \frac{2 }{m^2 g} (\pbn (x, p))^2.
\Ee
Since $s^* < 0$, \eqref{est3:s^*} shows that
\be \label{est4:s^*}
- s^*
\leq \sqrt{ \frac{8 \sqrt{(m c)^2 + |\pb (s;x,p)|^2}}{c m^3 g^2}} | \pbn (x, p) |.
\ee
Finally, combining \eqref{est3:pb3tb} and \eqref{est4:s^*}, we conclude that
\be \notag
\tb (x, p) = \tb (x, p) + s^* - s^* \leq 
\Big( \frac{2}{m g}  + \frac{\sqrt{8} \sqrt[4]{(m c)^2 + |\pb (s;x,p)|^2}}{\sqrt{ c m^3 g^2}} \Big) | \pbn (x, p) |,
\ee
and prove \eqref{est:tbpb3}.
\end{proof} 

\begin{prop} \label{lem:tB}

Recall the dynamic characteristics \eqref{ODE_F} and its self-consistent potentials $\Phi(x)$ and $\Psi(t,x)$ in \eqref{eqtn:Dphi} and \eqref{Poisson_f}, respectively. Suppose the conditions \eqref{Uest:DPhi}, \eqref{Uest:DxPsi} and \eqref{Uest:Dxphi_F} hold.
Then for all $(t,x,p) \in [0, \infty) \times  \bar\O \times \R^3$ satisfying $t - \tBp (t,x,p) \geq 0$, the backward exit time \eqref{tb} is bounded by
\Be \label{est:tB}
\begin{split}
\tBp (t,x,p) 
& \leq \frac{1}{\min \{ \frac{g}{4 \sqrt{2}}, \frac{c}{\sqrt{10}} \}} \times \max \left\{ x_3, \ \frac{2 c}{m_{\pm} g} \big( \sqrt{(m_{\pm} c)^2 + |p|^2} - m_{\pm} c + \frac{3 m_{\pm} g}{2 c} x_{3} - p_3 \big) \right\} 
\\& \qquad + \frac{4}{m_{\pm} g} |p| + 1.
\end{split}
\Ee
Moreover, the forward exit time \eqref{tb} is bounded by
\be \label{est:tF}
\begin{split}
\tFp (t,x,p) 
& \leq \frac{1}{\min \{ \frac{g}{4 \sqrt{2}}, \frac{c}{\sqrt{10}} \}} \times \max \left\{ x_3, \ \frac{2 c}{m_{\pm} g} \big( \sqrt{(m_{\pm} c)^2 + |p|^2} - m_{\pm} c + \frac{3 m_{\pm} g}{2 c} x_{3} + p_3 \big) \right\} 
\\& \qquad + \frac{4}{m_{\pm} g} |p| + 1.
\end{split}
\ee
\end{prop}
 
\begin{proof}

In the proof, we adopt the abuse of notation about $\pm$ in \eqref{abuse}.
We always assume that $t - \tB(t,x,p) \geq 0$ throughout this proof. 

\smallskip

First, we show \eqref{est:tB}. 
If $\tB (t,x,p) < \frac{4}{m g} |p| + 1$, then \eqref{est:tB} directly follows. Without loss of generality, from now on we assume that 
\Be \label{assume:tB}
\tB (t,x,p) \geq \frac{4}{m g} |p| + 1.
\Ee
Further, we always assume that $s \in [t- \tB (t,x,p), t]$. In particular, this implies that 
\Be \label{cond:t-s}
t-s \leq \tB(t,x,p).
\Ee
	
\smallskip	

\textbf{Step 1.}
From the assumption \eqref{Uest:Dxphi_F}, we get
\be \label{Bootstrap2}
e_{\pm} \sup_{0 \leq t < \infty}\| \nabla_x \big( \Phi + \Psi (t) \big) \|_{L^\infty(\O)}  
\leq \min \left\{m_+, m_- \right\} \times \frac{g}{2}.
\ee
From $\V \times B = (B_3 \V_2, - B_3 \V_1, 0)$ and \eqref{Bootstrap2}, the vertical 3-force is bounded from above as
\Be \label{upper_V_3_dy}
\frac{d \P_{3} (s;t, x, p)}{ds} 
\leq - \frac{m g}{2}, \ \text{for all} \ 0 \leq s \leq t < \infty.
\Ee
Since $p_3 = \P_{3} (t;t,x,p)$, this ensures that
\be \label{delta_V_3_dy}
\P_{3} (s;t, x,p) - p_3  
\geq \frac{m g}{2} |s-t|, 
\ \text{for $s - t \in [- \tB (t,x,p), 0]$ with $0 \leq s \leq t < \infty$.}
\ee
Note that \eqref{upper_V_3_dy} also implies $\P_{3} (s;x , p)$ is strictly decreasing along characteristics \eqref{ODE_F}.
Moreover, from \eqref{sign:v3_dy} the vertical particle velocity satisfies
\be \label{upper_Velocity_3_dy}
\sgn \big( \V_{3} (s;t,x,p) \big) = \sgn \big( \P_{3} (s;t,x,p) \big).
\ee

\smallskip

\textbf{Step 2.}
Recall the assumption that $t - \tB(t,x,p) \geq 0$.
From \eqref{ODE_F}, \eqref{upper_V_3_dy} and \eqref{upper_Velocity_3_dy}, we obtain that for any $(t, x, p) \in [0, \infty) \times \bar\O \times \R^3$, there exists a unique vertical peak $\X_3 (s;t,x,p)$ along the characteristics for $s - t \in [- \tB (t,x,p), 0]$ and $t < \infty$, that is, 
\be \label{def:s*_tB}
\X_3 (s^*;t,x,p) = \max_{s - t \in [- \tB (t,x,p), 0]} { \X_3 (s;t,x,p) }.
\ee
Using $\P_3(s;t, x, p)$ is monotone in $s$, together with \eqref{ODE_F} and \eqref{upper_Velocity_3_dy}, we derive that for $s - t \in [- \tB (t,x, p), 0]$, 
\be \label{V3s*}
s^* =
\begin{cases}
t, & \text{if } p_3 \geq 0, \\[2pt]
t - \tB (t,x, p) < s^* < t, & \text{if } p_3 < 0,
\end{cases}
\ \text{ and } \
P_3 (s^*;t,x,p) =
\begin{cases}
 p_3, & \text{if } p_3 \geq 0, \\[2pt]
0, & \text{if } p_3 < 0.
\end{cases}
\ee

Here we claim that at $s = s^*$ in \eqref{def:s*_tB},
\be \label{est:x3_dy}
\X _{3} (s^*;t,x,p)
\leq \max \left\{ x_3, \ \frac{2 c}{m g} \big( \sqrt{(m c)^2 + |p|^2} - mc + \frac{3 m g}{2 c} x_{3} - p_3 \big) \right\}.
\ee
To prove this claim, we consider two cases: $p_3 < 0$ and $p_3 \geq 0$.

\smallskip

\textbf{\underline{Case 1:} $p_3 < 0$.} 
By integrating  \eqref{upper_V_3_dy} over $(s^*, t)$, we obtain, for $p_3 = \P_3 (t;t,x,p) < 0$,
\be \notag
p_3 - \P_3 (s^*;t,x,p) = \int^{t}_{s^*} \frac{d \P_{3} (s;t,x,p)}{ds} \dd s \leq - \frac{m g}{2} (t - s^*).
\ee
From \eqref{V3s*} and $p_3 = \P_3 (t;t,x,p) < 0$, we have
\be \label{est:tb1}
t - s^* \leq - \frac{2}{m g} p_3.
\ee 
Using \eqref{Uest:DxPsi}, \eqref{eq:energy_conservation_dy} and \eqref{est:tb1}, then for $p_3 = \P_3 (t;t,x,p) < 0$ and $s^* \leq  s \leq t$,
\Be \label{est1:energy_conservation_dy}
\begin{split}
& \ \ \ \ \sqrt{(m c)^2 + |\P (s;t,x,p)|^2} + \frac{1}{c} \big( e \Phi_h (\X (s;t,x,p)) + m g \X_{3} (s;t,x,p) \big) 
\\& \leq \sqrt{(m c)^2 + |p|^2} + \frac{1}{c} \big( e | \Phi_h (x) | + m g x_3 \big) 
\\& \qquad \qquad + \sup_{s^* \leq  s \leq t} \Big| e \nabla_x \Psi (s, \X (s;t,x,p)) \cdot \frac{\V (s;t,x,p)}{c} \Big| \times (t - s^*) 
\\& \leq \sqrt{(m c)^2 + |p|^2} + \frac{1}{c} \big( e | \Phi_h (x) | + m g x_3 \big) + \frac{m g}{2} (t - s^*)
\\& \leq \sqrt{(m c)^2 + |p|^2} + \frac{1}{c} \big( e | \Phi_h (x) | + m g x_3 \big) - p_3.
\end{split}
\Ee 
On the other hand, using \eqref{Uest:DPhi} and $\Phi_h = \Psi = 0$ on $\p\O$, we derive that
\be \label{est2:energy_conservation_dy}
\begin{split}
& \sqrt{(m c)^2 + |\P (s^*;t,x,p)|^2} + \frac{1}{c} \big( e \Phi_h (\X (s^*;t,x,p)) + m g \X_{3} (s^*;t,x,p) \big) 
\\& \ \ \ \ \geq mc + \frac{m g}{2 c} \X _{3} (s^*;t,x,p), 
\\& \sqrt{(m c)^2 + |p|^2} + \frac{1}{c} \big( e | \Phi_h (x) | + m g x_3 \big) - p_3 
\leq \sqrt{(m c)^2 + |p|^2} + \frac{3 m g}{2 c} x_{3} - p_3. 
\end{split}
\ee
Combining \eqref{est1:energy_conservation_dy} with \eqref{est2:energy_conservation_dy}, we derive that
\be \notag
mc + \frac{m g}{2 c} \X _{3} (s^*;t,x,p)
\leq \sqrt{(m c)^2 + |p|^2} + \frac{3 m g}{2 c} x_{3} - p_3.
\ee
This implies that
\be \label{est1:peak_dy}
\X _{3} (s^*;t,x,p)
\leq \frac{2 c}{m g} \big( \sqrt{(m c)^2 + |p|^2} - mc + \frac{3 m g}{2 c} x_{3} - p_3 \big).
\ee

\smallskip

\textbf{\underline{Case 2:} $p_3 \geq 0$.} 
From \eqref{V3s*}, for $p_3 = \P_3 (t;t,x,p) \geq 0$, $s^* = t$. Thus, we obtain that
\be \label{est2:peak_dy}
\X_{3} (s^*;t,x,p) = x_3.
\ee
Finally we conclude \eqref{est:x3_dy} by combining \eqref{est1:peak_dy} and \eqref{est2:peak_dy}.

\smallskip

\textbf{Step 3.}
From $\V \times B = (B_3 \V_2, - B_3 \V_1, 0)$ and \eqref{ODE_F}, we derive that
\be \notag
\begin{split}
& \ \ \ \ \frac{d}{ds} \big( |\P_{1} (s;t, x,p)|^2 + |\P_{2} (s;t, x,p)|^2 \big)
\\& = 2 (\P_{1}, \P_{2}) \cdot \frac{d}{ds} \big( \P_{1} (s;t, x,p), \P_{2} (s;t, x,p) \big)
\\& = 2 (\P_{1}, \P_{2}) \cdot \big( \frac{e}{c} (B_3 \V_2, - B_3 \V_1) - e \nabla_{x_\|} \Psi  (s, \X) - e \nabla_{x_\|} \Phi_h (\X) \big)
\\& = - 2 (\P_{1}, \P_{2}) \cdot \big( e \nabla_{x_1} (\Phi(\X) + \Psi(s,\X)), e \nabla_{x_2} (\Phi(\X) + \Psi(s,\X)) \big).
\end{split}
\ee
From \eqref{Uest:Dxphi_F} and Cauchy–Schwarz inequality,
\be \notag
\Big| \frac{d}{ds} \big( |\P_{1} (s;t, x,p)|^2 + |\P_{2} (s;t, x,p)|^2 \big) \Big|
\leq \sqrt{|\P_{1} (s;t, x,p)|^2 + |\P_{2} (s;t, x,p)|^2} \times \frac{m g}{\sqrt{2}},
\ee
and
\be \notag
\Big| \frac{d}{ds} \sqrt{|\P_{1} (s;t, x,p)|^2 + |\P_{2} (s;t, x,p)|^2} \Big|
\leq \frac{m g}{2 \sqrt{2}}.
\ee
By Gronwall's inequality and $p = \P (t;t,x,p)$, we derive that for $s - t \in [-\tB (t,x,p), 0]$,
\be \label{back_delta_V_12_dy}
\sqrt{|\P_{1} (s;t, x,p)|^2 + |\P_{2} (s;t, x,p)|^2}
\leq \sqrt{p^2_1 + p^2_2} + \frac{m g}{2\sqrt{2}} (t-s).
\ee

Now we choose 
\Be \label{bar s}
\bar s :  = t - \frac{4}{m g} |p| - 1.
\Ee
Note that $t- \bar s = \frac{4}{m g} |p| + 1 \leq \tB(t,x,p) $ due to \eqref{assume:tB}, so that it satisfies the condition of \eqref{cond:t-s} for all $s \in [t- \tB (t,x,p), \bar s]$.  
From \eqref{delta_V_3_dy}, we derive for all $s \in [t- \tB (t,x,p), \bar s]$,
\be \notag
\P_{3} ( s;t,x,p) + | p_3 | 
\geq \P_{3} ( s;t,x,p) - p_3 
\geq \frac{m g}{2} | s-t |.
\ee
Therefore, we further derive that 
\be \label{est1:lower_V_3_dy}
\P_{3} ( s;t,x,p) \geq \frac{m g}{2} (t - s) - |p_3| 
\geq \frac{m g}{2} (t - s) - |p|, 
\ \text{for all } s \in [t- \tB (t,x,p), \bar s].
\ee
Using $t - s \geq t - \bar{s} = \frac{4}{mg} |p| + 1$, together with \eqref{est1:lower_V_3_dy}, we obtain 
\be \label{est:lower_V_3_dy}
\P_{3} ( s;t,x,p) \geq \frac{m g}{4} (t- s) + \frac{m g}{4}, 
\ \text{for all } s \in [t- \tB (t,x,p), \bar s].
\ee
Using \eqref{back_delta_V_12_dy} and our choice \eqref{bar s}, we have that 
\be \label{est:upper_V_12_dy}
\begin{split}
& \ \ \ \ \sqrt{ |\P_{1} ( s;t, x,p)|^2 + |\P_{2} (s;t, x,p)|^2}
\\& \leq |p| + \frac{m g}{2\sqrt{2}} (t-s) < \frac{m g}{2} (t- s) - \frac{m g}{4}
\\& \leq \frac{m g}{2} (t- s), 
\ \text{for all } s \in [t- \tB (t,x,p), \bar s].
\end{split}
\ee
Using \eqref{est:lower_V_3_dy} and \eqref{est:upper_V_12_dy}, we deduce that for all $s \in [t- \tB (t,x,p), \bar s]$,
\be \label{est:lower_P3_dy}
\begin{split}
\P_{3} ( s;t, x,p) & \geq \frac{m g}{4}, 
\\ 4 | \P_{3} (s;t, x,p) |^2  & \geq \frac{(m g)^2}{4} (t-s)^2 \geq |\P_{1} (s;t, x,p) |^2 + |\P_{2} (s;t, x,p)|^2.
\end{split}
\ee
Hence, we obtain that for all $s \in [t- \tB (t,x,p), \bar s]$,
\be \notag
\begin{split}
\V_3 (s;t, x,p) = \frac{ \P_{3} (s;t, x,p) }{\sqrt{m^2 + |\P (s;t, x,p) |^2 / c^2 }} > \frac{\P_3 (s;t, x,p)}{\sqrt{m^2 + 5 |\P_3 (s;t, x,p) |^2 / c^2 }}.
\end{split}
\ee
Using $m^2 + 5 |\P_3 (s;t, x,p) |^2 / c^2 \leq \max \{ 2 m^2, 10 |\P_3 (s;t, x,p) |^2 / c^2 \}$ and $\P_{3} ( s;t, x,p) \geq \frac{m g}{4}$, we derive that
\be \label{est:lower_HV_3_dy}
\begin{split}
\V_3 (s;t, x,p) 
& > \min \{ \frac{\P_3 (s;t, x,p)}{\sqrt{2m^2}}, \frac{\P_3 (s;t, x,p)}{\sqrt{10 |\P_3 (s;t, x,p) |^2 / c^2}} \}
\\& \geq c_1 := \min \{ \frac{g}{4 \sqrt{2}}, \frac{c}{\sqrt{10}} \}, 
\ \text{for all } s \in [t- \tB (t,x,p), \bar s].
\end{split}  
\ee
Using \eqref{est:lower_HV_3_dy} together with \eqref{ODE_F}, we derive that 
\be \notag
\X_{3} (s^*;t,x,p) \geq \X_{3} (\bar{s};t,x,p)
= \int_{t- \tB (t,x,p)}^{\bar s } \V_3 (\tau;t,x,p) \dd \tau
> c_1 (\bar{s} - (t - \tB (t,x,p))).
\ee
On the other hand, \eqref{est:x3_dy} further implies that 
\Be
c_1 (\bar{s} - (t - \tB (t,x,p))) \leq \max \left\{ x_3, \ \frac{2 c}{m g} \big( \sqrt{(m c)^2 + |p|^2} - mc + \frac{3 m g}{2 c} x_{3} - p_3 \big) \right\}.
\Ee
From the choice of $\bar{s}$ in \eqref{bar s_tF}, finally we derive 
\be \notag
\begin{split}
\tB (t,x,p) 
\leq \frac{1}{c_1} \times \max \left\{ x_3, \ \frac{2 c}{m g} \big( \sqrt{(m c)^2 + |p|^2} - mc + \frac{3 m g}{2 c} x_{3} - p_3 \big) \right\} + \frac{4}{m g} |p| + 1,
\end{split}
\ee
and thus conclude \eqref{est:tB}.

\smallskip 

Second, we prove \eqref{est:tF}.
If $\tF (t,x,p) < \frac{4}{m g} |p| + 1$, then \eqref{est:tF} directly follows. Without loss of generality, from now on we assume that 
\Be \label{assume:tF}
\tF (t,x,p) \geq \frac{4}{m g} |p| + 1.
\Ee
Following the step 1 in the proof of \eqref{est:tB}, we have that for any $(t, x, p) \in [0, \infty) \times \bar\O \times \R^3$, there exists a unique vertical peak $\X_3 (s;t,x,p)$ along the characteristics for $s - t \in [0, \tF (t,x,p)]$ and $t < \infty$, that is, 
\be \label{def:s*_tF}
\X_3 (s^{*};t,x,p) = \max_{s - t \in [0, \tF (t,x,p)]} { \X_3 (s;t,x,p) }.
\ee
Using $\P_3(s;t, x, p)$ is monotone in $s$, together with \eqref{ODE_F} and \eqref{upper_Velocity_3_dy}, we derive that for $s - t \in [0, \tF (t,x,p)]$, 
\be \label{V3s*_tF}
s^* =
\begin{cases}
t \leq s^* < t + \tF (t,x,p), & \text{if } p_3 \geq 0, \\[2pt]
t, & \text{if } p_3 < 0,
\end{cases}
\ \text{ and } \
P_3 (s^*;t,x,p) =
\begin{cases}
0, & \text{if } p_3 \geq 0, \\[2pt]
p_3, & \text{if } p_3 < 0.
\end{cases}
\ee

\smallskip

\textbf{Step 1.}
Similarly, we first claim that at $s = s^*$ in \eqref{def:s*_tF},
\be \label{est:x3_dy_tF}
\X _{3} (s^*;t,x,p)
\leq \max \left\{ x_3, \ \frac{2 c}{m g} \big( \sqrt{(m c)^2 + |p|^2} - mc + \frac{3 m g}{2 c} x_{3} + p_3 \big) \right\}.
\ee
To prove this claim, we consider two cases: $p_3 < 0$ and $p_3 \geq 0$.

\smallskip

\textbf{\underline{Case 1:} $p_3 < 0$.}
From \eqref{V3s*_tF}, for $p_3 = \P_3 (t;t,x,p) \geq 0$, $s^* = t$. Thus, we obtain that
\be \label{est2:peak_dy_tF}
\X_{3} (s^*;t,x,p) = x_3.
\ee

\smallskip

\textbf{\underline{Case 2:} $p_3 \geq 0$.} 
By integrating  \eqref{upper_V_3_dy} over $(t, s^*)$, we obtain, for $p_3 = \P_3 (t;t,x,p) \geq 0$,
\be \label{delta_V_3_dy_tF}
\P_3 (s^*;t,x,p) - p_3 = \int^{s^*}_{t} \frac{d \P_{3} (s;t,x,p)}{ds} \dd s \leq - \frac{m g}{2} (s^* - t).
\ee
From \eqref{V3s*_tF} and $p_3 = \P_3 (t;t,x,p) \geq 0$, we have
\be \label{est:tb1_tF}
s^* - t \leq \frac{2}{m g} p_3.
\ee
Using \eqref{Uest:DxPsi}, \eqref{eq:energy_conservation_dy} and \eqref{est:tb1_tF}, we have, for $p_3 = \P_3 (t;t,x,p) \geq 0$ and $t \leq s \leq s^*$,
\Be \label{est1:energy_conservation_dy_tF}
\begin{split}
& \ \ \ \ \sqrt{(m c)^2 + |\P (s;t,x,p)|^2} + \frac{1}{c} \big( e \Phi_h (\X (s;t,x,p)) + m g \X_{3} (s;t,x,p) \big) 
\\& \leq \sqrt{(m c)^2 + |p|^2} + \frac{1}{c} \big( e | \Phi_h (x) | + m g x_3 \big) 
\\& \qquad \qquad + \sup_{s^* \leq  s \leq t} \Big| e \nabla_x \Psi (s, \X (s;t,x,p)) \cdot \frac{\V (s;t,x,p)}{c} \Big| \times (s^* - t) 
\\& \leq \sqrt{(m c)^2 + |p|^2} + \frac{1}{c} \big( e | \Phi_h (x) | + m g x_3 \big) + \frac{m g}{2} (s^* - t)
\\& \leq \sqrt{(m c)^2 + |p|^2} + \frac{1}{c} \big( e | \Phi_h (x) | + m g x_3 \big) + p_3.
\end{split}
\Ee 
On the other hand, using \eqref{Uest:DPhi} and $\Phi_h = \Psi = 0$ on $\p\O$, we derive that
\be \label{est2:energy_conservation_dy_tF}
\begin{split}
& \sqrt{(m c)^2 + |\P (s^*;t,x,p)|^2} + \frac{1}{c} \big( e \Phi_h (\X (s^*;t,x,p)) + m g \X_{3} (s^*;t,x,p) \big) 
\\& \ \ \ \ \geq mc + \frac{m g}{2 c} \X _{3} (s^*;t,x,p), 
\\& \sqrt{(m c)^2 + |p|^2} + \frac{1}{c} \big( e | \Phi_h (x) | + m g x_3 \big) + p_3 
\leq \sqrt{(m c)^2 + |p|^2} + \frac{3 m g}{2 c} x_{3} + p_3. 
\end{split}
\ee
Combining \eqref{est1:energy_conservation_dy_tF} with \eqref{est2:energy_conservation_dy_tF}, we derive that
\be \notag
mc + \frac{m g}{2 c} \X _{3} (s^*;t,x,p)
\leq \sqrt{(m c)^2 + |p|^2} + \frac{3 m g}{2 c} x_{3} + p_3.
\ee
This implies that
\be \label{est1:peak_dy_tF}
\X _{3} (s^*;t,x,p)
\leq \frac{2 c}{m g} \big( \sqrt{(m c)^2 + |p|^2} - mc + \frac{3 m g}{2 c} x_{3} + p_3 \big).
\ee
Finally we conclude \eqref{est:x3_dy_tF} by combining \eqref{est2:peak_dy_tF} and \eqref{est1:peak_dy_tF}.

\smallskip

\textbf{Step 2.}
Now we are ready to prove \eqref{est:tF}.
From \eqref{back_delta_V_12_dy} in the proof of \eqref{est:tB}, we derive that for $s - t \in [0, \tF (t,x,p)]$,
\be \label{back_delta_V_12_dy_tF}
\sqrt{|\P_{1} (s;t, x,p)|^2 + |\P_{2} (s;t, x,p)|^2}
\leq \sqrt{p^2_1 + p^2_2} + \frac{m g}{2\sqrt{2}} (s - t).
\ee
Now we choose 
\Be \label{bar s_tF}
\bar s :  = t + \frac{4}{m g} |p| + 1.
\Ee
Note that $\bar s - t = \frac{4}{m g} |p| + 1 \leq \tF (t,x,p) $ due to \eqref{assume:tF}, thus for all $s \in [ \bar s, t + \tF (t,x,p)]$,
\Be \label{cond:t-s_tF}
t-s \leq \tF (t,x,p).
\Ee 
From \eqref{delta_V_3_dy_tF}, we derive, for all $s \in [ \bar s, t + \tF (t,x,p)]$,
\be \notag
\P_{3} ( s;t,x,p) - | p_3 | 
\leq \P_{3} ( s;t,x,p) - p_3 
\leq - \frac{m g}{2} (s-t).
\ee
Therefore, we further derive that 
\be \label{est1:lower_V_3_dy_tF}
- \P_{3} ( s;t,x,p) \geq \frac{m g}{2} (s - t) - |p_3| 
\geq \frac{m g}{2} (s - t) - |p|, 
\ \text{for all } s \in [ \bar s, t + \tF (t,x,p)].
\ee
Using $s - t \geq \bar{s} - t = \frac{4}{mg} |p| + 1$, together with \eqref{est1:lower_V_3_dy_tF}, we obtain 
\be \label{est:lower_V_3_dy_tF}
- \P_{3} ( s;t,x,p) \geq \frac{m g}{4} (s - t) + \frac{m g}{4}, 
\ \text{for all } s \in [ \bar s, t + \tF (t,x,p)].
\ee
Using \eqref{back_delta_V_12_dy_tF} and our choice \eqref{bar s_tF}, we have that 
\be \label{est:upper_V_12_dy_tF}
\sqrt{ |\P_{1} ( s;t, x,p)|^2 + |\P_{2} (s;t, x,p)|^2}
\leq \frac{m g}{2} (s - t), 
\ \text{for all } s \in [ \bar s, t + \tF (t,x,p)].
\ee
Using \eqref{est:lower_V_3_dy_tF} and \eqref{est:upper_V_12_dy_tF}, we deduce that for all $s \in [ \bar s, t + \tF (t,x,p)]$,
\be \label{est:lower_P3_dy_tF}
4 | \P_{3} (s;t, x,p) |^2  
\geq \frac{(m g)^2}{4} (t-s)^2 \geq |\P_{1} (s;t, x,p) |^2 + |\P_{2} (s;t, x,p)|^2.
\ee
Hence, we obtain that for all $s \in s \in [ \bar s, t + \tF (t,x,p)]$,
\be \notag
\begin{split}
- \V_3 (s;t, x,p) = \frac{ - \P_{3} (s;t, x,p) }{\sqrt{m^2 + |\P (s;t, x,p) |^2 / c^2 }} > \frac{ - \P_3 (s;t, x,p)}{\sqrt{m^2 + 5 |\P_3 (s;t, x,p) |^2 / c^2 }}.
\end{split}
\ee
Using $m^2 + 5 |\P_3 (s;t, x,p) |^2 / c^2 \leq \max \{ 2 m^2, 10 |\P_3 (s;t, x,p) |^2 / c^2 \}$ and \eqref{est:lower_V_3_dy_tF}, we derive 
\be \label{est:lower_HV_3_dy_tF}
\begin{split}
- \V_3 (s;t, x,p) 
& > \min \{ \frac{ - \P_3 (s;t, x,p)}{\sqrt{2m^2}}, \frac{ |\P_3 (s;t, x,p)| }{\sqrt{10 |\P_3 (s;t, x,p) |^2 / c^2}} \}
\\& \geq c_1 := \min \{ \frac{g}{4 \sqrt{2}}, \frac{c}{\sqrt{10}} \}, 
\ \text{for all } s \in s \in [ \bar s, t + \tF (t,x,p)].
\end{split}  
\ee
Using \eqref{est:lower_HV_3_dy_tF} together with \eqref{ODE_F} and $\X_{3} (t + \tF (t,x,p);t,x,p) = 0$, we derive that 
\be \notag
\X_{3} (s^*;t,x,p) \geq \X_{3} (\bar{s};t,x,p)
= \int^{t + \tF (t,x,p)}_{\bar s } - \V_3 (\tau;t,x,p) \dd \tau
> c_1 (t + \tF (t,x,p) - \bar{s}).
\ee
From \eqref{est:x3_dy_tF} and the choice of $\bar{s}$ in \eqref{bar s_tF}, finally we derive 
\be \notag
\begin{split}
\tF (t,x,p) 
\leq \frac{1}{c_1} \times \max \left\{ x_3, \ \frac{2 c}{m g} \big( \sqrt{(m c)^2 + |p|^2} - mc + \frac{3 m g}{2 c} x_{3} + p_3 \big) \right\} + \frac{4}{m g} |p| + 1,
\end{split}
\ee
and thus conclude \eqref{est:tF}.
\end{proof}

The following corollary is a direct consequence of \eqref{est:x3_dy} and \eqref{est:x3_dy_tF} in Proposition \ref{lem:tB}.

\begin{corollary} \label{cor:max_X_dy} 

Recall the dynamic characteristics \eqref{ODE_F} and its self-consistent potentials $\Phi(x)$ and $\Psi(t,x)$ in \eqref{eqtn:Dphi} and \eqref{Poisson_f}, respectively. Suppose the conditions \eqref{Uest:DPhi}, \eqref{Uest:DxPsi} and \eqref{Uest:Dxphi_F} hold.
Consider $(t,x,p) \in [0, \infty) \times  \bar\O \times \R^3$, then for all $s \in [\max \{0, t - \tBp (t,x,p) \}, t + \tFp (t,x,p)]$,
\Be \notag
\X_{\pm, 3} (s; t,x,p) 
\leq \frac{4 c}{m_{\pm} g} |p| + 3 x_3.
\Ee
Furthermore, if $t - \tBp (t,x,p) < 0$, then
\Be \notag
\begin{split}
t \leq \frac{1}{\min \{ \frac{g}{4 \sqrt{2}}, \frac{c}{\sqrt{10}} \}} \times \big( \frac{4 c}{m_{\pm} g} |p| + 3 x_{3} \big) + \frac{4}{m_{\pm} g} |p| + 1.
\end{split}
\Ee
Note that when $t - \tBp (t,x,p) < 0$, instead of boundary, we have to backward $(t,x,p)$ to the initial data $(0, \X(0; t,x,p), \P(0; t,x,p))$, thus the above estimate can help to control time $t$ by $x$ and $p$. These will be used in showing the decay in time in the Asymptotic Stability Criterion (see more details in Proposition \ref{prop:decay} and Theorem \ref{theo:AS}).
\end{corollary}

Similar to the steady case, we build two estimates on $\tB (t,x,p)$ in which $\tB$ can be controlled by $\pB (t,x,p)$ and $\pBn (t,x,p)$. These will be used in the regularity estimates in Section \ref{sec:RD}.

\begin{prop} \label{lem2:tB}

Recall the dynamic characteristics \eqref{ODE_F} and its self-consistent potentials $\Phi(x)$ and $\Psi(t,x)$ in \eqref{eqtn:Dphi} and \eqref{Poisson_f}, respectively. Suppose 
the condition \eqref{Uest:DPhi} holds. Further, we assume
\be \label{Bootstrap_first}
\sup_{0 \leq t < \infty}\| \nabla_x \big( \Phi + \Psi (t) \big) \|_{L^\infty(\O)}
\leq \min \left\{\frac{m_+}{e_+}, \frac{m_-}{e_-} \right\} \times \frac{g}{2}.
\ee
Then for all $(t,x,p) \in [0, \infty) \times  \bar\O \times \R^3$ satisfying $t - \tBp (t,x,p) \geq 0$, the backward exit time \eqref{tb} is bounded by
\Be \label{est:tBpB3}
\begin{split}
\tBp (t,x, p) 
& \leq | \pBpn (t,x, p) | \Big( \frac{2}{m_{\pm} g}  + \frac{ 16 (\sqrt{c m_{\pm}} + | \pBn (t,x,p) |)}{c (m_{\pm})^2 g} 
\\& \qquad \qquad \qquad \qquad \qquad \times \big(1 + \sqrt{(m_{\pm} c)^2 + |\pBp (t,x,p)|^2} + 2 \pBpn (t,x, p) \big) \Big).
\end{split}
\Ee
Moreover, for all $s \in [t - \tBp (t,x,p), t + \tFp (t,x,p)]$,
\Be \label{est:x3pB3}
\X_{\pm, 3} (s;t,x,p) \leq \frac{2}{(m_{\pm})^2 g} (\pBpn (t,x, p))^2.
\Ee
\end{prop}

\begin{proof}

In the proof, we adopt the abuse of notation about $\pm$ in \eqref{abuse}.
From the assumption \eqref{Bootstrap_first}, then for $0 \leq t < \infty$,
\Be \notag
\big\| e_{\pm} \nabla_x \big( \Phi_h (x) + \Psi (t, x) \big) \big\|_{L^\infty (\bar{\O})} \leq \min \left\{m_+, m_- \right\} \times \frac{g}{2}.
\Ee
From the dynamic characteristics \eqref{ODE_F}, this shows that
\be \label{lower_upper_dotp_3_dy}
- \frac{3 m g }{2} 
\leq \frac{d P_{3} (s;t,x,p)}{ds} 
= - e \nabla_{x_3} \big( \Psi (s, \X(s;t,x,p)) + \Phi_h (\X (s;t,x,p) ) \big) - m g 
\leq - \frac{m g }{2}.
\ee
Using the zero Dirichlet boundary condition of $\Phi_h$ in \eqref{eqtn:Dphi}, together with \eqref{Uest:DPhi}, we obtain that for $x \in \O$,
\be \label{est:phih_dy}
- \frac{m_{\pm} g x_3}{2} \leq e_{\pm} \Phi_h (x) \leq \frac{m_{\pm} g x_3}{2}.
\ee
From \eqref{Uest:DPhi} and \eqref{Bootstrap_first}, we derive
\be \label{est:DxPsi_first}
\sup_{0 \leq t < \infty}\| \nabla_x \Psi (t) \|_{L^\infty(\O)} \leq \min \left\{\frac{m_+}{e_+}, \frac{m_-}{e_-} \right\} \times g.
\ee
Using Lemma \ref{lem:conservation_law} and \eqref{est:DxPsi_first}, together with the zero Dirichlet boundary condition of $\Phi_h$, then for $(t,x,p) \in [0, \infty) \times  \bar\O \times \R^3$ and $t - \tB(t,x,p) \leq s \leq t$,
\be \label{est1:pB>p}
\begin{split}
& \sqrt{(m c)^2 + |\P (s;t,x,p)|^2} + \frac{1}{c} \big( e_{\pm} \Phi_h (\X (s;t,x,p)) + m g \X_{3} (s;t,x,p) \big)
\\& = \sqrt{(m c)^2 + |\pB (t,x,p)|^2} - \frac{1}{c} \int^{s}_{\tau = t - \tB (t,x,p) } e \nabla_x \Psi (\tau, \X (\tau;t,x,p)) \cdot \V (\tau;t,x,p) \dd \tau
\\& \leq \sqrt{(m c)^2 + |\pB (t,x,p)|^2}
+ (s - t + \tB (t,x,p) ) m g.
\end{split}
\ee
where we use $\big| \V (\tau;t,x,p) = \frac{\P (\tau;t,x,p)}{\sqrt{m^2_{\pm} + |\P (\tau;t,x,p)|^2 / c^2}} \big| \leq c$ in the last inequality.
Further, we apply \eqref{est:phih_dy} in \eqref{est1:pB>p}, and get
\be \label{est:pB>p}
\begin{split}
\sqrt{(m c)^2 + |\P (s;t,x,p)|^2}
\leq \sqrt{(m c)^2 + |\pB (t,x,p)|^2}
+ (s - t + \tB (t,x,p) ) m g.
\end{split}
\ee
Similar as the proof in Proposition \ref{lem2:tb}, $\P_{3} (s;t,x, p)$ is strictly decreasing along characteristics \eqref{ODE_F}, and the vertical particle velocity satisfies
\be \notag
\sgn \big( \V_{3} (s; t,x,p) \big) = \sgn \big( \P_{3} (s; t,x,p) \big).
\ee
Hence there exists a unique time $s^*$ to achieve a vertical peak, such that 
\Be \notag
\X_{3} (s^*; t,x,p) = \max_{s \in [t - \tB (t,x, p), \ t + \tF (t,x,p)]} { \X_{3} (s; t,x,p) }.
\Ee
Moreover, at $s=s^*$, the vertical velocities satisfies that 
\Be \label{eq:V3s^*_dy}
0 = \frac{d}{ds} \X_{3} (s; t, x, p)|_{s= s^*} = \V_{3} (s^*; t, x, p) =  c \frac{\P_{3} (s^*;t,x,p)}{\P^0 (s^*; t,x,p) },
\Ee
where $\P^0 = \sqrt{(m_\pm c)^2 + |\P|^2}$.
Further, we note that 
\be \label{eq:sign_V3s^*_dy}
\sgn \big( \V_{3} (s; t,x,p) \big) = \sgn \big( \P_{3} (s; t,x,p) :=
\begin{cases}
& > 0, \ \text{ if } t - \tB (t,x,p) \leq s < s^*, \\[5pt]
& < 0, \ \text{ if } s^* < s \leq  t + \tF (t,x,p).
\end{cases}
\ee

For $p_3 \geq 0$, using \eqref{eq:sign_V3s^*_dy} and $\P_{3} (t; t,x,p) = p_3 \geq 0$, we derive that
\be \label{est:s^*_dy}
t -\tB (t,x,p) \leq t \leq s^*, \text{ for } p_3 \geq 0.
\ee
Using \eqref{lower_upper_dotp_3_dy} and \eqref{est:s^*_dy}, we get
\be \label{est:pB3tB}
\begin{split}
& \P_{3} (s^*;t,x,p) - \P_{3} (t - \tB (t,x,p);t,x,p)
\\& = \int^{s^*}_{t - \tB (t,x,p)} \big( - e \nabla_{x_3} \big( \Psi (s, \X(s;t,x,p)) + \Phi_h (X (s; t,x,p) ) \big) - m g \big) \dd s
\\& \leq - \frac{m g }{2} (s^* - t + \tB (t,x, p)) \leq - \frac{m g }{2} \tB (t,x, p).
\end{split}
\ee
Together with \eqref{est:pB3tB} and \eqref{eq:V3s^*_dy}, we have
\be \notag
- \pBn (t,x, p) = \P_{3} (s^*;t,x,p) - \P_{3} (t - \tB (t,x,p); t,x,p)
\leq - \frac{m g }{2} \tB (t,x, p).
\ee
This shows that
\be \label{est2:pB3tB}
\frac{2}{m g} \pBn (t,x, p) \geq \tB (t,x, p), \text{ for } p_3 \geq 0.
\ee

For $p_3 < 0$, using \eqref{eq:sign_V3s^*_dy} and $\P_{3} (t; t,x,p) = p_3 < 0$, we derive that
\be \label{est2:s^*_dy}
t - \tB (t,x,p)  \leq s^* < t \leq t + \tF (t,x,p), \text{ for } p_3< 0.
\ee

\smallskip

\textbf{Step 1. } 
First, we claim that \eqref{est:x3pB3} holds.
Similar from \eqref{est:pB3tB} and \eqref{est2:pB3tB}, we get
\be \notag
- \pBn (t,x, p) = \P_{3} (s^*; t,x,p) - \P_{3} (t - \tB (t,x, p); t,x,p)
\leq - \frac{m g }{2} (s^* - t + \tB (t,x, p)),
\ee
and hence
\be \label{est3:pB3tB}
\frac{2}{m g} \pBn (t,x, p) \geq s^* - t + \tB (t,x, p).
\ee
Furthermore, we have
\Be \label{eq:x3s^*_dy}
\begin{split} 
\X_3 (s^*; t,x,p) 
&= \X_3 (s^*; t,x,p) - \X_3 (t - \tB(t,x,p); t,x,p) 
\\& = \int_{t-\tB(x,p)}^{s^*} \V_{3} (s; t, x, p) \dd s
= \int_{t-\tB(t,x,p)}^{s^*} c \frac{\P_{3} (s; t,x,p)}{\P^0 (s; t,x,p) } \dd s.
\end{split}
\Ee
Since $\P_{3} (t - \tB (t,x,p); t,x,p) = \pBn (t,x,p)$, together with \eqref{lower_upper_dotp_3_dy}, we derive that, for any $t - \tB (t,x, p) \leq s \leq s^*$,
\Be \notag
\begin{split}
\P_{3} ( s; t,x,p) - \P_{3} (t - \tB (t,x,p); t,x,p)
& = \int^{s}_{t - \tB (t,x,p)} \frac{d\P_3(\tau; t,x,p) }{d \tau } \dd \tau 
\\& \leq - \frac{m g}{2} ( s - t + \tB (x,p)).
\end{split}
\Ee
This, together with \eqref{eq:sign_V3s^*_dy}, shows that for $t - \tB (t,x, p) \leq s \leq s^*$,
\be \label{est:p3pB3}
0 \leq \P_{3} ( s; t,x,p)
\leq \pBn (t,x,p) - \frac{m g}{2} ( s - t + \tB (t,x,p)).
\ee
Inputting \eqref{est:p3pB3} into \eqref{eq:x3s^*_dy}, we obtain
\Be \label{est:x3s^*_dy}
\begin{split} 
\X_3 (s^*; t,x,p) 
& = \int_{t-\tB(x,p)}^{s^*} c \frac{\P_{3} (s; t,x,p)}{\P^0 (s; t,x,p) } \dd s
\\& \leq c \int_{t-\tB(x,p)}^{s^*} \frac{\pBn (t,x,p) - \frac{m g}{2} ( s - t + \tB (t,x,p))}{\P^0 (s; t,x,p) } \dd s
\end{split}
\Ee
Since $\P^0 (s; t,x,p) = \sqrt{(mc)^2 + |\P (s; t,x,p)|^2} \geq m c$, we bound \eqref{est:x3s^*_dy} by
\Be \label{est2:x3s^*_dy}
\begin{split} 
\X_3 (s^*; t,x,p) 
& \leq \int_{t-\tB(x,p)}^{s^*} \frac{\pBn (t,x,p) - \frac{m g}{2} ( s - t + \tB (x,p))}{m } \dd s
\\& = (s^* - t + \tB (t,x, p)) \frac{\pBn (t,x,p)}{m} -
\frac{g}{2} \int_{t-\tB(x,p)}^{s^*} ( s - t + \tB (x,p) ) \dd s 
\\& =  (s^* - t + \tB (t,x, p)) \frac{\pBn (t,x,p)}{m} -
\frac{g}{4} ( s^* - t + \tB (x,p) )^2
\\& \leq \frac{2}{m^2 g} (\pBn (t,x, p))^2.
\end{split}
\Ee
where the last inequality follows from \eqref{est3:pB3tB} and $( s^* - t + \tB (x,p) )^2 \geq 0$. Thus, we prove \eqref{est:x3pB3}.

\smallskip

\textbf{Step 2. }  
Second, we claim that \eqref{est:tBpB3} holds. The proof is already complete in the case of $p_3\geq 0$ due to \eqref{est2:pB3tB}. We only consider the other case $p_3<0$. Then, from \eqref{est2:s^*_dy}, we note that  $s^* \leq t$. 
Hence we have that
\Be \label{eq:x3x3s^*_dy}
\X_3 (t; t,x,p) - \X_3 (s^*; t,x,p) 
= \int^{t}_{s^*} \V_{3} (s; t, x, p) \dd s
= \int^{t}_{s^*} c \frac{\P_{3} (s; t,x,p)}{\P^0 (s; t,x,p) } \dd s.
\Ee
Since $\X_3 (t; t,x,p) = x_3 \geq 0$, \eqref{eq:x3x3s^*_dy} implies
\Be \label{est:x3x3s^*_dy}
- \int^{t}_{s^*} c \frac{\P_{3} (s; t,x,p)}{\P^0 (s; t,x,p) } \dd s \leq \X_3 (t; t,x,p) - \int^{t}_{s^*} c \frac{\P_{3} (s; t,x,p)}{\P^0 (s; t,x,p) } \dd s = \X_3 (s^*; t,x,p).
\Ee 
Recall that $\P_{3} (s^*; t,x,p) = 0$, from \eqref{lower_upper_dotp_3_dy} we derive, for $s^* \leq s \leq t$,
\Be \label{est:p3p3s^*_dy}
\P_{3} ( s; t,x,p) = \P_{3} ( s; t,x,p) - \P_{3} (s^*; x,p)
= \int^{s}_{s^*} \frac{d \P_3(\tau;x,p) }{d \tau } \dd \tau 
\leq - \frac{m g}{2} ( s - s^*).
\Ee
Using \eqref{est:p3p3s^*_dy} and \eqref{est:pB>p}, 
together with \eqref{est3:pB3tB}, we have
\be \label{est2:x3x3s^*_dy}
\begin{split}
& - \int^{t}_{s^*} c \frac{\P_{3} (s; t,x,p)}{\P^0 (s; t,x,p) } \dd s 
\\& \geq \int^{t}_{s^*} c \frac{ \frac{m g}{2} ( s - s^*) }{\sqrt{(m c)^2 + |\pB (t,x,p)|^2}
+ (s - t + \tB (t,x,p) ) m g} \dd s
\\& = \int^{t}_{s^*} c \frac{ \frac{m g}{2} ( s - s^*) }{ \sqrt{(m c)^2 + |\pB (t,x,p)|^2}
+ \big( (s - s^*) + (s^* - t + \tB (t,x,p) ) \big) m g } \dd s
\\& \geq \int^{t}_{s^*} c \frac{ \frac{m g}{2} ( s - s^*) }{ \sqrt{(m c)^2 + |\pB (t,x,p)|^2}
+ (s - s^*) m g + 2 \pBn (t,x, p) } \dd s
\\& \geq \frac{c}{2 (1 + \sqrt{(m c)^2 + |\pB (t,x,p)|^2} + 2 \pBn (t,x, p) )} 
\underbrace{ \int^{t}_{s^*} \frac{ m g ( s - s^*) }{ 1 + (s - s^*) m g } \dd s }_{\eqref{est2:x3x3s^*_dy}_*}.
\end{split}
\ee
Then we compute
\be \label{est3:x3x3s^*_dy}
\begin{split}
\eqref{est2:x3x3s^*_dy}_* 
& = ( s - s^*) \big|^{t}_{s^*} - \frac{\ln (1 + m g ( s - s^*) ) }{m g} \big|^{t}_{s^*}
\\& = \frac{1}{m g} \big( m g ( t - s^*) - \ln (1 + m g ( t - s^*) ) \big).
\end{split}
\ee
Consider the function $A (x) = x - \ln (1 + x)$. Now we claim that
\be \label{est:A(x)}
A (x) \geq
\begin{cases}
\frac{x^2}{8}, & \text{ for } 0 \leq x < 1, \\[5pt]
\frac{x}{4}, & \text{ for } 1 \leq x.
\end{cases}
\ee

For $0 \leq x < 1$, we compute that
\be \notag
A (x) \big|_{x = 0} = \frac{x^2}{8} \big|_{x = 0} = 0,
\ee
and
\be \notag
A' (x) = 1 - \frac{1}{1+x} \big|_{x = 0} = \frac{x}{4} \big|_{x = 0} = 0.
\ee
Further, from $0 \leq x < 1$, we get
\be \notag
A'' (x) = \frac{1}{(1+x)^2} \geq \frac{1}{4},
\ee
and thus derive the first part of \eqref{est:A(x)}.

For $1 \leq x$, we have
\be \notag
A (x) \big|_{x = 1} > \frac{1}{4}.
\ee
From $1 \leq x$, we get
\be \notag
A' (x) = 1 - \frac{1}{1+x} \geq \frac{1}{4},
\ee
and thus prove the second part of \eqref{est:A(x)}.

Applying \eqref{est:A(x)} on \eqref{est3:x3x3s^*_dy}, we obtain
\Be \label{est4:x3x3s^*_dy}
\eqref{est2:x3x3s^*_dy}_*
\geq
\begin{cases}
\frac{m g ( t - s^*)^2}{8}, & \text{ for } 0 \leq m g ( t - s^*) < 1, \\[5pt]
\frac{ t - s^* }{4}, & \text{ for } 1 \leq m g ( t - s^*).
\end{cases}
\Ee
Inputting \eqref{est2:x3x3s^*_dy} into \eqref{est:x3x3s^*_dy}, together with \eqref{est2:x3s^*_dy}, we derive 
\Be \notag
\frac{c \times \eqref{est2:x3x3s^*_dy}_*}{2 (1 + \sqrt{(m c)^2 + |\pB (t,x,p)|^2} + 2 \pBn (t,x, p) )} 
\leq \X_3 (s^*; x,p) \leq \frac{2}{m^2 g} (\pBn (t,x, p))^2,
\Ee
and thus,
\Be \label{est3:s^*_dy}
\frac{4}{c m^2 g} (\pBn (t,x, p))^2 (1 + \sqrt{(m c)^2 + |\pB (t,x,p)|^2} + 2 \pBn (t,x, p) )
\geq \eqref{est2:x3x3s^*_dy}_*
\Ee
Since $s^* < t$, \eqref{est4:x3x3s^*_dy} and \eqref{est3:s^*_dy} show that for $0 \leq m g ( t - s^*) < 1$,
\be \label{est4:s^*_dy}
t - s^*
\leq |\pBn (t,x, p)| \sqrt{\frac{32}{c m^3 g^2} (1 + \sqrt{(m c)^2 + |\pB (t,x,p)|^2} + 2 \pBn (t,x, p))},
\ee
and for $1 \leq m g ( t - s^*)$,
\be \label{est5:s^*_dy}
t - s^*
\leq \frac{16}{c m^2 g } (\pBn (t,x, p))^2 (1 + \sqrt{(m c)^2 + |\pB (t,x,p)|^2} + 2 \pBn (t,x, p) ).
\ee
Finally, combining \eqref{est3:pB3tB}, \eqref{est4:s^*_dy} and \eqref{est5:s^*_dy}, we conclude that
\be \notag
\begin{split}
\tB (t,x, p) 
& = s^* - t + \tB (t,x, p) + t - s^* 
\\& \leq | \pBn (t,x, p) | \Big( \frac{2}{m g}  + \frac{ 16 (\sqrt{c m} + | \pBn (t,x,p) |)}{c m^2 g} 
\\& \qquad \qquad \qquad \qquad \qquad \quad \times \big(1 + \sqrt{(m c)^2 + |\pB (t,x,p)|^2} + 2 \pBn (t,x, p) \big) \Big),
\end{split}
\ee
and prove \eqref{est:tBpB3}.
\end{proof} 

Based on Proposition \ref{lem2:tB}, in the following proof, we show that when $\nabla_x \Psi (t)$ is sufficiently small, $\tB (t,x, p) $ can be controlled by a linear form of $|\pB (t,x,p)|$. We skip some details that directly follow from the proof of Proposition \ref{lem2:tB}.

\begin{prop} \label{lem3:tB} 

Recall the dynamic characteristics \eqref{ODE_F} and its self-consistent potentials $\Phi(x)$ and $\Psi(t,x)$ in \eqref{eqtn:Dphi} and \eqref{Poisson_f}, respectively. Suppose the conditions \eqref{Uest:DPhi}, \eqref{Uest:DxPsi} and \eqref{Uest:Dxphi_F} hold.
Then for all $(t,x,p) \in [0, \infty) \times  \bar\O \times \R^3$ satisfying $t - \tBp (t,x,p) \geq 0$, the backward exit time \eqref{tb} is bounded by
\Be \label{est2:tBpB3}
\begin{split}
\tBp (t,x, p) 
& \leq \frac{8}{m_{\pm} g} \sqrt{(m_{\pm} c)^2 + |\pBp (t,x,p)|^2}.
\end{split}
\Ee
Moreover, for all $(t,x,p) \in [0, \infty) \times  \bar\O \times \R^3$ satisfying $t - \tBp (t,x,p) \geq 0$,
\Be \label{est:x3ppB}
\sqrt{(m_{\pm} c)^2 + |p|^2} + \frac{m_{\pm} g}{2 c} x_3
\leq \frac{7}{6} \sqrt{(m_{\pm} c)^2 + |\pBp (t,x,p)|^2}.
\ee
\end{prop}

\begin{proof}

In the proof, we adopt the abuse of notation about $\pm$ in \eqref{abuse}.
Applying Lemma \ref{lem:conservation_law} with $s = t$, together with \eqref{Uest:DPhi} and the zero Dirichlet boundary condition of $\Phi_h$, we have, for $(t,x,p) \in [0, \infty) \times  \bar\O \times \R^3$,
\be \label{est1:pB>p+tB}
\begin{split}
& \sqrt{(m c)^2 + |p|^2} + \frac{1}{c} \big( e_{\pm} \Phi_h (x) + m g x_{3} \big)
\\& = \sqrt{(m c)^2 + |\pB (t,x,p)|^2} - \frac{1}{c} \int^{t}_{\tau = t - \tB (t,x,p) } e \nabla_x \Psi (\tau, \X (\tau;t,x,p)) \cdot \V (\tau;t,x,p) \dd \tau
\\& \leq \sqrt{(m c)^2 + |\pB (t,x,p)|^2}
+ \frac{m g}{48} \tB (t,x,p).
\end{split}
\ee
where last inequality follows from $\sup_{0 \leq t < \infty}\| e \nabla_x \Psi (t) \|_{L^\infty(\O)} \leq \frac{m g}{48}$ in \eqref{Uest:DxPsi} and $| \V | \leq c$.
Further, from \eqref{est:phih_dy} and \eqref{est1:pB>p+tB}, we get
\be \label{est2:pB>p+tB}
\begin{split}
\sqrt{(m c)^2 + |p|^2} + \frac{m g}{2 c} x_3
\leq \sqrt{(m c)^2 + |\pB (t,x,p)|^2}
+ \frac{m g}{48} \tB (t,x,p).
\end{split}
\ee

Recall from the proof of Proposition \ref{lem2:tB}, there exists a unique time $s^*$, such that 
\Be \notag
\X_{3} (s^*; t,x,p) = \max_{s \in [t - \tB (t,x, p), \ t + \tF (t,x,p)]} { \X_{3} (s; t,x,p) }.
\Ee
and
\be \notag
\sgn \big( \V_{3} (s; t,x,p) \big) = \sgn \big( \P_{3} (s; t,x,p) :=
\begin{cases}
& > 0, \ \text{ if } t - \tB (t,x,p) \leq s < s^*, \\[5pt]
& < 0, \ \text{ if } s^* < s \leq  t + \tF (t,x,p).
\end{cases}
\ee

For $p_3 \geq 0$, from \eqref{est2:pB3tB}, we have
\be \notag
\frac{2}{m g} \pBn (t,x, p) \geq \tB (t,x, p),
\ee
and conclude \eqref{est2:tBpB3}.

For $p_3 < 0$, from \eqref{est2:s^*_dy}, we derive that
\be \notag
t - \tB (t,x,p)  \leq s^* < t \leq t + \tF (t,x,p).
\ee
From \eqref{est3:pB3tB}, we get
\be \label{est:pB3tB_s^*-t}
\frac{2}{m g} \pBn (t,x, p) \geq s^* - t + \tB (t,x, p).
\ee
On the other hand, \eqref{est:p3p3s^*_dy} shows that
\Be \notag
p_3
= \int^{t}_{s^*} \frac{d \P_3(\tau;x,p) }{d \tau } \dd \tau 
\leq - \frac{m g}{2} ( t - s^*), \text{ for } p_3< 0.
\Ee
This implies that
\be \label{est:p3_t-s^*_dy}
\frac{m g}{2} ( t - s^*)
\leq - p_3 \leq \sqrt{(m c)^2 + |p|^2}.
\ee
Combining \eqref{est:pB3tB_s^*-t} with \eqref{est:p3_t-s^*_dy}, we obtain 
\be \label{est:pB3tB_dy}
\begin{split}
\tB (t,x, p) 
& = \big( s^* - t + \tB (t,x, p) \big) + ( t - s^*)
\\& \leq \frac{2}{m g} \pBn (t,x, p) + ( t - s^*).
\end{split}
\ee
Inputting \eqref{est:pB3tB_dy} into \eqref{est2:pB>p+tB}, together with \eqref{est:p3_t-s^*_dy}, we derive 
\be \label{est2:pB3tB_dy}
\begin{split}
\frac{m g}{2} ( t - s^*)
& \leq \sqrt{(m c)^2 + |p|^2} + \frac{m g}{2 c} x_3
\\& \leq \sqrt{(m c)^2 + |\pB (t,x,p)|^2}
+ \frac{m g}{48} \big( \frac{2}{m g} \pBn (t,x, p) + ( t - s^*) \big).
\end{split}
\ee
This shows that
\be \label{est3:pB3tB_dy}
\begin{split}
\frac{m g}{4} (t - s^*)
\leq \sqrt{(m c)^2 + |\pB (t,x,p)|^2} + \frac{1}{24} \pBn (t,x, p).
\end{split}
\ee
Together with \eqref{est:pB3tB_dy}, we conclude \eqref{est2:tBpB3}.

\smallskip

Now applying \eqref{est2:tBpB3} on \eqref{est2:pB>p+tB}, we deduce
\be \notag
\begin{split}
& \sqrt{(m c)^2 + |p|^2} + \frac{m g}{2 c} x_3
\\& \leq \sqrt{(m c)^2 + |\pB (t,x,p)|^2}
+ \frac{m g}{48} \times \frac{8}{m g} \sqrt{(m c)^2 + |\pB (t,x,p)|^2}
\\& = \frac{7}{6} \sqrt{(m c)^2 + |\pB (t,x,p)|^2}.
\end{split}
\ee
and conclude \eqref{est:x3ppB}.
\end{proof}

\section{Steady Solutions}

In this section, we construct a weak solution to the steady problem, and prove its regularity, and finally study a uniqueness theorem. 

\subsection{Construction} \label{sec:CS} 

We construct solutions to the steady problem \eqref{VP_h}-\eqref{eqtn:Dphi} via the following sequences: for any $\ell \in \N$,
\begin{align}
v_\pm \cdot \nabla_x h^{\ell+1}_{\pm} 
+ \big( e_{\pm} ( \frac{v_\pm}{c} \times B 
- \nabla_x \Phi^\ell ) - \nabla_x (m_{\pm} g x_3) \big) \cdot \nabla_p h^{\ell+1}_{\pm} = 0 & \ \ \text{in} \ \O \times \R^3, \label{eqtn:hk} \\
h^{\ell+1}_{\pm} (x,p) = G_{\pm} (x,p) & \ \ \text{on} \ \gamma_-, \label{bdry:hk} \\	
\rho^\ell = \int_{\R^3} ( e_+ h^\ell_+ + e_{-} h^\ell_{-} ) \dd p & \ \ \text{in} \ \O, \label{eqtn:rhok} \\
- \Delta \Phi^\ell = \rho^\ell & \ \ \text{in} \ \O, \label{eqtn:phik} \\
\Phi^\ell =0  & \ \ \text{on} \ \p\O, \label{bdry:phik}
\end{align}
where $v_\pm = \frac{p}{\sqrt{ m_\pm^2 + |p|^2/ c^2}}$ and the initial setting $h^0_{\pm} = 0$ and $(\rho^0,  \nabla_x \Phi^0) = (0, \mathbf{0})$.

\smallskip 

First, we construct $h^1_{\pm}$ solving \eqref{eqtn:hk} and \eqref{bdry:hk} from $(\rho^0,  \nabla_x \Phi^0) = (0, \mathbf{0})$. We consider the Lagrangian formulation of the following characteristics:
\Be \label{z1}
\begin{split}
\big( X^{1}_{\pm} (t;x,p), P^{1}_{\pm} (t;x,p) \big) |_{t=0} = (x, p) \ \ \text{at} \ t = 0, &
\\ \frac{d X^1_{\pm} }{d t} = V^1_{\pm} = \frac{P^1_{\pm}}{\sqrt{m^2_{\pm} + |P^1_{\pm}|^2 / c^2}}, \ \
\frac{d P^1_{\pm} }{d t} =  {e_{\pm}} V^1_{\pm} \times \frac{B}{c} - {m_{\pm}} g \mathbf{e}_3. &
\end{split}
\Ee
Then from \eqref{eqtn:rhok}, as long as it is well-defined, we obtain
\[
\rho^1 (x) = \int_{\R^3} ( e_+ h^1_+ + e_{-} h^1_{-} ) \dd p.
\]
Using the Green's function in \eqref{phi_rho}, together with \eqref{eqtn:phik}, we derive $\Phi^1$ and $\nabla_x \Phi^1$ by
\[
- \Delta \Phi^1 (x) = \rho^1 (x).
\]

\smallskip

Second, we construct $(h^{\ell+1}_{\pm}, \rho^{\ell}, \nabla_x \Phi^{\ell})$ solving \eqref{eqtn:hk}-\eqref{bdry:phik} for $\ell \geq 1$ by iterating the process. Given $\nabla_x \Phi^{\ell}$, we construct $h^{\ell+1}_{\pm}$ along the characteristics $(X^{\ell+1}_{\pm}, P^{\ell+1}_{\pm})$ as follows:
\be \label{VP_h^l+1}
\big( X^{\ell+1}_{\pm} (t;x,p)，P^{\ell+1}_{\pm} (t;x,p) \big) |_{t=0} = (x, p) \ \ \text{at} \ t = 0,
\ee
and   
\Be \label{bdry:h^l+1}
\begin{split}
& \frac{d X^{\ell+1}_{\pm} }{d t} = V^{\ell+1}_{\pm} = \frac{P^{\ell+1}_{\pm}}{\sqrt{m^2_{\pm} + |P^{\ell+1}_{\pm}|^2 / c^2}}, 
\\& \frac{d P^{\ell+1}_{\pm} }{d t} =  {e_{\pm}} ( V^{\ell+1}_{\pm} \times \frac{B}{c} - \nabla_x \Phi^{\ell} (X^{\ell+1}_{\pm} )) - {m_{\pm}} g \mathbf{e}_3. 
\end{split}
\Ee
Using the Peano theorem, together with the boundness of $\rho^\ell$ (see \eqref{Uest:rho^k}) and the continuity of $\nabla_x \Phi^\ell$ (see \eqref{phi_rho}), we conclude the existence of solutions (not necessarily unique). 
Then we obtain $\rho^{\ell+1}$ from \eqref{eqtn:rhok} and solve $\nabla_x \Phi^{\ell+1}$ from \eqref{eqtn:phik} and \eqref{bdry:phik}.

Now suppose that $(X^{\ell+1}, P^{\ell+1})$ exists, then   
\Be 
\begin{split}
P^{\ell+1}_{\pm} (t;x,p) & = p + \int^t_0 \Big( e_{\pm}\big( V^{\ell+1}_{\pm} \times \frac{B}{c} - \nabla_x  \Phi^{\ell}_{\pm} (X^{\ell+1}_{\pm} (s;x,p)) \big) - m_{\pm} g \mathbf{e}_3 \Big)  \dd s,
\\ X^{\ell+1}_{\pm} (t;x,p) & = x + \int^t_0 V^{\ell+1}_{\pm} (s;x,p) \dd s.
\end{split}
\Ee

Analogous to Definition \ref{def:tb}, we define the backward exit time, position, and momentum for every iteration, that is, for any $\ell \geq 0$,
\Be \label{def:tb_l}
\begin{split} 
& \tbp^\ell (x,p) : = \sup\{ s  \in [0,\infty) : X_{\pm, 3}^\ell (-\tau; x,p) \ \text{for all } \tau \in (0, s)\} \geq 0, 
\\& \xbp^{\ell} (x,p) := X^{\ell}_{\pm} (-\tbp^{\ell}(x,p); x,p), \ \
\pbp^{\ell} (x,p) := P^{\ell}_{\pm} (-\tbp^{\ell}(x,p); x,p).
\end{split}
\Ee
Given $\nabla_x \Phi^\ell$, the weak solution $h^{\ell+1}_{\pm}$ to \eqref{eqtn:hk}-\eqref{bdry:hk} follows
\Be \label{form:h^k}
\begin{split}
h^{\ell+1}_{\pm} (x,p) 
= h^{\ell+1}_{\pm}  (X_{\pm}^{\ell+1}  ( -t;x,p), P^{\ell+1}_{\pm} (-t;x,p) )  \ \ \text{for all} \ t \in [0, \tbp^{\ell+1} (x,p)].
\end{split}
\Ee
Similar to \eqref{w^h}, we define a weight function which is invariant along the characteristics
\Be \label{def:w^k}
w^{\ell+1}_{\pm, \beta} (x,p) 
= e^{ \beta \left( \sqrt{(m_{\pm} c)^2 + |p|^2} + \frac{1}{c} ( e_{\pm} \Phi^{\ell} (x) + m_{\pm} g x_3 ) \right) }
\ \ \text{in} \ \O \times \R^3.
\Ee
Note that, at the boundary the weight function satisfies
\Be \label{wk:bdry}
w^{\ell+1}_{\pm, \beta} (x, p) 
= e^{ \beta \sqrt{(m_{\pm} c)^2 + |p|^2} } 
\ \ \text{on} \ \p\O.
\Ee
Using \eqref{bdry:hk}, \eqref{form:h^k}, and \eqref{wk:bdry}, as long as $\tb^{\ell+1} (x,p) < \infty$, then we have  
\Be
\begin{split} \label{form:h^k_G}
w^{\ell+1}_{\pm, \beta} (x,p) h^{\ell+1}_{\pm} (x,p) 
& = w^{\ell+1}_{\pm, \beta} (\xbp^{\ell+1} (x,p), \pbp^{\ell+1} (x,p) )  G_{\pm} (\xbp^{\ell+1} (x,p), \pbp^{\ell+1} (x,p) )
\\&	= e^{ \beta \sqrt{(m_{\pm} c)^2 + | \pbp^{\ell+1} (x,p)|^2} } G_{\pm} (\xbp^{\ell+1} (x,p), \pbp^{\ell+1} (x,p) ).
\end{split}	
\Ee

So far, we have constructed the sequences $(h^{\ell+1}_{\pm}, \rho^{\ell}, \nabla_x \Phi^{\ell})$ for $\ell \geq 1$.
The main challenge now is to demonstrate the strong convergence of each component in the sequences. Without establishing this convergence, the limits of these sequences may not correspond to a steady solution.

Using the mathematical induction, we obtain the uniform-in-$\ell$ estimates on $\| h_{\pm}^{\ell+1} \|_{L^\infty (\bar \O \times \R^3)} $ and $|\rho^{\ell} (x)|$.
Together with the elliptic estimates (see Lemma \ref{lem:rho_to_phi}), we derive the following key estimate:
\be \notag
\| \nabla_x \Phi^ \ell \|_{L^\infty (\bar{\O})} 
\leq  \min \big( \frac{m_{+}}{e_{+}}, \frac{m_{-}}{e_{-}} \big) \times \  \frac{g}{2},
\ \text{ for any }
\ell \in \N.
\ee
This result ensures that the downward gravitational force dominates in the field (see Proposition \ref{prop:Unif_steady} below).

\begin{prop} \label{prop:Unif_steady}

Suppose the condition \eqref{condition:beta} holds for some $g, \beta > 0$. 
Then, under the above construction $(h^{\ell+1}_{\pm}, \rho^\ell, \nabla_x \Phi^\ell)$ satisfies the following uniform-in-$\ell$ estimates:
\begin{align}
& \| h_{\pm}^{\ell+1} \|_{L^\infty (\bar \O \times \R^3)} 
\leq \|  G_{\pm} \|_{L^\infty (\gamma_-)},
\label{Uest:h^k} 
\\& \| w^{\ell+1}_{\pm, \beta} h_{\pm}^{\ell+1}   \|_{L^\infty (\bar \O \times \R^3)} 
\leq \| e^{ \beta \sqrt{(m_{\pm} c)^2 + |p|^2}} G_{\pm} \|_{L^\infty (\gamma_-)}, \label{Uest:wh^k}
\\& 
| \rho^{\ell} (x) |
\leq \frac{e_+}{\beta} \| e^{ \beta \sqrt{(m_{+} c)^2 + |p|^2}} G_+ \|_{L^\infty(\gamma_-)} e^{-  \beta \frac{ m_+ }{2 c} g x_3} + \frac{e_-}{\beta} \| e^{ \beta \sqrt{(m_{-} c)^2 + |p|^2}} G_- \|_{L^\infty(\gamma_-)} e^{- \beta \frac{ m_{-} }{2 c} g x_3}, \label{Uest:rho^k}
\\& 
\| \nabla_x \Phi^ \ell \|_{L^\infty (\bar{\O})} 
\leq  \min \big( \frac{m_{+}}{e_{+}}, \frac{m_{-}}{e_{-}} \big) \times \  \frac{g}{2}. 
\label{Uest:DPhi^k}
\end{align} 
Furthermore, $\Phi^\ell \in H^1_{0, \text{loc}} (\O)$ and 
\Be\label{lower_w_st}
w^{\ell+1}_{\pm, \beta} (x,p) 
\geq e^{ \beta \left( \sqrt{(m_{\pm} c)^2 + |p|^2} + \frac{m_{\pm}}{2 c} g x_3 \right) }.
\Ee
\end{prop}

\begin{proof}

Under the initial setting $(\rho^0, \nabla_x \Phi^0) = (0, \mathbf{0})$, together with the invariance of $h^1_{\pm}$ and the weight function $w^{1}_{\pm, \beta} (x,p)$ along the characteristics \eqref{z1}, we deduce \eqref{Uest:h^k}-\eqref{Uest:DPhi^k} hold for $\ell = 0$.

\smallskip

Now we prove this by induction. In the following proof, we adopt \text{the abuse of notation about $\pm$} in \eqref{abuse}.
Assume a positive integer $k > 0$ and suppose that \eqref{Uest:h^k}-\eqref{Uest:DPhi^k} hold for $0 \leq \ell \leq k$.
From \eqref{eqtn:rhok}, then for $\ell = k + 1$, 
\be \label{est1:rho_k+1}
\begin{split}
| \rho^{k+1} (x) |
& = \Big| \int_{\R^3} ( e_+ h^{k+1}_+ + e_{-} h^{k+1}_{-} ) \dd p \Big|
\\& \leq \Big| \int_{\R^3} e_{+} h^{k+1}_+ \dd p \Big| + \Big| \int_{\R^3} e_{-} h^{k+1}_{-} \dd p \Big|
\\& \leq \Big|\int_{\R^3} w^{k+1}_{+, \beta}  h^{k+1}_{+} (x,p) \frac{e_{+}}{w^{k+1}_{+, \beta} (x,p)} \dd p \Big| 
+ \Big|\int_{\R^3} w^{k+1}_{-, \beta}  h^{k+1}_{-} (x,p) \frac{e_{-}}{w^{k+1}_{-, \beta} (x,p)} \dd p \Big|.
\end{split}
\ee
Using \eqref{Uest:wh^k} for $\ell=k$, we have
\be \label{est2:rho_k+1}
\Big|\int_{\R^3} w^{k+1}_{\beta} h^{k+1} (x,p) \frac{e}{w^{k+1}_{\beta} (x,p)} \dd p \Big| 
\leq \| e^{ \beta \sqrt{(m c)^2 + |p|^2}} G \|_{L^\infty(\gamma_-)} \int_{\R^3} \frac{e}{w^{k+1}_{\beta} (x, p)} \dd p.
\ee
Using \eqref{Uest:DPhi^k} for $\ell=k$ and the zero Dirichlet boundary condition \eqref{bdry:phik}, we derive 
\Be \notag
\begin{split}
|\Phi ^{k } (x_\parallel, x_3)| 
& = \Big| \Phi^{k} (x_\parallel, 0) + \int ^{x_3}_0 \nabla_{x_3} \Phi^{k} (x_\parallel, y) \dd y \Big| 
\\& \leq \int ^{x_3}_0 \big| \nabla_{x_3} \Phi^{k }  (x_\parallel, y) \big| \dd y 
\leq \min \big( \frac{m_{+}}{e_{+}}, \frac{m_{-}}{e_{-}} \big) \times \  \frac{g}{2} x_3. 
\end{split}
\Ee
Thus, we derive \eqref{lower_w_st} for $\ell=k+1$ via
\Be\notag
w^ {k+1 }_\beta(x,p)
\geq e^{ \beta \left( \sqrt{(m c)^2 + |p|^2} + \frac{1}{c} ( - \frac{1}{2} m g x_3 + m g x_3 ) \right) } 
\geq e^{ \beta \left( \sqrt{(m c)^2 + |p|^2} + \frac{m}{2 c} g x_3 \right) }.
\Ee
This shows that
\Be \label{est3:rho_k+1}
\int_{\R^3} \frac{1}{w_\beta^{k+1} (x,p)} \dd p
\leq \int_{\R^3} e^{ - \beta \left( \sqrt{(m c)^2 + |p|^2} + \frac{m}{2 c} g x_3 \right) } \dd p 
\lesssim \frac{1}{\beta} e^{-  \beta \frac{m}{2 c} g x_3}.
\Ee
Together with \eqref{est1:rho_k+1}-\eqref{est3:rho_k+1}, we obtain
\be \notag
\begin{split}
| \rho^{k+1} (x) |
& \leq \Big|\int_{\R^3} w^{k+1}_{+, \beta}  h^{k+1}_{+} (x,p) \frac{e_{+}}{w^{k+1}_{+, \beta} (x,p)} \dd p \Big| 
+ \Big|\int_{\R^3} w^{k+1}_{-, \beta}  h^{k+1}_{-} (x,p) \frac{e_{-}}{w^{k+1}_{-, \beta} (x,p)} \dd p \Big|
\\& \leq \| w^{\ell+1}_{+, \beta} h_{+}^{\ell+1} \|_{L^\infty (\bar \O \times \R^3)}  \frac{e_+ }{\beta} e^{-  \beta \frac{ m_+ }{2 c} g x_3} + \| w^{\ell+1}_{-, \beta} h_{-}^{\ell+1} \|_{L^\infty (\bar \O \times \R^3)} \frac{ e_- }{\beta} e^{- \beta \frac{ m_{-} }{2 c} g x_3}.
\end{split}
\ee
Using \eqref{Uest:wh^k} for $\ell=k$, we conclude \eqref{Uest:rho^k} for $\ell=k+1$.

Using \eqref{phi_rho}, we obtain $\Phi^{k+1} \in H^1_{0, \text{loc}} (\O)$ which is the weak solution to \eqref{eqtn:phik}-\eqref{bdry:phik}.
Applying \eqref{est:nabla_phi} in Lemma \ref{lem:rho_to_phi} with $A, B$ as follows:
\be \notag
A = \frac{1}{\beta} \big( e_+ \| e^{ \beta \sqrt{(m_{+} c)^2 + |p|^2}} G_+ \|_{L^\infty(\gamma_-)} + e_{-} \| e^{ \beta \sqrt{(m_{-} c)^2 + |p|^2}} G_- \|_{L^\infty(\gamma_-)} \big),
\ \text{and } \
B =  \beta \frac{ \hat{m} }{2 c} g,
\ee
with $\hat{m} = \min ( m_{+}, m_{-} )$, then we derive that  
\Be \notag
|\p_{x_j} \Phi^{k+1} (x)| 
\leq \frac{\mathfrak{C} }{\beta} \big( e_+ \| e^{ \beta \sqrt{(m_{+} c)^2 + |p|^2}} G_+ \|_{L^\infty(\gamma_-)} + e_{-} \| e^{ \beta \sqrt{(m_{-} c)^2 + |p|^2}} G_- \|_{L^\infty(\gamma_-)} \big) \times
\big( 1 + \frac{2 c}{ \beta \hat{m} g} \big).
\Ee
From the assumption of $\beta$ in \eqref{condition:beta}, we conclude that \eqref{Uest:DPhi^k} holds for $\ell=k+1$. 

Since $\| \nabla_x \Phi^{k} \|_{L^\infty (\bar \O)} \leq  \min \big( \frac{m_{+}}{e_{+}}, \frac{m_{-}}{e_{-}} \big) \times \  \frac{g}{2}$ holds, 
using \eqref{est:tb^h} in Proposition \ref{lem:tb}, we derive
\Be
\tb^{k+1}(x,p) \leq \frac{4}{m g} \big( \sqrt{(m c)^2 + |p|^2} + \frac{3}{2c} m g x_{3} \big) < \infty
\ \ \text{for all} \
(x, p) \in \bar \O \times \R^3.
\Ee  
From \eqref{form:h^k_G}, we show \eqref{Uest:h^k} and \eqref{Uest:wh^k} for $\ell=k+1$.

\smallskip

Now we have proved \eqref{Uest:h^k}-\eqref{Uest:DPhi^k} hold for $\ell = k+1$. Therefore, we complete the proof by induction. Using \eqref{Uest:DPhi^k} and the zero Dirichlet boundary condition \eqref{bdry:phik}, we can check that $\Phi^\ell \in H^1_{0, \text{loc}} (\O)$ and $\eqref{lower_w_st}$ holds.
\end{proof}

\subsection{A Priori Estimate} 
\label{sec:RS}

In this section, we establish a priori estimate of $(h_{\pm}, \rho, \Phi )$ solving \eqref{VP_h}-\eqref{eqtn:Dphi}. 
Based on the corresponding characteristic trajectory given in \eqref{ODE_h}, we have
\Be \notag
\nabla_{x} h_{\pm} (x,p) = \nabla_{x} \xbp (x,p) \cdot \nabla_{x_\parallel} G_{\pm} (\xbp, \pbp) + \nabla_{x} \pbp (x,p) \cdot \nabla_p G_{\pm} (\xbp, \pbp).
\Ee

Through direct computation (see Lemmas \ref{lem:XV_xv}-\ref{lem:nabla_zb}), a critical term $|\vbpn|^{-1}$ appears in the expansions of $\nabla_{x} \xbp$ and $\nabla_{x} \pbp$. To handle this, we introduce the kinetic distance function $\alpha_\pm (x, p)$ defined in $\eqref{alpha}$, where $\alpha_{\pm} (x, p) = |v_{\pm, 3}|$ for $x \in \p\O$. 
By analyzing the characteristic trajectory, we establish a lower bound for $|v_{\pm, 3}|$ using the kinetic distance function. Subsequently, we derive a key estimate for $| \nabla_p h_{\pm} |$ in Proposition \ref{prop:Reg}.

\begin{lemma} \label{VL} 

Let $\alpha_{\pm} (x,p)$ be defined as in \eqref{alpha}.
Suppose the assumption \eqref{Uest:DPhi} holds. 
Recall the characteristics $Z_{\pm} (t;x,p) = (X_{\pm} (t;x,p), P_{\pm} (t;x,p) )$ solving \eqref{ODE_h}. For all $(x,p) \in \O \times \R^3$ and $t  \in  [- \tbp (x,p), 0]$, then
\Be \label{est:alpha}
\begin{split}
\alpha_{\pm} (X_{\pm} (t;x,p), P_{\pm} (t;x,p)) 
& \leq  \alpha_{\pm} (x,p) e^{ (4g + \| \p_{x_3}^2  \Phi_h \|_{L^\infty(\bar{\O})} + \|  \nabla_{x_\parallel} \p_{x_3} \Phi_h \|_{L^\infty(\p\O)} ) |t| },\\
\alpha_{\pm} (X_{\pm} (t;x,p), P_{\pm} (t;x,p)) 
& \geq \alpha_{\pm} (x,p) e^{  - (4g + \| \p_{x_3}^2 \Phi_h \|_{L^\infty(\bar{\O})} + \|  \nabla_{x_\parallel} \p_{x_3} \Phi_h \|_{L^\infty(\p\O)}	 ) |t| }.
\end{split}
\Ee
Furthermore, we have  
\begin{align}
| \vbpn (x,p)| 
& \geq \alpha_{\pm} (x,p) e^{ - (4g + \| \p_{x_3}^2   \Phi_h \|_{L^\infty(\bar{\O})} + \|  \nabla_{x_\parallel} \p_{x_3} \Phi_h \|_{L^\infty(\p\O)}	) \tbp (x,p) } \label{est1:alpha}
\\& \geq {\alpha_{\pm} (x,p) }
 e^{ -\frac{4}{m_{\pm} g} (4g + 2 \| \nabla_x ^2 \Phi_h \|_{\infty} ) |\pbp|}. \label{est1:xb_x/w}
\end{align}
\end{lemma}

\begin{proof}

For the sake of simplicity, we abuse the notation as in \eqref{abuse}.
Recall from \eqref{alpha},
\be \notag
\alpha (x,p) = \sqrt{ |x_3|^2  + |v_3|^2 +  2 ( e \p_{x_3} \Phi_h (x_\parallel , 0) + m g ) \frac{c x_3}{ p^0 } },
\ee
where $p^0 = \sqrt{ (m c)^2 + |p|^2}$ and $v =  \frac{p}{ \sqrt{ m^2 + |p|^2/ c^2} }$.
Under direct computation,
\be \notag
\nabla_p (v_3) = 
\big(\frac{- v_1 v_3}{ c p^0 }, \ \frac{- v_2 v_3}{ c p^0 }, \ \frac{c}{p^0} - \frac{(v_3)^2}{c p^0} \big) 
\ \text{ and } \ 
\nabla_p \big(\frac{c}{ p^0} \big) = - \frac{ v}{ (p^0)^2}.
\ee
Thus we get
\Be \label{est1:alpha_steady}
\begin{split}
& \big[ v \cdot \nabla_x + \big( e ( \frac{v}{c} \times B - \nabla_x \Phi_h ) - \nabla_x (m g x_3) \big) \cdot \nabla_p \big] \alpha (x, p)
\\& = \frac{ v_3 x_3 
- \left[ v_3 ( \frac{c}{p^0} - \frac{(v_3)^2}{c p^0} ) ( e \p_{x_3} \Phi_h (x) + m g ) 
- \frac{ (v_3)^2 }{ c p^0 } v_{\|} \cdot e \p_{x_{\|}} \Phi_h (x) \right]}{ \sqrt{ |x_3|^2  + |v_3|^2 +  2 (\p_{x_3} \Phi_h (x_\parallel , 0) + g ) \frac{c x_3}{ p^0 } } }
\\& \ \ \ \ \frac{ + \left[ v_\parallel \cdot \nabla_{x_\parallel} e \p_{x_3} \Phi_h   (x_\parallel, 0) \frac{c x_3}{ p^0 } + \frac{c v_3}{ p^0 } ( e \p_{x_3} \Phi_h (x_\parallel , 0) + m g ) \right] }{ * }
\\& \ \ \ \ \frac{ + \big( e \p_{x_3} \Phi_h (x_\parallel, 0) + m g \big) \frac{ x_3 }{ (p^0)^2 } v \cdot \big( e \nabla_x \Phi_h (x) + \nabla_x (m g x_3) \big) }{*} .
\end{split}
\Ee
Using the fundamental theorem of calculus, we have
\be \label{est2:alpha_steady}
\begin{split}
& - \frac{c v_3}{ p^0 } ( e \p_{x_3} \Phi_h (x) + m g ) + \frac{c v_3}{ p^0 } ( e \p_{x_3} \Phi_h (x_\parallel , 0) + m g )
\\& = \frac{c v_3}{ p^0 } \times \int^0_{x_3} \p_{x_3} \p_{x_3} \Phi_h (x_\parallel, y_3) \dd y_3 
\\& \leq \frac{c v_3}{ p^0 } x_3 \| \p_{x_3} \p_{x_3} \Phi_h \|_{L^\infty(\bar{\O})}.
\end{split}
\ee
Using \eqref{est2:alpha_steady}, we obtain
\be \label{este:alpha_steady}
\begin{split}
\eqref{est1:alpha_steady} 
& \leq \frac{ v_3 x_3 + \frac{c v_3}{ p^0 } x_3 \| \p_{x_3} \p_{x_3} \Phi_h \|_{L^\infty(\bar{\O})}
+ \frac{(v_3)^3}{ c p^0 } ( e \p_{x_3} \Phi_h (x)+ m g ) 
+ \frac{(v_3)^2}{ c p^0 } v_{\|} \cdot e \p_{x_{\|}} \Phi_h (x) }{ \sqrt{ |x_3|^2  + |v_3|^2 +  2 (\p_{x_3} \Phi_h (x_\parallel , 0) + g ) \frac{c x_3}{ p^0 } } }
\\& \frac{+ \ v_\parallel \cdot \nabla_{x_\parallel} e \p_{x_3} \Phi_h (x_\parallel, 0) \frac{c x_3}{ p^0 }
+ ( e \p_{x_3} \Phi_h (x_\parallel, 0) + m g ) \frac{ x_3 }{ (p^0)^2} v \cdot \big( e \nabla_x \Phi_h (x) + \nabla_x (m g x_3) \big) }{*}
\\& \leq \frac{  (|x_3|^2  + |v_3|^2) \big( \frac{1}{2} + \frac{c }{2 p^0 } \| \p_{x_3} \p_{x_3} \Phi_h \|_{L^\infty(\bar{\O})} 
+ \frac{ 3 |v| }{c p^0 } ( e \| \nabla_{x} \Phi_h \|_{L^\infty(\p\O)} + m g ) \big)}{ \sqrt{ |x_3|^2  + |v_3|^2 +  2 (\p_{x_3} \Phi_h (x_\parallel , 0) + g ) \frac{c x_3}{ p^0 } } }
\\& \frac{ + \ 2 |v_\parallel| \|  \nabla_{x_\parallel} \p_{x_3} \Phi_h \|_{L^\infty(\p\O)} \frac{c x_3}{ p^0 }
+ ( e \p_{x_3} \Phi_h (x_\parallel, 0) + m g )  \frac{c x_3}{ p^0 } \big( \frac{ v }{ c p^0 } \cdot \nabla_x (e \Phi_h (x) + m g x_3) \big) }{*} .
\end{split}
\ee
Using \eqref{Uest:DPhi}, we have 
\Be\notag
\begin{split}
& \big| \big[ v \cdot \nabla_x + \big( e ( \frac{v}{c} \times B - \nabla_x \Phi_h ) - \nabla_x (m g x_3) \big) \cdot \nabla_p \big] \alpha (x, p) \big|
\\& \leq  \big( 4g + \frac{c}{ p^0 } \| \p_{x_3} \p_{x_3} \Phi_h \|_{L^\infty(\bar{\O})} + \frac{2 |v_\parallel| }{m g} \| \nabla_{x_\parallel} \p_{x_3} \Phi_h \|_{L^\infty(\p\O)} \big) \alpha (x,p) 
\\& \leq \big( 4g + \| \p_{x_3} \p_{x_3} \Phi_h \|_{L^\infty(\bar{\O})} + \|  \nabla_{x_\parallel} \p_{x_3} \Phi_h \|_{L^\infty(\p\O)} \big) \alpha (x,p). 
\end{split}
\Ee
Using the Gronwall's inequality, we conclude both inequalities of \eqref{est:alpha}. 

\smallskip

For \eqref{est1:alpha}, we set $t =  - \tb (x, p)$ in \eqref{est:alpha} and get 
\[
\alpha (\xb, \pb) = | \vbn (x,p)|.
\]
Thus we obtain the first inequality \eqref{est1:alpha}.
Then using \eqref{ODE_h}, \eqref{est1:alpha} and Proposition \ref{lem:tb}, we have 
\Be \notag
\tb (x, p) \leq \frac{4}{m g} |\pb (x, p)|,
\Ee
and therefore prove \eqref{est1:xb_x/w}. 
\end{proof}

\begin{lemma} \label{lem:XV_xv}
	
Suppose \eqref{ODE_h} holds for all $(t,x,p)$. Then for $1 \leq i, \ j \leq 3$,
\Be \label{XV_x}
\begin{split} 
\p_{x_i} X_{\pm, j} (t;x,p) 
& = \delta_{ij} + \int^t_0 \p_{x_i} P_{\pm} (s;x,p) \cdot \nabla_p V_{\pm, j} (s;x,p) \dd s 	 	
\\&	=  \delta_{ij} + \int^t_0 \int^s_0 \Big(
\p_{x_i} P_{\pm} (\tau;x,p) \cdot \nabla_p \mathcal{E}_{\pm, j} (\tau, X_{\pm} (\tau;x,p), P_{\pm} (\tau;x,p))  
\\& \qquad\qquad\qquad\quad + \p_{x_i} X_{\pm} (\tau;x,p) \cdot \nabla_x \mathcal{E}_{\pm, j} (\tau, X_{\pm} (\tau;x,p), P_{\pm} (\tau;x,p)) \Big) \dd \tau \dd s,
\\ \p_{x_i} P_{\pm, j} (t;x,p) 
& = \int^t_0 e \big( \p_{x_i} P_{\pm} (s;x,p) \cdot \nabla_p ( V_{\pm} (s;x,p) \times \frac{B}{c} )_{j} 
\\& \qquad\qquad\qquad - \p_{x_i} X_{\pm} (s;x,p) \cdot \nabla_x \p_{x_j} \Phi_h ( X_{\pm} (s;x,p)) \big) \dd s,
\end{split}
\Ee
and
\Be \label{XV_v}
\begin{split}
& \p_{p_i} X_{\pm, j} (t;x,p)   
\\& = \int^t_0 \p_{p_i} P_{\pm} (s;x,p) \cdot \nabla_p V_{\pm, j} (s;x,p) \dd s 	
\\&	= t \p_{p_i} v_{\pm, j} + \int^t_0 \int^s_0 \Big(
\p_{p_i} P_{\pm} (\tau;x,p) \cdot \nabla_p \mathcal{E}_{\pm, j} (\tau, X_{\pm} (\tau; x, p), P_{\pm} (\tau; x, p))
\\& \qquad\qquad\qquad\qquad\qquad + \p_{p_i} X_{\pm} (\tau;x,p) \cdot \nabla_x \mathcal{E}_{\pm, j} (\tau, X_{\pm} (\tau; x, p), P_{\pm} (\tau; x, p)) \Big) \dd \tau \dd s,
\\& \p_{p_i} P_{\pm, j} (t;x,p) 
= \delta_{ij} + \int^t_0 e \big( \p_{p_i} P_{\pm} (s;x,p) \cdot \nabla_p ( V_{\pm} (s;x,p) \times \frac{B}{c} )_{j} 
\\& \qquad\qquad\qquad\qquad\qquad\qquad - \p_{p_i} X_{\pm} (s;x,p) \cdot \nabla_x \p_{x_j} \Phi_h ( X_{\pm} (s;x,p)) \big) \dd s.
\end{split}
\Ee
Here, we have defined $\mathcal{E}_\pm  (s, X_\pm (s; x, p), P_\pm(s; x, p) )$ as
\Be \notag
\begin{split}
& \mathcal{E}_\pm  (s, X_\pm (s; x, p), P_\pm(s; x, p) ) 
:= \frac{1}{\sqrt{m_\pm^2 + |P_\pm(s;x,p)|^2 / c^2 }} \frac{\dd}{\dd s} P_\pm (s;x,p) 
\\& \qquad\qquad\qquad\qquad\qquad - \frac{ P_\pm (s;x,p)}{ c^2 \big( m_\pm^2 + |P_\pm(s;x,p)|^2 / c^2 \big)^{3/2} } \big( P_\pm (s;x,p) \cdot \frac{\dd}{\dd s} P_\pm (s;x,p) \big).
\end{split}\
\Ee
Moreover, Suppose the assumption \eqref{Uest:DPhi} holds. Then
\begin{align}
	|\nabla_{x} X_{\pm} (t;x,p)| & \leq \min \big\{  e^{\frac{t^2 + 2t}{2} (\| \nabla_x^2 \Phi  \|_\infty + e B_3 + m_{\pm} g ) }, e^{ (1 + B_3 + \| \nabla_x ^2 \Phi  \|_\infty)t} \big\}, \label{est:X_x} \\
	|\nabla_{x} P_{\pm} (t;x,p)| &\leq   \min \big\{ t ( B_3 + \| \nabla_x ^2 \Phi \|_\infty) e^{ (1+ B_3 + \| \nabla_x ^2 \Phi \|_\infty)t},
	e^{ (1+ B_3 + \| \nabla_x ^2 \Phi \|_\infty)t}
	\big\}, \label{est:V_x} \\
	|\nabla_{p} X_{\pm} (t;x,p)| &\leq   \min \big\{ t e^{ (1 + B_3 + \| \nabla_x ^2 \Phi  \|_\infty)t},
	e^{ (1 + B_3 + \| \nabla_x ^2 \Phi  \|_\infty)t}
	\big\}, \label{est:X_v} \\
	|\nabla_{p} P_{\pm} (t;x,p)| & \leq  \min \big\{ 1 + t (B_3 + \| \nabla_x ^2 \Phi \|_{\infty}) e^{(1 + B_3 + \|\nabla_x^2 \Phi \|_\infty) t},
	e^{ (1 + B_3 + \| \nabla_x ^2 \Phi  \|_\infty)t}
	\big\}. \label{est:V_v}
\end{align}
\end{lemma}

\begin{proof}

For the sake of simplicity, we abuse the notation as in \eqref{abuse}.
Recall from \eqref{ODE_h}, the characteristics $(X (s;x,p), P (s;x,p))$ follows
\Be \label{ODE_h_XV}
\begin{split} 
\frac{d X (s;x,p) }{d s} & = V (s;x,p) = \frac{P (s;x,p)}{\sqrt{m^2 + |P (s;x,p)|^2 / c^2}}, \\
\frac{d P (s;x,p) }{d s} & = e \big( V (s;x,p) \times \frac{B}{c} - \nabla_x \Phi_h (X (s;x,p) ) \big) - m g \mathbf{e}_3,
\end{split}
\Ee
Then we have   
\Be \label{form:XV}
\begin{split}
P (t;x,p) & = p + \int^t_0 \Big( e \big( V (s;x,p) \times \frac{B}{c} - \nabla_x  \Phi (X (s;x,p)) \big) - m g \mathbf{e}_3 \Big)  \dd s,
\\ X (t;x,p) & = x + \int^t_0 V (s;x,p) \dd s.
\end{split}
\Ee
Further, we compute that
\be \notag
\begin{split}
\frac{\dd}{\dd s} V (s;x,p) 
& = \frac{1}{\sqrt{m^2 + |P(s;x,p)|^2 / c^2 }} \frac{\dd}{\dd s} P (s;x,p) 
\\& \ \ \ \ - \frac{ P (s;x,p)}{ c^2 \big( m^2 + |P(s;x,p)|^2 / c^2 \big)^{3/2} } \big( P (s;x,p) \cdot \frac{\dd}{\dd s} P (s;x,p) \big)
\\& := \mathcal{E} (s, X(s; x, p), P(s; x, p) ). 
\end{split}
\ee
Thus, we rewrite \eqref{form:XV} as
\Be \label{form2:XV}
\begin{split}
X (t;x,p) 
& = x + v t + \int^t_0\int^s_0 \mathcal{E} (\tau, X (\tau; x, p), P (\tau; x, p)) \dd \tau \dd s,
\end{split}
\Ee
where $v = \frac{p}{\sqrt{m^2 + |p|^2 / c^2}}$.
From \eqref{form:XV} and \eqref{form2:XV}, we directly compute \eqref{XV_x} and \eqref{XV_v}. 

\smallskip

Next, from \eqref{ODE_h_XV}, together with $\p_{p_i} v = \frac{\delta_{ii}}{\sqrt{ m^2 + |p|^2/ c^2}} - \frac{v_i v}{c^2 \sqrt{ m^2 + |p|^2/ c^2}}$ , we have
\be \label{est1:diff_dvds}
\begin{split}
& \Big| \nabla_x \big( \frac{\dd P (s;x,p) }{\dd s} \big)\Big| 
= \big| e \nabla^2_x \Phi_h (X (s;x,p)) \big|,
\\& \Big| \nabla_p \big( \frac{\dd P (s;x,p) }{\dd s} \big) \Big| 
\lesssim \frac{ e B_3 }{ c \sqrt{ m^2 + |P (s;x,p)|^2/ c^2}}.
\end{split}
\ee
From \eqref{est1:diff_dvds}, by setting $\frac{P^0 (s;x,p)}{c} = \sqrt{m^2 + |P(s;x,p)|^2 / c^2 }$, we compute that
\be \label{est:E_x}
\begin{split}
& \Big| \nabla_x \Big( \mathcal{E} \big( s, X(s; x, p), P(s; x, p) \big) \Big) \Big|
\\& \lesssim \frac{1}{\langle P(s; x, p) \rangle} \big| \nabla^2_x \Phi_h (X (s;x,p)) \big| + \frac{\big| V (s;x,p) \big|^2}{ c P^0 (s;x,p) } \big| \nabla^2_x \Phi_h (X (s;x,p)) \big|
\\& \lesssim \frac{1}{ \frac{P^0 (s;x,p)}{c} } \big| \nabla^2_x \Phi_h (X (s;x,p)) \big|,
\end{split}
\ee
and 
\be \label{est:E_v}
\begin{split}
& \Big| \nabla_p \Big( \mathcal{E} \big( s, X(s; x, p), P(s; x, p) \big) \Big) \Big|
\\& \leq \Big| \nabla_p \big( \frac{1}{\frac{P^0 (s;x,p)}{c} } \big) \frac{\dd}{\dd s} P (s;x,p) \Big| 
+ \Big| \frac{1}{ \frac{P^0 (s;x,p)}{c} } \nabla_p \big( \frac{\dd}{\dd s} P (s;x,p) \big) \Big|
\\& \ \ \ \ + \Big| \nabla_p \big( \frac{1}{ c P^0 (s;x,p) } V (s;x,p) \cdot \frac{\dd}{\dd s} P (s;x,p) V (s;x,p) \big) \Big|
\\& \lesssim \frac{1}{ c \big( \frac{P^0 (s;x,p)}{c} \big)^2} \big( e ( B_3 + \big| \nabla_x \Phi_h (X (s;x,p)) \big| ) + m g \big).
\end{split}
\ee
For $\p_{x_i} X_{j}$ and $\p_{p_i} X_{j}$, we change the order of integrals in each last double integral to get
\Be \label{form:X_xv}
\begin{split}
\p_{x_i} X_{j} (t;x,p) 
& =  \delta_{ij} + \int^t_0 (t - \tau) \Big(
\p_{x_i} P (\tau;x,p) \cdot \nabla_p \mathcal{E}_j (\tau, X (\tau;x,p), P (\tau;x,p))  
\\& \qquad\qquad\qquad + \p_{x_i} X (\tau;x,p) \cdot \nabla_x \mathcal{E}_j (\tau, X(\tau;x,p), P (\tau;x,p))  \Big) \dd \tau, \\
\p_{p_i} X_{j} (t;x,p)   
& = t \p_{p_i} v_j + \int^t_0 (t - \tau) \Big(
\p_{p_i} P (\tau;x,p) \cdot \nabla_p \mathcal{E}_j (\tau, X (\tau; x, p), P (\tau; x, p))
\\& \qquad\qquad\qquad + \p_{p_i} X (\tau;x,p) \cdot \nabla_x \mathcal{E}_j (\tau, X (\tau; x, p), P (\tau; x, p)) \Big) \dd \tau.
\end{split}
\Ee  
From \eqref{est:E_x} and \eqref{est:E_v}, together with \eqref{Uest:DPhi} and the first line of \eqref{form:X_xv}, we get
\Be \label{est0:X_x}
\begin{split}
\big| \p_{x_i} X_{j} (t;x,p) \big|
& \leq 1 + \int^t_0 (t - \tau) \Big( \big| \p_{x_i} P (\tau;x,p) \big| \frac{e ( B_3 + \big| \nabla_x \Phi_h (X (s;x,p)) \big| ) + m g}{ c \big( \frac{P^0 (s;x,p)}{c} \big)^2}  
\\& \qquad\qquad\qquad\qquad\qquad + \big| \p_{x_i} X (\tau;x,p) \big| \frac{1}{ \frac{P^0 (s;x,p)}{c} } \| \nabla_x ^2 \Phi \|_{\infty} \Big) \dd \tau
\\& \lesssim 1 + \int^t_0 (t - \tau) \Big( \big| \p_{x_i} P (\tau;x,p) \big| \frac{e B_3 + m g}{ c m^2}
+ \big| \p_{x_i} X (\tau;x,p) \big| \frac{\| \nabla_x ^2 \Phi \|_{\infty}}{m}  \Big) \dd \tau.
\end{split}
\Ee  
From \eqref{XV_x}, we have
\Be \label{est0:V_x}
\begin{split} 
\big| \p_{x_i} P_j (t;x,p) \big|
& \leq e \int^t_0 \big| \p_{x_i} P (s;x,p) \big| \frac{B_3}{P^0 (s;x,p)} + \| \nabla_x ^2 \Phi \|_{\infty} \big| \p_{x_i} X (s;x,p) \big| \dd s
\\& \lesssim e \int^t_0 \big| \p_{x_i} P (s;x,p) \big| \frac{ B_3}{c m} + \| \nabla_x ^2 \Phi \|_{\infty} \big| \p_{x_i} X (s;x,p) \big| \dd s
\end{split}
\Ee
Using \eqref{est0:X_x} and \eqref{est0:V_x}, we obtain
\Be \label{est:X_x+V_x}
\begin{split}
& \big| \p_{x_i} X (t;x,p) \big| + \big| \p_{x_i} P (t;x,p) \big|
\\& \lesssim 1 + \int^t_0 (t - \tau) \Big( \big| \p_{x_i} P (\tau;x,p) \big| \frac{e B_3 + m g}{ c m^2}
+ \big| \p_{x_i} X (\tau;x,p) \big| \frac{\| \nabla_x ^2 \Phi \|_{\infty}}{m}  \Big) \dd \tau
\\& \qquad\qquad + e \int^t_0 \big| \p_{x_i} P (s;x,p) \big| \frac{ B_3}{c m} + \| \nabla_x ^2 \Phi \|_{\infty} \big| \p_{x_i} X (s;x,p) \big| \dd s
\\& \lesssim 1 + \int^t_0 (t - \tau + 1) (\| \nabla_x ^2 \Phi \|_{\infty} + e B_3 + mg)
\Big( \big| \p_{x_i} P (\tau;x,p) \big| + \big| \p_{x_i} X (\tau;x,p) \big| \Big) \dd \tau
\end{split}
\Ee  
Applying the Gronwall's inequality, we derive that 
\Be \label{est1:X_x+V_x}
\begin{split}
\big| \p_{x_i} X (t;x,p) \big| + \big| \p_{x_i} P (t;x,p) \big|
& \lesssim e^{\int^t_0 (t - \tau + 1) (\| \nabla_x ^2 \Phi \|_{\infty} + e B_3 + mg) \dd \tau}
\\& \lesssim e^{ \frac{t^2 + 2t}{2} (\| \nabla_x ^2 \Phi \|_{\infty} + e B_3 + mg)}.
\end{split}
\Ee 
From \eqref{est:E_x} and \eqref{est:E_v}, together with \eqref{Uest:DPhi} and the second line of \eqref{form:X_xv}, we get
\Be \label{est0:X_v}
\begin{split}
\big| \p_{p_i} X_{j} (t;x,p) \big|
& \leq \frac{t}{m} + \int^t_0 (t - \tau) \Big( \big| \p_{p_i} P (\tau;x,p) \big| \frac{e ( B_3 + \big| \nabla_x \Phi_h (X (s;x,p)) \big| ) + m g}{ c \big( \frac{P^0 (s;x,p)}{c} \big)^2}  
\\& \qquad\qquad\qquad\qquad\qquad + \big| \p_{p_i} X (\tau;x,p) \big| \frac{1}{ \frac{P^0 (s;x,p)}{c} } \| \nabla_x ^2 \Phi \|_{\infty} \Big) \dd \tau
\\& \lesssim \frac{t}{m} + \int^t_0 (t - \tau) \Big( \big| \p_{p_i} P (\tau;x,p) \big| \frac{e B_3 + m g}{ c m^2}
+ \big| \p_{p_i} X (\tau;x,p) \big| \frac{\| \nabla_x ^2 \Phi \|_{\infty}}{m}  \Big) \dd \tau.
\end{split}
\Ee  
From \eqref{XV_v}, we have
\Be \label{est0:V_v}
\begin{split} 
\big| \p_{p_i} P_j (t;x,p) \big|
& \leq 1+ e \int^t_0 \big| \p_{p_i} P (s;x,p) \big| \frac{B_3}{P^0 (s;x,p)} + \| \nabla_x ^2 \Phi \|_{\infty} \big| \p_{p_i} X (s;x,p) \big| \dd s
\\& \lesssim 1 + e \int^t_0 \big| \p_{p_i} P (s;x,p) \big| \frac{ B_3}{c m} + \| \nabla_x ^2 \Phi \|_{\infty} \big| \p_{p_i} X (s;x,p) \big| \dd s.
\end{split}
\Ee
Using \eqref{est0:X_v} and \eqref{est0:V_v}, we obtain
\Be \label{est:X_v+V_v}
\begin{split}
& \big| \p_{p_i} X (t;x,p) \big| + \big| \p_{p_i} P (t;x,p) \big|
\\& \lesssim (1 + \frac{t}{m} ) + \int^t_0 (t - \tau) \Big( \big| \p_{p_i} P (\tau;x,p) \big| \frac{e B_3 + m g}{ c m^2}
+ \big| \p_{p_i} X (\tau;x,p) \big| \frac{\| \nabla_x ^2 \Phi \|_{\infty}}{m}  \Big) \dd \tau
\\& \qquad\qquad + e \int^t_0 \big| \p_{p_i} P (s;x,p) \big| \frac{ B_3}{c m} + \| \nabla_x ^2 \Phi \|_{\infty} \big| \p_{p_i} X (s;x,p) \big| \dd s
\\& \lesssim (1 + t ) + \int^t_0 (t - \tau + 1) (\| \nabla_x ^2 \Phi \|_{\infty} + e B_3 + mg)
\Big( \big| \p_{x_i} P (\tau;x,p) \big| + \big| \p_{x_i} X (\tau;x,p) \big| \Big) \dd \tau.
\end{split}
\Ee  
Applying the Gronwall's inequality, we derive that 
\Be \label{est1:X_v+V_v}
\begin{split}
\big| \p_{p_i} X_{j} (t;x,p) \big| + \big| \p_{p_i} P_j (t;x,p) \big|
& \lesssim (1 + t) e^{\int^t_0 (t - \tau + 1) (\| \nabla_x ^2 \Phi \|_{\infty} + e B_3 + mg) \dd \tau}
\\& \lesssim (1 + t) e^{ \frac{t^2 + 2t}{2} (\| \nabla_x ^2 \Phi \|_{\infty} + e B_3 + mg)}.
\end{split}
\Ee 

\smallskip

On the other hand, using the first equality of $\p_{x_i} X_{j}$ and $\p_{x_i} P_{j}$ in \eqref{XV_x}, we derive 
\Be \notag
\begin{split} 
& \big| \p_{x_i} X (t;x,p) \big| + \big| \p_{x_i} P (t;x,p) \big|
\\& \leq 1 + \int^t_0 \frac{\big| \p_{x_i} P (s;x,p) \big|}{ m} \dd s 	
+ e \int^t_0 \big| \p_{x_i} P (s;x,p) \big| \frac{ B_3}{c m} + \| \nabla_x ^2 \Phi \|_{\infty} \big| \p_{x_i} X (s;x,p) \big| \dd s
\\& \lesssim 1 + \int^t_0 (1 + B_3 + \| \nabla_x ^2 \Phi \|_{\infty}) \Big( \big| \p_{x_i} X (t;x,p) \big| + \big| \p_{x_i} P (t;x,p) \big| \Big) \dd s.
\end{split}
\Ee
Using the Gronwall's inequality, we derive that  
\Be \label{est2:X_x+V_x}
\big| \p_{x_i} X (t;x,p) \big| + \big| \p_{x_i} P (t;x,p) \big| \leq e^{(1 + B_3 + \|\nabla_x^2 \Phi \|_\infty) t}.
\Ee
From \eqref{est0:V_x} and \eqref{est2:X_x+V_x}, we have
\Be \label{est1:V_x}
\big| \p_{x_i} P_j (t;x,p) \big|
\lesssim t (B_3 + \| \nabla_x ^2 \Phi \|_{\infty}) e^{(1 + B_3 + \|\nabla_x^2 \Phi \|_\infty) t}.
\Ee
Similarly, using the first equality of $\p_{p_i} X_{j}$ and $\p_{p_i} P_{j}$ in \eqref{XV_v}, we derive 
\Be \notag
\begin{split} 
& \big| \p_{p_i} X (t;x,p) \big| + \big| \p_{p_i} P (t;x,p) \big|
\\& \leq 1 + \int^t_0 \frac{\big| \p_{p_i} P (s;x,p) \big|}{ m} \dd s + e \int^t_0 \big| \p_{p_i} P (s;x,p) \big| \frac{ B_3}{c m} + \| \nabla_x ^2 \Phi \|_{\infty} \big| \p_{p_i} X (s;x,p) \big| \dd s
\\& \lesssim 1 + \int^t_0 (1 + B_3 + \| \nabla_x ^2 \Phi \|_{\infty}) \Big( \big| \p_{p_i} X (t;x,p) \big| + \big| \p_{p_i} P (t;x,p) \big| \Big) \dd s.
\end{split}
\Ee
Using the Gronwall's inequality, we derive that  
\Be \label{est2:X_v+V_v}
\big| \p_{p_i} X (t;x,p) \big| + \big| \p_{p_i} P (t;x,p) \big| \leq e^{(1 + B_3 + \|\nabla_x^2 \Phi \|_\infty) t}.
\Ee
From \eqref{est0:V_v} and \eqref{est2:X_v+V_v}, we have
\Be \label{est2:V_x}
\big| \p_{p_i} P_j (t;x,p) \big|
\lesssim 1 + t (B_3 + \| \nabla_x ^2 \Phi \|_{\infty}) e^{(1 + B_3 + \|\nabla_x^2 \Phi \|_\infty) t}.
\Ee
Finally, combining \eqref{est2:X_x+V_x}, \eqref{est1:V_x}, \eqref{est2:X_v+V_v}, and \eqref{est2:V_x}, we conclude \eqref{est:X_x}-\eqref{est:V_v}.
\end{proof}

\begin{lemma} \label{lem:exp_txvb}

Recall $(\tbp (x,p), \xbp (x,p), \pbp (x,p))$ in Definition \ref{def:tb}. Then
\begin{align}
\p_{x_i} \tbp (x,p) 
& = \frac{ \p_{x_i} X_{\pm, 3} (- \tbp (x,p) ; x,p )}{ \vbpn (x,p) } \notag \\
& = \frac{1}{ \vbpn (x,p)} \Big( \delta_{i3} + \int^{- \tbp (x,p)}_0 \p_{x_i} P_{\pm} (s;x,p) \cdot \nabla_p V_{\pm, 3} (s;x,p) \dd s  \Big), \label{tb_x} \\[5pt]
\p_{p_i} \tbp (x,p)  
& = \frac{\p_{p_i} X_{\pm, 3} (-\tbp (x,p) ; x,p) }{ \vbpn (x,p) } \notag \\
& = \frac{1}{ \vbpn (x,p)}
\int^{-\tbp (x,p)}_0 \p_{p_i} P_{\pm} (s;x,p) \cdot \nabla_p V_{\pm, 3} (s;x,p) \dd s, \label{tb_v}
\end{align}
\Be \label{xb_x}
\p_{x_i} \xbp (x,p) = - \frac{ \p_{x_i} X_{\pm, 3} ( - \tbp  (x,p); x,p) }{ \vbpn (x,p) } \vbp (x,p)+ \p_{x_i} X_{\pm} ( - \tbp (x,p); x,p),
\Ee
\Be \label{xb_v}
\p_{p_i} \xbp (x,p) = - \frac{ \p_{p_i} X_{\pm, 3} ( - \tbp  (x,p);x,p) }{ \vbpn (x,p)}  \vbp (x,p) + \p_{p_i} X_{\pm} ( - \tbp (x,p);x,p),
\Ee 
and
\Be \label{vb_x}
\begin{split}
& \p_{x_i} \pbp (x,p) 
\\& =  - \frac{ \p_{x_i} X_{\pm, 3} ( - \tbp (x,p);x,p) }{ \vbpn (x,p) }
\Big( e \big( \vbp (x,p) \times \frac{B}{c} - \nabla_x \big( \Phi_h (\xbp (x,p) ) \big) - m_{\pm} g \mathbf{e}_3 \Big)
\\& \qquad\qquad + \p_{x_i} P_{\pm} ( - \tbp (x,p); x,p),
\end{split}
\Ee
\Be \label{vb_v}
\begin{split}
& \p_{p_i} \pbp (x,p) 
\\& =  - \frac{ \p_{p_i} X_{\pm, 3} ( - \tbp (x,p);x,p) }{ \vbpn (x,p) }
\Big( e \big( \vbp (x,p) \times \frac{B}{c} - \nabla_x \big( \Phi_h (\xbp (x,p) ) \big) - m_{\pm} g \mathbf{e}_3 \Big)
\\& \qquad\qquad + \p_{p_i} P_{\pm} ( - \tbp (x,p); x,p).
\end{split}
\Ee 	
\end{lemma}

\begin{proof}

For the sake of simplicity, we abuse the notation as in \eqref{abuse}.
Recall from Definition \ref{def:tb}, $X_3 (-\tb (x,p); x, p)=0$. By taking the derivatives on $X_3 (-\tb (x,p); x, p)$, together with \eqref{ODE_h}, we obtain that
\Be \notag
- \p_{x_i} \tb (x, p) V_3 (- \tb (x,p); x,p) + \p_{x_i} X_3 (- \tb (x,p) ; x,p ) = 0,
\Ee
and
\be \notag
- \p_{p_i} \tb(x,p) V_3 (- \tb (x,p); x,p) + \p_{p_i} X_3 (- \tb (x,p) ; x,p ) = 0.
\ee
This shows that
\Be \notag
\begin{split}
& - \p_{x_i} \tb (x, p) \vbn (x,p) + \p_{x_i} X_3 (- \tb (x,p); x,p ) = 0,
\\& - \p_{p_i} \tb(x,p) \vbn (x,p) + \p_{p_i} X_3 (- \tb (x,p); x,p ) = 0.
\end{split}
\ee
Now using the identities of $\p_{x_i} X_{j} (t;x,p) $ and $\p_{p_i} X_{j} (t;x,p)$ in \eqref{XV_x} and \eqref{XV_v}, we conclude $\eqref{tb_x}$ and $\eqref{tb_v}$.

\smallskip

From \eqref{def:zb^h}, $\xb (x,p) = X ( -\tb (x, p); x, p)$. Taking the derivatives on $\xb (x,p)$, together with \eqref{ODE_h}, we derive that
\Be \notag
\p_{x_i} x_{\b} = - \p_{x_i} \tb  v_{\b} + \p_{x_i} X  ( - \tb ; x,p),
\Ee 
and
\Be \notag
\p_{p_i} x_{\b} = - \p_{p_i} \tb  v_{\b}  + \p_{p_i} X  ( - \tb ; x,p).
\Ee 
Using the identities in \eqref{tb_x} and \eqref{tb_v}, we conclude \eqref{xb_x} and \eqref{xb_v}. 

\smallskip

Again from \eqref{def:zb^h}, $\pb (x,p) = P ( -\tb (x, p); x, p)$. Taking the derivatives on $\pb (x,p)$, together with \eqref{ODE_h}, we get 
\Be \notag
\p_{x_i} \pb  (x,p) = - \p_{x_i} \tb  (x,p) \dot{P} (- \tb  (x,p);x,p) + \p_{x_i} P ( - \tb  (x,p); x,p),
\Ee
and
\Be \notag
\p_{p_i} \pb  (x,p) = - \p_{p_i} \tb  (x,p) \dot{P} (- \tb  (x,p);x,p) + \p_{p_i} P ( - \tb  (x,p); x,p).
\Ee
Finally, using the identities of $\p_{x_i} P_{j} (t;x,p) $ and $\p_{p_i} P_{j} (t;x,p)$ in \eqref{XV_x} and \eqref{XV_v}, we conclude \eqref{vb_x} and \eqref{vb_v}. 
\end{proof}

\begin{lemma} \label{lem:nabla_zb}

Suppose the assumption \eqref{Uest:DPhi} holds. Recall $(\tbp (x,p), \xbp (x,p), \pbp (x,p))$ in Definition \ref{def:tb}. Then
\Be \label{est:xb_x}
\begin{split}
&| \p_{x_i} \xbp (x,p) | 
\leq \frac{ | \vbp (x,p) | }{ | \vbpn (x,p) | } \delta_{i3} 
\\& \ \ \ \ + \Big(1 +  \frac{|\vbp (x,p)|}{|\vbpn (x,p)|}
\frac{|\tbp (x,p)|^2}{2} ( \| \nabla_x^2 \Phi  \|_\infty + e B_3 + m_{\pm} g ) \Big) \times e^{ (1 + B_3 + \| \nabla_x ^2 \Phi  \|_\infty) \tbp},
\end{split}
\Ee
\Be \label{est:xb_v}
\begin{split}
& |\p_{p_i} \xbp (x,p)| 
\leq \frac{|\vbp  (x,p)| |\tbp (x,p)|  }{|\vbpn  (x,p)| \sqrt{ (m_{\pm})^2 + |p|^2/ c^2}}  
\\& \ \ \ \ + \Big(1 +   \frac{|\vbp (x,p)|}{|\vbpn (x,p)|}
\frac{|\tbp (x,p)|^2}{2} ( \| \nabla_x^2 \Phi  \|_\infty + e B_3 + m_{\pm} g ) \Big) \times e^{ (1 + B_3 + \| \nabla_x ^2 \Phi  \|_\infty) \tbp},
\end{split}
\Ee
\Be \label{est:vb_x}
\begin{split}
& |\p_{x_i} \pbp (x,p)| \leq \frac{ |e B_3 + m_{\pm} g| }{|\vbpn  (x,p)|} \delta_{i3}
\\& \ \ \ \ + \Big(1 +  \frac{|e B_3 + m_{\pm} g|}{|\vbpn (x,p)|}
\frac{|\tbp (x,p)|^2}{2} ( \| \nabla_x^2 \Phi  \|_\infty + e B_3 + m_{\pm} g ) \Big) \times e^{ (1 + B_3 + \| \nabla_x ^2 \Phi  \|_\infty) \tbp},
\end{split}
\Ee
\Be \label{est:vb_v}
\begin{split}
& |\p_{p_i} \pbp (x,p)| \leq \frac{|e B_3 + m_{\pm} g| |\tbp (x,p)|  }{|\vbpn  (x,p)| \sqrt{ (m_{\pm})^2 + |p|^2/ c^2}}  
\\& \ \ \ \ + \Big(1 +   \frac{|e B_3 + m_{\pm} g|}{|\vbpn (x,p)|}
\frac{|\tbp (x,p)|^2}{2} ( \| \nabla_x^2 \Phi \|_\infty + e B_3 + m_{\pm} g ) \Big) \times e^{ (1 + B_3 + \| \nabla_x ^2 \Phi  \|_\infty) \tbp}.
\end{split}
\Ee 
\end{lemma}

\begin{proof}	

For the sake of simplicity, we abuse the notation as in \eqref{abuse}. 
From \eqref{XV_x} in Lemma \ref{lem:XV_xv}, we have 
\be \notag
\begin{split}
\p_{x_i} X_{ 3} (- \tb (x,p) ; x,p )
& = \delta_{i3} + \int^{- \tb (x,p)}_0 \int^s_0 \Big(
\p_{x_i} P (\tau;x,p) \cdot \nabla_p \mathcal{E}_3 (\tau, X (\tau;x,p), P (\tau;x,p))  
\\& \qquad\qquad\qquad + \p_{x_i} X (\tau;x,p) \cdot \nabla_x \mathcal{E}_3 (\tau, X(\tau;x,p), P (\tau;x,p))  \Big) \dd \tau \dd s.
\end{split}
\Ee
Using \eqref{est:E_x} and \eqref{est:E_v}, together with \eqref{Uest:DPhi}, we get
\be \label{est:X3x}
\begin{split}
\big| \p_{x_i} X_{ 3} (- \tb (x,p) ; x,p ) \big|
& \leq \delta_{i3} + \frac{|\tb (x,p)|^2}{2}
( \| \nabla_x^2 \Phi  \|_\infty + e B_3 + mg ) e^{ (1 + B_3 + \| \nabla_x ^2 \Phi  \|_\infty) \tb}.
\end{split}
\Ee
Now using Lemma \ref{lem:exp_txvb}, we have
\Be \notag
\begin{split}
| \p_{x_i} \xb (x,p) | 
& = \Big| - \frac{ \p_{x_i} X _{  3} ( - \tb  (x,p);x,p) }{v_{\b, 3} (x,p) } \vb (x,p)+ \p_{x_i} X ( - \tb (x,p) ;x,p) \Big|
\\& \leq \Big| \frac{ \p_{x_i} X _{  3} ( - \tb  (x,p);x,p) }{v_{\b, 3} (x,p) } \vb (x,p) \Big| + \big| \p_{x_i} X ( - \tb (x,p) ;x,p) \big|
\end{split}
\Ee
From \eqref{est:X3x}, together with \eqref{est:X_x} and \eqref{est:V_x}, we conclude \eqref{est:xb_x} by
\Be \notag
\begin{split}
| \p_{x_i} \xb (x,p) | 
& \leq \frac{ | \vb (x,p) | }{ | v_{\b, 3} (x,p) | }   \Big( \delta_{i3} + \frac{|\tb (x,p)|^2}{2}
( \| \nabla_x^2 \Phi  \|_\infty + e B_3 + mg )
 \times e^{ (1 + B_3 + \| \nabla_x ^2 \Phi  \|_\infty) \tb} \Big) 
\\& \qquad + e^{ (1 + B_3 + \| \nabla_x ^2 \Phi  \|_\infty) \tb}.
\end{split}
\Ee
We skip the rest of proof, since \eqref{est:xb_v}-\eqref{est:vb_v} follow similarly from Lemma \ref{lem:XV_xv} and \ref{lem:exp_txvb}.
\end{proof}

\begin{prop} \label{prop:Reg}

Suppose $(h_{\pm}, \rho , \nabla_x \Phi )$ solve \eqref{VP_h}-\eqref{eqtn:Dphi} in the sense of Definition \ref{weak_sol}. Suppose the condition \eqref{Uest:DPhi} holds. 
Assume $\| e^{{\tilde \beta } \sqrt{(m_{\pm} c)^2 + |p|^2}} \nabla_{x_\parallel, p} G_{\pm} (x,p) \|_{L^\infty (\gamma_-)} < \infty$ holds for $\tilde \beta >0$, and let $\hat{m} = \min ( m_{+}, m_{-} )$ satisfy
\Be \label{choice_beta}
\frac{8}{ \hat{m} g} (1 + B_3 + \| \nabla_x ^2 \Phi  \|_\infty) \leq \tilde \beta.
\Ee
Then, for  $(x, p) \in \bar \O \times \R^3$, 
\Be \label{est:rho_x}
\begin{split}
e^{ \frac{\tilde \beta \hat{m} g}{4 c} x_3 } |\nabla_{x_i} \rho  (x)|   
& \lesssim \| e^{\tilde \beta \sqrt{(m_+ c)^2 + |p|^2}} \nabla_{x_\parallel, p} G_+ \|_{L^\infty (\gamma_-)}  
\times \Big(
1 + \mathbf{1}_{|x_3| \leq 1} \frac{1}{\sqrt{ m_+ g x_3 }}
\Big)
\\& \ \ \ \ + \| e^{\tilde \beta \sqrt{(m_- c)^2 + |p|^2}} \nabla_{x_\parallel, p} G_- \|_{L^\infty (\gamma_-)}  
\times \Big(
1 + \mathbf{1}_{|x_3| \leq 1} \frac{1}{\sqrt{ m_- g x_3 }}
\Big),
\end{split}
\Ee
and
\Be \label{est:phi_C2}
\begin{split}
\| \nabla_x^2 \Phi  \|_\infty
& \lesssim \frac{1}{\beta} \big( e_+ \| w_{+, \beta} G_+ \|_{L^\infty(\gamma_-)} + e_{-} \| w_{-, \beta} G_- \|_{L^\infty(\gamma_-)} \big)
\\& \ \ \ \ + \sum\limits_{i = \pm} \| e^{\tilde \beta \sqrt{(m_i c)^2 + |p|^2}} \nabla_{x_\parallel, p} G_i \|_{L^\infty (\gamma_-)}.
\end{split}	
\Ee
Moreover, 
\Be \label{est:hk_v}
\begin{split}
& e^{ \frac{\tilde \beta}{2}|p^0_{\pm}|} e^{  \frac{\tilde \beta m_{\pm} g}{4 c} x_3} | \nabla_p h_{\pm} (x,p)| 
\\& \lesssim \big( 1 + \frac{ \| \nabla_x^2 \Phi  \|_\infty + e B_3 + m_{\pm} g}{(m_{\pm} g)^2} \big) \| e^{\tilde \beta \sqrt{(m_{\pm} c)^2 + |p|^2}} \nabla_{x_\parallel, p} G_{\pm} \|_{L^\infty (\gamma_-)},
\end{split}
\Ee
and
\Be \label{est:hk_x}
\begin{split}
& e^{ \frac{\tilde \beta}{2}|p^0_{\pm}|} e^{  \frac{\tilde \beta m_{\pm} g}{4 c} x_3} | \nabla_x h_{\pm} (x,p)|  
\\& \lesssim \big( \frac{\delta_{i3}}{\alpha_{\pm} (x,p)} + \frac{ \| \nabla_x^2 \Phi  \|_\infty + e B_3 + m_{\pm} g}{(m_{\pm} g)^2} \big) \| e^{\tilde \beta \sqrt{(m_{\pm} c)^2 + |p|^2}} \nabla_{x_\parallel, p} G_{\pm} \|_{L^\infty (\gamma_-)}.
\end{split}
\Ee
\end{prop}

\begin{proof}	

For the sake of simplicity, we abuse the notation as in \eqref{abuse}.
From \eqref{form:h^k_G}, we get $h (x,p) = G (\xb, \pb)$. By taking the derivative on $h (x,p)$, we have 
\Be \label{form:nabla_h}
\nabla_{x, p} h (x,p) = \nabla_{x, p} \xb (x,p) \cdot \nabla_{x_\parallel} G (\xb, \pb) + \nabla_{x, p} \pb (x,p) \cdot \nabla_p G (\xb, \pb).
\Ee
Here, for simplicity, we 
use the notation: $p_\pm^0 = \sqrt{(m_\pm c)^2 + |p|^2}$ in the rest of proof.

\smallskip

\textbf{Step 1. Proof of \eqref{est:rho_x}.}
From \eqref{def:rho} and \eqref{form:nabla_h}, we have
\Be \label{est1:rho_x}
\begin{split}
& \ \ \ \ | \nabla_{x_i} \rho (x) | 
\\& = \Big| \int_{\R^3} e_+ \nabla_{x_i} h_+ (x,p) + e_{-} \nabla_{x_i} h_{-} (x,p) \dd p \Big| 
\\& \leq \sum_{j = \pm} e_{j} \Big|\int_{\R^3} \nabla_{x_i} \xbj \cdot \nabla_{x_\parallel} G_j (\xbj, \pbj)  + \nabla_{x_i} \pbj \cdot \nabla_p G_j (\xbj, \pbj) \dd p \Big| 
\\& \leq \sum_{j = \pm} e_{j} \bigg\{ \| e^{\tilde \beta |p^0_j|} \nabla_{x_\parallel} G_j \|_{L^\infty (\gamma_-)}  
\times \int_{\R^3} \frac{| \nabla_{x_i} \xbj (x,p) |}{ e^{\tilde \beta |\pbj^0 (x,p)| }} \dd p  
\\& \qquad\qquad\qquad\qquad\qquad + \| e^{\tilde \beta |p^0_j|} \nabla_{p} G_j \|_{L^\infty (\gamma_-)}  
\times \int_{\R^3} \frac{| \nabla_{x_i} \pbj (x,p) |}{ e^{\tilde \beta |\pbj^0 (x,p)| }} \dd p \bigg\}
\\& \leq \sum_{j = \pm} e_{j} \| e^{\tilde \beta |p^0_j|} \nabla_{x_\parallel, p} G_j \|_{L^\infty (\gamma_-)}  
\times \bigg\{ 
\underbrace{ \int_{\R^3} \frac{| \nabla_{x_i} \xbj (x,p) |}{ e^{\tilde \beta |\pbj^0 (x,p)| }} \dd p }_{\eqref{est1:rho_x}_1}
 + \underbrace{ \int_{\R^3} \frac{| \nabla_{x_i} \pbj (x,p) |}{ e^{\tilde \beta |\pbj^0 (x,p)| }} \dd p }_{\eqref{est1:rho_x}_2}
 \bigg\}.
\end{split}
\Ee 

First, we bound $\eqref{est1:rho_x}_1$. From $\tb (x,p) \leq \frac{4}{m g} |\pb (x, p)|$ in \eqref{est:tb^h}, together with \eqref{est:tbpb3} and \eqref{est:xb_x}, we derive that
\begin{align}
& \frac{| \nabla_{x_i} \xb (x,p) |}{ e^{\tilde \beta |\pb^0 (x,p)| } }
\leq \frac{ | \vb (x,p) | }{ | v_{\b, 3} (x,p) | } \delta_{i3} e^{- \tilde{\beta} |\pb^0 (x,p)| }
\label{est:xb_x/w1}
\\& \ \ \ \ + \Big(1 +  \frac{|\vb (x,p)|}{|\vbn (x,p)|}
\frac{2 |\pb (x,p)|}{m g} \Big( \frac{2}{m g}  + \frac{\sqrt{8} \sqrt[4]{(m c)^2 + |\pb (s;x,p)|^2}}{\sqrt{ c m^3 g^2}} \Big) | \pbn (x, p) |  \notag
\\& \qquad\qquad\qquad \times ( \| \nabla_x^2 \Phi  \|_\infty + e B_3 + mg ) \Big) e^{ \frac{4}{m g} (1 + B_3 + \| \nabla_x ^2 \Phi  \|_\infty) |\pb|} e^{- \tilde{\beta} |\pb^0 | }.
\label{est:xb_x/w2}
\end{align}	
Using \eqref{est1:xb_x/w} and $\tilde{\beta}$ in \eqref{choice_beta}, we get that
\Be \label{est2:xb_x/w}
\begin{split}
\eqref{est:xb_x/w1} 
& \leq \frac{ \delta_{i3} }{ {\alpha (x,p) } } e^{- \frac{\tilde \beta}{2} |\pb^0 (x,p)| } 
\leq \frac{ \delta_{i3} }{\alpha (x,p)} e^{ - \frac{\tilde \beta}{2} \big( |p^0| + \frac{m g}{2 c} x_3 \big) },
\\ \eqref{est:xb_x/w2} 
& \lesssim e^{- \frac{\tilde \beta}{2} |\pb^0 (x,p)| }
+ ( \| \nabla_x^2 \Phi  \|_\infty + e B_3 + mg ) e^{- \frac{\tilde \beta}{2} |\pb^0 (x,p)| }
\lesssim e^{- \frac{\tilde \beta}{2} \big( |p^0| + \frac{m g}{2 c} x_3 \big)},
\end{split}
\Ee
where we use the bound from the condition \eqref{Uest:DPhi} and Lemma \ref{lem:conservation_law} as follows:
\be \label{lower:vb}
\pb^0 (x,p) = \sqrt{(m_{\pm} c)^2 + |\pb (x,p)|^2} \geq \sqrt{(m_{\pm} c)^2 + |p|^2} + \frac{m g}{2 c} x_3.
\ee
Recall from $\alpha (x, p)$ in \eqref{alpha}, together with the condition \eqref{Uest:DPhi}, we compute
\Be \label{est:1/alpha}
\begin{split}
& \int_{\R^3} \frac{1}{\alpha (x,p)} e^{ - \frac{\tilde \beta}{2} |p^0| } \dd p
\\& = \int_{\R^3} \frac{1}{\sqrt{ |x_3|^2  + |v_{3} |^2 +  2 ( e \p_{x_3} \Phi_h (x_\parallel , 0) + m g ) \frac{c x_3}{ p^0 } }} e^{ - \frac{\tilde \beta}{2} |p^0| } \dd p 
\\& = \mathbf{1}_{|x_3| > 1} \int_{\R^3} \frac{1}{\sqrt{ |x_3|^2  + |v_{3} |^2 +  2 ( e \p_{x_3} \Phi_h (x_\parallel , 0) + m g ) \frac{c x_3}{ p^0 } }} e^{ - \frac{\tilde \beta}{2} |p^0| } \dd p  \\& \ \ \ \ +  \mathbf{1}_{|x_3| \leq 1} \int_{\R^3} \frac{1}{\sqrt{ |x_3|^2  + |v_{3} |^2 +  2 ( e \p_{x_3} \Phi_h (x_\parallel , 0) + m g ) \frac{c x_3}{ p^0 } }} e^{ - \frac{\tilde \beta}{2} |p^0| } \dd p 
\\& \leq \mathbf{1}_{|x_3| > 1} \int_{\R^3} \frac{e^{ - \frac{\tilde \beta}{2} |p^0| }}{\sqrt{ |x_3|^2}} \dd p + \mathbf{1}_{|x_3| \leq 1} \int_{\R^3} \frac{e^{ - \frac{\tilde \beta}{2} |p^0| }}{\sqrt{ |x_3|^2  + |v_{3} |^2 + m g \frac{c x_3}{ p^0 } }}  \dd p 
\\& \lesssim \mathbf{1}_{|x_3| > 1} + \mathbf{1}_{|x_3| \leq 1} \int_{\R^3} \frac{e^{ - \frac{\tilde \beta}{2} |p^0| } \sqrt{ \frac{p^0}{c} } }{\sqrt{ m g x_3 }} \dd p 
\lesssim 1 + \mathbf{1}_{|x_3| \leq 1} \frac{1}{\sqrt{ m g x_3 }}.
\end{split}
\Ee
Combining \eqref{est:1/alpha} with \eqref{est2:xb_x/w}, we obtain that
\Be \label{est3:xb_x/w}
\begin{split}
\int_{\R^3} \eqref{est:xb_x/w1} \dd p 
& \leq \int_{\R^3} \frac{1}{\alpha (x,p)} e^{ - \frac{\tilde \beta}{2} \big( |p^0| + \frac{m g}{2 c} x_3 \big) } \dd p 
\\& \leq e^{ - \frac{\tilde \beta m g}{4 c} x_3 } \int_{\R^3} \frac{1}{\alpha (x,p)} e^{ - \frac{\tilde \beta}{2} |p^0| } \dd p
\\& \lesssim e^{ - \frac{\tilde \beta m g}{4 c} x_3 } \Big(
1 + \mathbf{1}_{|x_3| \leq 1} \frac{1}{\sqrt{ m g x_3 }}
\Big),
\end{split}
\Ee
and
\Be \label{est4:xb_x/w}
\begin{split}
\int_{\R^3} \eqref{est:xb_x/w2} \dd p 
& \leq \int_{\R^3} e^{- \frac{\tilde \beta}{2} \big( |p^0| + \frac{m g}{2 c} x_3 \big)} \dd p
\\& \leq e^{ - \frac{\tilde \beta m g}{4 c} x_3 } \int_{\R^3} e^{ - \frac{\tilde \beta}{2} |p^0| } \dd p
\lesssim e^{ - \frac{\tilde \beta m g}{4 c} x_3 }.
\end{split}
\Ee
From \eqref{est3:xb_x/w} and \eqref{est4:xb_x/w}, we derive that
\be \label{est5:xb_x/w}
\eqref{est1:rho_x}_1 =
\int_{\R^3} \frac{| \nabla_x \xb (x,p) |}{ e^{\tilde \beta |\pb^0 (x,p)| }} \dd p
\lesssim e^{ - \frac{\tilde \beta m g}{4 c} x_3 } \Big(
1 + \mathbf{1}_{|x_3| \leq 1} \frac{ \delta_{i3} }{\sqrt{ m g x_3 }}
\Big).
\ee

Second, we bound $\eqref{est1:rho_x}_2$. From $\tb (x,p) \leq \frac{4}{m g} |\pb (x, p)|$ in \eqref{est:tb^h}, together with \eqref{est:vb_x}, we derive that
\be \label{est:vb_x/w}
\begin{split}
& \frac{| \nabla_{x_i} \pb (x,p) |}{ e^{\tilde \beta |\pb^0 (x,p)| }} \leq \frac{ |e B_3 + m g| }{|v_{\b, 3}  (x,p)|}  \delta_{i3} e^{- \tilde{\beta} |\pb^0 (x,p)| }
\\& \ \ \ \ + \Big(1 +  \frac{|e B_3 + m g|}{|\vbn (x,p)|}
\frac{8 |\pb (x,p)|^2}{(m g)^2} ( \| \nabla_x^2 \Phi  \|_\infty + e B_3 + mg ) \Big) e^{ \frac{4}{m g} (1 + B_3 + \| \nabla_x ^2 \Phi  \|_\infty) |\pb| } e^{- \tilde{\beta} |\pb^0| }.
\end{split}
\ee
Similarly, from \eqref{est1:xb_x/w}, $\tilde{\beta}$ in \eqref{choice_beta}, together with \eqref{est:tbpb3} and \eqref{lower:vb}, we get that
\Be \notag
\eqref{est:vb_x/w} 
\lesssim \frac{\delta_{i3}}{\alpha (x,p)} e^{ - \frac{\tilde \beta}{2} \big( |p^0| + \frac{m g}{2 c} x_3 \big) }
+ e^{ - \frac{\tilde \beta}{2} \big( |p^0| + \frac{m g}{2 c} x_3 \big) }.
\Ee
Using the same argument of \eqref{est3:xb_x/w}-\eqref{est5:xb_x/w}, we derive that 
\Be \label{est1:vb_x/w}
\eqref{est1:rho_x}_2 = \int_{\R^3} \frac{| \nabla_{x_i} \pb (x,p) |}{ e^{\tilde \beta |\pb^0 (x,p)| }} \dd p
\lesssim e^{ - \frac{\tilde \beta m g}{4 c} x_3 } \Big(
1 + \mathbf{1}_{|x_3| \leq 1} \frac{\delta_{i3}}{\sqrt{ m g x_3 }}
\Big).
\Ee

Inputting \eqref{est5:xb_x/w} and \eqref{est1:vb_x/w} into \eqref{est1:rho_x}, we have
\Be \label{est2:rho_x}
\begin{split}
| \nabla_{x_i} \rho (x) | 
\lesssim \sum_{j = \pm} \| e^{\tilde \beta |p^0_j|^2} \nabla_{x_\parallel, p} G_j \|_{L^\infty (\gamma_-)}  
\times e^{ - \frac{\tilde \beta m_j g}{4 c} x_3 } \Big(
1 + \mathbf{1}_{|x_3| \leq 1} \frac{\delta_{i3}}{\sqrt{ m_j g x_3 }}
\Big).
\end{split}
\Ee 
and conclude \eqref{est:rho_x}.

\smallskip

\textbf{Step 2. Proof of \eqref{est:phi_C2}.}
In order to apply \eqref{est:nabla^2phi}, we start with $|\rho |_{C^{0,\delta}(\O)}$. 
Given any $x \in \O$, consider $0<h<1$ such that $x \pm h \mathbf{e}_i \in \O$ with $1 \leq i \leq 3$.
Using \eqref{est:rho_x}, we get
\Be \label{rho_DQ}
\begin{split}
& \frac{|\rho(x \pm h \mathbf{e}_i) - \rho(x)|}{h^\delta} 
\\& \leq \frac{1}{h^\delta} \int^h_0 | \nabla_{x_i} \rho (x \pm \tau \mathbf{e}_i) | \dd \tau 
\\& \leq \sum_{i = \pm} \| e^{\tilde \beta |p^0|} \nabla_{x_\parallel, p} G_i \|_{L^\infty (\gamma_-)} \times \frac{1}{h^\delta} \int^h_0 \big( 1 + \mathbf{1}_{|x_3| \leq 1} \frac{\delta_{i3}}{\sqrt{ m g (x \pm \tau \mathbf{e}_i)_3 }} \big) \dd \tau.
\end{split}
\Ee

For $i = 1, 2$, from $\delta_{i3} = 0$ and \eqref{rho_DQ}, we have
\Be \label{est:rho_DQ_12}
\begin{split}
\frac{1}{h^\delta} \int^h_0 \big( 1 + \mathbf{1}_{|x_3| \leq 1} \frac{\delta_{i3}}{\sqrt{ m g (x \pm \tau \mathbf{e}_i)_3 }} \big) \dd \tau
& \leq h^{1-\delta} \lesssim_{\delta} 1.
\end{split}
\Ee

For $i = 3$, from \eqref{rho_DQ} we obtain
\Be \label{est:rho_DQ_3}
\begin{split}
\frac{1}{h^\delta} \int^h_0 \big( 1 + \mathbf{1}_{|x_3| \leq 1} \frac{1}{\sqrt{ m g (x \pm \tau \mathbf{e}_i)_3 }} \big) \dd \tau
& \lesssim h^{1-\delta} + \frac{1}{h^\delta} \int^h_0  \frac{1}{\sqrt{ x_3 \pm \tau }}  \dd \tau.
\end{split}
\Ee
Consider $\frac{|\rho(x + h \mathbf{e}_i) - \rho(x)|}{h^\delta}$, then the integrand in \eqref{est:rho_DQ_3} is $\frac{1}{\sqrt{ x_3 + \tau }}$. Note that $x + h \mathbf{e}_i \in \O$ for any $0< h < 1$. Picking $0 < \delta < 1/2$, we get
\Be \label{est1:rho_DQ_3}
\begin{split}
h^{-\delta } \Big|	\int^h_0  \frac{1}{\sqrt{ x_3 + \tau }}  \dd \tau \Big|
& \lesssim  h^{-\delta } \Big| \sqrt{ x_3 + h } - \sqrt{x_3} \Big|
\\& =  h^{-\delta } \big| \frac{h}{\sqrt{ x_3 + h } + \sqrt{x_3} } \big| 
\lesssim h^{1-\delta } \big| \frac{1}{ \sqrt{h} } \big|
= h^{\frac{1}{2} - \delta} \lesssim_{\delta} 1.
\end{split}
\Ee 
Consider $\frac{|\rho(x - h \mathbf{e}_i) - \rho(x)|}{h^\delta}$, then the integrand in \eqref{est:rho_DQ_3} is $\frac{1}{\sqrt{ x_3 - \tau }}$. Note that $x + h \mathbf{e}_i \in \O$ for any $0< h < \min \{ x_3, 1\}$. Picking $0 < \delta < 1/2$, we get
\Be \label{est2:rho_DQ_3}
\begin{split}
h^{-\delta } \Big|	\int^{\min \{h, x_3\}}_0  \frac{1}{\sqrt{ x_3 - \tau }}  \dd \tau \Big|
& \lesssim  h^{-\delta } \Big| \sqrt{ x_3} - \sqrt{ x_3 - \min \{h, x_3\}} \Big|
\\& =  h^{-\delta } \big| \frac{\min \{h, x_3\}}{ \sqrt{ x_3} + \sqrt{ x_3 - \min \{h, x_3\}} } \big| 
\lesssim h^{\frac{1}{2} - \delta} \lesssim_{\delta} 1.
\end{split}
\Ee 
Inputting \eqref{est:rho_DQ_12}-\eqref{est2:rho_DQ_3} into \eqref{rho_DQ}, we bound $|\rho |_{C^{0,\delta}(\O)}$ by
\Be \label{est:rho_Hol}
\begin{split}
|\rho |_{C^{0,\delta}(\O)}   \leq \sup_{0 < h < 1}\frac{|\rho(x \pm h \mathbf{e}_i) - \rho(x)|}{h^\delta}   \lesssim_\delta
\sum_{i = \pm} \| e^{\tilde \beta |p^0| } \nabla_{x_\parallel, p} G_i \|_{L^\infty (\gamma_-)}.
\end{split}
\Ee

On the other hand, similar to \eqref{est1:rho_x}, we have
\Be \label{est1:rho_Hol}
\begin{split}
| \rho (x) | 
& = \Big| \int_{\R^3} e_+ h_+ (x,p) + e_{-} h_{-} (x,p) \dd p \Big| 
\\& \leq \sum_{i = \pm} e_{i} \int_{\R^3} \frac{1}{w_{i, \beta} (x, p)} \dd p \times \| w_{i, \beta} G_i \|_{L^\infty(\gamma_-)}
\\& \leq \frac{1}{\beta} \big( e_+ \| w_{+, \beta} G_+ \|_{L^\infty(\gamma_-)} e^{-  \beta \frac{ m_+ }{2 c} g x_3} + e_{-} \| w_{-, \beta} G_- \|_{L^\infty(\gamma_-)} e^{- \beta \frac{ m_{-} }{2 c} g x_3} \big),
\end{split}
\Ee 
where the last inequality follows from \eqref{Uest:DPhi}.

Now using \eqref{est:nabla^2phi} and \eqref{est:rho_Hol}, together with \eqref{est1:rho_Hol}, we have
\Be  \notag
\begin{split}
\| \nabla_x^2 \Phi \|_{L^\infty(\O)} 
& \lesssim_\delta \|\rho \|_{L^\infty(\O)} + \|\rho \|_{C^{0,\delta}(\O)} + \frac{1}{\beta} \big( e_+ \| w_{+, \beta} G_+ \|_{L^\infty(\gamma_-)} + e_{-} \| w_{-, \beta} G_- \|_{L^\infty(\gamma_-)} \big),
\end{split}
\Ee 
and thus conclude \eqref{est:phi_C2}.

\smallskip

\textbf{Step 3. Proof of \eqref{est:hk_v}.} 
From \eqref{form:nabla_h}, we get
\Be \notag
\nabla_{p} h (x,p)
= \frac{\nabla_{p} \xb (x,p)}{e^{\tilde{\beta} |\pb^0 (x,p)| }} \cdot e^{\tilde{\beta} |\pb^0| } \nabla_{x_\parallel} G (\xb, \pb) + \frac{\nabla_{p} \pb (x,p)}{e^{\tilde{\beta} |\pb^0 (x,p)|}} \cdot e^{\tilde{\beta} |\pb^0|} \nabla_p G (\xb, \pb).
\Ee
This shows that
\be \label{est1:hk_v}
| \nabla_{p} h (x,p) | 
\leq \Big( \frac{ | \nabla_p \xb (x,p) |}{e^{\tilde{\beta} |\pb^0 (x,p)|}} + \frac{ |\nabla_p \pb (x,p)|}{e^{\tilde{\beta} |\pb^0 (x,p)|}} \Big) \| e^{\tilde{\beta} |p^0|}   \nabla_{x_\parallel, p} G   \|_{L^\infty (\gamma_-)}.
\ee
Using \eqref{est:xb_v}, we get
\Be \notag
\begin{split}
& \frac{ |\p_{p_i} \xb (x,p)| }{ e^{\tilde{\beta} |\pb^0 (x,p)|} }
\\& \leq \frac{|\vb (x,p)| }{|v_{\b, 3}  (x,p)| \sqrt{ m^2 + |p|^2/ c^2}} \times \big( \frac{2}{m g}  + \frac{\sqrt{8} \sqrt[4]{(m c)^2 + |\pb (x,p)|^2}}{\sqrt{ c m^3 g^2}} \big) | \pbn (x, p) | e^{- \tilde{\beta} |\pb^0 (x,p)|} 
\\& \ \ \ \ + \Big( 1 + \frac{|\vb (x,p)|}{|\vbn (x,p)|}
\frac{|\tb (x,p)|}{2} \times \big( \frac{2}{m g}  + \frac{\sqrt{8} \sqrt[4]{(m c)^2 + |\pb (x,p)|^2}}{\sqrt{ c m^3 g^2}} \big) | \pbn (x, p) | 
\\& \qquad \qquad \quad \times ( \| \nabla_x^2 \Phi  \|_\infty + e B_3 + mg ) \Big) e^{ (1 + B_3 + \| \nabla_x ^2 \Phi  \|_\infty) \tb} e^{- \tilde{\beta} |\pb^0 (x,p)|}
\end{split}
\Ee
Further, from \eqref{est:tbpb3}, we have 
\Be \label{est1:xb_v_a}
\begin{split}
& \frac{ |\p_{p_i} \xb (x,p)| }{ e^{\tilde{\beta} |\pb^0 (x,p)|} }
\\& \leq \frac{|\pb (x,p)|}{\sqrt{ m^2 + |p|^2/ c^2}} \big( \frac{2}{m g} + \frac{\sqrt{8} \sqrt[4]{(m c)^2 + |\pb (x,p)|^2}}{\sqrt{ c m^3 g^2}} \big) e^{- \tilde{\beta} |\pb^0 (x,p)| } 
\\& \ \ \ \ + \Big( 1 + \frac{|\pb (x,p)| |\tb (x,p)|}{2}  \big( \frac{2}{m g}  + \frac{\sqrt{8} \sqrt[4]{(m c)^2 + |\pb (x,p)|^2}}{\sqrt{ c m^3 g^2}} \big) ( \| \nabla_x^2 \Phi  \|_\infty + e B_3 + mg ) \Big)
\\& \qquad \quad \times e^{ (1 + B_3 + \| \nabla_x ^2 \Phi  \|_\infty) \tb} e^{- \tilde{\beta} |\pb^0 (x,p)|}.
\end{split}
\Ee
Next, using \eqref{lower:vb} and $\tilde{\beta}$ in \eqref{choice_beta}, together with \eqref{est:tb^h}, we obtain 
\Be \label{est1:xb_v}
\begin{split}
\eqref{est1:xb_v_a}
& \leq \frac{|\pb (x,p)|}{m^2 g} \big (2  + \frac{\sqrt{8} \sqrt[4]{(m c)^2 + |\pb (x,p)|^2}}{\sqrt{ c m}} \big) e^{- \tilde{\beta} |\pb^0 (x,p)|} 
\\& \ \ \ \ + \Big( 1 + \frac{4 |\pb (x,p)|^2}{(m g)^2}  \big( 1 + \frac{\sqrt{2} \sqrt[4]{(m c)^2 + |\pb (x,p)|^2}}{\sqrt{ c m}} \big) ( \| \nabla_x^2 \Phi  \|_\infty + e B_3 + mg ) \Big)
\\& \qquad \quad \times e^{ (1 + B_3 + \| \nabla_x ^2 \Phi  \|_\infty) \frac{4}{m g} |\pb (x,p)| } e^{- \tilde{\beta} |\pb^0 (x,p)|}
\\& \lesssim \frac{1}{m^2 g} e^{- \frac{\tilde{\beta}}{2} |\pb^0 (x,p)|} + \big( 1 + \frac{\| \nabla_x^2 \Phi  \|_\infty + e B_3 + mg}{(m g)^2} \big) e^{- \frac{\tilde{\beta}}{2} |\pb^0 (x,p)|} 
\\& \lesssim \big( 1 + \frac{ \| \nabla_x^2 \Phi  \|_\infty + e B_3 + mg}{(m g)^2} \big) e^{ - \frac{\tilde \beta}{2} \big( |p^0| + \frac{m g}{2 c} x_3 \big) }.
\end{split}
\Ee

Analogously, using \eqref{est:vb_v}, \eqref{lower:vb}, and $\tilde{\beta}$ in \eqref{choice_beta}, together with \eqref{est:tb^h}, \eqref{est:tbpb3} and \eqref{est:pb>p}, we obtain that
\Be \label{est1:vb_v}
\begin{split}
& \frac{ |\p_{p_i} \pb (x,p)| }{e^{\tilde{\beta} |\pb^0 (x,p)|}}
\\& \leq \frac{|e B_3 + m g| }{|v_{\b, 3}  (x,p)| \sqrt{ m^2 + |p|^2/ c^2}} \times \big( \frac{2}{m g} + \frac{\sqrt{8} \sqrt[4]{(m c)^2 + |\pb (x,p)|^2}}{\sqrt{ c m^3 g^2}} \big) | \pbn (x, p) | e^{- \tilde{\beta} |\pb^0 (x,p)|}
\\& \ \ \ \ + \Big( 1 + \frac{|e B_3 + m g|}{|\vbn (x,p)|}
\frac{|\tb (x,p)|}{2} \times \big( \frac{2}{m g}  + \frac{\sqrt{8} \sqrt[4]{(m c)^2 + |\pb (x,p)|^2}}{\sqrt{ c m^3 g^2}} \big) | \pbn (x, p) | 
\\& \qquad \qquad \quad \times ( \| \nabla_x^2 \Phi  \|_\infty + e B_3 + mg ) \Big) e^{ (1 + B_3 + \| \nabla_x ^2 \Phi  \|_\infty) \tb} e^{- \tilde{\beta} |\pb^0 (x,p)|}
\\& \leq \frac{|e B_3 + m g|}{m^2 g} \big( 2 + \frac{\sqrt{8} \sqrt[4]{(m c)^2 + |\pb (x,p)|^2}}{\sqrt{ c m}} \big) \sqrt{m^2 + |\pb (x,p)|^2 / c^2} \times e^{- \tilde{\beta} |\pb^0 (x,p)|} 
\\& \ \ \ \ + \Big( 1 + \frac{ 4 |e B_3 + m g| |\pb (x,p)|}{(mg)^2}  \big( 1  + \frac{\sqrt{2} \sqrt[4]{(m c)^2 + |\pb (x,p)|^2}}{\sqrt{ c m}} \big) \sqrt{m^2 + |\pb (x,p)|^2 / c^2}
\\& \qquad \qquad \quad \times ( \| \nabla_x^2 \Phi  \|_\infty + e B_3 + mg ) \Big) e^{ (1 + B_3 + \| \nabla_x ^2 \Phi  \|_\infty) \tb} e^{- \tilde{\beta} |\pb^0 (x,p)|}
\\& \lesssim \frac{|e B_3 + m g|}{m^2 g} e^{- \frac{\tilde{\beta}}{2} |\pb^0 (x,p)|} + \big( 1 + \frac{\| \nabla_x^2 \Phi  \|_\infty + e B_3 + mg}{(m g)^2} \big) e^{- \frac{\tilde{\beta}}{2} |\pb^0 (x,p)|}
\\& \lesssim \big( 1 + \frac{ \| \nabla_x^2 \Phi  \|_\infty + e B_3 + mg}{(m g)^2} \big) e^{ - \frac{\tilde \beta}{2} \big( |p^0| + \frac{m g}{2 c} x_3 \big) }.
\end{split}
\Ee

Inputting \eqref{est1:xb_v} and \eqref{est1:vb_v} into \eqref{est1:hk_v}, we conclude \eqref{est:hk_v}.

\smallskip

\textbf{Step 4. Proof of \eqref{est:hk_x}.} 
Again from \eqref{form:nabla_h}, we get
\Be \notag
\nabla_{x} h (x,p)
= \frac{\nabla_{x} \xb (x,p)}{e^{\tilde{\beta} |\pb^0 (x,p)|}} \cdot e^{\tilde{\beta} |\pb^0|} \nabla_{x_\parallel} G (\xb, \pb) + \frac{\nabla_{x} \pb (x,p)}{e^{\tilde{\beta} |\pb^0 (x,p)|}} \cdot e^{\tilde{\beta} |\pb^0|} \nabla_p G (\xb, \pb).
\Ee
This shows that
\be \label{est1:hk_x}
| \nabla_{x} h (x,p) | 
\leq \Big( \frac{ | \nabla_x \xb (x,p) |}{e^{\tilde{\beta} |\pb^0 (x,p)|}} + \frac{ |\nabla_x \pb (x,p)|}{e^{\tilde{\beta} |\pb^0 (x,p)|}} \Big) \| e^{\tilde{\beta} |p^0|} \nabla_{x_\parallel, p} G   \|_{L^\infty (\gamma_-)}.
\ee

Using \eqref{est:xb_x}, \eqref{est1:xb_x/w} and \eqref{est:tbpb3}, we have
\Be \label{est1:xb_x_a}
\begin{split}
& \frac{ |\p_{x_i} \xb (x,p)| }{ e^{\tilde{\beta} |\pb^0 (x,p)|} }
\\& \leq \frac{|\vb (x,p)| \delta_{i3}}{\alpha  (x,p)} e^{ \frac{4}{m g} (4g + 2 \| \nabla_x ^2 \Phi \|_{\infty} ) |\pb |} e^{- \tilde{\beta} |\pb^0 (x,p)|}
\\& \ \ \ \ + \Big( 1 + \frac{|\vb (x,p)|}{|\vbn (x,p)|}
\frac{|\tb (x,p)|}{2} \times \big( \frac{2}{m g}  + \frac{\sqrt{8} \sqrt[4]{(m c)^2 + |\pb (x,p)|^2}}{\sqrt{ c m^3 g^2}} \big) | \pbn (x, p) | 
\\& \qquad \qquad \quad \times ( \| \nabla_x^2 \Phi  \|_\infty + e B_3 + mg ) \Big) e^{ (1 + B_3 + \| \nabla_x ^2 \Phi  \|_\infty) \tb} e^{- \tilde{\beta} |\pb^0 (x,p)|}
\\& \leq \frac{|\pb (x,p)| \delta_{i3}}{\alpha  (x,p)} \frac{1}{\sqrt{m^2 + |\pb (x,p)|^2 / c^2}} \times e^{ \frac{4}{m g} (4g + 2 \| \nabla_x ^2 \Phi \|_{\infty} ) |\pb |} e^{- \tilde{\beta} |\pb^0 (x,p)|}
\\& \ \ \ \ + \Big( 1 + \frac{|\pb (x,p)| |\tb (x,p)|}{2}  \big( \frac{2}{m g}  + \frac{\sqrt{8} \sqrt[4]{(m c)^2 + |\pb (x,p)|^2}}{\sqrt{ c m^3 g^2}} \big) ( \| \nabla_x^2 \Phi  \|_\infty + e B_3 + mg ) \Big)
\\& \qquad \quad \times e^{ (1 + B_3 + \| \nabla_x ^2 \Phi  \|_\infty) \tb} e^{- \tilde{\beta} |\pb^0 (x,p)|}.
\end{split}
\Ee
Further, using \eqref{lower:vb} and $\tilde{\beta}$ in \eqref{choice_beta}, together with \eqref{est:tb^h}, we obtain 
\Be \label{est1:xb_x}
\begin{split}
\eqref{est1:xb_x_a}
& \leq \frac{c \delta_{i3}}{\alpha (x,p)} e^{ \frac{4}{m g} (4g + 2 \| \nabla_x ^2 \Phi \|_{\infty} ) |\pb |} e^{- \tilde{\beta} |\pb^0 (x,p)|}  
\\& \ \ \ \ + \Big( 1 + \frac{4 |\pb (x,p)|^2}{(m g)^2}  \big( 1 + \frac{\sqrt{2} \sqrt[4]{(m c)^2 + |\pb (x,p)|^2}}{\sqrt{ c m}} \big) ( \| \nabla_x^2 \Phi  \|_\infty + e B_3 + mg ) \Big)
\\& \qquad \quad \times e^{ (1 + B_3 + \| \nabla_x ^2 \Phi  \|_\infty) \frac{4}{m g} |\pb (x,p)| } e^{- \tilde{\beta} |\pb^0 (x,p)|}
\\& \lesssim \frac{\delta_{i3}}{\alpha (x,p)} e^{- \frac{\tilde{\beta}}{2} |\pb^0 (x,p)|} + \big( 1 + \frac{\| \nabla_x^2 \Phi  \|_\infty + e B_3 + mg}{(m g)^2} \big) e^{- \frac{\tilde{\beta}}{2} |\pb^0 (x,p)|} 
\\& \lesssim \big( \frac{\delta_{i3}}{\alpha (x,p)} + \frac{ \| \nabla_x^2 \Phi  \|_\infty + e B_3 + mg}{(m g)^2} \big) e^{ - \frac{\tilde \beta}{2} \big( |p^0| + \frac{m g}{2 c} x_3 \big) }.
\end{split}
\Ee

Analogously, using \eqref{est:vb_x}, \eqref{lower:vb}, and $\tilde{\beta}$ in \eqref{choice_beta}, together with \eqref{est:tb^h}, \eqref{est:tbpb3} and \eqref{est:pb>p}, we obtain that
\Be \label{est1:vb_x}
\begin{split}
& \frac{ |\p_{x_i} \pb (x,p)| }{e^{\tilde{\beta} |\pb^0 (x,p)|}}
\\& \leq \frac{ |e B_3 + m g| \delta_{i3}}{\alpha  (x,p)} e^{ \frac{4}{m g} (4g + 2 \| \nabla_x ^2 \Phi \|_{\infty} ) |\pb |} e^{- \tilde{\beta} |\pb^0 (x,p)|}
\\& \ \ \ \ + \Big( 1 + \frac{|e B_3 + m g|}{|\vbn (x,p)|}
\frac{|\tb (x,p)|}{2} \times \big( \frac{2}{m g}  + \frac{\sqrt{8} \sqrt[4]{(m c)^2 + |\pb (x,p)|^2}}{\sqrt{ c m^3 g^2}} \big) | \pbn (x, p) | 
\\& \qquad \qquad \quad \times ( \| \nabla_x^2 \Phi  \|_\infty + e B_3 + mg ) \Big) e^{ (1 + B_3 + \| \nabla_x ^2 \Phi  \|_\infty) \tb} e^{- \tilde{\beta} |\pb^0 (x,p)|}
\\& \leq \frac{ |e B_3 + m g| \delta_{i3}}{\alpha  (x,p)} e^{ \frac{4}{m g} (4g + 2 \| \nabla_x ^2 \Phi \|_{\infty} ) |\pb |} e^{- \tilde{\beta} |\pb^0 (x,p)|}
\\& \ \ \ \ + \Big( 1 + \frac{ 4 |e B_3 + m g| |\pb (x,p)|}{(mg)^2}  \big( 1  + \frac{\sqrt{2} \sqrt[4]{(m c)^2 + |\pb (x,p)|^2}}{\sqrt{ c m}} \big) \sqrt{m^2 + |\pb (x,p)|^2 / c^2}
\\& \qquad \qquad \quad \times ( \| \nabla_x^2 \Phi  \|_\infty + e B_3 + mg ) \Big) e^{ (1 + B_3 + \| \nabla_x ^2 \Phi  \|_\infty) \tb} e^{- \tilde{\beta} |\pb^0 (x,p)|}
\\& \lesssim \frac{|e B_3 + m g| \delta_{i3}}{\alpha  (x,p)} e^{- \frac{\tilde{\beta}}{2} |\pb^0 (x,p)|} + \big( 1 + \frac{\| \nabla_x^2 \Phi  \|_\infty + e B_3 + mg}{(m g)^2} \big) e^{- \frac{\tilde{\beta}}{2} |\pb^0 (x,p)|}
\\& \lesssim \big( \frac{\delta_{i3}}{\alpha  (x,p)} + \frac{ \| \nabla_x^2 \Phi  \|_\infty + e B_3 + mg}{(m g)^2} \big) e^{ - \frac{\tilde \beta}{2} \big( |p^0| + \frac{m g}{2 c} x_3 \big) }.
\end{split}
\Ee

Inputting \eqref{est1:xb_x} and \eqref{est1:vb_x} into \eqref{est1:hk_x}, we conclude \eqref{est:hk_x}.
\end{proof}

\begin{remark}
Proposition \ref{prop:Reg} shows the crucial regularities on $h_{\pm}$, $\rho$ and $\nabla_x \Phi$ under the assumption \eqref{choice_beta}. 
We remark that the term $B_3$ in \eqref{choice_beta} will not appear in the non-relativistic model, in which the backward exit time $\tBp$ admits a better estimate.
\end{remark}

Collecting the results in Proposition \ref{prop:Reg},
we conclude the following regularity estimate.

\begin{theorem}[Regularity Estimate] \label{theo:RS}
 
Suppose $(h, \rho_h, \Phi_h )$ solves \eqref{VP_h}-\eqref{eqtn:Dphi} in the sense of Definition \ref{weak_sol}. Suppose \eqref{Uest:DPhi} holds. Moreover, we assume
\be \notag
\| e^{{\tilde \beta } \sqrt{(m_{\pm} c)^2 + |p|^2}} \nabla_{x_\parallel, p} G_{\pm} (x,p) \|_{L^\infty (\gamma_-)} < \infty, \text{ for } \tilde \beta >0.
\ee
Furthermore, let $\beta,\tilde \beta >0$ satisfy  
\Be \label{condition:G}
\begin{split}
& \frac{1}{\beta} \big( e_+ \| e^{ \beta \sqrt{(m_{+} c)^2 + |p|^2} } G_+ \|_{L^\infty(\gamma_-)} + e_{-} \| e^{ \beta \sqrt{(m_{\-} c)^2 + |p|^2}} G_- \|_{L^\infty(\gamma_-)} \big)
\\& \ \ \ \ + \| e^{\tilde \beta \sqrt{(m_{+} c)^2 + |p|^2}} \nabla_{x_\parallel, p} G_+ \|_{L^\infty (\gamma_-)} + \| e^{\tilde \beta \sqrt{(m_{-} c)^2 + |p|^2}} \nabla_{x_\parallel, p} G_- \|_{L^\infty (\gamma_-)}
\\& \leq \frac{m_{\pm} g}{8} \tilde \beta - (1 + B_3).
\end{split}	
\Ee
and
\Be \label{condition:Phi_xx}
\frac{8}{m_{\pm} g} (1 + B_3 + \| \nabla_x ^2 \Phi  \|_\infty) \leq \tilde \beta \leq \beta.
\Ee
Then, $h, \rho$ and $\nabla_x \Phi $ are locally Lipschitz continuous and the following bounds hold:
\Be \label{theo:rho_x}
\begin{split}
e^{ \frac{\tilde \beta m g}{4 c} x_3 } |\p_{x_i} \rho  (x)|   
& \lesssim \| e^{\tilde \beta \sqrt{(m_+ c)^2 + |p|^2}} \nabla_{x_\parallel, p} G_+ \|_{L^\infty (\gamma_-)}  
\times \Big(
1 + \mathbf{1}_{|x_3| \leq 1} \frac{1}{\sqrt{ m_+ g x_3 }}
\Big)
\\& \ \ \ \ + \| e^{\tilde \beta \sqrt{(m_- c)^2 + |p|^2}} \nabla_{x_\parallel, p} G_- \|_{L^\infty (\gamma_-)}  
\times \Big(
1 + \mathbf{1}_{|x_3| \leq 1} \frac{1}{\sqrt{ m_- g x_3 }}
\Big),
\end{split}
\Ee
and
\Be \label{theo:phi_C2}
\begin{split}
\| \nabla_x^2 \Phi  \|_\infty
& \lesssim \frac{1}{\beta} \big( e_+ \| w_{+, \beta} G_+ \|_{L^\infty(\gamma_-)} + e_{-} \| w_{-, \beta} G_- \|_{L^\infty(\gamma_-)} \big)
\\& \ \ \ \ + \| e^{\tilde \beta \sqrt{(m_+ c)^2 + |p|^2}} \nabla_{x_\parallel, p} G_+ \|_{L^\infty (\gamma_-)} + \| e^{\tilde \beta \sqrt{(m_- c)^2 + |p|^2}} \nabla_{x_\parallel, p} G_- \|_{L^\infty (\gamma_-)}.
\end{split}	
\Ee
Furthermore,
\Be \label{theo:hk_v}
\begin{split}
& e^{ \frac{\tilde \beta}{2}|p^0|} e^{  \frac{\tilde \beta m_{\pm} g}{4 c} x_3} | \nabla_p h_{\pm} (x,p)| 
\\& \lesssim \big( 1 + \frac{ \| \nabla_x^2 \Phi  \|_\infty + e B_3 + mg}{(m g)^2} \big) \| e^{\tilde \beta \sqrt{(m_{\pm} c)^2 + |p|^2}} \nabla_{x_\parallel, p} G_{\pm} \|_{L^\infty (\gamma_-)},
\end{split}
\Ee
and
\Be \label{theo:hk_x}
\begin{split}
& e^{ \frac{\tilde \beta}{2}|p^0|} e^{  \frac{\tilde \beta m_{\pm} g}{4 c} x_3} | \nabla_x h_{\pm} (x,p)|  
\\& \lesssim \big( \frac{\delta_{i3}}{\alpha  (x,p)} + \frac{ \| \nabla_x^2 \Phi  \|_\infty + e B_3 + mg}{(m g)^2} \big) \| e^{\tilde \beta \sqrt{(m_{\pm} c)^2 + |p|^2}} \nabla_{x_\parallel, p} G_{\pm} \|_{L^\infty (\gamma_-)}.
\end{split}
\Ee
\end{theorem}

\begin{proof}

Using the condition \eqref{condition:G}, we can compute that \eqref{choice_beta}-\eqref{est:phi_C2} hold.
We omit the rest of the proof since it follows directly from Proposition \ref{prop:Reg}.
\end{proof}

\subsection{Uniqueness Theorem} 
\label{sec:US}

In this section, we prove Theorem \ref{theo:US}: a conditional uniqueness theorem subjected to a regularity of the solutions. 

The basic idea is as follows: suppose there are two steady-state solutions, $(h_{1, \pm}, \rho_1, \Phi_1)$ and $(h_{2, \pm}, \rho_2, \Phi_2)$, to the system \eqref{VP_h}-\eqref{eqtn:Dphi}.
By considering the difference $h_{1, \pm} - h_{2, \pm}$, we demonstrate that $h_{1, \pm} = h_{2, \pm}$ almost everywhere in $\Omega \times \mathbb{R}^3$, thereby establishing the uniqueness of the solution.

\begin{theorem}[Uniqueness Theorem] \label{theo:US}

Let $(h_{1, \pm}, \rho_1, \Phi_1)$ and $(h_{2, \pm}, \rho_2, \Phi_2)$ solve \eqref{VP_h}-\eqref{eqtn:Dphi} in the sense of Definition \ref{weak_sol}. 
Assume both follow all assumptions in Theorem \ref{theo:CS} and satisfy \eqref{Uest:wh}-\eqref{Uest:DPhi}. 
Further, assume there exist $\e, \bar \beta > 0$, such that
\Be \label{condition_unique}
\| w_{\pm, \bar \beta} \nabla_p h_{i, \pm} \|_\infty < \e \beta^2 
\ \ \text{for }  i = 1, 2,
\Ee
where $w_{\pm, \bar \beta}(x,v)$ is defined in \eqref{w^h}.
Then $h_{1, \pm} = h_{2, \pm}$ a.e. in $\O \times \R^3$ and $\Phi_1 = \Phi_2$ a.e. in $\O$. 
\end{theorem}

\begin{proof}

We only prove the case when the condition \eqref{condition_unique} holds for $i= 2$, since the case when \eqref{condition_unique} holds for  $i = 1$ can be deduced in a similar way. 
For the sake of simplicity, we abuse the notation as in \eqref{abuse}.

\smallskip

Since both $(h_{1, \pm}, \rho_1, \Phi_1)$ and $(h_{2, \pm}, \rho_2, \Phi_2)$ solve \eqref{VP_h}-\eqref{eqtn:Dphi}, we have 
\Be \label{VP_diff}
\begin{split}
& v_\pm \cdot \nabla_x ( h_{1, \pm} - h_{2, \pm} )
+ \big( e_{\pm} ( \frac{v_\pm}{c} \times B 
- \nabla_x \Phi_1 ) - \nabla_x (m_{\pm} g x_3) \big) \cdot \nabla_p ( h_{1, \pm} - h_{2, \pm} ) 
\\& = e_{\pm} \nabla_x (\Phi_1 - \Phi_2 ) \cdot \nabla_p h_{2, \pm} \ \ \text{in} \ \O \times \R^3, 
\end{split}
\ee
and
\be \label{VP_diff_bdy}
h_1 (x,p) - h_2 (x,p) = 0 \ \ \text{in} \ \gamma_-,
\Ee
where
\Be \notag
v_\pm =  c \frac{p}{p^0_\pm} = \frac{p}{\sqrt{ m_\pm^2 + |p|^2/ c^2}}.
\Ee
Suppose $(X (s;x,p), P (s;x,p))$ is the characteristics in \eqref{ODE_h} with $\Phi_h = \Phi_1$. Then, from \eqref{VP_diff} and \eqref{VP_diff_bdy}, we get
\Be \label{diff:h}
\begin{split}
& (h_1- h_2) (x,p) 
\\& = \int^0_{-\tb(x, p)} e (\nabla_x\Phi_1(X (s;x,p)) -\nabla_x \Phi_2(X (s;x,p))) \cdot \nabla_p h_2(X (s;x,p), P (s;x,p)) \dd s,
\end{split}
\Ee
where $\tb(x,p)$ is the backward exit time of the characteristics $(X (s;x,p), P (s;x,p))$.
From the condition \eqref{condition_unique}, we obtain that
\Be \label{bound:diff_h}
\begin{split}
& |\eqref{diff:h}| 
\\& \leq e \| \nabla_x \Phi_1 - \nabla_x \Phi_2 \|_\infty \int^0_{-\tb(x, p)} | \nabla_p h_2(X (s;x,p), P (s;x,p)) | \dd s
\\& \leq e \| \nabla_x \Phi_1 - \nabla_x \Phi_2 \|_\infty \| w_{\bar\beta} \nabla_p h_2 \|_\infty \int^0_{-\tb(x, p)} \frac{1}{ w_{\bar\beta} (X (s;x,p), P (s;x,p))} \dd s
\\& \leq e \| \nabla_x \Phi_1 - \nabla_x \Phi_2 \|_\infty \| w_{\bar\beta} \nabla_p h_2 \|_\infty 
\underbrace{\tb(x, p) \sup_{ s \in [- \tb(x,p), 0]}   \left( \frac{1}{w_{\bar \beta}(X(s;x,p), P(s;x,p))}\right )}_{\eqref{bound:diff_h}_*}.
\end{split}
\Ee
Since $\Phi$ satisfies \eqref{Uest:DPhi}, we have
\Be \label{lower_w}
w_{\bar \beta} (x,p) = e^{ \bar \beta \big( \sqrt{(m c)^2 + |p|^2} + \frac{1}{c} ( e \Phi_1 (x) + m g x_3 ) \big) } 
\geq e^{ \bar \beta \big( \sqrt{(m c)^2 + |p|^2} + \frac{1}{2 c} m g x_3 \big) },
\Ee
and
\Be \label{upper_w}
w_{\bar \beta} (x,p) = e^{ \bar \beta \big( \sqrt{(m c)^2 + |p|^2} + \frac{1}{c} ( e \Phi_1 (x) + m g x_3 ) \big) } 
\leq e^{ \bar \beta \big( \sqrt{(m c)^2 + |p|^2} + \frac{3}{2 c} m g x_3 \big) }.
\Ee
Using Proposition \ref{lem:tb}, together with \eqref{w:invar} and \eqref{lower_w}, we get 
\Be \label{est:tb/w}
\begin{split}
\eqref{bound:diff_h}_* 
= \frac{\tb (x, p)}{w_{\bar \beta} (x,p)}
& \leq \frac{\frac{4}{m g} \big( \sqrt{(m c)^2 + |p|^2} + \frac{3}{2c} m g x_{3} \big)}{e^{ \bar \beta \big( \sqrt{(m c)^2 + |p|^2} + \frac{1}{2 c} m g x_3 \big) }}
\\& \leq \frac{4}{m g} \big( \sqrt{(m c)^2 + |p|^2} + \frac{3}{2c} m g x_{3} \big) e^{ - \bar \beta \big( \sqrt{(m c)^2 + |p|^2} + \frac{1}{2 c} m g x_3 \big) }
\\& \leq \frac{12}{m g} \big( \sqrt{(m c)^2 + |p|^2} + \frac{1}{2c} m g x_{3} \big) e^{ - \bar \beta \big( \sqrt{(m c)^2 + |p|^2} + \frac{1}{2 c} m g x_3 \big) }
\\& \leq \frac{48}{m g \bar \beta} e^{ - \frac{3 \bar \beta}{4} \big( \sqrt{(m c)^2 + |p|^2} + \frac{1}{2 c} m g x_3 \big) }.
\end{split}
\Ee
Applying \eqref{est:tb/w} and \eqref{bound:diff_h} into \eqref{diff:h}, we have
\be \label{bound2:diff_h}
\begin{split}
|h_1 (x, p) - h_2 (x, p)| 
\leq e \| \nabla_x \Phi_1 - \nabla_x \Phi_2 \|_\infty \| w_{\bar\beta} \nabla_p h_2 \|_\infty \frac{48}{m g \bar \beta} e^{ - \frac{3 \bar \beta}{4} \big( \sqrt{(m c)^2 + |p|^2} + \frac{1}{2 c} m g x_3 \big) }.
\end{split}
\ee
This shows that
\be \label{bound3:diff_h}
\begin{split}
e^{ \frac{3 \bar \beta}{4} \big( \sqrt{(m c)^2 + |p|^2} + \frac{1}{2 c} m g x_3 \big) } |h_1 (x, p) - h_2 (x, p)|
\leq \frac{48 e}{m g \bar \beta}  \| \nabla_x \Phi_1 - \nabla_x \Phi_2 \|_\infty \| w_{\bar\beta} \nabla_p h_2 \|_\infty.
\end{split}
\ee

Note that both $(h_{1, \pm}, \rho_1, \Phi_1)$ and $(h_{2, \pm}, \rho_2, \Phi_2)$ satisfy \eqref{Uest:wh}-\eqref{Uest:DPhi}, from \eqref{Uest:rho} we deduce
\be \label{Uest:rho_12}
| \rho_{1, 2} (x) |
\leq \frac{1}{\beta} \big( e_+ \| w_{+, \beta} G_+ \|_{L^\infty(\gamma_-)} e^{-  \beta \frac{ m_+ }{2 c} g x_3} + e_{-} \| w_{-, \beta} G_- \|_{L^\infty(\gamma_-)} e^{- \beta \frac{ m_{-} }{2 c} g x_3} \big)
\ee
Consider $\hat{m} = \min \{ m_{-},  m_{+} \}$, then we pick $\beta'$ as follows:
\be \label{def:beta'}
0 < \beta' =  \min\{ \frac{\bar \beta}{4}, \beta \} \times \hat{m}.
\ee
Hence $\beta' \leq \beta \hat{m}$, and \eqref{Uest:rho_12} shows that
\be \notag
\| e^{\beta' \frac{g}{2c} x_3} \rho_{1, 2} (x) \|_\infty
\leq \frac{1}{\beta} \big( e_+ \| w_{+, \beta} G_+ \|_{L^\infty(\gamma_-)} + e_{-} \| w_{-, \beta} G_- \|_{L^\infty(\gamma_-)} \big) < \infty.
\ee
Set $A = \| e^{\beta' \frac{g}{2c} x_3} (\rho_1 - \rho_2 ) \|_\infty$ and $B = \beta' \frac{g}{2c}$, we get 
\be \label{est:rho_12}
| \rho_1 - \rho_2 | \leq A e^{- B x_3}.
\ee
Using \eqref{est:nabla_phi} in Lemma \ref{lem:rho_to_phi}, together with $- \Delta ( \Phi_1 - \Phi_2 ) = \rho_1 - \rho_2$, we obtain
\Be \label{est:phi_12_x}
\begin{split}
\| \nabla_x \Phi_1 - \nabla_x \Phi_2 \|_\infty 
&\leq \mathfrak{C} A \big( 1 +  \frac{1}{B} \big)
\\& \leq \mathfrak{C} (1 + \frac{2 c}{\beta^\prime g} )
\| e^{\beta' \frac{g}{2c} x_3} (\rho_1 - \rho_2 ) \|_\infty.
\end{split}
\Ee
Now setting $\hat{\beta} = \frac{\beta'}{\min \{ m_{-},  m_{+} \}}$, we can bound $\rho_1 - \rho_2$ by
\Be \notag
\begin{split}
\rho_1 - \rho_2 
& = \int_{\R^3} e_+ (h_{1, +} - h_{2, +} ) + e_{-} ( h_{1, -} - h_{2, -} ) \dd p
\\& \leq e_+ \| w_{+, \hat{\beta}} (h_{1, +} - h_{2, +} ) \|_{L^{\infty} (\O \times \R^3)} \times \int_{\R^3} \frac{1}{w_{+, \hat{\beta}}} \dd p
\\& \ \ \ \ + e_- \| w_{-, \hat{\beta}} (h_{1, -} - h_{2, -} ) \|_{L^{\infty} (\O \times \R^3)} \times \int_{\R^3} \frac{1}{w_{-, \hat{\beta}}} \dd p.
\end{split}
\Ee
This, together with \eqref{lower_w}, shows that
\Be \notag
\begin{split}
| \rho_1 - \rho_2 |
& \leq e_+ \| w_{+, \hat{\beta}} (h_{1, +} - h_{2, +} ) \|_{L^{\infty} (\O \times \R^3)} \times \frac{1}{\hat{\beta}} e^{- \hat{\beta} \frac{m _+ g}{2c} x_3}
\\& \ \ \ \ + e_- \| w_{-, \hat{\beta}} (h_{1, -} - h_{2, -} ) \|_{L^{\infty} (\O \times \R^3)} \times \frac{1}{\hat{\beta}} e^{- \hat{\beta} \frac{m_- g}{2c} x_3}
\\& \leq e_+ \| w_{+, \hat{\beta}} (h_{1, +} - h_{2, +} ) \|_{L^{\infty} (\O \times \R^3)} \times \frac{1}{\hat{\beta}} e^{- \beta' \frac{g}{2c} x_3}
\\& \ \ \ \ + e_- \| w_{-, \hat{\beta}} (h_{1, -} - h_{2, -} ) \|_{L^{\infty} (\O \times \R^3)} \times \frac{1}{\hat{\beta}} e^{- \beta' \frac{g}{2c} x_3}.
\end{split}
\Ee
Note that \eqref{def:beta'} implies $\hat{\beta} \leq \frac{\bar{\beta}}{4}$. Further, using \eqref{upper_w}, we derive that
\Be \label{est:e^beta*rho_12}
\begin{split}
e^{\beta' \frac{g}{2c} x_3} | \rho_1 - \rho_2 |
& \leq e_+ \| e^{ \hat{\beta} \big( \sqrt{(m_+ c)^2 + |p|^2} + \frac{3}{2 c} m_+ g x_3 \big) } (h_{1, +} - h_{2, +} ) \|_{L^{\infty} (\O \times \R^3)} \times \frac{1}{\hat{\beta}}
\\& \ \ \ \ + e_- \| e^{ \hat{\beta} \big( \sqrt{(m_- c)^2 + |p|^2} + \frac{3}{2 c} m_- g x_3 \big) } (h_{1, -} - h_{2, -} ) \|_{L^{\infty} (\O \times \R^3)} \times \frac{1}{\hat{\beta}}
\\& \leq e_+ \| e^{ \frac{\bar{\beta}}{4} \big( \sqrt{(m_+ c)^2 + |p|^2} + \frac{3}{2 c} m_+ g x_3 \big) } (h_{1, +} - h_{2, +} ) \|_{L^{\infty} (\O \times \R^3)} \times \frac{1}{\hat{\beta}}
\\& \ \ \ \ + e_- \| e^{ \frac{\bar{\beta}}{4} \big( \sqrt{(m_- c)^2 + |p|^2} + \frac{3}{2 c} m_- g x_3 \big) } (h_{1, -} - h_{2, -} ) \|_{L^{\infty} (\O \times \R^3)} \times \frac{1}{\hat{\beta}}.
\end{split}
\Ee
Inputting \eqref{est:e^beta*rho_12} into \eqref{est:phi_12_x}, we obtain
\Be \label{est2:phi_12_x}
\begin{split}
& \| \nabla_x \Phi_1 - \nabla_x \Phi_2 \|_\infty 
\\& \leq \mathfrak{C} (1 + \frac{2 c}{\beta^\prime g} ) \frac{1}{\hat{\beta}} \Big(
e_+ \| e^{ \frac{\bar{\beta}}{4} \big( \sqrt{(m_+ c)^2 + |p|^2} + \frac{3}{2 c} m_+ g x_3 \big) } (h_{1, +} - h_{2, +} ) \|_{L^{\infty} (\O \times \R^3)}
\\& \qquad \qquad \qquad \qquad + e_- \| e^{ \frac{\bar{\beta}}{4} \big( \sqrt{(m_- c)^2 + |p|^2} + \frac{3}{2 c} m_- g x_3 \big) } (h_{1, -} - h_{2, -} ) \|_{L^{\infty} (\O \times \R^3)}  \Big).
\end{split}
\Ee
Using \eqref{bound3:diff_h}, together with \eqref{est2:phi_12_x}, we bound
\be \label{bound4:diff_h}
\begin{split}
& e^{ \frac{3 \bar \beta}{4} \big( \sqrt{(m_+ c)^2 + |p|^2} + \frac{1}{2 c} m_+ g x_3 \big) } |h_{1, +} (x, p) - h_{2, +} (x, p)|
\\& \ \ \ \ + e^{ \frac{3 \bar \beta}{4} \big( \sqrt{(m_- c)^2 + |p|^2} + \frac{1}{2 c} m_- g x_3 \big) } |h_{1, -} (x, p) - h_{2, -} (x, p)|
\\& \leq \mathfrak{C} (1 + \frac{2 c}{\beta^\prime g} ) \frac{1}{\hat{\beta}} \Big(
e_+ \| e^{ \frac{\bar{\beta}}{4} \big( \sqrt{(m_+ c)^2 + |p|^2} + \frac{3}{2 c} m_+ g x_3 \big) } (h_{1, +} - h_{2, +} ) \|_{L^{\infty} (\O \times \R^3)}
\\& \qquad \qquad \qquad \qquad + e_- \| e^{ \frac{\bar{\beta}}{4} \big( \sqrt{(m_- c)^2 + |p|^2} + \frac{3}{2 c} m_- g x_3 \big) } (h_{1, -} - h_{2, -} ) \|_{L^{\infty} (\O \times \R^3)}  \Big)
\\& \ \ \ \ \times \big( \frac{48 e_+}{m g \bar \beta} \| w_{+, \bar\beta} \nabla_p h_{2, +} \|_\infty + \frac{48 e_-}{m g \bar \beta} \| w_{-, \bar\beta} \nabla_p h_{2, -} \|_\infty \big).
\end{split}
\ee
Together with $\beta'$ in \eqref{def:beta'} and $\hat{\beta} = \frac{\beta'}{\min \{ m_{-},  m_{+} \}}$, we get
\be \label{bound5:diff_h}
\begin{split}
\eqref{bound4:diff_h}
& \leq \mathfrak{C} (1 + \frac{2 c}{\beta^\prime g} ) \frac{ \max\{ e_+, e_- \} }{\hat{\beta}} \big( \frac{48 e_+}{m g \bar \beta} \| w_{+, \bar\beta} \nabla_p h_{2, +} \|_\infty + \frac{48 e_-}{m g \bar \beta} \| w_{-, \bar\beta} \nabla_p h_{2, -} \|_\infty \big)
\\& \ \ \ \ \times \Big( \| e^{ \frac{3 \bar{\beta}}{4} \big( \sqrt{(m_+ c)^2 + |p|^2} + \frac{1}{2 c} m_+ g x_3 \big) } (h_{1, +} - h_{2, +} ) \|_{L^{\infty} (\O \times \R^3)} 
\\& \qquad \qquad \qquad + \| e^{ \frac{3 \bar{\beta}}{4} \big( \sqrt{(m_- c)^2 + |p|^2} + \frac{1}{2 c} m_- g x_3 \big) } (h_{1, -} - h_{2, -} ) \|_{L^{\infty} (\O \times \R^3)}  \Big).
\end{split}
\ee
Under the condition \eqref{condition_unique}, we bound \eqref{bound5:diff_h} by
\be \label{bound6:diff_h}
\begin{split}
\eqref{bound5:diff_h} 
& \leq \frac{1}{2} \Big( \| e^{ \frac{3 \bar{\beta}}{4} \big( \sqrt{(m_+ c)^2 + |p|^2} + \frac{1}{2 c} m_+ g x_3 \big) } (h_{1, +} - h_{2, +} ) \|_{L^{\infty} (\O \times \R^3)} 
\\& \qquad \qquad \qquad + \| e^{ \frac{3 \bar{\beta}}{4} \big( \sqrt{(m_- c)^2 + |p|^2} + \frac{1}{2 c} m_- g x_3 \big) } (h_{1, -} - h_{2, -} ) \|_{L^{\infty} (\O \times \R^3)}  \Big).
\end{split}
\ee
Finally, together with \eqref{bound4:diff_h}-\eqref{bound6:diff_h}, we derive that
\Be \notag
\begin{split}
& \| e^{ \frac{3 \bar{\beta}}{4} \big( \sqrt{(m_+ c)^2 + |p|^2} + \frac{1}{2 c} m_+ g x_3 \big) } (h_{1, +} - h_{2, +} ) \|_{L^{\infty} (\O \times \R^3)} 
\\& \qquad \qquad + \| e^{ \frac{3 \bar{\beta}}{4} \big( \sqrt{(m_- c)^2 + |p|^2} + \frac{1}{2 c} m_- g x_3 \big) } (h_{1, -} - h_{2, -} ) \|_{L^{\infty} (\O \times \R^3)}
\\& \leq \frac{1}{2} \Big( \| e^{ \frac{3 \bar{\beta}}{4} \big( \sqrt{(m_+ c)^2 + |p|^2} + \frac{1}{2 c} m_+ g x_3 \big) } (h_{1, +} - h_{2, +} ) \|_{L^{\infty} (\O \times \R^3)} 
\\& \qquad \qquad + \| e^{ \frac{3 \bar{\beta}}{4} \big( \sqrt{(m_- c)^2 + |p|^2} + \frac{1}{2 c} m_- g x_3 \big) } (h_{1, -} - h_{2, -} ) \|_{L^{\infty} (\O \times \R^3)}  \Big),
\end{split}
\Ee
and conclude the uniqueness. 
\end{proof}

\subsection{Proof of the Main Theorem: Stationary Problem} 
\label{sec:EX_SS}

In this section, we prove the main result, Theorem \ref{theo:CS}.
Utilizing the a priori estimates from Section \ref{sec:RS} and applying mathematical induction, we derive the uniform-in-$\ell$ estimates on $\nabla_p h^{\ell}_{\pm}$ in Proposition \ref{prop:Unif_D2xDp}.

Then we consider the difference functions $ h^{\ell+1}_{\pm} - h^{\ell}_{\pm}$ for any $\ell \geq 1$. They satisfy the following:
\Be \notag
\begin{split}
& v_\pm \cdot \nabla_x ( h^{\ell+1}_{\pm} - h^{\ell}_{\pm} )
+ \big( e_{\pm} ( \frac{v_\pm}{c} \times B 
- \nabla_x \Phi^{\ell} ) - \nabla_x (m_{\pm} g x_3) \big) \cdot \nabla_p ( h^{\ell+1}_{\pm} - h^{\ell}_{\pm} ) 
\\& = e_{\pm} \nabla_x (\Phi^{\ell} - \Phi^{\ell-1} ) \cdot \nabla_p h^{\ell}_{\pm}.
\end{split}
\ee
By applying the uniform bound on $\nabla_p h^{\ell}_{\pm}$, we can demonstrate that the sequences $\{ h^{\ell+1}_{\pm} \}^{\infty}_{\ell=0}$, and $\{ \rho^\ell \}^{\infty}_{\ell=0}$, $\{ \nabla_x \Phi^\ell \}^{\infty}_{\ell=0}$ are Cauchy sequences in $L^{\infty} (\O \times \R^3)$ and $L^{\infty} (\O)$ respectively (see Proposition \ref{prop:cauchy}).
Therefore, we establish their strong convergence and conclude the existence of a steady solution.

\begin{prop} \label{prop:Unif_D2xDp}

Suppose the condition \eqref{condition:beta} holds for some $g, \beta > 0$.
Moreover, the condition \eqref{condition:G} holds with $\| e^{{\tilde \beta } \sqrt{(m_{\pm} c)^2 + |p|^2}} \nabla_{x_\parallel, p} G_{\pm} (x,p) \|_{L^\infty (\gamma_-)} < \infty$ for some $\beta \geq \tilde \beta >0$. Then $(h^{\ell+1}_{\pm}, \rho^\ell, \nabla_x  \Phi^\ell)$ from the construction satisfies the following uniform-in-$\ell$ estimates:
\be \label{Uest:DDPhi^l}
\frac{8}{m_{\pm} g} (1 + B_3 + \| \nabla_x ^2 \Phi^{\ell} \|_\infty) 
\leq \tilde \beta,
\ee
and
\be \label{Uest:h_v^l}
\begin{split}
& e^{ \frac{\tilde \beta}{2}|p^0_{\pm}|} e^{  \frac{\tilde \beta m_{\pm} g}{4 c} x_3} | \nabla_p h^{\ell+1}_{\pm} (x,p)| 
\\& \lesssim \big( 1 + \frac{ \| \nabla_x^2 \Phi^{\ell} \|_\infty + e B_3 + m_{\pm} g}{(m_{\pm} g)^2} \big) \| e^{\tilde \beta \sqrt{(m_{\pm} c)^2 + |p|^2}} \nabla_{x_\parallel, p} G_{\pm} \|_{L^\infty (\gamma_-)}.
\end{split}
\ee
\end{prop}

\begin{proof}

Under the initial setting $(\rho^0, \nabla_x \Phi^0) = (0,0)$, together with the condition \eqref{condition:G}, we deduce that $\| \nabla^2_x \Phi^0 \|_{\infty} = 0$, and thus \eqref{Uest:DDPhi^l} holds for $\ell = 0$. From the construction, $h^1_{\pm}$ is the solution to \eqref{VP_h} and \eqref{bdry:h} with $\nabla_x \Phi_h = \nabla_x \Phi^0 = 0$, that is,
\be \label{VP^0}
\begin{split}
& v_\pm \cdot \nabla_x h^1_{\pm}
+ \big( e_{\pm} ( \frac{v_\pm}{c} \times B 
- \nabla_x \Phi^0 ) - \nabla_x (m_{\pm} g x_3) \big) \cdot \nabla_p h^1_{\pm} = 0 \ \ \text{in} \ \O \times \R^3, 
\\& h^1_{\pm} (x,p) = G_{\pm} (x,p) \ \ \text{in} \ \gamma_-. 
\end{split}
\ee
From \eqref{z1}, we consider the characteristic $(X^1, P^1)$  and the corresponding $(\tb^1, \xb^1, \pb^1)$ for \eqref{VP^0}.
By replacing $\Phi_h$ with $\Phi^0$, Lemmas \ref{VL}-\ref{lem:nabla_zb} hold for the characteristic $(X^1, P^1)$ and the corresponding $(\tb^1, \xb^1, \pb^1)$.
This is because we only work with Vlasov equations under the assumption \eqref{Uest:DPhi} in the proof of these Lemmas, and $\nabla_x \Phi^0 = 0$ clearly satisfies this assumption.

Analogously, by replacing $\Phi_h$ with $\Phi^0$ in Proposition \ref{prop:Reg}, we can check that $\| \nabla^2_x \Phi^0 \|_{\infty}$ satisfies the assumption \eqref{choice_beta}.
From Proposition \ref{prop:Reg}, we bound $\| \nabla^2_x \Phi^0 \|_{\infty}$ by
\Be \notag
\begin{split}
\| \nabla_x^2 \Phi^0 \|_\infty
& \lesssim \frac{1}{\beta} \big( e_+ \| w_{+, \beta} G_+ \|_{L^\infty(\gamma_-)} + e_{-} \| w_{-, \beta} G_- \|_{L^\infty(\gamma_-)} \big)
\\& \ \ \ \ + \big( \| e^{\tilde \beta \sqrt{(m_+ c)^2 + |p|^2}} \nabla_{x_\parallel, p} G_+ \|_{L^\infty (\gamma_-)} + \| e^{\tilde \beta \sqrt{(m_- c)^2 + |p|^2}} \nabla_{x_\parallel, p} G_- \|_{L^\infty (\gamma_-)}
\big) .
\end{split}	
\Ee
Following step 3 (Proof of \eqref{est:hk_v}) in the proof Proposition \ref{prop:Reg}, we deduce that \eqref{Uest:h_v^l} holds for $\ell = 0$.

\smallskip

Now we prove this by induction. 
In the following proof, we abuse the notation as in \eqref{abuse}.
Assume a positive integer $k > 0$ and suppose that \eqref{Uest:DDPhi^l} and \eqref{Uest:h_v^l} hold for $0 \leq \ell \leq k$.
From \eqref{eqtn:hk}-\eqref{bdry:phik}, then for $\ell=k+1$, 
\be \label{rho_phi^k+1}
\begin{split}	
\rho^{k+1} = \int_{\R^3} ( e_+ h^{k+1}_+ + e_{-} h^{k+1}_{-} ) \dd p & \ \ \text{in} \ \O, \\
- \Delta  \Phi^{k+1} = \rho^{k+1} & \ \ \text{in} \ \O, \\
\Phi^{k+1} =0  & \ \ \text{on} \ \p\O,
\end{split}
\ee
and
\be \label{h^k+2}
\begin{split}
& v_\pm \cdot \nabla_x h^{k+2}_{\pm} 
+ \big( e_{\pm} ( \frac{v_\pm}{c} \times B 
- \nabla_x \Phi^{k+1} ) - \nabla_x (m_{\pm} g x_3) \big) \cdot \nabla_p h^{k+2}_{\pm} = 0 \ \ \text{in} \ \O \times \R^3, 
\\& h^{k+2}_{\pm} (x,p) = G_{\pm} (x,p) \ \ \text{on} \ \gamma_-.
\end{split}
\ee
On the other side, from the construction, $h^{k+1}_{\pm}$ is the solution to \eqref{VP_h} and \eqref{bdry:h} with $\nabla_x \Phi_h = \nabla_x \Phi^{k}$, that is,
\be \label{VP^k}
\begin{split}
& v_\pm \cdot \nabla_x h^{k+1}_{\pm}
+ \big( e_{\pm} ( \frac{v_\pm}{c} \times B 
- \nabla_x \Phi^k ) - \nabla_x (m_{\pm} g x_3) \big) \cdot \nabla_p h^{k+1}_{\pm} = 0 \ \ \text{in} \ \O \times \R^3, 
\\& h^{k+1}_{\pm} (x,p) = G_{\pm} (x,p) \ \ \text{on} \ \gamma_-. 
\end{split}
\ee
Recall \eqref{Uest:DPhi^k} in Proposition \ref{prop:Unif_steady}, we have, for any $\ell \geq 0$,
\be \label{est:Phi_x^0}
\| \nabla_x \Phi^ \ell \|_{L^\infty (\bar{\O})} 
\leq  \min \big( \frac{m_{+}}{e_{+}}, \frac{m_{-}}{e_{-}} \big) \times \  \frac{g}{2}.
\ee 
From \eqref{VP_h^l+1}, \eqref{bdry:h^l+1} and \eqref{def:tb_l}, we consider the characteristic $(X^{k+1}, P^{k+1})$ for \eqref{VP^k} and the corresponding $(\tb^{k+1}, \xb^{k+1}, \pb^{k+1})$.
Similar as the initial case, by replacing $\Phi_h$ with $\Phi^k$, Lemmas \ref{VL}-\ref{lem:nabla_zb} hold for $(X^{k+1}, P^{k+1})$ and $(\tb^{k+1}, \xb^{k+1}, \pb^{k+1})$.

\smallskip

First, we show \eqref{Uest:DDPhi^l} holds for $\ell = k+1$.
From \eqref{rho_phi^k+1}, \eqref{VP^k} and the condition \eqref{condition:G}, we have
\Be \notag
\begin{split}
& \ \ \ \ | \nabla_x \rho^{k+1} (x) | 
\\& = \Big| \int_{\R^3} e_+ \nabla_x h^{k+1}_+ (x,p) + e_{-} \nabla_x h^{k+1}_{-} (x,p) \dd p \Big| 
\\& \leq \sum_{i = \pm} e_{i} \Big|\int_{\R^3} \nabla_x \xb^{k+1} \cdot \nabla_{x_\parallel} G_i (\xb^{k+1}, \pb^{k+1})  + \nabla_x \pb^{k+1} \cdot \nabla_p G_i (\xb^{k+1}, \pb^{k+1}) \dd p \Big|.
\end{split}
\Ee 
Applying Lemma \ref{lem:nabla_zb} on $(\tb^{k+1}, \xb^{k+1}, \pb^{k+1})$, together with the step 1 (Proof of \eqref{est:rho_x}) in the proof of Proposition \ref{prop:Reg}, we deduce that
\Be \label{est:rho_x^k+1}
\begin{split}
e^{ \frac{\tilde \beta m g}{4 c} x_3 } |\p_{x_i} \rho^{k+1} (x)|   
& \lesssim \| e^{\tilde \beta \sqrt{(m_+ c)^2 + |p|^2}} \nabla_{x_\parallel, p} G_+ \|_{L^\infty (\gamma_-)}  
\times \Big(
1 + \mathbf{1}_{|x_3| \leq 1} \frac{1}{\sqrt{ m_+ g x_3 }}
\Big)
\\& \ \ \ \ + \| e^{\tilde \beta \sqrt{(m_- c)^2 + |p|^2}} \nabla_{x_\parallel, p} G_- \|_{L^\infty (\gamma_-)}  
\times \Big(
1 + \mathbf{1}_{|x_3| \leq 1} \frac{1}{\sqrt{ m_- g x_3 }}
\Big),
\end{split}
\Ee
Using \eqref{est:rho_x^k+1}, together with the step 2 (Proof of \eqref{est:phi_C2}) in the proof of Proposition \ref{prop:Reg}, we bound
\Be \notag
\begin{split}
& \frac{|\rho^{k+1} (x \pm h \mathbf{e}_i) - \rho^{k+1} (x)|}{h^\delta} 
\\& \leq \frac{1}{h^\delta} \int^h_0 | \nabla_x \rho^{k+1} (x \pm \tau \mathbf{e}_i) | \dd \tau 
\\& \leq \sum_{i = \pm} \| e^{\tilde \beta |p^0|} \nabla_{x_\parallel, p} G_i \|_{L^\infty (\gamma_-)} \times \frac{1}{h^\delta} \int^h_0 \big( 1 + \mathbf{1}_{|x_3| \leq 1} \frac{1}{\sqrt{ m g (x \pm \tau \mathbf{e}_i)_3 }} \big) \dd \tau.
\end{split}
\Ee
Then we bound $|\rho^{k+1} |_{C^{0,\delta}(\O)}$ by
\Be \label{est:rho_Hol^k+1}
\begin{split}
|\rho^{k+1} |_{C^{0,\delta}(\O)} & \leq \sup_{0 < h < 1}\frac{|\rho^{k+1} (x \pm h \mathbf{e}_i) - \rho^{k+1} (x)|}{h^\delta} \\& \lesssim_\delta
\sum_{i = \pm} \| e^{\tilde \beta |p^0| } \nabla_{x_\parallel, p} G_i \|_{L^\infty (\gamma_-)}.
\end{split}
\Ee
Using \eqref{est:nabla^2phi}, \eqref{Uest:rho^k} and \eqref{est:rho_Hol^k+1}, we derive
\Be \label{est:DDPhi^k+1}
\begin{split}
\| \nabla_x^2 \Phi^{k+1} \|_{L^\infty(\O)} 
& \lesssim_\delta \|\rho \|_{L^\infty(\O)} + \|\rho \|_{C^{0,\delta}(\O)}
\\& \qquad + \frac{1}{\beta} \big( e_+ \| e^{ \beta \sqrt{(m_{+} c)^2 + |p|^2}} G_+ \|_{L^\infty(\gamma_-)}  + e_- \| e^{ \beta \sqrt{(m_{-} c)^2 + |p|^2}} G_- \|_{L^\infty(\gamma_-)} \big) .
\end{split}
\Ee 
Inputting \eqref{rho_phi^k+1} and the upper bound on $|\rho^{k+1} (x)|$ from \eqref{Uest:rho^k} into \eqref{est:DDPhi^k+1}, we have
\Be \notag
\begin{split}
\| \nabla_x^2 \Phi^{k+1} \|_\infty
& \lesssim \frac{1}{\beta} \big( e_+ \| w_{+, \beta} G_+ \|_{L^\infty(\gamma_-)} + e_{-} \| w_{-, \beta} G_- \|_{L^\infty(\gamma_-)} \big)
\\& \ \ \ \ + \big( \| e^{\tilde \beta \sqrt{(m_+ c)^2 + |p|^2}} \nabla_{x_\parallel, p} G_+ \|_{L^\infty (\gamma_-)} + \| e^{\tilde \beta \sqrt{(m_- c)^2 + |p|^2}} \nabla_{x_\parallel, p} G_- \|_{L^\infty (\gamma_-)}
\big).
\end{split}	
\Ee
Under the condition \eqref{condition:G}, this shows that
\be \label{est1:DDPhi^k+1}
\frac{8}{m g} (1 + B_3 + \| \nabla_x ^2 \Phi^{k+1} \|_\infty) \leq \tilde \beta,
\ee
and \eqref{Uest:DDPhi^l} holds for $\ell = k+1$.

\smallskip

Second, we show \eqref{Uest:h_v^l} holds for $\ell = k+1$.
From \eqref{VP_h^l+1}, \eqref{bdry:h^l+1} and \eqref{def:tb_l}, we consider the characteristic $(X^{k+2}, P^{k+2})$ for \eqref{h^k+2} and the corresponding $(\tb^{k+2}, \xb^{k+2}, \pb^{k+2})$.
From \eqref{h^k+2}, we have
\Be \notag
\nabla_{p} h^{k+2} (x,p)
= \nabla_{p} \xb^{k+2} (x,p) \cdot \nabla_{x_\parallel} G (\xb^{k+2}, \pb^{k+2}) + \nabla_{p} \pb^{k+2} (x,p) \cdot \nabla_p G (\xb^{k+2}, \pb^{k+2}).
\Ee
Similarly, by replacing $\Phi_h$ with $\Phi^{k+1}$, Lemmas \ref{VL}-\ref{lem:nabla_zb} hold for the characteristic $(X^{k+2}, P^{k+2})$ and $(\tb^{k+2}, \xb^{k+2}, \pb^{k+2})$.
Applying Lemma \ref{lem:nabla_zb} on $(\tb^{k+2}, \xb^{k+2}, \pb^{k+2})$ and the upper bound on $\| \nabla_x ^2 \Phi^{k+1} \|_\infty$ from \eqref{est1:DDPhi^k+1}, together with the step 3 (Proof of \eqref{est:hk_v}) in the proof of Proposition \ref{prop:Reg}, we conclude that \eqref{Uest:h_v^l} holds for $\ell = k+1$.

\smallskip

Now we have proved \eqref{Uest:DDPhi^l} and \eqref{Uest:h_v^l} hold for $\ell = k+1$. Therefore, we complete the proof by induction.
\end{proof}

\begin{prop} \label{prop:cauchy}

Suppose the condition \eqref{condition:beta} holds for some $g, \beta > 0$.
Moreover, the condition \eqref{condition:G} holds with $\| e^{{\tilde \beta } \sqrt{(m_{\pm} c)^2 + |p|^2}} \nabla_{x_\parallel, p} G_{\pm} (x,p) \|_{L^\infty (\gamma_-)} < \infty$, and for some $\beta \geq \tilde \beta >0$,
\be \label{weight_G_xv}
(1 + \frac{\tilde{\beta}}{m_{\pm} g} ) \| e^{\tilde \beta \sqrt{(m_{\pm} c)^2 + |p|^2}} \nabla_{x_\parallel, p} G_{\pm} \|_{L^\infty (\gamma_-)} \leq \frac{1}{4}.
\ee
Then for any $\ell \geq 1$, $h^{\ell}_{\pm}$ from the construction satisfies that there exists some $\bar \beta > 0$, 
\be \label{est:h_cauchy}
\begin{split}
& \| e^{ \frac{3 \bar{\beta}}{4} \big( \sqrt{(m_+ c)^2 + |p|^2} + \frac{1}{2 c} m_+ g x_3 \big) } (h^{\ell+1}_{+} - h^{\ell}_{+} ) \|_{L^{\infty} (\O \times \R^3)} 
\\& \qquad \qquad + \| e^{ \frac{3 \bar{\beta}}{4} \big( \sqrt{(m_- c)^2 + |p|^2} + \frac{1}{2 c} m_- g x_3 \big) } (h^{\ell+1}_{-} - h^{\ell}_{-} ) \|_{L^{\infty} (\O \times \R^3)}
\\& \leq \frac{1}{2} \Big( \| e^{ \frac{3 \bar{\beta}}{4} \big( \sqrt{(m_+ c)^2 + |p|^2} + \frac{1}{2 c} m_+ g x_3 \big) } (h^\ell_{+} - h^{\ell-1}_{+} ) \|_{L^{\infty} (\O \times \R^3)} 
\\& \qquad \qquad + \| e^{ \frac{3 \bar{\beta}}{4} \big( \sqrt{(m_- c)^2 + |p|^2} + \frac{1}{2 c} m_- g x_3 \big) } (h^\ell_{-} - h^{\ell-1}_{-} ) \|_{L^{\infty} (\O \times \R^3)}  \Big),
\end{split}
\Ee
Furthermore, $\{ h^{\ell+1}_{\pm} \}^{\infty}_{\ell=0}$, and $\{ \rho^\ell \}^{\infty}_{\ell=0}$, $\{ \nabla_x \Phi^\ell \}^{\infty}_{\ell=0}$ from the construction are Cauchy sequences in $L^{\infty} (\O \times \R^3)$ and $L^{\infty} (\O)$ respectively.
\end{prop}

\begin{proof}

For the sake of simplicity, we abuse the notation as in \eqref{abuse}.
We first prove \eqref{est:h_cauchy}. 
Given $\ell \geq 1$, from the construction in \eqref{eqtn:hk}-\eqref{bdry:phik}, we have 
\Be \label{VP_diff^l+1}
\begin{split}
& v_\pm \cdot \nabla_x ( h^{\ell+1}_{\pm} - h^{\ell}_{\pm} )
+ \big( e_{\pm} ( \frac{v_\pm}{c} \times B 
- \nabla_x \Phi^{\ell} ) - \nabla_x (m_{\pm} g x_3) \big) \cdot \nabla_p ( h^{\ell+1}_{\pm} - h^{\ell}_{\pm} ) 
\\& = e_{\pm} \nabla_x (\Phi^{\ell} - \Phi^{\ell-1} ) \cdot \nabla_p h^{\ell}_{\pm} \ \ \text{in} \ \O \times \R^3, 
\end{split}
\ee
with 
\be \label{VP_diff_bdy^l+1}
h^{\ell+1}_{\pm} (x,p) - h^{\ell}_{\pm} (x,p) = 0 \ \ \text{in} \ \gamma_-.
\Ee
Following \eqref{VP_h^l+1}, \eqref{bdry:h^l+1} and \eqref{def:tb_l}, we consider the characteristic $(X^{\ell+1}, P^{\ell+1})$ for \eqref{VP_diff^l+1}
and the corresponding backward exit time $\tb^{\ell+1} (x,p)$. Thus, from \eqref{VP_diff^l+1} and \eqref{VP_diff_bdy^l+1},
\Be \label{diff:h^l+1}
\begin{split}
& (h^{\ell+1} - h^{\ell} ) (x,p) 
\\& = \int^0_{-\tb^{\ell+1} (x, p)} e (\nabla_x \Phi^{\ell} - \nabla_x \Phi^{\ell-1} ) \cdot \nabla_p h^{\ell} ( X^{\ell+1} (s;x,p), P^{\ell+1} (s;x,p)) \dd s.
\end{split}
\Ee
Using Proposition \ref{prop:Unif_D2xDp}, together with the conditions \eqref{condition:beta} and \eqref{condition:G}, then for any $\ell \geq 1$,
\be \label{est:h_v^l+1}
\begin{split}
& e^{ \frac{\tilde \beta}{2} |p^0|} e^{  \frac{\tilde \beta m g}{4 c} x_3} | \nabla_p h^{\ell}_{\pm} (x,p)| 
\\& \lesssim \big( 1 + \frac{ \| \nabla_x^2 \Phi^{\ell} \|_\infty + e B_3 + mg}{(m g)^2} \big) \| e^{\tilde \beta \sqrt{(m_{\pm} c)^2 + |p|^2}} \nabla_{x_\parallel, p} G_{\pm} \|_{L^\infty (\gamma_-)}
\\& \lesssim (1 + \frac{\tilde{\beta}}{m g} ) \| e^{\tilde \beta \sqrt{(m_{\pm} c)^2 + |p|^2}} \nabla_{x_\parallel, p} G_{\pm} \|_{L^\infty (\gamma_-)}.
\end{split}
\ee
From \eqref{Uest:DPhi^k} in Proposition \ref{prop:Unif_steady}, then for any $\ell \geq 1$,
\be \label{est:phi_x^l}
\| \nabla_x \Phi^{\ell} \|_{L^\infty (\bar{\O})} 
\leq  \min \big( \frac{m_{+}}{e_{+}}, \frac{m_{-}}{e_{-}} \big) \times \  \frac{g}{2}. 
\ee
Recall $w^{\ell+1}_{\pm, \beta} (x,p)$ defined in \eqref{def:w^k}, from \eqref{est:phi_x^l} we get for any $\ell \geq 1$,
\be \label{upper_w^l+1}
w^{\ell+1}_{\pm, \beta} (x,p) 
= e^{ \beta \left( \sqrt{(m_{\pm} c)^2 + |p|^2} + \frac{1}{c} ( e_{\pm} \Phi^{\ell} (x) + m_{\pm} g x_3 ) \right) } 
\leq e^{ \beta \left( \sqrt{(m_{\pm} c)^2 + |p|^2} + \frac{3}{2 c} ( m_{\pm} g x_3 ) \right) },
\ee
and
\Be \label{lower_w^l+1}
w^{k+1}_{\beta} (x,p) 
\geq e^{\beta \big( \sqrt{(m c)^2 + |p|^2} + \frac{1}{2 c} m g x_3 \big) }.
\Ee
By setting $\bar{\beta} = \frac{\tilde \beta}{6}$, we can rewrite \eqref{est:h_v^l+1} as
\be \label{est1:h_v^l+1}
w^{\ell+1}_{\pm, \bar{\beta}} (x,p)  | \nabla_p h^{\ell}_{\pm} (x,p)| 
\lesssim (1 + \frac{\tilde{\beta}}{m g} ) \| e^{\tilde \beta \sqrt{(m_{\pm} c)^2 + |p|^2}} \nabla_{x_\parallel, p} G_{\pm} \|_{L^\infty (\gamma_-)}.
\ee
Thus, we obtain that for any $\ell \geq 1$,
\Be \label{bound:diff_h^l+1}
\begin{split}
& |\eqref{diff:h^l+1}| 
\\& \leq e \| \nabla_x \Phi^{\ell} - \nabla_x \Phi^{\ell-1} \|_\infty \int^0_{-\tb^{\ell+1} (x, p)} | \nabla_p h^{\ell} (X^{\ell+1} (s;x,p), P^{\ell+1} (s;x,p)) | \dd s
\\& \leq e \| \nabla_x \Phi^{\ell} - \nabla_x \Phi^{\ell-1} \|_\infty \| w^{\ell+1}_{\bar\beta} \nabla_p h^{\ell} \|_\infty \int^0_{-\tb^{\ell+1} (x, p)} \frac{1}{ w^{\ell+1}_{\bar\beta} (X^{\ell+1} (s;x,p), P^{\ell+1} (s;x,p)) } \dd s
\\& \leq e \| \nabla_x \Phi^{\ell} - \nabla_x \Phi^{\ell-1} \|_\infty \| w^{\ell+1}_{ \bar{\beta} } \nabla_p h^{\ell} \|_\infty 
\\& \qquad \qquad \times
\underbrace{\tb^{\ell+1} (x, p) \sup_{ s \in [- \tb(x,p), 0]}   \left( \frac{1}{ w^{\ell+1}_{\bar{\beta}} (X^{\ell+1} (s;x,p), P^{\ell+1} (s;x,p)) }\right )}_{\eqref{bound:diff_h^l+1}_*}.
\end{split}
\Ee
Furthermore, by replacing $\Phi_h$ with $\Phi^{\ell}$, Proposition \ref{lem:tb} holds for $\tb^{\ell+1} (x,p)$, that is,
\Be \label{est:tb^l+1}
\tb^{\ell+1} (x,p) \leq \min \left\{ \frac{4}{m g} \big( \sqrt{(m c)^2 + |p|^2} + \frac{3}{2c} m g x_{3} \big), \ \frac{4}{m g} |\pb (x, p)| \right\}. 
\Ee
Using \eqref{est:tb^l+1}, together with \eqref{w:invar} and \eqref{lower_w^l+1}, we get 
\Be \label{est:tb/w^l+1}
\begin{split}
\eqref{bound:diff_h^l+1}_* 
= \frac{\tb (x, p)}{w^{\ell+1}_{\bar \beta} (x,p)}
& \leq \frac{\frac{4}{m g} \big( \sqrt{(m c)^2 + |p|^2} + \frac{3}{2c} m g x_{3} \big)}{e^{ \bar \beta \big( \sqrt{(m c)^2 + |p|^2} + \frac{1}{2 c} m g x_3 \big) }}
\\& \leq \frac{4}{m g} \big( \sqrt{(m c)^2 + |p|^2} + \frac{3}{2c} m g x_{3} \big) e^{ - \bar \beta \big( \sqrt{(m c)^2 + |p|^2} + \frac{1}{2 c} m g x_3 \big) }
\\& \leq \frac{12}{m g} \big( \sqrt{(m c)^2 + |p|^2} + \frac{1}{2c} m g x_{3} \big) e^{ - \bar \beta \big( \sqrt{(m c)^2 + |p|^2} + \frac{1}{2 c} m g x_3 \big) }
\\& \leq \frac{48}{m g \bar \beta} e^{ - \frac{3 \bar \beta}{4} \big( \sqrt{(m c)^2 + |p|^2} + \frac{1}{2 c} m g x_3 \big) }.
\end{split}
\Ee
Applying \eqref{est:tb/w^l+1} and \eqref{bound:diff_h^l+1} into \eqref{diff:h^l+1}, we have
\be \label{bound2:diff_h^l+1}
\begin{split}
|h^{\ell+1} (x, p) - h^{\ell} (x, p)| 
\leq e \| \nabla_x \Phi^{\ell} - \nabla_x \Phi^{\ell-1} \|_\infty \| w^{\ell+1}_{ \bar{\beta} } \nabla_p h^{\ell} \|_\infty \frac{48}{m g \bar \beta} e^{ - \frac{3 \bar \beta}{4} \big( \sqrt{(m c)^2 + |p|^2} + \frac{1}{2 c} m g x_3 \big) }.
\end{split}
\ee
This shows that
\be \label{bound3:diff_h^l+1}
\begin{split}
e^{ \frac{3 \bar \beta}{4} \big( \sqrt{(m c)^2 + |p|^2} + \frac{1}{2 c} m g x_3 \big) } |h^{\ell+1} (x, p) - h^{\ell} (x, p)|
\leq \frac{48 e}{m g \bar \beta} \| \nabla_x \Phi^{\ell} - \nabla_x \Phi^{\ell-1} \|_\infty \| w^{\ell+1}_{ \bar{\beta} } \nabla_p h^{\ell} \|_\infty.
\end{split}
\ee

On the other hand, from \eqref{Uest:rho^k} in Proposition \ref{prop:Unif_steady}, we obtain
\be \label{Uest:rho_12^l+1}
\begin{split}
| \rho^{\ell} (x) |
& \leq \frac{1}{\beta} \big( e_+ \| w_{+, \beta} G_+ \|_{L^\infty(\gamma_-)} e^{-  \beta \frac{ m_+ }{2 c} g x_3} + e_{-} \| w_{-, \beta} G_- \|_{L^\infty(\gamma_-)} e^{- \beta \frac{ m_{-} }{2 c} g x_3} \big),
\\ | \rho^{\ell-1} (x) |
& \leq \frac{1}{\beta} \big( e_+ \| w_{+, \beta} G_+ \|_{L^\infty(\gamma_-)} e^{-  \beta \frac{ m_+ }{2 c} g x_3} + e_{-} \| w_{-, \beta} G_- \|_{L^\infty(\gamma_-)} e^{- \beta \frac{ m_{-} }{2 c} g x_3} \big).
\end{split}
\ee
Consider $\hat{m} = \min \{ m_{-},  m_{+} \}$, then we pick $\beta'$ as follows:
\be \label{def:beta'^l+1}
0 < \beta' =  \min\{ \frac{\bar \beta}{4}, \beta \} \times \hat{m}.
\ee
Hence, $\beta' \leq \beta \hat{m}$ and \eqref{Uest:rho_12^l+1} shows that
\be \notag
\begin{split}
\| e^{\beta' \frac{g}{2c} x_3} \rho^{\ell} (x) \|_\infty
& \leq \frac{1}{\beta} \big( e_+ \| w_{+, \beta} G_+ \|_{L^\infty(\gamma_-)} + e_{-} \| w_{-, \beta} G_- \|_{L^\infty(\gamma_-)} \big) < \infty, \\
\| e^{\beta' \frac{g}{2c} x_3} \rho^{\ell-1} (x) \|_\infty
& \leq \frac{1}{\beta} \big( e_+ \| w_{+, \beta} G_+ \|_{L^\infty(\gamma_-)} + e_{-} \| w_{-, \beta} G_- \|_{L^\infty(\gamma_-)} \big) < \infty.
\end{split}
\ee
Set $A = \| e^{\beta' \frac{g}{2c} x_3} (\rho^{\ell} - \rho^{\ell-1} ) \|_\infty$ and $B = \beta' \frac{g}{2c}$, we get 
\be \label{est:rho_12^l+1}
| \rho^{\ell} - \rho^{\ell-1} | \leq A e^{- B x_3}.
\ee
Using \eqref{est:nabla_phi} in Lemma \ref{lem:rho_to_phi}, together with $- \Delta ( \Phi^{\ell} - \Phi^{\ell-1} ) = \rho^{\ell} - \rho^{\ell-1}$, we obtain
\Be \label{est:phi_12_x^l+1}
\begin{split}
\| \nabla_x \Phi^{\ell} - \nabla_x \Phi^{\ell-1} \|_\infty 
&\leq \mathfrak{C} A \big( 1 +  \frac{1}{B} \big)
\\& \leq \mathfrak{C} (1 + \frac{2 c}{\beta^\prime g} )
\| e^{\beta' \frac{g}{2c} x_3} ( \rho^{\ell} - \rho^{\ell-1} ) \|_\infty.
\end{split}
\Ee
Now setting $\hat{\beta} = \frac{\beta'}{\min \{ m_{-},  m_{+} \}}$, we bound $\rho^{\ell} - \rho^{\ell-1}$ by
\Be \notag
\begin{split}
\rho^{\ell} - \rho^{\ell-1}
& = \int_{\R^3} e_+ (h^{\ell}_{+} - h^{\ell-1}_{+} ) + e_{-} ( h^{\ell}_{-} - h^{\ell-1}_{-} ) \dd p
\\& \leq e_+ \| w^{\ell}_{+, \hat{\beta}} (h^{\ell}_{+} - h^{\ell-1}_{+} ) \|_{L^{\infty} (\O \times \R^3)} \times \int_{\R^3} \frac{1}{w^{\ell}_{+, \hat{\beta}}} \dd p
\\& \ \ \ \ + e_- \| w^{\ell}_{-, \hat{\beta}} (h^{\ell}_{-} - h^{\ell-1}_{-}) \|_{L^{\infty} (\O \times \R^3)} \times \int_{\R^3} \frac{1}{w^{\ell}_{-, \hat{\beta}}} \dd p.
\end{split}
\Ee
This, together with \eqref{lower_w^l+1}, shows that
\Be \notag
\begin{split}
| \rho^{\ell} - \rho^{\ell-1} |
& \leq e_+ \| w^{\ell}_{+, \hat{\beta}} (h^{\ell}_{+} - h^{\ell-1}_{+} ) \|_{L^{\infty} (\O \times \R^3)} \times \frac{1}{\hat{\beta}} e^{- \hat{\beta} \frac{m _+ g}{2c} x_3}
\\& \ \ \ \ + e_- \| w^{\ell}_{-, \hat{\beta}} (h^{\ell}_{-} - h^{\ell-1}_{-}) \|_{L^{\infty} (\O \times \R^3)} \times \frac{1}{\hat{\beta}} e^{- \hat{\beta} \frac{m_- g}{2c} x_3}
\\& \leq e_+ \| w^{\ell}_{+, \hat{\beta}} (h^{\ell}_{+} - h^{\ell-1}_{+} ) \|_{L^{\infty} (\O \times \R^3)} \times \frac{1}{\hat{\beta}} e^{- \beta' \frac{g}{2c} x_3}
\\& \ \ \ \ + e_- \| w^{\ell}_{-, \hat{\beta}} (h^{\ell}_{-} - h^{\ell-1}_{-}) \|_{L^{\infty} (\O \times \R^3)} \times \frac{1}{\hat{\beta}} e^{- \beta' \frac{g}{2c} x_3}.
\end{split}
\Ee
Note that \eqref{def:beta'^l+1} implies $\hat{\beta} \leq \frac{\bar{\beta}}{4}$. Further, using \eqref{upper_w^l+1}, we derive that
\Be \label{est:e^beta*rho_12^l+1}
\begin{split}
e^{\beta' \frac{g}{2c} x_3} | \rho^{\ell} - \rho^{\ell-1} |
& \leq e_+ \| e^{ \hat{\beta} \big( \sqrt{(m_+ c)^2 + |p|^2} + \frac{3}{2 c} m_+ g x_3 \big) }  (h^{\ell}_{+} - h^{\ell-1}_{+} ) \|_{L^{\infty} (\O \times \R^3)} \times \frac{1}{\hat{\beta}}
\\& \ \ \ \ + e_- \| e^{ \hat{\beta} \big( \sqrt{(m_- c)^2 + |p|^2} + \frac{3}{2 c} m_- g x_3 \big) } (h^{\ell}_{-} - h^{\ell-1}_{-}) \|_{L^{\infty} (\O \times \R^3)} \times \frac{1}{\hat{\beta}}
\\& \leq e_+ \| e^{ \frac{\bar{\beta}}{4} \big( \sqrt{(m_+ c)^2 + |p|^2} + \frac{3}{2 c} m_+ g x_3 \big) }  (h^{\ell}_{+} - h^{\ell-1}_{+} ) \|_{L^{\infty} (\O \times \R^3)} \times \frac{1}{\hat{\beta}}
\\& \ \ \ \ + e_- \| e^{ \frac{\bar{\beta}}{4} \big( \sqrt{(m_- c)^2 + |p|^2} + \frac{3}{2 c} m_- g x_3 \big) } (h^{\ell}_{-} - h^{\ell-1}_{-}) \|_{L^{\infty} (\O \times \R^3)} \times \frac{1}{\hat{\beta}}.
\end{split}
\Ee
Inputting \eqref{est:e^beta*rho_12^l+1} into \eqref{est:phi_12_x^l+1}, we obtain
\Be \label{est2:phi_12_x^l+1}
\begin{split}
& \| \nabla_x \Phi^{\ell} - \nabla_x \Phi^{\ell-1} \|_\infty 
\\& \leq \mathfrak{C} (1 + \frac{2 c}{\beta^\prime g} ) \frac{1}{\hat{\beta}} \Big(
e_+ \| e^{ \frac{\bar{\beta}}{4} \big( \sqrt{(m_+ c)^2 + |p|^2} + \frac{3}{2 c} m_+ g x_3 \big) } (h^{\ell}_{+} - h^{\ell-1}_{+}) \|_{L^{\infty} (\O \times \R^3)}
\\& \qquad \qquad \qquad \qquad + e_- \| e^{ \frac{\bar{\beta}}{4} \big( \sqrt{(m_- c)^2 + |p|^2} + \frac{3}{2 c} m_- g x_3 \big) } (h^{\ell}_{-} - h^{\ell-1}_{-}) \|_{L^{\infty} (\O \times \R^3)}  \Big).
\end{split}
\Ee
Using \eqref{bound3:diff_h^l+1}, together with \eqref{est2:phi_12_x^l+1}, we bound
\be \label{bound4:diff_h^l+1}
\begin{split}
& e^{ \frac{3 \bar \beta}{4} \big( \sqrt{(m_+ c)^2 + |p|^2} + \frac{1}{2 c} m_+ g x_3 \big) } |h^{\ell+1}_{+} (x, p) - h^{\ell}_{+} (x, p)|
\\& \ \ \ \ + e^{ \frac{3 \bar \beta}{4} \big( \sqrt{(m_- c)^2 + |p|^2} + \frac{1}{2 c} m_- g x_3 \big) } |h^{\ell+1}_{-} (x, p) - h^{\ell}_{-} (x, p)|
\\& \leq \mathfrak{C} (1 + \frac{2 c}{\beta^\prime g} ) \frac{1}{\hat{\beta}} \Big(
e_+ \| e^{ \frac{\bar{\beta}}{4} \big( \sqrt{(m_+ c)^2 + |p|^2} + \frac{3}{2 c} m_+ g x_3 \big) } (h^{\ell}_{+} - h^{\ell-1}_{+}) \|_{L^{\infty} (\O \times \R^3)}
\\& \qquad \qquad \qquad \qquad + e_- \| e^{ \frac{\bar{\beta}}{4} \big( \sqrt{(m_- c)^2 + |p|^2} + \frac{3}{2 c} m_- g x_3 \big) } (h^{\ell}_{-} - h^{\ell-1}_{-}) \|_{L^{\infty} (\O \times \R^3)}  \Big)
\\& \ \ \ \ \times \big( \frac{48 e_+}{m g \bar \beta} \| w^{\ell+1}_{+, \bar\beta} \nabla_p h^{\ell}_{+} \|_\infty + \frac{48 e_-}{m g \bar \beta} \| w^{\ell+1}_{-, \bar\beta} \nabla_p h^{\ell}_{-} \|_\infty \big).
\end{split}
\ee
Together with $\beta'$ in \eqref{def:beta'^l+1} and $\hat{\beta} = \frac{\beta'}{\min \{ m_{-},  m_{+} \}}$, we get
\be \label{bound5:diff_h^l+1}
\begin{split}
\eqref{bound4:diff_h^l+1}
& \leq \mathfrak{C} (1 + \frac{2 c}{\beta^\prime g} ) \frac{ \max\{ e_+, e_- \} }{\hat{\beta}} \big( \frac{48 e_+}{m g \bar \beta} \| w^{\ell+1}_{+, \bar\beta} \nabla_p h^{\ell}_{+} \|_\infty + \frac{48 e_-}{m g \bar \beta} \| w^{\ell+1}_{-, \bar\beta} \nabla_p h^{\ell}_{-} \|_\infty \big)
\\& \ \ \ \ \times \Big( \| e^{ \frac{3 \bar{\beta}}{4} \big( \sqrt{(m_+ c)^2 + |p|^2} + \frac{1}{2 c} m_+ g x_3 \big) } (h^{\ell}_{+} - h^{\ell-1}_{+} ) \|_{L^{\infty} (\O \times \R^3)} 
\\& \qquad \qquad \qquad + \| e^{ \frac{3 \bar{\beta}}{4} \big( \sqrt{(m_- c)^2 + |p|^2} + \frac{1}{2 c} m_- g x_3 \big) } (h^{\ell}_{-} - h^{\ell-1}_{-}) \|_{L^{\infty} (\O \times \R^3)}  \Big)
\\& \lesssim_{\bar \beta, \beta'} 
\big\{ \| w^{\ell+1}_{+, \bar\beta} \nabla_p h^{\ell}_{+} \|_\infty + \| w^{\ell+1}_{-, \bar\beta} \nabla_p h^{\ell}_{-} \|_\infty \big\}
\\& \qquad \qquad \times \Big( \| e^{ \frac{3 \bar{\beta}}{4} \big( \sqrt{(m_+ c)^2 + |p|^2} + \frac{1}{2 c} m_+ g x_3 \big) } (h^{\ell}_{+} - h^{\ell-1}_{+} ) \|_{L^{\infty} (\O \times \R^3)} 
\\& \qquad \qquad \qquad + \| e^{ \frac{3 \bar{\beta}}{4} \big( \sqrt{(m_- c)^2 + |p|^2} + \frac{1}{2 c} m_- g x_3 \big) } (h^{\ell}_{-} - h^{\ell-1}_{-} ) \|_{L^{\infty} (\O \times \R^3)}  \Big).
\end{split}
\ee
Under the condition \eqref{weight_G_xv}, together with \eqref{est1:h_v^l+1}, we derive for any $\ell \geq 1$,
\be \notag
w^{\ell+1}_{\pm, \bar{\beta}} (x,p)  | \nabla_p h^{\ell}_{\pm} (x,p)| 
\lesssim (1 + \frac{\tilde{\beta}}{m g} ) \| e^{\tilde \beta \sqrt{(m_{\pm} c)^2 + |p|^2}} \nabla_{x_\parallel, p} G_{\pm} \|_{L^\infty (\gamma_-)} \leq \frac{1}{4}.
\ee 
Thus, we bound \eqref{bound5:diff_h^l+1} by
\be \label{bound6:diff_h^l+1}
\begin{split}
\eqref{bound5:diff_h^l+1} 
& \leq \frac{1}{2} \Big( \| e^{ \frac{3 \bar{\beta}}{4} \big( \sqrt{(m_+ c)^2 + |p|^2} + \frac{1}{2 c} m_+ g x_3 \big) } (h^{\ell}_{+} - h^{\ell-1}_{+}) \|_{L^{\infty} (\O \times \R^3)} 
\\& \qquad \qquad \qquad + \| e^{ \frac{3 \bar{\beta}}{4} \big( \sqrt{(m_- c)^2 + |p|^2} + \frac{1}{2 c} m_- g x_3 \big) } (h^{\ell}_{-} - h^{\ell-1}_{-} ) \|_{L^{\infty} (\O \times \R^3)}  \Big).
\end{split}
\ee
Finally, together with \eqref{bound4:diff_h^l+1}-\eqref{bound6:diff_h^l+1}, we derive that
\Be \notag
\begin{split}
& \| e^{ \frac{3 \bar{\beta}}{4} \big( \sqrt{(m_+ c)^2 + |p|^2} + \frac{1}{2 c} m_+ g x_3 \big) } (h^{\ell+1}_{+} (x, p) - h^{\ell}_{+} (x, p)) \|_{L^{\infty} (\O \times \R^3)} 
\\& \qquad \qquad + \| e^{ \frac{3 \bar{\beta}}{4} \big( \sqrt{(m_- c)^2 + |p|^2} + \frac{1}{2 c} m_- g x_3 \big) } (h^{\ell+1}_{-} (x, p) - h^{\ell}_{-} (x, p) ) \|_{L^{\infty} (\O \times \R^3)}
\\& \leq \frac{1}{2} \Big( \| e^{ \frac{3 \bar{\beta}}{4} \big( \sqrt{(m_+ c)^2 + |p|^2} + \frac{1}{2 c} m_+ g x_3 \big) } (h^{\ell}_{+} - h^{\ell-1}_{+} ) \|_{L^{\infty} (\O \times \R^3)} 
\\& \qquad \qquad + \| e^{ \frac{3 \bar{\beta}}{4} \big( \sqrt{(m_- c)^2 + |p|^2} + \frac{1}{2 c} m_- g x_3 \big) } (h^{\ell}_{-} - h^{\ell-1}_{-} ) \|_{L^{\infty} (\O \times \R^3)}  \Big),
\end{split}
\Ee
and conclude \eqref{est:h_cauchy}. 

\smallskip

Next, using the condition \eqref{condition:beta} and \eqref{Uest:wh^k} in Proposition \ref{prop:Unif_steady}, we have, for any $k \geq 0$,
\be
\| w^{k+1}_{\pm, \beta} h_{\pm}^{k+1} \|_{L^\infty (\bar \O \times \R^3)} 
\leq \| e^{ \beta \sqrt{(m_{\pm} c)^2 + |p|^2}} G_{\pm} \|_{L^\infty (\gamma_-)} < \infty.
\ee
Since $\bar{\beta} = \frac{\tilde \beta}{6}$ and $0 < \tilde \beta \leq \beta$, we obtain that for any $k \geq 0$,
\be
\begin{split}
& \| e^{ \frac{3 \bar{\beta}}{4} \big( \sqrt{(m_{\pm} c)^2 + |p|^2} + \frac{1}{2 c} m_{\pm} g x_3 \big) } (h^{k}_{{\pm}} - h^{k-1}_{{\pm}} ) \|_{L^{\infty} (\O \times \R^3)} 
\leq 2 \| w^{k+1}_{\pm, \beta} h_{\pm}^{k+1}   \|_{L^\infty (\bar \O \times \R^3)}  < \infty.
\end{split}
\ee
Together with \eqref{est:h_cauchy}, we derive for any $\ell \geq 0$,
\Be \label{bound7:diff_h^l+1}
\begin{split}
& \| e^{ \frac{3 \bar{\beta}}{4} \big( \sqrt{(m_+ c)^2 + |p|^2} + \frac{1}{2 c} m_+ g x_3 \big) } (h^{\ell+1}_{+} (x, p) - h^{\ell}_{+} (x, p)) \|_{L^{\infty} (\O \times \R^3)} 
\\& \qquad \qquad + \| e^{ \frac{3 \bar{\beta}}{4} \big( \sqrt{(m_- c)^2 + |p|^2} + \frac{1}{2 c} m_- g x_3 \big) } (h^{\ell+1}_{-} (x, p) - h^{\ell}_{-} (x, p) ) \|_{L^{\infty} (\O \times \R^3)}
\\& \leq \frac{1}{2^{\ell-1}} \Big( \| e^{ \frac{3 \bar{\beta}}{4} \big( \sqrt{(m_+ c)^2 + |p|^2} + \frac{1}{2 c} m_+ g x_3 \big) } (h^{2}_{+} - h^{1}_{+} ) \|_{L^{\infty} (\O \times \R^3)} 
\\& \qquad \qquad \qquad + \| e^{ \frac{3 \bar{\beta}}{4} \big( \sqrt{(m_- c)^2 + |p|^2} + \frac{1}{2 c} m_- g x_3 \big) } (h^{2}_{-} - h^{1}_{-} ) \|_{L^{\infty} (\O \times \R^3)}  \Big)
\\& \leq \frac{2}{2^{\ell-1}} \big( \| e^{ \beta \sqrt{(m_{+} c)^2 + |p|^2}} G_{+} \|_{L^\infty +  (\gamma_-)} + \| e^{ \beta \sqrt{(m_{-} c)^2 + |p|^2}} G_{-} \|_{L^\infty (\gamma_-)} \big).
\end{split}
\Ee
Hence, we conclude that $\{ h^{\ell+1} \}^{\infty}_{\ell=0}$ forms a Cauchy sequences in $L^{\infty} (\O \times \R^3)$.

\smallskip

Now inputting \eqref{bound7:diff_h^l+1} into \eqref{est:e^beta*rho_12^l+1}, we deduce for any $\ell \geq 0$,
\Be \label{est:cauchy_rho}
\begin{split}
& e^{\beta' \frac{g}{2c} x_3} | \rho^{\ell} - \rho^{\ell-1} |
\\& \leq e_+ \| e^{ \frac{\bar{\beta}}{4} \big( \sqrt{(m_+ c)^2 + |p|^2} + \frac{3}{2 c} m_+ g x_3 \big) }  (h^{\ell}_{+} - h^{\ell-1}_{+} ) \|_{L^{\infty} (\O \times \R^3)} \times \frac{1}{\hat{\beta}}
\\& \ \ \ \ + e_- \| e^{ \frac{\bar{\beta}}{4} \big( \sqrt{(m_- c)^2 + |p|^2} + \frac{3}{2 c} m_- g x_3 \big) } (h^{\ell}_{-} - h^{\ell-1}_{-}) \|_{L^{\infty} (\O \times \R^3)} \times \frac{1}{\hat{\beta}}
\\& \leq \frac{2}{2^{\ell-1}} \big( \| e^{ \beta \sqrt{(m_{+} c)^2 + |p|^2}} G_{+} \|_{L^\infty +  (\gamma_-)} + \| e^{ \beta \sqrt{(m_{-} c)^2 + |p|^2}} G_{-} \|_{L^\infty (\gamma_-)} \big) \times \frac{e_+ + e_-}{\hat{\beta}}.
\end{split}
\Ee
Similarly, we input \eqref{bound7:diff_h^l+1} into \eqref{est2:phi_12_x^l+1}, and obtain for any $\ell \geq 0$,
\Be \label{est:cauchy_phi_x}
\begin{split}
& \| \nabla_x \Phi^{\ell} - \nabla_x \Phi^{\ell-1} \|_\infty 
\\& \leq \mathfrak{C} (1 + \frac{2 c}{\beta^\prime g} ) \frac{1}{\hat{\beta}} \Big(
e_+ \| e^{ \frac{\bar{\beta}}{4} \big( \sqrt{(m_+ c)^2 + |p|^2} + \frac{3}{2 c} m_+ g x_3 \big) } (h^{\ell}_{+} - h^{\ell-1}_{+}) \|_{L^{\infty} (\O \times \R^3)}
\\& \qquad \qquad \qquad \qquad + e_- \| e^{ \frac{\bar{\beta}}{4} \big( \sqrt{(m_- c)^2 + |p|^2} + \frac{3}{2 c} m_- g x_3 \big) } (h^{\ell}_{-} - h^{\ell-1}_{-}) \|_{L^{\infty} (\O \times \R^3)}  \Big)
\\& \leq \frac{2}{2^{\ell-1}}  \mathfrak{C} (1 + \frac{2 c}{\beta^\prime g} ) \frac{e_+ + e_-}{\hat{\beta}}
 \big( \| e^{ \beta \sqrt{(m_{+} c)^2 + |p|^2}} G_{+} \|_{L^\infty +  (\gamma_-)} + \| e^{ \beta \sqrt{(m_{-} c)^2 + |p|^2}} G_{-} \|_{L^\infty (\gamma_-)} \big).
\end{split}
\Ee
Therefore, we conclude that $\{ \rho^\ell \}^{\infty}_{\ell=0}$ and $\{ \nabla_x \Phi^\ell \}^{\infty}_{\ell=0}$ are both Cauchy sequences in $L^{\infty} (\O)$.
\end{proof}

Finally, using the Cauchy sequences of $L^\infty$-spaces in Proposition \ref{prop:cauchy}, we construct a weak solution of $(h_{\pm}, \rho, \Phi)$ solving \eqref{VP_h}-\eqref{eqtn:Dphi}. 

\begin{proof}[\textbf{Proof of Theorem \ref{theo:CS}}]

From the assumption, then for some $g, \beta > 0$,
\be \notag
\| e^{ \beta \sqrt{(m_{\pm} c)^2 + |p|^2} } G_{\pm} \|_{L^\infty(\gamma_-)} < \infty, 
\ee
and for some $\beta \geq \tilde{\beta} > 0$,
\be \notag
\| e^{{\tilde \beta } \sqrt{(m_{\pm} c)^2 + |p|^2}} \nabla_{x_\parallel, p} G_{\pm} (x,p) \|_{L^\infty (\gamma_-)} < \infty,
\ee
and the conditions \eqref{condition:beta}, \eqref{condition:tilde_beta} and \eqref{condition:G_xv} hold for $g, \beta, \tilde{\beta} > 0$.

\smallskip

\textbf{Step 1. Regularity: Proof of \eqref{Uest:wh}-\eqref{Uest:DPhi}.}
From \eqref{est:h_cauchy}, \eqref{est:cauchy_rho} and \eqref{est:cauchy_phi_x}, the arguments of the Cauchy sequences of $L^\infty$-spaces in Proposition \ref{prop:cauchy}, there exists 
\be \notag
h_{\pm} (x, p) \in L^\infty (\bar \O \times \R^3),
\ \text{ and } \ \rho(x), \nabla_x \Phi (x) \in L^\infty (\bar \O) 
\ \text{ with } \ 
\Phi= 0 \ \ \text{on} \ \p\O,
\ee
such that as $k \to \infty$,
\begin{align}
e^{ \frac{3 \bar{\beta}}{4} ( \sqrt{(m_{\pm} c)^2 + |p|^2} + \frac{1}{2 c} m_{\pm} g x_3) } h^{k}_{\pm}
\to e^{ \frac{3 \bar{\beta}}{4} ( \sqrt{(m_{\pm} c)^2 + |p|^2} + \frac{1}{2 c} m_{\pm} g x_3 ) } h_{\pm}
& \ \text{ in } \ L^\infty (\bar \O \times \R^3),
\label{weakconv_whst} \\
e^{\beta' \frac{g}{2c} x_3} \rho^{k} \to e^{\beta' \frac{g}{2c} x_3} \rho 
& \ \text{ in } \ L^\infty (\bar \O),
\label{weakconv_rhost} \\
\nabla_x \Phi^k \to \nabla_x \Phi
& \ \text{ in } \ L^\infty (\bar \O),
\label{ae_converge_par_Phist}
\end{align} 
where $\bar{\beta} = \frac{\tilde \beta}{6}$ and $\beta' =  \min\{ \frac{\bar \beta}{4}, \beta \} \times \min \{ m_{-},  m_{+} \}$. This clearly implies that
\be \label{strong_conv_h_rho}
\begin{split}
h^{k}_{\pm} \to h_{\pm}
&\ \text{ in } \ L^\infty (\bar \O \times \R^3) 
\ \text{ as } \ k \to \infty, 
\\ \rho^{k} \to \rho 
& \ \text{ in } \ L^\infty (\bar \O)
\ \text{ as } \ k \to \infty.
\end{split}
\ee
Using \eqref{Uest:h^k}, \eqref{Uest:rho^k} and \eqref{Uest:DPhi^k} in Proposition \ref{prop:Unif_steady}, together with the $L^\infty$ convergence in \eqref{weakconv_whst}-\eqref{ae_converge_par_Phist}, we prove that $(h_{\pm}, \rho, \nabla_x \Phi)$ satisfies \eqref{Uest:rho} and \eqref{Uest:DPhi}.

Then, from zero Dirichlet boundary condition of \eqref{bdry:phik} and $\Phi= 0$ on $\p\O$, together with \eqref{ae_converge_par_Phist}, we obtain
\be \label{ae_converge_Phist}
\Phi^{k} (x) \to \Phi(x)
\ \text{ a.e. in $\O$}
\ \ \text{as} \ \ k \rightarrow \infty.
\ee
We consider the weight functions $w_{\pm} (x, p)$ as follows:
\Be \notag
w_{\pm, \beta} (x,p) = e^{ \beta \left( \sqrt{(m_{\pm} c)^2 + |p|^2} + \frac{1}{c} ( e_{\pm} \Phi (x) + m_{\pm} g x_3 ) \right) }.
\Ee
Using \eqref{ae_converge_Phist}, we have
\be \label{wi_conv_w}
w^{k}_{\pm, \beta} (x,p) = e^{ \beta \left( \sqrt{(m_{\pm} c)^2 + |p|^2} + \frac{1}{c} ( e_{\pm} \Phi^{k} (x) + m_{\pm} g x_3 ) \right) }
\rightarrow  w_{\pm, \beta} (x,p)
\ \text{ a.e. in $\O$}
\ \ \text{as} \ \ k \rightarrow \infty.
\ee
From \eqref{strong_conv_h_rho} and \eqref{wi_conv_w}, we get that
\Be \notag
w^{k}_{\pm, \beta} (x,p) h^k_{\pm}
\to w_{\pm, \beta} (x,p) h_{\pm}
\ \text{ a.e. in $\O$}
\ \ \text{as} \ \ k \rightarrow \infty.
\Ee
Thus, we conclude \eqref{Uest:wh} by using \eqref{Uest:wh^k}.

\smallskip

\textbf{Step 2. Existence.}
Next we will prove that $(h_{\pm}, \rho, \Phi)$ obtained in 
\eqref{weakconv_whst}-\eqref{ae_converge_par_Phist} is the solution to \eqref{VP_h}-\eqref{eqtn:Dphi} in the sense of Definition \ref{weak_sol}.

First, we show $(h_{\pm}, \nabla_x \Phi)$ is a weak solution of \eqref{VP_h}-\eqref{bdry:h} and $(\rho, \Phi)$ is a weak solution of \eqref{eqtn:Dphi}.
Recall from the construction and Proposition \ref{prop:Unif_steady}, for any $k \in \N$, $h^{k}_{\pm} (x,p)$ is the weak solution to \eqref{eqtn:hk} and \eqref{bdry:hk} with the field containing $\nabla_x \Phi^{k}$.
Therefore, given any $k \in \N$, we have, for any $\psi \in  C^\infty_{c} (\bar \O \times \R^3)$,
\begin{align}
& \ \ \ \ \iint_{\O \times \R^3} h^{k+1}_{\pm} v_{\pm} \cdot \nabla_x \psi \dd p \dd x 
- \big( \int_{\gamma_+} h^{k+1}_{\pm} \psi \dd \gamma
- \int_{\gamma_-} G_{\pm} \psi \dd \gamma \big) \label{weak_stk} \\
& = \iint_{\O \times \R^3} h^{k+1}_{\pm} \big( e_{\pm} ( \frac{v_{\pm}}{c} \times B 
- \nabla_x \Phi^{k} ) - \nabla_x (m_{\pm} g x_3) \big) \cdot \nabla_p \psi \dd p \dd x, \label{weak_stk_2}
\end{align}
Moreover, from \eqref{eqtn:phik} and \eqref{bdry:phik}, then for any $\varphi \in H^1_0 (\O) \cap C^\infty_c (\bar \O)$, 
\Be \label{weak_Poisson_k}
\int_{\O} \nabla_x \Phi^{k} \cdot \nabla_x \varphi \dd x 
= \int_{\O} \rho^{k} \varphi \dd x.
\Ee
Suppose two test functions $\psi (x, p), \varphi (x)$ have compact support in $x$ as follows:
\Be \notag
\begin{split}
\text{spt}_x (\psi)&:= \overline{  \{x \in \O:  \psi(x, p) \neq 0 \ \text{for some } p \in \R^3\}} \subset\subset \O,
\\ \text{spt} (\varphi)&:= \overline{\{x \in \O:  \varphi(x) \neq0 \}} \subset\subset \O.
\end{split}
\Ee

For the first term in \eqref{weak_stk}, from $\psi \in  C^\infty_{c} (\bar \O \times \R^3)$ and $L^{\infty}$ convergence $h^{k}_{\pm} \to h_{\pm}$ in \eqref{strong_conv_h_rho}, we derive that
\Be \notag
\iint_{\O \times \R^3} h^{k + 1}_{\pm} v_{\pm} \cdot \nabla_x \psi \dd p \dd x
\rightarrow 
\iint_{\O \times \R^3} h_{\pm} v_{\pm} \cdot \nabla_x \psi \dd p \dd x
\ \text{ as } \ k \to \infty.
\Ee
Moreover, using the $L^{\infty}$ convergence $\nabla_x \Phi^k \to \nabla_x \Phi$ in \eqref{ae_converge_par_Phist}, together with $\psi \in  C^\infty_{c} (\bar \O \times \R^3)$, we obtain
\Be \label{conv:hDPDpsi}
\iint_{\O \times \R^3} h^{k+1}_{\pm} \nabla_x \Phi^{k} \cdot \nabla_p \psi \dd p \dd x
\rightarrow 
\iint_{\O \times \R^3} h_{\pm} \nabla_x \Phi \cdot \nabla_p \psi \dd p \dd x
\ \text{ as } \ k \to \infty.
\Ee
Analogously, we deduce the convergence of every term in \eqref{weak_stk} and \eqref{weak_stk_2}.
From the convergence of every other term in \eqref{weak_stk} and \eqref{weak_stk_2}, we conclude that $(h_{\pm}, \nabla_x \Phi)$ solves \eqref{weak_form}, that is, it is a weak solution of \eqref{VP_h}-\eqref{bdry:h}.

On the other side, from the $L^{\infty}$ convergence $\nabla_x \Phi^k \to \nabla_x \Phi$ in \eqref{ae_converge_par_Phist}, together with $\varphi \in H^1_0 (\O) \cap C^\infty_c (\bar \O)$, we obtain
\be \label{conv:DPhiDphi}
\int_{\O} \nabla_x \Phi^{k} \cdot \nabla_x \varphi \dd x 
\to \int_{\O} \nabla_x \Phi \cdot \nabla_x \varphi \dd x
\ \text{ as } \ k \to \infty.
\ee
Using the $L^{\infty}$ convergence $\rho^k \to \rho$ in \eqref{strong_conv_h_rho}, we further have
\be \label{conv:rhophi}
\int_{\O} \rho^{k} \varphi \dd x 
\to \int_{\O} \rho \varphi \dd x
\ \text{ as } \ k \to \infty.
\ee
Inputting \eqref{conv:DPhiDphi} and \eqref{conv:rhophi} into \eqref{weak_Poisson_k}, we conclude $(\rho, \Phi)$ is a weak solution of \eqref{eqtn:Dphi}. 

\smallskip

Second, we show $(h_{\pm}, \rho)$ is a weak solution of \eqref{def:rho}.
Recall from the construction, for any $k \in \N$, $(h^k_{\pm}, \rho^k)$ is the weak solution to \eqref{eqtn:rhok}. Therefore, given any $k \in \N$, then for any $\vartheta \in H^1_0 (\O) \cap C^\infty_c (\bar \O)$, 
\Be \label{weak_rho_k}
\int_{\O} \vartheta (x) \rho^{k} (x) \dd x 
= \int_{\O} \vartheta (x) \int _{\R^3} ( e_+ h_{+}^{k} (x, p) + e_{-} h^{k}_{-} (x, p) ) \dd p \dd x.
\Ee
Using the $L^{\infty}$ convergence $\rho^k \to \rho$ in \eqref{strong_conv_h_rho}, together with $\vartheta \in H^1_0 (\O) \cap C^\infty_c (\bar \O)$, we get
\be \label{conv:rhotheta}
\int_{\O} \vartheta \rho^{k}  \dd x 
\to \int_{\O} \vartheta \rho \dd x
\ \text{ as } \ k \to \infty.
\ee

Now we claim that for any $\vartheta \in H^1_0 (\O) \cap C^\infty_c (\bar \O)$, 
\Be \label{conv:theta_hk+}
\begin{split}
\int_{\O} \vartheta (x) \int _{\R^3} h_{+}^{k} (x, p) \dd p \dd x  
\to \int_{\O} \vartheta (x) \int _{\R^3} h_{+} (x, p) \dd p \dd x
\ \text{ as } \ k \to \infty.
\end{split}
\Ee
For the left-hand side of \eqref{conv:theta_hk+}, using \eqref{weakconv_whst}, we have
\be \notag
\begin{split}
& \int_{\O} \vartheta (x) \int _{\R^3} h_{+}^{k} (x, p) \dd p \dd x 
\\& = \int_{\O} \int_{\R^3} e^{ \frac{3 \bar{\beta}}{4} ( \sqrt{(m_{+} c)^2 + |p|^2} + \frac{1}{2 c} m_{+} g x_3) } h_{+}^{k}
\frac{\vartheta (x)}{e^{ \frac{3 \bar{\beta}}{4} ( \sqrt{(m_{+} c)^2 + |p|^2} + \frac{1}{2 c} m_{+} g x_3) } } \dd p \dd x.
\end{split}
\ee
Since $\vartheta \in H^1_0 (\O) \cap C^\infty_c (\bar \O)$, we can check that
\be \label{est:weight_theta}
\int_{\O} \int_{\R^3} \frac{|\vartheta (x)|}{e^{ \frac{3 \bar{\beta}}{4} ( \sqrt{(m_{+} c)^2 + |p|^2} + \frac{1}{2 c} m_{+} g x_3) } } \dd p \dd x < \infty.
\ee
We repeat the above steps on the right-hand side of \eqref{conv:theta_hk+}, then 
\be \label{est:theta_hk+}
\begin{split}
& \Big| \int_{\O} \vartheta (x) \int _{\R^3} h_{+}^{k} (x, p) \dd p \dd x - \int_{\O} \vartheta (x) \int _{\R^3} h_{+} (x, p) \dd p \dd x \Big|
\\& = \Big| \int_{\O} \int _{\R^3} \big( e^{ \frac{3 \bar{\beta}}{4} ( \sqrt{(m_{\pm} c)^2 + |p|^2} + \frac{1}{2 c} m_{+} g x_3) } h^{k}_{+}
- e^{ \frac{3 \bar{\beta}}{4} ( \sqrt{(m_{+} c)^2 + |p|^2} + \frac{1}{2 c} m_{+} g x_3 ) } h_{+} \big)
\\& \qquad \qquad \qquad \times \frac{\vartheta (x)}{e^{ \frac{3 \bar{\beta}}{4} ( \sqrt{(m_{+} c)^2 + |p|^2} + \frac{1}{2 c} m_{+} g x_3) } } \dd p \dd x  \Big|
\\& \leq \| e^{ \frac{3 \bar{\beta}}{4} ( \sqrt{(m_{\pm} c)^2 + |p|^2} + \frac{1}{2 c} m_{+} g x_3) } h^{k}_{+}
- e^{ \frac{3 \bar{\beta}}{4} ( \sqrt{(m_{+} c)^2 + |p|^2} + \frac{1}{2 c} m_{+} g x_3 ) } h_{+} \|_{L^\infty (\bar \O \times \R^3)}
\\& \qquad \times \int_{\O} \int_{\R^3} \frac{|\vartheta (x)|}{e^{ \frac{3 \bar{\beta}}{4} ( \sqrt{(m_{+} c)^2 + |p|^2} + \frac{1}{2 c} m_{+} g x_3) } } \dd p \dd x.
\end{split}
\ee
Using \eqref{weakconv_whst}, together with \eqref{est:weight_theta}, we prove the claim.
Analogously, we can obtain a similar argument on $h^k_{-}$ and $h_{-}$ as follows:
\Be \label{conv:theta_hk-}
\begin{split}
\int_{\O} \vartheta (x) \int _{\R^3} h_{-}^{k} (x, p) \dd p \dd x  
\to \int_{\O} \vartheta (x) \int _{\R^3} h_{-} (x, p) \dd p \dd x
\ \text{ as } \ k \to \infty.
\end{split}
\Ee
Together with \eqref{conv:theta_hk+} and \eqref{conv:theta_hk-}, we deduce 
\Be \notag
\begin{split}
& \int_{\O} \vartheta (x) \int _{\R^3} ( e_+ h_{+}^{k} (x, p) + e_{-} h^{k}_{-} (x, p) ) \dd p \dd x
\\& \to \int_{\O} \vartheta (x) \int _{\R^3} ( e_+ h_{+}(x, p) + e_{-} h_{-} (x, p) ) \dd p \dd x
\ \text{ as } \ k \to \infty.
\end{split}
\Ee
Therefore we prove that 
\be \notag
\rho(x) = \int _{\R^3} ( e_+ h_{+} (x, p) + e_{-} h_{-} (x, p) ) \dd p,
\ee
and thus $(h_{\pm}, \rho)$ is a weak solution of \eqref{def:rho}.

\smallskip

\textbf{Step 3. Regularity: Proof of \eqref{Uest:Phi_xx}-\eqref{est_final:hk_x}.}
Recall from \eqref{strong_conv_h_rho}, we have the $L^{\infty}$ convergence $\rho^k \to \rho$. This implies that
\be \notag
|\rho |_{C^{0,\delta}(\O)} \leq \sup_{k \in \N} |\rho^{k+1} |_{C^{0,\delta}(\O)}.
\ee
Together with \eqref{Uest:rho}, \eqref{est:rho_Hol^k+1} and \eqref{est:nabla^2phi}, we conclude that $\Phi$ constructed in \eqref{ae_converge_Phist} satisfies \eqref{Uest:Phi_xx}.
From Theorem \ref{theo:RS}, together with \eqref{condition:tilde_beta}, \eqref{condition:G_xv} and \eqref{Uest:Phi_xx}, one can check that the solution $(h_{\pm}, \rho, \nabla_x \Phi)$ satisfies \eqref{est_final:rho_x}, \eqref{est_final:phi_C2}, \eqref{est_final:hk_v}, and \eqref{est_final:hk_x}.

\smallskip

\textbf{Step 4. Uniqueness.}
Finally, using the weighted bound on $| \nabla_p h_{\pm} (x,p)|$ in \eqref{est_final:hk_v}, we apply Theorem \ref{theo:US}, and thus conclude the uniqueness of the solution $(h_{\pm}, \rho, \Phi)$.
\end{proof}

\section{Dynamic Solutions}  
\label{sec:DS}

In this section, we construct solutions to the dynamic problem \eqref{eqtn:f}-\eqref{Poisson_f}, and study their properties such as regularity and uniqueness. 

We begin with a brief overview of the steps outlined in this section. 
Given a steady solution 
$(h, \rho, \Phi)$, we construct a sequence $(f^{\ell+1}_{\pm}, \varrho^{\ell}, \nabla_x \Psi^{\ell})$ for $\ell \geq 1$ as described in \eqref{eqtn:fell}-\eqref{bdry:Psi_fell}.
Using mathematical induction, we establish that for any $\ell \in \N$,
\be \notag
\begin{split}
\sup_{0 \leq t < \infty} \| \nabla_x \big( \Phi + \Psi^{\ell} (t) \big) \|_{L^\infty(\O)}  
	& \leq \min \left\{\frac{m_+}{e_+}, \frac{m_-}{e_-} \right\} \times \frac{g}{2},
	\\ \sup_{0 \leq t < \infty}\| \nabla_x \Psi^{\ell} (t) \|_{L^\infty(\O)} & \leq \min \left\{\frac{m_+}{e_+}, \frac{m_-}{e_-} \right\} \times \frac{g}{48},
\end{split}
\ee
and
\be \notag
\sup_{0 \leq t < \infty} \| e^{ \frac{\beta}{2} \sqrt{(m_{\pm} c)^2 + |p|^2} } e^{ \frac{m_{\pm} g}{4 c} \beta x_3} f^{\ell+1}_{\pm} (t,x,p)  \|_{L^\infty  (\O \times \R^3)} < \infty.
\ee
These results demonstrate the dominance of downward gravity in the field and provide a weighted $L^\infty$ estimate on $f^{\ell+1}_{\pm}$ in Proposition \ref{prop:DC}.

Next, suppose $f_{\pm} (t,x,p)$ solves \eqref{eqtn:f}-\eqref{Poisson_f}. We can express $f_{\pm} (t,x,p)$ as follows:
\Be \notag
\begin{split}
	f_{\pm} (t,x,p) 
	& = \mathbf{1}_{t \leq \tBp (t,x,p)} f_{\pm} (0, \Z_{\pm} (0;t,x,p) ) 
	\\& \ \ \ \ + \int^t_{ \max\{0, t - \tBp (t,x,p)\}} e_{\pm} \nabla_x \Psi (s, \X_{\pm} (s;t,x,p)) \cdot \nabla_p h_{\pm} ( \Z_{\pm} (0;t,x,p) ) \dd s. 
\end{split}
\Ee 
Using the bound on the backward exit time $\tBp$ in \eqref{est:tB}, we obtain 
\be \notag
\mathbf{1}_{t \leq \tBp (t,x,p)} e^{ - \frac{\beta}{4} \sqrt{(m_{\pm} c)^2 + |p|^2} } e^{ - \frac{m_{\pm} g}{8 c} \beta x_3}
\leq e^{- \frac{m_{\pm} g}{24} \beta (t-1) }.
\ee
Combined with the weighted $L^\infty$ estimate on $f^{\ell+1}_{\pm} (t,x,p)$, this result leads to the conclusion of asymptotic stability in Theorem \ref{theo:AS}.

The remaining steps are analogous to the steady case. We perform a priori estimates on on $(F_{\pm}, \phi_F)$ solving \eqref{VP_F}, \eqref{Poisson_F}, \eqref{VP_0}, \eqref{Dbc:F} and \eqref{bdry:F}.
By employing the modified kinetic distance function $\alpha_{\pm, F}$ in \eqref{alpha_F}, we derive the weighted $L^\infty$ estimate on $\nabla_p F_{\pm}$ in Theorem \ref{theo:RD}.
Using the a priori estimates for the sequence $(f^{\ell+1}_{\pm}, \varrho^{\ell}, \nabla_x \Psi^{\ell})$ for $\ell \geq 1$, we show that all components in the sequence are Cauchy sequences in $L^{\infty} (\R_+ \times \O \times \R^3)$ or $L^{\infty} (\R_+ \times \O)$ (Proposition \ref{prop:cauchy_dy}).
This allows us to establish strong convergence and conclude the existence of a dynamic solution.
Finally, the uniqueness follows from a similar argument as in the proof of Proposition \ref{prop:cauchy}.

\subsection{Construction} \label{sec:DC}

Suppose $(h_{\pm}, \nabla_x \Phi)$ is the steady solution to \eqref{VP_h}-\eqref{eqtn:Dphi}. We construct solutions to the dynamic problem \eqref{eqtn:f}-\eqref{Poisson_f} via the following sequences: for any $\ell \in \N$,
\be \label{eqtn:fell}
\begin{split}
& \p_t f^{\ell+1}_{\pm}  + v_\pm \cdot \nabla_x f^{\ell+1}_{\pm} + \Big( e_{\pm} \big( \frac{v_\pm}{c} \times B - \nabla_x ( \Phi + \Psi^{\ell} ) \big) - \nabla_x ( m_\pm g x_3) \Big) \cdot \nabla_p f^{\ell+1}_{\pm} 
\\& \ \ \ \ = e_{\pm} \nabla_x \Psi^{\ell} \cdot \nabla_p h_{\pm} \ \ \text{in} \ \R_+ \times  \O \times \R^3,
\end{split}
\ee
and
\begin{align}
f^{\ell+1}_{\pm} (t,x,p) = 0 
& \ \ \text{on} \ \gamma_-,
\label{bdry:fell} \\
f^{\ell+1}_{\pm, 0} (x, p) = F_{\pm, 0} - h_{\pm} 
& \ \ \text{in} \ \{ 0 \} \times \O \times \R^3, 
\label{initial:fell} \\
\varrho^{\ell} (t, x) 
=  \int_{\R^3} ( e_+ f^{\ell}_+ + e_{-} f^{\ell}_{-} ) \dd p 
& \ \ \text{in} \ \R_+ \times \O, &
\label{varrhoell}	\\
- \Delta \Psi^\ell (t, x) = \varrho^\ell (t, x) 
& \ \ \text{in} \ \R_+ \times \O,
\label{Poisson_fell} \\
\Psi^\ell (t, x) = 0 
& \ \ \text{on} \ \R_+ \times \p\O,
\label{bdry:Psi_fell}
\end{align}
where $v_\pm = \frac{p}{\sqrt{ m_\pm^2 + |p|^2/ c^2}}$ and the initial setting $f^0_{\pm} = 0$ and $(\varrho^0,  \nabla_x \Psi^0) = (0, \mathbf{0})$.

\smallskip

First, we construct $f^1_{\pm}$ solving \eqref{eqtn:fell}-\eqref{initial:fell} from $(\varrho^0, \nabla_x \Psi^0) = (0, \mathbf{0})$. 
We consider the Lagrangian formulation of the following characteristics:
\Be \label{Z1}
\Z^{1}_{\pm} (s;t,x,p) = ( \X^{1}_{\pm} (s;t,x,p), \P^{1}_{\pm} (s;t,x,p) ),
\Ee
which solves
\Be \notag 
\big( \X^{1}_{\pm} (t;t,x,p), \P^{1}_{\pm} (t;t,x,p) \big) = (x, p), 
\Ee
and
\Be \label{ODE1}
\\ \frac{d \X^1_{\pm}}{d t} = \V^1_{\pm} = \frac{ \P^1_{\pm} }{\sqrt{m^2_{\pm} + |P^1_{\pm}|^2 / c^2}}, \ \
\frac{d \P^1_{\pm} }{d t} = e_{\pm} \big( \V^1_{\pm} \times \frac{B}{c} - \nabla_x \Phi \big) - {m_{\pm}} g \mathbf{e}_3. 
\Ee
Then from \eqref{varrhoell}, as long as it is well-defined, we obtain
\[
\varrho^{1} (t, x) 
=  \int_{\R^3} ( e_+ f^{1}_+ + e_{-} f^{1}_{-} ) \dd p.
\]
Using \eqref{phi_rho} and Lemma \ref{lemma:G}, together with \eqref{Poisson_fell} and \eqref{bdry:Psi_fell}, we derive $\Psi^1$ and $\nabla_x \Psi^1$ by
\be \notag
- \Delta \Psi^1 (t, x) = \varrho^1 (t, x).
\ee

\smallskip

Second, we construct $(f^{\ell+1}_{\pm}, \varrho^{\ell}, \nabla_x \Psi^{\ell})$ solving \eqref{eqtn:fell}-\eqref{bdry:Psi_fell} for $\ell \geq 1$ by iterating the process. Given $\nabla_x \Psi^{\ell}$, we construct $f^{\ell+1}_{\pm}$ along the characteristics as follows:
\Be \label{Zell}
\Z^{\ell+1}_{\pm} (s;t,x,p) 
= (\X^{\ell+1}_{\pm} (s;t,x,p), \V^{\ell+1}_{\pm} (s;t,x,p) ), 
\Ee
which solves
\Be \notag 
\big( \X^{\ell+1}_{\pm} (t;t,x,p), \P^{\ell+1}_{\pm} (t;t,x,p) \big) = (x, p), 
\Ee
and
\Be \label{ODEell}
\begin{split}
& \frac{d \X^{\ell+1}_{\pm} }{ds} = \V^{\ell+1}_{\pm} = \frac{ \P^{\ell+1}_{\pm} }{\sqrt{m^2_{\pm} + |P^{\ell+1}_{\pm}|^2 / c^2}},   
\\& \frac{d \P^{\ell+1}_{\pm} }{ds} = e_{\pm} \big( \V^{\ell+1}_{\pm} \times \frac{B}{c} - \nabla_x ( \Phi + \Psi^\ell ) \big) - {m_{\pm}} g \mathbf{e}_3.
\end{split}
\Ee
Using the Peano theorem, together with the bound of $\varrho^\ell$ (see Proposition \ref{prop:DC}) and the continuity of $\nabla_x \Psi^\ell$ (see \eqref{phi_rho}), we conclude the existence of solutions (not necessarily unique). 
Then we obtain $\varrho^{\ell+1}$ from \eqref{varrhoell} and solve $\nabla_x \Psi^{\ell+1}$ from \eqref{Poisson_fell} and \eqref{bdry:Psi_fell}.

\smallskip

Analogous to Definition \ref{def:tb}, we define the backward exit time, position, and momentum for every iteration: for any $\ell \in \N$,
\Be \label{def:tB_l}
\begin{split} 
&	\tBp^\ell(t,x,p) : = \sup\{ s  \in [0,\infty) : \X_{\pm, 3}^\ell (t-\tau; t,x,p)  \ \text{for all }  \tau \in (0, s)\}  \geq 0, 
\\& \xBp^{\ell} (t,x,p) := \X^{\ell}_{\pm} (t - \tBp^{\ell}(t,x,p); t,x,p), \ \
\pBp^{\ell} (t,x,p) := \P^{\ell}_{\pm} (t - \tBp^{\ell}(t,x,p); t,x,p).
\end{split}
\Ee

\smallskip

Further, similar to \eqref{def:flux}, \eqref{cont_eqtn}, and \eqref{identity:Psi_t}, we define the following: for any $\ell \in \N$,
\begin{align}
b^{\ell} (t,x) := \int_{\R^3} (v_+ e_+ f^{\ell}_+ + v_{-} e_{-} f^{\ell}_{-} ) \dd p \ \ \text{in} \ \R_+ \times \O,\label{def:flux_ell} \\
\p_t \varrho^\ell + \nabla_x \cdot b^\ell =0 \ \ \text{in} \ \R_+ \times \O,
\label{cont_eqtn_ell} \\
\p_t \Psi^\ell (t,x)  = (-\Delta_0)^{-1} \p_t \varrho^\ell(t,x)  = - (-\Delta_0)^{-1}  (\nabla_x \cdot b^\ell) (t,x)\ \ \text{in} \ \R_+ \times \O. \label{identity:Psi_t_ell}
\end{align}

\smallskip

In addition, similar to \eqref{w^F}, we define the dynamic weight function for every iteration: for any $\ell \in \N$,
\Be \label{w^ell}
\w^{\ell+1}_{\pm} (t,x,p) = \w^{\ell+1}_{\pm, \beta} (t,x,p) = e^{ \beta \left( \sqrt{(m_{\pm} c)^2 + |p|^2} + \frac{1}{c} \big( e \big( \Phi (x) + \Psi^\ell (t,x) \big) + m_{\pm} g x_3 \big) \right) }.
\Ee
and at the initial time $t=0$, 
\Be \label{w^ell_t=0}
\w^{\ell+1}_{\pm} (0,x,p) = \w^{\ell+1}_{\pm, \beta, 0} (x,p)
= e^{ \beta \left( \sqrt{(m_{\pm} c)^2 + |p|^2} + \frac{1}{c} \big( e_{\pm} ( \Phi (x) + \Psi^{\ell} (0,x) ) + m_{\pm} g x_3 \big) \right) }.
\Ee
From Remark \ref{rmk:w_dy} and \eqref{identity:Psi_t_ell}, each dynamic weight is not invariant along the dynamic characteristics: for any $\ell \in \N$ and $s \in [t - \tBp^{\ell+1} (t,x,p), t + \tFp^{\ell+1} (t,x,p)]$,
\Be \label{dDTE^ell}
\begin{split}
& \frac{d}{ds} \Big( \sqrt{(m_{\pm} c)^2 + |\P^{\ell+1}_{\pm} (s;t,x,p)|^2} + \frac{1}{c} \big( e_{\pm} \big( \Phi (\X^{\ell+1}_{\pm} (s;t,x,p)) 
\\& \qquad + \Psi^{\ell} (s, \X^{\ell+1}_{\pm} (s;t,x,p)) \big) + m_{\pm} g \X^{\ell+1}_{\pm, 3} (s;t,x,p) \big) \Big) 
\\& = \frac{e_{\pm}}{c} \p_t \Psi^{\ell} (s,\X^{\ell+1}_{\pm} (s;t,x,p)) 
\\& = \frac{e_{\pm}}{c} (-\Delta_0)^{-1} \p_t \varrho^\ell(t,x) 
= - \frac{e_{\pm}}{c} (-\Delta_0)^{-1}  (\nabla_x \cdot b^{\ell} ) (s, \X^{\ell+1}_{\pm} (s;t,x,p)).
\end{split}
\Ee
Using \eqref{w:bdry} and \eqref{bdry:Psi_fell}, we derive that, for any $\ell \in \N$,
\be \label{w_bdry}
\w^{\ell+1}_{\pm, \beta} (t,x,p) 
= w_{\pm, \beta} (x,p)
= e^{\beta \sqrt{(m_{\pm} c)^2 + |p|^2} } 
\ \ \text{on} \ (x,p) \in \p\O \times \R^3.
\ee
On the other hand, \eqref{W_t=0} and \eqref{initial:fell} show that, for any $\ell \in \N$, 
\be \label{w_initial}
\w^{\ell+1}_{\pm, \beta} (0,x,p) = \w_{\pm, \beta, 0} (x, p)
\ \ \text{in} \ (x,p) \in \bar \O \times \R^3.
\ee

\smallskip

Finally, we define the following: for any $\ell \in \N$,
\be \label{def:Fell}
F^{\ell+1}_{\pm} (t,x,p) = h_{\pm} (x,p) + f^{\ell+1}_{\pm} (t,x,p).
\ee
Under direct computation, we obtain that, for any $\ell \in \N$,
\be \label{eqtn:Fell}
\begin{split}
& \p_t F^{\ell+1}_{\pm}  + v_\pm \cdot \nabla_x F^{\ell+1}_{\pm} + \Big( e_{\pm} \big( \frac{v_\pm}{c} \times B - \nabla_x ( \Phi + \Psi^{\ell} ) \big) - \nabla_x ( m_\pm g x_3) \Big) \cdot \nabla_p F^{\ell+1}_{\pm} 
\\& \ \ \ \ = 0 \ \ \text{in} \ \R_+ \times  \O \times \R^3,
\end{split}
\ee
and
\be \label{bdry_initial:Fell}
F^{\ell+1}_{\pm} (0, x, p) = F_{\pm, 0} 
\ \ \text{in} \ \O \times \R^3,
\ \
F^{\ell+1}_{\pm} (t, x, p) = G_{\pm} (x,p)
\ \ \text{on} \ \gamma_-.
\ee
We also define 
\be \label{eqtn:phiFell}
\phi_{F^{\ell}} (t, x) = \Phi (x) + \Psi^{\ell} (t, x),
\ee
which solves the following Poisson equation: 
\Be \label{Poisson_Fell}
\begin{split}
- \Delta \phi_{F^{\ell}} (t, x)
= \int_{\R^3} ( e_+ F^{\ell}_+ + e_{-} F^{\ell}_{-} ) \dd p
\ \ \text{in} \ \R_+ \times \O,
\ \
\phi_{F^{\ell}} (t, x) = 0 \ \ \text{on} \ \R_+ \times \p\O.
\end{split}
\Ee
One can check that for any $\ell \in \N$, $F^{\ell+1}_{\pm}$ and $f^{\ell+1}_{\pm}$ share the same characteristic $\Z$ in \eqref{Z1}, \eqref{ODE1} and \eqref{Zell}, \eqref{ODEell}. Thus, they also share the same backward exit time, position, and momentum for every iteration.

\begin{lemma} \label{lem:tB_ell}

(a)
Suppose the following condition holds: for any $\ell \in \N$,
\Be \label{Bootstrap_ell}
\sup_{0 \leq t < \infty} \| \nabla_x \big( \Phi + \Psi^{\ell} (t) \big) \|_{L^\infty(\O)}  
\leq \min \left\{\frac{m_+}{e_+}, \frac{m_-}{e_-} \right\} \times \frac{g}{2}.
\Ee 
Then for $(t,x,p) \in [0, \infty) \times  \bar\O \times \R^3$ satisfying $t - \tBp^{\ell + 1} (t,x,p) \geq 0$, the backward exit time is bounded by	
\Be \label{est:tB_ell}
\begin{split}
\tBp^{\ell + 1} (t,x,p) 
& \leq \frac{1}{\min \{ \frac{g}{4 \sqrt{2}}, \frac{c}{\sqrt{10}} \}} \times \max \left\{ x_3, \ \frac{2 c}{m_{\pm} g} \big( \sqrt{(m_{\pm} c)^2 + |p|^2} - m_{\pm} c + \frac{3 m_{\pm} g}{2 c} x_{3} - p_3 \big) \right\}
\\& \qquad + 1 + \frac{4}{m_{\pm} g} |p|,
\\ \tBp^{\ell + 1} (t,x,p) 
& \leq | \pBpn^{\ell + 1} (t,x, p) | \Big( \frac{2}{m_{\pm} g}  + \frac{ 16 (\sqrt{c (m_{\pm})^2 g} + | \pBpn^{\ell + 1} (t,x,p) |)}{c (m_{\pm})^3 g^2} 
\\& \qquad \qquad \qquad \qquad \qquad \times \big(1 + \sqrt{(m_{\pm} c)^2 + |\pBp^{\ell + 1} (t,x,p)|^2} + 2 \pBpn^{\ell + 1} (t,x, p) \big) \Big).
\end{split}
\Ee
Moreover, the forward exit time is bounded by
\be \label{est:tF_ell}
\begin{split}
\tFp^{\ell + 1} (t,x,p) 
& \leq \frac{1}{\min \{ \frac{g}{4 \sqrt{2}}, \frac{c}{\sqrt{10}} \}} \times \max \left\{ x_3, \ \frac{2 c}{m_{\pm} g} \big( \sqrt{(m_{\pm} c)^2 + |p|^2} - m_{\pm} c + \frac{3 m_{\pm} g}{2 c} x_{3} + p_3 \big) \right\}
\\& \qquad + \frac{4}{m_{\pm} g} |p| + 1.
\end{split}
\ee

(b)
In addition, if for any $\ell \in \N$,
\be \label{Bootstrap_psi_ell}
\sup_{0 \leq t < \infty}\| \nabla_x \Psi^{\ell} (t) \|_{L^\infty(\O)} \leq \min \left\{\frac{m_+}{e_+}, \frac{m_-}{e_-} \right\} \times \frac{g}{48}.
\ee
Then for $(t,x,p) \in [0, \infty) \times  \bar\O \times \R^3$ satisfying $t - \tB^{\ell + 1} (t,x,p) \geq 0$, the backward exit time is bounded by
\Be \label{est2:tB_ell}
\begin{split}
\tBp^{\ell + 1} (t,x,p) 
& \leq \frac{8}{m_{\pm} g} \sqrt{(m_{\pm} c)^2 + |\pBp^{\ell + 1} (t,x,p)|^2}.
\end{split}
\Ee
Moreover, for all $s \in [t - \tBp^{\ell + 1} (t,x,p), t + \tFp^{\ell + 1} (t,x,p)]$,
\Be \label{est:x3ppB_ell}
\sqrt{(m_{\pm} c)^2 + |p|^2} + \frac{m_{\pm} g}{2 c} x_3
\leq \frac{7}{6} \sqrt{(m_{\pm} c)^2 + |\pBp^{\ell + 1} (t,x,p)|^2}.
\ee
\end{lemma}

\begin{proof}

Here we omit the proof since it directly follows from Propositions \ref{lem:tB}, \ref{lem2:tB} and \ref{lem3:tB}.
\end{proof}

\subsection{Asymptotic Stability Criterion} \label{sec:AS}

In this section, we always assume $f_{\pm} (t,x,p)$ is a Lagrangian solution to \eqref{eqtn:f}-\eqref{Poisson_f} for a given  $\nabla_p h$.
The main purpose of this section is to prove Theorem \ref{theo:AS}, which shows a conditional asymptotic stability result of the dynamic perturbation. 

From \eqref{Lform:f}, the Lagrangian formulation of $f$ is 
\Be \label{form:f}
f_{\pm} (t,x,p) = \mathcal{I}_{\pm} (t,x,p) + \mathcal{N}_{\pm} (t,x,p), 
\Ee 
where 
\be \label{form:I}
\mathcal{I}_{\pm} (t,x,p ) 
:= \mathbf{1}_{t \leq \tB (t,x,p)} f_{\pm} (0, \Z_{\pm} (0;t,x,p) ),  
\ee
and
\be \label{form:N}
\mathcal{N}_{\pm} (t,x,p ) 	
:= \int^t_{ \max\{0, t - \tB (t,x,p)\}} e_{\pm} \nabla_x \Psi (s, \X_{\pm} (s;t,x,p)) \cdot \nabla_p h_{\pm} ( \Z_{\pm} (0;t,x,p) ) \dd s. 
\ee

\begin{lemma} \label{lem:w/w}

Suppose the conditions \eqref{Uest:DPhi}, \eqref{Uest:DxPsi} and \eqref{Uest:Dxphi_F} hold. Recall $w_\beta(x,p)$ in \eqref{w^h}, $\w_\beta (t,x,p)$ in \eqref{w^F}, and $\Z = (\X, \P)$ solving \eqref{ODE_F}. 
Further, we assume that
\be \label{est:beta_b_dy}
\beta \frac{e_{\pm}}{c} \| (-\Delta_0)^{-1} (\nabla_x \cdot b ) \|_{L^\infty_{t,x}} 
\leq 1 \leq \frac{\beta}{2}.
\ee
Consider $(t,x,p) \in [0, \infty) \times  \bar\O \times \R^3$, for any $s, s^\prime \in [\max \{0, t -\tBp (t,x,p) \}, t + \tFp (t,x,p)]$,
\Be \label{est:1/w_h}
\begin{split}
\frac{\w_{\pm, \beta} (s^\prime, \Z_{\pm} (s^\prime;t,x,p) )}{\w_{\pm, \beta} (s, \Z_{\pm} (s ;t,x,p))}
& \leq e^{\beta \frac{e_{\pm}}{c} \| (-\Delta_0)^{-1} (\nabla_x \cdot b ) \|_{L^\infty_{t,x}} 
\big( \frac{ \frac{4 c}{m_{\pm} g} |p| + 4 x_{3} }{c_1} + \frac{8}{m_{\pm} g} |p| + 2 \big) }  ,
\\ \frac{1}{\w_{\pm, \beta} (s, \Z_{\pm} (s;t,x,p))}  
& \leq e^{2 \beta \frac{e_{\pm}}{c} \| (-\Delta_0)^{-1} (\nabla_x \cdot b ) \|_{L^\infty_{t,x}} } e^{ - \frac{\beta}{2} \sqrt{(m_{\pm} c)^2 + |p|^2} } e^{ - \frac{m_{\pm} g}{4 c} \beta x_3},
\end{split}
\Ee
where $c_1 := \min \{ \frac{g}{4 \sqrt{2}}, \frac{c}{\sqrt{10}} \}$.
Moreover, recall $w_{\pm, \beta} (x,p)$ in \eqref{w^h}, then
\Be \label{est:1/w}
\frac{1}{w_{\pm, \beta} ( \Z_{\pm} (s;t,x,p))}   
\leq e^{2 \beta \frac{e_{\pm}}{c} \| (-\Delta_0)^{-1} (\nabla_x \cdot b ) \|_{L^\infty_{t,x}} } e^{- \frac{\beta}{4} \sqrt{(m_{\pm} c)^2 + |p|^2} } e^{- \frac{m_{\pm} g}{8 c} \beta x_3}.
\Ee
\end{lemma}

\begin{proof}

We abuse the notation as in \eqref{abuse} in the proof.
From \eqref{dDTE}, we get
\Be \label{dEC}
\begin{split}
& \frac{d}{ds} \Big( \sqrt{(m_{\pm} c)^2 + |\P_{\pm} (s;t,x,p)|^2} + \frac{1}{c} \big( e_{\pm} \big( \Phi_h (\X_{\pm} (s;t,x,p)) 
\\& \qquad \qquad + \Psi (s, \X_{\pm} (s;t,x,p)) \big) + m_{\pm} g \X_{\pm, 3} (s;t,x,p) \big) \Big)
\\& = \frac{e_{\pm}}{c} \p_t \Psi (s,\X_{\pm} (s;t,x,p)) 
= \frac{e_{\pm}}{c} (-\Delta_0)^{-1} \p_t \varrho (s, \X_{\pm} (s;t,x,p)) 
\\& = - \frac{e_{\pm}}{c} (-\Delta_0)^{-1}  (\nabla_x \cdot b ) (s, \X_{\pm} (s;t,x,p)).
\end{split}
\Ee
Using \eqref{dEC}, together with the dynamic weight function $\w_{\pm, \beta} (t,x,p)$ in \eqref{w^F}, we derive
\Be	\label{w(Z)_s}
\begin{split}
& \frac{d}{ds} \w_\beta (s, \X (s;t,x,p), \P (s;t,x,p))
\\& = \w_\beta (s, \X  (s;t,x,p), \P (s;t,x,p)) \big( - \beta \frac{e}{c} (-\Delta_0)^{-1}  (\nabla_x \cdot b ) (s, \X (s;t,x,p)) \big).
\end{split}
\Ee 
Thus, if $\max \{0, t-\tB  (t,x,p) \} \leq s, s^\prime \leq t + \tF  (t,x,p)$, then 
\Be \label{est:w_ell}
\begin{split}
& \w_\beta  (s, \X (s;t,x,p), \P (s;t,x,p))
\\& = \w _\beta (s^\prime,\X  (s^\prime;t,x,p), \P (s^\prime;t,x,p)) 
e^{ - \beta \frac{e}{c}  
 \int^s_{s^\prime} (-\Delta_0)^{-1} (\nabla_x \cdot b  )(\tau, \X (\tau;t,x,p)) \dd \tau }.
\end{split}
\Ee 
Using \eqref{cont_eqtn}, \eqref{est:tB} and \eqref{est:tF}, we bound the integration term above by 
\Be \label{est1:w_ell}
\int^s_{s^\prime} (-\Delta_0)^{-1} (\nabla_x \cdot b  )(\tau, \X (\tau;t,x,p)) \dd \tau
\leq |s-s^\prime| \| (-\Delta_0)^{-1} (\nabla_x \cdot b ) \|_{L^\infty_{t,x}},
\ee
and further, by setting $c_1 := \min \{ \frac{g}{4 \sqrt{2}}, \frac{c}{\sqrt{10}} \}$, we have
\Be \label{est:expo_w_ell}
\begin{split}
& |s-s^\prime| \| (-\Delta_0)^{-1} (\nabla_x \cdot b ) \|_{L^\infty_{t,x}}
\\& \leq |t_{\mathbf{B }} (t,x,p)+ t_{\mathbf{F} }  (t,x,p)| \| (-\Delta_0)^{-1} (\nabla_x \cdot b ) \|_{L^\infty_{t,x}}
\\& \leq 
\Big\{ \frac{x_3 + \frac{2 c}{m g} \big( \sqrt{(m c)^2 + |p|^2} - mc + \frac{3 m g}{2 c} x_{3} + |p_3| \big)}{c_1} + \frac{8}{m g} |p| + 2 \Big\}
\| (-\Delta_0)^{-1} (\nabla_x \cdot b ) \|_{L^\infty_{t,x}}
\\& \leq 
\big( \frac{ \frac{4 c}{m g} |p| + 4 x_{3} }{c_1} + \frac{8}{m g} |p| + 2 \big)
\| (-\Delta_0)^{-1} (\nabla_x \cdot b ) \|_{L^\infty_{t,x}}.
\end{split}
\Ee 
and obtain the first part of \eqref{est:1/w_h}.
Further, from \eqref{Uest:Dxphi_F}, we bound
\Be \label{est:wbeta}
\w_{\beta} (t,x,p) = e^{ \beta \left( \sqrt{(m c)^2 + |p|^2} + \frac{1}{c} \big( e \big( \Phi_h (x) + \Psi(t,x) \big) + m g x_3 \big) \right) }
\geq e^{ \beta \big( \sqrt{(m c)^2 + |p|^2} + \frac{m g}{2 c} x_3 \big) }.
\Ee
Inputting $s^\prime = t$ in \eqref{est:w_ell}, together with \eqref{est:beta_b_dy} and \eqref{est:expo_w_ell}, we get
\Be \notag
\begin{split}
& \w _\beta  (s,\X  (s;t,x,p), \V  (s;t,x,p)) 
\\& \geq \w_\beta (t,x,p) e^{- \beta \frac{e}{c} \| (-\Delta_0)^{-1} (\nabla_x \cdot b ) \|_{L^\infty_{t,x}} 
\big( \frac{ \frac{4 c}{m g} |p| + 4 x_{3} }{c_1} + \frac{8}{m g} |p| + 2 \big) }  
\\& \geq e^{ \beta \big( \sqrt{(m c)^2 + |p|^2} + \frac{m g}{2 c} x_3 \big) } e^{- \beta \frac{e}{c} \| (-\Delta_0)^{-1} (\nabla_x \cdot b ) \|_{L^\infty_{t,x}} 
\big( \frac{ \frac{4 c}{m g} |p| + 4 x_{3} }{c_1} + \frac{8}{m g} |p| + 2 \big) }  
\\& \geq e^{- 2 \beta \frac{e}{c} \| (-\Delta_0)^{-1} (\nabla_x \cdot b ) \|_{L^\infty_{t,x}} } e^{\frac{\beta}{2} \sqrt{(m c)^2 + |p|^2} } e^{\frac{m g}{4 c} \beta x_3}.
\end{split}
\Ee
and conclude the second part of \eqref{est:1/w_h}.
	
\smallskip	

Now we prove \eqref{est:1/w}. Using \eqref{ODE_F}, \eqref{eq:energy_conservation_dy} and \eqref{dEC}, we compute that 
\Be	\label{w(Z)_s:h}
\begin{split}
& \frac{d}{ds} w_\beta (\Z (s;t,x,p) ) 
\\& = \beta w_\beta ( \Z (s;t,x,p) ) \times
\frac{d}{ds} \Big( \sqrt{(m c)^2 + |\P (s;t,x,p)|^2} 
\\& \qquad \qquad \qquad \qquad \qquad \qquad
+ \frac{1}{c} \big( e \Phi_h (\X (s;t,x,p)) + m g \X_{3} (s;t,x,p) \big) \Big)
\\& =  \beta w_\beta (\Z (s;t,x,p) ) \big( - \frac{e}{c} \nabla_x \Psi (s, \X (s;t,x,p)) \cdot \V (s;t,x,p) \big).
\\&  = \beta w_\beta (\Z  (s;t,x,p) ) 	 
\frac{e}{c} \big( \p_t \Psi (s, \X  (s;t,x,p)) - \frac{d}{ds} \Psi (s, \X (s;t,x,p)) \big),
\end{split}
\Ee 
where the last equality follows from
\Be \label{dsPsi}
\frac{d}{ds} \Psi (s, \X (s;t,x,p)) 
= \p_t \Psi (s, \X (s;t,x,p)) + \nabla_x \Psi (s, \X  (s;t,x,p)) \cdot \V (s;t,x,p).
\Ee
Together with \eqref{dDTE}, we have 
\Be	\notag
\begin{split}
& \frac{d}{ds} w_\beta (\Z (s;t,x,p) ) 
\\& = w_\beta (\Z  (s;t,x,p) ) 
\big( - \beta \frac{e}{c} (-\Delta_0)^{-1}  (\nabla_x \cdot b ) (s, \X (s;t,x,p)) - \beta \frac{e}{c} \frac{d}{ds} \Psi (s, \X (s;t,x,p)) \big).
\end{split}
\Ee 
This shows that, for $s, s^\prime \in [t-\tB (t,x,p), t +t_{\mathbf{F} } (t,x,p)]$,  
\Be \label{est:w_h}
\begin{split}
\frac{w_\beta  (\Z  (s;t,x,p) ) }{w _\beta (\Z  (s^\prime;t,x,p)) }
& = e^{ - \beta \frac{e}{c} \int^s_{s^\prime} (-\Delta_0)^{-1} (\nabla_x \cdot b  )(\tau, \X (\tau;t,x,p)) \dd \tau } e^{ - \beta \frac{e}{c} \int^s_{s^\prime} \frac{d}{d \tau} \Psi  (\tau , \X (\tau ;t,x,p)) \dd \tau }
\\& = e^{ - \beta \frac{e}{c} \int^s_{s^\prime} (-\Delta_0)^{-1} (\nabla_x \cdot b  )(\tau, \X (\tau;t,x,p)) \dd \tau } e^{ - \beta \frac{e}{c} ( \Psi  (s , \X (s ;t,x,p)) - \Psi  (s^\prime , \X (s^\prime ;t,x,p)) ) }.
\end{split}
\Ee
From \eqref{Uest:DxPsi} and the Dirichlet boundary condition of \eqref{Poisson_f}, together with Corollary \ref{cor:max_X_dy}, we get, for $s, s^\prime \in [\max \{0, t - \tB  (t,x,p) \}, t + \tF (t,x,p)]$,
\Be \label{est:expo_w}
\begin{split}
& \big| \Psi  (s, \X (s;t,x,p)) - \Psi (s^\prime,\X  (s^\prime;t,x,p)) \big| 
\\& = \Big| \int^{\X_3 (s;t,x,p)}_{0} \Psi_{x_3} (s, (\X_{\|}  (s;t,x,p), \tau) ) \dd \tau 
- \int^{\X_3 (s';t,x,p)}_{0} \Psi_{x_3} (s^\prime, (\X_{\|} (s^\prime;t,x,p), \tau) ) \dd \tau \Big| 
\\& \leq \sup_{0 \leq t < \infty}\| \nabla_x \Psi (t) \|_{L^\infty(\O)} \times \big( \X_3 (s;t,x,p) + \X_3 (s';t,x,p) \big)
\\& \leq \frac{m g}{24 e} (\frac{4 c}{m g} |p| + 3 x_3).
\end{split}
\Ee 
Inputting $s^\prime = t$ in \eqref{est:w_h}, together with \eqref{est:beta_b_dy}, \eqref{est1:w_ell}, \eqref{est:expo_w_ell} and \eqref{est:expo_w}, we derive that
\Be \notag
\begin{split}
& w_\beta(\Z  (s;t,x,p)) 
\\& \geq w_\beta(t,x,p) e^{- \beta \frac{e}{c} \| (-\Delta_0)^{-1} (\nabla_x \cdot b ) \|_{L^\infty_{t,x}} 
\big( \frac{ \frac{4 c}{m g} |p| + 4 x_{3} }{c_1} + \frac{8}{m g} |p| + 2 \big) }  
e^{ - \beta \frac{e}{c} \big( \Psi  (s , \X (s ;t,x,p)) - \Psi  (s^\prime , \X (s^\prime ;t,x,p)) \big) }
\\& \geq e^{ \beta \big( \sqrt{(m c)^2 + |p|^2} + \frac{m g}{2 c} x_3 \big) } e^{- \beta \frac{e}{c} \| (-\Delta_0)^{-1} (\nabla_x \cdot b ) \|_{L^\infty_{t,x}} 
\big( \frac{ \frac{4 c}{m g} |p| + 4 x_{3} }{c_1} + \frac{8}{m g} |p| + 2 \big) } e^{ - \beta \frac{e}{c} \big( \frac{m g}{24 e} (\frac{4 c}{m g} |p| + 3 x_3) \big) }
\\& \geq e^{- 2 \beta \frac{e}{c} \| (-\Delta_0)^{-1} (\nabla_x \cdot b ) \|_{L^\infty_{t,x}} } e^{\frac{\beta}{2} \sqrt{(m c)^2 + |p|^2} } e^{\frac{m g}{4 c} \beta x_3} e^{ - \beta \frac{m g}{24 c} (\frac{4 c}{m g} |p| + 3 x_3) }
\\& \geq e^{- 2 \beta \frac{e}{c} \| (-\Delta_0)^{-1} (\nabla_x \cdot b ) \|_{L^\infty_{t,x}} } e^{\frac{\beta}{4} \sqrt{(m c)^2 + |p|^2} } e^{\frac{m g}{8 c} \beta x_3},
\end{split}
\Ee
and thus, we conclude \eqref{est:1/w}.
\end{proof}

\begin{lemma} \label{lem:Ddb}

Recall $\mathfrak{G}$ in Lemma \ref{lemma:G} and $b (t, x)$ in \eqref{def:flux}, then
\Be \label{D^-1_Db}
\begin{split}
(-\Delta_0)^{-1} (\nabla_x \cdot b ) (t,x)    
= \int_{\O} b (t, y) \cdot \nabla_y \mathfrak{G} \dd y.
\end{split}
\Ee
Moreover, if
$\sup\limits_{0 \leq t < \infty } \| e^{ \frac{\beta}{2} \big( \sqrt{(m_{\pm} c)^2 + |p|^2} + \frac{m_{\pm} g}{2 c} x_3 \big) } f_{\pm} (t) \|_{\infty} < \infty$,
then
\Be \label{est:D^-1_Db}
\begin{split}
& \sup_{0 \leq t < \infty} \| (-\Delta_0)^{-1}  (\nabla_x \cdot b  (t,x)) \|_{L^\infty (\O)} 
\\& \leq \frac{ 8 \pi^2 \hat{e} c }{\beta^3} \sum\limits_{i = \pm} \sup_{0 \leq t < \infty } \| e^{ \frac{\beta}{2} \big( \sqrt{(m_i c)^2 + |p|^2} + \frac{m_i g}{2 c} x_3 \big) } f_i (t) \|_{\infty} (1 + \frac{4 c}{\hat{m} g \beta}),
\end{split}
\Ee
where $\hat{m} = \min \{ m_+, m_- \}$ and $\hat{e} = \max \{ e_+, e_- \}$.
\end{lemma}

\begin{proof}

From \eqref{phi_rho}, we directly get
\be \notag
(-\Delta_0)^{-1}  (\nabla_x \cdot b ) (t,x)   
= - \int_{\O} \mathfrak{G} (\nabla \cdot b) (t, y) \dd y.
\ee
Using zero Dirichlet boundary condition of \eqref{Dphi} and integration by parts, we get
\be \notag
\int_{\O} \mathfrak{G} (\nabla \cdot b) (t, y) \dd y
= - \int_{\O} b (t, y)   \cdot  \nabla_y \mathfrak{G} \dd y,
\ee
and thus derive \eqref{D^-1_Db}.

\smallskip

Next, we compute that
\be \notag
\begin{split}
|b (t,x)| 
& = \big| \int_{\R^3} (v_+ e_+ f_+ + v_{-} e_{-} f_{-} ) \dd p \big|
\\& \leq \big| \int_{\R^3} v_+ e_+ f_+ \dd p \big| + \big| \int_{\R^3} v_{-} e_{-} f_{-} \dd p \big|
\\& = \big| \int_{\R^3} \frac{e_+ p}{\sqrt{ m_{+}^2 + |p|^2/ c^2}}  f_+ \dd p \big| + \big| \int_{\R^3} \frac{e_- p}{\sqrt{ m_{-}^2 + |p|^2/ c^2}} f_{-} \dd p \big|.
\end{split}
\ee	
From $\sup\limits_{0 \leq t < \infty } \| e^{ \frac{\beta}{2} \big( \sqrt{(m_{\pm} c)^2 + |p|^2} + \frac{m_{\pm} g}{2 c} x_3 \big) } f_{\pm} (t) \|_{\infty} < \infty$, we derive that
\Be \notag
\begin{split}
& | b (t,x) |  
\\& \leq e^{-\frac{m_+ g}{4 c} \beta x_3} \int_{\R^3} \frac{e_+ |p|}{\sqrt{ m_{+}^2 + |p|^2/ c^2}}  e^{ - \frac{\beta}{2} \sqrt{(m_+ c)^2 + |p|^2}} \dd p \sup_{0 \leq t < \infty } \| e^{ \frac{\beta}{2} \big( \sqrt{(m_+ c)^2 + |p|^2} + \frac{m_+ g}{2 c} x_3 \big) } f_+ (t) \|_{\infty}
\\& + e^{-\frac{m_- g}{4 c} \beta x_3} \int_{\R^3} \frac{e_- |p|}{\sqrt{ m_{-}^2 + |p|^2/ c^2}}  e^{ - \frac{\beta}{2} \sqrt{(m_- c)^2 + |p|^2}} \dd p \sup_{0 \leq t < \infty } \| e^{ \frac{\beta}{2} \big( \sqrt{(m_- c)^2 + |p|^2} + \frac{m_- g}{2 c} x_3 \big) } f_- (t) \|_{\infty}
\\& \leq e^{-\frac{m_+ g}{4 c} \beta x_3} e_+ c  \int_{\R^3}  e^{ \frac{\beta}{2} |p|} \dd p \sup_{0 \leq t < \infty } \| e^{ \frac{\beta}{2} \big( \sqrt{(m_+ c)^2 + |p|^2} + \frac{m_+ g}{2 c} x_3 \big) } f_+ (t) \|_{\infty}
\\& \ \ \ \ + e^{-\frac{m_- g}{4 c} \beta x_3} e_- c \int_{\R^3}   e^{ \frac{\beta}{2} |p|} \dd p \sup_{0 \leq t < \infty } \| e^{ \frac{\beta}{2} \big( \sqrt{(m_- c)^2 + |p|^2} + \frac{m_- g}{2 c} x_3 \big) } f_- (t) \|_{\infty}.
\end{split}
\Ee
Note that $\int_{\R^3}   e^{ \frac{\beta}{2} |p|} \dd p \leq \frac{8 \pi^2}{\beta^3}$. By setting $\hat{m} = \min \{ m_+, m_- \}$ and $\hat{e} = \max \{ e_+, e_- \}$, we get
\Be \label{est:b}
\begin{split} 
| b (t,x) |  
& \leq e^{-\frac{m_+ g}{4 c} \beta x_3}  \frac{ 8 \pi^2 e_+ c }{\beta^3} \sup_{0 \leq t < \infty } \| e^{ \frac{\beta}{2} \big( \sqrt{(m_+ c)^2 + |p|^2} + \frac{m_+ g}{2 c} x_3 \big) } f_+ (t) \|_{\infty} 
\\& \ \ \ \ + e^{-\frac{m_- g}{4 c} \beta x_3}  \frac{ 8 \pi^2 e_- c }{\beta^3} \sup_{0 \leq t < \infty } \| e^{ \frac{\beta}{2} \big( \sqrt{(m_- c)^2 + |p|^2} + \frac{m_- g}{2 c} x_3 \big) } f_- (t) \|_{\infty}
\\& \leq \frac{ 8 \pi^2 \hat{e} c }{\beta^2} \sum\limits_{i = \pm} \sup_{0 \leq t < \infty } \| e^{ \frac{\beta}{2} \big( \sqrt{(m_i c)^2 + |p|^2} + \frac{m_i g}{2 c} x_3 \big) } f_i (t) \|_{\infty} \times e^{- \frac{\hat{m} g}{4 c} \beta x_3}.
\end{split}
\Ee 
Recall from \eqref{D^-1_Db}, we have
\be \notag
| (-\Delta_0)^{-1}  (\nabla_x \cdot b ) (t,x) | 
= \big| \int_{\O} b (t, y)   \cdot  \nabla_y \mathfrak{G} \dd y \big|.
\ee
Following the proof of \eqref{est:nabla_phi} in Lemma \ref{lem:rho_to_phi}, together with \eqref{est:b}, we conclude \eqref{est:D^-1_Db}.
\end{proof}

\begin{prop} \label{prop:decay}

Suppose the condition \eqref{Uest:DPhi}, \eqref{Uest:DxPsi} and \eqref{Uest:Dxphi_F} hold. 
Let $\hat{m} = \min\{m_+, m_- \}$, we further assume that for $g, \beta>0$,
\be \label{Bootstrap_f_first}
\sum\limits_{i = \pm} \sup_{0 \leq t < \infty } \| e^{ \frac{\beta}{2} \big( \sqrt{(m_i c)^2 + |p|^2} + \frac{m_i g}{2 c} x_3 \big) } f_i (t) \|_{\infty} 
\leq \frac{ \beta^2 \ln (2) }{ 8 \big( \pi \max \{ e_+, e_- \} \big)^2 (1 + \frac{4 c}{\hat{m} g \beta})}.
\ee	
Then 
\be \label{est:I}
| \mathcal{I}_{\pm} (t,x,p) |
\leq e^{2 + \frac{m_{\pm} g}{24} \beta} \| \w_{\pm, \beta, 0 }   f_{\pm, 0 } \|_{L^\infty_{x,p}} 
e^{- \frac{m_{\pm} g}{24} \beta t } 
e^{ - \frac{\beta}{4} \sqrt{(m_{\pm} c)^2 + |p|^2} } e^{ - \frac{m_{\pm} g}{8 c} \beta x_3}.
\ee
Moreover, let $\lambda = \frac{g}{48} \hat{m} \beta$, $\nu = \frac{g}{16 c} \hat{m} \beta$. Assume that
\be \notag
\sup\limits_{s \in [0, t] } e^{ \lambda s} \| e^{ \nu x_3} \varrho (s) \|_{L^\infty(\O)} < \infty.
\ee 
Then
\be \label{est:N}
\begin{split}
| \mathcal{N}_{\pm} (t,x,p) |
& \leq  e^{- \frac{\beta}{8} \sqrt{(m_{\pm} c)^2 + |p|^2} } e^{- \frac{m_{\pm} g}{16 c} \beta x_3} e^{ - \lambda t} 
\\& \ \ \ \ \times e^{ \lambda } 4 e_{\pm} \sup\limits_{s \in [0, t] } e^{ \lambda s} \| e^{ \nu x_3} \varrho (s) \|_{L^\infty(\O)} (1 + \frac{1}{\nu}) \| w_{\pm, \beta} \nabla_p h_{\pm} \|_{\infty}.
\end{split}
\ee
\end{prop}

\begin{proof} 

We abuse the notation as in \eqref{abuse} in the proof.
From \eqref{est:D^-1_Db} and \eqref{Bootstrap_f_first}, we derive that
\Be \label{est1:D^-1 Db}
\begin{split}
& \beta \frac{e_{\pm}}{c} \| (-\Delta_0)^{-1} (\nabla_x \cdot b ) \|_{L^\infty_{t,x}}
\\& \leq \beta \frac{e_{\pm}}{c} \frac{ 8 \pi^2 \hat{e} c }{\beta^3} \sum\limits_{i = \pm} \sup_{0 \leq t < \infty } \| e^{ \frac{\beta}{2} \big( \sqrt{(m_i c)^2 + |p|^2} + \frac{m_i g}{2 c} x_3 \big) } f_i (t) \|_{\infty} (1 + \frac{4 c}{\hat{m} g \beta}) \leq \ln (2).
\end{split}
\Ee

\smallskip

First, we prove \eqref{est:I}. From Corollary \ref{cor:max_X_dy}, we have
\be \notag
\mathbf{1}_{t \leq \tB  (t,x,p)} 
\leq \mathbf{1}_{t \leq \frac{1}{\min \{ \frac{g}{4 \sqrt{2}}, \frac{c}{\sqrt{10}} \}} \times \big( \frac{4 c}{m g} |p| + 3 x_{3} \big) + \frac{4}{m g} |p| + 1}.
\ee
This implies that
\be \label{est:1_t<tB}
\begin{split} 
& \mathbf{1}_{t \leq \tB  (t,x,p)} e^{ - \frac{\beta}{4} \sqrt{(m c)^2 + |p|^2} } e^{ - \frac{m g}{8 c} \beta x_3}
\\& \leq \mathbf{1}_{t \leq \frac{1}{\min \{ \frac{g}{4 \sqrt{2}}, \frac{c}{\sqrt{10}} \}} \times \big( \frac{4 c}{m g} |p| + 3 x_{3} \big) + \frac{4}{m g} |p| + 1} 
\times e^{ - \frac{\beta}{4} \sqrt{(m c)^2 + |p|^2} } e^{ - \frac{m g}{8 c} \beta x_3}
\leq e^{- \frac{m g}{24} \beta (t-1) }.
\end{split}	
\ee
Using \eqref{form:I} and \eqref{est1:D^-1 Db}, together with \eqref{est:1_t<tB} and Lemma \ref{lem:w/w}, we bound $\mathcal{I} (t,x,p )$ as
\Be \label{bound:tB/w}
\begin{split} 
| \mathcal{I} (t,x,p ) |
& = \big| \mathbf{1}_{t \leq \tB (t,x,p)} f (0,\X (0;t,x,p), \P (0;t,x,p)) \big|
\\& \leq \|\w_{\beta, 0 }   f_{0  } \|_{L^\infty_{x,p}}    \frac{ \mathbf{1}_{t \leq \tB  (t,x,p)}}{\w _\beta (0,\X  (0;t,x,p),\V  (0;t,x,p))}  
\\& \leq  \|\w_{\beta, 0 }   f_{0 } \|_{L^\infty_{x,p}}    
\frac{\mathbf{1}_{t \leq \tB  (t,x,p)} }{e^{ \frac{\beta}{4} \sqrt{(m c)^2 + |p|^2} } e^{ \frac{m g}{8 c} \beta x_3} }
\times \frac{ e^{2 \beta \frac{e}{c} \| (-\Delta_0)^{-1} (\nabla_x \cdot b ) \|_{L^\infty_{t,x}} } }{ e^{ \frac{\beta}{4} \sqrt{(m c)^2 + |p|^2} + \frac{m g}{8 c} \beta x_3} }
\\& \leq \|\w_{\beta, 0 }   f_{0  } \|_{L^\infty_{x,p}} 
e^{- \frac{m g}{24} \beta (t-1) }
\times \frac{4}{e^{ \frac{\beta}{4} \sqrt{(m c)^2 + |p|^2} + \frac{m g}{8 c} \beta x_3} }
\\& \leq 4 e^{\frac{m g}{24} \beta} \|\w_{\beta, 0 }   f_{0  } \|_{L^\infty_{x,p}} 
e^{- \frac{m g}{24} \beta t } 
e^{ - \frac{\beta}{4} \sqrt{(m c)^2 + |p|^2} - \frac{m g}{8 c} \beta x_3}.
\end{split}	
\Ee
and conclude \eqref{est:I}.

\smallskip

Second, we prove \eqref{est:N}. 
Using \eqref{est:1/w} and \eqref{est1:D^-1 Db}, we obtain 
\Be \label{h_v}
\begin{split}
| \nabla_p h (\Z (s;t,x,p) ) |
& \leq \frac{ 1}{w _\beta (\X (s;t,x,p), \V (s;t,x,p))} \| w_\beta \nabla_p h \|_{\infty}
\\& \leq e^{2 \beta \frac{e}{c} \| (-\Delta_0)^{-1} (\nabla_x \cdot b ) \|_{L^\infty_{t,x}} } e^{- \frac{\beta}{4} \sqrt{(m c)^2 + |p|^2} } e^{- \frac{m g}{8 c} \beta x_3}
\| w_\beta \nabla_p h \|_{\infty} 
\\& \leq 4 e^{- \frac{\beta}{4} \sqrt{(m c)^2 + |p|^2} } e^{- \frac{m g}{8 c} \beta x_3}
\| w_\beta \nabla_p h \|_{\infty}.
\end{split}
\Ee 	
On the other hand, from the assumption $\sup\limits_{s \in [0, t] } e^{ \lambda s} \| e^{ \nu x_3} \varrho (s) \|_{L^\infty(\O)} < \infty$, we have
\be \label{cond:varrho}
| \varrho (s, x) | 
\leq  \big( \sup\limits_{s \in [0, t] } e^{ \lambda s} \| e^{ \nu x_3} \varrho (s) \|_{L^\infty(\O)} \big) 
e^{ - \lambda s} e^{ - \nu x_3},
\text{ for any } (s, x) \in [0, t] \times \O.
\ee
Recall from \eqref{Poisson_f}, we have
\be \notag
- \Delta_x \Psi (t,x) = \varrho (t,x) \text{ in } \R_+ \times \O, \ \
\Psi(t,x) = 0 \text{ on } \R_+ \times \p\O.
\ee 
Applying Lemma \ref{lem:rho_to_phi} on \eqref{cond:varrho}, we obtain that, for any $(s, x) \in [0, t] \times \O$,
\be \label{est:Drho}
| \nabla_x \Psi (s, x) | 
\leq  \big( \sup\limits_{s \in [0, t] } e^{ \lambda s} \| e^{ \nu x_3} \varrho (s) \|_{L^\infty(\O)} \big) 
e^{ - \lambda s} (1 + \frac{1}{\nu}).
\ee
Using \eqref{form:N}, \eqref{h_v} and \eqref{est:Drho}, we bound $\mathcal{N} (t,x,p )$ as
\Be \label{est1:N}
\begin{split}
& | \mathcal{N} (t,x,p ) |
\\& = \big| \int^t_{ \max\{0, t - \tB (t,x,p)\}} e \nabla_x \Psi (s, \X (s;t,x,p)) \cdot \nabla_p h ( \X (s;t,x,p), \P (s;t,x,p)) \dd s \big|
\\& \leq \int^t_{ \max\{0, t - \tB (t,x,p)\}} e | \nabla_x \Psi (s, \X (s;t,x,p)) | | \nabla_p h_{\pm} ( \X (s;t,x,p), \P (s;t,x,p)) | \dd s 
\\& \leq \int^t_{ \max\{0, t - \tB (t,x,p)\}} e \big( \sup\limits_{s \in [0, t] } e^{ \lambda s} \| e^{ \nu x_3} \varrho (s) \|_{L^\infty(\O)} \big) 
e^{ - \lambda s} (1 + \frac{1}{\nu})
\\& \qquad \qquad \qquad \qquad \qquad \times 4 e^{- \frac{\beta}{4} \sqrt{(m c)^2 + |p|^2} } e^{- \frac{m g}{8 c} \beta x_3}
\| w_\beta \nabla_p h \|_{\infty} \dd s 
\\& = 4 e \sup\limits_{s \in [0, t] } e^{ \lambda s} \| e^{ \nu x_3} \varrho (s) \|_{L^\infty(\O)} (1 + \frac{1}{\nu}) \| w_\beta \nabla_p h \|_{\infty}
\\& \qquad \qquad \times 
\underbrace{
e^{- \frac{\beta}{4} \sqrt{(m c)^2 + |p|^2} } e^{- \frac{m g}{8 c} \beta x_3} \int^t_{ \max\{0, t - \tB (t,x,p)\}} e^{ - \lambda s} \dd s}_{\eqref{est1:N}_*}.
\end{split}
\Ee
Using \eqref{est:tB} in Proposition \ref{lem:tB} and following \eqref{bound:tB/w}, we have
\Be \label{est2:N}
\begin{split}
\eqref{est1:N}_*
& \leq e^{- \frac{\beta}{4} \sqrt{(m c)^2 + |p|^2} } e^{- \frac{m g}{8 c} \beta x_3} \tB (t,x,p) e^{ - \lambda (t - \tB  (t,x,p)) }
\\& = \big( e^{- \frac{\beta}{8} \sqrt{(m c)^2 + |p|^2} } e^{- \frac{m g}{16 c} \beta x_3} \tB (t,x,p) e^{ \lambda {\tB  (t,x,p)}} \big) \times
e^{- \frac{\beta}{8} \sqrt{(m c)^2 + |p|^2} } e^{- \frac{m g}{16 c} \beta x_3} e^{ - \lambda t}
\\& \lesssim e^{ \lambda } e^{- \frac{\beta}{8} \sqrt{(m c)^2 + |p|^2} } e^{- \frac{m g}{16 c} \beta x_3} e^{ - \lambda t}.
\end{split}
\Ee
Therefore, we conclude \eqref{est:N} by inputting \eqref{est2:N} into \eqref{est1:N}.
\end{proof}

Now we are ready to prove the main result of this section.

\begin{theorem}[Asymptotic Stability Criterion] \label{theo:AS}

Suppose $(h, \Phi)$ solves \eqref{VP_h}-\eqref{eqtn:Dphi} in the sense of Definition \ref{weak_sol}. 
Suppose $(f, \varrho,   \Psi)$ is a solution to \eqref{eqtn:f}-\eqref{Poisson_f}.
Set $\hat{m} = \min\{m_+, m_- \}$, we assume the following conditions hold: for $g, \beta>0$,
\be \label{Bootstrap_f}
\sum\limits_{i = \pm} \sup_{0 \leq t < \infty } \| e^{ \frac{\beta}{2} \big( \sqrt{(m_i c)^2 + |p|^2} + \frac{m_i g}{2 c} x_3 \big) } f_i (t) \|_{\infty} 
\leq \frac{ \beta^2 \ln (2) }{ 8 \big( \pi \max \{ e_+, e_- \} \big)^2 (1 + \frac{4 c}{\hat{m} g \beta})},
\ee	
and	
\be \label{Bootstrap} 
\begin{split}
\sup_{0 \leq t < \infty}\| \nabla_x \big( \Phi + \Psi (t) \big) \|_{L^\infty(\O)}  
& \leq \min \left\{\frac{m_+}{e_+}, \frac{m_-}{e_-} \right\} \times \frac{g}{2},
\\ \sup_{0 \leq t < \infty} \| \nabla_x \Psi (t) \|_{L^\infty(\O)} & \leq \min \left\{\frac{m_+}{e_+}, \frac{m_-}{e_-} \right\} \times \frac{g}{48}.
\end{split}
\ee
Set $\lambda = \frac{g}{48} \hat{m} \beta$, $\nu = \frac{g}{16 c} \hat{m} \beta$, we further assume that
\Be \label{condition:Dvh}
4 e^{ \lambda } (1 + \frac{1}{\nu}) \big( e^2_+ \| w_{+,\beta} \nabla_p h_+ \|_{\infty} + e^2_- \| w_{-,\beta} \nabla_p h_- \|_{\infty} \big) \frac{512}{\beta^3} \leq \frac{1}{2},
\Ee
and $\w_{\pm, \beta, 0} (x, p)$ defined in \eqref{W_t=0} satisfies
\be \notag 
\|\w_{\pm, \beta, 0} f_{\pm, 0} \|_{L^\infty (\O \times \R^3)} <\infty.
\ee
Then, $(f(t), \varrho(t))$ is bounded by
\Be \label{decay:varrho}
\begin{split}
& \sup_{0 \leq t < \infty} e^{ \lambda t} \| e^{ \nu x_3} \varrho (t, x)\|_{L^\infty(\O)}
\\& \leq \big( e_+ e^{2 + \frac{m_+ g}{24} \beta} \| \w_{+, \beta, 0 } f_{+, 0 } \|_{L^\infty_{x,p}} + e_- e^{2 + \frac{m_- g}{24} \beta} \| \w_{-,\beta, 0 } f_{-, 0 } \|_{L^\infty_{x,p}} \big) \frac{128}{\beta^3},
\end{split}
\Ee 
and
\Be \label{decay_f}
\begin{split}
& \sup_{0 \leq t < \infty} e^{ \lambda t} \|  e^{\frac{\beta}{8} \sqrt{(m_{\pm} c)^2 + |p|^2} + \frac{m_{\pm} g}{16 c} \beta x_3} f_{\pm} (t,x,p) \|_{L^\infty (\O \times \R^3)}
\\& \leq 2 e^2 \big( e^{\frac{m_+ g}{24} \beta} \| \w_{+,\beta, 0 } f_{+, 0 } \|_{L^\infty_{x,p}}
+ e^{\frac{m_- g}{24} \beta} \| \w_{-,\beta, 0 } f_{-, 0 } \|_{L^\infty_{x,p}} \big).
\end{split}
\Ee
Moreover, $ b (t, x)$ defined in \eqref{def:flux} is bounded by
\Be \label{Uest:D^-1_Db}
\sup_{0 \leq t < \infty} \| (-\Delta_0)^{-1}  (\nabla_x \cdot b  (t,x)) \|_{L^\infty (\O)} 
\lesssim \frac{ c }{\max \{ e_+, e_- \} \beta}.
\Ee
\end{theorem}

\begin{proof}

We abuse the notation as in \eqref{abuse} in the proof.
From \eqref{def:varrho}, then for any $(t, x) \in [0, \infty) \times \bar\O$,
\be \notag
| \varrho(t,x) |
\leq e_+ \int_{\R^3} | f_+ (t,x,p) | \dd p + e_- \int_{\R^3} | f_{-} (t,x,p) | \dd p.
\ee
Using \eqref{form:f}, together with \eqref{est:I} and \eqref{est:N}, we derive that for any $(t, x) \in [0, \infty) \times \bar\O$,
\Be \label{est:varrho}
\begin{split}
& e^{ \lambda t} e^{ \nu x_3} |\varrho (t,x)|
\\& \leq e^{ \lambda t} e^{ \nu x_3}
\sum\limits_{i = \pm} e_i \Big\{ \int_{\R^3} | \mathcal{I} _i (t,x,p)| \dd p +  \int_{\R^3} |\mathcal{N}_i (t,x,p) |\dd p \Big\}   
\\& \leq e^{ \lambda t} e^{ \nu x_3}
\sum\limits_{i = \pm} e_i \Big\{ \int_{\R^3} e^{2 + \frac{m_i g}{24} \beta} \| \w_{i, \beta, 0 } f_{i, 0 } \|_{L^\infty_{x,p}} 
e^{- \frac{m_i g}{24} \beta t } 
e^{ - \frac{\beta}{4} \sqrt{(m_i c)^2 + |p|^2} } e^{ - \frac{m_i g}{8 c} \beta x_3} \dd p 
\\& \ \ \ \ + \int_{\R^3} e^{- \frac{\beta}{8} \sqrt{(m_i c)^2 + |p|^2} } e^{- \frac{m_i g}{16 c} \beta x_3} e^{ - \lambda t} e^{ \lambda } 4 e_i \sup\limits_{s \in [0, t] } e^{ \lambda s} \| e^{ \nu x_3} \varrho (s) \|_{L^\infty(\O)} (1 + \frac{1}{\nu}) \| w_{\beta, i} \nabla_p h_i \|_{\infty} \dd p \Big\} 
\\& \leq \sum\limits_{i = \pm} e_i \Big\{ e^{2 + \frac{m_i g}{24} \beta} \| \w_{i, \beta, 0 } f_{i, 0 } \|_{L^\infty_{x,p}}  \int_{\R^3} 
e^{ - \frac{\beta}{4} \sqrt{(m_i c)^2 + |p|^2} } \dd p 
\\& \ \ \ \ + e^{ \lambda } 4 e_i  (1 + \frac{1}{\nu}) \| w_{\beta, i} \nabla_p h_i \|_{\infty} \int_{\R^3} e^{- \frac{\beta}{8} \sqrt{(m_i c)^2 + |p|^2} } \dd p  
\times \sup\limits_{s \in [0, t] } e^{ \lambda s} \| e^{ \nu x_3} \varrho (s) \|_{L^\infty(\O)} \Big\}.
\end{split}	
\Ee
Together with the condition \eqref{condition:Dvh}, we get, for any $(t, x) \in [0, \infty) \times \bar\O$,
\be \label{est2:varrho}
\begin{split}
& \eqref{est:varrho}
\\& \leq \big( e_+ e^{2 + \frac{m_+ g}{24} \beta} \| \w_{+,\beta, 0 } f_{+, 0 } \|_{L^\infty_{x,p}} + e_- e^{2 + \frac{m_- g}{24} \beta} \| \w_{-,\beta, 0 } f_{-, 0 } \|_{L^\infty_{x,p}} \big) \frac{64}{\beta^3}
\\& \ \ \ \ + 4 e^{ \lambda } (1 + \frac{1}{\nu}) \big( e^2_+ \| w_{+,\beta} \nabla_p h_+ \|_{\infty} + e^2_- \| w_{-,\beta} \nabla_p h_- \|_{\infty} \big) \frac{512}{\beta^3} \times \sup\limits_{s \in [0, t] } e^{ \lambda s} \| e^{ \nu x_3} \varrho (s) \|_{L^\infty(\O)}
\\& \leq \big( e_+ e^{2 + \frac{m_+ g}{24} \beta} \| \w_{+,\beta, 0 } f_{+, 0 } \|_{L^\infty_{x,p}} + e_- e^{2 + \frac{m_- g}{24} \beta} \| \w_{-,\beta, 0 } f_{-, 0 } \|_{L^\infty_{x,p}} \big) \frac{64}{\beta^3}
\\& \ \ \ \ + \frac{1}{2} \sup\limits_{s \in [0, t] } e^{ \lambda s} \| e^{ \nu x_3} \varrho (s) \|_{L^\infty(\O)}.
\end{split}	
\Ee
This implies that for any $(t, x) \in [0, \infty) \times \bar\O$,
\be \notag
\frac{1}{2} e^{ \lambda t} e^{ \nu x_3} |\varrho (t,x)|
\leq \big( e_+ e^{2 + \frac{m_+ g}{24} \beta} \| \w_{+,\beta, 0 } f_{+, 0 } \|_{L^\infty_{x,p}} + e_- e^{2 + \frac{m_- g}{24} \beta} \| \w_{-,\beta, 0 } f_{-, 0 } \|_{L^\infty_{x,p}} \big) \frac{64}{\beta^3},
\ee
and thus we conclude \eqref{decay:varrho}.

\smallskip

Next, inputting \eqref{decay:varrho} into \eqref{est:I} and \eqref{est:N}, together with $f_{\pm} (t,x,p) = \mathcal{I}_{\pm} (t,x,p) + \mathcal{N}_{\pm} (t,x,p)$, we derive that for any $(t,x,p) \in [0, \infty) \times  \bar\O \times \R^3$,
\be \notag
\begin{split}
e^{ \lambda t} | f (t,x,p) |
& \leq e^{ \lambda t} \big( | \mathcal{I} (t,x,p) | + | \mathcal{N} (t,x,p) | \big)
\\& \leq e^{2 + \frac{m g}{24} \beta} \| \w_{\beta, 0 }   f_{0 } \|_{L^\infty_{x,p}} 
e^{ - \frac{\beta}{4} \sqrt{(m c)^2 + |p|^2} } e^{ - \frac{m g}{8 c} \beta x_3} 
+ \big( e^{2 + \frac{m_+ g}{24} \beta} \| \w_{+,\beta, 0 } f_{+, 0 } \|_{L^\infty_{x,p}}
\\& \ \ \ \ + e^{2 + \frac{m_- g}{24} \beta} \| \w_{-,\beta, 0 } f_{-, 0 } \|_{L^\infty_{x,p}} \big) e^{- \frac{\beta}{8} \sqrt{(m_{\pm} c)^2 + |p|^2} } e^{- \frac{m_{\pm} g}{16 c} \beta x_3},
\end{split}
\ee
and conclude \eqref{decay_f}.  
Using \eqref{decay_f} and \eqref{est:D^-1_Db} in Lemma \ref{lem:Ddb}, we conclude \eqref{Uest:D^-1_Db}.
\end{proof}

\subsection{A Priori Estimate} 
\label{sec:RD} 

In this section, we establish a priori estimate of $(F, \phi_F)$ solving \eqref{VP_F}, \eqref{Poisson_F}, \eqref{VP_0},  \eqref{Dbc:F} and \eqref{bdry:F}.  
Recall from \eqref{def:F}, $F(t,x,p)= h(x,p)+f(t,x,p)$ where $(h, \rho, \Phi)$ solves \eqref{VP_h}-\eqref{eqtn:Dphi} and $(f, \varrho, \Psi)$ solves \eqref{eqtn:f}-\eqref{Poisson_f} respectively.

Throughout this section, we assume a compatibility condition:
\Be \label{CC_F0=G}
F_{\pm, 0} (x, p) = G_{\pm} (x, p) \ \ \text{on} \ (x, p) \in \gamma_-.
\Ee
Further, we always suppose \eqref{Bootstrap} holds, which is analogous to the condition \eqref{Uest:DPhi} in the case of steady solution $(h, \Phi)$.
Again we utilize the kinetic distance function $\alpha_{\pm, F} (t, x, p)$ of $\eqref{alpha_F}$.
Therefore, we omit the proofs of some results that are similar to Lemmas in Section \ref{sec:RS}. 

\begin{lemma} \label{VL:dyn} 

Let $\alpha_{\pm, F} (t,x,p)$ be defined as in \eqref{alpha_F}.
Suppose the assumption \eqref{Bootstrap} holds. Recall the characteristics $\Z_{\pm} (s;t,x,p) = (\X_{\pm} (s;t,x,p), \V_{\pm} (s;t,x,p) )$ solving \eqref{ODE_F}. For all $(t,x,p) \in \R_+ \times \O \times \R^3$ satisfying $t - \tBp (t,x,p) \geq 0$ and $s \in [ t - \tBp (t,x,p), t]$,
\Be \label{est:alpha:dyn}
\begin{split}
& \alpha_{\pm, F} (s, \X_{\pm} (s;t,x,p), \P_{\pm} (s;t,x,p))
\\& \leq \alpha_{\pm, F} (t,x,p) 
e^{ \big( 4g + \| \p_{x_3}^2  \phi_F \|_{L^\infty(\bar{\O})} + e_{\pm} \| \p_t \p_{x_3} \phi_F (t, x_\parallel , 0) \|_{L^\infty(\p\O)} + \|  \nabla_{x_\parallel} \p_{x_3} \phi_F \|_{L^\infty(\p\O)} \big) |t - s| },
\end{split}
\Ee
and
\Be \label{est:alpha:dyn_lower}
\begin{split}
& \alpha_{\pm, F} (s, \X_{\pm} (s;t,x,p), \P_{\pm} (s;t,x,p))
\\& \geq \alpha_{\pm, F} (t,x,p) 
e^{ - \big( 4g + \| \p_{x_3}^2 \phi_F \|_{L^\infty(\bar{\O})} + e_{\pm} \| \p_t \p_{x_3} \phi_F (t, x_\parallel , 0) \|_{L^\infty(\p\O)} + \| \nabla_{x_\parallel} \p_{x_3} \phi_F \|_{L^\infty(\p\O)} \big) |t - s| }.
\end{split}
\Ee
Furthermore, we have
\begin{align}
& | \vBpn (t,x,p) | 
\notag \\
& \geq \alpha_{\pm, F} (t,x,p) 
e^{  - \big( 4g + \| \p_{x_3}^2 \phi_F \|_{L^\infty(\bar{\O})} + e_{\pm} \| \p_t \p_{x_3} \phi_F (t, x_\parallel , 0) \|_{L^\infty(\p\O)} + \|  \nabla_{x_\parallel} \p_{x_3} \phi_F \|_{L^\infty(\p\O)} \big) |\tBp| }
\label{est1:alpha:dyn} \\
& \geq {\alpha_{\pm, F} (t,x,p) }
 e^{ -\frac{8}{m_{\pm} g} \big( 4g + 2 \| \nabla_x ^2 \phi_F \|_{\infty} + e_{\pm} \| \p_t \p_{x_3} \phi_F (t, x_\parallel , 0) \|_{L^\infty(\p\O)} \big) |\pBp^0 |}, 
\label{est1:xb_x/w:dyn}
\end{align}
where $\pBp^0 = \sqrt{(m_{\pm} c)^2 + |\pBp (t,x,p)|^2}$.
\end{lemma}

\begin{proof}

We abuse the notation as in \eqref{abuse} in the proof.
Recall from \eqref{alpha_F},
\be \notag
\alpha_F (t,x,p) = \sqrt{ |x_3|^2  + |v_3|^2 +  2 ( e \p_{x_3} \phi_F (t, x_\parallel , 0) + m g ) \frac{c x_3}{ p^0 } },
\ee
where $p^0 = \sqrt{ (m c)^2 + |p|^2}$ and $v =  \frac{p}{ \sqrt{ m^2 + |p|^2/ c^2} }$.
From Lemma \ref{VL}, together with the following:
\be
\p_t \alpha_F (t,x,p) = 2 e \p_t \p_{x_3} \phi_F (t, x_\parallel , 0) \frac{c x_3}{ p^0 },
\ee
then we get
\Be \label{est1:alpha_dy}
\begin{split}
& \big[ \p_t + v \cdot \nabla_x + \big( e ( \frac{v}{c} \times B - \nabla_x \phi_F ) - \nabla_x (m g x_3) \big) \cdot \nabla_p \big] \alpha_F (t, x, p)
\\& = \frac{ 2 e \p_t \p_{x_3} \phi_F (t, x_\parallel , 0) \frac{c x_3}{ p^0 } }{ \sqrt{ |x_3|^2  + |v_3|^2 +  2 (\p_{x_3} \phi_F (x_\parallel , 0) + g ) \frac{c x_3}{ p^0 } } }
\\& \ \ \ \ \frac{ + v_3 x_3 
- \left[ v_3 (\frac{c}{p^0} - \frac{(v_3)^2}{ c p^0}) ( e \p_{x_3} \phi_F (x) + m g ) 
- \frac{ (v_3)^2 }{ c p^0 } v_{\|} \cdot e \p_{x_{\|}} \phi_F (x) \right] }{ * }
\\& \ \ \ \ \frac{ + \left[ v_\parallel \cdot \nabla_{x_\parallel} e \p_{x_3} \phi_F   (x_\parallel, 0) \frac{c x_3}{ p^0 } + \frac{c v_3}{ p^0 } ( e \p_{x_3} \phi_F (x_\parallel , 0) + m g ) \right] }{ * }
\\& \ \ \ \ \frac{ + \big( e \p_{x_3} \phi_F (x_\parallel, 0) + m g \big) \frac{ x_3 }{ (p^0)^2} v \cdot \big( e \nabla_x \phi_F (x) + \nabla_x (m g x_3) \big) }{*} .
\end{split}
\Ee
Using the fundamental theorem of calculus in \eqref{est2:alpha_steady}, we obtain
\be \notag
\begin{split}
\eqref{est1:alpha_dy} 
& \leq \frac{ v_3 x_3 + \frac{c v_3}{ p^0 } x_3 \| \p_{x_3} \p_{x_3} \phi_F \|_{L^\infty(\bar{\O})}
+ \frac{(v_3)^3}{ c p^0 } ( e \p_{x_3} \phi_F (x)+ m g ) 
+ \frac{(v_3)^2}{ c p^0 } v_{\|} \cdot e \p_{x_{\|}} \phi_F (x) }{ \sqrt{ |x_3|^2  + |v_3|^2 +  2 (\p_{x_3} \phi_F (x_\parallel , 0) + g ) \frac{x_3}{\langle v \rangle} } }
\\& \ \ \ \ \frac{ +  2 e \| \p_t \p_{x_3} \phi_F (t, x_\parallel , 0) \|_{L^\infty(\p\O)} \frac{c x_3}{ p^0 }
+ \ v_\parallel \cdot \nabla_{x_\parallel} e \p_{x_3} \phi_F (x_\parallel, 0) \frac{c x_3}{ p^0 } }{*}
\\& \ \ \ \ \frac{+ ( e \p_{x_3} \phi_F (x_\parallel, 0) + m g ) \frac{ x_3 }{ (p^0)^2} v \cdot \big( e \nabla_x \phi_F (x) + \nabla_x (m g x_3) \big) }{*}
\\& \leq \frac{  (|x_3|^2  + |v_3|^2) \big( \frac{1}{2} + \frac{c }{2 p^0} \| \p_{x_3} \p_{x_3} \phi_F \|_{L^\infty(\bar{\O})} 
+ \frac{ 3 |v| }{ c p^0 } ( e \| \nabla_{x} \phi_F \|_{L^\infty(\p\O)} + m g ) \big)}{ \sqrt{ |x_3|^2  + |v_3|^2 +  2 (\p_{x_3} \phi_F (x_\parallel , 0) + g ) \frac{x_3}{\langle v \rangle} } }
\\& \ \ \ \ \frac{ +  2 e \| \p_t \p_{x_3} \phi_F (t, x_\parallel , 0) \|_{L^\infty(\p\O)} \frac{c x_3}{ p^0 } + \ 2 |v_\parallel| \|  \nabla_{x_\parallel} \p_{x_3} \phi_F \|_{L^\infty(\p\O)} \frac{c x_3}{ p^0 } }{*} 
\\& \ \ \ \ \frac{ + ( e \p_{x_3} \phi_F (x_\parallel, 0) + m g )  \frac{c x_3}{ p^0 } \big( \frac{ v }{ (p^0)^2} \cdot \nabla_x (e \phi_F (x) + m g x_3) \big) }{*}.
\end{split}
\ee
Using \eqref{Bootstrap}, we have 
\Be\notag
\begin{split}
& \big| \big[ \p_t + v \cdot \nabla_x + \big( e ( \frac{v}{c} \times B - \nabla_x \phi_F ) - \nabla_x (m g x_3) \big) \cdot \nabla_p \big] \alpha_F (t, x, p) \big|
\\& \leq \big( 4g + \| \p_{x_3} \p_{x_3} \phi_F \|_{L^\infty(\bar{\O})} + e \| \p_t \p_{x_3} \phi_F (t, x_\parallel , 0) \|_{L^\infty(\p\O)} + \|  \nabla_{x_\parallel} \p_{x_3} \phi_F \|_{L^\infty(\p\O)} \big) \alpha_F (t,x,p).
\end{split}
\Ee
Using the Gronwall's inequality, we conclude both \eqref{est:alpha:dyn} and \eqref{est:alpha:dyn_lower}. 

\smallskip

For \eqref{est1:alpha:dyn}, we set $s = t - \tB (t,x,p)$ in \eqref{est:alpha:dyn_lower} and get 
\[
\alpha_F (t - \tB (t,x,p), \X (t - \tB (t,x,p);t,x,p),\P (t - \tB (t,x,p);t,x,p)) = | \vBn (t,x,p)|.
\]
Thus we obtain the first inequality \eqref{est1:alpha:dyn}. Together with \eqref{ODE_F} and Proposition \ref{lem3:tB}, we have 
\Be \notag
\tB (t, x, p) \leq \frac{8}{m g} \sqrt{(m c)^2 + |\pB (t,x,p)|^2},
\Ee
and therefore prove \eqref{est1:xb_x/w:dyn}. 
\end{proof}

\smallskip

Now suppose the assumption \eqref{Bootstrap} holds. We refer to Section \ref{sec:RS} and collect some basic properties and estimates on $F_{\pm} (t,x,p)$.

\smallskip

(a)
From \eqref{VP_0} and \eqref{bdry:F}, we rewrite $F_{\pm} (t,x,p)$ as
\Be \label{form:F}
\begin{split} 
F_{\pm} (t,x,p) 
& = \mathbf{1}_{t \leq \tBp (t,x,p) } F_{\pm, 0} (\X_{\pm} (0;t,x,p), \P_{\pm} (0;t,x,p))
\\& \ \ \ \ + \mathbf{1}_{t > \tBp (t,x,p) }  G_{\pm} ( \xBp (t,x,p), \pBp (t,x,p)).
\end{split}
\Ee
Therefore, the derivatives of $F_\pm (t,x,p)$ are give by
\Be \label{DF_x}
\begin{split}
\p_{x_i} F_{\pm} (t,x,p) 
& = \mathbf{1}_{t \leq \tBp (t,x,p) }  
\{ \p_{x_i} \X_{\pm} (0;t,x,p) \cdot \nabla_x F_{\pm, 0} (\Z_{\pm} (0;t,x,p) ) 
\\& \qquad \qquad \qquad  + \p_{x_i} \P_{\pm} (0;t,x,p) \cdot \nabla_p F_{\pm, 0} (\Z_{\pm} (0;t,x,p) ) \}
\\& \ \ \ \ + \mathbf{1}_{t >\tBp (t,x,p) }  
\{ \p_{x_i} \xBp \cdot \nabla_{x_\parallel} G_{\pm} (\xBp (t,x,p), \pBp (t,x,p))
\\& \qquad \qquad \qquad + \p_{x_i} \pBp \cdot \nabla_p G_{\pm} (\xBp (t,x,p), \pBp (t,x,p)) \},
\end{split}
\Ee
and
\Be \label{DF_v}
\begin{split}
\p_{p_i} F_{\pm} (t,x,p) 
& = \mathbf{1}_{t \leq \tBp (t,x,p) }  
\{ \p_{p_i} \X_{\pm} (0;t,x,p) \cdot \nabla_x F_{\pm, 0} (\Z (0;t,x,p) ) 
\\& \qquad \qquad \qquad + \p_{p_i} \P_{\pm} (0;t,x,p) \cdot \nabla_p F_{\pm, 0} (\Z_{\pm} (0;t,x,p) ) \}
\\& \ \ \ \ + \mathbf{1}_{t > \tBp (t,x,p) }  
\{ \p_{p_i} \xBp \cdot \nabla_{x_\parallel} G_{\pm} (\xBp (t,x,p), \pBp (t,x,p))
\\& \qquad \qquad \qquad + \p_{p_i} \pBp \cdot \nabla_p G_{\pm} (\xBp (t,x,p), \pBp (t,x,p)) \},
\end{split}
\Ee
where the characteristics $\Z_{\pm} (s;t,x,p) = (\X_{\pm} (s;t,x,p), \P_{\pm} (s;t,x,p))$ solves \eqref{ODE_F}.

\smallskip

(b)
Analogous to the proof of Lemma \ref{lem:XV_xv}, we can conclude the following:
\begin{align}
|\nabla_{x} \X_{\pm} (s;t,x,p)| & \leq \min \big\{  e^{\frac{|t-s|^2 + 2|t-s| }{2} (\| \nabla_x^2 \phi_F  \|_\infty + e_{\pm} B_3 + m_{\pm} g ) }, e^{ (1 + B_3 + \| \nabla_x ^2 \phi_F  \|_\infty) |t-s|} \big\}, 
\label{est:X_x:dyn} \\
|\nabla_{x} \P_{\pm} (s;t,x,p)| & \leq \min \big\{ |t-s| ( B_3 + \| \nabla_x ^2 \phi_F \|_\infty) e^{ (1+ B_3 + \| \nabla_x ^2 \phi_F \|_\infty) |t-s| },
\notag \\
& \qquad \qquad \qquad \qquad \qquad e^{ (1+ B_3 + \| \nabla_x ^2 \phi_F \|_\infty) |t-s|}
\big\}, 
\label{est:V_x:dyn} \\
|\nabla_{p} \X_{\pm} (s;t,x,p)| &\leq   \min \big\{ |t-s| e^{ (1 + B_3 + \| \nabla_x ^2 \phi_F  \|_\infty) |t-s|},
e^{ (1 + B_3 + \| \nabla_x ^2 \phi_F  \|_\infty) |t-s|}
\big\}, 
\label{est:X_v:dyn} \\
|\nabla_{p} \P_{\pm} (s;t,x,p)| & \leq  \min \big\{ 1 + |t-s| (B_3 + \| \nabla_x ^2 \phi_F \|_{\infty}) e^{(1 + B_3 + \|\nabla_x^2 \phi_F \|_\infty) |t-s|},
\notag \\
& \qquad \qquad \qquad \qquad \qquad e^{ (1 + B_3 + \| \nabla_x ^2 \phi_F  \|_\infty) |t-s|}
\big\},
\label{est:V_v:dyn}
\end{align}
where we denote 
$\| \nabla_x^2\phi_{F} \|_\infty := \sup\limits_{s \in [t-\tB (t,x,p),t]} \| \nabla_x^2\phi_{F} (s, x)  \|_{L^{\infty}(\bar{\O})}$.

\smallskip

(c)
Following the proof of Lemmas \ref{lem:exp_txvb} and \ref{lem:nabla_zb}, together with \eqref{est:X_x:dyn}-\eqref{est:V_v:dyn}, then we have
\Be \label{est:xb_x:dyn}
\begin{split}
& | \p_{x_i} \xBp (t,x,p) | 
\leq \frac{ | \vBp (t,x,p) | }{ | \vBpn (t,x,p) | }  
 \delta_{i3} 
\\& \ \ \ \ + \Big(1 +  \frac{|\vBp (t,x,p)|}{|\vBpn (t,x,p)|}
\frac{|\tBp (t,x,p)|^2}{2} ( \| \nabla_x^2 \phi_F \|_\infty + e_{\pm} B_3 + m_{\pm} g ) \Big) e^{ (1 + B_3 + \| \nabla_x ^2 \phi_F  \|_\infty) |\tBp| },
\end{split}
\Ee
\Be \label{est:xb_v:dyn}
\begin{split}
& |\p_{p_i} \xBp (t,x,p)| 
\leq \frac{|\vBp (t,x,p)| |\tBp (t,x,p)|  }{|\vBpn (t,x,p)| \sqrt{ (m_{\pm})^2 + |p|^2/ c^2}}  
\\& \ \ \ \ + \Big(1 +   \frac{|\vBp (t,x,p)|}{|\vBpn (t,x,p)|}
\frac{|\tBp (t,x,p)|^2}{2} ( \| \nabla_x^2 \phi_F  \|_\infty + e_{\pm} B_3 + m_{\pm} g ) \Big) e^{ (1 + B_3 + \| \nabla_x ^2 \phi_F  \|_\infty) |\tBp| },
\end{split}
\Ee
\Be \label{est:vb_x:dyn}
\begin{split}
& |\p_{x_i} \pBp (t,x,p)| \leq \frac{ |e_{\pm} B_3 + m_{\pm} g| }{|\vBpn (t,x,p)|} \delta_{i3}
\\& \ \ \ \ + \Big(1 +  \frac{|e_{\pm} B_3 + m_{\pm} g|}{|\vBpn (t,x,p)|}
\frac{|\tBp (t,x,p)|^2}{2} ( \| \nabla_x^2 \phi_F  \|_\infty + e_{\pm} B_3 + m_{\pm} g ) \Big) e^{ (1 + B_3 + \| \nabla_x ^2 \phi_F  \|_\infty) | \tBp|},
\end{split}
\Ee
\Be \label{est:vb_v:dyn}
\begin{split}
& |\p_{p_i} \pBp (t,x,p)| \leq \frac{|e_{\pm} B_3 + m_{\pm} g| |\tBp (t,x,p)|  }{|\vBpn (t,x,p)| \sqrt{ (m_{\pm})^2 + |p|^2/ c^2}}  
\\& \ \ \ \ + \Big(1 + \frac{|e_{\pm} B_3 + m_{\pm} g|}{|\vBpn (t,x,p)|}
\frac{|\tBp (t,x,p)|^2}{2} ( \| \nabla_x^2 \phi_F  \|_\infty + e_{\pm} B_3 + m_{\pm}g ) \Big) e^{ (1 + B_3 + \| \nabla_x ^2 \phi_F  \|_\infty) |\tBp|},
\end{split}
\Ee
where
$\| \nabla_x^2\phi_{F} \|_\infty = \sup\limits_{s \in [t-\tB (t,x,p),t]} \| \nabla_x^2\phi_{F} (s, x)  \|_{L^{\infty}(\bar{\O})}$.

\begin{prop} \label{RE:dyn}

Suppose $(F, \phi_F)$ solves \eqref{VP_F}, \eqref{Poisson_F}, \eqref{VP_0}, \eqref{Dbc:F} and \eqref{bdry:F} under the compatibility condition \eqref{CC_F0=G}.
Assume that \eqref{Bootstrap} and \eqref{est:beta_b_dy} hold. Further, we assume 
\be \label{est1:D3tphi_F}
(1 + B_3 + \| \nabla_x ^2 \phi_F  \|_\infty) + \| \p_t \p_{x_3} \phi_F (t, x_\parallel , 0) \|_{L^\infty(\p\O)} 
\leq \frac{\hat{m} g}{24} \tilde{\beta},
\ee
where $\hat{m} = \min \{ m_+, m_- \}$.
Consider $(t,x,p) \in [0, \infty) \times  \bar\O \times \R^3$, then
\Be \label{est:F_v:dyn}
\begin{split} 
& \| e^{ \frac{\tilde{\beta}}{4} \sqrt{(m_{\pm} c)^2 + |p|^2} } e^{ \frac{m_{\pm} g}{8 c} \tilde{\beta} x_3} \nabla_p F_{\pm} (t,x,p) \|_{L^\infty(\O \times \R^3)} 
\\& \lesssim 2 e^{ \frac{m_{\pm} g}{24} \tilde{\beta} }
e^{2 \tilde{\beta} \frac{e_{\pm}}{c} \| (-\Delta_0)^{-1} (\nabla_x \cdot b ) \|_{L^\infty_{t,x}} } \| \w_{\pm, \tilde \beta, 0} \nabla_{x,p} F_{\pm, 0} \|_{L^\infty (\O \times \R^3)}
\\& \ \ \ \ + \big( 1 + \frac{ \| \nabla_x^2 \phi_F \|_\infty + e_{\pm} B_3 + m_{\pm} g}{(m_{\pm} g)^2} \big) \| e^{\tilde{\beta} |p^0_{\pm}|} \nabla_{x_\parallel,p} G_{\pm} \|_{L^\infty (\gamma_-)},
\end{split}
\Ee 
and 
\Be \label{est:F_x:dyn}
\begin{split} 
& e^{ \frac{\tilde{\beta}}{4} \sqrt{(m_{\pm} c)^2 + |p|^2} } e^{ \frac{m_{\pm} g}{8 c} \tilde{\beta} x_3} \big| \nabla_{x_i} F_{\pm} (t,x,p) \big|
\\& \lesssim 2 e^{ \frac{m_{\pm} g}{24} \tilde{\beta} }
e^{2 \tilde{\beta} \frac{e_{\pm}}{c} \| (-\Delta_0)^{-1} (\nabla_x \cdot b ) \|_{L^\infty_{t,x}} } \| \w_{\pm, \tilde \beta, 0}   \nabla_{x,p} F_{\pm, 0} \|_{L^\infty (\O \times \R^3)}
\\& \ \ \ \ + \big( \frac{\delta_{i3}}{\alpha_{\pm, F} (t,x,p)} + 1 + \frac{ \| \nabla_x^2 \phi_F \|_\infty + e_{\pm} B_3 + m_{\pm} g}{(m_{\pm} g)^2} \big) \| e^{\tilde{\beta} |p^0_{\pm}|} \nabla_{x_\parallel,p} G_{\pm} \|_{L^\infty (\gamma_-)},
\end{split}
\Ee 
where 
$\| \nabla_x^2\phi_{F} \|_\infty = \sup\limits_{s \in [t-\tB (t,x,p),t]} \| \nabla_x^2\phi_{F} (s, x)  \|_{L^{\infty}(\bar{\O})}$, and $\alpha_{\pm, F} (t,x,p)$ is defined in \eqref{alpha_F}.
\end{prop}

\begin{proof}

We follow the proof of \eqref{est:hk_v} and \eqref{est:hk_x} in Proposition \ref{prop:Reg}.
Here, for simplicity, we 
use the notation: $p_\pm^0 = \sqrt{(m_\pm c)^2 + |p|^2}$ in the rest of proof. 
We will also use \eqref{abuse}.

\smallskip

\textbf{Step 1. Proof of \eqref{est:F_v:dyn}.}
From \eqref{DF_v}, we have
\Be \notag
\begin{split}
\p_{p_i} F (t,x,p) 
& = \mathbf{1}_{t \leq \tB(t,x,p) }  
\{ \p_{p_i} \X (0;t,x,p) \cdot \nabla_x F_0 (\Z (0;t,x,p) ) 
\\& \qquad \qquad \qquad \qquad + \p_{p_i} \P (0;t,x,p) \cdot \nabla_p F_0 (\Z (0;t,x,p) ) \}
\\& \ \ \ \ + \mathbf{1}_{t >\tB (t,x,p) }  
\{ \p_{p_i} \xB \cdot \nabla_{x_\parallel} G (\xB (t,x,p), \pB (t,x,p))
\\& \qquad \qquad \qquad \qquad + \p_{p_i} \pB \cdot \nabla_p G (\xB (t,x,p), \pB (t,x,p)) \},
\end{split}
\Ee
This shows that
\begin{align}
& | \nabla_p F (t,x,p) | 
\notag \\
& \leq \mathbf{1}_{t \leq \tB (t,x,p)}
\frac{| \nabla_p \X  (0;t,x,p)| +| \nabla_p \P  (0;t,x,p)|  }{\w_{\tilde \beta} (0, \Z (0;t,x,p) )} \| \w_{\tilde \beta, 0}   \nabla_{x,p} F_0  \|_{L^\infty (\O \times \R^3)} \label{est1:F_v1} \\
& \ \ \ \ +	\mathbf{1}_{t > \tB (t,x,p)} \frac{ | \nabla_p \xB (t,x,p) | + |\nabla_p \pB  (t,x,p)|}{ e^{\tilde{\beta} |\pB^0 (t,x,p)| } } \| e^{\tilde{\beta} |p^0|} \nabla_{x_\parallel,p} G \|_{L^\infty (\gamma_-)}.
\label{est1:F_v2}
\end{align} 

For \eqref{est1:F_v1}, using \eqref{est:X_v:dyn} and \eqref{est:V_v:dyn}, together with \eqref{est:1/w_h} and Corollary \ref{cor:max_X_dy}, we get
\Be \label{est2:F_v1}
\begin{split}
\eqref{est1:F_v1} 
& \leq \mathbf{1}_{t \leq \tB (t,x,p)} e^{2 \tilde{\beta} \frac{e}{c} \| (-\Delta_0)^{-1} (\nabla_x \cdot b ) \|_{L^\infty_{t,x}} } e^{ - \frac{\tilde{\beta}}{2} \sqrt{(m c)^2 + |p|^2} } e^{ - \frac{m g}{4 c} \tilde{\beta} x_3}
\\& \qquad \times 2 e^{ (1 + B_3 + \| \nabla_x ^2 \phi_F  \|_\infty) |t| } \| \w_{\tilde \beta, 0}   \nabla_{x,p} F_0  \|_{L^\infty (\O \times \R^3)}
\\& \leq 2 e^{2 \tilde{\beta} \frac{e}{c} \| (-\Delta_0)^{-1} (\nabla_x \cdot b ) \|_{L^\infty_{t,x}} } e^{ - \frac{\tilde{\beta}}{2} \sqrt{(m c)^2 + |p|^2} } e^{ - \frac{m g}{4 c} \tilde{\beta} x_3} \| \w_{\tilde \beta, 0}   \nabla_{x,p} F_0  \|_{L^\infty (\O \times \R^3)}
\\& \qquad \times e^{ \frac{\tilde{\beta}}{4} \sqrt{(m c)^2 + |p|^2} } e^{ \frac{m g}{8 c} \tilde{\beta} x_3} e^{ \frac{m g}{24} \tilde{\beta} }
\\& \leq 2 e^{ \frac{m g}{24} \tilde{\beta} } e^{2 \tilde{\beta} \frac{e}{c} \| (-\Delta_0)^{-1} (\nabla_x \cdot b ) \|_{L^\infty_{t,x}} } e^{ - \frac{\tilde{\beta}}{4} \sqrt{(m c)^2 + |p|^2} } e^{ - \frac{m g}{8 c} \tilde{\beta} x_3} \| \w_{\tilde \beta, 0}   \nabla_{x,p} F_0  \|_{L^\infty (\O \times \R^3)},
\end{split}
\Ee
where the second last inequality follows from \eqref{est:1_t<tB} and \eqref{est1:D3tphi_F}.

For \eqref{est1:F_v2}, using \eqref{est:xb_v:dyn}, we derive that
\Be \notag
\begin{split}
& \frac{ |\p_{p_i} \xB (t,x,p)| }{ e^{\tilde{\beta} |\pB^0 (t,x,p)|} }
\\& \leq \frac{|\vB  (t,x,p)| | \pBn (t,x, p) | }{|\vBn  (t,x,p)| \sqrt{ m^2 + |p|^2/ c^2}} \Big( \frac{2}{m g}  + \frac{ 16 (\sqrt{c m^2 g} + | \pBn (t,x,p) |)}{c m^3 g^2} 
\\& \qquad \qquad \times \big(1 + \sqrt{(m c)^2 + |\pB (t,x,p)|^2} + 2 \pBn (t,x, p) \big) \Big) e^{- \tilde{\beta} |\pB^0 (t,x,p)|}
\\& \ \ \ \ + \Big( 1 +   \frac{|\vB (t,x,p)| | \pBn (t,x, p) |}{2 |\vBn (t,x,p)|} |\tB (t,x,p)| \Big( \frac{2}{m g}  + \frac{ 16 (\sqrt{c m^2 g} + | \pBn (t,x,p) |)}{c m^3 g^2} 
\\& \qquad \qquad \times \big(1 + \sqrt{(m c)^2 + |\pB (t,x,p)|^2} + 2 \pBn (t,x, p) \big) \Big)
( \| \nabla_x^2 \phi_F  \|_\infty + e B_3 + mg ) \Big) 
\\& \qquad \qquad \qquad \qquad \times e^{ (1 + B_3 + \| \nabla_x ^2 \phi_F  \|_\infty) \tB (t,x,p) } e^{- \tilde{\beta} |\pB^0 (t,x,p)|}.
\end{split}
\Ee
Using Proposition \ref{lem2:tB}, we deduce
\Be \label{est1:xb_v_a:dyn}
\begin{split}
& \frac{ |\p_{p_i} \xB (t,x,p)| }{ e^{\tilde{\beta} |\pB^0 (t,x,p)|} }
\\& \leq \frac{|\pB  (t,x,p)| }{ \sqrt{ m^2 + |p|^2/ c^2}} \Big( \frac{2}{m g}  + \frac{ 16 (\sqrt{c m} + | \pBn (t,x,p) |)}{c m^2 g} 
\\& \qquad \qquad \times \big(1 + \sqrt{(m c)^2 + |\pB (t,x,p)|^2} + 2 \pBn (t,x, p) \big) \Big) e^{- \tilde{\beta} |\pB^0 (t,x,p)|}
\\& \ \ \ \ + \Big( 1 +   \frac{|\pB (t,x,p)|}{2 } |\tB (t,x,p)| \Big( \frac{2}{m g}  + \frac{ 16 (\sqrt{c m} + | \pBn (t,x,p) |)}{c m^2 g} 
\\& \qquad \qquad \times \big(1 + \sqrt{(m c)^2 + |\pB (t,x,p)|^2} + 2 \pBn (t,x, p) \big) \Big)
( \| \nabla_x^2 \phi_F  \|_\infty + e B_3 + mg ) \Big) 
\\& \qquad \qquad \qquad \qquad \times e^{ (1 + B_3 + \| \nabla_x ^2 \phi_F  \|_\infty) \tB (t,x,p) } e^{- \tilde{\beta} |\pB^0 (t,x,p)|}.
\end{split}
\Ee
From Proposition \ref{lem3:tB}, we have
\Be \notag
\begin{split}
\tB (t,x, p) 
& \leq \frac{8}{m g} \sqrt{(m c)^2 + |\pB (t,x,p)|^2}.
\end{split}
\Ee
and
\Be \notag 
|p^0| + \frac{m g}{2 c} x_3
= \sqrt{(m c)^2 + |p|^2} + \frac{m g}{2 c} x_3
\leq \frac{7}{6} |\pB^0 (t,x,p)|.
\ee
Together with \eqref{est1:D3tphi_F}, we further obtain
\Be \label{est1:xb_v:dyn}
\begin{split}
\eqref{est1:xb_v_a:dyn}
& \lesssim \frac{1}{m^2 g} e^{- \frac{7 \tilde{\beta}}{12} |\pB^0 (t,x,p)|} + \big( 1 + \frac{\| \nabla_x^2 \phi_F \|_\infty + e B_3 + mg}{(m g)^2} \big) e^{- \frac{7 \tilde{\beta}}{12} |\pB^0 (t,x,p)|} 
\\& \lesssim \big( 1 + \frac{ \| \nabla_x^2 \phi_F \|_\infty + e B_3 + mg}{(m g)^2} \big) e^{ - \frac{\tilde \beta}{2} \big( |p^0| + \frac{m g}{2 c} x_3 \big) }.
\end{split}
\Ee
Analogously, using \eqref{est:xb_v:dyn}, Propositions \ref{lem2:tB} and \ref{lem3:tB}, we derive that
\Be \label{est1:vb_v:dyn}
\begin{split}
\frac{ |\p_{p_i} \pB (t,x,p)| }{e^{\tilde{\beta} |\pB^0 (t,x,p)|}}
& \lesssim \frac{|e B_3 + m g|}{m^2 g} e^{- \frac{7 \tilde{\beta}}{12} |\pB^0 (t,x,p)|} + \big( 1 + \frac{\| \nabla_x^2 \phi_F \|_\infty + e B_3 + mg}{(m g)^2} \big) e^{- \frac{7 \tilde{\beta}}{12} |\pB^0 (t,x,p)|}
\\& \lesssim \big( 1 + \frac{ \| \nabla_x^2 \phi_F \|_\infty + e B_3 + mg}{(m g)^2} \big) e^{ - \frac{\tilde \beta}{2} \big( |p^0| + \frac{m g}{2 c} x_3 \big) }.
\end{split}
\Ee
Thus, we conclude \eqref{est:F_v:dyn} by inputting \eqref{est1:xb_v:dyn}, \eqref{est1:vb_v:dyn}, and \eqref{est2:F_v1} into \eqref{est1:F_v1} and \eqref{est1:F_v2}.

\smallskip

\textbf{Step 2. Proof of \eqref{est:F_x:dyn}.}
Analogous to step 1, from \eqref{DF_x}, we have
\begin{align}
& | \nabla_{x_i} F (t,x,p) | 
\notag \\
& \leq \mathbf{1}_{t \leq \tB (t,x,p)}
\frac{| \nabla_{x_i} \X  (0;t,x,p)| +| \nabla_{x_i} \P  (0;t,x,p)|  }{\w_{\tilde \beta} (0, \Z (0;t,x,p) )} \| \w_{\tilde \beta, 0}   \nabla_{x,p} F_0  \|_{L^\infty (\O \times \R^3)} \label{est:F_x1} \\
& \ \ \ \ +	\mathbf{1}_{t > \tB (t,x,p)} \frac{ | \nabla_{x_i} \xB (t,x,p) | + |\nabla_{x_i} \pB  (t,x,p)|}{ e^{\tilde{\beta} |\pB^0 (t,x,p)| } } \| e^{\tilde{\beta} |p^0|} \nabla_{x_\parallel,p} G \|_{L^\infty (\gamma_-)}.
\label{est:F_x2}
\end{align} 

For \eqref{est:F_x1}, using \eqref{est:X_x:dyn}, \eqref{est:V_x:dyn} and \eqref{est2:F_v1}, together with \eqref{est:1/w_h}, \eqref{est:1_t<tB}, \eqref{est1:D3tphi_F} and Corollary \ref{cor:max_X_dy}, we get
\Be \label{est1:F_x1}
\begin{split} 
\eqref{est:F_x1}
& \leq \mathbf{1}_{t \leq \tB (t,x,p)} e^{2 \tilde{\beta} \frac{e}{c} \| (-\Delta_0)^{-1} (\nabla_x \cdot b ) \|_{L^\infty_{t,x}} } e^{ - \frac{\tilde{\beta}}{2} \sqrt{(m c)^2 + |p|^2} } e^{ - \frac{m g}{4 c} \tilde{\beta} x_3}
\\& \qquad \times 2 e^{ (1 + B_3 + \| \nabla_x ^2 \phi_F  \|_\infty) |t| } \| \w_{\tilde \beta, 0}   \nabla_{x,p} F_0  \|_{L^\infty (\O \times \R^3)}
\\& \leq 2 e^{ \frac{m g}{24} \tilde{\beta} }
e^{2 \tilde{\beta} \frac{e}{c} \| (-\Delta_0)^{-1} (\nabla_x \cdot b ) \|_{L^\infty_{t,x}} } e^{ - \frac{\tilde{\beta}}{4} \sqrt{(m c)^2 + |p|^2} } e^{ - \frac{m g}{8 c} \tilde{\beta} x_3} \| \w_{\tilde \beta, 0}   \nabla_{x,p} F_0  \|_{L^\infty (\O \times \R^3)}.
\end{split}
\Ee

For \eqref{est:F_x2}, using \eqref{est:xb_x:dyn}, \eqref{est1:xb_x/w:dyn} in Lemma \ref{VL:dyn}, together with Proposition \ref{lem2:tB}, we derive
\Be \label{est1:xb_x_a:dyn}
\begin{split}
& \frac{ |\p_{x_i} \xB (t,x,p)| }{ e^{\tilde{\beta} |\pB^0 (t,x,p)|} }
\\& \leq \frac{ | \vB (t,x,p) | }{ | \vBn (t,x,p) | } \delta_{i3}  e^{- \tilde{\beta} |\pB^0 (t,x,p)|}
\\& \ \ \ \ + \Big( 1 +   \frac{|\vB (t,x,p)| | \pBn (t,x, p) |}{2 |\vBn (t,x,p)|} |\tB (t,x,p)| \Big[ \frac{2}{m g}  + \frac{ 16 (\sqrt{c m} + | \pBn (t,x,p) |)}{c m^2 g} 
\\& \qquad \qquad \times \big(1 + \sqrt{(m c)^2 + |\pB (t,x,p)|^2} + 2 \pBn (t,x, p) \big) \Big]
( \| \nabla_x^2 \phi_F  \|_\infty + e B_3 + mg ) \Big) 
\\& \qquad \qquad \qquad \qquad \times e^{ (1 + B_3 + \| \nabla_x ^2 \phi_F  \|_\infty) \tB (t,x,p) } e^{- \tilde{\beta} |\pB^0 (t,x,p)|}
\\& \leq \frac{ | \vB (t,x,p) | }{ \alpha_F (t,x,p) }  
e^{ \frac{8}{m g} (4g + 2 \| \nabla_x ^2 \phi_F \|_{\infty} + e \| \p_t \p_{x_3} \phi_F (t, x_\parallel , 0) \|_{L^\infty(\p\O)} ) |\pB^0 |} \delta_{i3} e^{- \tilde{\beta} |\pB^0 (t,x,p)|}
\\& \ \ \ \ + \Big( 1 +   \frac{|\pB (t,x,p)|}{2 } |\tB (t,x,p)| \Big[ \frac{2}{m g}  + \frac{ 16 (\sqrt{c m} + | \pBn (t,x,p) |)}{c m^2 g} 
\\& \qquad \qquad \times \big(1 + \sqrt{(m c)^2 + |\pB (t,x,p)|^2} + 2 \pBn (t,x, p) \big) \Big]
( \| \nabla_x^2 \phi_F  \|_\infty + e B_3 + mg ) \Big) 
\\& \qquad \qquad \qquad \qquad \times e^{ (1 + B_3 + \| \nabla_x ^2 \phi_F \|_\infty) \tB (t,x,p) } e^{- \tilde{\beta} |\pB^0 (t,x,p)|}.
\end{split}
\Ee
Together with \eqref{est1:D3tphi_F} and Proposition \ref{lem3:tB}, we further have
\Be \label{est1:xb_x:dyn}
\begin{split}
\eqref{est1:xb_x_a:dyn}
& \lesssim \frac{\delta_{i3}}{\alpha_F (t,x,p)} e^{- \frac{7 \tilde{\beta}}{12} |\pB^0 (t,x,p)|} + \big( 1 + \frac{\| \nabla_x^2 \phi_F \|_\infty + e B_3 + mg}{(m g)^2} \big) e^{- \frac{7 \tilde{\beta}}{12} |\pB^0 (t,x,p)|} 
\\& \lesssim \big( \frac{\delta_{i3}}{\alpha_F (t,x,p)} + 1 + \frac{ \| \nabla_x^2 \phi_F \|_\infty + e B_3 + mg}{(m g)^2} \big) e^{ - \frac{\tilde \beta}{2} \big( |p^0| + \frac{m g}{2 c} x_3 \big) }.
\end{split}
\Ee
Analogous to \eqref{est1:vb_v:dyn}, using \eqref{est:vb_x:dyn}, Propositions \ref{lem2:tB} and \ref{lem3:tB}, we derive that
\Be \label{est1:vb_x:dyn}
\begin{split}
\frac{ |\p_{x_i} \pB (x,p)| }{e^{\tilde{\beta} |\pB^0 (x,p)|}}
& \lesssim \frac{\delta_{i3}}{\alpha_F (t,x,p)}  e^{- \frac{7 \tilde{\beta}}{12} |\pB^0 (x,p)|} + \big( 1 + \frac{\| \nabla_x^2 \phi_F \|_\infty + e B_3 + mg}{(m g)^2} \big) e^{- \frac{7 \tilde{\beta}}{12} |\pB^0 (x,p)|}
\\& \lesssim \big( \frac{\delta_{i3}}{\alpha_F (t,x,p)} + 1 + \frac{ \| \nabla_x^2 \phi_F \|_\infty + e B_3 + mg}{(m g)^2} \big) e^{ - \frac{\tilde \beta}{2} \big( |p^0| + \frac{m g}{2 c} x_3 \big) }.
\end{split}
\Ee
Finally, we conclude \eqref{est:F_x:dyn} by \eqref{est1:F_x1} and inputting \eqref{est1:xb_x:dyn} and \eqref{est1:vb_x:dyn} into \eqref{est:F_x2}.
\end{proof}

\begin{lemma} \label{lem:D3tphi_F}

Suppose that \eqref{Bootstrap} holds. Further, we assume 
\be \label{est:D2xphi_F}
1 + B_3 + \| \nabla_x ^2 \phi_F \|_\infty 
\leq \frac{\hat{m} g}{24} \tilde{\beta},
\ee
where $\hat{m} = \min \{ m_+, m_- \}$.
Consider $(t,x) \in [0, \infty) \times \bar\O$, then
\be \label{est:D.b}
\begin{split}
|\nabla_x \cdot b (t, x)|  
& \lesssim C_1 \frac{ \mathbf{1}_{|x_3| \leq 1} }{ \sqrt{ x_3 } } + C_2  e^{ - \frac{\hat{m} g}{8 c} \tilde{\beta} x_3},
\end{split}
\ee
and
\be \label{est:D3tphi_F} 
| \p_{x_3} \p_t \phi_{F} (t,x)| 
\lesssim C_1 + C_2 ( 1 + \frac{8 c}{\hat{m} g \tilde{\beta}} ),  
\ee
where $\hat{m} = \min\{m_+, m_- \}$, and
\begin{align}
C_{1} & = \frac{ 1 }{{\tilde{\beta}}^3} ( \frac{e_+}{\sqrt{ m^2_+ g  }} + \frac{e_-}{\sqrt{ m^2_- g  }} ),
\label{express:C} \\
C_{2} & = \sum\limits_{i = \pm} \frac{e_i }{{\tilde{\beta}}^3} \Big( 2 e^{ \frac{m_i g}{24} \tilde{\beta} } 
e^{2 \tilde{\beta} \frac{e_i}{c} \| (-\Delta_0)^{-1} (\nabla_x \cdot b ) \|_{L^\infty_{t,x}} } \| \w_{\pm, \tilde \beta, 0}   \nabla_{x,p} F_{\pm, 0} \|_{L^\infty (\O \times \R^3)}
\\& \qquad \qquad \quad + \big( 1 + \frac{ \| \nabla_x^2 \phi_F \|_\infty + e B_3 + m_i g }{(m_i g)^2} \big) \| e^{\tilde{\beta} |p^0|} \nabla_{x_\parallel,p} G_i \|_{L^\infty (\gamma_-)} + \frac{1}{\sqrt{ m_i g}} \Big).
\label{express:D}
\end{align}
\end{lemma}

\begin{proof}

We abuse the notation as in \eqref{abuse} in the proof.
First, we prove \eqref{est:D.b}.
From \eqref{VP_h} and the regularity of $h (x,p)$ in \eqref{Uest:wh}, we deduce
\be \label{eq:int_v*h_x=0}
\int_{\R^3} v_{\pm} \cdot \nabla_x h_{\pm} (x,p) \dd p = 0.
\ee
Together with \eqref{def:flux} and $F (t,x,p) = h (x,p) + f (t,x,p)$, we have
\Be \label{est:D.b1}
\begin{split}
|\nabla_x \cdot b (t,x)| 
& = \big| \int_{\R^3} (e_+ v_+ \cdot \nabla_x f_+ (t,x,p) + e_{-} v_{-} \cdot \nabla_x f_{-} (t,x,p) ) \dd p \big|
\\& = \big| \int_{\R^3} (e_+ v_+ \cdot \nabla_x (F_+ - h_+ ) + e_{-} v_{-} \cdot \nabla_x ( F_{-} - h_- ) ) \dd p \big|
\\& = \big| \int_{\R^3} (e_+ v_+ \cdot \nabla_x F_+ + e_{-} v_{-} \cdot \nabla_x F_{-} ) \dd p \big| 
\\& \leq e_+ \int_{\R^3} | v_+ \cdot \nabla_x F_+ (t,x,p) | \dd p + e_- \int_{\R^3} | v_- \cdot \nabla_x F_- (t,x,p) | \dd p.
\end{split}
\Ee 
Recall \eqref{est:F_x1}  and \eqref{est:F_x2} in the proof of \eqref{est:F_x:dyn}, we bound $\nabla_x F (t,x,p)$ by
\begin{align}
& | \nabla_{x_i} F (t,x,p) | 
\notag \\
& \leq \mathbf{1}_{t \leq \tB (t,x,p)}
\frac{| \nabla_{x_i} \X  (0;t,x,p)| +| \nabla_{x_i} \P  (0;t,x,p)|  }{\w_{\tilde \beta} (0, \Z (0;t,x,p) )} \| \w_{\tilde \beta, 0}   \nabla_{x,p} F_0  \|_{L^\infty (\O \times \R^3)} \label{est:F_x1_b} \\
& \ \ \ \ +	\mathbf{1}_{t > \tB (t,x,p)} \frac{ | \nabla_{x_i} \xB (t,x,p) | + |\nabla_{x_i} \pB  (t,x,p)|}{ e^{\tilde{\beta} |\pB^0 (t,x,p)| } } \| e^{\tilde{\beta} |p^0|} \nabla_{x_\parallel,p} G \|_{L^\infty (\gamma_-)}.
\label{est:F_x2_b}
\end{align} 

Following from \eqref{est1:F_x1}, together with \eqref{est:D2xphi_F}, we get
\Be \label{est1:F_x1_b}
\eqref{est:F_x1_b}
\leq 2 e^{ \frac{m g}{24} \tilde{\beta} }
e^{2 \tilde{\beta} \frac{e}{c} \| (-\Delta_0)^{-1} (\nabla_x \cdot b ) \|_{L^\infty_{t,x}} } e^{ - \frac{\tilde{\beta}}{4} \sqrt{(m c)^2 + |p|^2} } e^{ - \frac{m g}{8 c} \tilde{\beta} x_3} \| \w_{\tilde \beta, 0} \nabla_{x,p} F_0  \|_{L^\infty (\O \times \R^3)}.
\Ee
Although we use the assumption \eqref{est1:D3tphi_F} in  \eqref{est1:F_x1}, we remark that the condition \eqref{est:D2xphi_F} is sufficient for the above estimate.

From \eqref{est:xb_x:dyn}, together with \eqref{est1:xb_x_a:dyn} and \eqref{est1:xb_x:dyn}, we have
\Be \label{est1:xb_x:dyn_b}
\begin{split}
& \frac{ |\p_{x_i} \xB (t,x,p)| }{ e^{\tilde{\beta} |\pB^0 (t,x,p)|} }
\\& \leq \frac{ | \vB (t,x,p) | }{ | \vBn (t,x,p) | } \delta_{i3}  e^{- \tilde{\beta} |\pB^0 (t,x,p)|} 
+ \big( 1 + \frac{\| \nabla_x^2 \phi_F \|_\infty + e B_3 + mg}{(m g)^2} \big) e^{- \frac{7 \tilde{\beta}}{12} |\pB^0 (t,x,p)|}.
\end{split}
\Ee

Now we show the upper bound on $\frac{|v_3|}{|\vBn (t,x,p)|}$, which is crucial in the estimate on \eqref{est:D.b1}. Here we split the proof into two cases.

\smallskip

\textbf{\underline{Case 1:} $v_3 \geq 0$.}
From \eqref{eq:sign_V3s^*_dy} and $\P_{3} (s;t,x, p)$ is strictly decreasing along characteristics, we obtain that
\be \notag
0 \leq v_3 \leq \vBn (t,x,p).
\ee
This shows
\be \label{est1:vb/vbn}
\frac{|v_3|}{|\vBn (t,x,p)|} \leq 1.
\ee

\smallskip

\textbf{\underline{Case 2:} $v_3 < 0$.}
Using \eqref{est:x3pB3} in Proposition \ref{lem2:tB}, we have
\Be \notag
0 \leq x_3 \leq \X_3 (s;t,x,p) \leq \frac{2}{m^2 g} (\pBn (t,x, p))^2.
\Ee
Together with $|v| = \big| \frac{p}{\sqrt{ m^2 + |p|^2/ c^2}} \big| \leq c$,, we derive that
\be \label{est2:vb/vbn}
\begin{split}
\frac{|v_3|}{|\vBn (t,x,p)|}
=  \frac{ |v_3| \sqrt{ {m_\pm}^2 + |\pB|^2/ c^2} }{|\pBn (t,x,p)|}
\leq \frac{ \sqrt{ 2 (m_\pm c)^2 + 2 |\pB|^2} }{\sqrt{m^2 g x_3}}.
\end{split}
\ee
Combining \eqref{est1:vb/vbn} with \eqref{est2:vb/vbn}, we deduce 
\be \label{est:vb/vbn}
\frac{|v_3|}{|\vBn (t,x,p)|}
\leq 1 + \frac{ \sqrt{ 2 (m_\pm c)^2 + 2 |\pB|^2} }{\sqrt{m^2 g x_3}}.
\ee

Analogous to \eqref{est1:xb_x:dyn_b}, using \eqref{est:vb_x:dyn} and \eqref{est1:vb_x:dyn}, we derive that
\Be \label{est1:vb_x:dyn_b}
\begin{split}
& \frac{ |\p_{x_i} \pb (x,p)| }{e^{\tilde{\beta} |\pb^0 (x,p)|}}
\\& \lesssim \frac{ |e B_3 + m g| }{|\vBn  (t,x,p)|} \delta_{i3} e^{- \tilde{\beta} |\pB^0 (t,x,p)|} 
+ \big( 1 + \frac{\| \nabla_x^2 \phi_F \|_\infty + e B_3 + mg}{(m g)^2} \big) e^{- \frac{7 \tilde{\beta}}{12} |\pb^0 (x,p)|}.
\end{split}
\Ee
Inputting \eqref{est1:xb_x:dyn_b} and \eqref{est1:vb_x:dyn_b} into \eqref{est:F_x2_b}, together with \eqref{est1:F_x1_b}, we bound 
\be \label{est:F_x_b}
\begin{split}
& | \nabla_{x_i} F (t,x,p) | 
\\& \leq 2 e^{ \frac{m g}{24} \tilde{\beta} }
e^{2 \tilde{\beta} \frac{e}{c} \| (-\Delta_0)^{-1} (\nabla_x \cdot b ) \|_{L^\infty_{t,x}} } e^{ - \frac{\tilde{\beta}}{4} \sqrt{(m c)^2 + |p|^2} } e^{ - \frac{m g}{8 c} \tilde{\beta} x_3} \| \w_{\tilde \beta, 0} \nabla_{x,p} F_0  \|_{L^\infty (\O \times \R^3)}
\\& \ \ \ \ + \Big\{ \frac{ | \vB (t,x,p) | }{ | \vBn (t,x,p) | } \delta_{i3}  e^{- \tilde{\beta} |\pB^0 (t,x,p)|} + \frac{ |e B_3 + m g| }{|\vBn  (t,x,p)|} \delta_{i3} e^{- \tilde{\beta} |\pB^0 (t,x,p)|} 
\\& \qquad \ \ \ \ + 2 \big( 1 + \frac{\| \nabla_x^2 \phi_F \|_\infty + e B_3 + mg}{(m g)^2} \big) e^{- \frac{7 \tilde{\beta}}{12} |\pB^0 (t,x,p)|} \Big\} \times \| e^{\tilde{\beta} |p^0|} \nabla_{x_\parallel,p} G \|_{L^\infty (\gamma_-)}.
\end{split} 
\ee
Using \eqref{est:x3ppB} in Proposition \ref{lem3:tB}, together with \eqref{est:vb/vbn}, we compute
\Be \label{est1:D.b3}
\begin{split}
& \int_{\R^3} |v \cdot \nabla_x F (t,x,p) | \dd p  
\\& \lesssim \Big( 2 e^{ \frac{m g}{24} \tilde{\beta} }
e^{2 \tilde{\beta} \frac{e}{c} \| (-\Delta_0)^{-1} (\nabla_x \cdot b ) \|_{L^\infty_{t,x}} } \| \w_{\tilde \beta, 0}   \nabla_{x,p} F_0  \|_{L^\infty (\O \times \R^3)}
\int_{\R^3} |v| e^{ - \frac{\tilde{\beta}}{4} \sqrt{(m c)^2 + |p|^2} } \dd p
\\& \ \ \ \ + \big( 1 + \frac{ \| \nabla_x^2 \phi_F \|_\infty + e B_3 + mg}{(m g)^2} \big) \| e^{\tilde{\beta} |p^0|} \nabla_{x_\parallel,p} G \|_{L^\infty (\gamma_-)} \int_{\R^3} |v| e^{ - \frac{\tilde{\beta}}{4} \sqrt{(m c)^2 + |p|^2} } \dd p
\\& \ \ \ \ + \int_{\R^3} |v| \big( 1 + \frac{ \sqrt{ 2 (m_\pm c)^2 + 2 |\pB|^2} }{\sqrt{m^2 g x_3}} \big) e^{ - \frac{\tilde{\beta}}{4} \sqrt{(m c)^2 + |p|^2} } \dd p \Big) \times e^{ - \frac{m g}{8 c} \tilde{\beta} x_3}.
\end{split}
\Ee 
Thus, we derive that
\be \notag
\begin{split}
& \int_{\R^3} |v \cdot \nabla_x F (t,x,p)| \dd p  
\\& \lesssim \frac{e^{ - \frac{m g}{8 c} \tilde{\beta} x_3}}{{\tilde{\beta}}^3} \Big( \big[ 2 e^{ \frac{m g}{24} \tilde{\beta} } 
e^{2 \tilde{\beta} \frac{e}{c} \| (-\Delta_0)^{-1} (\nabla_x \cdot b ) \|_{L^\infty_{t,x}} } \| \w_{\tilde \beta, 0}   \nabla_{x,p} F_0  \|_{L^\infty (\O \times \R^3)}
\\& \qquad \qquad \qquad + \big( 1 + \frac{ \| \nabla_x^2 \phi_F \|_\infty + e B_3 + mg}{(m g)^2} \big) \| e^{\tilde{\beta} |p^0|} \nabla_{x_\parallel,p} G \|_{L^\infty (\gamma_-)} \big] + \frac{1}{ \sqrt{ m^2 g  x_3 } } \Big).
\end{split}
\ee
and thus we conclude \eqref{est:D.b}. 

\smallskip

Second, we prove \eqref{est:D3tphi_F}.
Since $\phi_F = \Phi (x) + \Psi (t, x)$, from \eqref{identity:Psi_t} and Lemma \ref{lem:Ddb}, we have 
\Be \label{D3tphi}
\begin{split}
\p_{x_3} \p_t \phi_{F} (t,x)
& = \p_{x_3} \p_t \Psi (t,x)
\\& = - \p_{x_3} (-\Delta_0)^{-1} (\nabla_x \cdot b ) (t,x)
= \int_{\O} \nabla\cdot b (y) \p_{x_3} \mathfrak{G} \dd y .
\end{split} 
\Ee
Using \eqref{est1:G_x}-\eqref{est:Green_x}, we compute 
\Be \label{est:D.b4}
\begin{split}
| \p_{x_3} \p_t \phi_{F} (t,x)| 
& \leq \int_{\O} \big| \nabla\cdot b (y) \big|  \big| \p_{x_3} \mathfrak{G} \big| \dd y
\\& \leq \int_{0}^{\infty} \int_{\R^2} 
\min \big( \frac{y_3}{|x-y|^3}, \frac{1}{|x-y|^2} \big) \times | \nabla \cdot b (y_\parallel, y_3) | \dd y_\parallel \dd y_3
\\& \leq \int_{0}^{1} \big( \frac{ C_1 }{ \sqrt{ y_3 } } + C_2 \big)
\int_{\R^2} \min \big( \frac{y_3}{|x-y|^3}, \frac{1}{|x-y|^2} \big) \dd y_\parallel \dd y_3
\\& \ \ \ \ + \int_{1}^{\infty} C_2  e^{ - \frac{\hat{m} g}{8 c} \tilde{\beta} y_3} \int_{\R^2} \min \big( \frac{y_3}{|x-y|^3}, \frac{1}{|x-y|^2} \big) \dd y_\parallel \dd y_3.
\end{split} 
\Ee
Following \eqref{est1:nabla_phi} and \eqref{est:I_2}, we derive
\Be \label{est1:D.b4}
\begin{split}
\eqref{est:D.b4}
& \leq \int_{0}^{1} \big( \frac{ C_1 }{ \sqrt{ y_3 } } + C_2 \big)
\big( \ln (1+ \frac{1}{2 |x_3 - y_3|^2} ) + y_3 \big) \dd y_3
\\& \ \ \ \ + \int_{1}^{\infty} C_2  e^{ - \frac{\hat{m} g}{8 c} \tilde{\beta} y_3} \big( \ln (1+ \frac{1}{2 |x_3 - y_3|^2} ) + y_3 \big) \dd y_3
\\& \leq \underbrace{\int_{0}^{1} \frac{ C_1 }{ \sqrt{ y_3 } } \big( \ln (1+ \frac{1}{2 |x_3 - y_3|^2} ) + y_3 \big) \dd y_3 }_{\eqref{est1:D.b4}_*}
+ C_2 ( 1 + \frac{8 c}{\hat{m} g \tilde{\beta}} ).
\end{split} 
\Ee

Using the fact that $x^{- \frac{1}{8}} \ln (1 + x) \lesssim 1$ for any $x \geq 2$, we bound 
\be \label{est2:D.b4}
\begin{split}
\eqref{est1:D.b4}_*
& \leq \int_{0}^{1} \frac{ C_1 }{ \sqrt{ y_3 } }  \ln (1+ \frac{1}{2 |x_3 - y_3|^2} ) \dd y_3 + C_1
\\& \leq C_1 + \mathbf{1}_{|x_3| > 2} C_1 + 
\mathbf{1}_{|x_3| \leq 2}
\int_{0}^{1} \frac{ C_1 }{ \sqrt{ y_3 } } \frac{1}{ |x_3 - y_3|^{\frac{1}{4}} } \dd y_3
\lesssim C_1,
\end{split} 
\Ee
where the last inequality follows from
\be \notag
\int_{0}^{1} \frac{ 1 }{ \sqrt{ y_3 } |x_3 - y_3|^{\frac{1}{4}} } \dd y_3 
\leq \int_{0}^{1} |y_3|^{- \frac{3}{4}} \dd y_3
+ \int_{0}^{1} |x_3 - y_3|^{- \frac{3}{4}} \dd y_3
\lesssim 1
\ee
Finally, inputting \eqref{est2:D.b4} into \eqref{est1:D.b4}, we conclude \eqref{est:D3tphi_F}.
\end{proof}

Collecting the results from Proposition \ref{RE:dyn} and Lemma \ref{lem:D3tphi_F}, here we illustrate the regularity estimate for the dynamical problem.
Similar to the steady case, this is important in the uniqueness theorem proved in Section \ref{sec:US}.

\begin{theorem}[Regularity Estimate]
\label{theo:RD}

Suppose $(F_{\pm}, \phi_F)$ solves \eqref{VP_F}, \eqref{Poisson_F}, \eqref{VP_0},  \eqref{Dbc:F} and \eqref{bdry:F} under the compatibility condition \eqref{CC_F0=G}.
Consider 
\be \notag
F_{\pm} (t,x,p) = h_{\pm} (x,p) + f_{\pm} (t,x,p),
\ee 
where $(h_{\pm}, \rho, \Phi)$ solves \eqref{VP_h}-\eqref{eqtn:Dphi} and $(f_{\pm}, \varrho, \Psi)$ solves \eqref{eqtn:f}-\eqref{Poisson_f} respectively.
Assume all conditions in Theorem \ref{theo:RS} and Theorem \ref{theo:AS}.
Further, we assume that \eqref{est:beta_b_dy}, \eqref{est1:D3tphi_F} and
\Be \label{condition:ML}
M \leq \beta e^{ - \frac{m_{\pm} g}{24} \beta}
\ \text{ and } \ 
L \leq \tilde{\beta} e^{ - \frac{m_{\pm} g}{24} \tilde{\beta} }.
\Ee
hold for some $g, \beta, \tilde \beta > 0$.
Consider $\alpha_{\pm, F} (t,x,p)$ defined in \eqref{alpha_F}, then for any $(t,x,p) \in [0, \infty) \times  \bar\O \times \R^3$,
\Be \label{est_final:F_v:dyn}
\begin{split} 
& \| e^{ \frac{\tilde{\beta}}{4} \sqrt{(m_{\pm} c)^2 + |p|^2} } e^{ \frac{m_{\pm} g}{8 c} \tilde{\beta} x_3} \nabla_p F_{\pm} (t,x,p) \|_{L^\infty(\O \times \R^3)} 
\\& \lesssim 2 e^{ 2 + \frac{m_{\pm} g}{24} \tilde{\beta} }
\| \w_{\pm, \tilde \beta, 0}   \nabla_{x,p} F_{\pm, 0} \|_{L^\infty (\O \times \R^3)}
+ \frac{1}{24} \tilde{\beta}
\| e^{\tilde{\beta} |p^0_{\pm}|} \nabla_{x_\parallel,p} G_{\pm} \|_{L^\infty (\gamma_-)},
\end{split}
\Ee 
and 
\Be \label{est_final:F_x:dyn}
\begin{split} 
& e^{ \frac{\tilde{\beta}}{4} \sqrt{(m_{\pm} c)^2 + |p|^2} } e^{ \frac{m_{\pm} g}{8 c} \tilde{\beta} x_3} \big| \nabla_x F (t,x,p) \big|
\\& \lesssim 2 e^{ 2 + \frac{m_{\pm} g}{24} \tilde{\beta}}
\| \w_{\pm, \tilde \beta, 0} \nabla_{x,p} F_{\pm, 0} \|_{L^\infty (\O \times \R^3)}
+ ( \frac{ \mathbf{1}_{|x_3| \leq 1} }{\alpha_{\pm, F} (t,x,p)} + \frac{1}{24} \tilde{\beta} ) \| e^{\tilde{\beta} |p^0_{\pm}|} \nabla_{x_\parallel,p} G_{\pm} \|_{L^\infty (\gamma_-)}.
\end{split}
\Ee 
Moreover, 
\be \label{est_final:phi_F}
\| \nabla_x^2 \phi_F (t) \|_{L^\infty (\bar \O)} + 
\| \p_t  \nabla_x  \phi_F (t) \|_{L^\infty (\bar \O)} 
\lesssim 1.
\ee
\end{theorem}

\begin{proof}

We abuse the notation as in \eqref{abuse} in the proof. 
Using Proposition \ref{RE:dyn}, together with the conditions \eqref{est:beta_b_dy} and \eqref{est1:D3tphi_F}, we obtain \eqref{est_final:F_v:dyn} and \eqref{est_final:F_x:dyn}.

\smallskip

For \eqref{est_final:phi_F}, recall from Lemma \ref{lem:D3tphi_F}, then for any $(t,x) \in [0, \infty) \times \bar\O$,
\be \label{est_final:D3tphi_F}
| \p_{x_3} \p_t \phi_{F}(t,x)| 
\lesssim C_1 + C_2 ( 1 + \frac{8 c}{\hat{m} g \tilde{\beta}} ),  
\ee
where $\hat{m} = \min \{ m_+, m_- \}$ and
\begin{align}
C_{1} & = \frac{ 1 }{{\tilde{\beta}}^3} ( \frac{e_+}{\sqrt{ m_+ g  }} + \frac{e_-}{\sqrt{ m_- g }} ),
\notag \\
C_{2} & = \sum\limits_{i = \pm} \frac{e_i }{{\tilde{\beta}}^3} \Big( 2 e^{ \frac{m_i g}{24} \tilde{\beta} }  
e^{2 \tilde{\beta} \frac{e_i}{c} \| (-\Delta_0)^{-1} (\nabla_x \cdot b ) \|_{L^\infty_{t,x}} } \| \w_{\pm, \tilde \beta, 0} \nabla_{x, p} F_{\pm, 0} \|_{L^\infty (\O \times \R^3)}
\notag \\
& \qquad \qquad \quad + \big( 1 + \frac{ \| \nabla_x^2 \phi_F \|_\infty + e B_3 + m_i g }{(m_i g)^2} \big) \| e^{\tilde{\beta} |p^0|} \nabla_{x_\parallel,p} G_i \|_{L^\infty (\gamma_-)} + \frac{1}{\sqrt{ m g}} \Big).
\notag
\end{align}
Thus, it remains to estimate $\| \nabla_x^2 \phi_F (t) \|_{L^\infty (\bar \O)}$. We will use a similar idea in the proof of \eqref{est:phi_C2} in Proposition \ref{prop:Reg}. 
	
Recall from \eqref{Poisson_F}, \eqref{eqtn:Dphi} and \eqref{Poisson_f}, we have 
\Be \notag
- \Delta_x \phi_F 
= \rho + \varrho (t ).
\Ee
Using \eqref{Uest:rho} and \eqref{decay:varrho}, we bound 
\Be \label{est:rho:infty} 
\begin{split}
& \| \rho (\cdot ) + \varrho (t, \cdot ) \|_{L^\infty (\O)} 
\\& \leq \frac{1}{\beta} \big( e_+ \| w_{+, \beta} G_+ \|_{L^\infty(\gamma_-)} + e_{-} \| w_{-, \beta} G_- \|_{L^\infty(\gamma_-)} \big)
\\& \ \ \ \ + \frac{32}{\beta^2} \big( e_+ e^{2 + \frac{m_+ g}{24} \beta} \| \w_{+, \beta, 0 } f_{+, 0 } \|_{L^\infty_{x,p}} + e_- e^{2 + \frac{m_- g}{24} \beta} \| \w_{-, \beta, 0 } f_{-, 0 } \|_{L^\infty_{x,p}} \big).
\end{split}
\Ee
Applying \eqref{est:nabla^2phi} in Lemma \ref{lem:rho_to_phi}, together with \eqref{Uest:rho} and \eqref{decay:varrho}, we derive that 
\Be \label{bound:D2phi}
\begin{split}
\| \nabla_x^2 \phi_F ( t ) \|_{L^\infty(\bar \O)}
& \lesssim_{\delta} \|\rho +  \varrho (t ) \|_{L^\infty(\O)} + \| \rho + \varrho (t ) \|_{C^{0,\delta}(\O)} 
\\& \qquad + \frac{1}{\beta} \Big[ \frac{1}{\beta} \big( e_+ \| w_{+, \beta} G_+ \|_{L^\infty(\gamma_-)} + e_{-} \| w_{-, \beta} G_- \|_{L^\infty(\gamma_-)} \big)
\\& \qquad \qquad + \frac{32}{\beta^2} \big( e_+ e^{2 + \frac{m_+ g}{24} \beta} \| \w_{+, \beta, 0 } f_{+, 0 } \|_{L^\infty_{x,p}} + e_- e^{2 + \frac{m_- g}{24} \beta} \| \w_{-, \beta, 0 } f_{-, 0 } \|_{L^\infty_{x,p}} \big) \Big].
\end{split}
\Ee

It remains to control $\| \rho + \varrho (t ) \|_{C^{0,\delta}(\O)}$. We use a similar idea in the proof of \eqref{est:phi_C2} in Proposition \ref{prop:Reg}.
Given any $x \in \O$, consider $0<h<1$ such that $x \pm h \mathbf{e}_i \in \O$ with $1 \leq i \leq 3$.
Then we compute
\Be \label{rho_varrho_DQ}
\begin{split}
& \frac{ \big| \big( \rho(x+ h e_i) + \varrho (t,x+ h e_i) \big) - \big( \rho(x) + \varrho (t,x) \big) \big| }{h^\delta}  
\\& \leq \sum_{i = \pm} \frac{ e_i }{h^\delta} \int^h_0 \big| \nabla_{x_i} \big( \rho (x+ \tau e_i) +  \varrho (x+ \tau e_i) \big) \big| \dd \tau 
\\& \leq \sum_{i = \pm} \frac{ e_i }{h^\delta} \int^h_0 \Big| \int_{\R^3} \nabla_{x_i} F (t, x+ \tau e_i, p) \dd p \Big| \dd \tau.
\end{split}
\Ee
Using \eqref{est_final:F_x:dyn} and \eqref{alpha_F}, we derive that
\Be \label{est:rho:holder}
\begin{split}
& \Big| \int_{\R^3} \nabla_{x_i} F (t, x, p) \dd p \Big|
\\& \lesssim  \int_{\R^3} \Big( 2 e^{ 2 + \frac{m g}{24} \tilde{\beta} } \| \w_{\tilde \beta, 0} \nabla_{x,p} F_0  \|_{L^\infty (\O \times \R^3)}
+ \frac{1}{24} \tilde{\beta} \| e^{\tilde{\beta} |p^0|} \nabla_{x_\parallel,p} G \|_{L^\infty (\gamma_-)} \Big)
e^{ - \frac{\tilde{\beta}}{4} \sqrt{(m c)^2 + |p|^2} } \dd p 
\\& \ \ \ \ + \| e^{\tilde{\beta} |p^0|} \nabla_{x_\parallel,p} G \|_{L^\infty (\gamma_-)} \times \int_{\R^3} \mathbf{1}_{|x_3| \leq 1} \frac{ \delta_{i3} }{\alpha_F (t,x,p)} e^{ - \frac{\tilde{\beta}}{4} \sqrt{(m c)^2 + |p|^2} } \dd p 
\\& \lesssim \frac{1}{{\tilde{\beta}}^3} \big( 2 e^{ 2 + \frac{m g}{24} \tilde{\beta} }
\| \w_{\tilde \beta, 0} \nabla_{x,p} F_0  \|_{L^\infty (\O \times \R^3)}
+ \frac{1}{24} \tilde{\beta} \| e^{\tilde{\beta} |p^0|} \nabla_{x_\parallel,p} G \|_{L^\infty (\gamma_-)} \big)
\\& \ \ \ \ + \frac{1}{{\tilde{\beta}}^3} \| e^{\tilde{\beta} |p^0|} \nabla_{x_\parallel,p} G \|_{L^\infty (\gamma_-)} \times \mathbf{1}_{|x_3| \leq 1} \frac{ \delta_{i3} }{ \sqrt{ m g  x_3 } }.
\end{split}
\Ee 
By picking $0 < \delta < 1/2$, we follow \eqref{est:rho_DQ_12}-\eqref{est2:rho_DQ_3}, then for $0< h < 1 $ and $1 \leq i \leq 3$, 
\be \label{est:rho:holder2}
\frac{1}{h^\delta} \int^h_0 \big( 1 + \mathbf{1}_{|x_3| \leq 1} \frac{\delta_{i3}}{\sqrt{ m g (x \pm \tau \mathbf{e}_i)_3 }} \big) \dd \tau \lesssim_{\delta} 1.
\ee
Using \eqref{condition:ML} and \eqref{est_final:D3tphi_F}, together with \eqref{est:rho:infty}-\eqref{est:rho:holder2}, we conclude \eqref{est_final:phi_F}.
\end{proof}

\subsection{Stability and Uniqueness} \label{Sec:SU}

In this section, we first prove Theorem \ref{theo:UD}: a stability theorem. The idea of the proof is similar to the proof of Proposition \ref{prop:cauchy}. Then we prove the uniqueness in Theorem \ref{theo:UA} which follows from Theorem \ref{theo:UD}.

\begin{theorem}[Stability Theorem]
\label{theo:UD}

Suppose $(F_{1, \pm}, \nabla_x \phi_{F_1})$ and $(F_{2, \pm}, \nabla_x \phi_{F_2})$ solve \eqref{VP_F}, \eqref{Poisson_F}, \eqref{Dbc:F} and \eqref{bdry:F} under the compatibility condition \eqref{CC_F0=G}.
Consider 
\be \notag
F_{1, \pm} (t,x,p) = h_{1, \pm} (x,p) + f_{1, \pm} (t,x,p),
\ee 
where $(h_{1, \pm}, \rho_1, \Phi_1)$ solves \eqref{VP_h}-\eqref{eqtn:Dphi} and $(f_{1, \pm}, \varrho_1, \Psi_1)$ solves \eqref{eqtn:f}-\eqref{Poisson_f} respectively.
We define 
\be \label{w1_dy}
\w_{1, \pm} (t,x,p) = e^{ \frac{1}{12} \tilde{\beta} \big( \sqrt{(m_{\pm} c)^2 + |p|^2} + \frac{1}{c} ( e_{\pm} \phi_{F_1} (x) + m_{\pm} g x_3 ) \big) },
\ee 
and
\be \label{b1_dy}
b_1 (t,x) = \int_{\R^3} \big(v_+ e_+ ( F_{1,+} - h_{1,+} ) + v_{-} e_{-} ( F_{1, -} - h_{1, -}) \big) \dd p.
\ee
Assume both $(F_{1, \pm}, \nabla_x \phi_{F_1})$ and $(F_{2, \pm}, \nabla_x \phi_{F_2})$ satisfy all assumptions in Theorem \ref{theo:RD}, then
\Be \label{est:F1-2}
\begin{split}
& \sum_{i = \pm} \| \w_{1,i} (F_{1,i} - F_{2,i} ) (t) \|_{L^{\infty} (\O \times \R^3)} 
\\& \leq \exp \Big\{ { t \Big( \frac{1}{12} \tilde{\beta} \sup_{s \in [0,t]} \| (-\Delta_0)^{-1} (\nabla_x \cdot b_1(s)) \|_{L^\infty(\O)}
+ \sum_{j = \pm} \sup_{s \in [0,t]} \| \w_{1, j} \nabla_p F_{2, j} (s) \|_{L^{\infty} (\O \times \R^3)} \Big) } \Big\}
\\& \ \ \ \ \times \sum_{i = \pm} \| \w_{1,i} (F_{1,i} - F_{2,i} ) (0) \|_{L^{\infty} (\O \times \R^3)}.
\end{split}
\Ee	
\end{theorem}

\begin{proof}

We abuse the notation as in \eqref{abuse} in the proof.
From \eqref{ODE_F} and \eqref{tb}, we consider the characteristic $\Z_1 = ( \X_1, \P_1 )$ for $(F_1, \phi_{F_1})$ and the corresponding backward exit time $\tBf (t,x,p)$. 

Now consider $\w_{1, \pm} (t,x,p)$ in \eqref{w1_dy}. 
From \eqref{VP_F}, \eqref{Poisson_F}, \eqref{Dbc:F} and \eqref{bdry:F}, together with \eqref{dEC}, we have
\Be \label{VP_diff_12_dy}
\begin{split}
& \p_t ( \w_{1, \pm} (F_{1, \pm} - F_{2, \pm} )) + v_\pm \cdot \nabla_x ( \w_{1, \pm} (F_{1, \pm} - F_{1, \pm} ))
\\& \ \ \ \ + \Big( e_{\pm} \big( \frac{v_\pm}{c} \times B - \nabla_x \phi_{F_1} \big) - \nabla_x ( m_\pm g x_3) \Big) \cdot \nabla_p ( \w_{1, \pm} (F_{1, \pm} - F_{2, \pm} ))
\\& = - \frac{e_{\pm}}{c} (-\Delta_0)^{-1}  (\nabla_x \cdot b_1 ) \frac{1}{12} \tilde{\beta} \w_{1, \pm} (F_{1, \pm} - F_{2, \pm} )
\\& \ \ \ \ + e_{\pm} \nabla_x ( \phi_{F_1} - \phi_{F_2})
\cdot \nabla_p ( \w_{1, \pm} F_{2, \pm} )
 \ \ \text{in} \ \R_+ \times  \O \times \R^3,
\end{split}
\ee
where $b_1 (t,x)$ is defined in \eqref{b1_dy}. Moreover,
\be \label{VP_diff_bdy_12_dy}
F_1 (t,x,p) - F_2 (t,x,p) =  0 \ \ \text{in} \ \gamma_-,
\Ee
and
\be \label{VP_diff_initial_12_dy}
F_{1, \pm} (0, x, p) - F_{2, \pm} (0, x, p) = F_{1, \pm, 0} (x,p) - F_{2, \pm, 0} (x,p) \ \ \ \ \text{in} \ \O \times \R^3.
\Ee

\smallskip

\textbf{\underline{Case 1:} $\tBf (t,x,p) < t$.}
From \eqref{VP_diff_12_dy} and \eqref{VP_diff_bdy_12_dy}, we get
\Be \label{energy:wF_12_tb<t}
\begin{split}
& \mathbf{1}_{\tBf (t,x,p) < t} \times \w_{1} (F_1- F_2) (t,x,p) 
\\& = \int^t_{t - \tBf (t,x,p)} - \frac{e}{c} (-\Delta_0)^{-1}  (\nabla_x \cdot b_1 ) \frac{1}{12} \tilde{\beta} \w_{1} (F_{1} - F_{2} ) \dd s 
\\& \ \ \ \ + \int^t_{t - \tBf (t,x,p)} e \nabla_x ( \phi_{F_1} - \phi_{F_2})
\cdot \nabla_p ( \w_{1} F_{2, \pm} ) \dd s 
\\& \leq \frac{e}{c} \frac{1}{12} \tilde{\beta} \sup_{s \in [0,t]} \| (-\Delta_0)^{-1} (\nabla_x \cdot b_1(s)) \|_{L^\infty(\O)}
\int^t_{t - \tBf (t,x,p)} 	\| \w_{1} (F_1  - F_2 )(s) \|_{L^{\infty} (\O \times \R^3)} \dd s 
\\& \ \ \ \ + \sup_{s \in [0,t]} \| \w_{1} \nabla_p F_2 (s) \|_{L^{\infty} (\O \times \R^3)} 
\int^t_{t - \tBf (t,x,p)} 
\underbrace{ \| \nabla_x (\phi_{F_1} - \phi_{F_2}) (s) \|_{L^{\infty}_x(\O)} }_{\eqref{energy:wF_12_tb<t}_*}
\dd s.
\end{split}
\Ee
Now we estimate $\eqref{energy:wF_12_tb<t}_*$. From \eqref{Poisson_F} and Lemma \ref{lemma:G}, we have
\Be \label{est:Dphi_12_tb<t}
\begin{split}
& \big| \nabla_x (\phi_{F_1} - \phi_{F_2}) (s, x) \big|
\\& \leq \Big| \int_{\O}  \nabla_x \mathfrak{G} 
\sum_{i = \pm} \int_{\R^3} \frac{e_i}{\w_{1,i} (s,y,p)} \big( \w_{1,i} (F_{1,i} - F_{2,i} ) \big) (s,y,p) \dd p \dd y \Big|
\\& \leq \sum_{i = \pm} \| \w_{1,i} (F_{1,i} - F_{2,i} ) (s) \|_{L^{\infty} (\O \times \R^3)} 
\times \Big\| \int_{\O}  \nabla_x \mathfrak{G} 
\int_{\R^3} \frac{e_i}{\w_{1,i} (s,y,p)}  \dd p \dd y \Big\|_{L^{\infty} _x(\O)}
\\& = \sum_{i = \pm} \Big( \| \w_{1,i} (F_{1,i} - F_{2,i} ) (s) \|_{L^{\infty} (\O \times \R^3)} 
\times \Big\| \int_{\R^3} \int_{\O}  \nabla_x \mathfrak{G} 
 \frac{e_i}{\w_{1,i} (s,y,p)}  \dd y \dd p \Big\|_{L^{\infty} _x(\O)} \Big).
\end{split}
\Ee
Using \eqref{w^F}, together with \eqref{Bootstrap}, we have
\Be \label{w^F_1}
\begin{split}
\w_{1, \pm } (t,x,p) 
& = e^{ \frac{1}{12} \tilde{\beta} \left( \sqrt{(m_{\pm} c)^2 + |p|^2} + \frac{1}{c} \big( e_{\pm} \big( \Phi_1 (x) + \Psi_1 (t,x) \big) + m_{\pm} g x_3 \big) \right) } 
\\& \geq e^{ \frac{1}{12} \tilde{\beta} \big( \sqrt{(m_{\pm} c)^2 + |p|^2} + \frac{m_{\pm} g}{2 c} x_3 \big) }.
\end{split}
\Ee
Inputting \eqref{w^F_1} into \eqref{est:Dphi_12_tb<t}, we derive
\begin{align}
\big| \nabla_x (\phi_{F_1} - \phi_{F_2}) (s, x) \big|
& \leq \sum_{i = \pm} \Big( \| \w_{1,i} (F_{1,i} - F_{2,i} ) (s) \|_{L^{\infty} (\O \times \R^3)} 
\times \int_{\R^3} \frac{e_i}{e^{ \frac{1}{12} \tilde{\beta} \sqrt{(m_{\pm} c)^2 + |p|^2} }} \dd p 
\notag \\
& \qquad \qquad \times \Big\| \int_{\O}  \nabla_x \mathfrak{G} 
e^{ - \frac{1}{12} \tilde{\beta} \frac{m_{\pm} g}{2 c} y_3  } \dd y \Big\|_{L^{\infty} _x(\O)} \Big).
\label{est1:Dphi_12_tb<t}
\end{align}
Applying \eqref{est:nabla_phi} in Lemma \ref{lem:rho_to_phi} on \eqref{est1:Dphi_12_tb<t}, we deduce that
\be
\Big\| \int_{\O}  \nabla_x \mathfrak{G} 
e^{ - \frac{1}{12} \tilde{\beta} \frac{m g}{2 c} y_3  } \dd y \Big\|_{L^{\infty} _x(\O)}
\lesssim 1 + \frac{2 c}{m g}.
\ee
Together with $\int_{\R^3} e_i e^{ - \frac{1}{12} \tilde{\beta} \sqrt{(m_{\pm} c)^2 + |p|^2} \dd p} \dd p \lesssim e_i$, we can compute that
\be \label{est2:Dphi_12_tb<t}
\| \nabla_x (\phi_{F_1} - \phi_{F_2}) (s) \|_{L^{\infty}_x(\O)}
\leq \sum_{i = \pm} \Big( \| \w_{1,i } (F_{1,i} - F_{2,i} ) (s) \|_{L^{\infty} (\O \times \R^3)} 
\times e_i ( 1 + \frac{2 c}{m g } ) \Big).
\ee
Inputting \eqref{est2:Dphi_12_tb<t} into \eqref{energy:wF_12_tb<t}, we get
\Be \label{est3:Dphi_12_tb<t}
\begin{split}
& \| \mathbf{1}_{\tBf (t,x,p) < t} \times \w_{1} (F_1- F_2) (t) \|_{L^{\infty} (\O \times \R^3)}
\\& \lesssim_{e, m, g} \frac{1}{12} \tilde{\beta} \sup_{s \in [0,t]} \| (-\Delta_0)^{-1} (\nabla_x \cdot b_1(s)) \|_{L^\infty(\O)}
\int^t_{t - \tBf (t,x,p)} 	\| \w_{1} (F_1  - F_2 )(s) \|_{L^{\infty} (\O \times \R^3)} \dd s 
\\& \qquad + \sup_{s \in [0,t]} \| \w_{1} \nabla_p F_2 (s) \|_{L^{\infty} (\O \times \R^3)}
\times \int^t_{t - \tBf (t,x,p)} \sum_{i = \pm} \| \w_{1,i} (F_{1,i} - F_{2,i} ) (s) \|_{L^{\infty} (\O \times \R^3)} \dd s.
\end{split}
\Ee

\smallskip

\textbf{\underline{Case 2:} $t \leq \tBf (t,x,p)$.}
Similar to case 1, using \eqref{VP_diff_12_dy} and \eqref{VP_diff_initial_12_dy}, we get
\Be \label{energy:wF_12_tb>t}
\begin{split}
& \mathbf{1}_{t \leq \tBf (t,x,p)} \times \w_{1} (F_1- F_2) (t,x,p) - \w_{1} (F_{1, 0} - F_{2, 0} ) (\Z_1 (0; t,x,p))
\\& = \int^t_{0} - \frac{e}{c} (-\Delta_0)^{-1}  (\nabla_x \cdot b_1 ) \frac{1}{12} \tilde{\beta} \w_{1} (F_{1} - F_{2} ) \dd s 
+ \int^t_{0} e \nabla_x ( \phi_{F_1} - \phi_{F_2})
\cdot \nabla_p ( \w_{1} F_{2, \pm} ) \dd s 
\end{split}
\Ee
Following \eqref{est:Dphi_12_tb<t}-\eqref{est3:Dphi_12_tb<t}, we derive  
\Be \label{est:Dphi_12_tb>t}
\begin{split}
& \mathbf{1}_{t \leq \tBf (t,x,p)} \times \w_{1} (F_1- F_2) (t,x,p) - \w_{1} (F_{1, 0} - F_{2, 0} ) (\Z_1 (0; t,x,p))
\\& \lesssim_{e, m, g} \frac{1}{12} \tilde{\beta} \sup_{s \in [0,t]} \| (-\Delta_0)^{-1} (\nabla_x \cdot b_1(s)) \|_{L^\infty(\O)}
\int^t_{0} 	\| \w_{1} (F_1  - F_2 )(s) \|_{L^{\infty} (\O \times \R^3)} \dd s 
\\& \qquad \ \ + \sup_{s \in [0,t]} \| \w_{1} \nabla_p F_2 (s) \|_{L^{\infty} (\O \times \R^3)}
\\& \qquad \qquad \qquad \times \int^t_{0} \sum_{i = \pm} \| \w_{1,i} (F_{1,i} - F_{2,i} ) (s) \|_{L^{\infty} (\O \times \R^3)} \dd s.
\end{split}
\Ee
This shows that
\Be \label{est1:Dphi_12_tb>t}
\begin{split}
& \| \mathbf{1}_{t \leq \tBf (t,x,p)} \times \w_{1} (F_1- F_2) (t) \|_{L^{\infty} (\O \times \R^3)} - \| \w_{1} (F_1- F_2) (0) \|_{L^{\infty} (\O \times \R^3)}
\\& \lesssim_{e, m, g} \frac{1}{12} \tilde{\beta} \sup_{s \in [0,t]} \| (-\Delta_0)^{-1} (\nabla_x \cdot b_1(s)) \|_{L^\infty(\O)}
\int^t_{0} 	\| \w_{1} (F_1  - F_2 )(s) \|_{L^{\infty} (\O \times \R^3)} \dd s 
\\& \qquad \ \ + \sup_{s \in [0,t]} \| \w_{1} \nabla_p F_2 (s) \|_{L^{\infty} (\O \times \R^3)}
\\& \qquad \qquad \qquad \times \int^t_{0} \sum_{i = \pm} \| \w_{1,i} (F_{1,i} - F_{2,i} ) (s) \|_{L^{\infty} (\O \times \R^3)} \dd s.
\end{split}
\Ee
Combining \eqref{est3:Dphi_12_tb<t} and \eqref{est1:Dphi_12_tb>t}, then for $i = \pm$,
\Be \notag
\begin{split}
& \| \w_{1,i} (F_{1,i}- F_{2,i}) (t) \|_{L^{\infty} (\O \times \R^3)} - \| \w_{1,i} (F_{1,i}- F_{2,i}) (0) \|_{L^{\infty} (\O \times \R^3)}
\\& \lesssim_{e, m, g} \frac{1}{12} \tilde{\beta} \sup_{s \in [0,t]} \| (-\Delta_0)^{-1} (\nabla_x \cdot b_1(s)) \|_{L^\infty(\O)}
\int^t_{0} 	\| \w_{1,i} (F_{1,i}- F_{2,i}) (s) \|_{L^{\infty} (\O \times \R^3)} \dd s 
\\& \qquad \qquad + \sup_{s \in [0,t]} \| \w_{1,i} \nabla_p F_{2, i} (s) \|_{L^{\infty} (\O \times \R^3)}
\times \int^t_{0} \sum_{j = \pm} \| \w_{1,j} (F_{1,j} - F_{2,j} ) (s) \|_{L^{\infty} (\O \times \R^3)} \dd s,
\end{split}
\Ee
Summing up two cases $i = \pm$, we obtain that 
\Be \label{est:Dphi_12}
\begin{split}
& \sum_{i = \pm} \| \w_{1,i} (F_{1,i} - F_{2,i} ) (t) \|_{L^{\infty} (\O \times \R^3)}
- \sum_{i = \pm} \| \w_{1,i} (F_{1,i} - F_{2,i} ) (0) \|_{L^{\infty} (\O \times \R^3)}
\\& \lesssim_{e, m, g} \Big( \frac{1}{12} \tilde{\beta} \sup_{s \in [0,t]} \| (-\Delta_0)^{-1} (\nabla_x \cdot b_1(s)) \|_{L^\infty(\O)}
+ \sum_{j = \pm} \sup_{s \in [0,t]} \| \w_{1, j} \nabla_p F_{2, j} (s) \|_{L^{\infty} (\O \times \R^3)} \Big)
\\& \qquad \qquad \times \int^t_{0} \sum_{j = \pm} \| \w_{1,j} (F_{1,j} - F_{2,j} ) (s) \|_{L^{\infty} (\O \times \R^3)} \dd s.
\end{split}
\Ee
Applying the Gronwall's inequality to \eqref{est:Dphi_12}, we conclude \eqref{est:F1-2}. 
\end{proof}

The uniqueness of the dynamical solution is a direct consequence of Theorem \ref{theo:UD}.

\begin{theorem}[Uniqueness Theorem] 
\label{theo:UA} 

Suppose $(F_{1, \pm}, \nabla_x \phi_{F_1})$ and $(F_{2, \pm}, \nabla_x \phi_{F_2})$ solve \eqref{VP_F}, \eqref{Poisson_F}, \eqref{Dbc:F} and \eqref{bdry:F} under the compatibility condition \eqref{CC_F0=G}.
Assume both follow all assumptions in Theorem \ref{theo:AS} and Theorem \ref{theo:RD}.
Further, we assume that both satisfy \eqref{Uest:wh_dy}, \eqref{Uest:D^-1_Db} and \eqref{est_final:F_v:dyn}.
Then $F_{1, \pm} = F_{2, \pm}$ a.e. in $\R_+ \times \O \times \R^3$ and $\nabla_x \phi_{F_1} = \nabla_x \phi_{F_2}$ a.e. in $\R_+ \times \O$. 
\end{theorem}

\begin{proof} 

Suppose $(F_{1, \pm}, \phi_{F_1})$ and $(F_{2, \pm}, \phi_{F_2})$ are two solutions constructed in Theorem \ref{theo:CD} under the same initial and boundary conditions.
From \eqref{est_final:F_v:dyn}, then for any $(t,x,p) \in [0, \infty) \times  \bar\O \times \R^3$ and $i = 1, 2$,
\be \notag
\| e^{ \frac{\beta}{4} \sqrt{(m c)^2 + |p|^2} } e^{ \frac{m g}{8 c} \beta x_3} \nabla_p F_{i, \pm} (t,x,p) \|_{L^\infty(\O \times \R^3)}  
< \infty.
\ee
Consider $\w_{1, \pm} (t,x,p)$ in \eqref{w1_dy}, from \eqref{Uest_final:F_v:dyn}, we get for $i = 1, 2$,
\be \notag
\sup_{s \in [0,t]} \| \w_{1, i} \nabla_p F_{2, i} (s) \|_{L^{\infty} (\O \times \R^3)}
< \infty.
\ee
Furthermore, using \eqref{Uest:D^-1_Db} in Theorem \ref{theo:AS}, we derive
\be \notag
\sup_{s \in [0,t]} \| (-\Delta_0)^{-1} (\nabla_x \cdot b_1(s)) \|_{L^\infty(\O)}
< \infty,
\ee
where $b_1 (t,x)$ is defined in \eqref{b1_dy}.
Since $(F_{1, \pm}, \phi_{F_1})$ and $(F_{2, \pm}, \phi_{F_2})$ share the same initial and boundary conditions, the uniqueness of the dynamical solution directly follows \eqref{est:F1-2} in Theorem \ref{theo:UD}.
\end{proof}

\subsection{Proof of the Main Theorem: Dynamical Problem} \label{sec:EX_DS}

In this section, we show the existence of dynamic solutions in Theorem \ref{theo:CD}. 
We start with showing some uniform-in-$\ell$ estimates on $( F^{\ell+1}, \phi_{F^{\ell}} )$ constructed in Section \ref{sec:DC}.
Then we obtain the existence of dynamic solutions by passing a limit on the sequences.
We follow the same idea of the proofs in Section \ref{sec:EX_SS}, and thus we omit some steps that are similar to Lemmas in Section \ref{sec:EX_SS}. 

\begin{lemma} 
\label{lem:w/w_ell}

Suppose the assumptions \eqref{Bootstrap_ell} and \eqref{Bootstrap_psi_ell} hold. Further, we assume that
\be \label{est:beta_b_dy_ell}
\sup_{0 \leq t < \infty} \beta \frac{e_{\pm}}{c} \| (-\Delta_0)^{-1} (\nabla_x \cdot b^{\ell} ) \|_{L^\infty_{t,x}} 
\leq 1 \leq \frac{\beta}{2}
\ \text{ for all } \ell \in \N.
\ee
Recall $\Z^{\ell+1}_{\pm} (s;t,x,p) $ in \eqref{Zell} and $\w_{\pm, \beta}^{\ell+1} (t,x,p)$ in \eqref{w^ell}. 
Consider $(t,x,p) \in [0, \infty) \times  \bar\O \times \R^3$, then for any $\ell \in \N$ and $s, s^\prime \in [\max \{0, t-\tBp^{\ell+1}  (t,x,p) \}, t + \tFp^{\ell+1} (t,x,p)]$,
\Be \label{est:1/w_h_ell}
\begin{split}
\frac{\w^{\ell+1}_{\pm, \beta}  (s^\prime, \Z^{\ell+1}_{\pm} (s^\prime;t,x,p) )}{\w^{\ell+1}_{\pm, \beta} (s, \Z^{\ell+1}_{\pm} (s;t,x,p))}
& \leq e^{\beta \frac{e_{\pm}}{c} \| (-\Delta_0)^{-1} (\nabla_x \cdot b^{\ell} ) \|_{L^\infty_{t,x}} 
\big( \frac{ \frac{4 c}{m_{\pm} g} |p| + 4 x_{3} }{c_1} + \frac{8}{m_{\pm} g} |p| + 2 \big) }  ,
\\ \frac{1}{\w^{\ell+1}_{\pm, \beta} (s, \Z^{\ell+1}_{\pm} (s;t,x,p))}  
& \leq e^{2 \beta \frac{e_{\pm}}{c} \| (-\Delta_0)^{-1} (\nabla_x \cdot b^{\ell} ) \|_{L^\infty_{t,x}} } e^{ - \frac{\beta}{2} \sqrt{(m_{\pm} c)^2 + |p|^2} } e^{ - \frac{m_{\pm} g}{4 c} \beta x_3},
\end{split}
\Ee
where $c_1 := \min \{ \frac{g}{4 \sqrt{2}}, \frac{c}{\sqrt{10}} \}$.
Moreover, recall $w_{\pm, \beta} (x,p)$ in \eqref{w^h}, then
\Be \label{est:1/w_ell}
\frac{1}{w_{\pm, \beta} ( \Z^{\ell+1}_{\pm} (s;t,x,p))}   
\leq e^{2 \beta \frac{e_{\pm}}{c} \| (-\Delta_0)^{-1} (\nabla_x \cdot b^{\ell} ) \|_{L^\infty_{t,x}} } e^{- \frac{\beta}{4} \sqrt{(m_{\pm} c)^2 + |p|^2} } e^{- \frac{m_{\pm} g}{8 c} \beta x_3}.
\Ee
\end{lemma}

\begin{proof}

The lemma is a direct consequence of Lemma \ref{lem:tB_ell} and Lemma \ref{lem:w/w} (see more details in Section \ref{sec:AS}).
\end{proof}

\begin{lemma} \label{lem:wf_ell}

Suppose the assumptions \eqref{Bootstrap_ell} and \eqref{est:beta_b_dy_ell} hold. Further, we assume that
\be \notag
\|  \w_{\pm, \beta, 0} F_{\pm, 0} \|_{L^\infty (\O \times \R^3)} 
+ \| e^{\beta \sqrt{(m_{\pm} c)^2 + |p|^2} } G_{\pm} \|_{L^\infty (\gamma_-)}
+ \| w_{\pm, \beta} h_{\pm} \|_{L^\infty  (\O \times \R^3)} < \infty.
\ee
Consider $(t,x,p) \in [0, \infty) \times  \bar\O \times \R^3$, then for any $\ell \in \N$,
\be \label{Uest:wf}
\begin{split}
& e^{ \frac{\beta}{2} \sqrt{(m_{\pm} c)^2 + |p|^2} } e^{ \frac{m_{\pm} g}{4 c} \beta x_3} | f^{\ell+1}_{\pm} (t,x,p) |  
\\& \leq e^{ 2 \beta \frac{e_{\pm}}{c} \| (-\Delta_0)^{-1} (\nabla_x \cdot b^{\ell} ) \|_{L^\infty_{t,x}} } \times \big( \|  \w_{\pm, \beta, 0} F_{\pm, 0} \|_{L^\infty (\O \times \R^3)} + \| e^{\beta \sqrt{(m_{\pm} c)^2 + |p|^2} } G_{\pm} \|_{L^\infty (\gamma_-)} \big) 
\\& \ \ \ \ + \| w_{\pm, \beta} h_{\pm} \|_{L^\infty  (\O \times \R^3)}.
\end{split} 
\ee
\end{lemma}

\begin{proof}

For the sake of simplicity, we abuse the notation as in \eqref{abuse}.
From \eqref{def:Fell}, we have
\Be \label{f=F-h_ell}
f^{\ell+1}(t,x,p)| 
= | F^{\ell+1} (t,x,p) - h(x,p) |
\leq | F^{\ell+1} (t,x,p) | + | h(x,p) |.
\Ee
Using the condition \eqref{Bootstrap_ell}, together with $w_\beta(x,p)$ in \eqref{w^h}, we get 
\Be \notag
w_{\beta} (x,p) \geq e^{ \beta \big( \sqrt{(m c)^2 + |p|^2} + \frac{m g}{2 c} x_3 \big) },
\Ee
and thus
\be \label{decom:wf}
f^{\ell+1}(t,x,p)|  
\leq | F^{\ell+1} (t,x,p) | + e^{ - \beta \big( \sqrt{(m c)^2 + |p|^2} + \frac{m g}{2 c} x_3 \big) } \| w_\beta h  \|_\infty.
\ee
	
Now we do the estimate on $| F^{\ell+1} (t,x,p) |$ in \eqref{decom:wf}. From \eqref{eqtn:Fell} and the characteristics \eqref{Z1}-\eqref{ODEell}, we obtain that, for $s \in [ \max\{0, t- t^{\ell+1 }_{\mathbf{B} } (t,x,p)\}, t- t^{\ell+1 }_{\mathbf{F} } (t,x,p)]$, 
\Be\begin{split}\notag
\frac{d}{ds} 
 F^{\ell+1}  (s, \mathcal{Z} ^{\ell+1} (s;t,x,p) )
 =0.
\end{split}\Ee
Using the initial and boundary conditions in \eqref{bdry_initial:Fell}, we derive that 
\be \notag
\begin{split}
F^{\ell+1} (t,x,p) 
& = \mathbf{1}_{ t \leq t_{\mathbf{B}}^{\ell+1} (t,x,p) } 
F_{0} (\mathcal{Z} ^{\ell +1} (0;t,x,p)) 
\\& \qquad + \mathbf{1}_{t > t_{\mathbf{B}}^{\ell+1  } (t,x,p)}
G ( \mathcal{Z}^{\ell+1} (t - \tB^{\ell +1} (t,x,p);t,x,p) ), 
\end{split} 
\ee 
and thus
\be \label{form:Fell}
\begin{split}
| F^{\ell+1} (t,x,p) |
& \leq \mathbf{1}_{ t \leq t_{\mathbf{B}}^{\ell+1} (t,x,p) } 
| F_{0} (\mathcal{Z} ^{\ell +1} (0;t,x,p)) |
\\& \ \ \ \ + \mathbf{1}_{t > t_{\mathbf{B}}^{\ell+1  } (t,x,p)}
| G ( \mathcal{Z}^{\ell+1} (t - \tB^{\ell +1} (t,x,p);t,x,p) ) |.
\end{split} 
\ee 

For the first term in the right-hand side of \eqref{form:Fell}, we rewrite it as
\be \notag
\mathbf{1}_{ t \leq t_{\mathbf{B}}^{\ell+1} (t,x,p) } | F_{0} (\mathcal{Z} ^{\ell +1} (0;t,x,p)) |  
= \frac{ \mathbf{1}_{ t \leq t_{\mathbf{B}}^{\ell+1} (t,x,p) }}{ \w^{\ell+1}_{\beta}  (0,\mathcal{Z} ^{\ell+1} (0;t,x,p)) } \w^{\ell+1} _{\beta, 0} | F_{0 }  (\mathcal{Z} ^{\ell+1} (0; t,x,p)) |.
\ee
From \eqref{w_initial} and \eqref{est:1/w_h_ell} in Lemma \ref{lem:w/w_ell}, we derive 
\Be \label{est:Fell1}
\begin{split}
& \mathbf{1}_{ t \leq t_{\mathbf{B}}^{\ell+1} (t,x,p) } | F_{0} (\mathcal{Z} ^{\ell +1} (0;t,x,p)) |  
\\& \leq \frac{1}{\w^{\ell+1}_{\beta} (0,\Z^{\ell+1} (0;t,x,v))} \|  \w_{ \beta, 0 } F_{0 } \|_{L^\infty_{x,v}}
\\& \leq e^{2 \beta \frac{e}{c} \| (-\Delta_0)^{-1} (\nabla_x \cdot b^{\ell} ) \|_{L^\infty_{t,x}} } e^{ - \frac{\beta}{2} \sqrt{(m c)^2 + |p|^2} } e^{ - \frac{m g}{4 c} \beta x_3} \times \|  \w_{ \beta, 0} F_{0 } \|_{L^\infty_{x,v}}.
\end{split}
\Ee

For the second term in the right-hand side of \eqref{form:Fell}, we rewrite it as
\be \notag
\begin{split}
& \mathbf{1}_{t > t_{\mathbf{B}}^{\ell+1  } (t,x,p)}
| G ( \mathcal{Z}^{\ell+1} (t - \tB^{\ell +1} (t,x,p);t,x,p) ) |  
\\& = \frac{ \mathbf{1}_{ t > t_{\mathbf{B}}^{\ell+1} (t,x,p) }}{ \w^{\ell+1}_{\beta} ( \mathcal{Z}^{\ell+1} (t - \tB^{\ell +1} (t,x,p);t,x,p) ) } \w^{\ell+1}_{\beta} | G ( \mathcal{Z}^{\ell+1} (t - \tB^{\ell +1} (t,x,p);t,x,p) ) |.
\end{split}
\ee
From \eqref{w_bdry} and \eqref{est:1/w_h_ell} in Lemma \ref{lem:w/w_ell}, we derive 
\Be \label{est:Fell2}
\begin{split}
& \mathbf{1}_{t > t_{\mathbf{B}}^{\ell+1  } (t,x,p)}
| G ( \mathcal{Z}^{\ell+1} (t - \tB^{\ell +1} (t,x,p);t,x,p) ) |   
\\& \leq \frac{1}{ \w^{\ell+1}_{\beta} ( \mathcal{Z}^{\ell+1} (t - \tB^{\ell +1} (t,x,p);t,x,p) ) } \| \w^{\ell+1}_{\beta} G \|_{L^\infty (\gamma_-)}
\\& = \frac{1}{ \w^{\ell+1}_{\beta} ( \mathcal{Z}^{\ell+1} (t - \tB^{\ell +1} (t,x,p);t,x,p) ) } \| e^{\beta \sqrt{(m_{\pm} c)^2 + |p|^2} } G \|_{L^\infty (\gamma_-)}
\\& \leq e^{2 \beta \frac{e}{c} \| (-\Delta_0)^{-1} (\nabla_x \cdot b^{\ell} ) \|_{L^\infty_{t,x}} } e^{ - \frac{\beta}{2} \sqrt{(m c)^2 + |p|^2} } e^{ - \frac{m g}{4 c} \beta x_3} \times \| e^{\beta \sqrt{(m_{\pm} c)^2 + |p|^2} } G \|_{L^\infty (\gamma_-)}.
\end{split}
\Ee
Inputting \eqref{est:Fell1} and \eqref{est:Fell2} into \eqref{form:Fell}, we obtain that
\be \notag
\begin{split}
& e^{ \frac{\beta}{2} \sqrt{(m c)^2 + |p|^2} } e^{ \frac{m g}{4 c} \beta x_3} | F^{\ell+1} (t,x,p) |
\\& \leq e^{2 \beta \frac{e}{c} \| (-\Delta_0)^{-1} (\nabla_x \cdot b^{\ell} ) \|_{L^\infty_{t,x}} } \times \big( \|  \w_{ \beta, 0} F_{0 } \|_{L^\infty_{x,v}} + \| e^{\beta \sqrt{(m_{\pm} c)^2 + |p|^2} } G \|_{L^\infty (\gamma_-)} \big).
\end{split}
\ee 
Together with \eqref{decom:wf}, we conclude \eqref{Uest:wf}.
\end{proof}

\begin{lemma} \label{lem:D^-1 Db_ell}

Consider $b^{\ell} (t, x)$ defined in \eqref{def:flux_ell}-\eqref{identity:Psi_t_ell}, then for any $\ell \in \N$,
\Be \label{D^-1 Db_ell}
\begin{split}
(-\Delta_0)^{-1} (\nabla_x \cdot b^\ell) (t,x) 
= - \int_{\O} \mathfrak{G} (x,y) \nabla  \cdot b^\ell (y) \dd y 
= \int_{\O} b^\ell (y) \cdot \nabla_y \mathfrak{G} (x,y) \dd y.
\end{split}
\Ee
Moreover, if
$\sup\limits_{0 \leq t < \infty } \| e^{ \frac{\beta}{2} \big( \sqrt{(m_{\pm} c)^2 + |p|^2} + \frac{m_{\pm} g}{2 c} x_3 \big) } f^\ell_{\pm} (t) \|_{\infty} < \infty$,
then
\Be \label{est:D^-1 Db_ell}
\begin{split}
& \sup_{0 \leq t < \infty} \| (-\Delta_0)^{-1}  (\nabla_x \cdot b^\ell (t,x)) \|_{L^\infty (\O)} 
\\& \leq \frac{ 8 \pi^2 \hat{e} c }{\beta^2} \sum\limits_{i = \pm} \sup_{0 \leq t < \infty } \| e^{ \frac{\beta}{2} \big( \sqrt{(m_i c)^2 + |p|^2} + \frac{m_i g}{2 c} x_3 \big) } f^\ell_i (t) \|_{L^\infty(\O \times \R^3)} (1 + \frac{4 c}{\hat{m} g \beta}),
\end{split}
\Ee
where $\hat{m} = \min \{ m_+, m_- \}$ and $\hat{e} = \max \{ e_+, e_- \}$.
\end{lemma}

\begin{proof}

Here we omit the proof since it directly follows from Lemma \ref{lem:Ddb}.
\end{proof}

\begin{prop} \label{prop:DC}

Consider $M$ in \eqref{set:M} as follows:
\Be \notag
M = 2 \sum\limits_{i = \pm} \big(
\| \w_{i, \beta, 0 } F_{i, 0} \|_{L^\infty (\O \times \R^3)}
+  \| e^{\beta \sqrt{(m_{i} c)^2 + |p|^2} } G_i \|_{L^\infty (\gamma_-)} \big).
\Ee
Suppose that \eqref{choice:g} holds for $\beta, g>0$.
From $\nabla_x \Phi$ in \eqref{VP_h}-\eqref{eqtn:Dphi} and $\Psi^{\ell}$ in \eqref{eqtn:fell}-\eqref{bdry:Psi_fell}, then for any $\ell \in \N$,
\be \label{Bootstrap_ell_2} 
\begin{split}
\sup_{0 \leq t < \infty} \| \nabla_x \big( \Phi + \Psi^{\ell} (t) \big) \|_{L^\infty(\O)}  
& \leq \min \left\{\frac{m_+}{e_+}, \frac{m_-}{e_-} \right\} \times \frac{g}{2},
\\ \sup_{0 \leq t < \infty}\| \nabla_x \Psi^{\ell} (t) \|_{L^\infty(\O)} & \leq \min \left\{\frac{m_+}{e_+}, \frac{m_-}{e_-} \right\} \times \frac{g}{48}.
\end{split}
\ee
Moreover, from $b^\ell$ in \eqref{def:flux_ell}-\eqref{identity:Psi_t_ell} and $f^{\ell+1}_{\pm}$ in \eqref{eqtn:fell}-\eqref{bdry:Psi_fell}, then for any $\ell \in \N$,
\begin{align}
e^{ 2 \beta \frac{e_{\pm}}{c} \| (-\Delta_0)^{-1} (\nabla_x \cdot b^{\ell} ) \|_{L^\infty_{t,x}} } \leq 2, &
\label{est:e^bell} \\
\sup_{0 \leq t < \infty} \| \varrho^{\ell} (t, x) \|_{L^\infty(\O)}
\leq \frac{2 M}{\beta^2} \big( e_+ e^{ - \frac{m_+ g}{4 c} \beta x_3} + e_- e^{ - \frac{m_+ g}{4 c} \beta x_3} \big), &
\label{Uest:varrhofell} \\
\sup_{0 \leq t < \infty} \| e^{ \frac{\beta}{2} \sqrt{(m_{\pm} c)^2 + |p|^2} } e^{ \frac{m_{\pm} g}{4 c} \beta x_3} f^{\ell+1}_{\pm} (t,x,p)  \|_{L^\infty  (\O \times \R^3)} \leq M. &
\label{Uest:wfell}
\end{align}
\end{prop}

\begin{proof}

Under the initial setting $f^0_{\pm} = 0$ and $(\varrho^0, \nabla_x \Psi^0) = (0, \mathbf{0})$, together with \eqref{Uest:DPhi} on $\nabla_x \Phi$, we check that \eqref{Bootstrap_ell_2}-\eqref{Uest:varrhofell} hold for $\ell = 0$.
Since $\nabla_x \Psi^0 = \mathbf{0}$, using the invariance of $h_{\pm}$ and $w_{\pm, \beta}$ from \eqref{VP_h}-\eqref{eqtn:Dphi} along the characteristics and \eqref{Uest:wh}, together with Lemma \ref{lem:wf_ell} and \eqref{est:e^bell} holding for $\ell = 0$, we deduce \eqref{Uest:wfell} holds for $\ell = 0$.

\smallskip

Now we prove this by induction. 
We abuse the notation as in \eqref{abuse} in the following proof.
Assume a positive integer $k > 0$ and suppose that \eqref{Bootstrap_ell_2}-\eqref{Uest:wfell} hold for $0 \leq \ell \leq k$.
Since \eqref{choice:g} holds for $\beta, g>0$, together with \eqref{Uest:wf} in Lemma \ref{lem:wf_ell}, we derive that 
\be \label{est:wf}
\begin{split}
& e^{ \frac{\beta}{2} \sqrt{(m c)^2 + |p|^2} } e^{ \frac{m g}{4 c} \beta x_3} | f^{k+1} (t,x,p) |  
\\& \leq e^{ 2 \beta \frac{e}{c} \| (-\Delta_0)^{-1} (\nabla_x \cdot b^{k} ) \|_{L^\infty_{t,x}} } \times \big( \|  \w_{ \beta, 0} F_{0 } \|_{L^\infty (\O \times \R^3)} + \| w_{\beta} G \|_{L^\infty (\gamma_-)} \big) 
+ \| w_\beta h \|_{L^\infty  (\O \times \R^3)}
\\& \leq 2 \big( \|  \w_{ \beta, 0} F_{0 } \|_{L^\infty (\O \times \R^3)} + \| w_{\beta} G \|_{L^\infty (\gamma_-)} \big) + \| w_\beta h \|_{L^\infty  (\O \times \R^3)} \leq M,
\end{split} 
\ee
where the last line follows from \eqref{est:e^bell} when $\ell = k$. Thus, \eqref{Uest:wfell} holds for $\ell = k+1$.

Next, using \eqref{est:D^-1 Db_ell} in Lemma \ref{lem:D^-1 Db_ell} and \eqref{Uest:wfell} for $\ell = k+1$, together with \eqref{choice:g}, we deduce
\Be \notag
\begin{split}
& e^{ 2 \beta \frac{e}{c} \| (-\Delta_0)^{-1} (\nabla_x \cdot b^{k+1} ) \|_{L^\infty_{t,x}} }
\\& \leq
e^{ 2 \beta \frac{e}{c} \frac{ 8 \pi^2 \hat{e} c }{\beta^2} \sum\limits_{i = \pm} \sup_{0 \leq t < \infty } \| e^{ \frac{\beta}{2} \big( \sqrt{(m_i c)^2 + |p|^2} + \frac{m_i g}{2 c} x_3 \big) } f^{k+1}_i (t) \|_{L^\infty(\O \times \R^3)} (1 + \frac{4 c}{\hat{m} g \beta}) }  
\leq e^{ \frac{32 \pi^2 \hat{e} e}{\beta} M } \leq 2,
\end{split} 
\Ee
and thus, we prove \eqref{est:e^bell} for $\ell = k+1$.	

Finally, following the proof of \eqref{Uest:DPhi^k}, we derive the first line of \eqref{Bootstrap_ell_2}. 
Note from \eqref{varrhoell} and \eqref{Poisson_fell}, we have \be \notag
- \Delta \Psi^\ell (t, x) = \varrho^{\ell} (t, x) = \int_{\R^3} ( e_+ f^{\ell}_+ + e_{-} f^{\ell}_{-} ) \dd p.
\ee
Applying \eqref{Uest:wfell} for $\ell = k+1$, we get
\be \notag
\begin{split}
& | \varrho^{\ell} (t, x) |
\\& = \big| \int_{\R^3} ( e_+ f^{k+1}_+ + e_{-} f^{k+1}_{-} ) \dd p \big|
\\& \leq \int_{\R^3} e_+ | f^{k+1}_+ | \dd p + \int_{\R^3} e_{-} | f^{k+1}_{-} | \dd p
\\& \leq \| e^{ \frac{\beta}{2} \sqrt{(m c)^2 + |p|^2} } e^{ \frac{m g}{4 c} \beta x_3} f^{k+1}(t,x,p)  \|_{L^\infty  (\O \times \R^3)}
\times \sum\limits_{i = \pm} e^{ - \frac{m_i g}{4 c} \beta x_3} \int_{\R^3} e_i e^{ - \frac{\beta}{2} \sqrt{(m_i c)^2 + |p|^2} } \dd p 
\\& \leq \frac{2 M}{\beta^2} \big( e_+ e^{ - \frac{m_+ g}{4 c} \beta x_3} + e_- e^{ - \frac{m_+ g}{4 c} \beta x_3} \big).
\end{split} 
\ee
Thus we conclude \eqref{Uest:varrhofell} for $\ell = k+1$.
Together with \eqref{est:nabla_phi} in Lemma \ref{lem:rho_to_phi}, we conclude the second line of \eqref{Bootstrap_ell_2} for $\ell = k+1$.

\smallskip

Now we have proved \eqref{est:e^bell} and \eqref{Uest:wfell} hold for $\ell = k+1$. Therefore, we complete the proof by induction.
\end{proof}

\begin{prop} \label{prop:Unif_D2xDp_dy}

Suppose the condition \eqref{condition:beta} holds for some $g, \beta > 0$.
Moreover, the condition \eqref{condition:G} and 
\be \label{condition:F0_G_dy}
\| \w_{\pm, \tilde \beta, 0}  \nabla_{x,p} F_{\pm, 0} \|_{L^\infty (\O \times \R^3)} +
\| e^{{\tilde \beta } \sqrt{(m_{\pm} c)^2 + |p|^2}} \nabla_{x_\parallel, p} G_{\pm} (x,p) \|_{L^\infty (\gamma_-)} < \tilde{\beta} e^{ - \frac{m_{\pm} g}{24} \tilde{\beta} },
\ee
hold for some $\beta \geq \tilde \beta >0$.
Then $(F^{\ell+1}_{\pm}, \nabla_x \phi_{F^\ell} )$ from the construction also satisfies the following uniform-in-$\ell$ estimates:
\be \label{Uest:DDPhi^l_dy}
\sup_{0 \leq t < \infty} 
\big\{ (1 + B_3 + \| \nabla_x ^2 \phi_{F^{\ell}}  \|_\infty) + \| \p_t \p_{x_3} \phi_{F^\ell} (t, x_\parallel , 0) \|_{L^\infty(\p\O)} \big\}
\leq \frac{\hat{m} g}{24} \tilde{\beta},
\ee
and
\Be \label{Uest:h_v^l_dy}
\begin{split} 
& \sup_{0 \leq t < \infty} \| e^{ \frac{\tilde{\beta}}{4} \sqrt{(m_{\pm} c)^2 + |p|^2} } e^{ \frac{m_{\pm} g}{8 c} \tilde{\beta} x_3} \nabla_p F^{\ell+1}_{\pm} (t,x,p) \|_{L^\infty(\O \times \R^3)} 
\\& \lesssim 4 e^{ \frac{m_{\pm} g}{24} \tilde{\beta} } 
\| \w_{\pm, \tilde \beta, 0}  \nabla_{x,p} F_{\pm, 0} \|_{L^\infty (\O \times \R^3)}
+ \frac{m_{\pm} g}{24} \tilde{\beta} \| e^{\tilde{\beta} |p^0_{\pm}|} \nabla_{x_\parallel,p} G_{\pm} \|_{L^\infty (\gamma_-)},
\end{split}
\Ee 
where $\hat{m} = \min \{ m_+, m_- \}$.
\end{prop}

\begin{proof}

We abuse the notation as in \eqref{abuse} in the proof.
From Theorem \ref{theo:RS}, there exists a solution $(h_{\pm}, \rho, \Phi)$ solving \eqref{VP_h}-\eqref{eqtn:Dphi}, where it satisfies \eqref{theo:rho_x}-\eqref{theo:hk_x} under all conditions in Theorem \ref{theo:RS}.
Using $h_{\pm} (x, p)$ and $\Phi (x)$, together with the construction in \eqref{eqtn:fell}-\eqref{bdry:Psi_fell} and \eqref{def:Fell}-\eqref{Poisson_Fell}, we obtain
$(F^{\ell}_{\pm}, \phi_{F^{\ell}})$ for any $\ell \in \N$.

\smallskip

\textbf{Step 1. }
Under the initial setting $f^0_{\pm} = 0$ and $(\varrho^0,  \nabla_x \Psi^0) = (0, \mathbf{0})$, together with \eqref{theo:phi_C2} and \eqref{condition:F0_G_dy}, we deduce that $\phi_{F^0} = \Phi$, and thus \eqref{Uest:DDPhi^l_dy} holds for $\ell = 0$.
From the construction, $F^1_{\pm}$ is the solution to \eqref{eqtn:Fell} and \eqref{bdry_initial:Fell} with $\nabla_x \Psi^0 = \mathbf{0}$, that is,
\be \label{VP^0_dy}
\begin{split}
& \p_t F^{1}_{\pm}  + v_\pm \cdot \nabla_x F^{1}_{\pm} + \Big( e_{\pm} \big( \frac{v_\pm}{c} \times B - \nabla_x \Phi  \big) - \nabla_x ( m_\pm g x_3) \Big) \cdot \nabla_p F^{1}_{\pm} 
= 0 \ \ \text{in} \ \R_+ \times  \O \times \R^3,
\\& F^{1}_{\pm} (0, x, p) = F_{\pm, 0} 
\ \ \text{in} \ \O \times \R^3,
\ \
F^{1}_{\pm} (t, x, p) = G_{\pm} (x,p)
\ \ \text{on} \ \gamma_-.
\end{split}
\ee
From \eqref{Z1} and \eqref{ODE1}, we consider the characteristic $(\X^1, \P^1)$ for \eqref{VP^0_dy} and the corresponding $(\tB^1, \xB^1, \pB^1)$.
Note that we only work with Vlasov equations under the assumption \eqref{Bootstrap}.
By replacing $\phi_F$ with $\Phi$, Lemma \ref{VL:dyn} and \eqref{est:xb_x:dyn}-\eqref{est:vb_v:dyn} hold for the characteristic $(\X^1, \P^1)$ and the corresponding $(\tB^1, \xB^1, \pB^1)$, since $\nabla_x \Phi$ satisfies the assumption \eqref{Bootstrap}.

Analogously, from \eqref{est:e^bell} and \eqref{Uest:DDPhi^l_dy} holding for $\ell = 0$, we check that $\| \nabla^2_x \Phi \|_{\infty}$ satisfies the assumptions \eqref{est:beta_b_dy} and \eqref{est1:D3tphi_F} in Proposition \ref{RE:dyn}.
Following the proof of \eqref{est:F_v:dyn} in Proposition \ref{RE:dyn} and Theorem \ref{theo:RD}, together with \eqref{condition:F0_G_dy}, we deduce \eqref{Uest:h_v^l_dy} holds for $\ell = 0$.

\smallskip

\textbf{Step 2. }
Now we prove this by induction. Assume a positive integer $k > 0$ and suppose that \eqref{Uest:DDPhi^l_dy} and \eqref{Uest:h_v^l_dy} hold for $0 \leq \ell \leq k$.
From \eqref{def:Fell}-\eqref{Poisson_Fell}, then for $\ell=k+1$, 
\be \label{rho_phi^k+1_dy}
\begin{split}	
- \Delta \phi_{F^{k+1}} (t, x) = \int_{\R^3} ( e_+ F^{k+1}_+ + e_{-} F^{k+1}_{-} ) \dd p & \ \ \text{in} \ \R_+ \times \O, \\ \phi_{F^{k+1}} (t, x) =0  & \ \ \text{on} \ \R_+ \times \p\O,
\end{split}
\ee
and
\be \label{h^k+2_dy}
\begin{split}
& \p_t F^{k+2}_{\pm}  + v_\pm \cdot \nabla_x F^{k+2}_{\pm} + \Big( e_{\pm} \big( \frac{v_\pm}{c} \times B - \nabla_x ( \Phi + \Psi^{k+1} ) \big) - \nabla_x ( m_\pm g x_3) \Big) \cdot \nabla_p F^{k+2}_{\pm} 
\\& \ \ \ \ = 0 \ \ \text{in} \ \R_+ \times  \O \times \R^3,
\end{split}
\ee
with
\be \label{bdry_initial:Fk+2_dy}
F^{k+2}_{\pm} (0, x, p) = F_{\pm, 0} 
\ \ \text{in} \ \O \times \R^3,
\ \
F^{k+2}_{\pm} (t, x, p) = G_{\pm} (x,p)
\ \ \text{on} \ \gamma_-.
\ee

On the other side, from the construction, $F^{k+1}_{\pm}$ is the solution to \eqref{VP_F}, \eqref{Poisson_F}, \eqref{VP_0},  \eqref{Dbc:F} and \eqref{bdry:F} with $\nabla_x \phi_F = \nabla_x ( \Phi + \Psi^{k} )$, that is,
\be \label{VP^k_dy}
\begin{split}
& \p_t F^{k+1}_{\pm}  + v_\pm \cdot \nabla_x F^{k+1}_{\pm} + \Big( e_{\pm} \big( \frac{v_\pm}{c} \times B - \nabla_x ( \Phi + \Psi^{k} ) \big) - \nabla_x ( m_\pm g x_3) \Big) \cdot \nabla_p F^{k+1}_{\pm} 
\\& \ \ \ \ = 0 \ \ \text{in} \ \R_+ \times  \O \times \R^3.
\end{split}
\ee
Recall \eqref{Bootstrap_ell_2} and \eqref{est:e^bell} in Proposition \ref{prop:DC}, then for any $\ell \geq 0$,
\be \label{est:Phi_x^0_dy}
\begin{split}
& \sup_{0 \leq t < \infty} \| \nabla_x \big( \Phi + \Psi^{\ell} (t) \big) \|_{L^\infty(\O)}  
\leq \min \left\{\frac{m_+}{e_+}, \frac{m_-}{e_-} \right\} \times \frac{g}{2},
\\& \sup_{0 \leq t < \infty}\| \nabla_x \Psi^{\ell} (t) \|_{L^\infty(\O)} 
\leq \min \left\{\frac{m_+}{e_+}, \frac{m_-}{e_-} \right\} \times \frac{g}{48}
\\& \sup_{0 \leq t < \infty} e^{ 2 \beta \frac{e}{c} \| (-\Delta_0)^{-1} (\nabla_x \cdot b^{\ell} ) \|_{L^\infty(\O)} } \leq 2.
\end{split}
\ee
From \eqref{Zell} and \eqref{ODEell}, we consider the characteristic $(\X^{k+1}, \P^{k+1})$ for \eqref{VP^k_dy} and the corresponding $(\tB^{k+1}, \xB^{k+1}, \pB^{k+1})$.
Similar as the initial case, by replacing $\phi_F$ with $\Phi + \Psi^k$, Lemma \ref{VL:dyn} and \eqref{est:xb_x:dyn}-\eqref{est:vb_v:dyn} hold for for $(X^{k+1}, P^{k+1})$ and $(\tb^{k+1}, \xb^{k+1}, \pb^{k+1})$ due to \eqref{est:Phi_x^0_dy}.

\smallskip

\textbf{Step 3. }
First, we show \eqref{Uest:DDPhi^l_dy} holds for $\ell = k+1$. Note that \eqref{Uest:DDPhi^l_dy} holds for $\ell = k$, that is, 
\be \label{Uest:DDPhi^k_dy}
(1 + B_3 + \| \nabla_x ^2 \phi_{F^k}  \|_\infty) + \| \p_t \p_{x_3} \phi_{F^k} (t, x_\parallel , 0) \|_{L^\infty(\p\O)} 
\leq \frac{\hat{m} g}{24} \tilde{\beta}.
\ee
By replacing $\phi_F$ with $\phi_{F^k}$ and from \eqref{est:Phi_x^0_dy}, we follow the proof of \eqref{est:F_x:dyn} in Proposition \ref{RE:dyn}, and derive that
\Be \label{est:rho_x^k+1_dy}
\begin{split} 
& e^{ \frac{\tilde{\beta}}{4} \sqrt{(m c)^2 + |p|^2} } e^{ \frac{m g}{8 c} \tilde{\beta} x_3} \big| \nabla_{x_i} F^{k+1} (t,x,p) \big|
\\& \lesssim 2 e^{ \frac{m g}{24} \tilde{\beta} }
e^{2 \tilde{\beta} \frac{e}{c} \| (-\Delta_0)^{-1} (\nabla_x \cdot b ) \|_{L^\infty_{t,x}} } \| \w_{\tilde \beta, 0}   \nabla_{x,p} F_0  \|_{L^\infty (\O \times \R^3)}
\\& \ \ \ \ + \big( \frac{\delta_{i3}}{\alpha_F (t,x,p)} + 1 + \frac{ \| \nabla_x^2 \phi_{F^k} \|_\infty + e B_3 + mg}{(m g)^2} \big) \| e^{\tilde{\beta} |p^0|} \nabla_{x_\parallel,p} G \|_{L^\infty (\gamma_-)},
\end{split}
\Ee 
Consider $F = F^{k+1}$. Using \eqref{est:Phi_x^0_dy} and \eqref{est:rho_x^k+1_dy}, we follow the proof of \eqref{est:D3tphi_F} in Lemma \ref{lem:D3tphi_F} and deduce that
\be \label{est:D3tphi_F_dy} 
| \p_{x_3} \p_t \phi_{F^{k+1}}(t,x)| 
\lesssim \frac{ 1 }{{\tilde{\beta}}^3} + \frac{ 1 }{{\tilde{\beta}}^3} ( 1 + \frac{8 c}{\hat{m} g \tilde{\beta}} ).
\ee
On the other hand, from \eqref{est:Phi_x^0_dy} and \eqref{est:rho_x^k+1_dy}, we follow the proof of \eqref{est_final:phi_F} in Theorem \ref{theo:RD} and obtain that
\Be \notag
\| \nabla_x^2 \phi_{F^{k+1}} (t) \|_{L^\infty (\bar \O)}
\lesssim \| \w_{\tilde \beta, 0}   \nabla_{x,v} F_0  \|_{L^\infty (\O \times \R^3)} 
+  \| e^{\tilde{\beta} |p^0|} \nabla_{x_\parallel,p} G   \|_{L^\infty (\gamma_-)}.
\Ee
Under the condition \eqref{condition:F0_G_dy}, the above and \eqref{est:D3tphi_F_dy} show that
\be \label{est1:DDPhi^k+1_dy}
(1 + B_3 + \| \nabla_x ^2 \phi_{F^{k+1}}  \|_\infty) + \| \p_t \p_{x_3} \phi_{F^{k+1}} (t, x_\parallel , 0) \|_{L^\infty(\p\O)} 
\leq \frac{\hat{m} g}{24} \tilde{\beta},
\ee
and \eqref{Uest:DDPhi^l_dy} holds for $\ell = k+1$.

\smallskip

\textbf{Step 4. }
Second, we show \eqref{Uest:h_v^l_dy} holds for $\ell = k+1$.
Note that \eqref{est:Phi_x^0_dy} and \eqref{Uest:DDPhi^l_dy} holding for $\ell = k+1$.
By replacing $\phi_F$ with $\phi_{F^{k+1}}$, we follow the proof of \eqref{est:F_v:dyn} in Proposition \ref{RE:dyn}, and get
\Be \notag
\begin{split} 
& \sup_{0 \leq t < \infty} \| e^{ \frac{\tilde{\beta}}{4} \sqrt{(m c)^2 + |p|^2} } e^{ \frac{m g}{8 c} \tilde{\beta} x_3} \nabla_p F^{k+2} (t,x,p) \|_{L^\infty(\O \times \R^3)} 
\\& \lesssim 2 e^{ \frac{m g}{24} \tilde{\beta} }
e^{2 \tilde{\beta} \frac{e}{c} \| (-\Delta_0)^{-1} (\nabla_x \cdot b ) \|_{L^\infty_{t,x}} } \| \w_{\tilde \beta, 0}   \nabla_{x,p} F_0  \|_{L^\infty (\O \times \R^3)}
\\& \ \ \ \ + \big( 1 + \frac{ \| \nabla_x^2 \phi_{F^{k+1}} \|_\infty + e B_3 + mg}{(m g)^2} \big) \| e^{\tilde{\beta} |p^0|} \nabla_{x_\parallel,p} G \|_{L^\infty (\gamma_-)}.
\end{split}
\Ee 
From the condition \eqref{condition:F0_G_dy}, we deduce that \eqref{Uest:h_v^l_dy} holds for $\ell = k+1$.

\smallskip

Now we have proved \eqref{Uest:DDPhi^l_dy} and \eqref{Uest:h_v^l_dy} hold for $\ell = k+1$. Therefore, we complete the proof by induction.
\end{proof}

Similar to Proposition \ref{prop:cauchy}, using the estimates  in Proposition \ref{prop:Unif_D2xDp_dy}, we show that $\{ F^{\ell+1}_{\pm} \}^{\infty}_{\ell=0}$ and $\{ \nabla_x \phi_{F^\ell} \}^{\infty}_{\ell=0}$ from the construction are both Cauchy sequences.

\begin{prop} \label{prop:cauchy_dy}

Consider $M$ and $L$ in \eqref{set:M} and \eqref{set:L}. Further, we assume that the conditions \eqref{condition:G} and 
\Be \label{condition2:F0_G_dy}
M \leq \beta e^{ - \frac{m_{\pm} g}{24} \beta}
\ \text{ and } \ 
L \leq \tilde{\beta} e^{ - \frac{m_{\pm} g}{24} \tilde{\beta} }.
\Ee
hold for some $g, \beta, \tilde \beta > 0$.
Then for any $\ell \geq 1$, $F^{\ell}_{\pm}$ from the construction satisfies that there exists some $\bar \beta > 0$,  
\be \label{est:h_cauchy_dy}
\begin{split}
& \sup_{0 \leq t < \infty} \Big( \| e^{ \frac{3 \bar{\beta}}{4} \big( \sqrt{(m_+ c)^2 + |p|^2} + \frac{1}{2 c} m_+ g x_3 \big) } (F^{\ell+1}_{+} - F^{\ell}_{+} ) \|_{L^{\infty} (\O \times \R^3)} 
\\& \qquad \qquad + \| e^{ \frac{3 \bar{\beta}}{4} \big( \sqrt{(m_- c)^2 + |p|^2} + \frac{1}{2 c} m_- g x_3 \big) } (F^{\ell+1}_{-} - F^{\ell}_{-} ) \|_{L^{\infty} (\O \times \R^3)} \Big)
\\& \leq \frac{1}{2} \sup_{0 \leq t < \infty} \Big( \| e^{ \frac{3 \bar{\beta}}{4} \big( \sqrt{(m_+ c)^2 + |p|^2} + \frac{1}{2 c} m_+ g x_3 \big) } (F^\ell_{+} - F^{\ell-1}_{+} ) \|_{L^{\infty} (\O \times \R^3)} 
\\& \qquad \qquad \qquad + \| e^{ \frac{3 \bar{\beta}}{4} \big( \sqrt{(m_- c)^2 + |p|^2} + \frac{1}{2 c} m_- g x_3 \big) } (F^\ell_{-} - F^{\ell-1}_{-} ) \|_{L^{\infty} (\O \times \R^3)}  \Big).
\end{split}
\Ee
Furthermore, $\{ F^{\ell+1}_{\pm} \}^{\infty}_{\ell=0}$, and $\{ \varrho^\ell \}^{\infty}_{\ell=0}$, $\{ \nabla_x \phi_{F^\ell} \}^{\infty}_{\ell=0}$ from the construction are Cauchy sequences in $L^{\infty} (\R_+ \times \O \times \R^3)$ and $L^{\infty} (\R_+ \times \O)$ respectively.
\end{prop}

\begin{proof}

We abuse the notation as in \eqref{abuse} in the proof.
We first prove \eqref{est:h_cauchy_dy}. 
Given $\ell \geq 1$, from the construction in \eqref{def:Fell}-\eqref{Poisson_Fell}, we have 
\Be \label{VP_diff^l+1_dy}
\begin{split}
& \p_t (F^{\ell+1}_{\pm} - F^{\ell}_{\pm}) + v_\pm \cdot \nabla_x (F^{\ell+1}_{\pm} - F^{\ell}_{\pm}) 
\\& \ \ \ \ + \Big( e_{\pm} \big( \frac{v_\pm}{c} \times B - \nabla_x ( \Phi + \Psi^{\ell} ) \big) - \nabla_x ( m_\pm g x_3) \Big) \cdot \nabla_p (F^{\ell+1}_{\pm} - F^{\ell}_{\pm})
\\& = e_{\pm} \nabla_x ( \Psi^{\ell} - \Psi^{\ell-1})
\cdot \nabla_p F^{\ell}_{\pm}
 \ \ \text{in} \ \R_+ \times  \O \times \R^3,
\end{split}
\ee
with 
\be \label{VP_diff_bdy^l+1_dy}
F^{\ell+1}_{\pm} (t,x,p) - F^{\ell}_{\pm} (t,x,p) = 0 \ \ \text{in} \ \gamma_-,
\Ee
and
\be \label{VP_diff_initial^l+1_dy}
F^{\ell+1}_{\pm} (0, x, p) - F^{\ell}_{\pm} (0, x, p) = 0 \ \ \ \ \text{in} \ \O \times \R^3.
\Ee
Following \eqref{Zell}, \eqref{ODEell} and \eqref{def:tB_l}, we consider the characteristic $(\X^{\ell+1}, \P^{\ell+1})$ for \eqref{VP_diff^l+1_dy} and the corresponding backward exit time $\tB^{\ell+1} (t,x,p)$. Thus, from \eqref{VP_diff^l+1_dy}, \eqref{VP_diff_bdy^l+1_dy} and \eqref{VP_diff_initial^l+1_dy}, we get
\Be \label{diff:h^l+1_dy}
\begin{split}
& (F^{\ell+1} - F^{\ell} ) (t,x,p) 
\\& = \int^t_{ \max\{0, t - \tB^{\ell+1} (t,x,p) \}} 
e \nabla_x ( \Psi^{\ell} - \Psi^{\ell-1})
\cdot \nabla_p F^{\ell} ( \X^{\ell+1} (s;t,x,p), \P^{\ell+1} (s;t,x,p)) \dd s.
\end{split}
\Ee
Using \eqref{Uest:h_v^l_dy} in Proposition \ref{prop:Unif_D2xDp_dy}, together with the condition \eqref{condition2:F0_G_dy}, then for any $\ell \geq 1$,
\be \label{est:h_v^l+1_dy}
\begin{split}
& e^{ \frac{\tilde \beta}{4} |p^0|} e^{  \frac{\tilde \beta m g}{8 c} x_3} | \nabla_p F^{\ell} (t,x,p)| 
\\& \lesssim 4 e^{ \frac{m g}{24} \tilde{\beta} } 
\| \w_{\tilde \beta, 0}  \nabla_{x,p} F_0  \|_{L^\infty (\O \times \R^3)}
+  \frac{m g}{24} \tilde{\beta} \| e^{\tilde{\beta} |p^0|} \nabla_{x_\parallel,p} G \|_{L^\infty (\gamma_-)}.
\end{split}
\ee
From \eqref{Bootstrap_ell_2} in Proposition \ref{prop:DC}, then for any $k \geq 1$,
\be \label{est:phi_x^l_dy}
\sup_{0 \leq t < \infty} \| \nabla_x \big( \Phi + \Psi^{k} (t) \big) \|_{L^\infty(\O)}
\leq  \min \big( \frac{m_{+}}{e_{+}}, \frac{m_{-}}{e_{-}} \big) \times \  \frac{g}{2}. 
\ee
Recall $\w^{\ell+1}_{\pm, \beta} (t,x,p)$ defined in \eqref{w^ell}, from \eqref{est:phi_x^l_dy} we get for any $\ell \geq 1$,
\be \label{upper_w^l+1_dy}
\begin{split}
\w^{\ell+1}_{\pm, \beta} (t,x,p) 
& = e^{ \beta \left( \sqrt{(m_{\pm} c)^2 + |p|^2} + \frac{1}{c} \big( e \big( \Phi (x) + \Psi^\ell (t,x) \big) + m_{\pm} g x_3 \big) \right) }
\\& \leq e^{ \beta \left( \sqrt{(m_{\pm} c)^2 + |p|^2} + \frac{3}{2 c} ( m_{\pm} g x_3 ) \right) },
\end{split}
\ee
and
\Be \label{lower_w^l+1_dy}
\w^{\ell +1}_{\bar \beta} (t,x,p) 
\geq e^{ \bar \beta \big( \sqrt{(m c)^2 + |p|^2} + \frac{1}{2 c} m g x_3 \big) }.
\Ee
By setting $\bar{\beta} = \frac{\tilde \beta}{6}$, from \eqref{upper_w^l+1_dy}, we can rewrite \eqref{est:h_v^l+1_dy} as
\be \label{est1:h_v^l+1_dy}
\begin{split}
& \w^{\ell+1}_{\pm, \bar{\beta}} (t,x,p)  | \nabla_p F^{\ell}_{\pm} (t,x,p)| 
\\& \lesssim 4 e^{ \frac{m g}{24} \tilde{\beta} } 
\| \w_{\tilde \beta, 0}  \nabla_{x,p} F_0  \|_{L^\infty (\O \times \R^3)}
+  \frac{m g}{24} \tilde{\beta} \| e^{\tilde{\beta} |p^0|} \nabla_{x_\parallel,p} G \|_{L^\infty (\gamma_-)}, 
\text{ for any } \ell \geq 1.
\end{split}
\ee
Thus, we obtain that
\Be \label{bound:diff_h^l+1_dy}
\begin{split}
& |\eqref{diff:h^l+1_dy}| 
\\& \leq e \| \nabla_x \Psi^{\ell} - \nabla_x \Psi^{\ell-1} \|_\infty 
\int^t_{ \max\{0, t - \tB^{\ell+1} (t,x,p) \}} 
\big| \nabla_p F^{\ell} ( \X^{\ell+1} (s;t,x,p), \P^{\ell+1} (s;t,x,p)) \big| \dd s
\\& \leq e \| \nabla_x \Psi^{\ell} - \nabla_x \Psi^{\ell-1} \|_\infty \| \w^{\ell+1}_{\bar\beta} \nabla_p F^{\ell} \|_\infty 
\int^t_{ \max\{0, t - \tB^{\ell+1} (t,x,p)\}}  \frac{1}{ \w^{\ell+1}_{\bar\beta} ( \Z^{\ell+1} (s;t,x,p) ) } \dd s
\\& \leq e \| \nabla_x \Psi^{\ell} - \nabla_x \Psi^{\ell-1} \|_\infty \| \w^{\ell+1}_{ \bar{\beta} } \nabla_p F^{\ell} \|_\infty 
\\& \qquad \qquad \times
\underbrace{\tB^{\ell+1} (t, x, p) \sup_{ s \in [t - \tB^{\ell+1} (t,x,p), t]}   \left( \frac{1}{ \w^{\ell+1}_{\bar{\beta}} ( \Z^{\ell+1} (s;t,x,p) ) } \right )}_{\eqref{bound:diff_h^l+1_dy}_*}.
\end{split}
\Ee
Furthermore, from \eqref{est:phi_x^l_dy} and Lemma \ref{lem:tB_ell}, we obtain
\Be \label{est:tb^l+1_dy}
\tB^{\ell + 1} (t,x,p) 
\lesssim \frac{4 c + 4}{m g} |p| + 3 x_{3} + 1.
\Ee
Using \eqref{est:tb^l+1_dy} and \eqref{est:1/w_h}, together with \eqref{est:e^bell} and \eqref{lower_w^l+1_dy}, we get 
\Be \label{est:tb/w^l+1_dy}
\begin{split}
\eqref{bound:diff_h^l+1_dy}_* 
& \leq \big(  \frac{4 c + 4}{m g} |p| + 3 x_{3} + 1 \big)
e^{2 \beta \frac{e}{c} \| (-\Delta_0)^{-1} (\nabla_x \cdot b ) \|_{L^\infty_{t,x}} } e^{ - \frac{\beta}{2} \sqrt{(m c)^2 + |p|^2} } e^{ - \frac{m g}{4 c} \beta x_3}
\\& \leq 2 \big(  \frac{4 c + 4}{m g} |p| + 3 x_{3} + 1 \big) e^{ - \bar \beta \big( \sqrt{(m c)^2 + |p|^2} + \frac{1}{2 c} m g x_3 \big) }
\\& \leq \frac{48}{m g \bar \beta} e^{ - \frac{3 \bar \beta}{4} \big( \sqrt{(m c)^2 + |p|^2} + \frac{1}{2 c} m g x_3 \big) }.
\end{split}
\Ee
Applying \eqref{est:tb/w^l+1_dy} and \eqref{bound:diff_h^l+1_dy} into \eqref{diff:h^l+1_dy}, we have
\be \label{bound2:diff_h^l+1_dy}
\begin{split}
& | F^{\ell+1} (t,x,p) - F^{\ell} (t,x,p) | 
\\& \leq e \| \nabla_x \Psi^{\ell} - \nabla_x \Psi^{\ell-1} \|_\infty \| \w^{\ell+1}_{ \bar{\beta} } \nabla_p F^{\ell} \|_\infty \times \frac{48}{m g \bar \beta} e^{ - \frac{3 \bar \beta}{4} \big( \sqrt{(m c)^2 + |p|^2} + \frac{1}{2 c} m g x_3 \big) }.
\end{split}
\ee
This shows that
\be \label{bound3:diff_h^l+1_dy}
\begin{split}
& e^{ \frac{3 \bar \beta}{4} \big( \sqrt{(m c)^2 + |p|^2} + \frac{1}{2 c} m g x_3 \big) } | F^{\ell+1} (t,x,p) - F^{\ell} (t,x,p) | 
\\& \leq \frac{48 e}{m g \bar \beta} \| \nabla_x \Psi^{\ell} - \nabla_x \Psi^{\ell-1} \|_\infty \| \w^{\ell+1}_{ \bar{\beta} } \nabla_p F^{\ell} \|_\infty.
\end{split}
\ee

On the other hand, from \eqref{Uest:varrhofell} in Proposition \ref{prop:DC} and $M$ in \eqref{set:M}, we obtain
\be \label{Uest:rho_12^l+1_dy}
\begin{split}
| \varrho^{\ell} (x) |
& \leq \frac{2 M}{\beta^2} \big( e_+ e^{ - \frac{m_+ g}{4 c} \beta x_3} + e_- e^{ - \frac{m_+ g}{4 c} \beta x_3} \big),
\\ | \varrho^{\ell-1} (x) |
& \leq \frac{2 M}{\beta^2} \big( e_+ e^{ - \frac{m_+ g}{4 c} \beta x_3} + e_- e^{ - \frac{m_+ g}{4 c} \beta x_3} \big).
\end{split}
\ee
Consider $\hat{m} = \min \{ m_{-},  m_{+} \}$, then we pick $\beta'$ as follows:
\be \label{def:beta'^l+1_dy}
0 < \beta' =  \min\{ \frac{\bar \beta}{4}, \beta \} \times \hat{m}.
\ee
Hence, $\beta' \leq \beta \hat{m}$ and \eqref{Uest:rho_12^l+1_dy} shows that
\be \notag
\| e^{\beta' \frac{g}{2c} x_3} \varrho^{\ell} (x) \|_\infty
\leq \frac{2 M}{\beta^2} < \infty, \ \
\| e^{\beta' \frac{g}{2c} x_3} \varrho^{\ell-1} (x) \|_\infty
\leq \frac{2 M}{\beta^2} < \infty.
\ee
Set $A = \| e^{\beta' \frac{g}{2c} x_3} (\varrho^{\ell} - \varrho^{\ell-1} ) \|_\infty$ and $B = \beta' \frac{g}{2c}$, we get 
\be \notag
| \varrho^{\ell} - \varrho^{\ell-1} | \leq A e^{- B x_3}.
\ee
Using \eqref{est:nabla_phi} in Lemma \ref{lem:rho_to_phi}, together with $- \Delta ( \Psi^{\ell} - \Psi^{\ell-1} ) = \varrho^{\ell} - \varrho^{\ell-1}$, we obtain
\Be \label{est:phi_12_x^l+1_dy}
\| \nabla_x \Psi^{\ell} - \nabla_x \Psi^{\ell-1} \|_\infty 
\leq \mathfrak{C} (1 + \frac{2 c}{\beta^\prime g} )
\| e^{\beta' \frac{g}{2c} x_3} ( \varrho^{\ell} - \varrho^{\ell-1} ) \|_\infty.
\Ee
Now setting $\hat{\beta} = \frac{\beta'}{\min \{ m_{-},  m_{+} \}}$, we bound $| \varrho^{\ell} - \varrho^{\ell-1} |$ as
\Be \notag
\begin{split}
| \varrho^{\ell} - \varrho^{\ell-1} |
& = \big| \int_{\R^3} e_+ (F^{\ell}_{+} - F^{\ell-1}_{+} ) + e_{-} ( F^{\ell}_{-} - F^{\ell-1}_{-} ) \dd p \big|
\\& \leq \sum\limits_{i = \pm} e_i \| \w^{\ell}_{i, \hat{\beta}} (F^{\ell}_{i} - F^{\ell-1}_{i} ) \|_{L^{\infty} (\O \times \R^3)} \times \int_{\R^3} \frac{1}{ \w^{\ell}_{i, \hat{\beta}}} \dd p.
\end{split}
\Ee
This, together with \eqref{lower_w^l+1_dy}, shows that
\Be \notag
\begin{split}
| \varrho^{\ell} - \varrho^{\ell-1} |
& \leq \sum\limits_{i = \pm} e_i \| \w^{\ell}_{i, \hat{\beta}} (F^{\ell}_{i} - F^{\ell-1}_{i} ) \|_{L^{\infty} (\O \times \R^3)} \times \frac{1}{\hat{\beta}} e^{- \hat{\beta} \frac{m _i g}{2c} x_3}
\\& \leq \sum\limits_{i = \pm} e_i \| \w^{\ell}_{i, \hat{\beta}} (F^{\ell}_{i} - F^{\ell-1}_{i} ) \|_{L^{\infty} (\O \times \R^3)} \times \frac{1}{\hat{\beta}} e^{- \beta' \frac{g}{2c} x_3}.
\end{split}
\Ee
Note that \eqref{def:beta'^l+1} implies $\hat{\beta} \leq \frac{\bar{\beta}}{4}$. Further, using \eqref{upper_w^l+1_dy}, we derive that
\Be \label{est:e^beta*rho_12^l+1_dy}
e^{\beta' \frac{g}{2c} x_3} | \varrho^{\ell} - \varrho^{\ell-1} |
\leq \sum\limits_{i = \pm} \frac{e_i}{\hat{\beta}} \| e^{ \hat{\beta} \big( \sqrt{(m_i c)^2 + |p|^2} + \frac{3}{2 c} m_i g x_3 \big) }  (F^{\ell}_{i} - F^{\ell-1}_{i} ) \|_{L^{\infty} (\O \times \R^3)}.
\Ee
Inputting \eqref{est:e^beta*rho_12^l+1_dy} into \eqref{est:phi_12_x^l+1_dy}, we obtain
\Be \label{est2:phi_12_x^l+1_dy}
\begin{split}
& \| \nabla_x \Phi^{\ell} - \nabla_x \Phi^{\ell-1} \|_\infty 
\\& \leq \mathfrak{C} (1 + \frac{2 c}{\beta^\prime g} ) \frac{1}{\hat{\beta}} \Big(
e_+ \| e^{ \frac{\bar{\beta}}{4} \big( \sqrt{(m_+ c)^2 + |p|^2} + \frac{3}{2 c} m_+ g x_3 \big) } (F^{\ell}_{+} - F^{\ell-1}_{+}) \|_{L^{\infty} (\O \times \R^3)}
\\& \qquad \qquad \qquad \qquad + e_- \| e^{ \frac{\bar{\beta}}{4} \big( \sqrt{(m_- c)^2 + |p|^2} + \frac{3}{2 c} m_- g x_3 \big) } (F^{\ell}_{-} - F^{\ell-1}_{-}) \|_{L^{\infty} (\O \times \R^3)}  \Big).
\end{split}
\Ee
Using \eqref{bound3:diff_h^l+1_dy}, together with \eqref{est2:phi_12_x^l+1_dy}, we get
\be \label{bound4:diff_h^l+1_dy}
\begin{split}
& e^{ \frac{3 \bar \beta}{4} \big( \sqrt{(m_+ c)^2 + |p|^2} + \frac{1}{2 c} m_+ g x_3 \big) } |F^{\ell+1}_{+} (t, x, p) - F^{\ell}_{+} (t, x, p)|
\\& \ \ \ \ + e^{ \frac{3 \bar \beta}{4} \big( \sqrt{(m_- c)^2 + |p|^2} + \frac{1}{2 c} m_- g x_3 \big) } |F^{\ell+1}_{-} (t, x, p) - F^{\ell}_{-} (t, x, p)|
\\& \leq \mathfrak{C} (1 + \frac{2 c}{\beta^\prime g} ) \frac{1}{\hat{\beta}} \Big(
e_+ \| e^{ \frac{\bar{\beta}}{4} \big( \sqrt{(m_+ c)^2 + |p|^2} + \frac{3}{2 c} m_+ g x_3 \big) } (F^{\ell}_{+} - F^{\ell-1}_{+}) \|_{L^{\infty} (\O \times \R^3)}
\\& \qquad \qquad \qquad \qquad + e_- \| e^{ \frac{\bar{\beta}}{4} \big( \sqrt{(m_- c)^2 + |p|^2} + \frac{3}{2 c} m_- g x_3 \big) } (F^{\ell}_{-} - F^{\ell-1}_{-}) \|_{L^{\infty} (\O \times \R^3)}  \Big)
\\& \ \ \ \ \times \Big[ \frac{48 e_+}{m g \bar \beta} \| \w^{\ell+1}_{+, \bar{\beta} } \nabla_p F^{\ell}_{+} \|_\infty + \frac{48 e_-}{m g \bar \beta} \| \w^{\ell+1}_{-, \bar{\beta} } \nabla_p F^{\ell}_{-} \|_\infty \Big].
\end{split}
\ee
Together with $\beta'$ in \eqref{def:beta'^l+1_dy} and $\hat{\beta} = \frac{\beta'}{\min \{ m_{-},  m_{+} \}}$, we get
\be \label{bound5:diff_h^l+1_dy}
\begin{split}
\eqref{bound4:diff_h^l+1_dy}
& \leq \mathfrak{C} (1 + \frac{2 c}{\beta^\prime g} ) \frac{ \max\{ e_+, e_- \} }{\hat{\beta}} \Big[ \frac{48 e_+}{m g \bar \beta} \| \w^{\ell+1}_{+, \bar{\beta} } \nabla_p F^{\ell}_{+} \|_\infty + \frac{48 e_-}{m g \bar \beta} \| \w^{\ell+1}_{-, \bar{\beta} } \nabla_p F^{\ell}_{-} \|_\infty \Big]
\\& \ \ \ \ \times \Big( \| e^{ \frac{3 \bar{\beta}}{4} \big( \sqrt{(m_+ c)^2 + |p|^2} + \frac{1}{2 c} m_+ g x_3 \big) } (F^{\ell}_{+} - F^{\ell-1}_{+} ) \|_{L^{\infty} (\O \times \R^3)} 
\\& \qquad \qquad \qquad + \| e^{ \frac{3 \bar{\beta}}{4} \big( \sqrt{(m_- c)^2 + |p|^2} + \frac{1}{2 c} m_- g x_3 \big) } (F^{\ell}_{-} - F^{\ell-1}_{-}) \|_{L^{\infty} (\O \times \R^3)}  \Big)
\\& \lesssim_{\bar \beta, \beta'} \frac{1}{ \tilde{\beta} }
\big\{ \| \w^{\ell+1}_{+, \bar{\beta} } \nabla_p F^{\ell}_{+} \|_\infty + \| \w^{\ell+1}_{-, \bar{\beta} } \nabla_p F^{\ell}_{-} \|_\infty \big\}
\\& \qquad \qquad \times \Big( \| e^{ \frac{3 \bar{\beta}}{4} \big( \sqrt{(m_+ c)^2 + |p|^2} + \frac{1}{2 c} m_+ g x_3 \big) } (F^{\ell}_{+} - F^{\ell-1}_{+} ) \|_{L^{\infty} (\O \times \R^3)} 
\\& \qquad \qquad \qquad + \| e^{ \frac{3 \bar{\beta}}{4} \big( \sqrt{(m_- c)^2 + |p|^2} + \frac{1}{2 c} m_- g x_3 \big) } (F^{\ell}_{-} - F^{\ell-1}_{-} ) \|_{L^{\infty} (\O \times \R^3)}  \Big).
\end{split}
\ee
Under the condition \eqref{condition2:F0_G_dy}, together with \eqref{est1:h_v^l+1_dy}, we derive 
\be \notag
\begin{split}
& \w^{\ell+1}_{\pm, \bar{\beta}} (t,x,p)  | \nabla_p F^{\ell}_{\pm} (t,x,p)| 
\\& \lesssim 4 e^{ \frac{m_{\pm} g}{24} \tilde{\beta} } 
\| \w_{\tilde \beta, 0}  \nabla_{x,p} F_0  \|_{L^\infty (\O \times \R^3)}
+  \frac{m_{\pm} g}{24} \tilde{\beta} \| e^{\tilde{\beta} |p^0|} \nabla_{x_\parallel,p} G \|_{L^\infty (\gamma_-)}
\leq \frac{1}{4}.
\end{split}
\ee 
Thus, we bound \eqref{bound5:diff_h^l+1_dy} by
\be \label{bound6:diff_h^l+1_dy}
\begin{split}
\eqref{bound5:diff_h^l+1_dy} 
& \leq \frac{1}{2} \Big( \| e^{ \frac{3 \bar{\beta}}{4} \big( \sqrt{(m_+ c)^2 + |p|^2} + \frac{1}{2 c} m_+ g x_3 \big) } (F^{\ell}_{+} (t, x, p) - F^{\ell-1}_{+}) (t, x, p) \|_{L^{\infty} (\O \times \R^3)} 
\\& \qquad \ \ + \| e^{ \frac{3 \bar{\beta}}{4} \big( \sqrt{(m_- c)^2 + |p|^2} + \frac{1}{2 c} m_- g x_3 \big) } (F^{\ell}_{-} (t, x, p) - F^{\ell-1}_{-} (t, x, p) ) \|_{L^{\infty} (\O \times \R^3)}  \Big).
\end{split}
\ee
Finally, together with \eqref{bound4:diff_h^l+1_dy}-\eqref{bound6:diff_h^l+1_dy}, we derive, for any $0 \leq t < \infty$,
\Be \notag
\begin{split}
& \| e^{ \frac{3 \bar{\beta}}{4} \big( \sqrt{(m_+ c)^2 + |p|^2} + \frac{1}{2 c} m_+ g x_3 \big) } (F^{\ell+1}_{+} (t, x, p) - F^{\ell}_{+} (t, x, p) ) \|_{L^{\infty} (\O \times \R^3)} 
\\& \qquad \qquad + \| e^{ \frac{3 \bar{\beta}}{4} \big( \sqrt{(m_- c)^2 + |p|^2} + \frac{1}{2 c} m_- g x_3 \big) } (F^{\ell+1}_{-} (t, x, p) - F^{\ell}_{-} (t, x, p) ) \|_{L^{\infty} (\O \times \R^3)}
\\& \leq \frac{1}{2} \Big( \| e^{ \frac{3 \bar{\beta}}{4} \big( \sqrt{(m_+ c)^2 + |p|^2} + \frac{1}{2 c} m_+ g x_3 \big) } (F^{\ell}_{+} (t, x, p) - F^{\ell-1}_{+} (t, x, p) ) \|_{L^{\infty} (\O \times \R^3)} 
\\& \qquad \qquad + \| e^{ \frac{3 \bar{\beta}}{4} \big( \sqrt{(m_- c)^2 + |p|^2} + \frac{1}{2 c} m_- g x_3 \big) } (F^{\ell}_{-} (t, x, p) - F^{\ell-1}_{-} (t, x, p) ) \|_{L^{\infty} (\O \times \R^3)}  \Big),
\end{split}
\Ee
and conclude \eqref{est:h_cauchy_dy}. 

\smallskip

Next, using the condition \eqref{condition:beta} and \eqref{Uest:wh^k} in Proposition \ref{prop:Unif_steady}, then for any $k \geq 0$,
\be
\| w^{k+1}_{\pm, \beta} h_{\pm}^{k+1} \|_{L^\infty (\bar \O \times \R^3)} 
\leq \| e^{ \beta \sqrt{(m_{\pm} c)^2 + |p|^2}} G_{\pm} \|_{L^\infty (\gamma_-)} < \infty.
\ee
Since $\bar{\beta} = \frac{\tilde \beta}{6}$ and $0 < \tilde \beta \leq \beta$, we obtain that for any $k \geq 2$,
\be \label{bound_weight_h_k-k-1}
\begin{split}
& \| e^{ \frac{3 \bar{\beta}}{4} \big( \sqrt{(m_{\pm} c)^2 + |p|^2} + \frac{1}{2 c} m_{\pm} g x_3 \big) } (h^{k}_{{\pm}} - h^{k-1}_{{\pm}} ) \|_{L^{\infty} (\O \times \R^3)} 
\leq 2 \| w^{k+1}_{\pm, \beta} h_{\pm}^{k+1}   \|_{L^\infty (\bar \O \times \R^3)} < \infty.
\end{split}
\ee
Consider $M$ in \eqref{set:M}, from \eqref{Uest:wfell} in Proposition \ref{prop:DC}, similarly we have for any $k \geq 2$,
\be \label{bound_weight_f_k-k-1}
\sup_{0 \leq t < \infty} \| e^{ \frac{3 \bar{\beta}}{4} \big( \sqrt{(m_{\pm} c)^2 + |p|^2} + \frac{1}{2 c} m_{\pm} g x_3 \big) } (f^{k}_{{\pm}} - f^{k-1}_{{\pm}} ) \|_{L^{\infty} (\O \times \R^3)} 
\leq 2 M < \infty.
\ee
Combining \eqref{bound_weight_h_k-k-1} with \eqref{bound_weight_f_k-k-1}, we derive that for any $k \geq 2$,
\be \label{bound_weight_F_k-k-1}
\sup_{0 \leq t < \infty} \| e^{ \frac{3 \bar{\beta}}{4} \big( \sqrt{(m_{\pm} c)^2 + |p|^2} + \frac{1}{2 c} m_{\pm} g x_3 \big) } (F^{k}_{{\pm}} - F^{k-1}_{{\pm}} ) \|_{L^{\infty} (\O \times \R^3)} 
\leq 2 M < \infty.
\ee
Using \eqref{bound_weight_F_k-k-1}, together with \eqref{est:h_cauchy_dy}, we derive for any $k \geq 0$,
\Be \label{bound7:diff_h^l+1_dy}
\begin{split}
& \sup_{0 \leq t < \infty} \Big( \| e^{ \frac{3 \bar{\beta}}{4} \big( \sqrt{(m_+ c)^2 + |p|^2} + \frac{1}{2 c} m_+ g x_3 \big) } (F^{k+1}_{+} (t, x, p) - F^{k}_{+} (t,x, p)) \|_{L^{\infty} (\O \times \R^3)} 
\\& \qquad \qquad + \| e^{ \frac{3 \bar{\beta}}{4} \big( \sqrt{(m_- c)^2 + |p|^2} + \frac{1}{2 c} m_- g x_3 \big) } (F^{k+1}_{-} (t, x, p) - F^{k}_{-} (t, x, p) ) \|_{L^{\infty} (\O \times \R^3)} \Big)
\\& \leq \frac{1}{2^{k-1}} \sup_{0 \leq t < \infty} \Big( \| e^{ \frac{3 \bar{\beta}}{4} \big( \sqrt{(m_+ c)^2 + |p|^2} + \frac{1}{2 c} m_+ g x_3 \big) } (F^{2}_{+} (t, x, p) - F^{1}_{+} (t, x, p) ) \|_{L^{\infty} (\O \times \R^3)} 
\\& \qquad \qquad \qquad \ \ + \| e^{ \frac{3 \bar{\beta}}{4} \big( \sqrt{(m_- c)^2 + |p|^2} + \frac{1}{2 c} m_- g x_3 \big) } (F^{2}_{-} (t, x, p) - F^{1}_{-} (t, x, p) ) \|_{L^{\infty} (\O \times \R^3)}  \Big)
\\& \leq \frac{1}{2^{k-1}} 2M.
\end{split}
\Ee
Hence, we conclude that $\{ F^{\ell+1} \}^{\infty}_{\ell=0}$ forms a Cauchy sequences in $L^{\infty} (\O \times \R^3)$.

\smallskip

Now inputting \eqref{bound7:diff_h^l+1_dy} into \eqref{est:e^beta*rho_12^l+1_dy}, we deduce for any $k \geq 0$,
\Be \label{est:cauchy_rho_dy}
\begin{split}
& e^{\beta' \frac{g}{2c} x_3} | \varrho^{k} (t, x) - \varrho^{k-1} (t, x) |
\\& \leq e_+ \| e^{ \frac{\bar{\beta}}{4} \big( \sqrt{(m_+ c)^2 + |p|^2} + \frac{3}{2 c} m_+ g x_3 \big) }  (F^{k}_{+} - F^{k-1}_{+} ) \|_{L^{\infty} (\O \times \R^3)} \times \frac{1}{\hat{\beta}}
\\& \ \ \ \ + e_- \| e^{ \frac{\bar{\beta}}{4} \big( \sqrt{(m_- c)^2 + |p|^2} + \frac{3}{2 c} m_- g x_3 \big) } (F^{k}_{-} - F^{k-1}_{-}) \|_{L^{\infty} (\O \times \R^3)} \times \frac{1}{\hat{\beta}}
\\& \leq \frac{2 M}{2^{k-1}} \times \frac{e_+ + e_-}{\hat{\beta}}.
\end{split}
\Ee
Similarly, we input \eqref{bound7:diff_h^l+1_dy} into \eqref{est2:phi_12_x^l+1_dy}, and obtain for any $k \geq 0$,
\Be \label{est:cauchy_phi_x_dy}
\begin{split}
& \| \nabla_x \Psi^{k} (t, x) - \nabla_x \Psi^{k-1} (t, x) \|_{L^\infty (\O)} 
\\& \leq \mathfrak{C} (1 + \frac{2 c}{\beta^\prime g} ) \frac{1}{\hat{\beta}} \Big(
e_+ \| e^{ \frac{\bar{\beta}}{4} \big( \sqrt{(m_+ c)^2 + |p|^2} + \frac{3}{2 c} m_+ g x_3 \big) } (F^{k}_{+} - F^{k-1}_{+}) \|_{L^{\infty} (\O \times \R^3)}
\\& \qquad \qquad \qquad \qquad + e_- \| e^{ \frac{\bar{\beta}}{4} \big( \sqrt{(m_- c)^2 + |p|^2} + \frac{3}{2 c} m_- g x_3 \big) } (F^{k}_{-} - F^{k-1}_{-}) \|_{L^{\infty} (\O \times \R^3)}  \Big)
\\& \leq \frac{2 M}{2^{k-1}}  \mathfrak{C} (1 + \frac{2 c}{\beta^\prime g} ) \frac{e_+ + e_-}{\hat{\beta}}.
\end{split}
\Ee
Thus, we deduce that $\{ \varrho^\ell \}^{\infty}_{\ell=0}$ and $\{ \nabla_x \Psi^\ell \}^{\infty}_{\ell=0}$ are both Cauchy sequences in $L^{\infty} (R_+ \times \O)$.
Since $\phi_{F^\ell} = \Phi (x) + \Psi^\ell$ with $ \Phi (x) \in L^{\infty} (\O)$ for any $\ell \geq 0$, we conclude that $\{ \nabla_x \phi_{F^\ell} \}_{\ell=0}$ is also a Cauchy sequence in $L^{\infty} (R_+ \times \O)$.
\end{proof}

Finally, using the Cauchy sequences of $L^\infty$-spaces in Proposition \ref{prop:cauchy_dy}, we construct a weak solution of $(F_{\pm}, \phi_F)$ solving \eqref{VP_F}, \eqref{Poisson_F}, \eqref{VP_0},  \eqref{Dbc:F} and \eqref{bdry:F}.

\begin{proof}[\textbf{Proof of Theorem \ref{theo:CD}}]

Besides the assumption in Theorem \ref{theo:CS}, we further assume \eqref{choice:g} which implies the condition \eqref{condition:Dvh} in Theorem \ref{theo:AS} and \eqref{condition:ML} in Theorem \ref{theo:RD}.
Hence, we can apply Theorems \ref{theo:AS} and \ref{theo:RD} in the following proof.

\smallskip

\textbf{Step 1. Regularity: Proof of \eqref{Uest:wh_dy}-\eqref{Uest:DxPsi} and \eqref{Uest:Dxphi_F}-\eqref{Uest:D2xD3tphi_F}.}
From \eqref{est:h_cauchy_dy}, \eqref{est:cauchy_rho_dy} and \eqref{est:cauchy_phi_x_dy}, the arguments of the Cauchy sequences of $L^\infty$-spaces in Proposition \ref{prop:cauchy_dy}, there exists 
\be \notag
F_{\pm} (t, x, p) \in L^\infty (\R_+ \times \bar \O \times \R^3)
\ \text{ and } \ \varrho(t, x), \nabla_x \phi_F (t, x) \in L^\infty (\R_+ \times \bar \O) 
\ \text{ with } \ 
\phi_F = 0 \ \ \text{on} \ \p\O,
\ee
such that as $k \to \infty$,
\be \label{weakconv_whst_dy}
e^{ \frac{3 \bar{\beta}}{4} ( \sqrt{(m_{\pm} c)^2 + |p|^2} + \frac{1}{2 c} m_{\pm} g x_3) } F^{k}_{\pm}
\to e^{ \frac{3 \bar{\beta}}{4} ( \sqrt{(m_{\pm} c)^2 + |p|^2} + \frac{1}{2 c} m_{\pm} g x_3 ) } F_{\pm}
\ \text{ in } \ L^\infty (\R_+ \times \bar \O \times \R^3),
\ee
and
\begin{align}
e^{\beta' \frac{g}{2c} x_3} \varrho^{k} \to e^{\beta' \frac{g}{2c} x_3} \varrho 
& \ \text{ in } \ L^\infty (\R_+ \times \bar \O) \ \text{ as } \ k \to \infty,
\label{weakconv_rhost_dy} \\
\nabla_x \phi_{F^k} \to \nabla_x \phi_F
& \ \text{ in } \ L^\infty (\R_+ \times \bar \O) \ \text{ as } \ k \to \infty,
\label{ae_converge_par_Phist_dy}
\end{align}
where $\bar{\beta} = \frac{\tilde \beta}{2}$ and $\beta' =  \min\{ \frac{\bar \beta}{4}, \beta \} \times \min \{ m_{-},  m_{+} \}$. This shows that
\be \label{strong_conv_h_rho_dy}
\begin{split}
F^{k}_{\pm} \to F_{\pm} 
&\ \text{ in } \ L^\infty (R_+ \times \bar \O \times \R^3) 
\ \text{ as } \ k \to \infty,
\\ \varrho^{k} \to \varrho 
& \ \text{ in } \ L^\infty (R_+ \times \bar \O)
\ \text{ as } \ k \to \infty.
\end{split}
\ee
Using \eqref{Bootstrap_ell_2} in Proposition \ref{prop:DC}, together with \eqref{eqtn:phiFell} and the $L^\infty$ convergence in \eqref{ae_converge_par_Phist_dy}, we prove that $\nabla_x \phi_F$ and $\nabla_x \Psi$ satisfy \eqref{Uest:Dxphi_F} and \eqref{Uest:DxPsi} respectively. 

Furthermore, the $L^{\infty}$ convergence $\varrho^k \to \varrho$ in \eqref{strong_conv_h_rho_dy} also implies
\be \notag
| \varrho |_{C^{0,\delta}(\O)} \leq \sup_{k \in \N} |\varrho^{k+1} |_{C^{0,\delta}(\O)}.
\ee
Together with \eqref{Uest:varrhofell}, \eqref{Uest:DDPhi^l_dy} and \eqref{est:nabla^2phi}, we deduce that $\| \nabla_x ^2 \phi_F \|_\infty $ satisfies \eqref{est:D2xphi_F}.
From Lemma \ref{lem:D3tphi_F}, this implies $| \p_{x_3} \p_t \phi_{F} (t,x)| $ satisfies \eqref{est:D3tphi_F}.
Therefore,  we conclude \eqref{Uest:D2xD3tphi_F}.

Similarly, from \eqref{form:Fell} and $F^k_{\pm} = h_{\pm} + f^k_{\pm}$ for any $k \geq 1$ in \eqref{def:Fell}, the $L^\infty$ convergence in \eqref{strong_conv_h_rho_dy} implies that
\be \label{strong_conv_f_rho_dy}
f^{k}_{\pm} \to f_{\pm} 
\ \text{ in } \ L^\infty (R_+ \times \bar \O \times \R^3) 
\ \text{ as } \ k \to \infty.
\ee
Using \eqref{Uest:wh}, together with \eqref{Uest:wfell}, we obtain \eqref{Uest:wh_dy}.

\smallskip

\textbf{Step 2. Existence.}
We omit the proof of showing $(F_{\pm}, \phi_F)$ obtained in 
\eqref{weakconv_whst_dy}-\eqref{ae_converge_par_Phist_dy} is the solution to \eqref{VP_F}, \eqref{Poisson_F}, \eqref{VP_0},  \eqref{Dbc:F} and \eqref{bdry:F} in the sense of Definition \ref{weak_sol_dy}, since it follows the same weak convergence argument of the proof of Theorem \ref{theo:CS}.

\smallskip

\textbf{Step 3. Regularity: Proof of \eqref{Uest_final:F_v:dyn}-\eqref{Uest_final:F_x:dyn}.}
Since $f_{\pm} = F_{\pm} - h_{\pm}$, from \eqref{est_final:hk_v} and \eqref{est_final:hk_x} in Theorem \ref{theo:CS}, together with \eqref{choice:g} and \eqref{Uest:D2xD3tphi_F}, \eqref{Uest:D^-1_Db} and Theorem \ref{theo:RD}, we conclude $f_{\pm}$ satisfies \eqref{Uest_final:F_v:dyn} and \eqref{Uest_final:F_x:dyn}.

\smallskip

\textbf{Step 4. Stability and Uniqueness: Proof of \eqref{Udecay:f}-\eqref{Udecay:DxPsi}.}
Finally, under the assumption \eqref{choice:g}, using Theorem \ref{theo:AS}, together with \eqref{Uest:wfell} and the $L^{\infty}$ convergence $f^k \to f$ in \eqref{strong_conv_f_rho_dy}, we prove that $(f(t), \varrho(t))$ satisfies \eqref{Udecay:f} and \eqref{Udecay:varrho}.
Applying \eqref{Udecay:varrho} in \eqref{est:nabla_phi} of Lemma \ref{lem:rho_to_phi}, we obtain \eqref{Udecay:DxPsi}.
On the other hand, from \eqref{Uest:wh_dy}, \eqref{Uest_final:F_v:dyn} and \eqref{Uest:D^-1_Db} in Theorem \ref{theo:AS}, we apply Theorem \ref{theo:UA}, and thus conclude the uniqueness of the solution $(F, \phi_F)$.
\end{proof}


\section{Asymptotic Behavior of Solutions as \texorpdfstring{$|x_\parallel | \rightarrow \infty$}{x->infinity}}
\label{sec:asymptotic}

Recall from Theorem \ref{theo:CS} that we established the existence and uniqueness of a stationary solution under the inflow boundary condition. Specifically, there exists a unique solution $(h_{\pm}, \rho, \Phi)$ to \eqref{VP_h}-\eqref{eqtn:Dphi} in the sense of Definition \ref{weak_sol}.
We now do some asymptotic analysis on this solution, beginning by introducing two inflow boundary conditions: for any $(x, p) \in \gamma_-$,
\be \notag
\begin{split}
\text{Isothermal J\"uttner distribution: }
& \frac{1}{ |e_{\pm}| m^2_{\pm} } e^{ - \frac{1}{2} \sqrt{(m_{\pm} c)^2 + |p|^2} },
\\
\text{Non-isothermal J\"uttner distribution: }
& \frac{1}{ |e_{\pm}| m^2_{\pm} T_{\pm} (x) } e^{- \frac{1}{2 T_{\pm} (x)} \sqrt{(m_{\pm} c)^2 + |p|^2} },
\end{split}
\ee
where $T_\pm (x)$ represent the non-isothermal wall temperature.

Before proving the main result, we outline the idea of the proof. Assume that the non-isothermal wall temperature $T_\pm(x)$ approaches a unit temperature as $|x_{\|}| \to \infty$.
Theorem \ref{theo:CS} establishes the existence of two solutions, $h_{I, \pm}$ and $h_{\pm}$, corresponding to the isothermal and non-isothermal J\"uttner distributions, respectively. Note that both solutions satisfy \eqref{Uest:wh}-\eqref{Uest:DPhi}.

Consider the difference $\mathrm{d}_{\pm} (x,p) = h_{I, \pm} (x,p) - h_{\pm} (x,p)$. 
Let $(X_{\pm} (s;x,p), P_{\pm} (s;x,p))$ denote the characteristics in the non-isothermal case, we have
\Be \notag
\begin{split}
\mathrm{d}_{\pm} (x,p) 
& = h_{\pm} (\xbp, \pbp) - h_{I, \pm} (\xbp, \pbp)
\\& \ \ \ \ + \int^0_{-\tbp} e \nabla \phi (X_{\pm} (s;x,p)) \cdot \nabla_p h_{I, \pm} (X_{\pm} (s;x,p), P_{\pm} (s;x,p)) \dd s.
\end{split}
\Ee
Next, we evaluate the weight function $\langle x_\parallel \rangle^3 | \hat{w}_{\pm} \mathrm{d}_{\pm} (x,p) |$ defined in \eqref{def:hat_w_shock}.
Following the steps in the proof of Theorem \ref{theo:US} and leveraging the regularity of $\nabla_p h_{I, \pm}$, we establish a bound for the weight function and conclude the asymptotic behavior of solutions as $|x_\parallel| \to \infty$.

\begin{theorem}
\label{thm:asymptotic_behavior}

Suppose that $(h_{I, \pm}, \rho_I, \Phi_I)$ and $(h_{\pm}, \rho, \Phi)$ are the unique solutions to \eqref{VP_h}-\eqref{eqtn:Dphi} in the sense of Definition \ref{weak_sol} under isothermal and non-isothermal J\"uttner distribution, respectively.
Assume that they both satisfy \eqref{Uest:wh}-\eqref{Uest:DPhi} and 
\be \label{condition:T_pm}
|T_{\pm} (x) - 1| \leq \frac{1}{(20 + |x|)^{4}}
\ \text{ with } \
x \in \p\O.
\ee
Further, we assume that there exists a constant $\beta'$, such that
\be \notag
2 (e_+ + e_-) < (\beta')^3 \leq \frac{1}{128}.
\ee
Then 
\Be
\begin{split}
| \nabla (\Phi - \Phi_I ) | 
& \lesssim \langle x  \rangle^{-3},
\\
\sum\limits_{i = \pm}
\| \langle x_\parallel \rangle^3 w_{i, \beta'} ( h_{I, i} - h_i ) \|_{L^{\infty} (\O \times \R^3)} 
& \lesssim 1,
\end{split}
\Ee
where $\langle x_\parallel \rangle = \sqrt{1+ |x_\parallel|^2}$ and 
$w_{\pm, \beta'} (x,p)$ defined in \eqref{w^h}.
\end{theorem}

\begin{proof}

In the proof, we adopt the abuse of notation about $\pm$ in \eqref{abuse}.

\textbf{Step 1.}
From the assumption, $(h_{I, \pm}, \rho_I, \Phi_I)$ solves the following VP system with an isothermal boundary condition:
\Be \label{VP_h_iso}
\begin{split}
v_\pm \cdot \nabla_x h_{I, \pm}
+ \big( e_{\pm} ( \frac{v_\pm}{c} \times B 
- \nabla_x \Phi_I ) - \nabla_x (m_{\pm} g x_3) \big) \cdot \nabla_p h_{I, \pm} = 0,
\\
\Delta \Phi_I (x) = \rho_I (x) = \int_{\R^3} ( e_+ h_{I, +} + e_{-} h_{I, -} ) \dd p
\ \text{  with  } \ 
\Phi_I |_{\p\O} = 0,
\\
h_{I, \pm} (x,p) |_{\gamma_-} 
= \frac{1}{ |e_{\pm}| m^2_{\pm} } e^{ - \frac{1}{2} \sqrt{(m_{\pm} c)^2 + |p|^2} }.
\end{split}
\Ee
Simultaneously, $(h_{\pm}, \rho, \Phi)$ solves the following VP system with a non-isothermal boundary condition:
\Be \label{VP_h_non}
\begin{split}
v_\pm \cdot \nabla_x h_{\pm}
+ \big( e_{\pm} ( \frac{v_\pm}{c} \times B 
- \nabla_x \Phi ) - \nabla_x (m_{\pm} g x_3) \big) \cdot \nabla_p h_{\pm} = 0,
\\
\Delta \Phi (x) = \rho (x) = \int_{\R^3} ( e_+ h_{+} + e_{-} h_{-} ) \dd p
\ \text{ with } \
\Phi |_{\p\O} = 0,
\\ 
h_{\pm} (x,p) |_{\gamma_-} 
= \frac{1}{ e_{\pm} m^2_{\pm} T_{\pm} (x) } e^{- \frac{1}{2 T_{\pm} (x)} \sqrt{(m_{\pm} c)^2 + |p|^2} }.
\end{split}
\Ee
Define two functions $\mathrm{d}_{\pm} (x, p)$ and $\phi (x)$ as
\Be
\begin{split}
\mathrm{d}_{\pm} (x,p) 
& := h_{\pm} (x, p) - h_{I, \pm} (x,p)
\ \ \text{in} \ \O \times \R^3,
\\ \phi (x) 
& := \Phi (x) - \Phi_I (x)
\ \ \text{in} \ \O.
\end{split}
\Ee
Under direct computation, $(\mathrm{d}_{\pm}, \phi)$ solves the following system:  
\Be \label{VP_diff_F}
\begin{split}
v_\pm \cdot \nabla_x \mathrm{d}_{\pm}
+ \big( e_{\pm} ( \frac{v_\pm}{c} \times B 
- \nabla_x \Phi ) - \nabla_x (m_{\pm} g x_3) \big) \cdot \nabla_p \mathrm{d}_{\pm} = e_{\pm} \nabla \phi \cdot \nabla_p h_{I, \pm},
\\ \Delta \phi = \int_{\R^3} ( e_+ \mathrm{d}_{+} + e_{-} \mathrm{d}_{-} ) \dd p
\ \text{ with } \
\phi |_{\p\O} = 0, 
\\ \mathrm{d}_{\pm} (x,p)|_{\gamma_-} 
= \frac{1}{ e_{\pm} m^2_{\pm} T_{\pm} (x) } e^{- \frac{1}{2 T_{\pm} (x)} \sqrt{(m_{\pm} c)^2 + |p|^2} } - \frac{1}{ e_{\pm} m^2_{\pm} } e^{ - \frac{1}{2} \sqrt{(m_{\pm} c)^2 + |p|^2} }. 
\end{split}
\Ee 
Consider the characteristics $Z_{\pm} (s;x,p) = (X_{\pm} (s;x,p), P_{\pm} (s;x,p))$ for \eqref{VP_diff_F}: 
\Be \label{char:diff_F}
\begin{split} 
\frac{d X_{\pm} (s;x,p) }{d s} 
& = V_{\pm} (s;x,p) = \frac{P_{\pm} (s;x,p)}{\sqrt{m^2_{\pm} + |P_{\pm} (s;x,p)|^2 / c^2}}, \\
\frac{d P_{\pm} (s;x,p) }{d s} 
& =  {e_{\pm}} ( V_{\pm} (s;x,p) \times \frac{B}{c} - \nabla_x \Phi (X_{\pm} (s;x,p) )) - {m_{\pm}} g \mathbf{e}_3,
\end{split}
\Ee
where $\Phi$ solves \eqref{VP_h_non}. At $s = 0$, we have
\[
Z_{\pm} (0; x, p) = (X_{\pm} (0; x, p), P_{\pm} (0; x,p)) = (x_{\pm}, p_{\pm}) = z_{\pm}.
\]
Along the characteristics \eqref{char:diff_F}, we further consider the backward exit time $\tbp$ and backward exit position and momentum as follows:
\[
(\xbp, \pbp ) = (X_{\pm} (-\tbp; x, p), P_{\pm} (-\tbp; x,p)).
\]
For any $(x, p) \in \O \times \R^3$, we have
\Be \label{eq:delta}
\begin{split}
\mathrm{d} (x,p) 
& = h (\xb, \pb) - h_{I} (\xb, \pb)
\\& \ \ \ \ + \int^0_{-\tb(x,p)} e \nabla \phi (X(s;x,p)) \cdot \nabla_p h_{I} (X(s;x,p), P(s;x,p)) \dd s.
\end{split}
\Ee
Following Lemma \ref{lem:conservation_law}, we get for all $s \in [-\tb (x,p), 0]$,
\Be \label{char:phi}
\begin{split}
& \sqrt{(m c)^2 + |P (s;x,p)|^2} + \frac{1}{c} \big( e \Phi (X (s;x,p)) + m g X_{3} (s;x,p) \big)
\\& = \sqrt{(m c)^2 + |\pb (x,p)|^2}. 
\end{split}
\Ee  
From \eqref{Uest:DPhi} in Theorem \ref{theo:CS}, we have 
\Be \label{condition:phi}
\| \nabla_x \Phi \|_{L^\infty (\bar{\O})} 
\leq \min \big( \frac{m_{+}}{e_{+}}, \frac{m_{-}}{e_{-}} \big) \times \frac{g}{2}. 
\Ee 
Using \eqref{est:tb^h} in Lemma \ref{lem:tb}, we have
\Be \label{est:tb_shock}
\tb (x,p) 
\leq \frac{4}{m g} \big( \sqrt{(m c)^2 + |p|^2} + \frac{3}{2c} m g x_{3} \big),
\Ee
and
\Be \label{est2:tb_shock}
\tb(x,p) 
\leq \frac{4}{m g} |\pb (x, p)|.
\Ee

\smallskip

\textbf{Step 2.}
Now we introduce the following weight: 
\Be \label{def:hat_w_shock}
\hat{w}_{\pm} (x,p) 
:= w_{\pm, \beta'} (x,p)
= e^{ \beta' \big( \sqrt{(m_{\pm} c)^2 + |p|^2} + \frac{1}{c} ( e_{\pm} \Phi (x) + m_{\pm} g x_3 ) \big) }.
\Ee
From \eqref{char:phi}, we have
\be \label{est:hat_w_conservation}
\hat{w}_{\pm} (x,p) = e^{ \beta' \sqrt{(m_{\pm} c)^2 + |\pbp (x,p)|^2} }.
\ee
Together with \eqref{eq:delta}, we derive a bound 
\Be \label{est:x^awh}
\begin{split}
& \langle x_\parallel \rangle^3 | \hat{w} \mathrm{d} (x,p) | 
\\& \leq \langle x_\parallel \rangle^3 e^{ \beta' \sqrt{(m c)^2 + |\pb (x,p)|^2} } | h (\xb, \pb) - h_I (\xb, \pb) |
\\& \ \ \ \ + 
\underbrace{
\langle x_\parallel \rangle^3 e^{ \beta' \left( \sqrt{(m c)^2 + |p|^2} + \frac{1}{c} ( e \Phi (x) + m g x_3 ) \right) } \int^0_{-\tb(x,p)} e \nabla \phi (X) \cdot \nabla_p h_{I} (X, P) \dd s
}_{\eqref{est:x^awh}^*}. 
\end{split}
\Ee
From \eqref{theo:hk_v}, we compute
\be \label{est2:x^awh}
\begin{split}
\eqref{est:x^awh}^*
& \lesssim \langle x_\parallel \rangle^3 e^{ \beta' \left( \sqrt{(m c)^2 + |p|^2} + \frac{1}{c} ( e \Phi (x) + m g x_3 ) \right) } 
\\& \ \ \ \ \times \int^0_{-\tb(x,p)}  |\nabla \phi (X(s;x,p))|  
e^{ - \frac{\tilde \beta}{2} \sqrt{(m c)^2 + | P (s:x,p) |^2}} e^{- \frac{\tilde \beta m g}{4 c} X_3 (s:x,p)} \dd s.
\end{split}
\ee
Suppose that $4 \beta' \leq \tilde \beta = 1$.
Together with \eqref{char:phi} and \eqref{condition:phi}, we bound
\be \label{est3:x^awh}
\begin{split}
\eqref{est2:x^awh}
& \lesssim \langle x_\parallel \rangle^3 \int^0_{-\tb(x,p)}  |\nabla \phi (X(s;x,p))|  
e^{ - \frac{\tilde \beta}{4} \sqrt{(m c)^2 + | P (s:x,p) |^2}} e^{  - \frac{\tilde \beta m g}{8 c} X_3 (s:x,p)} \dd s
\\& \lesssim \langle x_\parallel \rangle^3 \int^0_{-\tb(x,p)}  |\nabla \phi (X(s;x,p))|  
e^{- \frac{1}{4} \big( \sqrt{(m c)^2 + | P (s:x,p) |^2} + \frac{m g}{2 c} X_3(s;x,p) \big) } \dd s.
\end{split}
\ee
Recall from \eqref{GreenF} and Lemma \ref{lem:rho_to_phi}, 
\Be \label{GreenF_shock}
\nabla \phi (x) = \nabla_x \mathfrak{G} * \int_{\R^3} ( e_+ \mathrm{d}_{+} (x, p) + e_{-} \mathrm{d}_{-} (x, p) ) \dd p,
\Ee
and 
\Be \label{est:G_x_shock}
|\nabla_x \mathfrak{G} (x,y)| 
\lesssim  \min \left\{ 
\frac{y_3}{|x-y|^3}, \frac{1}{|x-y|^2}
\right\}.
\Ee
From \eqref{GreenF_shock}, we have
\Be \label{est:phi_x}
\begin{split}
& | \nabla \phi (X(s;x,p)) |
\\& \leq \sum\limits_{i = \pm} e_i \int_{\O} \big| \nabla_x \mathfrak{G} (X(s;x,p), y) \big| \int_{\R^3} \big| \mathrm{d}_i (y, p^\prime ) \big| \dd p^\prime \dd y
\\& \leq \sum\limits_{i = \pm} e_i \int_{\O} \big| \nabla_x \mathfrak{G} (X(s;x,p), y) \big| \int_{\R^3} \frac{1}{ \langle y_\parallel \rangle^3 \hat{w}_i (y,p^\prime) } \langle y_\parallel \rangle^3 | \hat{w}_i \mathrm{d}_i (y, p^\prime ) | \dd p^\prime \dd y 
\\& \leq \sum\limits_{i = \pm} e_i \| \langle x_\parallel \rangle^3 \hat{w}_i \mathrm{d}_i (x, p ) \|_{L^{\infty} (\O \times \R^3)} \int_{\O} \big| \nabla_x \mathfrak{G} (X(s;x,p), y) \big|  \int_{\R^3} \frac{1}{ \langle y_\parallel \rangle^3 \hat{w}_i (y,p^\prime) } \dd p^\prime \dd y
\\& \lesssim \sum\limits_{i = \pm} \frac{e_i}{ (\beta')^3 } \| \langle x_\parallel \rangle^3 \hat{w}_i \mathrm{d}_i (x, p ) \|_{L^{\infty} (\O \times \R^3)} \times \int_{\O} \frac{e^{- \beta' \frac{m_i g}{2 c} y_3}}{\langle y_\parallel \rangle^3 } \big| \nabla_x \mathfrak{G} (X(s;x,p), y) \big| 
 \dd y,
\end{split}
\Ee
where the last inequality follows from \eqref{condition:phi} and the zero Dirichlet boundary.
Inputting \eqref{est:phi_x} into \eqref{est3:x^awh}, we derive that
\be \label{est4:x^awh}
\begin{split}
\eqref{est3:x^awh} 
& \lesssim \langle x_\parallel \rangle^3 \int^0_{-\tb (x,p)} e^{- \frac{1}{4} \big( \sqrt{(m c)^2 + | P (s:x,p) |^2} + \frac{m g}{2 c} X_{3} (s;x,p) \big) } \dd s
\\& \ \ \ \ \times \sum\limits_{i = \pm} 
\Big\{ \frac{e_i}{ (\beta')^3 } \| \langle x_\parallel \rangle^3 \hat{w}_i \mathrm{d}_i (x, p ) \|_{L^{\infty} (\O \times \R^3)}
\int_{\O} \frac{ e^{- \beta' \frac{m_i g}{2 c} y_3} }{\langle y_\parallel \rangle^3 } | \nabla_x \mathfrak{G}(X(s;x,p), y)| \dd y \Big\}.
\end{split}
\ee

\smallskip

\textbf{Step 3.}
Now we claim that for any $s \in [-\tb (x,p), 0]$,
\Be \label{est:x^awh_*}
\int_{\O} \frac{ e^{- \beta' \frac{m g}{2 c} y_3} }{\langle y_\parallel \rangle^3} | \nabla_x \mathfrak{G}(X(s;x,p), y)| \dd y
\lesssim \frac{1}{\langle  X_\parallel (s;x,p) \rangle^3}.
\Ee
To prove the claim, we split the integration on the left-hand side of \eqref{est:x^awh_*} into two parts:
\be \label{est5:x^awh}
\begin{split}
& \int_{\O} \frac{1_{|X(s;x,p) - y| < 1}}{\langle y_\parallel \rangle^3} 
e^{- \beta' \frac{m g}{2 c} y_3} |  \nabla_x \mathfrak{G}(X(s;x,p), y)| \dd y
\\& \ \ \ \ + \int_0^\infty \int_{\R^2} \frac{1_{|X(s;x,p) - y| \geq 1}}{\langle y_\parallel \rangle^3} 
e^{- \beta' \frac{m g}{2 c} y_3} |  \nabla_x \mathfrak{G}(X(s;x,p), y)| \dd y_\parallel \dd y_3
\\& \lesssim 
\underbrace{
\int_{\O} \frac{1_{|X(s;x,p) - y| < 1}}{\langle y_\parallel \rangle^3} 
e^{- \beta' \frac{m g}{2 c} y_3} \frac{1}{|X (s;x,p) - y|^2} \dd y
}_{\eqref{est5:x^awh}_1}
\\& \ \ \ \ + 
\underbrace{
\int_0^\infty \int_{\R^2} \frac{1_{|X(s;x,p) - y| \geq 1}}{\langle y_\parallel \rangle^3} 
e^{- \beta' \frac{m g}{2 c} y_3} \frac{y_3}{|X (s;x,p) - y|^3} \dd y_\parallel \dd y_3
}_{\eqref{est5:x^awh}_2},
\end{split}
\ee
where the last inequality follows from \eqref{est:G_x_shock}.

For $\eqref{est5:x^awh}_1$, assume that $|X(s;x,p) - y| < 1$, we have
\be
|X_{\parallel} (s;x,p) - y_{\parallel}| + |X_3 (s;x,p) - y_3| \lesssim |X(s;x,p) - y| < 1.
\ee
This implies that
\be
| y_{\parallel} | = 
| y_{\parallel} - X_{\parallel} (s;x,p) + X_{\parallel} (s;x,p) |
\geq | X_{\parallel} (s;x,p) | - 1,
\ee
and thus
\be
\langle y_{\parallel} \rangle \gtrsim \langle X_{\parallel} (s;x,p) \rangle.
\ee
Together with the spherical coordinates, we obtain
\be \label{est:est5:x^awh_1}
\begin{split}
& \int_{\O} \frac{1_{|X(s;x,p) - y| < 1}}{\langle y_\parallel \rangle^3} 
e^{- \beta' \frac{m g}{2 c} y_3} \frac{1}{|X (s;x,p) - y|^2} \dd y
\\& \lesssim \frac{1}{\langle  X_\parallel (s;x,p) \rangle^3} \int_{\O} e^{- \beta' \frac{m g}{2 c} y_3} \frac{1_{|X(s;x,p) - y| < 1}}{|X (s;x,p) - y|^2} \dd y 
\\& \lesssim \frac{1}{\langle  X_\parallel (s;x,p) \rangle^3} \int^{1}_{0} e^{- \beta' \frac{m g}{2 c} y_3} \dd r \lesssim \frac{1}{\langle  X_\parallel (s;x,p) \rangle^3}.
\end{split}
\ee
Thus, we deduce that
\be \label{est2:est5:x^awh_1}
\eqref{est5:x^awh}_1 
\lesssim \frac{1}{\langle  X_\parallel (s;x,p) \rangle^3}.
\ee

For $\eqref{est5:x^awh}_2$, we bound
\be \label{est:est5:x^awh_2}
\begin{split}
\eqref{est5:x^awh}_2
& = \int_0^\infty e^{- \beta' \frac{m g}{2 c} y_3} y_3 \int_{\R^2} \frac{1_{|X(s;x,p) - y| \geq 1}}{\langle y_\parallel \rangle^3} 
 \frac{1}{|X (s;x,p) - y|^3} \dd y_\parallel \dd y_3 
\\& \lesssim \int_0^\infty e^{- \beta' \frac{m g}{2 c} y_3} y_3 \int_{\R^2} \frac{1}{\langle y_\parallel \rangle^3} 
\frac{1}{ \langle X_\parallel (s;x,p) - y_\parallel \rangle^3 } \dd y_\parallel \dd y_3, 
\end{split}
\ee
where the second line follows from
\be
|X (s;x,p) - y| \gtrsim \langle X_\parallel (s;x,p) - y_\parallel \rangle
\ \text{ when } \
|X(s;x,p) - y| \geq 1.
\ee
Under direct computation, we have
\be \label{est:<X_12>}
\langle X_\parallel (s;x,p) \rangle
\lesssim \langle X_\parallel (s;x,p) - y_\parallel \rangle + \langle y_\parallel \rangle.
\ee
From \eqref{est:<X_12>}, we get
\be \notag
\begin{split}
\langle y_\parallel \rangle
\langle X_\parallel (s;x,p) - y_\parallel \rangle
& \gtrsim \big( \langle X_\parallel (s;x,p) - y_\parallel \rangle + \langle y_\parallel \rangle \big) \times \min \{\langle y_\parallel \rangle, \langle X_\parallel (s;x,p) - y_\parallel \rangle \}
\\& \gtrsim \langle X_\parallel (s;x,p) \rangle \times \min \{\langle y_\parallel \rangle, \langle X_\parallel (s;x,p) - y_\parallel \rangle \},
\end{split}
\ee
and thus
\be \label{est2:<X_12>}
\begin{split}
\frac{1}{\langle y_\parallel \rangle
\langle X_\parallel (s;x,p) - y_\parallel \rangle}
& \lesssim 
\frac{1}{\langle X_\parallel (s;x,p) \rangle}
\times \frac{1}{\min \{\langle y_\parallel \rangle, \langle X_\parallel (s;x,p) - y_\parallel \rangle \}}
\\& \lesssim 
\frac{1}{\langle X_\parallel (s;x,p) \rangle} \times  
\Big( \frac{1}{\langle y_\parallel \rangle} +
\frac{1}{\langle X_\parallel (s;x,p) - y_\parallel \rangle} \Big).
\end{split}
\ee
Inputting \eqref{est2:<X_12>} into \eqref{est:est5:x^awh_2}, we obtain
\be \label{est2:est5:x^awh_2}
\begin{split}
& \int_0^\infty e^{- \beta' \frac{m g}{2 c} y_3} y_3 \int_{\R^2} \frac{1}{\langle y_\parallel \rangle^3} 
\frac{1}{ \langle X_\parallel (s;x,p) - y_\parallel \rangle^3 } \dd y_\parallel \dd y_3
\\& \lesssim \frac{1}{\langle X_\parallel (s;x,p) \rangle^3} \int_0^\infty e^{- \beta' \frac{m g}{2 c} y_3} y_3 \int_{\R^2} \Big( \frac{1}{\langle y_\parallel \rangle} +
\frac{1}{\langle X_\parallel (s;x,p) - y_\parallel \rangle} \Big)^3 \dd y_\parallel \dd y_3
\\& \lesssim \frac{1}{\langle X_\parallel (s;x,p) \rangle^3} \int_0^\infty e^{- \beta' \frac{m g}{2 c} y_3} y_3 \int_{\R^2} \big( \frac{1}{\langle y_\parallel \rangle^3} +
\frac{1}{\langle X_\parallel (s;x,p) - y_\parallel \rangle^3} \big) \dd y_\parallel \dd y_3
\end{split}
\ee
Using the cylindrical coordinates $(y_1, y_2, y_3) =  ( r \cos \varphi, r \sin \varphi, y_3)$, together with change of coordinates $z_\parallel = X_\parallel (s;x,p) - y_\parallel$, we derive that
\be \label{est3:est5:x^awh_2}
\begin{split}
& \int_0^\infty e^{- \beta' \frac{m g}{2 c} y_3} y_3 \int_{\R^2} \big( \frac{1}{\langle y_\parallel \rangle^3} +
\frac{1}{\langle X_\parallel (s;x,p) - y_\parallel \rangle^3} \big) \dd y_\parallel \dd y_3
\\& \lesssim \int_0^\infty e^{- \beta' \frac{m g}{2 c} y_3} y_3 \dd y_3\int^{\infty}_{0} \frac{r}{1 + r^3} \dd r \lesssim \frac{2 c}{\beta' m g}.
\end{split}
\ee
Applying the above into \eqref{est2:est5:x^awh_2}, we get
\be \label{est4:est5:x^awh_2}
\eqref{est5:x^awh}_2
\lesssim \frac{2 c}{\beta' m g} \frac{1}{\langle X_\parallel (s;x,p) \rangle^3}.
\ee
Combining with \eqref{est2:est5:x^awh_1} and \eqref{est4:est5:x^awh_2}, we conclude \eqref{est:x^awh_*}.

Inputting \eqref{est:x^awh_*} into \eqref{est4:x^awh}, we derive that 
\be \label{est8:x^awh}
\begin{split}
\eqref{est4:x^awh}
& \lesssim \Big( \frac{e_+}{ (\beta')^3 } \| \langle x_\parallel \rangle^3 \hat{w}_+ \mathrm{d}_+ (x, p ) \|_{L^{\infty} (\O \times \R^3)} + \frac{e_-}{ (\beta')^3 } \| \langle x_\parallel \rangle^3 \hat{w}_- \mathrm{d}_- (x, p ) \|_{L^{\infty} (\O \times \R^3)} \Big)
\\& \qquad \times 
\underbrace{
\int^0_{-\tb(x,p)} e^{- \frac{1}{4} \big( \sqrt{(m c)^2 + | P (s:x,p) |^2} + \frac{m g}{2 c} X_{3} (s;x,p) \big) }  
\frac{\langle x_\parallel \rangle^3}{\langle  X_{\parallel} (s;x,p) \rangle^3}
\dd s
}_{\eqref{est8:x^awh}_*}.
\end{split}
\ee

Now we consider the relation between $\langle x_\parallel \rangle$ and $\langle  X_\parallel (s;x,p) \rangle$ as follows:
\Be \notag
\begin{split}
X (s;x,p) 
= x + \int^s_0 V (\tau; x,p) \dd \tau
= x + \int^s_0 \frac{P (\tau;x,p)}{\sqrt{m^2 + |P (\tau;x,p)|^2 / c^2}} \dd \tau.
\end{split}
\Ee
Thus, we obtain that
\be \label{est:|x_12|}
| x_{\parallel} |
\leq | X_{\parallel} (s;x,p) | + c s.
\ee
Inputting \eqref{est:|x_12|} into $\eqref{est8:x^awh}_*$, together with $\langle y \rangle \geq 1$ for any $y \in \R^3$, we have
\be \label{est:est8:x^awh_*}
\begin{split}
\eqref{est8:x^awh}_*
& = \int^0_{-\tb(x,p)} e^{- \frac{1}{4} \big( \sqrt{(m c)^2 + | P (s:x,p) |^2} + \frac{m g}{2 c} X_{3} (s;x,p) \big) }   
\frac{\langle x_\parallel \rangle^3}{\langle  X_\parallel (s;x,p) \rangle^3}
\dd s
\\& \lesssim \int^0_{-\tb(x,p)} e^{- \frac{1}{4} \big( \sqrt{(m c)^2 + | P (s:x,p) |^2} + \frac{m g}{2 c} X_{3} (s;x,p) \big) }  
(1 + c s )^3 \dd s.
\end{split}
\ee
Recall from \eqref{char:phi}, for all $s \in [-\tb (x,p), 0]$,
\Be 
\begin{split}
& \sqrt{(m c)^2 + |P (s;x,p)|^2} + \frac{1}{c} \big( e \Phi (X (s;x,p)) + m g X_{3} (s;x,p) \big)
\\& = \sqrt{(m c)^2 + |p|^2} + \frac{1}{c} \big( e \Phi (x) + m g x_{3} \big). 
\end{split}
\Ee  
Together with \eqref{condition:phi}, for all $s \in [-\tb (x,p), 0]$,
\Be 
\begin{split}
& \sqrt{(m c)^2 + | P (s;x,p) |^2} + \frac{m g}{2 c} X_{3} (s;x,p) 
\\& \geq \frac{1}{2} \Big( \sqrt{(m c)^2 + |P (s;x,p)|^2} + \frac{1}{c} \big( e \Phi (X (s;x,p)) + m g X_{3} (s;x,p) \big) \Big)
\\& = \frac{1}{2} \big( \sqrt{(m c)^2 + |p|^2} + \frac{1}{c} \big( e \Phi (x) + m g x_{3} \big) \big)
\geq \frac{1}{2} \big( \sqrt{(m c)^2 + |p|^2} + \frac{m g}{2 c} x_{3} \big). 
\end{split}
\Ee
Using $\tb(x,p) \leq \frac{4}{m g} \big( \sqrt{(m c)^2 + |p|^2} + \frac{3}{2c} m g x_{3} \big)$ in \eqref{est:tb_shock}, we get
\be \label{est2:est8:x^awh_*}
\begin{split}
\eqref{est:est8:x^awh_*}
& \lesssim e^{- \frac{1}{8} \big( \sqrt{(m c)^2 + |p|^2} + \frac{m g}{2 c} x_{3} \big)} 
\int^0_{-\tb(x,p)}  
(1 + c s )^3
\dd s
\\& \lesssim e^{- \frac{1}{8} \big( \sqrt{(m c)^2 + |p|^2} + \frac{m g}{2 c} x_{3} \big)} \times
\big( 1 + |\tb(x,p)|^4 \big) \lesssim 1.
\end{split}
\ee 
This shows that $\eqref{est8:x^awh}_* \lesssim 1$, and thus
\be \label{est6:x^awh}
\eqref{est8:x^awh}
\lesssim \frac{e_+}{ (\beta')^3 } \| \langle x_\parallel \rangle^3 \hat{w}_+ \mathrm{d}_+ (x, p ) \|_{L^{\infty} (\O \times \R^3)} + \frac{e_-}{ (\beta')^3 } \| \langle x_\parallel \rangle^3 \hat{w}_- \mathrm{d}_- (x, p ) \|_{L^{\infty} (\O \times \R^3)},
\ee
Applying \eqref{est6:x^awh} into \eqref{est:x^awh}, we get
\Be \label{est7:x^awh}
\begin{split}
& \langle x_\parallel \rangle^3 | \hat{w} \mathrm{d} (x,p) | 
\\& \lesssim \langle x_\parallel \rangle^3 e^{ \beta' \sqrt{(m c)^2 + |\pb (x,p)|^2} } | h (\xb, \pb) - h_I (\xb, \pb) |
\\& \ \ \ \ + \Big( \frac{e_+}{ (\beta')^3 } \| \langle x_\parallel \rangle^3 \hat{w}_+ \mathrm{d}_+ (x, p ) \|_{L^{\infty} (\O \times \R^3)} + \frac{e_-}{ (\beta')^3 } \| \langle x_\parallel \rangle^3 \hat{w}_- \mathrm{d}_- (x, p ) \|_{L^{\infty} (\O \times \R^3)} \Big). 
\end{split}
\Ee

\smallskip

\textbf{Step 4.}
From \eqref{char:phi}, \eqref{est:|x_12|} and \eqref{est2:tb_shock}, we get
\be \label{est2:|x_12|}
| x_{\parallel} | 
\leq | \xb | + c | \tb(x,p) | 
\lesssim | \xb | + \frac{4 c}{m g} |\pb (x, p)|,
\ee
and thus
\be \label{est3:|x_12|}
\langle x_{\parallel} \rangle
\lesssim 1 + | \xb | + \frac{4 c}{m g} |\pb (x, p)|.
\ee
From the assumption \eqref{condition:T_pm}, we get that
\be \label{cond:T(xb)}
|\frac{1}{T (\xb)} - 1| \leq \frac{1}{(10 + |\xb| )^{4}} \leq \frac{1}{10}.
\ee

\smallskip

\textbf{\underline{Case 1:} $|\xb| \leq |\pb|$.}
From \eqref{cond:T(xb)}, we have
\Be \notag
\begin{split}
& | h (\xb, \pb) - h_I (\xb, \pb) |
\\& = \Big| \frac{1}{ e m^2 T (\xb) } e^{- \frac{1}{2 T (\xb)} \sqrt{(m c)^2 + |\pb|^2} } - \frac{1}{ e m^2 } e^{- \frac{1}{2 } \sqrt{(m c)^2 + |\pb|^2} } \Big|
\lesssim e^{ - \frac{1}{3} \sqrt{(m c)^2 + |\pb|^2} }.
\end{split}
\ee
Using $\beta' \leq \frac{1}{4}$ and \eqref{est3:|x_12|}, together with $|\xb| \leq |\pb|$, we obtain
\be \label{est:T(xb)}
\begin{split}
& \langle x_\parallel \rangle^3 e^{ \beta' \sqrt{(m c)^2 + |\pb (x,p)|^2} } | h (\xb, \pb ) - h_I (\xb, \pb) |
\\& \leq (1 + | \pb |)^3 e^{\frac{1}{4} \sqrt{(m c)^2 + |\pb |^2} } e^{- \frac{1}{3} \sqrt{(m c)^2 + |\pb |^2} } \lesssim 1.
\end{split}
\ee

\smallskip

\textbf{\underline{Case 2:} $|\xb| > |\pb|$.}
Now we consider
\Be \notag
\begin{split}
| h (\xb, \pb) - h_I (\xb, \pb) |
& = \Big| \frac{1}{ e m^2 T (\xb) } e^{- \frac{1}{2 T (\xb)} \sqrt{(m c)^2 + |\pb|^2} } - \frac{1}{ e m^2 } e^{- \frac{1}{2 } \sqrt{(m c)^2 + |\pb|^2} } \Big|
\\& = \frac{1}{ e m^2 } e^{- \frac{1}{2} \sqrt{(m c)^2 + |\pb|^2} } \times \Big| \frac{ e^{ \frac{1}{2 } \sqrt{(m c)^2 + |\pb|^2} (1 - \frac{1}{T(\xb)}) } }{ T (\xb) } - 1 \Big|.
\end{split}
\Ee
Together with \eqref{est3:|x_12|}, \eqref{est:T(xb)} and $|\xb| > |\pb|$, we have
\Be \label{est:x^awh_muT-mu}
\begin{split}
& \langle x_\parallel \rangle^3 e^{ \beta' \sqrt{(m c)^2 + |\pb (x,p)|^2} } | h (\xb, \pb ) - h_I (\xb, \pb) |
\\& \lesssim (1 + | \xb |)^3 e^{ \beta' \sqrt{(m c)^2 + |\pb (x,p)|^2} } | h (\xb, \pb ) - h_I (\xb, \pb) |
\\& \lesssim 1 + 
\underbrace{
| \xb |^3 e^{- \frac{1}{4} \sqrt{(m c)^2 + |\pb |^2} } \times \Big| \frac{ e^{ \frac{1}{2 } \sqrt{(m c)^2 + |\pb|^2} (1 - \frac{1}{T(\xb)}) } }{ T (\xb) } - 1 \Big|
}_{\eqref{est:x^awh_muT-mu}_*}.
\end{split}
\Ee
From \eqref{cond:T(xb)}, we have
\be \label{equ:T(xb)}
\Big| \frac{ e^{ \frac{1}{2 } \sqrt{(m c)^2 + |\pb|^2} (1 - \frac{1}{T(\xb)}) } }{ T (\xb) } \Big|
\leq \Big| e^{ \frac{1}{2} \sqrt{(m c)^2 + |\pb|^2} \frac{1}{( 10 + |\xb| )^4 } } \big( 1 + (10 + |\xb| )^{-4} \big) \Big|.
\ee
Since $|\xb| > |\pb|$, we bound
\be \notag
\Big| e^{ \frac{1}{2} \sqrt{(m c)^2 + |\pb|^2} \frac{1}{( 10 + |\xb| )^4 } } \big( 1 + (10 + |\xb| )^{-4} \big) \Big|
\leq \Big| e^{ \frac{m c}{ 20 ( 10 + |\xb| )^{3} } } \big( 1 + (10 + |\xb| )^{-4} \big) \Big|.
\ee
Using the Taylor expansion, we get
\be \notag
e^{ \frac{m c}{ 20 ( 10 + |\xb| )^{3} } }
\leq 1 + \frac{m c}{ 10 ( 10 + |\xb| )^{3} },
\ee
and thus
\be \label{est3:T(xb)}
\Big| e^{ \frac{1}{2} \sqrt{(m c)^2 + |\pb|^2} \frac{1}{( 10 + |\xb| )^4 } } \big( 1 + (10 + |\xb| )^{-4} \big) \Big|
\lesssim 1 + \frac{m c}{ 10} ( 10 + |\xb| )^{-3}.
\ee

Analogously, the Taylor expansion also implies that
\be \label{est4:T(xb)}
1 - \frac{m c}{ 10} ( 10 + |\xb| )^{-3} 
\lesssim \Big| \frac{ e^{ \frac{1}{2 } \sqrt{(m c)^2 + |\pb|^2} (1 - \frac{1}{T(\xb)}) } }{ T (\xb) } \Big|.
\ee
Using \eqref{est3:T(xb)} and \eqref{est4:T(xb)}, we obtain
\be \label{est5:T(xb)}
1 - \frac{m c}{ 10} ( 10 + |\xb| )^{-3}
\lesssim \Big| \frac{ e^{ \frac{1}{2 } \sqrt{(m c)^2 + |\pb|^2} (1 - \frac{1}{T(\xb)}) } }{ T (\xb) } \Big|
\lesssim 1 + \frac{m c}{ 10} ( 10 + |\xb| )^{-3}.
\ee
This shows that
\be \label{est7:T(xb)}
\Big| \frac{ e^{ \frac{1}{2 } \sqrt{(m c)^2 + |\pb|^2} (1 - \frac{1}{T(\xb)}) } }{ T (\xb) } - 1 \Big|
\lesssim ( 10 + |\xb| )^{-3}.
\ee
From \eqref{est5:T(xb)} and \eqref{est7:T(xb)}, we compute
\be \label{est6:T(xb)}
\begin{split}
& \eqref{est:x^awh_muT-mu}_*
\leq | \xb |^3 \times \Big| \frac{ e^{ \frac{1}{2 } \sqrt{(m c)^2 + |\pb|^2} (1 - \frac{1}{T(\xb)}) } }{ T (\xb) } - 1 \Big|
\lesssim | \xb |^3 \times ( 10 + |\xb| )^{-3} \lesssim 1.
\end{split}
\ee
Thus, we derive that 
\be \label{est9:T(xb)}
\langle x_\parallel \rangle^3 e^{ \beta' \sqrt{(m c)^2 + |\pb (x,p)|^2} } | h (\xb, \pb ) - h_I (\xb, \pb) | 
\lesssim 1.
\ee

\smallskip

\textbf{Step 5.}
Combining \eqref{est:T(xb)} and \eqref{est9:T(xb)}, together with \eqref{est7:x^awh}, we obtain that
\be \notag
\begin{split}
& \langle x_\parallel \rangle^3 | \hat{w} \mathrm{d} (x,p) | 
\\& \lesssim 1 + \Big( \frac{e_+}{ (\beta')^3 } \| \langle x_\parallel \rangle^3 \hat{w}_+ \mathrm{d}_+ (x, p ) \|_{L^{\infty} (\O \times \R^3)} + \frac{e_-}{ (\beta')^3 } \| \langle x_\parallel \rangle^3 \hat{w}_- \mathrm{d}_- (x, p ) \|_{L^{\infty} (\O \times \R^3)} \Big). 
\end{split}
\ee
Consider the above inequality in both $+$ and $-$ cases, we get
\be \notag
\begin{split}
& \langle x_\parallel \rangle^3 | \hat{w}_+ \mathrm{d}_+ (x,p) | + \langle x_\parallel \rangle^3 | \hat{w}_- \mathrm{d}_- (x,p) | 
\\& \lesssim 1 + \Big( \frac{e_+}{ (\beta')^3 } \| \langle x_\parallel \rangle^3 \hat{w}_+ \mathrm{d}_+ (x, p ) \|_{L^{\infty} (\O \times \R^3)} + \frac{e_-}{ (\beta')^3 } \| \langle x_\parallel \rangle^3 \hat{w}_- \mathrm{d}_- (x, p ) \|_{L^{\infty} (\O \times \R^3)} \Big). 
\end{split}
\ee
From the assumption $2 (e_+ + e_-) < (\beta')^3$, we conclude
\Be \label{est9:x^awh}
\sum\limits_{i = \pm}
\| \langle x_\parallel \rangle^3 \hat{w}_i \mathrm{d}_i (x, p ) \|_{L^{\infty} (\O \times \R^3)}
\lesssim 1.
\Ee
Inputting \eqref{est9:x^awh} and \eqref{est:x^awh_*} into \eqref{est:phi_x}, we derive that
\Be \label{est2:phi_x}
\begin{split}
& | \nabla \phi (X(s;x,p)) |
\\& \leq \sum\limits_{i = \pm} \frac{e_i}{ (\beta')^3 } \| \langle x_\parallel \rangle^3 \hat{w}_i \mathrm{d}_i (x, p ) \|_{L^{\infty} (\O \times \R^3)} \times \int_{\O}  \frac{e^{- \beta' \frac{m_i g}{2 c} y_3}}{\langle y_\parallel \rangle^3 } \big| \nabla_x \mathfrak{G}(X(s;x,p), y) \big| \dd y
\\& \lesssim \sum\limits_{i = \pm} \frac{e_i}{ (\beta')^3 } \| \langle x_\parallel \rangle^3 \hat{w}_i \mathrm{d}_i (x, p ) \|_{L^{\infty} (\O \times \R^3)} \times \frac{1}{\langle  X_\parallel (s;x,p) \rangle^3}.
\end{split}
\ee
This can be rewritten as
\Be \notag
|\nabla \phi | \lesssim \langle x  \rangle^{-3},
\Ee
and we conclude this theorem.
\end{proof}

\begin{remark}
\label{rmk:iso_explicit_solution}

Consider the Vlasov-Poisson system \eqref{VP_h}-\eqref{eqtn:Dphi} with $m_+ = m_-$ and $e_+ = - e_-$.

\smallskip

$(a)$
There exists an explicit solution to the isothermal case \eqref{VP_h_iso} given by
\be \label{eq:iso_explicit_solution}
h_{I, \pm} (x,p) 
= \frac{1}{ |e_{\pm}| m^2_{\pm} } e^{ - \frac{1}{2} \big( \sqrt{(m_{\pm} c)^2 + |p|^2} + \frac{1}{c} m_{\pm} g x_3  \big)  }.
\ee
Under direct computation, we have 
\be \notag
\Delta \Phi_I (x) = \rho_I (x)
= \int_{\R^3} ( e_+ h_{I, +} + e_{-} h_{I, -} ) \dd p = 0
\ \text{ with } \ 
\Phi_I |_{\p\O} = 0.
\ee
Thus, $\Phi_I (x) \equiv 0$ is a solution. 
From \eqref{est:hat_w_conservation}, we conclude $(h_{I, \pm}, \rho_I, \Phi_I)$ is a solution to \eqref{VP_h_iso}.

\smallskip

$(b)$
Furthermore, suppose $(h_{\pm}, \rho, \Phi)$ is the unique solution to \eqref{VP_h}-\eqref{eqtn:Dphi} in the sense of Definition \ref{weak_sol} under non-isothermal J\"uttner distribution.
Assume $(h_{\pm}, \rho, \Phi)$ satisfy \eqref{Uest:wh}-\eqref{Uest:DPhi} and 
\be \notag
|T_{\pm} (x) - 1| \leq \frac{1}{(20 + |x| )^{4}}
\ \text{ with } \
x \in \p\O.
\ee
Consider $h_{I, \pm} (x, p)$ in \eqref{eq:iso_explicit_solution}, then
\Be \notag
\begin{split}
| \nabla \Phi | 
& \lesssim \langle x  \rangle^{-3},
\\
\sum\limits_{i = \pm}
\| \langle x_\parallel \rangle^3 w_{i, \beta'} ( h_{I, i} - h_i ) \|_{L^{\infty} (\O \times \R^3)} 
& \lesssim 1,
\end{split}
\Ee
where $\langle x_\parallel \rangle = \sqrt{1+ |x_\parallel|^2}$ and $w_{\pm, \beta'} (x,p)$ in \eqref{w^h}.
\end{remark}

\section{Specular Boundary Condition}
\label{sec:specular}

In this section, we consider two species of relativistic Vlasov system \eqref{VP_F}-\eqref{Dbc:F} under the mixture of inflow boundary and specular boundary conditions as follows:
\Be \label{bdry:F_spec}
F_{\pm} (t,x,p) = G_{\pm} (x,p) + \varepsilon F_{\pm} (t,x, \tilde{p})
\ \ \text{on} \ (x, p) \in \gamma_-,
\Ee
where $\tilde{p} = (p_1, p_2, - p_3)$.

\subsection{Steady Solutions}
\label{sec:ss_spec}

We start with the steady problem to \eqref{VP_F}-\eqref{Dbc:F} and \eqref{bdry:F_spec} for $F_\pm (x,p) = h_{\pm} (x,p)$ and $\phi_F (x) = \Phi_h (x)$:
\begin{align}
v_\pm \cdot \nabla_x h_{\pm}
+ \big( e_{\pm} ( \frac{v_\pm}{c} \times B 
- \nabla_x \Phi_h ) - \nabla_x (m_{\pm} g x_3) \big) \cdot \nabla_p h_{\pm} = 0 \ \ & \text{in} \ \O \times \R^3, \label{VP_h_spec} \\
h_{\pm} (x,p) = G_{\pm} (x,p) + \varepsilon h_{\pm} (x, \tilde{p}) 
\ \ & \text{in} \ \gamma_-, \label{bdry:h_spec}
\end{align}
where $B= (0,0, B_3)$, $\tilde{p} = (p_1, p_2, - p_3)$, and the relativistic velocity $v_\pm$ is defined in \eqref{totalE}.
A steady electric potential is determined by solving:
\Be \label{eqtn:Dphi_spec}
\begin{split}
- \Delta_x \Phi_h (x) = \rho_h (x) \ \ \text{in } \O, \ \ \text{and} \  \ 
\Phi_h  = 0 \ \ \text{on } \p\O,
\end{split}
\Ee
where the steady local charge density is given by
\Be \label{def:rho_spec}
\rho_h (x) = \int_{\R^3} ( e_+ h_+ + e_{-} h_{-} ) \dd p. 
\Ee
For simplicity, we often let $(-\Delta_0)^{-1} \rho_h$ denote $\Phi_h$ solving \eqref{eqtn:Dphi_spec}. 
Moreover, for the reader's convenience, we write down the characteristics $Z_{\pm} (s;x,p) = (X_{\pm} (s;x,p), P_{\pm} (s;x,p))$ for the steady problem \eqref{VP_h_spec} which is the same as \eqref{VP_h}: 
\Be
\begin{split} \label{ODE_h_spec}
\frac{d X_{\pm} (s;x,p) }{d s} & = V_{\pm} (s;x,p) = \frac{P_{\pm} (s;x,p)}{\sqrt{m^2_{\pm} + |P_{\pm} (s;x,p)|^2 / c^2}}, \\
\frac{d P_{\pm} (s;x,p) }{d s} & =  {e_{\pm}} ( V_{\pm} (s;x,p) \times \frac{B}{c} - \nabla_x \Phi_h (X_{\pm} (s;x,p) )) - {m_{\pm}} g \mathbf{e}_3,
\end{split}
\Ee
where $\Phi_h $ solves \eqref{eqtn:Dphi} and $\mathbf{e}_3= (0,0,1)^{\intercal}$. At $s = 0$, we have
\[
Z_{\pm} (0; x, p) = (X_{\pm} (0; x, p), P_{\pm} (0; x,p)) = (x, p) = z_{\pm}.
\]

\medskip

Similar to Definition \ref{weak_sol}, we provide a precise definition of weak solutions to the steady problem \eqref{VP_h_spec}-\eqref{def:rho_spec}.

\begin{definition} \label{weak_sol_spec}

(a) We say that $(h_{\pm}, \nabla_x  \Phi_h) \in \big( L^2_{loc}(\O \times \R^3) \; \cap \; L^2_{loc}(\p\O \times \R^3; \dd \gamma) \big) \times L^2_{loc}(\O \times \R^3)$ is a weak solution of \eqref{VP_h_spec} and \eqref{bdry:h_spec}, if all terms below are bounded and the following condition holds
for any test function $\psi \in C^\infty_c (\bar \O \times \R^3)$, 
\Be \label{weak_form_spec}
\begin{split}
& \iint_{\O \times \R^3} h_{\pm} (x,p) v \cdot \nabla_x \psi(x, p) \dd p \dd x 
\\& \ \ \ \ - \iint_{\O \times \R^3} h_{\pm} (x, p) \big( e_{\pm} ( \frac{v_\pm}{c} \times B 
- \nabla_x \Phi_h ) - \nabla_x (m_{\pm} g x_3) \big) \cdot \nabla_p \psi (x,p)  \dd p \dd x 
\\& = \int_{ \p\O \times \{ p_3 <0 \}} h_{\pm} (x,p) \psi(x,p) \dd \gamma - \int_{ \p\O \times \{ p_3 > 0 \}} G_{\pm} (x,p) \psi(x,p) \dd \gamma 
\\& \ \ \ \ - \int_{ \p\O \times \{ p_3 > 0 \}} \varepsilon h_{\pm} (x, \tilde{p}) \psi(x,p) \dd \gamma,
\end{split}
\Ee
where $\dd \gamma := |v_3| \dd S_x \dd p$ denotes the phase boundary measure.
	
(b) We say that $(h_{\pm}, \rho_h) \in L^2_{loc}(\O \times \R^3) \times L^2_{loc}(\O  )$ is a weak solution of \eqref{def:rho_spec}, if all terms below are bounded and the following condition holds for any test function $\vartheta \in H^1_0 (\O) \cap C^\infty_c (\bar \O)$, 
\Be \label{weak_form_2_spec}
\int_{\O} \vartheta (x) \rho_h (x) \dd x 
= \int_{\O} \vartheta (x) \int _{\R^3} ( e_+ h_{+} (x, p) + e_{-} h_{-} (x, p) ) \dd p \dd x.
\Ee
	
(c) We say that $(\rho_h, \Phi_h) \in L^2_{loc}(\O  ) \times W^{1,2}_{loc} (\O)$ is a weak solution of \eqref{eqtn:Dphi_spec}, if all terms below are bounded and the following condition holds for any test function $\varphi \in H^1_0 (\O) \cap C^\infty_c (\bar \O)$, 
\Be \label{weak_form_3_spec}
\int_{\O} \nabla_x \Phi_h \cdot \nabla_x \varphi \dd x 
= \int_{\O} \rho_h \varphi \dd x.
\Ee
\end{definition}

To state the main theorem for the steady problem, we recall $w_{\pm} (x,p)$ in Definition \ref{w^h} and $\alpha_{\pm} (x, p)$ in Definition \ref{alpha}. Now we state the main theorem of the steady problem under the boundary condition \eqref{bdry:h_spec} as follows.

\begin{theorem} \label{theo:CS_spec}

Assume there exists a constant $\px > 0$ such that 
\be \label{condition:G_support_spec}
G_{\pm} (x, p) = 0 
\ \text{ for any } \
|p| \geq \px,
\ee
and suppose $\| e^{ \beta \sqrt{(m_{\pm} c)^2 + |p|^2} }G_{\pm} \|_{L^\infty(\gamma_-)} < \infty$ and $g, \beta > 0$ with $g \beta \gg 1$ satisfy the following condition:
\Be \label{condition:beta_spec}
\beta \geq
\frac{\mathfrak{C} }{\min \big( \frac{m_{+}}{e_{+}}, \frac{m_{-}}{e_{-}} \big) \times \  \frac{g}{2}} \big( e_{\pm} \| e^{ \beta \sqrt{(m_{\pm} c)^2 + |p|^2} } G_{\pm} \|_{L^\infty(\gamma_-)} \big) \times
\big( 1 + \frac{2 c}{ \beta g  } \frac{1}{\min\{m_+, m_- \}} \big),
\Ee
where $\mathfrak{C}>0$ is defined in \eqref{est:nabla_phi}. 
Furthermore, let $\beta \geq \tilde \beta >0$ ($g \tilde \beta \gg1 $) satisfy the
following condition:
\Be \label{condition:tilde_beta_spec}
\begin{split}
& \big( e_+ \| e^{ \beta \sqrt{(m_{+} c)^2 + |p|^2} } G_+ \|_{L^\infty(\gamma_-)} + e_{-} \| e^{ \beta \sqrt{(m_{-} c)^2 + |p|^2}} G_- \|_{L^\infty(\gamma_-)} \big)
\\& \ \ \ \ \times \log \Big( e+ \| e^{\tilde \beta \sqrt{(m_{+} c)^2 + |p|^2}} \nabla_{x_\parallel, p} G_+ \|_{L^\infty (\gamma_-)} + \| e^{\tilde \beta \sqrt{(m_{-} c)^2 + |p|^2}} \nabla_{x_\parallel, p} G_- \|_{L^\infty (\gamma_-)} \Big) 
\\& \leq \beta (\frac{m_+ + m_- }{8} g\tilde  \beta - (1 + B_3) ),
\end{split}	
\Ee
and
\be \label{condition:G_xv_spec}
(1 + \frac{\tilde{\beta}}{m g} ) \| e^{\tilde \beta \sqrt{(m_{\pm} c)^2 + |p|^2}} \nabla_{x_\parallel, p} G_{\pm} \|_{L^\infty (\gamma_-)} \leq \frac{1}{4}.
\ee
Finally, suppose $\varepsilon$ in the boundary condition \eqref{bdry:h_spec} satisfies that	
\be \label{condition:epilon_G_spec}
\varepsilon \big( 1 + \tilde{\beta} \big) {e^{\frac{\tilde{\beta}}{2} \sqrt{(m c)^2 + |\px|^2}}} \leq \frac{1}{4}.
\ee	
Then there exists a unique solution $(h_{\pm}, \rho, \Phi_h)$ to \eqref{VP_h_spec}-\eqref{def:rho_spec} in the sense of Definition \ref{weak_sol}.  More-
over, the following estimates hold:
\begin{align}
& \| w_{\pm, \beta} h_{\pm} \|_{L^\infty (\bar \O \times \R^3)} 
\leq \| e^{ \beta\sqrt{(m_{\pm} c)^2 + |p|^2} } G_{\pm} \|_{L^\infty (\gamma_-)},
\label{Uest:wh_spec} \\
& | \rho (x) | \leq \frac{1}{\beta} \big( e_+ \| w_{+, \beta} G_+ \|_{L^\infty(\gamma_-)} e^{-  \beta \frac{ m_+ }{2 c} g x_3} + e_{-} \| w_{-, \beta} G_- \|_{L^\infty(\gamma_-)} e^{- \beta \frac{ m_{-} }{2 c} g x_3} \big), 
\label{Uest:rho_spec} \\
& \| \nabla_x \Phi_h \|_{L^\infty (\bar{\O})} \leq  \min \big( \frac{m_{+}}{e_{+}}, \frac{m_{-}}{e_{-}} \big) \times \  \frac{g}{2},
\label{Uest:DPhi_spec} \\
& \frac{8}{m g} (1 + B_3 + \| \nabla_x ^2 \Phi _h \|_\infty) \leq \tilde \beta \leq \beta.
\label{Uest:Phi_xx_spec}
\end{align}
In addition, $h_\pm$ satisfies the following:
\be \label{Uest:h_support_spec}
h_{\pm} (x, p) = 0 
\ \text{ for any } \
x \in \p\O
\text{ and }
|p| \geq \px,
\ee
and $\rho$ and $\nabla_x \Phi _h$ satisfy the following bounds:
\begin{align}
& e^{ \frac{\tilde \beta m g}{4 c} x_3 } |\p_{x_i} \rho  (x)|   
\lesssim \| e^{\tilde \beta \sqrt{(m_+ c)^2 + |p|^2}} \nabla_{x_\parallel, p} G_+ \|_{L^\infty (\gamma_-)} \times \Big( 1 + \mathbf{1}_{|x_3| \leq 1} \frac{1}{\sqrt{ m_+ g x_3 }} \Big)
\notag \\
& \qquad \qquad \qquad \qquad + \| e^{\tilde \beta \sqrt{(m_- c)^2 + |p|^2}} \nabla_{x_\parallel, p} G_- \|_{L^\infty (\gamma_-)} \times \Big( 1 + \mathbf{1}_{|x_3| \leq 1} \frac{1}{\sqrt{ m_- g x_3 }} \Big),
\label{est_final:rho_x_spec} \\
& \| \nabla_x^2 \Phi_h  \|_\infty
\lesssim \frac{1}{\beta} \big( e_+ \| w_{+, \beta} G_+ \|_{L^\infty(\gamma_-)} + e_{-} \| w_{-, \beta} G_- \|_{L^\infty(\gamma_-)} \big)
\notag \\
& \ \ \times \log \big( e+ \| e^{\tilde \beta \sqrt{(m_+ c)^2 + |p|^2}} \nabla_{x_\parallel, p} G_+ \|_{L^\infty (\gamma_-)} + \| e^{\tilde \beta \sqrt{(m_- c)^2 + |p|^2}} \nabla_{x_\parallel, p} G_- \|_{L^\infty (\gamma_-)} \big).
\label{est_final:phi_C2_spec}
\end{align}
Furthermore, given $\alpha_\pm (x, p)$ defined in $\eqref{alpha}$, then $h_{\pm}$ satisfies that
\begin{align}
& e^{ \frac{\tilde \beta}{2}|p^0_{\pm}|} e^{  \frac{\tilde \beta m_{\pm} g}{4 c} x_3} | \nabla_p h_{\pm} (x,p)| 
\notag \\
& \lesssim \big( 1 + \frac{ \| \nabla_x^2 \Phi _h \|_\infty + e B_3 + m_{\pm} g}{(m_{\pm} g)^2} \big) \| e^{\tilde \beta \sqrt{(m_{\pm} c)^2 + |p|^2}} \nabla_{x_\parallel, p} G_{\pm} \|_{L^\infty (\gamma_-)},
\label{est_final:hk_v_spec} \\
& e^{ \frac{\tilde \beta}{2} |p^0_{\pm} |} e^{  \frac{\tilde \beta m_{\pm} g}{4 c} x_3} | \nabla_x h_{\pm} (x,p)|  
\notag \\
& \lesssim \big( \frac{\delta_{i3}}{\alpha_{\pm} (x,p)} + \frac{ \| \nabla_x^2 \Phi _h\|_\infty + e B_3 + m_{\pm} g}{(m_{\pm} g)^2} \big) \| e^{\tilde \beta \sqrt{(m_{\pm} c)^2 + |p|^2}} \nabla_{x_\parallel, p} G_{\pm} \|_{L^\infty (\gamma_-)}.
\label{est_final:hk_x_spec}
\end{align}
\end{theorem}

\smallskip


To prove Theorem \ref{theo:CS_spec}, we use approaches similar to those employed in the steady problem, focusing solely on the inflow boundary condition \eqref{VP_h}-\eqref{def:rho}.
In the following, we provide a brief overview of the steps outlined in this subsection.


We begin by constructing the sequences $(h^{\ell+1}_{\pm}, \rho^{\ell}, \nabla_x \Phi^{\ell})$ for $\ell \geq 1$ in \eqref{eqtn:hk_spec}-\eqref{bdry:phik_spec}. 
This construction closely mirrors that under pure inflow boundary conditions, with the primary difference being the boundary conditions applied to $h^{\ell+1}_{\pm}$.

Fortunately, the change in boundary conditions does not affect the characteristic trajectory determined by the Vlasov equations \eqref{VP_h} or \eqref{VP_h_spec}, so all properties established in Section \ref{sec:char} remain valid for the steady problem \eqref{VP_h_spec}-\eqref{def:rho_spec}. In particular, the conservation laws in Lemma \ref{lem:conservation_law}, together with the assumption $G_{\pm} (x, p) = 0$ for any $|p| \geq \px > 0$ and the use of mathematical induction, allows us to demonstrate that for every $\ell \in \N$,
\be \notag
h^{\ell+1}_{\pm} (x, p) = 0 
\ \text{ for any } \
x \in \p\O
\text{ and }
|p| \geq \px.
\ee
Thus, $\{ h^{\ell}_{\pm} \}^{\infty}_{\ell=0}$ has uniformly compact support with respect to $p$ (see Proposition \ref{prop:Unif_steady_spec}).

Similar to the case with pure inflow boundary conditions, we establish a priori estimates for $(h_{\pm}, \rho,  \Phi )$ solving \eqref{VP_h_spec}-\eqref{def:rho_spec}.
Due to the specular boundary condition, some regularity terms are altered. For example: consider $\tpbp = (p_{\mathbf{b}, \pm, 1}, p_{\mathbf{b}, \pm, 2}, - p_{\mathbf{b}, \pm, 3})$, then
\Be \notag
\begin{split}
\nabla_{x, p} h_{\pm} (x,p)
& = \nabla_{x, p} \xbp (x,p) \cdot \nabla_{x_\parallel} G_{\pm} (\xbp, \pbp) + \nabla_{x, p} \pbp (x,p) \cdot \nabla_p G_{\pm} (\xbp, \pbp)
\\& \ \ \ \ + \varepsilon \nabla_{x, p} \xbp (x,p) \cdot \nabla_{x_\parallel} h_{\pm} (\xbp, \tpbp) + \varepsilon \nabla_{x, p} \tpbp (x,p) \cdot \nabla_p h_{\pm} (\xbp, \tpbp).
\end{split}
\Ee
Suppose $h_{\pm} (x,p)$ has compact support with respect to $p$. Using the condition \eqref{condition:epilon_G_spec} on $\varepsilon$, along with the computation of the characteristic trajectory, we deduce the key estimate on $| \nabla_{x,p} h_{\pm} |$ as presented in Proposition \ref{prop:Reg_spec} and Theorem \ref{theo:RS_spec}.
The rest of the steps follow similarly from Sections \ref{sec:US} and \ref{sec:EX_SS}.
We conclude the existence and the uniqueness of the steady solution.

\subsubsection{Construction} \label{sec:CS_spec} 

Now we construct solutions to the steady problem \eqref{VP_h_spec}-\eqref{def:rho_spec} via the following sequences: for any $\ell \in \N$,
\begin{align}
v_\pm \cdot \nabla_x h^{\ell+1}_{\pm} 
+ \big( e_{\pm} ( \frac{v_\pm}{c} \times B 
- \nabla_x \Phi^\ell ) - \nabla_x (m_{\pm} g x_3) \big) \cdot \nabla_p h^{\ell+1}_{\pm} = 0 & \ \ \text{in} \ \O \times \R^3, \label{eqtn:hk_spec} \\
h^{\ell+1}_{\pm} (x,p) = G_{\pm} (x,p) + \varepsilon h^{\ell}_{\pm} (x, \tilde{p})  & \ \ \text{on} \ \gamma_-, \label{bdry:hk_spec} \\	
\rho^\ell = \int_{\R^3} ( e_+ h^\ell_+ + e_{-} h^\ell_{-} ) \dd p & \ \ \text{in} \ \O, \label{eqtn:rhok_spec} \\
- \Delta \Phi^\ell = \rho^\ell & \ \ \text{in} \ \O, \label{eqtn:phik_spec} \\
\Phi^\ell =0  & \ \ \text{on} \ \p\O, \label{bdry:phik_spec}
\end{align}
where $\tilde{p} = (p_1, p_2, - p_3)$ and the initial setting $h^0_{\pm} = 0$ and $(\rho^0,  \nabla_x \Phi^0) = (0, \mathbf{0})$.

\smallskip 

First, we construct $h^1_{\pm}$ solving \eqref{eqtn:hk_spec} and \eqref{bdry:hk_spec} from $h^0_{\pm} = 0$ and $(\rho^0,  \nabla_x \Phi^0) = (0, \mathbf{0})$. We consider the Lagrangian formulation of the following characteristics:
\Be \label{z1_spec}
\begin{split}
\big( X^{1}_{\pm} (t;x,p), P^{1}_{\pm} (t;x,p) \big) |_{t=0} = (x, p) \ \ \text{at} \ t = 0, &
\\ \frac{d X^1_{\pm} }{d t} = V^1_{\pm} = \frac{P^1_{\pm}}{\sqrt{m^2_{\pm} + |P^1_{\pm}|^2 / c^2}}, \ \
\frac{d P^1_{\pm} }{d t} =  {e_{\pm}} V^1_{\pm} \times \frac{B}{c} - {m_{\pm}} g \mathbf{e}_3. &
\end{split}
\Ee
Then from \eqref{eqtn:rhok_spec}, as long as it is well-defined, we obtain
\[
\rho^1 (x) = \int_{\R^3} ( e_+ h^1_+ + e_{-} h^1_{-} ) \dd p.
\]
Using the Green's function in \eqref{phi_rho}, together with \eqref{eqtn:phik_spec}, we derive $\Phi^1$ and $\nabla_x \Phi^1$ by
\[
- \Delta \Phi^1 (x) = \rho^1 (x).
\]

\smallskip

Second, we construct $(h^{\ell+1}_{\pm}, \rho^{\ell}, \nabla_x \Phi^{\ell})$ solving \eqref{eqtn:hk_spec}-\eqref{bdry:phik_spec} for $\ell \geq 1$ by iterating the process. Given $\nabla_x \Phi^{\ell}$, we construct $h^{\ell+1}_{\pm}$ along the characteristics $(X^{\ell+1}_{\pm}, P^{\ell+1}_{\pm})$ as follows:
\be \label{VP_h^l+1_spec}
\big( X^{\ell+1}_{\pm} (t;x,p)，P^{\ell+1}_{\pm} (t;x,p) \big) |_{t=0} = (x, p) \ \ \text{at} \ t = 0,
\ee
and   
\Be \label{bdry:h^l+1_spec}
\begin{split}
& \frac{d X^{\ell+1}_{\pm} }{d t} = V^{\ell+1}_{\pm} = \frac{P^{\ell+1}_{\pm}}{\sqrt{m^2_{\pm} + |P^{\ell+1}_{\pm}|^2 / c^2}}, 
\\& \frac{d P^{\ell+1}_{\pm} }{d t} =  {e_{\pm}} ( V^{\ell+1}_{\pm} \times \frac{B}{c} - \nabla_x \Phi^{\ell} (X^{\ell+1}_{\pm} )) - {m_{\pm}} g \mathbf{e}_3. 
\end{split}
\Ee
Using the Peano theorem, together with the boundness of $\rho^\ell$ (see \eqref{Uest:rho^k_spec}) and the continuity of $\nabla_x \Phi^\ell$ (see \eqref{phi_rho}), we conclude the existence of solutions (not necessarily unique). 
Then we obtain $\rho^{\ell+1}$ from \eqref{eqtn:rhok_spec} and solve $\nabla_x \Phi^{\ell+1}$ from \eqref{eqtn:phik_spec} and \eqref{bdry:phik_spec}.

Recall that the change in boundary conditions doesn't influence the characteristic trajectory, and thus the characteristics in \eqref{VP_h^l+1_spec} and \eqref{bdry:h^l+1_spec} are the same as the characteristics in \eqref{VP_h^l+1} and \eqref{bdry:h^l+1}.
Hence, for every iteration $\ell \geq 0$, suppose $(X^{\ell}, P^{\ell})$ exists we reuse the backward exit time $\tbp^\ell$, position $\xbp^{\ell}$, and momentum $\pbp^{\ell}$ defined in \eqref{def:tb_l}.
Given $\nabla_x \Phi^\ell$, the weak solution $h^{\ell+1}_{\pm}$ to \eqref{eqtn:hk_spec}-\eqref{bdry:hk_spec} satisfies
\Be \label{form:h^k_spec}
\begin{split}
h^{\ell+1}_{\pm} (x,p) 
= h^{\ell+1}_{\pm}  (X_{\pm}^{\ell+1}  ( -t;x,p), P^{\ell+1}_{\pm} (-t;x,p) )  \ \ \text{for all} \ t \in [0, \tbp^{\ell+1} (x,p)].
\end{split}
\Ee
Moreover, we reuse the weight function $w^{\ell+1}_{\pm, \beta}$ defined in \eqref{def:w^k} as follows:
\Be \label{def:w^k_spec}
w^{\ell+1}_{\pm, \beta} (x,p) 
= e^{ \beta \left( \sqrt{(m_{\pm} c)^2 + |p|^2} + \frac{1}{c} ( e_{\pm} \Phi^{\ell} (x) + m_{\pm} g x_3 ) \right) }
\ \ \text{in} \ \O \times \R^3,
\Ee
The weight function is invariant along the characteristics, and at the boundary it satisfies
\Be \label{wk:bdry_spec}
w^{\ell+1}_{\pm, \beta} (x, p) 
= e^{ \beta \sqrt{(m_{\pm} c)^2 + |p|^2} } 
\ \ \text{on} \ \p\O.
\Ee

Using \eqref{bdry:hk_spec}, \eqref{form:h^k_spec}, and \eqref{wk:bdry_spec}, as long as $\tb^{\ell+1} (x,p) < \infty$, then we have  
\Be \label{form:h^k_G_spec}
\begin{split} 
& w^{\ell+1}_{\pm, \beta} (x,p) h^{\ell+1}_{\pm} (x,p) 
\\& = w^{\ell+1}_{\pm, \beta} (\xbp^{\ell+1} (x,p), \pbp^{\ell+1} (x,p) ) \big\{ G_{\pm} (\xbp^{\ell+1} (x,p), \pbp^{\ell+1} (x,p) ) + \varepsilon h^{\ell}_{\pm} (\xbp^{\ell+1} (x,p), \tpbp^{\ell+1} (x,p) ) \big\}
\\&	= e^{ \beta \sqrt{(m_{\pm} c)^2 + | \pbp^{\ell+1} (x,p)|^2} } \big\{ G_{\pm} (\xbp^{\ell+1} (x,p), \pbp^{\ell+1} (x,p) ) + \varepsilon h^{\ell}_{\pm} (\xbp^{\ell+1} (x,p), \tpbp^{\ell+1} (x,p) ) \big\}.
\end{split}	
\Ee
where $\tpbp^{\ell+1} = (p_{\mathbf{b}, \pm, 1}, p_{\mathbf{b}, \pm, 2}, - p_{\mathbf{b}, \pm, 3})$.

\begin{prop} \label{prop:Unif_steady_spec}

Suppose the condition \eqref{condition:beta_spec} holds for some $g, \beta > 0$.
Then, under the above construction $(h^{\ell+1}_{\pm}, \rho^\ell, \nabla_x \Phi^\ell)$ satisfies the following uniform-in-$\ell$ estimates:
\begin{align}
& \| h_{\pm}^{\ell+1} \|_{L^\infty (\bar \O \times \R^3)} 
\leq (1 + \varepsilon ) \|  G_{\pm} \|_{L^\infty (\gamma_-)},
\label{Uest:h^k_spec} 
\\& \| w^{\ell+1}_{\pm, \beta} h_{\pm}^{\ell+1}   \|_{L^\infty (\bar \O \times \R^3)} 
\leq (1 + \varepsilon ) \| e^{ \beta \sqrt{(m_{\pm} c)^2 + |p|^2}} G_{\pm} \|_{L^\infty (\gamma_-)}, \label{Uest:wh^k_spec}
\\& 
| \rho^{\ell} (x) |
\leq \frac{e_+}{\beta} \| e^{ \beta \sqrt{(m_{+} c)^2 + |p|^2}} G_+ \|_{L^\infty(\gamma_-)} e^{-  \beta \frac{ m_+ }{2 c} g x_3} + \frac{e_-}{\beta} \| e^{ \beta \sqrt{(m_{-} c)^2 + |p|^2}} G_- \|_{L^\infty(\gamma_-)} e^{- \beta \frac{ m_{-} }{2 c} g x_3}, \label{Uest:rho^k_spec}
\\& 
\| \nabla_x \Phi^ \ell \|_{L^\infty (\bar{\O})} 
\leq  \min \big( \frac{m_{+}}{e_{+}}, \frac{m_{-}}{e_{-}} \big) \times \  \frac{g}{2}. 
\label{Uest:DPhi^k_spec}
\end{align} 
Furthermore, $\Phi^\ell \in H^1_{0, \text{loc}} (\O)$ and 
\Be\label{lower_w_st_spec}
w^{\ell+1}_{\pm, \beta} (x,p) 
\geq e^{ \beta \left( \sqrt{(m_{\pm} c)^2 + |p|^2} + \frac{m_{\pm}}{2 c} g x_3 \right) }.
\Ee
In addition, suppose the condition \eqref{condition:G_support_spec} holds for some $\px > 0$, then for any $\ell \in \N$,
\be \label{Uest:h^k_support_spec}
h^{\ell+1}_{\pm} (x, p) = 0 
\ \text{ for any } \
x \in \p\O
\text{ and }
|p| \geq \px.
\ee
\end{prop}

\begin{proof}

Under the initial setting $h^0_{\pm} = 0$ and $(\rho^0, \nabla_x \Phi^0) = (0, \mathbf{0})$, together with the invariance of $h^1_{\pm}$ and the weight function $w^{1}_{\pm, \beta} (x,p)$ along the characteristics \eqref{z1_spec}, we deduce \eqref{Uest:h^k_spec}-\eqref{Uest:DPhi^k_spec} and \eqref{Uest:h^k_support_spec} hold for $\ell = 0$.

\smallskip

Now we prove this by induction. In the following proof, we adopt \text{the abuse of notation about $\pm$} in \eqref{abuse}.
Assume a positive integer $k > 0$ and suppose that \eqref{Uest:h^k_spec}-\eqref{Uest:DPhi^k_spec} hold for $0 \leq \ell \leq k$.
From \eqref{eqtn:rhok_spec}, together with \eqref{est1:rho_k+1}, then for $\ell = k + 1$, 
\be \label{est1:rho_k+1_spec}
\begin{split}
| \rho^{k+1} (x) |
\leq \Big|\int_{\R^3} w^{k+1}_{+, \beta}  h^{k+1}_{+} (x,p) \frac{e_{+}}{w^{k+1}_{+, \beta} (x,p)} \dd p \Big| 
+ \Big|\int_{\R^3} w^{k+1}_{-, \beta}  h^{k+1}_{-} (x,p) \frac{e_{-}}{w^{k+1}_{-, \beta} (x,p)} \dd p \Big|
\end{split}
\ee
Using \eqref{Uest:wh^k_spec} for $\ell=k$, we have
\be \label{est2:rho_k+1_spec}
\Big|\int_{\R^3} w^{k+1}_{\beta} h^{k+1} (x,p) \frac{e}{w^{k+1}_{\beta} (x,p)} \dd p \Big| 
\leq (1 + \varepsilon ) \| e^{ \beta \sqrt{(m c)^2 + |p|^2}} G \|_{L^\infty(\gamma_-)} \int_{\R^3} \frac{e}{w^{k+1}_{\beta} (x, p)} \dd p.
\ee
Using \eqref{Uest:wh^k_spec}, \eqref{Uest:DPhi^k_spec} for $\ell=k$ and following \eqref{est3:rho_k+1}, we conclude \eqref{Uest:rho^k_spec} for $\ell=k+1$.

Using \eqref{phi_rho}, we obtain $\Phi^{k+1} \in H^1_{0, \text{loc}} (\O)$ which is the weak solution to \eqref{eqtn:phik_spec}-\eqref{bdry:phik_spec}.
Applying \eqref{est:nabla_phi} in Lemma \ref{lem:rho_to_phi} with $A, B$ as follows:
\be \notag
A = \frac{1 + \varepsilon }{\beta} \big( e_+ \| e^{ \beta \sqrt{(m_{+} c)^2 + |p|^2}} G_+ \|_{L^\infty(\gamma_-)} + e_{-} \| e^{ \beta \sqrt{(m_{-} c)^2 + |p|^2}} G_- \|_{L^\infty(\gamma_-)} \big),
\ \text{and } \
B =  \beta \frac{ \hat{m} }{2 c} g,
\ee
with $\hat{m} = \min ( m_{+}, m_{-} )$, together with the assumption of $\beta$ in \eqref{condition:beta_spec}, then we conclude that \eqref{Uest:DPhi^k_spec} holds for $\ell=k+1$. 

Since $\| \nabla_x \Phi^{k} \|_{L^\infty (\bar \O)} \leq  \min \big( \frac{m_{+}}{e_{+}}, \frac{m_{-}}{e_{-}} \big) \times \  \frac{g}{2}$ holds, 
using \eqref{est:tb^h} in Proposition \ref{lem:tb}, we derive
\Be \notag
\tb^{k+1}(x,p) \leq \frac{4}{m g} \big( \sqrt{(m c)^2 + |p|^2} + \frac{3}{2c} m g x_{3} \big) < \infty
\ \ \text{for all} \
(x, p) \in \bar \O \times \R^3.
\Ee  
From \eqref{form:h^k_spec}, \eqref{form:h^k_G_spec} for $\ell=k$, together with $|\pbp^{\ell+1}| = |\tpbp^{\ell+1}|$, we show 
\Be \label{est3:rho_k+1_spec}
\begin{split}
h^{\ell+1}_{\pm} (x,p) 
= G_{\pm} (\xbp^{\ell+1} (x,p), \pbp^{\ell+1} (x,p) ) + \varepsilon h^{\ell}_{\pm} (\xbp^{\ell+1} (x,p), \tpbp^{\ell+1} (x,p) ),
\end{split}
\Ee
and
\Be \notag
\begin{split} 
& w^{\ell+1}_{\pm, \beta} (x,p) h^{\ell+1}_{\pm} (x,p) 
\\&	= e^{ \beta \sqrt{(m_{\pm} c)^2 + | \pbp^{\ell+1} (x,p)|^2} } \big\{ G_{\pm} (\xbp^{\ell+1} (x,p), \pbp^{\ell+1} (x,p) ) + \varepsilon h^{\ell}_{\pm} (\xbp^{\ell+1} (x,p), \tpbp^{\ell+1} (x,p) ) \big\}
\\&	= e^{ \beta \sqrt{(m_{\pm} c)^2 + | \pbp^{\ell+1} (x,p)|^2} } G_{\pm} (\xbp^{\ell+1} (x,p), \pbp^{\ell+1} (x,p) ) 
\\& \ \ \ \ + \varepsilon e^{ \beta \sqrt{(m_{\pm} c)^2 + | \tpbp^{\ell+1} (x,p)|^2} } h^{\ell}_{\pm} (\xbp^{\ell+1} (x,p), \tpbp^{\ell+1} (x,p) ).
\end{split}	
\Ee
Together with \eqref{Uest:h^k_spec}, \eqref{Uest:wh^k_spec} for $\ell=k$, we conclude \eqref{Uest:h^k_spec} and \eqref{Uest:wh^k_spec} for $\ell=k+1$.
Furthermore, from \eqref{est3:rho_k+1_spec} and \eqref{Uest:h^k_support_spec} for $\ell=k$, together with Lemma \ref{lem:conservation_law}, we conclude \eqref{Uest:h^k_support_spec} for $\ell=k+1$.

\smallskip

Now we have proved \eqref{Uest:h^k_spec}-\eqref{Uest:DPhi^k_spec} and \eqref{Uest:h^k_support_spec} hold for $\ell = k+1$. Therefore, we complete the proof by induction. Using \eqref{Uest:DPhi^k_spec} and the zero Dirichlet boundary condition \eqref{bdry:phik_spec}, we can check that $\Phi^\ell \in H^1_{0, \text{loc}} (\O)$ and $\eqref{lower_w_st_spec}$ holds.
\end{proof}

\subsubsection{A Priori Estimate} \label{sec:RS_spec}

In this section, we establish a priori estimate of $(h , \rho,  \Phi )$ solving \eqref{VP_h_spec}-\eqref{eqtn:Dphi_spec}. 
Further, we always suppose \eqref{Uest:DPhi_spec} holds, which is the same as the condition \eqref{Uest:DPhi} in the case of the steady problem only under the effect of the inflow boundary conditions \eqref{VP_h}-\eqref{def:rho}.
Again we utilize the kinetic distance function $\alpha_{\pm} (x, p)$ defined in $\eqref{alpha}$.

Since the characteristics in \eqref{ODE_h_spec} are the same as the characteristics in \eqref{ODE_h}, Lemmas \ref{VL} - \ref{lem:nabla_zb} hold for the the characteristics $Z_{\pm} (t;x,p) = (X_{\pm} (t;x,p), P_{\pm} (t;x,p) )$ solving \eqref{ODE_h_spec}, and we omit the proofs of these results.

\begin{prop} \label{prop:Reg_spec}

Suppose $(h_{\pm}, \rho , \nabla_x \Phi)$ solve \eqref{VP_h_spec}-\eqref{eqtn:Dphi_spec} in the sense of Definition \ref{weak_sol_spec}. Suppose the condition \eqref{Uest:DPhi_spec} holds. 
Assume that there exists $\px > 0$ such that 
\be \label{condition:h_support_spec}
h_{\pm} (x, p) = 0 
\ \text{ for any } \
|p| \geq \px.
\ee
Further, there exists $\beta, \tilde \beta >0$ such that
\be \label{choice_beta_spec}
\| w_{\pm, \beta} h_{\pm} \|_{L^\infty (\bar \O \times \R^3)} 
\leq (1 + \varepsilon ) \| e^{ \beta \sqrt{(m_{\pm} c)^2 + |p|^2}} G_{\pm} \|_{L^\infty (\gamma_-)},
\ee
and
\be \notag
\| e^{{\tilde \beta } \sqrt{(m_{\pm} c)^2 + |p|^2}} \nabla_{x_\parallel, p} G_{\pm} (x,p) \|_{L^\infty (\gamma_-)} < \infty,
\ee
and let $\hat{m} = \min ( m_{+}, m_{-} )$ satisfy
\Be \label{choice_tbeta_spec}
\frac{8}{ \hat{m} g} (1 + B_3 + \| \nabla_x ^2 \Phi  \|_\infty) \leq \frac{\tilde \beta}{2}.
\Ee
Finally, suppose $\varepsilon$ in the boundary condition \eqref{bdry:h_spec} satisfies that	
\be \label{condition2:epilon_G_spec}
\varepsilon \big( 1 + \tilde{\beta} \big) {e^{\frac{\tilde{\beta}}{2} \sqrt{(m c)^2 + |\px|^2}}} \leq \frac{1}{4}.
\ee	
Then, for  $(x, p) \in \bar \O \times \R^3$, 
\Be \label{est:rho_x_spec}
\begin{split}
e^{ \frac{\tilde \beta \hat{m} g}{4 c} x_3 } |\nabla_{x_i} \rho  (x)|   
& \lesssim \| e^{\tilde \beta \sqrt{(m_+ c)^2 + |p|^2}} \nabla_{x_\parallel, p} G_+ \|_{L^\infty (\gamma_-)}  
\times \Big(
1 + \mathbf{1}_{|x_3| \leq 1} \frac{1}{\sqrt{ m_+ g x_3 }}
\Big)
\\& \ \ \ \ + \| e^{\tilde \beta \sqrt{(m_- c)^2 + |p|^2}} \nabla_{x_\parallel, p} G_- \|_{L^\infty (\gamma_-)}  
\times \Big(
1 + \mathbf{1}_{|x_3| \leq 1} \frac{1}{\sqrt{ m_- g x_3 }}
\Big),
\end{split}
\Ee
and
\Be \label{est:phi_C2_spec}
\begin{split}
\| \nabla_x^2 \Phi  \|_\infty
& \lesssim \frac{1}{\beta} \big( e_+ \| w_{+, \beta} G_+ \|_{L^\infty(\gamma_-)} + e_{-} \| w_{-, \beta} G_- \|_{L^\infty(\gamma_-)} \big)
\\& \ \ \ \ + \sum\limits_{i = \pm} \| e^{\tilde \beta \sqrt{(m_i c)^2 + |p|^2}} \nabla_{x_\parallel, p} G_i \|_{L^\infty (\gamma_-)}.
\end{split}	
\Ee
Moreover, 
\Be \label{est:hk_v_spec}
\begin{split}
& e^{ \frac{\tilde \beta}{2}|p^0_{\pm}|} e^{  \frac{\tilde \beta m_{\pm} g}{4 c} x_3} | \nabla_p h_{\pm} (x,p)| 
\\& \lesssim \big( 1 + \frac{ \| \nabla_x^2 \Phi  \|_\infty + e B_3 + m_{\pm} g}{(m_{\pm} g)^2} \big) \| e^{\tilde \beta \sqrt{(m_{\pm} c)^2 + |p|^2}} \nabla_{x_\parallel, p} G_{\pm} \|_{L^\infty (\gamma_-)},
\end{split}
\Ee
and
\Be \label{est:hk_x_spec}
\begin{split}
& e^{ \frac{\tilde \beta}{2}|p^0_{\pm}|} e^{  \frac{\tilde \beta m_{\pm} g}{4 c} x_3} | \nabla_x h_{\pm} (x,p)|  
\\& \lesssim \big( \frac{\delta_{i3}}{\alpha_{\pm} (x,p)} + \frac{ \| \nabla_x^2 \Phi  \|_\infty + e B_3 + m_{\pm} g}{(m_{\pm} g)^2} \big) \| e^{\tilde \beta \sqrt{(m_{\pm} c)^2 + |p|^2}} \nabla_{x_\parallel, p} G_{\pm} \|_{L^\infty (\gamma_-)}.
\end{split}
\Ee
\end{prop}

\begin{proof}	

For the sake of simplicity, we abuse the notation as in \eqref{abuse}.
Moreover, we use the notation: $p_\pm^0 = \sqrt{(m_\pm c)^2 + |p|^2}$ in the rest of proof.
From \eqref{form:h^k_G_spec}, we get 
\be \notag
h (x,p) = G (\xb, \pb) + \varepsilon h (\xb, \tpb),
\ee
where $\tpb = (p_{\mathbf{b}, 1}, p_{\mathbf{b}, 2}, - p_{\mathbf{b}, 3})$. By taking the derivative on $h (x,p)$, we have
\Be \label{form:nabla_h_spec}
\begin{split}
\nabla_{x, p} h (x,p)
& = \frac{\nabla_{x, p} \xb (x,p)}{e^{\tilde{\beta} |\pb^0 (x,p)| }} \cdot e^{\tilde{\beta} |\pb^0| } \nabla_{x_\parallel} G (\xb, \pb) + \frac{\nabla_{x, p} \pb (x,p)}{e^{\tilde{\beta} |\pb^0 (x,p)|}} \cdot e^{\tilde{\beta} |\pb^0|} \nabla_p G (\xb, \pb)
\\& \ \ \ \ + \varepsilon \nabla_{x, p} \xb (x,p) \cdot \nabla_{x_\parallel} h (\xb, \tpb) + \varepsilon \nabla_{x, p} \tpb (x,p) \cdot \nabla_p h (\xb, \tpb).
\end{split}
\Ee
Together with Lemma \ref{lem:conservation_law} and $|\pb| = |\tpb|$, we further get
\Be \label{weight_form:nabla_h_spec}
\begin{split}
& e^{\frac{\tilde{\beta}}{2} \big( |p^0| + \frac{1}{c} ( e_{\pm} \Phi (x) + m_{\pm} g x_3 ) \big)} \nabla_{x, p} h (x,p)
\\& = \frac{\nabla_{x, p} \xb (x,p)}{e^{\frac{\tilde{\beta}}{2} |\pb^0 (x,p)| }} \cdot e^{\tilde{\beta} |\pb^0| } \nabla_{x_\parallel} G (\xb, \pb) + \frac{\nabla_{x, p} \pb (x,p)}{e^{\frac{\tilde{\beta}}{2} |\pb^0 (x,p)|}} \cdot e^{\tilde{\beta} |\pb^0|} \nabla_p G (\xb, \pb)
\\& \ \ \ \ + \varepsilon \nabla_{x, p} \xb (x,p) \cdot e^{\frac{\tilde{\beta}}{2} |\tpb^0| } \nabla_{x_\parallel} h (\xb, \tpb) + \varepsilon \nabla_{x, p} \tpb (x,p) \cdot e^{\frac{\tilde{\beta}}{2} |\tpb^0| } \nabla_p h (\xb, \tpb).
\end{split}
\Ee

\smallskip

\textbf{Step 1. Proof of \eqref{est:hk_v_spec}.} 
From \eqref{condition:h_support_spec}, \eqref{weight_form:nabla_h_spec} and $\tpb = (p_{\mathbf{b}, 1}, p_{\mathbf{b}, 2}, - p_{\mathbf{b}, 3})$, we get
\be \label{est1:hk_v_spec}
\begin{split}
& e^{\frac{\tilde{\beta}}{2} \big( |p^0| + \frac{1}{c} ( e_{\pm} \Phi (x) + m_{\pm} g x_3 ) \big)} | \nabla_{p} h (x,p) | 
\\& \leq \Big( \frac{ | \nabla_p \xb (x,p) |}{e^{\frac{\tilde{\beta}}{2} |\pb^0 (x,p)|}} + \frac{ |\nabla_p \pb (x,p)|}{e^{\frac{\tilde{\beta}}{2} |\pb^0 (x,p)|}} \Big) \| e^{\tilde{\beta} |p^0|}   \nabla_{x_\parallel, p} G \|_{L^\infty (\gamma_-)}
\\& \ \ \ \ + \Big( \mathbf{1}_{|\pb| < \px} \varepsilon | \nabla_{p} \xb | \Big) \| e^{\frac{\tilde{\beta}}{2} \big( |p^0| + \frac{1}{c} ( e_{\pm} \Phi (x) + m_{\pm} g x_3 ) \big)} \nabla_{x_\parallel} h (x, p) \|_{L^\infty (\bar \O \times \R^3)}
\\& \ \ \ \ + \Big( \mathbf{1}_{|\pb| < \px} \varepsilon | \nabla_{p} \pb | \Big) \| e^{\frac{\tilde{\beta}}{2} \big( |p^0| + \frac{1}{c} ( e_{\pm} \Phi (x) + m_{\pm} g x_3 ) \big)} \nabla_{p} h (x, p) \|_{L^\infty (\bar \O \times \R^3)},
\end{split}
\Ee
and
\be \label{est1:hk_x_spec}
\begin{split}
& e^{\frac{\tilde{\beta}}{2} \big( |p^0| + \frac{1}{c} ( e_{\pm} \Phi (x) + m_{\pm} g x_3 ) \big)} | \nabla_{x} h (x,p) | 
\\& \leq \Big( \frac{ | \nabla_x \xb (x,p) |}{e^{\frac{\tilde{\beta}}{2} |\pb^0 (x,p)|}} + \frac{ |\nabla_x \pb (x,p)|}{e^{\frac{\tilde{\beta}}{2} |\pb^0 (x,p)|}} \Big) \| e^{\tilde{\beta} |p^0|}   \nabla_{x_\parallel, p} G \|_{L^\infty (\gamma_-)}
\\& \ \ \ \ + \Big( \mathbf{1}_{|\pb| < \px} \varepsilon | \nabla_{x} \xb | \Big) \| e^{\frac{\tilde{\beta}}{2} \big( |p^0| + \frac{1}{c} ( e_{\pm} \Phi (x) + m_{\pm} g x_3 ) \big)} \nabla_{x_\parallel} h (x, p) \|_{L^\infty (\bar \O \times \R^3)}
\\& \ \ \ \ + \Big( \mathbf{1}_{|\pb| < \px} \varepsilon | \nabla_{x} \pb | \Big) \| e^{\frac{\tilde{\beta}}{2} \big( |p^0| + \frac{1}{c} ( e_{\pm} \Phi (x) + m_{\pm} g x_3 ) \big)} \nabla_{p} h (x, p) \|_{L^\infty (\bar \O \times \R^3)}.
\end{split}
\Ee
Combining \eqref{est1:hk_v_spec} and \eqref{est1:hk_x_spec}, we derive that
\begin{align}
& e^{\frac{\tilde{\beta}}{2} \big( |p^0| + \frac{1}{c} ( e_{\pm} \Phi (x) + m_{\pm} g x_3 ) \big)} \big( | \nabla_{p} h (x,p) | + | \nabla_{x_\parallel} h (x,p) | \big) 
\notag \\
& \leq \Big( \frac{ | \nabla_p \xb (x,p) |}{e^{\frac{\tilde{\beta}}{2} |\pb^0 (x,p)|}} + \frac{ |\nabla_p \pb (x,p)|}{e^{\frac{\tilde{\beta}}{2} |\pb^0 (x,p)|}} \Big) \| e^{\tilde{\beta} |p^0|}  \nabla_{x_\parallel, p} G \|_{L^\infty (\gamma_-)}
\label{est1:hk_v+x_spec_1} \\
& \ \ \ \ + \Big( \frac{ | \nabla_{x_\parallel} \xb (x,p) |}{e^{\frac{\tilde{\beta}}{2} |\pb^0 (x,p)|}} + \frac{ |\nabla_{x_\parallel} \pb (x,p)|}{e^{\frac{\tilde{\beta}}{2} |\pb^0 (x,p)|}} \Big) \| e^{\tilde{\beta} |p^0|}  \nabla_{x_\parallel, p} G \|_{L^\infty (\gamma_-)}
\label{est1:hk_v+x_spec_2} \\
& \ \ \ \ + \mathbf{1}_{|\pb| < \px} \Big( \varepsilon | \nabla_{p} \xb | + \varepsilon | \nabla_{x_\parallel} \xb | \Big) 
\| e^{\frac{\tilde{\beta}}{2} \big( |p^0| + \frac{1}{c} ( e_{\pm} \Phi (x) + m_{\pm} g x_3 ) \big)} \nabla_{x_\parallel} h (x, p) \|_{L^\infty (\bar \O \times \R^3)}
\label{est1:hk_v+x_spec_3} \\
& \ \ \ \ + \mathbf{1}_{|\pb| < \px} \Big( \varepsilon | \nabla_{p} \pb | + \varepsilon | \nabla_{x_\parallel} \pb | \Big) 
\| e^{\frac{\tilde{\beta}}{2} \big( |p^0| + \frac{1}{c} ( e_{\pm} \Phi (x) + m_{\pm} g x_3 ) \big)} \nabla_{p} h (x, p) \|_{L^\infty (\bar \O \times \R^3)}.
\label{est1:hk_v+x_spec_4}
\end{align}
Replacing $\tilde{\beta}$ with $\tilde{\beta} / 2$ in \eqref{est1:xb_v} and \eqref{est1:vb_v}, together with \eqref{choice_tbeta_spec}, we get
\be \label{est1:xb+pb_v_spec}
\frac{ | \nabla_p \xb (x,p) |}{e^{\frac{\tilde{\beta}}{2} |\pb^0 (x,p)|}} + \frac{ |\nabla_p \pb (x,p)|}{e^{\frac{\tilde{\beta}}{2} |\pb^0 (x,p)|}} 
\lesssim 1 + \frac{ \| \nabla_x^2 \Phi  \|_\infty + e B_3 + mg}{(m g)^2}.
\ee
Analogously, replacing $\tilde{\beta}$ with $\tilde{\beta} / 2$ in \eqref{est1:xb_x} and \eqref{est1:vb_x}, together with \eqref{choice_tbeta_spec}, we get
\be \label{est1:xb+pb_x12_spec}
\frac{ | \nabla_{x_\parallel} \xb (x,p) |}{e^{\frac{\tilde{\beta}}{2} |\pb^0 (x,p)|}} + \frac{ |\nabla_{x_\parallel} \pb (x,p)|}{e^{\frac{\tilde{\beta}}{2} |\pb^0 (x,p)|}} 
\lesssim \frac{ \| \nabla_x^2 \Phi  \|_\infty + e B_3 + mg}{(m g)^2}.
\ee
Inputting \eqref{est1:xb+pb_v_spec} and \eqref{est1:xb+pb_x12_spec} into \eqref{est1:hk_v+x_spec_1} and \eqref{est1:hk_v+x_spec_2} respectively, we obtain
\be \label{est2:hk_v+x_spec}
\eqref{est1:hk_v+x_spec_1} + \eqref{est1:hk_v+x_spec_2}
\lesssim \big( 1 + \frac{ \| \nabla_x^2 \Phi  \|_\infty + e B_3 + m_{\pm} g}{(m_{\pm} g)^2} \big) \| e^{\tilde \beta \sqrt{(m_{\pm} c)^2 + |p|^2}} \nabla_{x_\parallel, p} G_{\pm} \|_{L^\infty (\gamma_-)}.
\ee

On the other hand, \eqref{est1:xb+pb_v_spec} and \eqref{est1:xb+pb_x12_spec}, together with \eqref{condition2:epilon_G_spec}, show
\be \label{est3:hk_v+x_spec_1}
\begin{split}
& \mathbf{1}_{|\pb| < \px} \Big( \varepsilon | \nabla_{p} \xb | + \varepsilon | \nabla_{x_\parallel} \xb | \Big) 
\\& \lesssim \mathbf{1}_{|\pb| < \px} \varepsilon \big( 1 + \frac{ \| \nabla_x^2 \Phi  \|_\infty + e B_3 + mg}{(m g)^2} \big) {e^{\frac{\tilde{\beta}}{2} |\pb^0 (x,p)|}}
\\& \lesssim \varepsilon \big( 1 + \frac{ \| \nabla_x^2 \Phi  \|_\infty + e B_3 + mg}{(m g)^2} \big) {e^{\frac{\tilde{\beta}}{2} \sqrt{(m c)^2 + |\px|^2}}} \leq \frac{1}{4},
\end{split}
\ee
and
\be \label{est3:hk_v+x_spec_2}
\mathbf{1}_{|\pb| < \px} \Big( \varepsilon | \nabla_{p} \pb | + \varepsilon | \nabla_{x_\parallel} \pb | \Big) 
\lesssim \varepsilon \big( 1 + \frac{ \| \nabla_x^2 \Phi \|_\infty + e B_3 + mg}{(m g)^2} \big) {e^{\frac{\tilde{\beta}}{2} \sqrt{(m c)^2 + |\px|^2}}} \leq \frac{1}{4}.
\ee
Inputting \eqref{est3:hk_v+x_spec_1} and \eqref{est3:hk_v+x_spec_2} into \eqref{est1:hk_v+x_spec_3} and \eqref{est1:hk_v+x_spec_4} respectively, we obtain
\be \label{est4:hk_v+x_spec}
\begin{split}
\eqref{est1:hk_v+x_spec_3} + \eqref{est1:hk_v+x_spec_4}
& \leq \frac{1}{4} 
\Big( \| e^{\frac{\tilde{\beta}}{2} \big( |p^0| + \frac{1}{c} ( e_{\pm} \Phi (x) + m_{\pm} g x_3 ) \big)} \nabla_{x_\parallel} h (x, p) \|_{L^\infty (\bar \O \times \R^3)} 
\\& \qquad \ \ + \| e^{\frac{\tilde{\beta}}{2} \big( |p^0| + \frac{1}{c} ( e_{\pm} \Phi (x) + m_{\pm} g x_3 ) \big)} \nabla_{p} h (x, p) \|_{L^\infty (\bar \O \times \R^3)} \Big).
\end{split}
\ee
Combining \eqref{est2:hk_v+x_spec} and \eqref{est4:hk_v+x_spec}, we obtain that
\be \label{est5:hk_v+x_spec}
\begin{split}
& \| e^{\frac{\tilde{\beta}}{2} \big( |p^0| + \frac{1}{c} ( e_{\pm} \Phi (x) + m_{\pm} g x_3 ) \big)} \nabla_{x_{\parallel}, p} h (x, p) \|_{L^\infty (\bar \O \times \R^3)} 
\\& \lesssim \big( 1 + \frac{ \| \nabla_x^2 \Phi  \|_\infty + e B_3 + m_{\pm} g}{(m_{\pm} g)^2} \big) \| e^{\tilde \beta \sqrt{(m_{\pm} c)^2 + |p|^2}} \nabla_{x_\parallel, p} G_{\pm} \|_{L^\infty (\gamma_-)}
\\& \ \ \ \ + \frac{1}{4} 
\Big( \| e^{\frac{\tilde{\beta}}{2} \big( |p^0| + \frac{1}{c} ( e_{\pm} \Phi (x) + m_{\pm} g x_3 ) \big)} \nabla_{x_\parallel} h (x, p) \|_{L^\infty (\bar \O \times \R^3)} 
\\& \qquad \ \ + \| e^{\frac{\tilde{\beta}}{2} \big( |p^0| + \frac{1}{c} ( e_{\pm} \Phi (x) + m_{\pm} g x_3 ) \big)} \nabla_{p} h (x, p) \|_{L^\infty (\bar \O \times \R^3)} \Big).
\end{split}
\ee
Uniting two weight $L^{\infty}$ estimates on $\nabla_{x_{\parallel}} h$ and $\nabla_{p} h$ in \eqref{est5:hk_v+x_spec}, we have
\be \label{est6:hk_v+x_spec}
\begin{split}
& \frac{1}{2} 
\Big( \| e^{\frac{\tilde{\beta}}{2} \big( |p^0| + \frac{1}{c} ( e_{\pm} \Phi (x) + m_{\pm} g x_3 ) \big)} \nabla_{x_\parallel} h (x, p) \|_{L^\infty (\bar \O \times \R^3)} 
\\& \qquad \ \ + \| e^{\frac{\tilde{\beta}}{2} \big( |p^0| + \frac{1}{c} ( e_{\pm} \Phi (x) + m_{\pm} g x_3 ) \big)} \nabla_{p} h (x, p) \|_{L^\infty (\bar \O \times \R^3)} \Big) 
\\& \lesssim 2 \big( 1 + \frac{ \| \nabla_x^2 \Phi  \|_\infty + e B_3 + m_{\pm} g}{(m_{\pm} g)^2} \big) \| e^{\tilde \beta \sqrt{(m_{\pm} c)^2 + |p|^2}} \nabla_{x_\parallel, p} G_{\pm} \|_{L^\infty (\gamma_-)}.
\end{split}
\ee
Furthermore, the condition \eqref{Uest:DPhi_spec} and Lemma \ref{lem:conservation_law} show that
\be \label{lower:vb_spec}
\pb^0 (x,p) = \sqrt{(m_{\pm} c)^2 + |\pb (x,p)|^2} \geq \sqrt{(m_{\pm} c)^2 + |p|^2} + \frac{m g}{2 c} x_3.
\ee
Together with \eqref{est6:hk_v+x_spec} and \eqref{lower:vb_spec}, we conclude \eqref{est:hk_v_spec}.

\smallskip

\textbf{Step 2. Proof of \eqref{est:hk_x_spec}.} 
The estimate \eqref{est6:hk_v+x_spec} indicates that
\be \notag
\begin{split}
& \frac{1}{2} 
\| e^{\frac{\tilde{\beta}}{2} \big( |p^0| + \frac{1}{c} ( e_{\pm} \Phi (x) + m_{\pm} g x_3 ) \big)} \nabla_{x_\parallel} h (x, p) \|_{L^\infty (\bar \O \times \R^3)} 
\\& \lesssim 2 \big( 1 + \frac{ \| \nabla_x^2 \Phi  \|_\infty + e B_3 + m_{\pm} g}{(m_{\pm} g)^2} \big) \| e^{\tilde \beta \sqrt{(m_{\pm} c)^2 + |p|^2}} \nabla_{x_\parallel, p} G_{\pm} \|_{L^\infty (\gamma_-)}.
\end{split}
\ee
Hence, we only need to bound the following:
\be \notag
e^{ \frac{\tilde \beta}{2}|p^0_{\pm}|} e^{  \frac{\tilde \beta m_{\pm} g}{4 c} x_3} | \nabla_{x_3} h_{\pm} (x,p)|.
\ee
Again from \eqref{est1:hk_x_spec}, we get
\be \label{est1:hk_x3_spec}
\begin{split}
& e^{\frac{\tilde{\beta}}{2} \big( |p^0| + \frac{1}{c} ( e_{\pm} \Phi (x) + m_{\pm} g x_3 ) \big)} | \nabla_{x_3} h (x,p) | 
\\& \leq \Big( \frac{ | \nabla_{x_3} \xb (x,p) |}{e^{\frac{\tilde{\beta}}{2} |\pb^0 (x,p)|}} + \frac{ |\nabla_{x_3} \pb (x,p)|}{e^{\frac{\tilde{\beta}}{2} |\pb^0 (x,p)|}} \Big) \| e^{\tilde{\beta} |p^0|}   \nabla_{x_\parallel, p} G \|_{L^\infty (\gamma_-)}
\\& \ \ \ \ + \Big( \mathbf{1}_{|\pb| < \px} \varepsilon | \nabla_{x_3} \xb | \Big) \| e^{\frac{\tilde{\beta}}{2} \big( |p^0| + \frac{1}{c} ( e_{\pm} \Phi (x) + m_{\pm} g x_3 ) \big)} \nabla_{x_\parallel} h (x, p) \|_{L^\infty (\bar \O \times \R^3)}
\\& \ \ \ \ + \Big( \mathbf{1}_{|\pb| < \px} \varepsilon | \nabla_{x_3} \pb | \Big) \| e^{\frac{\tilde{\beta}}{2} \big( |p^0| + \frac{1}{c} ( e_{\pm} \Phi (x) + m_{\pm} g x_3 ) \big)} \nabla_{p} h (x, p) \|_{L^\infty (\bar \O \times \R^3)}.
\end{split}
\Ee
Replacing $\tilde{\beta}$ with $\tilde{\beta} / 2$ in \eqref{est1:xb_x} and \eqref{est1:vb_x}, together with \eqref{choice_tbeta_spec}, we get
\be \label{est1:xb+pb_x3_spec}
\frac{ | \nabla_{x_3} \xb (x,p) |}{e^{\frac{\tilde{\beta}}{2} |\pb^0 (x,p)|}} + \frac{ |\nabla_{x_3} \pb (x,p)|}{e^{\frac{\tilde{\beta}}{2} |\pb^0 (x,p)|}} 
\lesssim \frac{1}{\alpha (x,p)} + \frac{ \| \nabla_x^2 \Phi  \|_\infty + e B_3 + mg}{(m g)^2}.
\ee
Together with \eqref{est6:hk_v+x_spec}, we obtain
\be \label{est2:hk_x3_spec}
\begin{split}
& e^{\frac{\tilde{\beta}}{2} \big( |p^0| + \frac{1}{c} ( e_{\pm} \Phi (x) + m_{\pm} g x_3 ) \big)} | \nabla_{x_3} h (x,p) | 
\\& \leq \Big( \frac{\delta_{i3}}{\alpha (x,p)} + \frac{ \| \nabla_x^2 \Phi  \|_\infty + e B_3 + mg}{(m g)^2} \Big) \| e^{\tilde{\beta} |p^0|}   \nabla_{x_\parallel, p} G \|_{L^\infty (\gamma_-)}
\\& \ \ \ \ + \Big( \mathbf{1}_{|\pb| < \px} \varepsilon | \nabla_{x_3} \xb | + \mathbf{1}_{|\pb| < \px} \varepsilon | \nabla_{x_3} \pb | \Big) 
\\& \qquad \ \ \times \big( 1 + \frac{ \| \nabla_x^2 \Phi  \|_\infty + e B_3 + m_{\pm} g}{(m_{\pm} g)^2} \big) \| e^{\tilde \beta \sqrt{(m_{\pm} c)^2 + |p|^2}} \nabla_{x_\parallel, p} G_{\pm} \|_{L^\infty (\gamma_-)}.
\end{split}
\Ee

On the other hand, using \eqref{est2:hk_x3_spec} and \eqref{condition2:epilon_G_spec}, we obtain
\be \label{est3:hk_x3_spec}
\begin{split}
& \mathbf{1}_{|\pb| < \px} \Big( \varepsilon | \nabla_{x_3} \xb | + \varepsilon | \nabla_{x_3} \pb | \Big) 
\\& \lesssim \mathbf{1}_{|\pb| < \px} \varepsilon \big( \frac{1}{\alpha (x,p)} + \frac{ \| \nabla_x^2 \Phi  \|_\infty + e B_3 + mg}{(m g)^2} \big) {e^{\frac{\tilde{\beta}}{2} |\pb^0 (x,p)|}}
\lesssim \frac{1}{\alpha (x,p)} + \frac{1}{4}.
\end{split}
\ee
Inputting \eqref{est3:hk_x3_spec} into \eqref{est2:hk_x3_spec}, we derive that
\be \label{est4:hk_x3_spec}
\begin{split}
& e^{\frac{\tilde{\beta}}{2} \big( |p^0| + \frac{1}{c} ( e_{\pm} \Phi (x) + m_{\pm} g x_3 ) \big)} | \nabla_{x_3} h (x,p) | 
\\& \leq \Big( \frac{1}{\alpha (x,p)} + \frac{ \| \nabla_x^2 \Phi  \|_\infty + e B_3 + mg}{(m g)^2} \Big) \| e^{\tilde{\beta} |p^0|}   \nabla_{x_\parallel, p} G \|_{L^\infty (\gamma_-)}
\\& \ \ \ \ + \Big( \frac{1}{\alpha (x,p)} + \frac{1}{4} \Big) \times \big( 1 + \frac{ \| \nabla_x^2 \Phi  \|_\infty + e B_3 + m_{\pm} g}{(m_{\pm} g)^2} \big) \| e^{\tilde \beta \sqrt{(m_{\pm} c)^2 + |p|^2}} \nabla_{x_\parallel, p} G_{\pm} \|_{L^\infty (\gamma_-)}.
\end{split}
\Ee
Together with \eqref{lower:vb_spec}, we conclude \eqref{est:hk_x_spec}.

\smallskip

\textbf{Step 3. Proof of \eqref{est:rho_x_spec}.}
Recall \eqref{est1:hk_x_spec} and Lemma \ref{lem:conservation_law}, we have
\be \label{est2:hk_x_spec}
\begin{split}
| \nabla_{x} h (x,p) | 
& \leq \Big( \frac{ | \nabla_x \xb (x,p) |}{e^{\tilde{\beta} |\pb^0 (x,p)|}} + \frac{ |\nabla_x \pb (x,p)|}{e^{\tilde{\beta} |\pb^0 (x,p)|}} \Big) \| e^{\tilde{\beta} |p^0|}   \nabla_{x_\parallel, p} G \|_{L^\infty (\gamma_-)}
\\& \ \ \ \ + \Big( \frac{ \mathbf{1}_{|\pb| < \px} \varepsilon | \nabla_{x} \xb | }{{e^{\frac{\tilde{\beta}}{2} |\pb^0 (x,p)|}}} \Big) \| e^{\frac{\tilde{\beta}}{2} \big( |p^0| + \frac{1}{c} ( e_{\pm} \Phi (x) + m_{\pm} g x_3 ) \big)} \nabla_{x_\parallel} h (x, p) \|_{L^\infty (\bar \O \times \R^3)}
\\& \ \ \ \ + \Big( \frac{ \mathbf{1}_{|\pb| < \px} \varepsilon | \nabla_{x} \pb | }{{e^{\frac{\tilde{\beta}}{2} |\pb^0 (x,p)|}}} \Big) \| e^{\frac{\tilde{\beta}}{2} \big( |p^0| + \frac{1}{c} ( e_{\pm} \Phi (x) + m_{\pm} g x_3 ) \big)} \nabla_{p} h (x, p) \|_{L^\infty (\bar \O \times \R^3)}.
\end{split}
\Ee
From \eqref{def:rho_spec} and \eqref{est2:hk_x_spec}, together with \eqref{est6:hk_v+x_spec}, we have
\begin{align}
& \ \ \ \ | \nabla_{x_i} \rho (x) | 
\notag
\\& = \Big| \int_{\R^3} e_+ \nabla_{x_i} h_+ (x,p) + e_{-} \nabla_{x_i} h_{-} (x,p) \dd p \Big| 
\notag
\\& \lesssim \sum_{j = \pm} e_{j} \| e^{\tilde \beta |p^0_j|} \nabla_{x_\parallel, p} G_j \|_{L^\infty (\gamma_-)}  
\times \bigg\{ 
\int_{\R^3} \frac{| \nabla_{x_i} \xbj (x,p) |}{ e^{\tilde \beta |\pbj^0 (x,p)| }} \dd p
 + \int_{\R^3} \frac{| \nabla_{x_i} \pbj (x,p) |}{ e^{\tilde \beta |\pbj^0 (x,p)| }} \dd p 
\bigg\} 
\label{est1:rho_x_spec_1}
\\& \ \ \ \ + \sum_{j = \pm} e_{j} \| e^{\tilde \beta |p^0_j|} \nabla_{x_\parallel, p} G_j \|_{L^\infty (\gamma_-)}  
\times \int_{\R^3} \frac{ \mathbf{1}_{|\pb| < \px} \varepsilon | \nabla_{x_i} \xbj | }{{e^{\frac{\tilde{\beta}}{2} |\pbj^0 (x,p)|}}} \dd p
\label{est1:rho_x_spec_2}
\\& \ \ \ \ + \sum_{j = \pm} e_{j} \| e^{\tilde \beta |p^0_j|} \nabla_{x_\parallel, p} G_j \|_{L^\infty (\gamma_-)}  
\times \int_{\R^3} \frac{ \mathbf{1}_{|\pb| < \px} \varepsilon | \nabla_{x_i} \pbj | }{{e^{\frac{\tilde{\beta}}{2} |\pbj^0 (x,p)|}}} \dd p.
\label{est1:rho_x_spec_3}
\end{align}

Following the estimates in \eqref{est1:rho_x}-\eqref{est2:rho_x}, we derive that
\be \label{est2:rho_x_spec_1}
\eqref{est1:rho_x_spec_1}
\lesssim \sum_{j = \pm} \| e^{\tilde \beta |p^0_j|^2} \nabla_{x_\parallel, p} G_j \|_{L^\infty (\gamma_-)}  
\times e^{ - \frac{\tilde \beta m_j g}{4 c} x_3 } \Big(
1 + \mathbf{1}_{|x_3| \leq 1} \frac{\delta_{i3}}{\sqrt{ m_j g x_3 }}
\Big).
\ee
Using \eqref{est:xb_x/w1}-\eqref{est2:xb_x/w}, together with  $\tilde{\beta}$ in \eqref{choice_tbeta_spec} and $\varepsilon$ in \eqref{condition2:epilon_G_spec}, we get 
\Be \label{est:xb_x/w_spec}
\begin{split}
\frac{ \mathbf{1}_{|\pb| < \px} \varepsilon | \nabla_{x_i} \xbj | }{{e^{\frac{\tilde{\beta}}{2} |\pbj^0 (x,p)|}}}
\lesssim
\frac{ \delta_{i3} }{\alpha (x,p)} e^{ - \frac{\tilde \beta}{2} \big( |p^0| + \frac{m g}{2 c} x_3 \big) } 
+ e^{- \frac{\tilde \beta}{2} \big( |p^0| + \frac{m g}{2 c} x_3 \big)}.
\end{split}
\Ee
Following \eqref{est:1/alpha}-\eqref{est5:xb_x/w}, we have
\be \label{est1:xb_x/w_spec}
\int_{\R^3} \frac{ \mathbf{1}_{|\pb| < \px} \varepsilon | \nabla_{x_i} \xb | }{{e^{\frac{\tilde{\beta}}{2} |\pbj^0 (x,p)|}}} \dd p
\lesssim e^{ - \frac{\tilde \beta m g}{4 c} x_3 } \Big(
1 + \mathbf{1}_{|x_3| \leq 1} \frac{ \delta_{i3} }{\sqrt{ m g x_3 }}
\Big).
\ee
From \eqref{est:vb_x/w}-\eqref{est1:vb_x/w}, together with  $\tilde{\beta}$ in \eqref{choice_tbeta_spec} and $\varepsilon$ in \eqref{condition2:epilon_G_spec}, we get 
\Be \label{est1:vb_x/w_spec}
\int_{\R^3} \frac{ \mathbf{1}_{|\pb| < \px} \varepsilon | \nabla_{x_i} \pb | }{{e^{\frac{\tilde{\beta}}{2} |\pbj^0 (x,p)|}}} \dd p
\lesssim e^{ - \frac{\tilde \beta m g}{4 c} x_3 } \Big(
1 + \mathbf{1}_{|x_3| \leq 1} \frac{\delta_{i3}}{\sqrt{ m g x_3 }}
\Big).
\Ee

Inputting \eqref{est1:xb_x/w_spec} and \eqref{est1:vb_x/w_spec} into \eqref{est1:rho_x_spec_2} and \eqref{est1:rho_x_spec_3}, together with \eqref{est2:rho_x_spec_1}, we have
\Be \notag
\begin{split}
| \nabla_{x_i} \rho (x) | 
\lesssim \sum_{j = \pm} \| e^{\tilde \beta |p^0_j|^2} \nabla_{x_\parallel, p} G_j \|_{L^\infty (\gamma_-)}  
\times e^{ - \frac{\tilde \beta m_j g}{4 c} x_3 } \Big(
1 + \mathbf{1}_{|x_3| \leq 1} \frac{\delta_{i3}}{\sqrt{ m_j g x_3 }}
\Big).
\end{split}
\Ee 
and conclude \eqref{est:rho_x_spec}.

\smallskip

\textbf{Step 4. Proof of \eqref{est:phi_C2_spec}.}
From the estimate on $| \nabla_{x_i} \rho (x) |$ in \eqref{est:rho_x_spec}, following \eqref{rho_DQ}-\eqref{est:rho_Hol} we obtain
\Be \label{est:rho_Hol_spec}
\begin{split}
|\rho |_{C^{0,\delta}(\O)} 
\lesssim_\delta
\sum_{i = \pm} \| e^{\tilde \beta |p^0| } \nabla_{x_\parallel, p} G_i \|_{L^\infty (\gamma_-)}.
\end{split}
\Ee
On the other hand, from \eqref{est1:rho_Hol} and \eqref{choice_beta_spec}, we have
\Be \label{est1:rho_Hol_spec}
\begin{split}
| \rho (x) | 
& = \Big| \int_{\R^3} e_+ h_+ (x,p) + e_{-} h_{-} (x,p) \dd p \Big| 
\\& \leq \sum_{i = \pm} e_{i} \int_{\R^3} \frac{1}{w_{i, \beta} (x, p)} \dd p \times \| w_{i, \beta} G_i \|_{L^\infty(\gamma_-)}
\\& \leq \frac{1}{\beta} \big( e_+ \| w_{+, \beta} G_+ \|_{L^\infty(\gamma_-)} e^{-  \beta \frac{ m_+ }{2 c} g x_3} + e_{-} \| w_{-, \beta} G_- \|_{L^\infty(\gamma_-)} e^{- \beta \frac{ m_{-} }{2 c} g x_3} \big).
\end{split}
\Ee 
Using \eqref{est:nabla^2phi} and \eqref{est:rho_Hol_spec}, together with \eqref{est1:rho_Hol_spec}, we derive
\Be  \notag
\begin{split}
\| \nabla_x^2 \Phi \|_{L^\infty(\O)} 
& \lesssim_\delta \|\rho \|_{L^\infty(\O)} + \|\rho \|_{C^{0,\delta}(\O)} + \frac{1}{\beta} \big( e_+ \| w_{+, \beta} G_+ \|_{L^\infty(\gamma_-)} + e_{-} \| w_{-, \beta} G_- \|_{L^\infty(\gamma_-)} \big),
\end{split}
\Ee 
and thus conclude \eqref{est:phi_C2_spec}.
\end{proof}

Collecting the results in Proposition \ref{prop:Reg_spec},
we conclude the following regularity estimate.

\begin{theorem}[Regularity Estimate] \label{theo:RS_spec}
 
Suppose $(h, \rho_h, \Phi_h )$ solves \eqref{VP_h_spec}-\eqref{eqtn:Dphi_spec} in the sense of Definition \ref{weak_sol_spec}. Suppose \eqref{Uest:DPhi_spec} holds. 
Assume that there exists $\px > 0$ such that 
\be \notag
h_{\pm} (x, p) = 0 
\ \text{ for any } \
|p| \geq \px.
\ee
There exists $\beta, \tilde \beta >0$ such that
\be \notag
\| w_{\pm, \beta} h_{\pm} \|_{L^\infty (\bar \O \times \R^3)} 
\leq (1 + \varepsilon ) \| e^{ \beta \sqrt{(m_{\pm} c)^2 + |p|^2}} G_{\pm} \|_{L^\infty (\gamma_-)},
\ee
and
\be \notag
\| e^{{\tilde \beta } \sqrt{(m_{\pm} c)^2 + |p|^2}} \nabla_{x_\parallel, p} G_{\pm} (x,p) \|_{L^\infty (\gamma_-)} < \infty,
\ee 
Furthermore, let $\beta,\tilde \beta >0$ satisfy  
\Be \label{condition:G_spec}
\begin{split}
& \frac{1}{\beta} \big( e_+ \| e^{ \beta \sqrt{(m_{+} c)^2 + |p|^2} } G_+ \|_{L^\infty(\gamma_-)} + e_{-} \| e^{ \beta \sqrt{(m_{\-} c)^2 + |p|^2}} G_- \|_{L^\infty(\gamma_-)} \big)
\\& \ \ \ \ + \| e^{\tilde \beta \sqrt{(m_{+} c)^2 + |p|^2}} \nabla_{x_\parallel, p} G_+ \|_{L^\infty (\gamma_-)} + \| e^{\tilde \beta \sqrt{(m_{-} c)^2 + |p|^2}} \nabla_{x_\parallel, p} G_- \|_{L^\infty (\gamma_-)}
\\& \leq \frac{m_{\pm} g}{16} \tilde \beta - (1 + B_3).
\end{split}	
\Ee
and
\Be \label{condition:Phi_xx_spec}
\frac{8}{m_{\pm} g} (1 + B_3 + \| \nabla_x ^2 \Phi  \|_\infty) \leq \frac{\tilde \beta}{2} \leq \frac{\beta}{2}.
\Ee
Finally, suppose $\varepsilon$ in the boundary condition \eqref{bdry:h_spec} satisfies that	
\be \notag
\varepsilon \big( 1 + \frac{ \| \nabla_x^2 \Phi  \|_\infty + e B_3 + mg}{(m g)^2} \big) {e^{\frac{\tilde{\beta}}{2} \sqrt{(m c)^2 + |\px|^2}}} \leq \frac{1}{4}.
\ee	
Then, $h, \rho$ and $\nabla_x \Phi $ are locally Lipschitz continuous and the following bounds hold:
\Be \label{theo:rho_x_spec}
\begin{split}
e^{ \frac{\tilde \beta m g}{4 c} x_3 } |\p_{x_i} \rho  (x)|   
& \lesssim \| e^{\tilde \beta \sqrt{(m_+ c)^2 + |p|^2}} \nabla_{x_\parallel, p} G_+ \|_{L^\infty (\gamma_-)}  
\times \Big(
1 + \mathbf{1}_{|x_3| \leq 1} \frac{1}{\sqrt{ m_+ g x_3 }}
\Big)
\\& \ \ \ \ + \| e^{\tilde \beta \sqrt{(m_- c)^2 + |p|^2}} \nabla_{x_\parallel, p} G_- \|_{L^\infty (\gamma_-)}  
\times \Big(
1 + \mathbf{1}_{|x_3| \leq 1} \frac{1}{\sqrt{ m_- g x_3 }}
\Big),
\end{split}
\Ee
and
\Be \label{theo:phi_C2_spec}
\begin{split}
\| \nabla_x^2 \Phi  \|_\infty
& \lesssim \frac{1}{\beta} \big( e_+ \| w_{+, \beta} G_+ \|_{L^\infty(\gamma_-)} + e_{-} \| w_{-, \beta} G_- \|_{L^\infty(\gamma_-)} \big)
\\& \ \ \ \ + \| e^{\tilde \beta \sqrt{(m_+ c)^2 + |p|^2}} \nabla_{x_\parallel, p} G_+ \|_{L^\infty (\gamma_-)} + \| e^{\tilde \beta \sqrt{(m_- c)^2 + |p|^2}} \nabla_{x_\parallel, p} G_- \|_{L^\infty (\gamma_-)}.
\end{split}	
\Ee
Furthermore,
\Be \label{theo:hk_v_spec}
\begin{split}
& e^{ \frac{\tilde \beta}{2}|p^0|} e^{  \frac{\tilde \beta m_{\pm} g}{4 c} x_3} | \nabla_p h_{\pm} (x,p)| 
\\& \lesssim \big( 1 + \frac{ \| \nabla_x^2 \Phi  \|_\infty + e B_3 + mg}{(m g)^2} \big) \| e^{\tilde \beta \sqrt{(m_{\pm} c)^2 + |p|^2}} \nabla_{x_\parallel, p} G_{\pm} \|_{L^\infty (\gamma_-)},
\end{split}
\Ee
and
\Be \label{theo:hk_x_spec}
\begin{split}
& e^{ \frac{\tilde \beta}{2}|p^0|} e^{  \frac{\tilde \beta m_{\pm} g}{4 c} x_3} | \nabla_x h_{\pm} (x,p)|  
\\& \lesssim \big( \frac{\delta_{i3}}{\alpha  (x,p)} + \frac{ \| \nabla_x^2 \Phi  \|_\infty + e B_3 + mg}{(m g)^2} \big) \| e^{\tilde \beta \sqrt{(m_{\pm} c)^2 + |p|^2}} \nabla_{x_\parallel, p} G_{\pm} \|_{L^\infty (\gamma_-)}.
\end{split}
\Ee
\end{theorem}

\begin{proof}

Using the condition \eqref{condition:G_spec}, we can compute that \eqref{choice_beta_spec}-\eqref{est:phi_C2_spec} hold.
We omit the rest of the proof, since it follows directly from Proposition \ref{prop:Reg_spec}.
\end{proof}

\subsubsection{Uniqueness Theorem} \label{sec:US_spec}

In this section, we prove Theorem \ref{theo:US_spec}: a conditional uniqueness theorem subjected to a regularity of the solutions. 

\begin{theorem}[Uniqueness Theorem] \label{theo:US_spec}

Let $(h_{1, \pm}, \rho_1, \Phi_1)$ and $(h_{2, \pm}, \rho_2, \Phi_2)$ solve \eqref{VP_h_spec}-\eqref{eqtn:Dphi_spec} in the sense of Definition \ref{weak_sol_spec}. 
Assume both follow all assumptions in Theorem \ref{theo:CS_spec} and satisfy \eqref{Uest:wh_spec}-\eqref{Uest:DPhi_spec}. 
Further, assume there exist $\bar \e, \bar \beta > 0$, such that
\Be \label{condition_unique_spec}
\| w_{\pm, \bar \beta} \nabla_p h_{i, \pm} \|_\infty < \bar \e \beta^2 
\ \ \text{for }  i = 1, 2,
\Ee
where $w_{\pm, \bar \beta}(x,v)$ is defined in \eqref{w^h}.
Then $h_{1, \pm} = h_{2, \pm}$ a.e. in $\bar \O \times \R^3$ and $\Phi_1 = \Phi_2$ a.e. in $\bar \O$. 
\end{theorem}

\begin{proof}

We only prove the case when the condition \eqref{condition_unique_spec} holds for $i= 2$, since the case when \eqref{condition_unique_spec} holds for  $i = 1$ can be deduced in a similar way. 
For the sake of simplicity, we abuse the notation as in \eqref{abuse}.

\smallskip

Since both $(h_{1, \pm}, \rho_1, \Phi_1)$ and $(h_{2, \pm}, \rho_2, \Phi_2)$ solve \eqref{VP_h_spec}-\eqref{eqtn:Dphi_spec}, we have 
\Be \label{VP_diff_spec}
\begin{split}
& v_\pm \cdot \nabla_x ( h_{1, \pm} - h_{2, \pm} )
+ \big( e_{\pm} ( \frac{v_\pm}{c} \times B 
- \nabla_x \Phi_1 ) - \nabla_x (m_{\pm} g x_3) \big) \cdot \nabla_p ( h_{1, \pm} - h_{2, \pm} ) 
\\& = e_{\pm} \nabla_x (\Phi_1 - \Phi_2 ) \cdot \nabla_p h_{2, \pm} \ \ \text{in} \ \O \times \R^3, 
\end{split}
\ee
and
\be \label{VP_diff_bdy_spec}
h_1 (x,p) - h_2 (x,p) = \varepsilon h_1 (x, \tilde{p}) - \varepsilon h_2 (x, \tilde{p})  
\ \ \text{in} \ \gamma_-,
\Ee
where
\Be \notag
v_\pm =  c \frac{p}{p^0_\pm} = \frac{p}{\sqrt{ m_\pm^2 + |p|^2/ c^2}}.
\Ee
Suppose $(X (s;x,p), P (s;x,p))$ is the characteristics in \eqref{ODE_h_spec} with $\Phi_h = \Phi_1$. Then, from \eqref{VP_diff_spec} and \eqref{VP_diff_bdy_spec}, we get
\begin{align}
& (h_1- h_2) (x,p) \notag
\\& = \varepsilon h_1 (\xb, \tpb) - \varepsilon h_2 (\xb, \tpb) 
\label{diff1:h_spec}
\\& \ \ \ \ + \int^0_{-\tb(x, p)} e (\nabla_x\Phi_1(X (s;x,p)) -\nabla_x \Phi_2(X (s;x,p))) \cdot \nabla_p h_2(X (s;x,p), P (s;x,p)) \dd s, 
\label{diff2:h_spec}
\end{align}
where $\tb(x,p)$ is the backward exit time of the characteristics $(X (s;x,p), P (s;x,p))$.
From the condition \eqref{condition_unique_spec}, together with \eqref{bound:diff_h} and \eqref{est:tb/w}, we obtain that
\Be \label{bound:diff_h_spec}
\begin{split}
|\eqref{diff2:h_spec}| 
& \leq e \| \nabla_x \Phi_1 - \nabla_x \Phi_2 \|_\infty \int^0_{-\tb(x, p)} | \nabla_p h_2(X (s;x,p), P (s;x,p)) | \dd s
\\& \leq e \| \nabla_x \Phi_1 - \nabla_x \Phi_2 \|_\infty \| w_{\bar\beta} \nabla_p h_2 \|_\infty 
\frac{48}{m g \bar \beta} e^{ - \frac{3 \bar \beta}{4} \big( \sqrt{(m c)^2 + |p|^2} + \frac{1}{2 c} m g x_3 \big) }.
\end{split}
\Ee
Inputting \eqref{bound:diff_h_spec} into \eqref{diff1:h_spec} and \eqref{diff2:h_spec}, we have
\be \label{bound2:diff_h_spec}
\begin{split}
& | h_1 (x, p) - h_2 (x, p) | 
\\& \leq  \varepsilon | h_1 (\xb, \tpb) - h_2 (\xb, \tpb) | 
\\& \ \ \ \ + \frac{48 e}{m g \bar \beta} \| \nabla_x \Phi_1 - \nabla_x \Phi_2 \|_\infty \| w_{\bar\beta} \nabla_p h_2 \|_\infty e^{ - \frac{3 \bar \beta}{4} \big( \sqrt{(m c)^2 + |p|^2} + \frac{1}{2 c} m g x_3 \big) }.
\end{split}
\ee
From \eqref{Uest:DPhi_spec}, together with Lemma \ref{lem:conservation_law} and $|\pb| = |\tpb|$, \eqref{bound2:diff_h_spec} shows that
\be \label{bound3:diff_h_spec}
\begin{split}
& e^{ \frac{3 \bar \beta}{4} \big( \sqrt{(m c)^2 + |p|^2} + \frac{1}{2 c} m g x_3 \big) } |h_1 (x, p) - h_2 (x, p)|
\\& \leq \varepsilon e^{ \frac{3 \bar \beta}{4} \big( \sqrt{(m c)^2 + |\tpb|^2} \big) } | h_1 (\xb, \tpb) - h_2 (\xb, \tpb) |
\\& \ \ \ \ + \frac{48 e}{m g \bar \beta}  \| \nabla_x \Phi_1 - \nabla_x \Phi_2 \|_\infty \| w_{\bar\beta} \nabla_p h_2 \|_\infty.
\end{split}
\ee
Note that both $(h_{1, \pm}, \rho_1, \Phi_1)$ and $(h_{2, \pm}, \rho_2, \Phi_2)$ satisfy \eqref{Uest:wh_spec}-\eqref{Uest:DPhi_spec}.
Following \eqref{est2:phi_12_x}, we bound \eqref{bound3:diff_h_spec} by
\be \label{bound4:diff_h_spec}
\begin{split}
& e^{ \frac{3 \bar \beta}{4} \big( \sqrt{(m_+ c)^2 + |p|^2} + \frac{1}{2 c} m_+ g x_3 \big) } |h_{1, +} (x, p) - h_{2, +} (x, p)|
\\& \ \ \ \ + e^{ \frac{3 \bar \beta}{4} \big( \sqrt{(m_- c)^2 + |p|^2} + \frac{1}{2 c} m_- g x_3 \big) } |h_{1, -} (x, p) - h_{2, -} (x, p)|
\\& \leq \varepsilon \Big(
\| e^{ \frac{3 \bar{\beta}}{4} \big( \sqrt{(m_+ c)^2 + |p|^2} + \frac{3}{2 c} m_+ g x_3 \big) } (h_{1, +} - h_{2, +} ) \|_{L^{\infty} (\bar \O \times \R^3)}
\\& \qquad \ \ + \| e^{ \frac{3 \bar{\beta}}{4} \big( \sqrt{(m_- c)^2 + |p|^2} + \frac{3}{2 c} m_- g x_3 \big) } (h_{1, -} - h_{2, -} ) \|_{L^{\infty} (\bar \O \times \R^3)}  \Big)
\\& \ \ \ \ + \mathfrak{C} (1 + \frac{2 c}{\beta^\prime g} ) \frac{1}{\hat{\beta}} \Big(
e_+ \| e^{ \frac{\bar{\beta}}{4} \big( \sqrt{(m_+ c)^2 + |p|^2} + \frac{3}{2 c} m_+ g x_3 \big) } (h_{1, +} - h_{2, +} ) \|_{L^{\infty} (\bar \O \times \R^3)}
\\& \qquad \qquad \qquad \qquad \ \ + e_- \| e^{ \frac{\bar{\beta}}{4} \big( \sqrt{(m_- c)^2 + |p|^2} + \frac{3}{2 c} m_- g x_3 \big) } (h_{1, -} - h_{2, -} ) \|_{L^{\infty} (\bar \O \times \R^3)}  \Big)
\\& \qquad \ \ \times \big( \frac{48 e_+}{m g \bar \beta} \| w_{+, \bar\beta} \nabla_p h_{2, +} \|_\infty + \frac{48 e_-}{m g \bar \beta} \| w_{-, \bar\beta} \nabla_p h_{2, -} \|_\infty \big).
\end{split}
\ee
Recall $\beta'$ in \eqref{def:beta'}, $\hat{\beta} = \frac{\beta'}{\min \{ m_{-},  m_{+} \}}$ and the condition \eqref{condition_unique_spec}. Following the steps in \eqref{bound5:diff_h} and \eqref{bound6:diff_h}, we get
\be \label{bound6:diff_h_spec}
\begin{split}
& e^{ \frac{3 \bar \beta}{4} \big( \sqrt{(m_+ c)^2 + |p|^2} + \frac{1}{2 c} m_+ g x_3 \big) } |h_{1, +} (x, p) - h_{2, +} (x, p)|
\\& \ \ \ \ + e^{ \frac{3 \bar \beta}{4} \big( \sqrt{(m_- c)^2 + |p|^2} + \frac{1}{2 c} m_- g x_3 \big) } |h_{1, -} (x, p) - h_{2, -} (x, p)|
\\& \leq (\frac{1}{2} + \varepsilon ) \times \Big(
\| e^{ \frac{3 \bar{\beta}}{4} \big( \sqrt{(m_+ c)^2 + |p|^2} + \frac{3}{2 c} m_+ g x_3 \big) } (h_{1, +} - h_{2, +} ) \|_{L^{\infty} (\bar \O \times \R^3)}
\\& \qquad \qquad \qquad \qquad + \| e^{ \frac{3 \bar{\beta}}{4} \big( \sqrt{(m_- c)^2 + |p|^2} + \frac{3}{2 c} m_- g x_3 \big) } (h_{1, -} - h_{2, -} ) \|_{L^{\infty} (\bar \O \times \R^3)}  \Big).
\end{split}
\ee
Using \eqref{bound6:diff_h_spec}, together with the condition \eqref{condition:epilon_G_spec} on $\varepsilon$, we derive 
\Be \notag
\begin{split}
& \| e^{ \frac{3 \bar{\beta}}{4} \big( \sqrt{(m_+ c)^2 + |p|^2} + \frac{1}{2 c} m_+ g x_3 \big) } (h_{1, +} - h_{2, +} ) \|_{L^{\infty} (\bar \O \times \R^3)} 
\\& \qquad \qquad + \| e^{ \frac{3 \bar{\beta}}{4} \big( \sqrt{(m_- c)^2 + |p|^2} + \frac{1}{2 c} m_- g x_3 \big) } (h_{1, -} - h_{2, -} ) \|_{L^{\infty} (\bar \O \times \R^3)}
\\& \leq \frac{3}{4} \Big( \| e^{ \frac{3 \bar{\beta}}{4} \big( \sqrt{(m_+ c)^2 + |p|^2} + \frac{1}{2 c} m_+ g x_3 \big) } (h_{1, +} - h_{2, +} ) \|_{L^{\infty} (\bar \O \times \R^3)} 
\\& \qquad \qquad + \| e^{ \frac{3 \bar{\beta}}{4} \big( \sqrt{(m_- c)^2 + |p|^2} + \frac{1}{2 c} m_- g x_3 \big) } (h_{1, -} - h_{2, -} ) \|_{L^{\infty} (\bar \O \times \R^3)}  \Big),
\end{split}
\Ee
and thus we conclude the uniqueness. 
\end{proof}

\subsubsection{Proof of the Main Theorem: Stationary Problem} \label{sec:EX_SS_spec}

\begin{prop} \label{prop:Unif_D2xDp_spec}

Suppose all assumptions in Theorem \ref{theo:CS_spec} hold for some $g, \beta, \tilde \beta > 0$. 
Then $(h^{\ell+1}_{\pm}, \rho^\ell, \nabla_x  \Phi^\ell)$ from the construction satisfies the following uniform-in-$\ell$ estimates:
\be \label{Uest:DDPhi^l_spec}
\frac{8}{m_{\pm} g} (1 + B_3 + \| \nabla_x ^2 \Phi^{\ell} \|_\infty) 
\leq \frac{\tilde \beta}{2},
\ee
and
\be \label{Uest:h_v^l_spec}
\begin{split}
& e^{ \frac{\tilde \beta}{2}|p^0_{\pm}|} e^{  \frac{\tilde \beta m_{\pm} g}{4 c} x_3} | \nabla_p h^{\ell+1}_{\pm} (x,p)| 
\\& \lesssim \big( 1 + \frac{ \| \nabla_x^2 \Phi^{\ell} \|_\infty + e B_3 + m_{\pm} g}{(m_{\pm} g)^2} \big) \| e^{\tilde \beta \sqrt{(m_{\pm} c)^2 + |p|^2}} \nabla_{x_\parallel, p} G_{\pm} \|_{L^\infty (\gamma_-)},
\end{split}
\ee
and
\be \label{Uest:h_x^l_spec}
\begin{split}
& e^{ \frac{\tilde \beta}{2}|p^0_{\pm}|} e^{ \frac{\tilde \beta m_{\pm} g}{4 c} x_3} | \nabla_x h^{\ell+1}_{\pm} (x,p)| 
\\& \lesssim \big( \frac{\delta_{i3}}{\alpha_{\pm} (x,p)} + \frac{ \| \nabla_x^2 \Phi^{\ell} \|_\infty + e B_3 + m_{\pm} g}{(m_{\pm} g)^2} \big) \| e^{\tilde \beta \sqrt{(m_{\pm} c)^2 + |p|^2}} \nabla_{x_\parallel, p} G_{\pm} \|_{L^\infty (\gamma_-)}.
\end{split}
\ee
\end{prop}

\begin{proof}

Under the initial setting $h^0_{\pm} = 0$ and $(\rho^0, \nabla_x \Phi^0) = (0,0)$, together with the condition \eqref{condition:tilde_beta_spec}, we deduce that $\| \nabla^2_x \Phi^0 \|_{\infty} = 0$, and thus \eqref{Uest:DDPhi^l_spec} holds for $\ell = 0$. From the construction, $h^1_{\pm}$ is the solution to \eqref{VP_h_spec} and \eqref{bdry:h_spec} with $\nabla_x \Phi_h = \nabla_x \Phi^0 = 0$, that is,
\be \label{VP^0_spec}
\begin{split}
& v_\pm \cdot \nabla_x h^1_{\pm}
+ \big( e_{\pm} ( \frac{v_\pm}{c} \times B 
- \nabla_x \Phi^0 ) - \nabla_x (m_{\pm} g x_3) \big) \cdot \nabla_p h^1_{\pm} = 0 \ \ \text{in} \ \O \times \R^3, 
\\& h^1_{\pm} (x,p) = G_{\pm} (x,p) \ \ \text{in} \ \gamma_-. 
\end{split}
\ee

From \eqref{z1_spec}, we consider the characteristic $(X^1, P^1)$  and the corresponding $(\tb^1, \xb^1, \pb^1)$ for \eqref{VP^0_spec}.
Following the steps in the proof of Theorem \ref{prop:Unif_D2xDp}, we deduce that Lemmas \ref{VL}-\ref{lem:nabla_zb} hold for the characteristic $(X^1, P^1)$ and the corresponding $(\tb^1, \xb^1, \pb^1)$.

Moreover, since all assumptions in Theorem \ref{theo:CS_spec} hold, one can check that $\| \nabla^2_x \Phi^0 \|_{\infty}$ satisfies the assumption \eqref{choice_tbeta_spec} in Proposition \ref{prop:Reg_spec}.
From Proposition \ref{prop:Reg_spec}, we bound $\| \nabla^2_x \Phi^0 \|_{\infty}$ by
\Be \notag
\begin{split}
\| \nabla_x^2 \Phi^0 \|_\infty
& \lesssim \frac{1}{\beta} \big( e_+ \| w_{+, \beta} G_+ \|_{L^\infty(\gamma_-)} + e_{-} \| w_{-, \beta} G_- \|_{L^\infty(\gamma_-)} \big)
\\& \ \ \ \ + \big( \| e^{\tilde \beta \sqrt{(m_+ c)^2 + |p|^2}} \nabla_{x_\parallel, p} G_+ \|_{L^\infty (\gamma_-)} + \| e^{\tilde \beta \sqrt{(m_- c)^2 + |p|^2}} \nabla_{x_\parallel, p} G_- \|_{L^\infty (\gamma_-)}
\big) .
\end{split}	
\Ee
Following the steps 1 and 2 (proof of \eqref{est:hk_v_spec} and \eqref{est:hk_x_spec}) in the proof of Proposition \ref{prop:Reg_spec}, we deduce that \eqref{Uest:h_v^l_spec} and \eqref{Uest:h_x^l_spec} hold for $\ell = 0$.

\smallskip

Now we prove this by induction. 
In the following proof, we abuse the notation as in \eqref{abuse}.
Assume a positive integer $k > 0$ and suppose that \eqref{Uest:DDPhi^l_spec} and \eqref{Uest:h_v^l_spec} hold for $0 \leq \ell \leq k$.
From \eqref{eqtn:hk_spec}-\eqref{bdry:phik_spec}, then for $\ell=k+1$, 
\be \label{rho_phi^k+1_spec}
\begin{split}	
\rho^{k+1} = \int_{\R^3} ( e_+ h^{k+1}_+ + e_{-} h^{k+1}_{-} ) \dd p & \ \ \text{in} \ \O, \\
- \Delta  \Phi^{k+1} = \rho^{k+1} & \ \ \text{in} \ \O, \\
\Phi^{k+1} =0  & \ \ \text{on} \ \p\O,
\end{split}
\ee
and
\be \label{h^k+2_spec}
\begin{split}
& v_\pm \cdot \nabla_x h^{k+2}_{\pm} 
+ \big( e_{\pm} ( \frac{v_\pm}{c} \times B 
- \nabla_x \Phi^{k+1} ) - \nabla_x (m_{\pm} g x_3) \big) \cdot \nabla_p h^{k+2}_{\pm} = 0 \ \ \text{in} \ \O \times \R^3, 
\\& h^{k+2}_{\pm} (x,p) = G_{\pm} (x,p) + \varepsilon h^{k+1}_{\pm} (x, \tilde{p}) \ \ \text{on} \ \gamma_-.
\end{split}
\ee
On the other side, from the construction, $h^{k+1}_{\pm}$ is the solution to \eqref{VP_h_spec} and \eqref{bdry:h_spec} with $\nabla_x \Phi_h = \nabla_x \Phi^{k}$, that is,
\be \label{VP^k_spec}
\begin{split}
& v_\pm \cdot \nabla_x h^{k+1}_{\pm}
+ \big( e_{\pm} ( \frac{v_\pm}{c} \times B 
- \nabla_x \Phi^k ) - \nabla_x (m_{\pm} g x_3) \big) \cdot \nabla_p h^{k+1}_{\pm} = 0 \ \ \text{in} \ \O \times \R^3, 
\\& h^{k+1}_{\pm} (x,p) = G_{\pm} (x,p) + \varepsilon h^{k}_{\pm} (x, \tilde{p}) \ \ \text{on} \ \gamma_-. 
\end{split}
\ee
Recall \eqref{Uest:DPhi^k_spec} in Proposition \ref{prop:Unif_steady_spec}, then for any $\ell \geq 0$,
\be \label{est:Phi_x^0_spec}
\| \nabla_x \Phi^ \ell \|_{L^\infty (\bar{\O})} 
\leq  \min \big( \frac{m_{+}}{e_{+}}, \frac{m_{-}}{e_{-}} \big) \times \  \frac{g}{2}.
\ee 
From \eqref{VP_h^l+1_spec}, \eqref{bdry:h^l+1_spec} and \eqref{def:tb_l}, we consider the characteristic $(X^{k+1}, P^{k+1})$ for \eqref{VP^k_spec} and the corresponding $(\tb^{k+1}, \xb^{k+1}, \pb^{k+1})$.
Similar as the initial case, by replacing $\Phi_h$ with $\Phi^k$, Lemmas \ref{VL}-\ref{lem:nabla_zb} hold for $(X^{k+1}, P^{k+1})$ and $(\tb^{k+1}, \xb^{k+1}, \pb^{k+1})$.

\smallskip

First, we show \eqref{Uest:DDPhi^l_spec} holds for $\ell = k+1$.
Since \eqref{Uest:h_x^l_spec} holds for $\ell = k$, we have
\be \notag
\begin{split}
& e^{ \frac{\tilde \beta}{2}|p^0_{\pm}|} e^{ \frac{\tilde \beta m_{\pm} g}{4 c} x_3} | \nabla_x h^{k+1}_{\pm} (x,p)| 
\\& \lesssim \big( \frac{\delta_{i3}}{\alpha_{\pm} (x,p)} + \frac{ \| \nabla_x^2 \Phi^{k} \|_\infty + e B_3 + m_{\pm} g}{(m_{\pm} g)^2} \big) \| e^{\tilde \beta \sqrt{(m_{\pm} c)^2 + |p|^2}} \nabla_{x_\parallel, p} G_{\pm} \|_{L^\infty (\gamma_-)}.
\end{split}
\Ee 
Recall from \eqref{rho_phi^k+1_spec}, we bound $\nabla_x \rho^{k+1} (x)$ by
\Be \notag
\begin{split}
& \ \ \ \ | \nabla_x \rho^{k+1} (x) | 
\\& = \Big| \int_{\R^3} e_+ \nabla_x h^{k+1}_+ (x,p) + e_{-} \nabla_x h^{k+1}_{-} (x,p) \dd p \Big|.
\end{split}
\Ee 
Applying Lemma \ref{lem:nabla_zb} on $(\tb^{k+1}, \xb^{k+1}, \pb^{k+1})$, together with the step 3 (Proof of \eqref{est:rho_x_spec}) in the proof of Proposition \ref{prop:Reg_spec}, we deduce that
\Be \label{est:rho_x^k+1_spec}
\begin{split}
e^{ \frac{\tilde \beta m g}{4 c} x_3 } |\p_{x_i} \rho^{k+1} (x)|   
& \lesssim \| e^{\tilde \beta \sqrt{(m_+ c)^2 + |p|^2}} \nabla_{x_\parallel, p} G_+ \|_{L^\infty (\gamma_-)}  
\times \Big(
1 + \mathbf{1}_{|x_3| \leq 1} \frac{1}{\sqrt{ m_+ g x_3 }}
\Big)
\\& \ \ \ \ + \| e^{\tilde \beta \sqrt{(m_- c)^2 + |p|^2}} \nabla_{x_\parallel, p} G_- \|_{L^\infty (\gamma_-)}  
\times \Big(
1 + \mathbf{1}_{|x_3| \leq 1} \frac{1}{\sqrt{ m_- g x_3 }}
\Big),
\end{split}
\Ee
Using \eqref{est:rho_x^k+1_spec}, together with the step 4 (Proof of \eqref{est:phi_C2_spec}) in the proof of Proposition \ref{prop:Reg_spec}, we bound
\Be \notag
\begin{split}
\| \nabla_x^2 \Phi^{k+1} \|_\infty
& \lesssim \frac{1}{\beta} \big( e_+ \| w_{+, \beta} G_+ \|_{L^\infty(\gamma_-)} + e_{-} \| w_{-, \beta} G_- \|_{L^\infty(\gamma_-)} \big)
\\& \ \ \ \ + \big( \| e^{\tilde \beta \sqrt{(m_+ c)^2 + |p|^2}} \nabla_{x_\parallel, p} G_+ \|_{L^\infty (\gamma_-)} + \| e^{\tilde \beta \sqrt{(m_- c)^2 + |p|^2}} \nabla_{x_\parallel, p} G_- \|_{L^\infty (\gamma_-)}
\big).
\end{split}	
\Ee
Under the condition \eqref{condition:tilde_beta_spec}, this shows that
\be \label{est1:DDPhi^k+1_spec}
\frac{8}{m g} (1 + B_3 + \| \nabla_x ^2 \Phi^{k+1} \|_\infty) \leq \tilde \beta,
\ee
and \eqref{Uest:DDPhi^l_spec} holds for $\ell = k+1$.

\smallskip

Second, we show \eqref{Uest:h_v^l_spec} and \eqref{Uest:h_x^l_spec} hold for $\ell = k+1$.
From \eqref{VP_h^l+1_spec}, \eqref{bdry:h^l+1_spec} and \eqref{def:tb_l}, we consider the characteristic $(X^{k+2}, P^{k+2})$ for \eqref{h^k+2_spec} and the corresponding $(\tb^{k+2}, \xb^{k+2}, \pb^{k+2})$.
From \eqref{h^k+2_spec} and \eqref{form:nabla_h_spec}, we have
\Be \notag
\begin{split}
\nabla_{x, p} h^{k+2} (x,p)
& = \nabla_{p} \xb^{k+2} (x,p) \cdot \nabla_{x_\parallel} G (\xb^{k+2}, \pb^{k+2}) + \nabla_{p} \pb^{k+2} (x,p) \cdot \nabla_p G (\xb^{k+2}, \pb^{k+2})
\\& \ \ \ \ + \varepsilon \nabla_{x, p} \xb (x,p) \cdot \nabla_{x_\parallel} h^{k+1} (\xb, \tpb) + \varepsilon \nabla_{x, p} \tpb (x,p) \cdot \nabla_p h^{k+1} (\xb, \tpb).
\end{split}
\Ee
Similarly, by replacing $\Phi_h$ with $\Phi^{k+1}$, Lemmas \ref{VL}-\ref{lem:nabla_zb} hold for the characteristic $(X^{k+2}, P^{k+2})$ and $(\tb^{k+2}, \xb^{k+2}, \pb^{k+2})$.
Applying Lemma \ref{lem:nabla_zb} on $(\tb^{k+2}, \xb^{k+2}, \pb^{k+2})$ and the upper bound on $\| \nabla_x ^2 \Phi^{k+1} \|_\infty$ from \eqref{est1:DDPhi^k+1_spec}, together with the steps 1 and 2 (proof of \eqref{est:hk_v_spec} and \eqref{est:hk_x_spec}) in the proof of Proposition \ref{prop:Reg_spec}, we conclude that \eqref{Uest:h_v^l_spec} and \eqref{Uest:h_x^l_spec} holds for $\ell = k+1$.

\smallskip

Now we have proved \eqref{Uest:DDPhi^l_spec}, \eqref{Uest:h_v^l_spec} and \eqref{Uest:h_x^l_spec} hold for $\ell = k+1$. Therefore, we complete the proof by induction.
\end{proof}

\begin{prop} \label{prop:cauchy_spec}

Suppose all assumptions in Theorem \ref{theo:CS_spec} hold for some $g, \beta, \tilde \beta > 0$. 
Then for any $\ell \geq 1$, $h^{\ell}_{\pm}$ from the construction satisfies that there exists some $\bar \beta > 0$, 
\be \label{est:h_cauchy_spec}
\begin{split}
& \| e^{ \frac{3 \bar{\beta}}{4} \big( \sqrt{(m_+ c)^2 + |p|^2} + \frac{1}{2 c} m_+ g x_3 \big) } (h^{\ell+1}_{+} - h^{\ell}_{+} ) \|_{L^{\infty} (\bar \O \times \R^3)} 
\\& \qquad \qquad + \| e^{ \frac{3 \bar{\beta}}{4} \big( \sqrt{(m_- c)^2 + |p|^2} + \frac{1}{2 c} m_- g x_3 \big) } (h^{\ell+1}_{-} - h^{\ell}_{-} ) \|_{L^{\infty} (\bar \O \times \R^3)}
\\& \leq \frac{3}{4} \Big( \| e^{ \frac{3 \bar{\beta}}{4} \big( \sqrt{(m_+ c)^2 + |p|^2} + \frac{1}{2 c} m_+ g x_3 \big) } (h^\ell_{+} - h^{\ell-1}_{+} ) \|_{L^{\infty} (\bar \O \times \R^3)} 
\\& \qquad \qquad + \| e^{ \frac{3 \bar{\beta}}{4} \big( \sqrt{(m_- c)^2 + |p|^2} + \frac{1}{2 c} m_- g x_3 \big) } (h^\ell_{-} - h^{\ell-1}_{-} ) \|_{L^{\infty} (\bar \O \times \R^3)}  \Big).
\end{split}
\Ee
Furthermore, $\{ h^{\ell+1}_{\pm} \}^{\infty}_{\ell=0}$, and $\{ \rho^\ell \}^{\infty}_{\ell=0}$, $\{ \nabla_x \Phi^\ell \}^{\infty}_{\ell=0}$ from the construction are Cauchy sequences in $L^{\infty} (\O \times \R^3)$ and $L^{\infty} (\O)$ respectively.
\end{prop}

\begin{proof}

For the sake of simplicity, we abuse the notation as in \eqref{abuse}.
We first prove \eqref{est:h_cauchy_spec}. 
Given $\ell \geq 1$, from the construction in \eqref{eqtn:hk_spec}-\eqref{bdry:phik_spec}, we have 
\Be \label{VP_diff^l+1_spec}
\begin{split}
& v_\pm \cdot \nabla_x ( h^{\ell+1}_{\pm} - h^{\ell}_{\pm} )
+ \big( e_{\pm} ( \frac{v_\pm}{c} \times B 
- \nabla_x \Phi^{\ell} ) - \nabla_x (m_{\pm} g x_3) \big) \cdot \nabla_p ( h^{\ell+1}_{\pm} - h^{\ell}_{\pm} ) 
\\& = e_{\pm} \nabla_x (\Phi^{\ell} - \Phi^{\ell-1} ) \cdot \nabla_p h^{\ell}_{\pm} \ \ \text{in} \ \O \times \R^3, 
\end{split}
\ee
with 
\be \label{VP_diff_bdy^l+1_spec}
h^{\ell+1}_{\pm} (x,p) - h^{\ell}_{\pm} (x,p) = \varepsilon h^{\ell}_{\pm} (x, \tilde{p}) - \varepsilon h^{\ell-1}_{\pm} (x, \tilde{p})  
\ \ \text{in} \ \gamma_-,
\Ee
where
\Be \notag
v_\pm =  c \frac{p}{p^0_\pm} = \frac{p}{\sqrt{ m_\pm^2 + |p|^2/ c^2}}.
\Ee
Following \eqref{VP_h^l+1_spec}, \eqref{bdry:h^l+1_spec} and \eqref{def:tb_l}, we consider the characteristic $(X^{\ell+1}, P^{\ell+1})$ for \eqref{VP_diff^l+1_spec}
and the corresponding backward exit time $\tb^{\ell+1} (x,p)$. Thus, from \eqref{VP_diff^l+1_spec} and \eqref{VP_diff_bdy^l+1_spec},
\begin{align}
& (h^{\ell+1} - h^{\ell} ) (x,p) \notag
\\& = \varepsilon h^{\ell} (\xb, \tpb) - \varepsilon h^{\ell-1} (\xb, \tpb) 
\label{diff1:h^l+1_spec}
\\& \ \ \ \ + \int^0_{-\tb^{\ell+1} (x, p)} e (\nabla_x \Phi^{\ell} - \nabla_x \Phi^{\ell-1} ) \cdot \nabla_p h^{\ell} ( X^{\ell+1} (s;x,p), P^{\ell+1} (s;x,p)) \dd s.
\label{diff2:h^l+1_spec}
\end{align}
Using Proposition \ref{prop:Unif_D2xDp_spec}, together with the conditions \eqref{condition:beta_spec} and \eqref{condition:tilde_beta_spec}, then for any $\ell \geq 1$,
\be \label{est:h_v^l+1_spec}
\begin{split}
& e^{ \frac{\tilde \beta}{2} |p^0|} e^{  \frac{\tilde \beta m g}{4 c} x_3} | \nabla_p h^{\ell}_{\pm} (x,p)| 
\\& \lesssim \big( 1 + \frac{ \| \nabla_x^2 \Phi^{\ell} \|_\infty + e B_3 + mg}{(m g)^2} \big) \| e^{\tilde \beta \sqrt{(m_{\pm} c)^2 + |p|^2}} \nabla_{x_\parallel, p} G_{\pm} \|_{L^\infty (\gamma_-)}
\\& \lesssim (1 + \frac{\tilde{\beta}}{m g} ) \| e^{\tilde \beta \sqrt{(m_{\pm} c)^2 + |p|^2}} \nabla_{x_\parallel, p} G_{\pm} \|_{L^\infty (\gamma_-)}.
\end{split}
\ee
From \eqref{Uest:DPhi^k_spec} in Proposition \ref{prop:Unif_steady_spec}, then for any $\ell \geq 1$,
\be \label{est:phi_x^l_spec}
\| \nabla_x \Phi^{\ell} \|_{L^\infty (\bar{\O})} 
\leq  \min \big( \frac{m_{+}}{e_{+}}, \frac{m_{-}}{e_{-}} \big) \times \  \frac{g}{2}. 
\ee
Following \eqref{bound:diff_h_spec}, together with \eqref{bound:diff_h} and \eqref{est:tb/w}, we obtain that for any $\ell \geq 1$,
\Be \label{bound:diff_h^l+1_spec}
\begin{split}
& |\eqref{diff2:h^l+1_spec}| 
\leq \frac{48 e}{m g \bar \beta} \| \nabla_x \Phi^{\ell} - \nabla_x \Phi^{\ell-1} \|_\infty \| w^{\ell+1}_{ \bar{\beta} } \nabla_p h^{\ell} \|_\infty e^{ - \frac{3 \bar \beta}{4} \big( \sqrt{(m c)^2 + |p|^2} + \frac{1}{2 c} m g x_3 \big) }.
\end{split}
\ee
Inputting \eqref{bound:diff_h^l+1_spec} into \eqref{diff1:h^l+1_spec} and \eqref{diff2:h^l+1_spec}, we have
\be \label{bound2:diff_h^l+1_spec}
\begin{split}
& | h^{\ell+1} (x, p) - h^{\ell} (x, p) | 
\\& \leq  \varepsilon | h^{\ell} (\xb, \tpb) - h^{\ell-1} (\xb, \tpb) | 
\\& \ \ \ \ + \frac{48 e}{m g \bar \beta} \| \nabla_x \Phi^{\ell} - \nabla_x \Phi^{\ell-1} \|_\infty \| w^{\ell+1}_{ \bar{\beta} } \nabla_p h^{\ell} \|_\infty e^{ - \frac{3 \bar \beta}{4} \big( \sqrt{(m c)^2 + |p|^2} + \frac{1}{2 c} m g x_3 \big) }.
\end{split}
\ee
From \eqref{est:phi_x^l_spec}, together with Lemma \ref{lem:conservation_law} and $|\pb| = |\tpb|$, \eqref{bound2:diff_h^l+1_spec} shows that
\be \label{bound3:diff_h^l+1_spec}
\begin{split}
& e^{ \frac{3 \bar \beta}{4} \big( \sqrt{(m c)^2 + |p|^2} + \frac{1}{2 c} m g x_3 \big) } | h^{\ell+1} (x, p) - h^{\ell} (x, p) | 
\\& \leq \varepsilon e^{ \frac{3 \bar \beta}{4} \big( \sqrt{(m c)^2 + |\tpb|^2} \big) } | h^{\ell} (\xb, \tpb) - h^{\ell-1} (\xb, \tpb) | 
\\& \ \ \ \ + \frac{48 e}{m g \bar \beta}  \| \nabla_x \Phi^{\ell} - \nabla_x \Phi^{\ell-1} \|_\infty \| w^{\ell+1}_{ \bar{\beta} } \nabla_p h^{\ell} \|_\infty.
\end{split}
\ee

Consider $\hat{m} = \min \{ m_{-},  m_{+} \}$ and $\bar{\beta} = \frac{\tilde \beta}{6}$, then we pick $\beta'$ as follows:
\be \label{def:beta'^l+1_spec}
0 < \beta' =  \min\{ \frac{\bar \beta}{4}, \beta \} \times \hat{m}.
\ee
Following \eqref{est2:phi_12_x}, we bound \eqref{bound3:diff_h^l+1_spec} by
\be \label{bound4:diff_h^l+1_spec}
\begin{split}
& e^{ \frac{3 \bar \beta}{4} \big( \sqrt{(m_+ c)^2 + |p|^2} + \frac{1}{2 c} m_+ g x_3 \big) } |h^{\ell+1}_{+} (x, p) - h^{\ell}_{+} (x, p)|
\\& \ \ \ \ + e^{ \frac{3 \bar \beta}{4} \big( \sqrt{(m_- c)^2 + |p|^2} + \frac{1}{2 c} m_- g x_3 \big) } |h^{\ell+1}_{-} (x, p) - h^{\ell}_{-} (x, p)|
\\& \leq \varepsilon \Big(
\| e^{ \frac{3 \bar{\beta}}{4} \big( \sqrt{(m_+ c)^2 + |p|^2} + \frac{3}{2 c} m_+ g x_3 \big) } (h^{\ell}_{+} - h^{\ell-1}_{+} ) \|_{L^{\infty} (\bar \O \times \R^3)}
\\& \qquad \ \ + \| e^{ \frac{3 \bar{\beta}}{4} \big( \sqrt{(m_- c)^2 + |p|^2} + \frac{3}{2 c} m_- g x_3 \big) } (h^{\ell}_{-} - h^{\ell-1}_{-} ) \|_{L^{\infty} (\bar \O \times \R^3)}  \Big)
\\& \ \ \ \ + \mathfrak{C} (1 + \frac{2 c}{\beta^\prime g} ) \frac{1}{\hat{\beta}} \Big(
e_+ \| e^{ \frac{\bar{\beta}}{4} \big( \sqrt{(m_+ c)^2 + |p|^2} + \frac{3}{2 c} m_+ g x_3 \big) } (h^{\ell}_{+} - h^{\ell-1}_{+} ) \|_{L^{\infty} (\bar \O \times \R^3)}
\\& \qquad \qquad \qquad \qquad \ \ + e_- \| e^{ \frac{\bar{\beta}}{4} \big( \sqrt{(m_- c)^2 + |p|^2} + \frac{3}{2 c} m_- g x_3 \big) } (h^{\ell}_{-} - h^{\ell-1}_{-} ) \|_{L^{\infty} (\bar \O \times \R^3)}  \Big)
\\& \qquad \ \ \times \big( \frac{48 e_+}{m g \bar \beta} \| w^{\ell+1}_{+, \bar\beta} \nabla_p h^{\ell}_{+} \|_\infty + \frac{48 e_-}{m g \bar \beta} \| w^{\ell+1}_{-, \bar\beta} \nabla_p h^{\ell}_{-} \|_\infty \big).
\end{split}
\ee
Using $\bar{\beta} = \frac{\tilde \beta}{6}$ and the condition \eqref{condition:G_xv_spec}, together with \eqref{est:h_v^l+1_spec}, we derive for any $\ell \geq 1$,
\be \notag
w^{\ell+1}_{\pm, \bar{\beta}} (x,p)  | \nabla_p h^{\ell}_{\pm} (x,p)| 
\lesssim (1 + \frac{\tilde{\beta}}{m g} ) \| e^{\tilde \beta \sqrt{(m_{\pm} c)^2 + |p|^2}} \nabla_{x_\parallel, p} G_{\pm} \|_{L^\infty (\gamma_-)} \leq \frac{1}{4}.
\ee 
Thus, we bound \eqref{bound4:diff_h^l+1_spec} by
\be \label{bound6:diff_h^l+1_spec}
\begin{split}
\eqref{bound4:diff_h^l+1_spec} 
& \leq ( \frac{1}{2} + \varepsilon ) \Big( \| e^{ \frac{3 \bar{\beta}}{4} \big( \sqrt{(m_+ c)^2 + |p|^2} + \frac{1}{2 c} m_+ g x_3 \big) } (h^{\ell}_{+} - h^{\ell-1}_{+}) \|_{L^{\infty} (\O \times \R^3)} 
\\& \qquad \qquad \qquad + \| e^{ \frac{3 \bar{\beta}}{4} \big( \sqrt{(m_- c)^2 + |p|^2} + \frac{1}{2 c} m_- g x_3 \big) } (h^{\ell}_{-} - h^{\ell-1}_{-} ) \|_{L^{\infty} (\O \times \R^3)}  \Big).
\end{split}
\ee
Finally, using the condition \eqref{condition:epilon_G_spec}, together with \eqref{bound4:diff_h^l+1_spec}-\eqref{bound6:diff_h^l+1_spec}, we derive that
\Be \notag
\begin{split}
& \| e^{ \frac{3 \bar{\beta}}{4} \big( \sqrt{(m_+ c)^2 + |p|^2} + \frac{1}{2 c} m_+ g x_3 \big) } (h^{\ell+1}_{+} (x, p) - h^{\ell}_{+} (x, p)) \|_{L^{\infty} (\O \times \R^3)} 
\\& \qquad \qquad + \| e^{ \frac{3 \bar{\beta}}{4} \big( \sqrt{(m_- c)^2 + |p|^2} + \frac{1}{2 c} m_- g x_3 \big) } (h^{\ell+1}_{-} (x, p) - h^{\ell}_{-} (x, p) ) \|_{L^{\infty} (\O \times \R^3)}
\\& \leq \frac{3}{4} \Big( \| e^{ \frac{3 \bar{\beta}}{4} \big( \sqrt{(m_+ c)^2 + |p|^2} + \frac{1}{2 c} m_+ g x_3 \big) } (h^{\ell}_{+} - h^{\ell-1}_{+} ) \|_{L^{\infty} (\O \times \R^3)} 
\\& \qquad \qquad + \| e^{ \frac{3 \bar{\beta}}{4} \big( \sqrt{(m_- c)^2 + |p|^2} + \frac{1}{2 c} m_- g x_3 \big) } (h^{\ell}_{-} - h^{\ell-1}_{-} ) \|_{L^{\infty} (\O \times \R^3)}  \Big),
\end{split}
\Ee
and conclude \eqref{est:h_cauchy_spec}. 

\smallskip

Next, using the condition \eqref{condition:beta_spec} and \eqref{Uest:wh^k_spec} in Proposition \ref{prop:Unif_steady_spec}, together with \eqref{est:h_cauchy} and \eqref{bound7:diff_h^l+1}, we derive for any $\ell \geq 0$,
\Be \label{bound7:diff_h^l+1_spec}
\begin{split}
& \| e^{ \frac{3 \bar{\beta}}{4} \big( \sqrt{(m_+ c)^2 + |p|^2} + \frac{1}{2 c} m_+ g x_3 \big) } (h^{\ell+1}_{+} (x, p) - h^{\ell}_{+} (x, p)) \|_{L^{\infty} (\O \times \R^3)} 
\\& \qquad \qquad + \| e^{ \frac{3 \bar{\beta}}{4} \big( \sqrt{(m_- c)^2 + |p|^2} + \frac{1}{2 c} m_- g x_3 \big) } (h^{\ell+1}_{-} (x, p) - h^{\ell}_{-} (x, p) ) \|_{L^{\infty} (\O \times \R^3)}
\\& \leq \big( \frac{3}{4} \big)^{\ell-1} \Big( \| e^{ \frac{3 \bar{\beta}}{4} \big( \sqrt{(m_+ c)^2 + |p|^2} + \frac{1}{2 c} m_+ g x_3 \big) } (h^{2}_{+} - h^{1}_{+} ) \|_{L^{\infty} (\O \times \R^3)} 
\\& \qquad \qquad \qquad + \| e^{ \frac{3 \bar{\beta}}{4} \big( \sqrt{(m_- c)^2 + |p|^2} + \frac{1}{2 c} m_- g x_3 \big) } (h^{2}_{-} - h^{1}_{-} ) \|_{L^{\infty} (\O \times \R^3)}  \Big)
\\& \leq 2 \big( \frac{3}{4} \big)^{\ell-1} \big( \| e^{ \beta \sqrt{(m_{+} c)^2 + |p|^2}} G_{+} \|_{L^\infty +  (\gamma_-)} + \| e^{ \beta \sqrt{(m_{-} c)^2 + |p|^2}} G_{-} \|_{L^\infty (\gamma_-)} \big).
\end{split}
\Ee
Hence, we conclude that $\{ h^{\ell+1} \}^{\infty}_{\ell=0}$ forms a Cauchy sequences in $L^{\infty} (\O \times \R^3)$.

\smallskip

Now following \eqref{est:cauchy_rho}, together with \eqref{bound7:diff_h^l+1_spec}, we deduce for any $\ell \geq 0$,
\Be \notag
\begin{split}
& e^{\beta' \frac{g}{2c} x_3} | \rho^{\ell} - \rho^{\ell-1} |
\\& \leq e_+ \| e^{ \frac{\bar{\beta}}{4} \big( \sqrt{(m_+ c)^2 + |p|^2} + \frac{3}{2 c} m_+ g x_3 \big) }  (h^{\ell}_{+} - h^{\ell-1}_{+} ) \|_{L^{\infty} (\O \times \R^3)} \times \frac{1}{\hat{\beta}}
\\& \ \ \ \ + e_- \| e^{ \frac{\bar{\beta}}{4} \big( \sqrt{(m_- c)^2 + |p|^2} + \frac{3}{2 c} m_- g x_3 \big) } (h^{\ell}_{-} - h^{\ell-1}_{-}) \|_{L^{\infty} (\O \times \R^3)} \times \frac{1}{\hat{\beta}}
\\& \leq 2 \big( \frac{3}{4} \big)^{\ell-1} \big( \| e^{ \beta \sqrt{(m_{+} c)^2 + |p|^2}} G_{+} \|_{L^\infty +  (\gamma_-)} + \| e^{ \beta \sqrt{(m_{-} c)^2 + |p|^2}} G_{-} \|_{L^\infty (\gamma_-)} \big) \times \frac{e_+ + e_-}{\hat{\beta}}.
\end{split}
\Ee
Similarly, following \eqref{est:cauchy_phi_x}, together with \eqref{bound7:diff_h^l+1_spec}, we obtain for any $\ell \geq 0$,
\Be \notag
\begin{split}
& \| \nabla_x \Phi^{\ell} - \nabla_x \Phi^{\ell-1} \|_\infty 
\\& \leq \mathfrak{C} (1 + \frac{2 c}{\beta^\prime g} ) \frac{1}{\hat{\beta}} \Big(
e_+ \| e^{ \frac{\bar{\beta}}{4} \big( \sqrt{(m_+ c)^2 + |p|^2} + \frac{3}{2 c} m_+ g x_3 \big) } (h^{\ell}_{+} - h^{\ell-1}_{+}) \|_{L^{\infty} (\O \times \R^3)}
\\& \qquad \qquad \qquad \qquad + e_- \| e^{ \frac{\bar{\beta}}{4} \big( \sqrt{(m_- c)^2 + |p|^2} + \frac{3}{2 c} m_- g x_3 \big) } (h^{\ell}_{-} - h^{\ell-1}_{-}) \|_{L^{\infty} (\O \times \R^3)}  \Big)
\\& \leq 2 \big( \frac{3}{4} \big)^{\ell-1} \mathfrak{C} (1 + \frac{2 c}{\beta^\prime g} ) \frac{e_+ + e_-}{\hat{\beta}}
 \big( \| e^{ \beta \sqrt{(m_{+} c)^2 + |p|^2}} G_{+} \|_{L^\infty +  (\gamma_-)} + \| e^{ \beta \sqrt{(m_{-} c)^2 + |p|^2}} G_{-} \|_{L^\infty (\gamma_-)} \big).
\end{split}
\Ee
Therefore, we conclude that $\{ \rho^\ell \}^{\infty}_{\ell=0}$ and $\{ \nabla_x \Phi^\ell \}^{\infty}_{\ell=0}$ are both Cauchy sequences in $L^{\infty} (\O)$.
\end{proof}

Finally, using the Cauchy sequences of $L^\infty$-spaces in Proposition \ref{prop:cauchy_spec}, we construct a weak solution of $(h_{\pm}, \rho, \Phi)$ solving \eqref{VP_h_spec}-\eqref{eqtn:Dphi_spec}. 

\begin{proof}[\textbf{Proof of Theorem \ref{theo:CS_spec}}]

From the assumption, then for some $g, \beta > 0$,
\be \notag
\| e^{ \beta \sqrt{(m_{\pm} c)^2 + |p|^2} } G_{\pm} \|_{L^\infty(\gamma_-)} < \infty, 
\ee
and for some $\beta \geq \tilde{\beta} > 0$,
\be \notag
\| e^{{\tilde \beta } \sqrt{(m_{\pm} c)^2 + |p|^2}} \nabla_{x_\parallel, p} G_{\pm} (x,p) \|_{L^\infty (\gamma_-)} < \infty,
\ee
and the conditions \eqref{condition:G_support_spec}-\eqref{condition:epilon_G_spec} hold for $g, \beta, \tilde{\beta} > 0$.

\smallskip

\textbf{Step 1. Regularity: Proof of \eqref{Uest:wh_spec}-\eqref{Uest:DPhi_spec} and \eqref{Uest:h_support_spec}.}
From the arguments of the Cauchy sequences of $L^\infty$-spaces in Proposition \ref{prop:cauchy_spec}, there exists 
\be \notag
h_{\pm} (x, p) \in L^\infty (\bar \O \times \R^3),
\ \text{ and } \ \rho(x), \nabla_x \Phi (x) \in L^\infty (\bar \O) 
\ \text{ with } \ 
\Phi= 0 \ \ \text{on} \ \p\O,
\ee
such that as $k \to \infty$,
\begin{align}
e^{ \frac{3 \bar{\beta}}{4} ( \sqrt{(m_{\pm} c)^2 + |p|^2} + \frac{1}{2 c} m_{\pm} g x_3) } h^{k}_{\pm}
\to e^{ \frac{3 \bar{\beta}}{4} ( \sqrt{(m_{\pm} c)^2 + |p|^2} + \frac{1}{2 c} m_{\pm} g x_3 ) } h_{\pm}
& \ \text{ in } \ L^\infty (\bar \O \times \R^3),
\label{weakconv_whst_spec} \\
e^{\beta' \frac{g}{2c} x_3} \rho^{k} \to e^{\beta' \frac{g}{2c} x_3} \rho 
& \ \text{ in } \ L^\infty (\bar \O),
\label{weakconv_rhost_spec} \\
\nabla_x \Phi^k \to \nabla_x \Phi
& \ \text{ in } \ L^\infty (\bar \O),
\label{ae_converge_par_Phist_spec}
\end{align} 
where $\bar{\beta} = \frac{\tilde \beta}{6}$ and $\beta' =  \min\{ \frac{\bar \beta}{4}, \beta \} \times \min \{ m_{-},  m_{+} \}$. This clearly implies that
\be \label{strong_conv_h_rho_spec}
\begin{split}
h^{k}_{\pm} \to h_{\pm}
&\ \text{ in } \ L^\infty (\bar \O \times \R^3) 
\ \text{ as } \ k \to \infty, 
\\ \rho^{k} \to \rho 
& \ \text{ in } \ L^\infty (\bar \O)
\ \text{ as } \ k \to \infty.
\end{split}
\ee
Using Proposition \ref{prop:Unif_steady_spec}, together with the $L^\infty$ convergence in \eqref{weakconv_whst_spec}-\eqref{ae_converge_par_Phist_spec}, we prove $(h_{\pm}, \rho, \nabla_x \Phi)$ satisfies \eqref{Uest:rho_spec}, \eqref{Uest:DPhi_spec} and \eqref{Uest:h_support_spec}.

Then, from zero Dirichlet boundary condition of \eqref{bdry:phik_spec} and $\Phi= 0$ on $\p\O$, together with \eqref{ae_converge_par_Phist}, we obtain
\be \label{ae_converge_Phist_spec}
\Phi^{k} (x) \to \Phi(x)
\ \text{ a.e. in $\O$}
\ \ \text{as} \ \ k \rightarrow \infty.
\ee
We consider the weight functions $w_{\pm} (x, p)$ as follows:
\Be \notag
w_{\pm, \beta} (x,p) = e^{ \beta \left( \sqrt{(m_{\pm} c)^2 + |p|^2} + \frac{1}{c} ( e_{\pm} \Phi (x) + m_{\pm} g x_3 ) \right) }.
\Ee
Using \eqref{ae_converge_Phist_spec}, we have
\be \label{wi_conv_w_spec}
w^{k}_{\pm, \beta} (x,p) = e^{ \beta \left( \sqrt{(m_{\pm} c)^2 + |p|^2} + \frac{1}{c} ( e_{\pm} \Phi^{k} (x) + m_{\pm} g x_3 ) \right) }
\rightarrow  w_{\pm, \beta} (x,p)
\ \text{ a.e. in $\O$}
\ \ \text{as} \ \ k \rightarrow \infty.
\ee
From \eqref{strong_conv_h_rho_spec} and \eqref{wi_conv_w_spec}, we get that
\Be \notag
w^{k}_{\pm, \beta} (x,p) h^k_{\pm}
\to w_{\pm, \beta} (x,p) h_{\pm}
\ \text{ a.e. in $\O$}
\ \ \text{as} \ \ k \rightarrow \infty.
\Ee
Thus, we conclude \eqref{Uest:wh_spec} by using \eqref{Uest:wh^k_spec}.

\smallskip

\textbf{Step 2. Existence.}
Next we will prove that $(h_{\pm}, \rho, \Phi)$ obtained in 
\eqref{weakconv_whst_spec}-\eqref{ae_converge_par_Phist_spec} is the solution to \eqref{VP_h_spec}-\eqref{eqtn:Dphi_spec} in the sense of Definition \ref{weak_sol_spec}.

Recall from the construction and Proposition \ref{prop:Unif_steady_spec}, for any $k \in \N$, $h^{k}_{\pm} (x,p)$ is the weak solution to \eqref{eqtn:hk_spec} and \eqref{bdry:hk_spec} with the field containing $\nabla_x \Phi^{k}$.
Therefore, given any $k \in \N$, then for any $\psi \in  C^\infty_{c} (\bar \O \times \R^3)$,
\begin{align}
& \ \ \ \ \iint_{\O \times \R^3} h^{k+1}_{\pm} v_{\pm} \cdot \nabla_x \psi \dd p \dd x 
- \big( \int_{\gamma_+} h^{k+1}_{\pm} \psi \dd \gamma
- \int_{\gamma_-} G_{\pm} \psi \dd \gamma 
- \int_{\gamma_-} \varepsilon h^{k}_{\pm} (x, \tilde{p}) \psi(x,p) \dd \gamma \big)
\label{weak_stk_spec} \\
& = \iint_{\O \times \R^3} h^{k+1}_{\pm} \big( e_{\pm} ( \frac{v_{\pm}}{c} \times B 
- \nabla_x \Phi^{k} ) - \nabla_x (m_{\pm} g x_3) \big) \cdot \nabla_p \psi \dd p \dd x, 
\label{weak_stk_2_spec}
\end{align}
Moreover, from \eqref{eqtn:phik} and \eqref{bdry:phik}, then for any $\varphi \in H^1_0 (\O) \cap C^\infty_c (\bar \O)$, 
\Be \label{weak_Poisson_k_spec}
\int_{\O} \nabla_x \Phi^{k} \cdot \nabla_x \varphi \dd x 
= \int_{\O} \rho^{k} \varphi \dd x.
\Ee
Suppose two test functions $\psi (x, p), \varphi (x)$ have compact support in $x$ as follows:
\Be \notag
\begin{split}
\text{spt}_x (\psi)&:= \overline{  \{x \in \O:  \psi(x, p) \neq 0 \ \text{for some } p \in \R^3\}} \subset\subset \O,
\\ \text{spt} (\varphi)&:= \overline{\{x \in \O:  \varphi(x) \neq0 \}} \subset\subset \O.
\end{split}
\Ee

For the first term in \eqref{weak_stk_spec}, from $\psi \in  C^\infty_{c} (\bar \O \times \R^3)$ and $L^{\infty}$ convergence $h^{k}_{\pm} \to h_{\pm}$ in \eqref{strong_conv_h_rho_spec}, we derive that
\Be \notag
\iint_{\O \times \R^3} h^{k + 1}_{\pm} v_{\pm} \cdot \nabla_x \psi \dd p \dd x
\rightarrow 
\iint_{\O \times \R^3} h_{\pm} v_{\pm} \cdot \nabla_x \psi \dd p \dd x
\ \text{ as } \ k \to \infty.
\Ee
Analogously, we deduce the convergence of every term in \eqref{weak_stk_spec} and \eqref{weak_stk_2_spec}.
From the convergence of every other term in \eqref{weak_stk_spec} and \eqref{weak_stk_2_spec}, we conclude that $(h_{\pm}, \nabla_x \Phi)$ solves \eqref{weak_form_spec}, that is, it is a weak solution of \eqref{VP_h_spec}-\eqref{bdry:h_spec}.

The rest of the proof directly follows from step 2 in the proof of Theorem \ref{theo:CS}, therefore we omit it.

\smallskip

\textbf{Step 3. Regularity: Proof of \eqref{Uest:Phi_xx_spec} and \eqref{est_final:rho_x_spec}-\eqref{est_final:hk_x_spec}.}
Recall from \eqref{strong_conv_h_rho_spec}, we have the $L^{\infty}$ convergence $\rho^k \to \rho$. This implies that
\be \notag
|\rho |_{C^{0,\delta}(\O)} \leq \sup_{k \in \N} |\rho^{k+1} |_{C^{0,\delta}(\O)}.
\ee
Together with \eqref{Uest:rho_spec} and \eqref{est:nabla^2phi}, we conclude that $\Phi$ constructed in \eqref{ae_converge_Phist_spec} satisfies \eqref{Uest:Phi_xx_spec}.
From Theorem \ref{theo:RS_spec}, together with \eqref{condition:tilde_beta_spec}, \eqref{condition:G_xv_spec} and \eqref{Uest:Phi_xx_spec}, one can check that the solution $(h_{\pm}, \rho, \nabla_x \Phi)$ satisfies \eqref{est_final:rho_x_spec}, \eqref{est_final:phi_C2_spec}, \eqref{est_final:hk_v_spec}, and \eqref{est_final:hk_x_spec}.

\smallskip

\textbf{Step 4. Uniqueness.}
Finally, using the weighted bound on $| \nabla_p h_{\pm} (x,p)|$ in \eqref{est_final:hk_v_spec}, we apply Theorem \ref{theo:US_spec}, and thus conclude the uniqueness of the solution $(h_{\pm}, \rho, \Phi)$.
\end{proof}

\subsection{Dynamic Solutions}
\label{sec:DS_spec}

We consider the dynamical problem \eqref{VP_F}-\eqref{Dbc:F} and \eqref{bdry:F_spec} for perturbation solutions around the steady solution $(h_\pm, \Phi_h)$ of Theorem \ref{theo:CS_spec}. 
Specifically,  the solutions are given by $F_\pm (t,x,p) = h_\pm (x,p) + f_\pm (t,x,p)$  and $\phi_F(t,x) = \Phi_h(x) + \Psi (t,x)$, where $(f_\pm, \Psi)$ solves the following equations:
\be \label{eqtn:f_spec}
\begin{split}
& \p_t f_{\pm}  + v_\pm \cdot \nabla_x f_{\pm} + \Big( e_{\pm} \big( \frac{v_\pm}{c} \times B - \nabla_x ( \Phi_h + \Psi ) \big) - \nabla_x ( m_\pm g x_3) \Big) \cdot \nabla_p f_{\pm} 
\\& = e_{\pm} \nabla_x \Psi \cdot \nabla_p h_{\pm} 
\ \ \text{in} \ \R_+ \times \O \times \R^3. 
\end{split}
\ee

For the initial condition of \eqref{VP_0}, we let 
\be \label{def:F_0_spec}
F_{\pm, 0} (x,p) = h_{\pm} (x,p) + f_{\pm, 0} (x,p) \ \ \text{in} \  \O \times \R^3,
\ee
which leads to the initial condition for $f_\pm$ given by
\be \label{f_0_spec}
\begin{split}
f_{\pm} (0,x,p) = f_{\pm, 0} (x,p) 
\ \ \text{in} \  \O \times \R^3.
\end{split}
\ee
The boundary conditions become pure specular boundary conditions:
\Be \label{bdry:f_spec} 
f_{\pm} (t,x,p) = \varepsilon f_{\pm} (t,x,\tilde{p}) \ \ \text{in} \ \R_+ \times \gamma_-,
\Ee
where $\tilde{p} = (p_1, p_2, - p_3)$.
An electric potential corresponding to the dynamical perturbation is determined by solving:
\be \label{Poisson_f_spec}
- \Delta_x	\Psi (t,x)
= \varrho (t,x) \ \  \text{in} \ \R_+ \times \O, 
\ \text{and }
\Psi(t,x) = 0  \ \   \text{on} \ \R_+ \times \p\O. ,
\ee
where a local charge density of the dynamical fluctuation is given by
\Be \label{def:varrho_spec}
\varrho(t,x) 
=  \int_{\R^3} ( e_+ f_+ + e_{-} f_{-} ) \dd p.
\Ee
For simplicity, we often let $(-\Delta_0)^{-1} \varrho$ denote $\Psi$ solving \eqref{Poisson_f_spec}. 
Moreover, for the reader's convenience, we write down the characteristics $\Z_{\pm} (s;t,x,p) = ( \X_{\pm} (s;t,x,p), \P_{\pm} (s;t,x,p) )$ for the dynamic problem \eqref{eqtn:f_spec} which is the same as \eqref{eqtn:f}: 
\Be \label{ODE_F_spec}
\begin{split}
\frac{d \X_{\pm} (s;t,x,p) }{ d s} & = \V_{\pm} (s;t,x,p) = \frac{\P_{\pm} (s;t,x,p)}{\sqrt{m^2_{\pm} + |\P_{\pm} (s;t,x,p)|^2 / c^2}}, \\
\frac{d \P_{\pm} (s;t,x,p)}{d s} & = {e_{\pm}} \big( \V_{\pm} (s;t,x,p) \times \frac{B}{c} - \nabla_x \Psi (s, \X(s;t,x,p)) 
\\& \qquad \qquad - \nabla_x \Phi_h (\X(s;t,x,p)) \big) - {m_{\pm}}  g  \mathbf{e}_3,
\end{split}
\Ee 
where $\Psi$ and $\Phi_h$ solve \eqref{Poisson_f} and \eqref{eqtn:Dphi}, respectively. At $s=t$, we have
\be \notag
\Z_{\pm} (t;t,x,p) = (\X_{\pm} (t;t,x,p), \P_{\pm} (t;t,x,p))  = (x, p) = z_{\pm}.
\ee

\medskip

Next, we provide a precise definition of weak solutions to the dynamical problem.

\begin{definition}
\label{weak_sol_dy_spec} 

Suppose $(h_{\pm}, \rho_h, \Phi_h)$ solves \eqref{VP_h_spec}-\eqref{def:rho_spec} in the sense of Definition \ref{weak_sol_spec}.  
	
(a) We say that $(f_{\pm}, \nabla_x \Psi) \in \big( L^2_{loc}(\R_+ \times \O \times \R^3) \; \cap \; L^2_{loc}(\R_+ \times \p\O \times \R^3; \dd \gamma) \big) \times L^2_{loc}(\R_+ \times \O \times \R^3)$ is a weak solution of \eqref{eqtn:f_spec}, \eqref{f_0_spec}, and \eqref{bdry:f_spec}, if all terms below are bounded and  and the following condition holds for any $t \in R_+$ and test function $\psi \in C^\infty_c (\R_+ \times \bar \O \times \R^3)$, 
\Be \label{weak_form_dy_spec}
\begin{split}
& \iint_{\R_+ \times \O \times \R^3} f_{\pm} (t,x,p) \p_t \psi(t, x, p) \dd p \dd x \dd t
+ \iint_{\R_+ \times \O \times \R^3} f_{\pm} (t,x,p) v \cdot \nabla_x \psi(t, x, p) \dd p \dd x \dd t
\\& - \iint_{\R_+ \times \O \times \R^3} f_{\pm} (t, x, p) \big( e_{\pm} ( \frac{v_\pm}{c} \times B 
- \nabla_x (\Phi + \Psi) ) - \nabla_x (m_{\pm} g x_3) \big) \cdot \nabla_p \psi (t,x,p)  \dd p \dd x \dd t
\\& = \iint_{\R_+ \times \p\O \times \{ p_3 <0 \}} f_{\pm} (t,x,p) \psi(t,x,p) \dd \gamma \dd t 
- \iint_{ \O \times \R^3} f_{\pm, 0} (x,p) \psi(0, x,p) \dd p \dd x
\\& - \iint_{\R_+ \times \p\O \times \{ p_3 > 0 \}} \varepsilon f_{\pm} (t, x, \tilde{p}) \psi(t,x,p) \dd \gamma,
\end{split}
\Ee
where $\dd \gamma := |v_3| \dd S_x \dd p$ denotes the boundary measure.
	
(b) We say that $(f_{\pm}, \varrho) \in L^2_{loc}(R_+ \times \O \times \R^3) \times L^2_{loc} (R_+ \times \O)$ is a weak solution of \eqref{def:varrho_spec}, if all terms below are bounded and the following condition holds for any $t \in R_+$ and test function $\vartheta \in H^1_0 (\O) \cap C^\infty_c (\bar \O)$, 
\Be \label{weak_form_2_dy_spec}
\int_{\O} \vartheta (x) \varrho (t, x) \dd x 
= \int_{\O} \vartheta (x) \int _{\R^3} ( e_+ f_{+} (t, x, p) + e_{-} f_{-} (t, x, p) ) \dd p \dd x.
\Ee
	
(c) We say that $(\varrho, \Psi) \in L^2_{loc}(\O  ) \times W^{1,2}_{loc} (\O)$ is a weak solution of \eqref{Poisson_f_spec}, if all terms below are bounded and the following condition holds for any $t \in R_+$ and test function $\varphi \in H^1_0 (\O) \cap C^\infty_c (\bar \O)$, 
\Be \label{weak_form_3_dy_spec}
\int_{\O} \nabla_x \Psi (t, x) \cdot \nabla_x \varphi (x) \dd x 
= \int_{\O} \varrho (t, x) \varphi (x) \dd x.
\Ee
\end{definition}

Recall $M, L$ and $\lambda$ defined in \eqref{set:M}, \eqref{set:L} and \eqref{lambda}, respectively.
Now we state the main theorem of the dynamic problem.

\begin{theorem}[Main theorem of the Dynamic problem]
\label{theo:CD_spec}

We assume that all assumptions in Theorem \ref{theo:CS_spec} hold with $g > 0, \beta \geq \tilde \beta >0$. 
Furthermore, assume there exists a constant $\tpx > 1$ such that the initial data $f_{\pm, 0} (x,p)$ in \eqref{f_0_spec} satisfies
\be \label{condition:f_0_spec}
\begin{split}
f_{\pm, 0} (x,p) = 0
\ \text{ for any } \
\frac{c}{m_{\pm} g} |p| + \frac{3}{4} x_{3} + 1 \geq \tpx.
\end{split}
\ee
and suppose that
\Be \label{choice:g_spec}
M \leq \beta e^{ - \frac{m_{\pm} g}{24} \beta}
\ \text{ and } \ 
L \leq \min \{ \tilde{\beta} e^{ - \frac{m_{\pm} g}{24} \tilde{\beta} }, \frac{1}{1024} \beta^2 e^{ - \frac{m_{\pm} g}{48} \beta } \}.
\Ee
Finally, suppose $\varepsilon$ in the boundary condition \eqref{bdry:f_spec} satisfies that	
\be \label{condition:epilon_G_dy_spec}
2 \varepsilon 
< \exp \Big\{ - \lambda \Big( \big( \frac{1}{\min \{ \frac{g}{4 \sqrt{2}}, \frac{c}{\sqrt{10}} \}} + 1 \big) \times \tpx + 1 \Big) \Big\}.
\ee	
Then there exists a unique solution $(f_\pm, \Psi)$ to \eqref{eqtn:f_spec}-\eqref{Poisson_f_spec} in the sense of Definition \ref{weak_sol_dy_spec}. 
Moreover, the following estimates hold:
\begin{align}
& \sup_{0 \leq t < \infty} \| e^{ \frac{\beta}{2} \sqrt{(m c)^2 + |p|^2} } e^{ \frac{m g}{4 c} \beta x_3} f_{\pm} (t,x,p)  \|_{L^\infty  (\O \times \R^3)} \leq M, 
\label{Uest:wh_dy_spec} \\
& \sup_{0 \leq t < \infty} \|   \nabla_x  \Psi \|_{L^\infty (\bar{\O})}  
\leq \min \left\{\frac{m_+}{e_+}, \frac{m_-}{e_-} \right\} \times \frac{g}{48}.
\label{Uest:DxPsi_spec}
\end{align}
In addition, $\phi_F =  \Phi_h + \Psi$ satisfies the following bounds:
\begin{align}
& \sup_{0 \leq t < \infty} \| \nabla_x \phi_F \|_{L^\infty (\bar{\O})} 
\leq \min \left\{\frac{m_+}{e_+}, \frac{m_-}{e_-} \right\} \times \frac{g}{2},
\label{Uest:Dxphi_F_spec} \\ 
& (1 + B_3 + \| \nabla_x ^2 \phi_F  \|_\infty) + \| \p_t \p_{x_3} \phi_F (t, x_\parallel , 0) \|_{L^\infty(\p\O)} 
\leq \frac{\hat{m} g}{24} \tilde{\beta}.  
\label{Uest:D2xD3tphi_F_spec}
\end{align}	
Furthermore, given $\alpha_{\pm, F} (t, x, p)$ defined in $\eqref{alpha_F}$, then $f_{\pm}$ satisfies that
\begin{align}
& \| e^{ \frac{\tilde{\beta}}{4} \sqrt{(m_{\pm} c)^2 + |p|^2} } e^{ \frac{m_{\pm} g}{8 c} \tilde{\beta} x_3} \nabla_p f_{\pm} (t,x,p) \|_{L^\infty(\O \times \R^3)} 
\lesssim 2 e^{ 2 } \tilde \beta + \frac{1}{24} {\tilde{\beta}}^2,
\label{Uest_final:F_v:dyn_spec} \\
& e^{ \frac{\tilde{\beta}}{4} \sqrt{(m_{\pm} c)^2 + |p|^2} } e^{ \frac{m_{\pm} g}{8 c} \tilde{\beta} x_3} \big| \nabla_x f_{\pm} (t,x,p) \big|
\lesssim 2 e^{ 2 } \tilde{\beta} + ( \frac{ \mathbf{1}_{|x_3| \leq 1} }{\alpha_{\pm, F} (t,x,p)} + \frac{1}{24} \tilde{\beta} ) \tilde{\beta}.
\label{Uest_final:F_x:dyn_spec}	
\end{align}
and
\be \label{Uest:f_support_spec} 
f^{\ell+1}_{\pm} (t, x, p) = 0
\ \text{ for any } \
\frac{4 c}{m_{\pm} g} |p| + 3 x_{3} \geq \tpx,
\ee
Finally, the solution $(f(t), \varrho(t), \nabla_x \Psi(t))$ exhibits exponential decay as follows:
\begin{align}
& \sup_{0 \leq t < \infty} e^{ \lambda t} \|  e^{\frac{\beta}{8} \sqrt{(m_{\pm} c)^2 + |p|^2} + \frac{m_{\pm} g}{16 c} \beta x_3} f_{\pm} (t ) \|_{L^\infty (\O \times \R^3)}
\lesssim \beta, 
\label{Udecay:f_spec} \\
& \sup_{0 \leq t < \infty} e^{ \lambda t} \| e^{  \frac{3 \lambda }{c}  x_3} \varrho (t)\|_{L^\infty(\O)}
\lesssim \frac{e_+ + e_-}{\beta^2}, 
\label{Udecay:varrho_spec} \\
& \sup_{0 \leq t < \infty} e^{ \lambda t} \|\nabla_x \Psi(t)\|_{L^\infty (\O)}
\lesssim \big( \frac{e_+ + e_-}{\beta^2} \big) \times (1 + \frac{16}{g \beta}). 
\label{Udecay:DxPsi_spec}
\end{align} 
\end{theorem}

\smallskip

To prove Theorem \ref{theo:CD_spec}, we follow steps similar to those used in the dynamic problem, considering only the inflow boundary condition \eqref{VP_F}-\eqref{bdry:F}.
In the following, we provide a brief overview of the steps outlined in this subsection.

Given a steady solution $(h_{\pm}, \nabla_x \Phi)$ to \eqref{VP_h_spec}-\eqref{eqtn:Dphi_spec}, we construct the sequences $(f^{\ell+1}_{\pm}, \varrho^{\ell}, \nabla_x \Psi^{\ell})$ for $\ell \geq 1$ in \eqref{eqtn:fell_spec}-\eqref{bdry:Psi_fell_spec}. 
Compared to the case with pure inflow boundary conditions, the only difference is the boundary conditions applied to $f^{\ell+1}_{\pm}$.

We remark that all properties in Section \ref{sec:char}, Lemma \ref{lem:w/w} and Lemma \ref{lem:Ddb} also apply to the dynamic problem \eqref{eqtn:f_spec}-\eqref{Poisson_f_spec} involving specular boundary conditions, as the change in boundary conditions does not affect the characteristic trajectory determined by the Vlasov equations \eqref{eqtn:f_spec}.

Next, suppose $f_{\pm} (t,x,p)$ is a solution to \eqref{eqtn:f_spec}-\eqref{Poisson_f_spec} and it satisfies that
\be \label{scheme:f_support_spec}
f_{\pm} (t, x, p) = 0
\ \text{ for any } \
\frac{4 c}{m_{\pm} g} |p| + 3 x_{3} \geq \tpx.
\ee
The specular boundary condition, together with the dominance of the downward gravity in the field and the weighted $L^\infty$ estimate on $f^{\ell+1}_{\pm} (t,x,p)$, allows us to conclude the asymptotic stability in Theorem \ref{theo:AS_spec}. 
Then we establish a priori estimate of $(F, \phi_F)$ solving \eqref{VP_F}-\eqref{Dbc:F} and \eqref{bdry:F_spec}.
Due to the specular boundary condition, some regularity terms are
altered (see Section \ref{sec:RD_spec}).
Thanks to the fact that the steady solution $h_{\pm} (x,p)$ has compact support with respect to $p$, and by using the assumption \eqref{scheme:f_support_spec} and the condition \eqref{condition:epilon_G_dy_spec}, we derive the weighted $L^\infty$ estimate on $\nabla_{x, p} F_{\pm}$ in Theorem \ref{theo:RD_spec}.

After that, we establish the key property that the sequence $\{ f^{\ell}_{\pm} \}^{\infty}_{\ell=0}$ satisfies that for every $\ell \geq 1$,
\be \label{scheme:f_l_support_spec}
f^{\ell}_{\pm} (t, x, p) = 0
\ \text{ for any } \
\frac{4 c}{m_{\pm} g} |p| + 3 x_{3} \geq \tpx.
\ee
Different from the steady problem, each dynamic weight is not invariant along the dynamic characteristics. For example, for any $s \in [ t - \tBp^{\ell+1} (t,x,p), t + \tFp^{\ell+1} (t,x,p)]$,
\Be \notag
\begin{split}
& \ \ \ \ \frac{d}{ds} \Big( \sqrt{(m_{\pm} c)^2 + |\P^{\ell+1}_{\pm} (s;t,x,p)|^2} + \frac{1}{c} \big( e_{\pm} \Phi (\X^{\ell+1}_{\pm} (s;t,x,p)) + m_{\pm} g \X^{\ell+1}_{\pm, 3} (s;t,x,p) \big) \Big) 
\\& \qquad = - \frac{e_{\pm}}{c} \nabla_x \Psi^{\ell} (s, \X^{\ell+1}_{\pm} (s;t,x,p)) \cdot \V^{\ell+1}_{\pm} (s;t,x,p).
\end{split}
\Ee 
Moreover, the specular boundary condition allows the trajectory to bounce infinitely.
This asks us to provide a better control on $\nabla_x \Psi^{\ell}$ to ensure that $|\P^{\ell+1}_{\pm}|$ remains uniformly bounded with respect to time $t$.
Using the asymptotic stability in Theorem \ref{theo:AS_spec} and mathematical induction, in Proposition \ref{prop:DC_spec} we obtain that
\be \notag
\sup_{0 \leq t < \infty} e^{ \lambda t} \|\nabla_x \Psi^{\ell} (t,x) \|_{L^\infty (\O)}
\lesssim \big( \frac{e_+ + e_-}{\beta^2} \big) \times (1 + \frac{16}{g \beta}),
\text{ for every }
\ell \in \N
.
\ee
This leads to the uniform-in-$\ell$ control on $f^{\ell}_{\pm}$ as presented in \eqref{scheme:f_l_support_spec}.
The remaining steps to establish existence and uniqueness are analogous to those for dynamic solutions under inflow boundary conditions. 

\subsubsection{Construction} \label{sec:DC_spec}

Suppose $(h_{\pm}, \nabla_x \Phi)$ is the steady solution to \eqref{VP_h_spec}-\eqref{eqtn:Dphi_spec}. We construct solutions to the dynamic problem \eqref{eqtn:f_spec}-\eqref{Poisson_f_spec} via the following sequences: for any $\ell \in \N$,
\be \label{eqtn:fell_spec}
\begin{split}
& \p_t f^{\ell+1}_{\pm}  + v_\pm \cdot \nabla_x f^{\ell+1}_{\pm} + \Big( e_{\pm} \big( \frac{v_\pm}{c} \times B - \nabla_x ( \Phi + \Psi^{\ell} ) \big) - \nabla_x ( m_\pm g x_3) \Big) \cdot \nabla_p f^{\ell+1}_{\pm} 
\\& \ \ \ \ = e_{\pm} \nabla_x \Psi^{\ell} \cdot \nabla_p h_{\pm} \ \ \text{in} \ \R_+ \times  \O \times \R^3,
\end{split}
\ee
and
\begin{align}
f^{\ell+1}_{\pm} (t,x,p) = \varepsilon f^{\ell}_{\pm} (t,x,\tilde{p})
& \ \ \text{on} \ \R_+ \times \gamma_-,
\label{bdry:fell_spec} \\
f^{\ell+1}_{\pm, 0} (x, p) = F_{\pm, 0} - h_{\pm} 
& \ \ \text{in} \ \{ 0 \} \times \O \times \R^3, 
\label{initial:fell_spec} \\
\varrho^{\ell} (t, x) 
=  \int_{\R^3} ( e_+ f^{\ell}_+ + e_{-} f^{\ell}_{-} ) \dd p 
& \ \ \text{in} \ \R_+ \times \O, &
\label{varrhoell_spec}	\\
- \Delta \Psi^\ell (t, x) = \varrho^\ell (t, x) 
& \ \ \text{in} \ \R_+ \times \O,
\label{Poisson_fell_spec} \\
\Psi^\ell (t, x) = 0 
& \ \ \text{on} \ \R_+ \times \p\O,
\label{bdry:Psi_fell_spec}
\end{align}
where $\tilde{p} = (p_1, p_2, - p_3)$ and the initial setting $f^0_{\pm} = 0$ and $(\varrho^0,  \nabla_x \Psi^0) = (0, \mathbf{0})$.

\smallskip

First, we construct $f^1_{\pm}$ solving \eqref{eqtn:fell_spec}-\eqref{initial:fell_spec} from $(\varrho^0, \nabla_x \Psi^0) = (0, \mathbf{0})$. 
We consider the Lagrangian formulation of the following characteristics:
\Be \label{Z1_spec}
\Z^{1}_{\pm} (s;t,x,p) = ( \X^{1}_{\pm} (s;t,x,p), \P^{1}_{\pm} (s;t,x,p) ),
\Ee
which solves
\Be \notag 
\big( \X^{1}_{\pm} (t;t,x,p), \P^{1}_{\pm} (t;t,x,p) \big) = (x, p), 
\Ee
and
\Be \label{ODE1_spec}
\\ \frac{d \X^1_{\pm}}{d t} = \V^1_{\pm} = \frac{ \P^1_{\pm} }{\sqrt{m^2_{\pm} + |P^1_{\pm}|^2 / c^2}}, \ \
\frac{d \P^1_{\pm} }{d t} = e_{\pm} \big( \V^1_{\pm} \times \frac{B}{c} - \nabla_x \Phi \big) - {m_{\pm}} g \mathbf{e}_3. 
\Ee
Then from \eqref{varrhoell_spec}, as long as it is well-defined, we obtain
\[
\varrho^{1} (t, x) 
=  \int_{\R^3} ( e_+ f^{1}_+ + e_{-} f^{1}_{-} ) \dd p.
\]
Using \eqref{phi_rho} and Lemma \ref{lemma:G}, together with \eqref{Poisson_fell_spec} and \eqref{bdry:Psi_fell_spec}, we derive $\Psi^1$ and $\nabla_x \Psi^1$ by
\be \notag
- \Delta \Psi^1 (t, x) = \varrho^1 (t, x).
\ee

\smallskip

Second, we construct $(f^{\ell+1}_{\pm}, \varrho^{\ell}, \nabla_x \Psi^{\ell})$ solving \eqref{eqtn:fell_spec}-\eqref{bdry:Psi_fell_spec} for $\ell \geq 1$ by iterating the process. Given $\nabla_x \Psi^{\ell}$, we construct $f^{\ell+1}_{\pm}$ along the characteristics as follows:
\Be \label{Zell_spec}
\Z^{\ell+1}_{\pm} (s;t,x,p) 
= (\X^{\ell+1}_{\pm} (s;t,x,p), \V^{\ell+1}_{\pm} (s;t,x,p) ), 
\Ee
which solves
\Be \notag 
\big( \X^{\ell+1}_{\pm} (t;t,x,p), \P^{\ell+1}_{\pm} (t;t,x,p) \big) = (x, p), 
\Ee
and
\Be \label{ODEell_spec}
\begin{split}
& \frac{d \X^{\ell+1}_{\pm} }{ds} = \V^{\ell+1}_{\pm} = \frac{ \P^{\ell+1}_{\pm} }{\sqrt{m^2_{\pm} + |P^{\ell+1}_{\pm}|^2 / c^2}},   
\\& \frac{d \P^{\ell+1}_{\pm} }{ds} = e_{\pm} \big( \V^{\ell+1}_{\pm} \times \frac{B}{c} - \nabla_x ( \Phi + \Psi^\ell ) \big) - {m_{\pm}} g \mathbf{e}_3.
\end{split}
\Ee
Using the Peano theorem, together with the bound of $\varrho^\ell$ (see Proposition \ref{prop:DC_spec}) and the continuity of $\nabla_x \Psi^\ell$ (see \eqref{phi_rho}), we conclude the existence of solutions (not necessarily unique). 
Then we obtain $\varrho^{\ell+1}$ from \eqref{varrhoell_spec} and solve $\nabla_x \Psi^{\ell+1}$ from \eqref{Poisson_fell_spec} and \eqref{bdry:Psi_fell_spec}.

\smallskip

Next, since the change in boundary conditions doesn't influence the characteristic trajectory, the characteristics in \eqref{Zell_spec} and \eqref{ODEell_spec} are the same as the characteristics in \eqref{Zell} and \eqref{ODEell}.
Hence, for every iteration $\ell \geq 0$, suppose $(\X^{\ell}, \P^{\ell})$ exists we reuse the backward exit time $\tBp^\ell$, position $\xBp^{\ell}$, and momentum $\pBp^{\ell}$ defined in \eqref{def:tB_l}.
Similar to \eqref{def:flux_ell}-\eqref{identity:Psi_t_ell}, we define the following: for any $\ell \in \N$,
\begin{align}
b^{\ell} (t,x) := \int_{\R^3} (v_+ e_+ f^{\ell}_+ + v_{-} e_{-} f^{\ell}_{-} ) \dd p \ \ \text{in} \ \R_+ \times \O,\label{def:flux_ell_spec} \\
\p_t\varrho^\ell + \nabla_x \cdot b^\ell =0 \ \ \text{in} \ \R_+ \times \O,
\label{cont_eqtn_ell_spec} \\
\p_t \Psi^\ell (t,x)  = (-\Delta_0)^{-1} \p_t \varrho^\ell(t,x)  = - (-\Delta_0)^{-1}  (\nabla_x \cdot b^\ell) (t,x)\ \ \text{in} \ \R_+ \times \O. \label{identity:Psi_t_ell_spec}
\end{align}
In addition, we reuse the weight function defined in \eqref{w^ell} and \eqref{w^ell_t=0} as follows:
\Be \label{w^ell_spec}
\w^{\ell+1}_{\pm} (t,x,p) = \w^{\ell+1}_{\pm, \beta} (t,x,p) = e^{ \beta \left( \sqrt{(m_{\pm} c)^2 + |p|^2} + \frac{1}{c} \big( e \big( \Phi (x) + \Psi^\ell (t,x) \big) + m_{\pm} g x_3 \big) \right) }.
\Ee
and at the initial time $t=0$, 
\Be \label{w^ell_t=0_spec}
\w^{\ell+1}_{\pm} (0,x,p) = \w^{\ell+1}_{\pm, \beta, 0} (x,p)
= e^{ \beta \left( \sqrt{(m_{\pm} c)^2 + |p|^2} + \frac{1}{c} \big( e_{\pm} ( \Phi (x) + \Psi^{\ell} (0,x) ) + m_{\pm} g x_3 \big) \right) }.
\Ee
Analogous to \eqref{dDTE^ell} and \eqref{eq:energy_conservation_dy}, each dynamic weight is not invariant along the dynamic characteristics: for any $s \in [ t - \tBp^{\ell+1} (t,x,p), t + \tFp^{\ell+1} (t,x,p)]$,
\Be \label{dDTE^ell_spec}
\begin{split}
& \frac{d}{ds} \Big( \sqrt{(m_{\pm} c)^2 + |\P^{\ell+1}_{\pm} (s;t,x,p)|^2} + \frac{1}{c} \big( e_{\pm} \big( \Phi (\X^{\ell+1}_{\pm} (s;t,x,p)) 
\\& \qquad + \Psi^{\ell} (s, \X^{\ell+1}_{\pm} (s;t,x,p)) \big) + m_{\pm} g \X^{\ell+1}_{\pm, 3} (s;t,x,p) \big) \Big) 
\\& = \frac{e_{\pm}}{c} \p_t \Psi^{\ell} (s, \X^{\ell+1}_{\pm} (s;t,x,p)) 
\\& = \frac{e_{\pm}}{c} (-\Delta_0)^{-1} \p_t \varrho^\ell(t,x) 
= - \frac{e_{\pm}}{c} (-\Delta_0)^{-1}  (\nabla_x \cdot b^{\ell} ) (s, \X^{\ell+1}_{\pm} (s;t,x,p)),
\end{split}
\Ee
and
\Be \label{dDTE^ell_2_spec}
\begin{split}
& \ \ \ \ \frac{d}{ds} \Big( \sqrt{(m_{\pm} c)^2 + |\P^{\ell+1}_{\pm} (s;t,x,p)|^2} + \frac{1}{c} \big( e_{\pm} \Phi (\X^{\ell+1}_{\pm} (s;t,x,p)) + m_{\pm} g \X^{\ell+1}_{\pm, 3} (s;t,x,p) \big) \Big) 
\\& \qquad = - \frac{e_{\pm}}{c} \nabla_x \Psi^{\ell} (s, \X^{\ell+1}_{\pm} (s;t,x,p)) \cdot \V^{\ell+1}_{\pm} (s;t,x,p).
\end{split}
\Ee 
Analogous to \eqref{w_bdry} and \eqref{w_initial}, we derive that
\begin{align}
\w^{\ell+1}_{\pm, \beta} (t,x,p) 
= e^{\beta \sqrt{(m_{\pm} c)^2 + |p|^2} } 
& \ \ \text{on} \ (t,x,p) \in \R_+ \times \p\O \times \R^3,
\label{w_bdry_spec} \\
\w^{\ell+1}_{\pm, \beta} (0,x,p) = \w_{\pm, \beta, 0} (x, p)
& \ \ \text{in} \ (x,p) \in \bar \O \times \R^3.
\label{w_initial_spec}
\end{align}

\smallskip

Finally, we define the following: for any $\ell \in \N$,
\be \label{def:Fell_spec}
F^{\ell+1}_{\pm} (t,x,p) = h_{\pm} (x,p) + f^{\ell+1}_{\pm} (t,x,p).
\ee
Under direct computation, we obtain that, for any $\ell \in \N$,
\be \label{eqtn:Fell_spec}
\begin{split}
& \p_t F^{\ell+1}_{\pm}  + v_\pm \cdot \nabla_x F^{\ell+1}_{\pm} + \Big( e_{\pm} \big( \frac{v_\pm}{c} \times B - \nabla_x ( \Phi + \Psi^{\ell} ) \big) - \nabla_x ( m_\pm g x_3) \Big) \cdot \nabla_p F^{\ell+1}_{\pm} 
\\& \ \ \ \ = 0 \ \ \text{in} \ \R_+ \times  \O \times \R^3,
\end{split}
\ee
and
\be \label{bdry_initial:Fell_spec}
F^{\ell+1}_{\pm} (0, x, p) = F_{\pm, 0} 
\ \ \text{in} \ \O \times \R^3,
\ \
F^{\ell+1}_{\pm} (t, x, p) = G_{\pm} (x,p) + \varepsilon F^{\ell}_{\pm} (t,x,\tilde{p})
\ \ \text{on} \ \gamma_-.
\ee
We also define 
\be \label{eqtn:phiFell_spec}
\phi_{F^{\ell}} (t, x) = \Phi (x) + \Psi^{\ell} (t, x),
\ee
which solves the following Poisson equation: 
\Be \label{Poisson_Fell_spec}
\begin{split}
- \Delta \phi_{F^{\ell}} (t, x)
= \int_{\R^3} ( e_+ F^{\ell}_+ + e_{-} F^{\ell}_{-} ) \dd p
\ \ \text{in} \ \R_+ \times \O,
\ \
\phi_{F^{\ell}} (t, x) = 0 \ \ \text{on} \ \R_+ \times \p\O.
\end{split}
\Ee

We remark that for any $\ell \in \N$, $F^{\ell+1}_{\pm}$ and $f^{\ell+1}_{\pm}$ share the same characteristic $\Z$ in \eqref{Z1_spec}, \eqref{ODE1_spec} and \eqref{Zell_spec}, \eqref{ODEell_spec}. Therefore, they share the same backward exit time, position, and momentum for every iteration.
Furthermore, one can check that Lemma \ref{lem:tB_ell} holds for the backward and forward exit time to the characteristics $\Z^{\ell+1}_{\pm} (s;t,x,p) = (\X^{\ell+1}_{\pm} (s;t,x,p), \P^{\ell+1}_{\pm} (s;t,x,p) )$ solving \eqref{eqtn:fell_spec}, and we omit the proof.

\subsubsection{Asymptotic Stability Criterion} \label{sec:AS_spec}

In this section, we always assume $f_{\pm} (t,x,p)$ is a Lagrangian solution to \eqref{eqtn:f_spec}-\eqref{Poisson_f_spec} for a given  $\nabla_p h$.
The main purpose of this section is to prove Theorem \ref{theo:AS_spec}, which shows a conditional asymptotic stability result of the dynamic perturbation. 

Further, we always suppose \eqref{Uest:DPhi_spec}, \eqref{Uest:DxPsi_spec} and \eqref{Uest:Dxphi_F_spec} hold, which are the same conditions as in the case of the dynamic problem only under the effect of the inflow boundary conditions \eqref{eqtn:f}-\eqref{Poisson_f}.
Since the characteristics in \eqref{ODE_F_spec} are the same as the characteristics in \eqref{ODE_F}, Lemmas \ref{lem:w/w} and \ref{lem:Ddb} hold for the the characteristics $\Z_{\pm} (s;t,x,p) = ( \X_{\pm} (s;t,x,p), \P_{\pm} (s;t,x,p) )$ solving \eqref{ODE_F_spec}, and we omit the proofs of these results.

From the boundary condition \eqref{f_0_spec}, together with \eqref{Lform:f}, the Lagrangian formulation of $f$ is 
\Be \label{form:f_spec}
f_{\pm} (t,x,p) 
= \mathcal{I}_{\pm} (t,x,p) 
+ \mathcal{N}_{\pm} (t,x,p)
+ \mathcal{S}_{\pm} (t,x,p), 
\Ee 
where 
\be \label{form:I_spec}
\mathcal{I}_{\pm} (t,x,p ) 
:= \mathbf{1}_{t \leq \tB (t,x,p)} f_{\pm} (0, \Z_{\pm} (0;t,x,p) ),  
\ee
and
\be \label{form:N_spec}
\mathcal{N}_{\pm} (t,x,p ) 	
:= \int^t_{ \max\{0, t - \tB (t,x,p)\}} e_{\pm} \nabla_x \Psi (s, \X_{\pm} (s;t,x,p)) \cdot \nabla_p h_{\pm} ( \Z_{\pm} (0;t,x,p) ) \dd s,
\ee
and
\be \label{form:S_spec}
\mathcal{S}_{\pm} (t,x,p ) 	
:= \mathbf{1}_{t > \tB (t,x,p)} \varepsilon f_{\pm} (t - \tBp (t,x,p), \xBp (t,x,p), \tpBp (t,x,p) ), 
\ee
where $\tpBp = (p_{\mathbf{B}, \pm, 1}, p_{\mathbf{B}, \pm, 2}, - p_{\mathbf{B}, \pm, 3})$.

\smallskip

We start with the estimates on $\mathcal{I}_{\pm} (t,x,p)$ and $\mathcal{N}_{\pm} (t,x,p)$.

\begin{prop} 
\label{prop:decay_spec}

Suppose the condition \eqref{Uest:DPhi_spec}, \eqref{Uest:DxPsi_spec} and \eqref{Uest:Dxphi_F_spec} hold. 
Let $\hat{m} = \min\{m_+, m_- \}$, we further assume that for $g, \beta>0$,
\be \label{Bootstrap_f_first_spec}
\sum\limits_{i = \pm} \sup_{0 \leq t < \infty } \| e^{ \frac{\beta}{2} \big( \sqrt{(m_i c)^2 + |p|^2} + \frac{m_i g}{2 c} x_3 \big) } f_i (t) \|_{\infty} 
\leq \frac{ \beta^2 \ln (2) }{ 8 \big( \pi \max \{ e_+, e_- \} \big)^2 (1 + \frac{4 c}{\hat{m} g \beta})}.
\ee	
Then 
\be \label{est:I_spec}
| \mathcal{I}_{\pm} (t,x,p) |
\leq e^{2 + \frac{m_{\pm} g}{24} \beta} \| \w_{\pm, \beta, 0 }   f_{\pm, 0 } \|_{L^\infty_{x,p}} 
e^{- \frac{m_{\pm} g}{24} \beta t } 
e^{ - \frac{\beta}{4} \sqrt{(m_{\pm} c)^2 + |p|^2} } e^{ - \frac{m_{\pm} g}{8 c} \beta x_3}.
\ee
Moreover, let $\lambda = \frac{g}{48} \hat{m} \beta$, $\nu = \frac{g}{16 c} \hat{m} \beta$. Assume that
\be \notag
\sup\limits_{s \in [0, t] } e^{ \lambda s} \| e^{ \nu x_3} \varrho (s) \|_{L^\infty(\O)} < \infty.
\ee 
Then
\be \label{est:N_spec}
\begin{split}
| \mathcal{N}_{\pm} (t,x,p) |
& \leq  e^{- \frac{\beta}{8} \sqrt{(m_{\pm} c)^2 + |p|^2} } e^{- \frac{m_{\pm} g}{16 c} \beta x_3} e^{ - \lambda t} 
\\& \ \ \ \ \times e^{ \lambda } 4 e_{\pm} \sup\limits_{s \in [0, t] } e^{ \lambda s} \| e^{ \nu x_3} \varrho (s) \|_{L^\infty(\O)} (1 + \frac{1}{\nu}) \| w_{\pm, \beta} \nabla_p h_{\pm} \|_{\infty}.
\end{split}
\ee
\end{prop}

\begin{proof} 

We omit the proof, since it follows directly from Proposition \ref{prop:decay}.
\end{proof}

Now we are ready to prove the main result of this section.

\begin{theorem}[Asymptotic Stability Criterion] \label{theo:AS_spec}
 
Suppose $(h, \Phi)$ solves \eqref{VP_h_spec}-\eqref{eqtn:Dphi_spec} in the sense of Definition \ref{weak_sol_spec}. 
Suppose $(f, \varrho, \Psi)$ is a solution to \eqref{eqtn:f_spec}-\eqref{Poisson_f_spec}.
Set $\hat{m} = \min\{m_+, m_- \}$, we assume the following conditions hold: for $g, \beta>0$,
\be \label{Bootstrap_f_spec}
\sum\limits_{i = \pm} \sup_{0 \leq t < \infty } \| e^{ \frac{\beta}{2} \big( \sqrt{(m_i c)^2 + |p|^2} + \frac{m_i g}{2 c} x_3 \big) } f_i (t) \|_{\infty} 
\leq \frac{ \beta^2 \ln (2) }{ 8 \big( \pi \max \{ e_+, e_- \} \big)^2 (1 + \frac{4 c}{\hat{m} g \beta})},
\ee	
and	
\be \label{Bootstrap_spec} 
\begin{split}
\sup_{0 \leq t < \infty}\| \nabla_x \big( \Phi + \Psi (t) \big) \|_{L^\infty(\O)}  
& \leq \min \left\{\frac{m_+}{e_+}, \frac{m_-}{e_-} \right\} \times \frac{g}{2},
\\ \sup_{0 \leq t < \infty} \| \nabla_x \Psi (t) \|_{L^\infty(\O)} & \leq \min \left\{\frac{m_+}{e_+}, \frac{m_-}{e_-} \right\} \times \frac{g}{48}.
\end{split}
\ee
Set $\lambda = \frac{g}{48} \hat{m} \beta$, $\nu = \frac{g}{16 c} \hat{m} \beta$, we further assume that
\Be \label{condition:Dvh_spec}
8 e^{ \lambda } (1 + \frac{1}{\nu}) \big( e^2_+ \| w_{+,\beta} \nabla_p h_+ \|_{\infty} + e^2_- \| w_{-,\beta} \nabla_p h_- \|_{\infty} \big) \frac{512}{\beta^3} \leq \frac{1}{2},
\Ee
and suppose there exists $\tpx > 1$ such that 
\be \label{condition2:f_support_spec} 
f_{\pm} (t, x, p) = 0
\ \text{ for any } \
\frac{4 c}{m_{\pm} g} |p| + 3 x_{3} \geq \tpx,
\ee
and $\w_{\pm, \beta, 0} (x, p)$ defined in \eqref{W_t=0} satisfies
\be \notag 
\|\w_{\pm, \beta, 0} f_{\pm, 0} \|_{L^\infty (\O \times \R^3)} <\infty.
\ee
Finally, suppose $\varepsilon$ in the boundary condition \eqref{bdry:f_spec} satisfies that	
\be \label{condition2:tilde_epilon_spec}
2 \varepsilon 
< \exp \Big\{ - \lambda \Big( \big( \frac{1}{\min \{ \frac{g}{4 \sqrt{2}}, \frac{c}{\sqrt{10}} \}} + 1 \big) \times \tpx + 1 \Big) \Big\}.
\ee
Then, $(f(t), \varrho(t), \Psi(t))$ is bounded by
\Be \label{decay_f_spec}
\begin{split}
& \sup_{0 \leq t < \infty} e^{ \lambda t} \|  e^{\frac{\beta}{8} \sqrt{(m_{\pm} c)^2 + |p|^2} + \frac{m_{\pm} g}{16 c} \beta x_3} f_{\pm} (t,x,p) \|_{L^\infty (\O \times \R^3)}
\\& \leq 3 e^2 \big( e^{\frac{m_+ g}{24} \beta} \| \w_{+,\beta, 0 } f_{+, 0 } \|_{L^\infty_{x,p}}
+ e^{\frac{m_- g}{24} \beta} \| \w_{-,\beta, 0 } f_{-, 0 } \|_{L^\infty_{x,p}} \big).
\end{split}
\Ee
and
\begin{align}
& \sup_{0 \leq t < \infty} e^{ \lambda t} \| e^{ \nu x_3} \varrho (t, x)\|_{L^\infty(\O)}
\leq \sum\limits_{i = \pm} e_i e^{2 + \frac{m_i g}{24} \beta} \| \w_{i, \beta, 0 } f_{i, 0 } \|_{L^\infty_{x,p}} \frac{1024}{\beta^3}, 
\label{decay:varrho_spec} \\
& \sup_{0 \leq t < \infty} e^{ \lambda t} \|\nabla_x \Psi(t)\|_{L^\infty (\O)}
\lesssim \sum\limits_{i = \pm} e_i e^{2 + \frac{m_i g}{24} \beta} \| \w_{i, \beta, 0 } f_{i, 0 } \|_{L^\infty_{x,p}} \frac{1024}{\beta^3}. 
\label{decay:psi_spec}
\end{align}
Moreover, the flux $b (t, x)$ defined in \eqref{def:flux} is bounded by
\Be \label{Uest:D^-1_Db_spec}
\sup_{0 \leq t < \infty} \| (-\Delta_0)^{-1}  (\nabla_x \cdot b  (t,x)) \|_{L^\infty (\O)} 
\lesssim \frac{ c }{\max \{ e_+, e_- \} \beta}.
\Ee
\end{theorem}

\begin{proof}

We abuse the notation as in \eqref{abuse} in the proof.
From \eqref{form:f_spec}, together with \eqref{est:I_spec} and \eqref{est:N_spec} in Proposition \ref{prop:decay_spec}, we derive for any $(t,x,p) \in [0, \infty) \times  \bar\O \times \R^3$,
\be \label{est2:varrho_spec}
\begin{split}
& e^{ \lambda t} | f (t,x,p) |
\\& \leq e^{ \lambda t} \big( | \mathcal{I} (t,x,p) | + | \mathcal{N} (t,x,p) | + | \mathcal{S} (t,x,p) | \big)
\\& \leq e^{2 + \frac{m g}{24} \beta} \| \w_{\beta, 0 }   f_{0 } \|_{L^\infty_{x,p}} 
e^{ - \frac{\beta}{4} \sqrt{(m c)^2 + |p|^2} } e^{ - \frac{m g}{8 c} \beta x_3} 
\\& \ \ \ \ + 4 e^{ \lambda } (1 + \frac{1}{\nu}) e_{\pm} \| w_{\beta} \nabla_p h \|_{\infty} e^{- \frac{\beta}{8} \sqrt{(m c)^2 + |p|^2} } e^{- \frac{m g}{16 c} \beta x_3} \times \sup\limits_{s \in [0, t] } e^{ \lambda s} \| e^{ \nu x_3} \varrho (s) \|_{L^\infty(\O)} 
\\& \ \ \ \ + \mathbf{1}_{t > \tB (t,x,p)} 
\underbrace{
\varepsilon e^{ \lambda \tB} | e^{ \lambda (t - \tB (t,x,p)} f (t - \tB (t,x,p), \xB (t,x,p), \tpB (t,x,p) ) |
}_{\eqref{est2:varrho_spec}_*}.
\end{split}
\ee
From Corollary \ref{cor:max_X_dy} and the condition \eqref{condition2:f_support_spec}, then
\Be \notag
\begin{split}
\tBp (t,x,p) 
& \leq \Big( \frac{1}{\min \{ \frac{g}{4 \sqrt{2}}, \frac{c}{\sqrt{10}} \}} + 1 \Big) \times \tpx + 1
\ \text{ for any } \
\frac{4 c}{m_{\pm} g} |p| + 3 x_{3} < \tpx.
\end{split}
\Ee
Together with the condition \eqref{condition2:tilde_epilon_spec}, this implies that for any $\frac{4 c}{m_{\pm} g} |p| + 3 x_{3} < \tpx$,
\be \label{est6:varrho_spec}
\begin{split}
\eqref{est2:varrho_spec}_*
& \leq 
\varepsilon e^{ \lambda \tB} \sup_{0 \leq t < \infty} e^{ \lambda t} \| e^{\frac{\beta}{8} \sqrt{(m c)^2 + |p|^2} + \frac{m g}{16 c} \beta x_3} f (t,x,p) \|_{L^\infty (\O \times \R^3)} e^{- \frac{\beta}{8} \sqrt{(m c)^2 + |p|^2} } e^{- \frac{m g}{16 c} \beta x_3}
\\& \leq \frac{1}{2} \sup_{0 \leq t < \infty} e^{ \lambda t} \| e^{\frac{\beta}{8} \sqrt{(m c)^2 + |p|^2} + \frac{m g}{16 c} \beta x_3} f (t,x,p) \|_{L^\infty (\O \times \R^3)} e^{- \frac{\beta}{8} \sqrt{(m c)^2 + |p|^2} } e^{- \frac{m g}{16 c} \beta x_3}.
\end{split}
\ee
Inputting \eqref{est6:varrho_spec} into \eqref{est2:varrho_spec}, together with \eqref{condition2:f_support_spec}, we derive 
\be \label{est3:varrho_spec}
\begin{split}
& \sup_{0 \leq t < \infty} e^{ \lambda t} \|  e^{\frac{\beta}{8} \sqrt{(m c)^2 + |p|^2} + \frac{m g}{16 c} \beta x_3} f (t,x,p) \|_{L^\infty (\O \times \R^3)}
\\& \leq 2 e^{2 + \frac{m g}{24} \beta} \| \w_{\beta, 0 }   f_{0 } \|_{L^\infty_{x,p}}  
+ \underbrace{
8 e^{ \lambda } (1 + \frac{1}{\nu}) e_{\pm} \| w_{\beta} \nabla_p h \|_{\infty}
}_{\eqref{est3:varrho_spec}_*}
\times \sup\limits_{s \in [0, t] } e^{ \lambda s} \| e^{ \nu x_3} \varrho (s) \|_{L^\infty(\O)}.
\end{split}
\ee

On the other hand, from \eqref{def:varrho_spec} we obtain for any $(t, x) \in [0, \infty) \times \bar\O$,
\be \notag
| \varrho(t,x) |
\leq e_+ \int_{\R^3} | f_+ (t,x,p) | \dd p + e_- \int_{\R^3} | f_{-} (t,x,p) | \dd p.
\ee
This, together with \eqref{est2:varrho_spec}, shows that for any $(t, x) \in [0, \infty) \times \bar\O$,
\Be \label{est:varrho_spec}
\begin{split}
& e^{ \lambda t} e^{ \nu x_3} |\varrho (t,x)|
\\& \leq e^{ \lambda t} e^{ \nu x_3}
\sum\limits_{i = \pm} e_i \int_{\R^3} | f_i (t,x,p) | \dd p   
\\& \leq \sum\limits_{i = \pm} e_i 
 \int_{\R^3} \frac{e^{ \nu x_3}}{e^{\frac{\beta}{8} \sqrt{(m_i c)^2 + |p|^2} + \frac{m_i g}{16 c} \beta x_3}} \dd p 
\times \sup_{0 \leq t < \infty} e^{ \lambda t} \|  e^{\frac{\beta}{8} \sqrt{(m_i c)^2 + |p|^2} + \frac{m_i g}{16 c} \beta x_3} f_i (t,x,p) \|_{L^\infty (\O \times \R^3)}
\\& \leq \sum\limits_{i = \pm} e_i \frac{512}{\beta^3} 
\big( e^{2 + \frac{m_i g}{24} \beta} \| \w_{i, \beta, 0 } f_{i, 0 } \|_{L^\infty_{x,p}} + \eqref{est3:varrho_spec}_* \times \sup\limits_{s \in [0, t] } e^{ \lambda s} \| e^{ \nu x_3} \varrho (s) \|_{L^\infty(\O)} \big)
\\& \leq \sum\limits_{i = \pm} e_i \frac{512}{\beta^3} e^{2 + \frac{m_i g}{24} \beta} \| \w_{i, \beta, 0 } f_{i, 0 } \|_{L^\infty_{x,p}} + \sum\limits_{i = \pm} e_i \frac{512}{\beta^3} \eqref{est3:varrho_spec}_* \times \sup\limits_{s \in [0, t] } e^{ \lambda s} \| e^{ \nu x_3} \varrho (s) \|_{L^\infty(\O)}.
\end{split}	
\Ee
From \eqref{est:varrho_spec} and the condition \eqref{condition:Dvh_spec}, we have
\be \label{est4:varrho_spec}
\sup\limits_{s \in [0, t] } e^{ \lambda s} \| e^{ \nu x_3} \varrho (s) \|_{L^\infty(\O)}
\leq \frac{1024}{\beta^3} \sum\limits_{i = \pm} e_i e^{2 + \frac{m_i g}{24} \beta} \| \w_{i, \beta, 0 } f_{i, 0 } \|_{L^\infty_{x,p}}.
\ee
Inputting \eqref{est4:varrho_spec} into \eqref{est3:varrho_spec}, we derive 
\be \label{est5:varrho_spec}
\begin{split}
& \sup_{0 \leq t < \infty} e^{ \lambda t} \|  e^{\frac{\beta}{8} \sqrt{(m c)^2 + |p|^2} + \frac{m g}{16 c} \beta x_3} f (t,x,p) \|_{L^\infty (\O \times \R^3)}
\\& \leq 2 e^{2 + \frac{m g}{24} \beta} \| \w_{\beta, 0 }   f_{0 } \|_{L^\infty_{x,p}}  
+ \sum\limits_{i = \pm} e^{2 + \frac{m_i g}{24} \beta} \| \w_{i, \beta, 0 } f_{i, 0 } \|_{L^\infty_{x,p}}.
\end{split}
\ee
Hence, we conclude \eqref{decay_f_spec} and \eqref{decay:varrho_spec}.
Applying \eqref{decay:varrho_spec} in \eqref{est:nabla_phi} of Lemma \ref{lem:rho_to_phi}, we obtain \eqref{decay:psi_spec}.
Finally, using \eqref{decay_f_spec}, together with \eqref{est:D^-1_Db} in Lemma \ref{lem:Ddb}, we conclude \eqref{Uest:D^-1_Db_spec}.
\end{proof}

\subsubsection{A Priori Estimate} 
\label{sec:RD_spec} 

In this section, we establish a priori estimate of $(F, \phi_F)$ solving \eqref{VP_F}, \eqref{Poisson_F}, \eqref{VP_0}, \eqref{Dbc:F} and \eqref{bdry:F_spec}.
Recall from \eqref{def:F}, $F(t,x,p)= h(x,p)+f(t,x,p)$ where $(h, \rho, \Phi)$ solves \eqref{VP_h_spec}-\eqref{eqtn:Dphi_spec} and $(f, \varrho, \Psi)$ solves \eqref{eqtn:f_spec}-\eqref{Poisson_f_spec}, respectively.

Throughout this section, we assume a compatibility condition:
\Be \label{CC_F0=G_spec}
F_{\pm, 0} (x, p) = G_{\pm} (x, p) + \varepsilon F_{\pm, 0} (x, \tilde{p})
\ \ \text{on} \ (x, p) \in \gamma_-,
\Ee
where $\tilde{p} = (p_1, p_2, - p_3)$.
Further, we always suppose \eqref{Bootstrap_spec} holds, which is analogous to the condition \eqref{Bootstrap} in the case of the dynamic problem only under the effect of the inflow boundary conditions \eqref{eqtn:f}-\eqref{Poisson_f}.

Suppose the assumption \eqref{Bootstrap_spec} holds. 
Since the characteristics in \eqref{ODE_F_spec} are the same as the characteristics in \eqref{ODE_F}, Lemma \ref{VL:dyn} and \eqref{est:X_x:dyn}-\eqref{est:vb_v:dyn} hold for the the characteristics $\Z_{\pm} (s;t,x,p) = ( \X_{\pm} (s;t,x,p), \P_{\pm} (s;t,x,p) )$ solving \eqref{ODE_F_spec}, and we omit the proof.
For the reader's convenience, we refer to Section \ref{sec:RD} and collect some basic properties and estimates on $F_{\pm} (t,x,p)$ under the mixture of inflow boundary and specular boundary conditions \eqref{bdry:F_spec} as follows.

\smallskip

(a)
From \eqref{VP_0} and \eqref{bdry:F_spec}, we rewrite $F_{\pm} (t,x,p)$ as
\Be \label{form:F_spec}
\begin{split} 
F_{\pm} (t,x,p) 
& = \mathbf{1}_{t \leq \tBp (t,x,p) } F_{\pm, 0} (\X_{\pm} (0;t,x,p), \P_{\pm} (0;t,x,p))
\\& \ \ \ \ + \mathbf{1}_{t > \tBp (t,x,p) }  G_{\pm} ( \xBp (t,x,p), \pBp (t,x,p))
\\& \ \ \ \ + \mathbf{1}_{t > \tBp (t,x,p) } \varepsilon F_{\pm} (t - \tBp (t,x,p), \xBp (t,x,p), \tpBp (t,x,p)),
\end{split}
\Ee
where $\tpBp = (p_{\mathbf{B}, \pm, 1}, p_{\mathbf{B}, \pm, 2}, - p_{\mathbf{B}, \pm, 3})$.
Therefore, the partial derivatives of $F_\pm (t,x,p)$ follow
\Be \label{DF_xv_spec}
\begin{split}
& \p_{x_i, p_i} F_{\pm} (t,x,p) 
\\& = \mathbf{1}_{t \leq \tBp (t,x,p) }  
\{ \p_{x_i, p_i} \X_{\pm} (0;t,x,p) \cdot \nabla_x F_{\pm, 0} (\Z_{\pm} (0;t,x,p) ) 
\\& \qquad \qquad \qquad \qquad + \p_{x_i, p_i} \P_{\pm} (0;t,x,p) \cdot \nabla_p F_{\pm, 0} (\Z_{\pm} (0;t,x,p) ) \}
\\& \ \ \ \ + \mathbf{1}_{t >\tBp (t,x,p) }  
\{ \p_{x_i, p_i} \xBp \cdot \nabla_{x_\parallel} G_{\pm} (\xBp (t,x,p), \pBp (t,x,p))
\\& \qquad \qquad \qquad \qquad \quad + \p_{x_i, p_i} \pBp \cdot \nabla_p G_{\pm} (\xBp (t,x,p), \pBp (t,x,p)) \}
\\& \ \ \ \ + \mathbf{1}_{t >\tBp (t,x,p) } \varepsilon 
\{ - \p_{x_i, p_i} \tBp \times \p_{t} F_{\pm} (t - \tBp (t,x,p), \xBp (t,x,p), \tpBp (t,x,p))
\\& \qquad \qquad \qquad \qquad \qquad + \p_{x_i, p_i} \xBp \cdot \nabla_{x_\parallel} F_{\pm} (t - \tBp (t,x,p), \xBp (t,x,p), \tpBp (t,x,p))
\\& \qquad \qquad \qquad \qquad \qquad + \p_{x_i, p_i} \tpBp \cdot \nabla_p F_{\pm} (t - \tBp (t,x,p), \xBp (t,x,p), \tpBp (t,x,p)) \},
\end{split}
\Ee
where the characteristics $\Z_{\pm} (s;t,x,p) = (\X_{\pm} (s;t,x,p), \P_{\pm} (s;t,x,p))$ solves \eqref{ODE_F_spec}.

\smallskip

(b)
We denote $\| \nabla_x^2\phi_{F} \|_\infty := \sup\limits_{s \in [t-\tB (t,x,p),t]} \| \nabla_x^2\phi_{F} (s, x)  \|_{L^{\infty}(\bar{\O})}$. 
Analogous to \eqref{est:X_x:dyn}-\eqref{est:V_v:dyn}, we conclude that
\begin{align}
|\nabla_{x} \X_{\pm} (s;t,x,p)| & \leq \min \big\{  e^{\frac{|t-s|^2 + 2|t-s| }{2} (\| \nabla_x^2 \phi_F  \|_\infty + e_{\pm} B_3 + m_{\pm} g ) }, e^{ (1 + B_3 + \| \nabla_x ^2 \phi_F  \|_\infty) |t-s|} \big\}, 
\label{est:X_x:dyn_spec} \\
|\nabla_{x} \P_{\pm} (s;t,x,p)| & \leq \min \big\{ |t-s| ( B_3 + \| \nabla_x ^2 \phi_F \|_\infty) e^{ (1+ B_3 + \| \nabla_x ^2 \phi_F \|_\infty) |t-s| },
\notag \\
& \qquad \qquad \qquad \qquad \qquad e^{ (1+ B_3 + \| \nabla_x ^2 \phi_F \|_\infty) |t-s|}
\big\}, 
\label{est:V_x:dyn_spec} \\
|\nabla_{p} \X_{\pm} (s;t,x,p)| &\leq   \min \big\{ |t-s| e^{ (1 + B_3 + \| \nabla_x ^2 \phi_F  \|_\infty) |t-s|},
e^{ (1 + B_3 + \| \nabla_x ^2 \phi_F  \|_\infty) |t-s|}
\big\}, 
\label{est:X_v:dyn_spec} \\
|\nabla_{p} \P_{\pm} (s;t,x,p)| & \leq  \min \big\{ 1 + |t-s| (B_3 + \| \nabla_x ^2 \phi_F \|_{\infty}) e^{(1 + B_3 + \|\nabla_x^2 \phi_F \|_\infty) |t-s|},
\notag \\
& \qquad \qquad \qquad \qquad \qquad e^{ (1 + B_3 + \| \nabla_x ^2 \phi_F  \|_\infty) |t-s|}
\big\}.
\label{est:V_v:dyn_spec}
\end{align}

\smallskip

(c)
From the proof of Lemma \ref{lem:exp_txvb}, we get
\Be \notag
\begin{split}
& - \p_{x_i} \tBp (t,x,p) \vBpn (x,p) + \p_{x_i} X_{\pm, 3} (t - \tBp (t,x,p); t,x,p ) = 0,
\\& - \p_{p_i} \tBp (t,x,p) \vBpn (x,p) + \p_{p_i} X_{\pm, 3} (t - \tBp (t,x,p); t,x,p ) = 0.
\end{split}
\ee
Together with Lemma \ref{lem:nabla_zb}, this shows that
\Be \label{est:tb_x:dyn_spec}
\begin{split}
& | \p_{x_i} \tBp (t,x,p) | 
\\& \leq \frac{1}{ | \vBpn (t,x,p) | }  
\Big( \delta_{i3} + \frac{|\tBp (t,x,p)|^2}{2} ( \| \nabla_x^2 \phi_F \|_\infty + e_{\pm} B_3 + m_{\pm} g ) \Big) e^{ (1 + B_3 + \| \nabla_x ^2 \phi_F  \|_\infty) |\tBp| },
\end{split}
\Ee
\Be \label{est:tb_v:dyn_spec}
\begin{split}
& |\p_{p_i} \tBp (t,x,p)| 
\leq \frac{|\tBp (t,x,p)|  }{|\vBpn (t,x,p)| \sqrt{ (m_{\pm})^2 + |p|^2/ c^2}}  
\\& \ \ \ \ + \frac{1}{|\vBpn (t,x,p)|}
\frac{|\tBp (t,x,p)|^2}{2} ( \| \nabla_x^2 \phi_F  \|_\infty + e_{\pm} B_3 + m_{\pm} g ) e^{ (1 + B_3 + \| \nabla_x ^2 \phi_F  \|_\infty) |\tBp| }.
\end{split}
\Ee

\smallskip

(d)
Furthermore, analogous to \eqref{est:xb_x:dyn}-\eqref{est:vb_v:dyn}, we conclude that
\Be \label{est:xb_x:dyn_spec}
\begin{split}
& | \p_{x_i} \xBp (t,x,p) | 
\leq \frac{ | \vBp (t,x,p) | }{ | \vBpn (t,x,p) | }  
 \delta_{i3} 
\\& \ \ \ \ + \Big(1 +  \frac{|\vBp (t,x,p)|}{|\vBpn (t,x,p)|}
\frac{|\tBp (t,x,p)|^2}{2} ( \| \nabla_x^2 \phi_F \|_\infty + e_{\pm} B_3 + m_{\pm} g ) \Big) e^{ (1 + B_3 + \| \nabla_x ^2 \phi_F  \|_\infty) |\tBp| },
\end{split}
\Ee
\Be \label{est:xb_v:dyn_spec}
\begin{split}
& |\p_{p_i} \xBp (t,x,p)| 
\leq \frac{|\vBp (t,x,p)| |\tBp (t,x,p)|  }{|\vBpn (t,x,p)| \sqrt{ (m_{\pm})^2 + |p|^2/ c^2}}  
\\& \ \ \ \ + \Big(1 +   \frac{|\vBp (t,x,p)|}{|\vBpn (t,x,p)|}
\frac{|\tBp (t,x,p)|^2}{2} ( \| \nabla_x^2 \phi_F  \|_\infty + e_{\pm} B_3 + m_{\pm} g ) \Big) e^{ (1 + B_3 + \| \nabla_x ^2 \phi_F  \|_\infty) |\tBp| },
\end{split}
\Ee
\Be \label{est:vb_x:dyn_spec}
\begin{split}
& |\p_{x_i} \pBp (t,x,p)| \leq \frac{ |e_{\pm} B_3 + m_{\pm} g| }{|\vBpn (t,x,p)|} \delta_{i3}
\\& \ \ \ \ + \Big(1 +  \frac{|e_{\pm} B_3 + m_{\pm} g|}{|\vBpn (t,x,p)|}
\frac{|\tBp (t,x,p)|^2}{2} ( \| \nabla_x^2 \phi_F  \|_\infty + e_{\pm} B_3 + m_{\pm} g ) \Big) e^{ (1 + B_3 + \| \nabla_x ^2 \phi_F  \|_\infty) | \tBp|},
\end{split}
\Ee
\Be \label{est:vb_v:dyn_spec}
\begin{split}
& |\p_{p_i} \pBp (t,x,p)| \leq \frac{|e_{\pm} B_3 + m_{\pm} g| |\tBp (t,x,p)|  }{|\vBpn (t,x,p)| \sqrt{ (m_{\pm})^2 + |p|^2/ c^2}}  
\\& \ \ \ \ + \Big(1 + \frac{|e_{\pm} B_3 + m_{\pm} g|}{|\vBpn (t,x,p)|}
\frac{|\tBp (t,x,p)|^2}{2} ( \| \nabla_x^2 \phi_F  \|_\infty + e_{\pm} B_3 + m_{\pm}g ) \Big) e^{ (1 + B_3 + \| \nabla_x ^2 \phi_F  \|_\infty) |\tBp|}.
\end{split}
\Ee

\begin{prop} 
\label{RE:dyn_spec}

Suppose $(F, \phi_F)$ solves \eqref{VP_F}, \eqref{Poisson_F}, \eqref{VP_0}, \eqref{Dbc:F} and \eqref{bdry:F_spec} under the compatibility condition \eqref{CC_F0=G_spec}.
Assume that \eqref{Bootstrap_spec} and 
\be \label{est:beta_b_dy_spec}
\beta \frac{e_{\pm}}{c} \| (-\Delta_0)^{-1} (\nabla_x \cdot b ) \|_{L^\infty_{t,x}} 
\leq 1 \leq \frac{\beta}{2}.
\ee
Further, we assume 
\be \label{est1:D3tphi_F_spec}
(1 + B_3 + \| \nabla_x ^2 \phi_F  \|_\infty) + \| \p_t \p_{x_3} \phi_F (t, x_\parallel , 0) \|_{L^\infty(\p\O)} 
\leq \frac{\hat{m} g}{24} \tilde{\beta},
\ee
where $\hat{m} = \min \{ m_+, m_- \}$.
Suppose there exists $\tpx > 1$ such that 
\be \label{condition3:F_support_spec} 
f_{\pm} (t, x, p) = 0
\ \text{ for any } \
\frac{4 c}{m_{\pm} g} |p| + 3 x_{3} \geq \tpx.
\ee
Set $\lambda = \frac{g}{48} \hat{m} \beta$, suppose $\varepsilon$ in the boundary condition \eqref{bdry:F_spec} satisfies that	
\be \label{condition3:tilde_epilon_spec}
2 \varepsilon 
< \exp \Big\{ - \lambda \Big( \big( \frac{1}{\min \{ \frac{g}{4 \sqrt{2}}, \frac{c}{\sqrt{10}} \}} + 1 \big) \times \tpx + 1 \Big) \Big\}.
\ee
Finally, suppose $\w_{\pm, \beta, 0} (x, p)$ defined in \eqref{W_t=0} satisfies
\be \notag 
\|\w_{\pm, \beta, 0} f_{\pm, 0} \|_{L^\infty (\O \times \R^3)} <\infty.
\ee
Consider $(t,x,p) \in [0, \infty) \times  \bar\O \times \R^3$, then
\Be \label{est:F_v:dyn_spec}
\begin{split} 
& \| e^{ \frac{\tilde{\beta}}{4} \sqrt{(m_{\pm} c)^2 + |p|^2} } e^{ \frac{m_{\pm} g}{8 c} \tilde{\beta} x_3} \nabla_p F_{\pm} (t,x,p) \|_{L^\infty(\O \times \R^3)} 
\\& \lesssim 24 e^{ \frac{m_{\pm} g}{24} \tilde{\beta} }
e^{2 \tilde{\beta} \frac{e_{\pm}}{c} \| (-\Delta_0)^{-1} (\nabla_x \cdot b ) \|_{L^\infty_{t,x}} } \| \w_{\pm, \tilde \beta, 0} \nabla_{x,p} F_{\pm, 0} \|_{L^\infty (\O \times \R^3)}
\\& \ \ \ \ + 12 \big( 1 + \frac{ \| \nabla_x^2 \phi_F \|_\infty + e_{\pm} B_3 + m_{\pm} g}{(m_{\pm} g)^2} \big) \| e^{\tilde{\beta} |p^0_{\pm}|} \nabla_{x_\parallel,p} G_{\pm} \|_{L^\infty (\gamma_-)},
\end{split}
\Ee 
and 
\Be \label{est:F_x:dyn_spec}
\begin{split} 
& e^{ \frac{\tilde{\beta}}{4} \sqrt{(m_{\pm} c)^2 + |p|^2} } e^{ \frac{m_{\pm} g}{8 c} \tilde{\beta} x_3} \big| \nabla_{x_i} F_{\pm} (t,x,p) \big|
\\& \lesssim \big( \frac{\delta_{i3}}{\alpha_{\pm, F} (t,x,p)} + 1 + \frac{ \| \nabla_x^2 \phi_F \|_\infty + e_{\pm} B_3 + m_{\pm} g}{(m_{\pm} g)^2} \big) 
\\& \qquad \times \Big\{ 24 e^{ \frac{m_{\pm} g}{24} \tilde{\beta} } e^{2 \tilde{\beta} \frac{e_{\pm}}{c} \| (-\Delta_0)^{-1} (\nabla_x \cdot b ) \|_{L^\infty_{t,x}} } \| \w_{\pm, \tilde \beta, 0}   \nabla_{x,p} F_{\pm, 0} \|_{L^\infty (\O \times \R^3)}
\\& \qquad \qquad + 12 \big( 1 + \frac{ \| \nabla_x^2 \phi_F \|_\infty + e_{\pm} B_3 + m_{\pm} g}{(m_{\pm} g)^2} \big) \| e^{\tilde{\beta} |p^0_{\pm}|} \nabla_{x_\parallel,p} G_{\pm} \|_{L^\infty (\gamma_-)} \Big\},
\end{split}
\Ee 
where 
$\| \nabla_x^2\phi_{F} \|_\infty = \sup\limits_{s \in [t-\tB (t,x,p),t]} \| \nabla_x^2\phi_{F} (s, x)  \|_{L^{\infty}(\bar{\O})}$, and $\alpha_{\pm, F} (t,x,p)$ is defined in \eqref{alpha_F}.
\end{prop}

\begin{proof}

We follow the proof of \eqref{est:F_v:dyn} and \eqref{est:F_x:dyn} in Proposition \ref{RE:dyn}.
Here, for simplicity, we 
use the notation: $p_\pm^0 = \sqrt{(m_\pm c)^2 + |p|^2}$ in the rest of proof. 
We will also use \eqref{abuse}.

\smallskip

\textbf{Step 1. Proof of \eqref{est:F_v:dyn_spec}.}
From \eqref{DF_xv_spec} and \eqref{VP_F}, we have
\Be \notag
\begin{split}
& \p_{p_i} F_{\pm} (t,x,p) 
\\& = \mathbf{1}_{t \leq \tBp (t,x,p) }  
\{ \p_{p_i} \X_{\pm} (0;t,x,p) \cdot \nabla_x F_{\pm, 0} (\Z (0;t,x,p) ) 
\\& \qquad \qquad \qquad \qquad \qquad + \p_{p_i} \P_{\pm} (0;t,x,p) \cdot \nabla_p F_{\pm, 0} (\Z_{\pm} (0;t,x,p) ) \}
\\& \ \ \ \ + \mathbf{1}_{t > \tBp (t,x,p) }  
\{ \p_{p_i} \xBp \cdot \nabla_{x_\parallel} G_{\pm} (\xBp (t,x,p), \pBp (t,x,p))
\\& \qquad \qquad \qquad \qquad \qquad + \p_{p_i} \pBp \cdot \nabla_p G_{\pm} (\xBp (t,x,p), \pBp (t,x,p)) \}
\\& \ \ \ \ + \mathbf{1}_{t >\tBp (t,x,p) } \varepsilon 
\big( - \p_{p_i} \tBp \big) \times \Big\{ c \frac{\tpBp}{\tpBp^0} \cdot \nabla_x F_{\pm} (t - \tBp, \xBp, \tpBp )
\\& \qquad \qquad \qquad + \Big( \frac{\tpBp}{\tpBp^0} \times e_{\pm} B - \nabla_x ( e_{\pm} \phi_F) - m_{\pm} g \mathbf{e}_3 \Big) \cdot \nabla_p F_{\pm} (t - \tBp, \xBp, \tpBp ) \Big\}
\\& \ \ \ \ + \mathbf{1}_{t >\tBp (t,x,p) } \varepsilon 
\{ \p_{p_i} \xBp \cdot \nabla_{x_\parallel} F_{\pm} (t - \tBp, \xBp, \tpBp )
\\& \qquad \qquad \qquad \qquad \qquad + \p_{p_i} \tpBp \cdot \nabla_p F_{\pm} (t - \tBp, \xBp, \tpBp ) \},
\end{split}
\Ee
where $\tpBp = (p_{\mathbf{B}, \pm, 1}, p_{\mathbf{B}, \pm, 2}, - p_{\mathbf{B}, \pm, 3})$.
Together with $|\tpB| = |\pB|$ and the assumption \eqref{Bootstrap_spec}, this shows that
\begin{align}
& | \nabla_p F (t,x,p) | 
\notag \\
& \leq \mathbf{1}_{t \leq \tB (t,x,p)}
\frac{| \nabla_p \X  (0;t,x,p)| +| \nabla_p \P  (0;t,x,p)|  }{\w_{\tilde \beta} (0, \Z (0;t,x,p) )} \| \w_{\tilde \beta, 0}   \nabla_{x,p} F_0  \|_{L^\infty (\O \times \R^3)} \label{est1:F_v1_spec} \\
& \ \ \ \ +	\mathbf{1}_{t > \tB (t,x,p)} \frac{ | \nabla_p \xB (t,x,p) | + |\nabla_p \pB (t,x,p)|}{ e^{\tilde{\beta} |\pB^0 (t,x,p)| } } \| e^{\tilde{\beta} |p^0|} \nabla_{x_\parallel,p} G \|_{L^\infty (\gamma_-)}
\label{est1:F_v2_spec} \\
& \ \ \ \ + \mathbf{1}_{t >\tB (t,x,p) } \varepsilon | \nabla_{p} \tB (t,x,p) | \times \frac{c}{|\tpB^0 (t,x,p)|} \big| \tpBn \p_{x_3} F (t - \tB, \xB, \tpB ) \big|
\label{est1:F_v7_spec} \\ 
& \ \ \ \ + \mathbf{1}_{t >\tB (t,x,p) } \varepsilon 
\frac{ | \nabla_p \tB (t,x,p) | }{ e^{ \frac{\tilde{\beta}}{4} |\pB^0 (t,x,p)| } } \big\{ c + e B_3 + \frac{3}{2} m g \big\}
\label{est1:F_v3_spec} \\ 
& \qquad \qquad \qquad \qquad \qquad \times \sup\limits_{0 \leq t < \infty} \| e^{ \frac{\tilde{\beta}}{4} \sqrt{(m c)^2 + |p|^2} } e^{ \frac{m g}{8 c} \tilde{\beta} x_3} \nabla_{x_\parallel,p} F (t,x,p) \|_{L^\infty (\bar{\O} \times \R^3)}
\label{est1:F_v4_spec}
\\& \ \ \ \ + \mathbf{1}_{t >\tB (t,x,p) } \varepsilon 
\frac{ | \nabla_p \xB (t,x,p) | + |\nabla_p \pB (t,x,p)|}{ e^{ \frac{\tilde{\beta}}{4} |\pB^0 (t,x,p)| } }
\label{est1:F_v5_spec} \\ 
& \qquad \qquad \qquad \qquad \qquad \times \sup\limits_{0 \leq t < \infty} \| e^{ \frac{\tilde{\beta}}{4} \sqrt{(m c)^2 + |p|^2} } e^{ \frac{m g}{8 c} \tilde{\beta} x_3} \nabla_{x_\parallel,p} F (t,x,p) \|_{L^\infty (\bar{\O} \times \R^3)}.
\label{est1:F_v6_spec}
\end{align}
Following from \eqref{est2:F_v1}, \eqref{est1:xb_v:dyn}, and \eqref{est1:vb_v:dyn}, we derive that 
\be \label{est2:F_v12_spec}
\begin{split}
& \eqref{est1:F_v1_spec} + \eqref{est1:F_v2_spec}
\\& \lesssim 2 e^{ \frac{m g}{24} \tilde{\beta} } e^{2 \tilde{\beta} \frac{e}{c} \| (-\Delta_0)^{-1} (\nabla_x \cdot b ) \|_{L^\infty_{t,x}} } e^{ - \frac{\tilde{\beta}}{4} \sqrt{(m c)^2 + |p|^2} } e^{ - \frac{m g}{8 c} \tilde{\beta} x_3} \| \w_{\tilde \beta, 0}   \nabla_{x,p} F_0  \|_{L^\infty (\O \times \R^3)}
\\& \ \ \ \ + \big( 1 + \frac{ \| \nabla_x^2 \phi_F \|_\infty + e B_3 + mg}{(m g)^2} \big) e^{ - \frac{\tilde \beta}{2} \big( |p^0| + \frac{m g}{2 c} x_3 \big) } \| e^{\tilde{\beta} |p^0|} \nabla_{x_\parallel,p} G \|_{L^\infty (\gamma_-)}.
\end{split}
\ee
From \eqref{est:tb_v:dyn_spec}, together with \eqref{est1:xb_v_a:dyn} and Proposition \ref{lem2:tB}, we get
\Be \label{est1:tb_v:dyn_spec}
\begin{split}
& |\p_{p_i} \tB (t,x,p)|
\\& \leq \frac{|\pB^0 (t,x,p)| }{ \sqrt{ (m c)^2 + |p|^2 }} \Big( \frac{2}{m g}  + \frac{ 16 (\sqrt{c m} + | \pBn (t,x,p) |)}{c m^2 g} 
\\& \qquad \qquad \times \big(1 + \sqrt{(m c)^2 + |\pB (t,x,p)|^2} + 2 \pBn (t,x, p) \big) \Big) 
\\& \ \ \ \ + \Big( \frac{|\pB^0 (t,x,p)|}{2 c} |\tB (t,x,p)| \Big( \frac{2}{m g}  + \frac{ 16 (\sqrt{c m} + | \pBn (t,x,p) |)}{c m^2 g} 
\\& \qquad \qquad \times \big(1 + \sqrt{(m c)^2 + |\pB (t,x,p)|^2} + 2 \pBn (t,x, p) \big) \Big)
( \| \nabla_x^2 \phi_F  \|_\infty + e B_3 + mg ) \Big) 
\\& \qquad \qquad \qquad \qquad \times e^{ (1 + B_3 + \| \nabla_x ^2 \phi_F  \|_\infty) \tB (t,x,p) }.
\end{split}
\Ee
Using the assumptions \eqref{condition3:F_support_spec} and \eqref{condition3:tilde_epilon_spec}, together with \eqref{est1:tb_v:dyn_spec}, we obtain that
\be \label{est2:F_v7_spec}
\varepsilon | \nabla_{p} \tB (t,x,p) | \times \frac{c}{|\tpB^0 (t,x,p)|}
\leq \frac{1}{8} e^{ - \frac{\tilde{\beta}}{4} \sqrt{(m c)^2 + |p|^2} } e^{ - \frac{m g}{8 c} \tilde{\beta} x_3}.
\ee
From \eqref{est1:xb_v:dyn} and \eqref{est1:vb_v:dyn}, we further derive
\be \label{est2:F_v35_spec}
\begin{split}
& \eqref{est1:F_v3_spec} + \eqref{est1:F_v5_spec}
\\& = \varepsilon \Big\{
\frac{ | \nabla_p \tB (t,x,p) | }{ e^{ \frac{\tilde{\beta}}{4} |\pB^0 (t,x,p)| } } \big\{ c + e B_3 + \frac{3}{2} m g \big\}
+ \frac{ | \nabla_p \xB (t,x,p) | + |\nabla_p \pB (t,x,p)|}{ e^{ \frac{\tilde{\beta}}{4} |\pB^0 (t,x,p)| } }
\Big\}
\\& \leq \frac{1}{8} e^{ - \frac{\tilde{\beta}}{4} \sqrt{(m c)^2 + |p|^2} } e^{ - \frac{m g}{8 c} \tilde{\beta} x_3}.
\end{split}
\ee
Inputting \eqref{est2:F_v12_spec}, \eqref{est2:F_v7_spec} and \eqref{est2:F_v35_spec} into \eqref{est1:F_v1_spec}-\eqref{est1:F_v6_spec}, we conclude that
\be \label{est2:F_v_spec}
\begin{split}
& | \nabla_p F (t,x,p) | 
\\& \lesssim 2 e^{ \frac{m g}{24} \tilde{\beta} } e^{2 \tilde{\beta} \frac{e}{c} \| (-\Delta_0)^{-1} (\nabla_x \cdot b ) \|_{L^\infty_{t,x}} } e^{ - \frac{\tilde{\beta}}{4} \sqrt{(m c)^2 + |p|^2} } e^{ - \frac{m g}{8 c} \tilde{\beta} x_3} \| \w_{\tilde \beta, 0}   \nabla_{x,p} F_0  \|_{L^\infty (\O \times \R^3)}
\\& \ \ \ \ + \big( 1 + \frac{ \| \nabla_x^2 \phi_F \|_\infty + e B_3 + mg}{(m g)^2} \big) e^{ - \frac{\tilde \beta}{2} \big( |p^0| + \frac{m g}{2 c} x_3 \big) } \| e^{\tilde{\beta} |p^0|} \nabla_{x_\parallel,p} G \|_{L^\infty (\gamma_-)}
\\& \ \ \ \ + \frac{1}{8} e^{ - \frac{\tilde{\beta}}{4} \sqrt{(m c)^2 + |p|^2} } e^{ - \frac{m g}{8 c} \tilde{\beta} x_3} 
\times \sup\limits_{0 \leq t < \infty} \| p_3 \p_{x_3} F (t,x,p) \|_{L^\infty (\bar{\O} \times \R^3)}
\\& \ \ \ \ + \frac{1}{8} e^{ - \frac{\tilde{\beta}}{4} \sqrt{(m c)^2 + |p|^2} } e^{ - \frac{m g}{8 c} \tilde{\beta} x_3}
\times \sup\limits_{0 \leq t < \infty} \| e^{ \frac{\tilde{\beta}}{4} \sqrt{(m c)^2 + |p|^2} } e^{ \frac{m g}{8 c} \tilde{\beta} x_3} \nabla_{x_\parallel,p} F (t,x,p) \|_{L^\infty (\bar{\O} \times \R^3)}
\\& \ \ \ \ + \frac{1}{8} e^{ - \frac{\tilde{\beta}}{4} \sqrt{(m c)^2 + |p|^2} } e^{ - \frac{m g}{8 c} \tilde{\beta} x_3} \times \sup\limits_{0 \leq t < \infty} \| e^{ \frac{\tilde{\beta}}{4} \sqrt{(m c)^2 + |p|^2} } e^{ \frac{m g}{8 c} \tilde{\beta} x_3} \nabla_{x_\parallel,p} F (t,x,p) \|_{L^\infty (\bar{\O} \times \R^3)}.
\end{split} 
\ee

Analogous to \eqref{est1:F_v1_spec}-\eqref{est1:F_v6_spec}, we get
\begin{align}
& | \nabla_{x} F (t,x,p) | 
\notag \\
& \leq \mathbf{1}_{t \leq \tB (t,x,p)}
\frac{| \nabla_{x} \X  (0;t,x,p)| +| \nabla_{x} \P  (0;t,x,p)|  }{\w_{\tilde \beta} (0, \Z (0;t,x,p) )} \| \w_{\tilde \beta, 0}   \nabla_{x,p} F_0  \|_{L^\infty (\O \times \R^3)} \label{est1:F_x1_spec} \\
& \ \ \ \ +	\mathbf{1}_{t > \tB (t,x,p)} \frac{ | \nabla_{x} \xB (t,x,p) | + |\nabla_{x} \pB (t,x,p)|}{ e^{\tilde{\beta} |\pB^0 (t,x,p)| } } \| e^{\tilde{\beta} |p^0|} \nabla_{x_\parallel,p} G \|_{L^\infty (\gamma_-)}
\label{est1:F_x2_spec} \\
& \ \ \ \ + \mathbf{1}_{t >\tB (t,x,p) } \varepsilon | \nabla_{x} \tB (t,x,p) | \times \frac{c}{|\tpB^0 (t,x,p)|} \big| \tpBn \p_{x_3} F (t - \tB, \xB, \tpB ) \big|
\label{est1:F_x7_spec} \\  
& \ \ \ \ + \mathbf{1}_{t >\tB (t,x,p) } \varepsilon 
\frac{ | \nabla_{x} \tB (t,x,p) | }{ e^{ \frac{\tilde{\beta}}{4} |\pB^0 (t,x,p)| } } \big\{ c + e B_3 + \frac{3}{2} m g \big\}
\label{est1:F_x3_spec} \\ 
& \qquad \qquad \qquad \qquad \qquad \times \sup\limits_{0 \leq t < \infty} \| e^{ \frac{\tilde{\beta}}{4} \sqrt{(m c)^2 + |p|^2} } e^{ \frac{m g}{8 c} \tilde{\beta} x_3} \nabla_{x_\parallel,p} F (t,x,p) \|_{L^\infty (\bar{\O} \times \R^3)}
\label{est1:F_x4_spec}
\\& \ \ \ \ + \mathbf{1}_{t >\tB (t,x,p) } \varepsilon 
\frac{ | \nabla_{x} \xB (t,x,p) | + |\nabla_{x} \pB (t,x,p)| }{ e^{ \frac{\tilde{\beta}}{4} |\pB^0 (t,x,p)| } }
\label{est1:F_x5_spec} \\ 
& \qquad \qquad \qquad \qquad \qquad \times \sup\limits_{0 \leq t < \infty} \| e^{ \frac{\tilde{\beta}}{4} \sqrt{(m c)^2 + |p|^2} } e^{ \frac{m g}{8 c} \tilde{\beta} x_3} \nabla_{x_\parallel,p} F (t,x,p) \|_{L^\infty (\bar{\O} \times \R^3)}.
\label{est1:F_x6_spec}
\end{align} 
Following from \eqref{est1:F_x1}, \eqref{est1:xb_x:dyn}, and \eqref{est1:vb_x:dyn}, we derive that 
\be \label{est2:F_x12_spec}
\begin{split}
& \eqref{est1:F_x1_spec} + \eqref{est1:F_x2_spec}
\\& \lesssim 2 e^{ \frac{m g}{24} \tilde{\beta} }
e^{2 \tilde{\beta} \frac{e}{c} \| (-\Delta_0)^{-1} (\nabla_x \cdot b ) \|_{L^\infty_{t,x}} } e^{ - \frac{\tilde{\beta}}{4} \sqrt{(m c)^2 + |p|^2} } e^{ - \frac{m g}{8 c} \tilde{\beta} x_3} \| \w_{\tilde \beta, 0}   \nabla_{x,p} F_0 \|_{L^\infty (\O \times \R^3)}
\\& \ \ \ \ + \big( \frac{\delta_{i3}}{\alpha_{F} (t,x,p)} + 1 + \frac{ \| \nabla_x^2 \phi_F \|_\infty + e B_3 + m g}{(m g)^2} \big) \| e^{\tilde{\beta} |p^0|} \nabla_{x_\parallel,p} G \|_{L^\infty (\gamma_-)}.
\end{split}
\ee
Using \eqref{est:tb_x:dyn_spec}, together with \eqref{est1:xb_x_a:dyn} and Proposition \ref{lem2:tB}, we get 
\Be \label{est1:tb_x12:dyn_spec}
\begin{split}
& |\p_{x_i} \tB (t,x,p)|
\\& \leq \frac{1}{ \alpha_F (t,x,p) }  
e^{ \frac{8}{m g} (4g + 2 \| \nabla_x ^2 \phi_F \|_{\infty} + e \| \p_t \p_{x_3} \phi_F (t, x_\parallel , 0) \|_{L^\infty(\p\O)} ) |\pB^0 |} \delta_{i3}
\\& \ \ \ \ + \Big( \frac{|\pB^0 (t,x,p)|}{2 c} |\tB (t,x,p)| \Big[ \frac{2}{m g}  + \frac{ 16 (\sqrt{c m} + | \pBn (t,x,p) |)}{c m^2 g} 
\\& \qquad \qquad \times \big(1 + \sqrt{(m c)^2 + |\pB (t,x,p)|^2} + 2 \pBn (t,x, p) \big) \Big]
( \| \nabla_x^2 \phi_F  \|_\infty + e B_3 + mg ) \Big) 
\\& \qquad \qquad \qquad \qquad \times e^{ (1 + B_3 + \| \nabla_x ^2 \phi_F \|_\infty) \tB (t,x,p) }.
\end{split}
\Ee
Using the assumptions \eqref{condition3:F_support_spec} and \eqref{condition3:tilde_epilon_spec}, together with \eqref{est1:tb_x12:dyn_spec}, we obtain that
\be \label{est2:F_x12_7_spec}
\varepsilon | \nabla_{x_{\|}} \tB (t,x,p) | \times \frac{c}{|\tpB^0 (t,x,p)|}
\leq \frac{1}{8} e^{ - \frac{\tilde{\beta}}{4} \sqrt{(m c)^2 + |p|^2} } e^{ - \frac{m g}{8 c} \tilde{\beta} x_3}.
\ee
From \eqref{est1:xb_x:dyn} and \eqref{est1:vb_x:dyn}, we further derive
\be \label{est2:F_x12_35_spec}
\begin{split}
& \varepsilon \Big\{
\frac{ | \nabla_{x_{\|}} \tB (t,x,p) | }{ e^{ \frac{\tilde{\beta}}{4} |\pB^0 (t,x,p)| } } \big\{ c + e B_3 + \frac{3}{2} m g \big\}
+ \frac{ | \nabla_{x_{\|}} \xB (t,x,p) | + |\nabla_{x_{\|}} \pB (t,x,p)|}{ e^{ \frac{\tilde{\beta}}{4} |\pB^0 (t,x,p)| } }
\Big\}
\\& \leq \frac{1}{8} e^{ - \frac{\tilde{\beta}}{4} \sqrt{(m c)^2 + |p|^2} } e^{ - \frac{m g}{8 c} \tilde{\beta} x_3}.
\end{split}
\ee
Inputting \eqref{est2:F_x12_spec}, \eqref{est2:F_x12_7_spec} and \eqref{est2:F_x12_35_spec} into \eqref{est1:F_x1_spec}-\eqref{est1:F_x6_spec}, we conclude that
\be \label{est2:F_x_spec}
\begin{split}
& | \nabla_{x_{\|}} F (t,x,p) | 
\\& \lesssim 2 e^{ \frac{m g}{24} \tilde{\beta} } e^{2 \tilde{\beta} \frac{e}{c} \| (-\Delta_0)^{-1} (\nabla_x \cdot b ) \|_{L^\infty_{t,x}} } e^{ - \frac{\tilde{\beta}}{4} \sqrt{(m c)^2 + |p|^2} } e^{ - \frac{m g}{8 c} \tilde{\beta} x_3} \| \w_{\tilde \beta, 0}   \nabla_{x,p} F_0  \|_{L^\infty (\O \times \R^3)}
\\& \ \ \ \ + \big( 1 + \frac{ \| \nabla_x^2 \phi_F \|_\infty + e B_3 + mg}{(m g)^2} \big) e^{ - \frac{\tilde \beta}{2} \big( |p^0| + \frac{m g}{2 c} x_3 \big) } \| e^{\tilde{\beta} |p^0|} \nabla_{x_\parallel,p} G \|_{L^\infty (\gamma_-)}
\\& \ \ \ \ + \frac{1}{8} e^{ - \frac{\tilde{\beta}}{4} \sqrt{(m c)^2 + |p|^2} } e^{ - \frac{m g}{8 c} \tilde{\beta} x_3} 
\times \sup\limits_{0 \leq t < \infty} \| p_3 \p_{x_3} F (t,x,p) \|_{L^\infty (\bar{\O} \times \R^3)}
\\& \ \ \ \ + \frac{1}{8} e^{ - \frac{\tilde{\beta}}{4} \sqrt{(m c)^2 + |p|^2} } e^{ - \frac{m g}{8 c} \tilde{\beta} x_3}
\times \sup\limits_{0 \leq t < \infty} \| e^{ \frac{\tilde{\beta}}{4} \sqrt{(m c)^2 + |p|^2} } e^{ \frac{m g}{8 c} \tilde{\beta} x_3} \nabla_{x_\parallel,p} F (t,x,p) \|_{L^\infty (\bar{\O} \times \R^3)}
\\& \ \ \ \ + \frac{1}{8} e^{ - \frac{\tilde{\beta}}{4} \sqrt{(m c)^2 + |p|^2} } e^{ - \frac{m g}{8 c} \tilde{\beta} x_3} \times \sup\limits_{0 \leq t < \infty} \| e^{ \frac{\tilde{\beta}}{4} \sqrt{(m c)^2 + |p|^2} } e^{ \frac{m g}{8 c} \tilde{\beta} x_3} \nabla_{x_\parallel,p} F (t,x,p) \|_{L^\infty (\bar{\O} \times \R^3)}.
\end{split} 
\ee

On the other hand, using the assumption \eqref{Bootstrap_spec} and $\alpha_{F} (t,x,p)$ defined in \eqref{alpha_F}, we get
\be \notag
\alpha_F (t,x,p) = \sqrt{ |x_3|^2  + |v_3|^2 +  2 ( e \p_{x_3} \phi_F (t, x_\parallel , 0) + m g ) \frac{c x_3}{ p^0 } } 
\geq |v_3|.
\ee
The above, together with \eqref{est1:tb_x12:dyn_spec}, shows that
\Be \label{est1:tb_x3:dyn_spec}
\begin{split}
& \big| p_3 \p_{x_3} \tB (t,x,p) \big|
\\& \leq \frac{p^0}{c}  
e^{ \frac{8}{m g} (4g + 2 \| \nabla_x ^2 \phi_F \|_{\infty} + e \| \p_t \p_{x_3} \phi_F (t, x_\parallel , 0) \|_{L^\infty(\p\O)} ) |\pB^0 |} 
\\& \ \ \ \ + |p_3| \Big( \frac{|\pB^0 (t,x,p)|}{2 c} |\tB (t,x,p)| \Big[ \frac{2}{m g}  + \frac{ 16 (\sqrt{c m} + | \pBn (t,x,p) |)}{c m^2 g} 
\\& \qquad \qquad \times \big(1 + \sqrt{(m c)^2 + |\pB (t,x,p)|^2} + 2 \pBn (t,x, p) \big) \Big]
( \| \nabla_x^2 \phi_F  \|_\infty + e B_3 + mg ) \Big) 
\\& \qquad \qquad \qquad \qquad \times e^{ (1 + B_3 + \| \nabla_x ^2 \phi_F \|_\infty) \tB (t,x,p) }.
\end{split}
\Ee
Analogous to \eqref{est1:xb_x:dyn} and \eqref{est1:vb_x:dyn}, we further derive that
\be \label{est1:xb+pb_x3:dyn_spec}
\begin{split}
& \big| p_3 \p_{x_3} \xB (t,x,p) \big| + \big| p_3 \p_{x_3} \pB (t,x,p) \big|
\\& \leq \Big( \frac{p^0}{c}  + 1 + \frac{ \| \nabla_x^2 \phi_F \|_\infty + e B_3 + m g}{(m g)^2} \Big) e^{\frac{5 \tilde{\beta}}{12} |\pB^0 (t,x,p)|}.
\end{split}
\ee
Now using \eqref{est1:F_x1_spec}-\eqref{est1:F_x6_spec}, together with \eqref{est2:F_x12_spec}, \eqref{est1:tb_x3:dyn_spec} and \eqref{est1:xb+pb_x3:dyn_spec}, we obtain 
\be \label{est1:p3_F_x3_spec}
\begin{split}
& | p_3 \p_{x_3} F (t,x,p) | 
\\& \lesssim 2 e^{ \frac{m g}{24} \tilde{\beta} }
e^{2 \tilde{\beta} \frac{e}{c} \| (-\Delta_0)^{-1} (\nabla_x \cdot b ) \|_{L^\infty_{t,x}} } e^{ - \frac{\tilde{\beta}}{4} \sqrt{(m c)^2 + |p|^2} } e^{ - \frac{m g}{8 c} \tilde{\beta} x_3} \| \w_{\tilde \beta, 0}   \nabla_{x,p} F_0 \|_{L^\infty (\O \times \R^3)}
\\& \ \ \ \ + \big( \frac{p^0}{c} + 1 + \frac{ \| \nabla_x^2 \phi_F \|_\infty + e B_3 + m g}{(m g)^2} \big) \| e^{\tilde{\beta} |p^0|} \nabla_{x_\parallel,p} G \|_{L^\infty (\gamma_-)} 
\\& \ \ \ \ + \varepsilon \Big\{ \frac{p^0}{c}  
e^{ \frac{8}{m g} (4g + 2 \| \nabla_x ^2 \phi_F \|_{\infty} + e \| \p_t \p_{x_3} \phi_F (t, x_\parallel , 0) \|_{L^\infty(\p\O)} ) |\pB^0 |} 
\\& \qquad \qquad + |p_3| \Big( \frac{|\pB^0 (t,x,p)|}{2 c} |\tB (t,x,p)| \Big[ \frac{2}{m g}  + \frac{ 16 (\sqrt{c m} + | \pBn (t,x,p) |)}{c m^2 g} 
\\& \qquad \qquad \times \big(1 + \sqrt{(m c)^2 + |\pB (t,x,p)|^2} + 2 \pBn (t,x, p) \big) \Big]
( \| \nabla_x^2 \phi_F  \|_\infty + e B_3 + mg ) \Big) 
\\& \qquad \qquad \qquad \times e^{ (1 + B_3 + \| \nabla_x ^2 \phi_F \|_\infty) \tB (t,x,p) } \Big\}
\\& \qquad \qquad \times \frac{c}{|\tpB^0 (t,x,p)|} \big| \tpBn \p_{x_3} F (t - \tB, \xB, \tpB ) \big|
\\& \ \ \ \ + \varepsilon \Big\{ \frac{p^0}{c}  
e^{ \frac{8}{m g} (4g + 2 \| \nabla_x ^2 \phi_F \|_{\infty} + e \| \p_t \p_{x_3} \phi_F (t, x_\parallel , 0) \|_{L^\infty(\p\O)} ) |\pB^0 |} 
\\& \qquad \qquad + |p_3| \Big( \frac{|\pB^0 (t,x,p)|}{2 c} |\tB (t,x,p)| \Big[ \frac{2}{m g}  + \frac{ 16 (\sqrt{c m} + | \pBn (t,x,p) |)}{c m^2 g} 
\\& \qquad \qquad \times \big(1 + \sqrt{(m c)^2 + |\pB (t,x,p)|^2} + 2 \pBn (t,x, p) \big) \Big]
( \| \nabla_x^2 \phi_F  \|_\infty + e B_3 + mg ) \Big) 
\\& \qquad \qquad \qquad \times e^{ (1 + B_3 + \| \nabla_x ^2 \phi_F \|_\infty) \tB (t,x,p) } \Big\} \big\{ c + e B_3 + \frac{3}{2} m g \big\}
\\& \qquad \qquad \times \sup\limits_{0 \leq t < \infty} \| e^{ \frac{\tilde{\beta}}{4} \sqrt{(m c)^2 + |p|^2} } e^{ \frac{m g}{8 c} \tilde{\beta} x_3} \nabla_{x_\parallel,p} F (t,x,p) \|_{L^\infty (\bar{\O} \times \R^3)}
\\& \ \ \ \ + \varepsilon 
\Big( \frac{p^0}{c}  + 1 + \frac{ \| \nabla_x^2 \phi_F \|_\infty + e B_3 + m g}{(m g)^2} \Big) e^{\frac{5 \tilde{\beta}}{12} |\pB^0 (t,x,p)|} e^{ - \frac{\tilde{\beta}}{4} |\pB^0 (t,x,p)| }
\\& \qquad \qquad \times \sup\limits_{0 \leq t < \infty} \| e^{ \frac{\tilde{\beta}}{4} \sqrt{(m c)^2 + |p|^2} } e^{ \frac{m g}{8 c} \tilde{\beta} x_3} \nabla_{x_\parallel,p} F (t,x,p) \|_{L^\infty (\bar{\O} \times \R^3)}.
\end{split} 
\ee
From the assumptions \eqref{condition3:F_support_spec} and \eqref{condition3:tilde_epilon_spec}, together with \eqref{est1:xb+pb_x3:dyn_spec} and \eqref{est1:tb_x3:dyn_spec}, we derive that
\be \label{est2:p3_F_x3_spec}
\begin{split}
& | p_3 \p_{x_3} F (t,x,p) | 
\\& \lesssim 2 e^{ \frac{m g}{24} \tilde{\beta} }
e^{2 \tilde{\beta} \frac{e}{c} \| (-\Delta_0)^{-1} (\nabla_x \cdot b ) \|_{L^\infty_{t,x}} } e^{ - \frac{\tilde{\beta}}{4} \sqrt{(m c)^2 + |p|^2} } e^{ - \frac{m g}{8 c} \tilde{\beta} x_3} \| \w_{\tilde \beta, 0}   \nabla_{x,p} F_0 \|_{L^\infty (\O \times \R^3)}
\\& \ \ \ \ + \big( \frac{p^0}{c} + 1 + \frac{ \| \nabla_x^2 \phi_F \|_\infty + e B_3 + m g}{(m g)^2} \big) \| e^{\tilde{\beta} |p^0|} \nabla_{x_\parallel,p} G \|_{L^\infty (\gamma_-)} 
\\& \ \ \ \ + \frac{1}{8} e^{ - \frac{\tilde{\beta}}{4} \sqrt{(m c)^2 + |p|^2} } e^{ - \frac{m g}{8 c} \tilde{\beta} x_3} 
\times \sup\limits_{0 \leq t < \infty} \| p_3 \p_{x_3} F (t,x,p) \|_{L^\infty (\bar{\O} \times \R^3)}
\\& \ \ \ \ + \frac{1}{8} e^{ - \frac{\tilde{\beta}}{4} \sqrt{(m c)^2 + |p|^2} } e^{ - \frac{m g}{8 c} \tilde{\beta} x_3}
\times \sup\limits_{0 \leq t < \infty} \| e^{ \frac{\tilde{\beta}}{4} \sqrt{(m c)^2 + |p|^2} } e^{ \frac{m g}{8 c} \tilde{\beta} x_3} \nabla_{x_\parallel,p} F (t,x,p) \|_{L^\infty (\bar{\O} \times \R^3)}
\\& \ \ \ \ + \frac{1}{8} e^{ - \frac{\tilde{\beta}}{4} \sqrt{(m c)^2 + |p|^2} } e^{ - \frac{m g}{8 c} \tilde{\beta} x_3} \times \sup\limits_{0 \leq t < \infty} \| e^{ \frac{\tilde{\beta}}{4} \sqrt{(m c)^2 + |p|^2} } e^{ \frac{m g}{8 c} \tilde{\beta} x_3} \nabla_{x_\parallel,p} F (t,x,p) \|_{L^\infty (\bar{\O} \times \R^3)}.
\end{split} 
\ee

Combining \eqref{est2:F_v_spec}, \eqref{est2:F_x_spec} and \eqref{est2:p3_F_x3_spec}, we derive that
\be \label{est:F_v+F_x+p3_F_x3_spec}
\begin{split}
& | \nabla_p F (t,x,p) | + | \nabla_{x_{\|}} F (t,x,p) | + | p_3 \p_{x_3} F (t,x,p) |
\\& \lesssim 6 e^{ \frac{m g}{24} \tilde{\beta} } e^{2 \tilde{\beta} \frac{e}{c} \| (-\Delta_0)^{-1} (\nabla_x \cdot b ) \|_{L^\infty_{t,x}} } e^{ - \frac{\tilde{\beta}}{4} \sqrt{(m c)^2 + |p|^2} } e^{ - \frac{m g}{8 c} \tilde{\beta} x_3} \| \w_{\tilde \beta, 0}   \nabla_{x,p} F_0  \|_{L^\infty (\O \times \R^3)}
\\& \ \ \ \ + 3 \big( 1 + \frac{ \| \nabla_x^2 \phi_F \|_\infty + e B_3 + mg}{(m g)^2} \big) e^{ - \frac{\tilde \beta}{2} \big( |p^0| + \frac{m g}{2 c} x_3 \big) } \| e^{\tilde{\beta} |p^0|} \nabla_{x_\parallel,p} G \|_{L^\infty (\gamma_-)}
\\& \ \ \ \ + \frac{3}{8} e^{ - \frac{\tilde{\beta}}{4} \sqrt{(m c)^2 + |p|^2} } e^{ - \frac{m g}{8 c} \tilde{\beta} x_3} 
\\& \qquad \qquad \times \Big\{ \sup\limits_{0 \leq t < \infty} \| p_3 \p_{x_3} F (t,x,p) \|_{L^\infty (\bar{\O} \times \R^3)}
\\& \qquad \qquad \qquad \qquad + \sup\limits_{0 \leq t < \infty} \| e^{ \frac{\tilde{\beta}}{4} \sqrt{(m c)^2 + |p|^2} } e^{ \frac{m g}{8 c} \tilde{\beta} x_3} \nabla_{x_\parallel,p} F (t,x,p) \|_{L^\infty (\bar{\O} \times \R^3)}
\\& \qquad \qquad \qquad \qquad + \sup\limits_{0 \leq t < \infty} \| e^{ \frac{\tilde{\beta}}{4} \sqrt{(m c)^2 + |p|^2} } e^{ \frac{m g}{8 c} \tilde{\beta} x_3} \nabla_{x_\parallel,p} F (t,x,p) \|_{L^\infty (\bar{\O} \times \R^3)} \Big\},
\end{split}
\ee
and this shows that
\be \label{est1:F_v+F_x+p3_F_x3_spec}
\begin{split}
& e^{ \frac{\tilde{\beta}}{4} \sqrt{(m c)^2 + |p|^2} } e^{ \frac{m g}{8 c} \tilde{\beta} x_3} \Big( | \nabla_p F (t,x,p) | + | \nabla_{x_{\|}} F (t,x,p) | + | p_3 \p_{x_3} F (t,x,p) | \Big)
\\& \lesssim 6 e^{ \frac{m g}{24} \tilde{\beta} } e^{2 \tilde{\beta} \frac{e}{c} \| (-\Delta_0)^{-1} (\nabla_x \cdot b ) \|_{L^\infty_{t,x}} } \| \w_{\tilde \beta, 0}   \nabla_{x,p} F_0  \|_{L^\infty (\O \times \R^3)}
\\& \ \ \ \ + 3 \big( 1 + \frac{ \| \nabla_x^2 \phi_F \|_\infty + e B_3 + mg}{(m g)^2} \big) \| e^{\tilde{\beta} |p^0|} \nabla_{x_\parallel,p} G \|_{L^\infty (\gamma_-)}
\\& \ \ \ \ + \frac{3}{4} \times \Big\{ \sup\limits_{0 \leq t < \infty} \| p_3 \p_{x_3} F (t,x,p) \|_{L^\infty (\bar{\O} \times \R^3)}
\\& \qquad \qquad \qquad \qquad + \sup\limits_{0 \leq t < \infty} \| e^{ \frac{\tilde{\beta}}{4} \sqrt{(m c)^2 + |p|^2} } e^{ \frac{m g}{8 c} \tilde{\beta} x_3} \nabla_{x_\parallel,p} F (t,x,p) \|_{L^\infty (\bar{\O} \times \R^3)} \Big\}.
\end{split}
\ee
Therefore, we derive 
\be \label{est2:F_v+F_x+p3_F_x3_spec}
\begin{split}
& \frac{1}{4} \times \Big\{ \sup\limits_{0 \leq t < \infty} \| e^{ \frac{\tilde{\beta}}{4} \sqrt{(m c)^2 + |p|^2} } e^{ \frac{m g}{8 c} \tilde{\beta} x_3} p_3 \p_{x_3} F (t,x,p) \|_{L^\infty (\bar{\O} \times \R^3)}
\\& \qquad \qquad + \sup\limits_{0 \leq t < \infty} \| e^{ \frac{\tilde{\beta}}{4} \sqrt{(m c)^2 + |p|^2} } e^{ \frac{m g}{8 c} \tilde{\beta} x_3} \nabla_{p} F (t,x,p) \|_{L^\infty (\bar{\O} \times \R^3)}
\\& \qquad \qquad + \sup\limits_{0 \leq t < \infty} \| e^{ \frac{\tilde{\beta}}{4} \sqrt{(m c)^2 + |p|^2} } e^{ \frac{m g}{8 c} \tilde{\beta} x_3} \nabla_{x_\parallel} F (t,x,p) \|_{L^\infty (\bar{\O} \times \R^3)} \Big\}
\\& \lesssim 6 e^{ \frac{m g}{24} \tilde{\beta} } e^{2 \tilde{\beta} \frac{e}{c} \| (-\Delta_0)^{-1} (\nabla_x \cdot b ) \|_{L^\infty_{t,x}} } \| \w_{\tilde \beta, 0}   \nabla_{x,p} F_0  \|_{L^\infty (\O \times \R^3)}
\\& \ \ \ \ + 3 \big( 1 + \frac{ \| \nabla_x^2 \phi_F \|_\infty + e B_3 + mg}{(m g)^2} \big) \| e^{\tilde{\beta} |p^0|} \nabla_{x_\parallel,p} G \|_{L^\infty (\gamma_-)}.
\end{split}
\ee
and thus we conclude \eqref{est:F_v:dyn_spec}.

\smallskip

\textbf{Step 2. Proof of \eqref{est:F_x:dyn_spec}.}
From \eqref{est2:F_v+F_x+p3_F_x3_spec}, we directly derive 
\be \label{est3:F_x_spec}
\begin{split}
& \sup\limits_{0 \leq t < \infty} \| e^{ \frac{\tilde{\beta}}{4} \sqrt{(m c)^2 + |p|^2} } e^{ \frac{m g}{8 c} \tilde{\beta} x_3} \nabla_{x_\parallel} F (t,x,p) \|_{L^\infty (\bar{\O} \times \R^3)}
\\& \lesssim 24 e^{ \frac{m g}{24} \tilde{\beta} } e^{2 \tilde{\beta} \frac{e}{c} \| (-\Delta_0)^{-1} (\nabla_x \cdot b ) \|_{L^\infty_{t,x}} } \| \w_{\tilde \beta, 0}   \nabla_{x,p} F_0  \|_{L^\infty (\O \times \R^3)}
\\& \ \ \ \ + 12 \big( 1 + \frac{ \| \nabla_x^2 \phi_F \|_\infty + e B_3 + mg}{(m g)^2} \big) \| e^{\tilde{\beta} |p^0|} \nabla_{x_\parallel,p} G \|_{L^\infty (\gamma_-)}.
\end{split}
\ee
Using \eqref{est1:F_x1_spec}-\eqref{est1:F_x6_spec}, together with \eqref{est1:xb_x:dyn}, \eqref{est1:vb_x:dyn} and  \eqref{est2:F_x12_spec}, we obtain that
\be \label{est4:F_x_spec}
\begin{split}
& e^{ \frac{\tilde{\beta}}{4} \sqrt{(m c)^2 + |p|^2} } e^{ \frac{m g}{8 c} \tilde{\beta} x_3} \big| \nabla_{x_3} F (t,x,p)  \big|
\\& \lesssim 2 e^{ \frac{m g}{24} \tilde{\beta} }
e^{2 \tilde{\beta} \frac{e}{c} \| (-\Delta_0)^{-1} (\nabla_x \cdot b ) \|_{L^\infty_{t,x}} } e^{ - \frac{\tilde{\beta}}{4} \sqrt{(m c)^2 + |p|^2} } e^{ - \frac{m g}{8 c} \tilde{\beta} x_3} \| \w_{\tilde \beta, 0}   \nabla_{x,p} F_0 \|_{L^\infty (\O \times \R^3)}
\\& \ \ \ \ + \big( \frac{\delta_{i3}}{\alpha_{F} (t,x,p)} + 1 + \frac{ \| \nabla_x^2 \phi_F \|_\infty + e B_3 + m g}{(m g)^2} \big) \| e^{\tilde{\beta} |p^0|} \nabla_{x_\parallel,p} G \|_{L^\infty (\gamma_-)}
\\& \ \ \ \ + \varepsilon \Big\{ | \nabla_{x} \tB (t,x,p) |  \frac{c}{|\tpB^0 (t,x,p)|} 
+ \frac{ | \nabla_{x} \tB (t,x,p) | }{ e^{ \frac{\tilde{\beta}}{4} |\pB^0 (t,x,p)| } } \big\{ c + e B_3 + \frac{3}{2} m g \big\}
\\& \qquad \qquad + \frac{ | \nabla_{x} \xB (t,x,p) | + |\nabla_{x} \pB (t,x,p)| }{ e^{ \frac{\tilde{\beta}}{4} |\pB^0 (t,x,p)| } } \Big\}
\\& \qquad \qquad \times \Big\{ \sup\limits_{0 \leq t < \infty} \| e^{ \frac{\tilde{\beta}}{4} \sqrt{(m c)^2 + |p|^2} } e^{ \frac{m g}{8 c} \tilde{\beta} x_3} p_3 \p_{x_3} F (t,x,p) \|_{L^\infty (\bar{\O} \times \R^3)}
\\& \qquad \qquad \qquad + \sup\limits_{0 \leq t < \infty} \| e^{ \frac{\tilde{\beta}}{4} \sqrt{(m c)^2 + |p|^2} } e^{ \frac{m g}{8 c} \tilde{\beta} x_3} \nabla_{p} F (t,x,p) \|_{L^\infty (\bar{\O} \times \R^3)}
\\& \qquad \qquad \qquad + \sup\limits_{0 \leq t < \infty} \| e^{ \frac{\tilde{\beta}}{4} \sqrt{(m c)^2 + |p|^2} } e^{ \frac{m g}{8 c} \tilde{\beta} x_3} \nabla_{x_\parallel} F (t,x,p) \|_{L^\infty (\bar{\O} \times \R^3)} \Big\}.
\end{split}
\ee
Together with \eqref{est1:tb_x12:dyn_spec} and \eqref{est2:F_v+F_x+p3_F_x3_spec}, we derive that
\be \notag
\begin{split}
& e^{ \frac{\tilde{\beta}}{4} \sqrt{(m c)^2 + |p|^2} } e^{ \frac{m g}{8 c} \tilde{\beta} x_3} \big| \nabla_{x_3} F (t,x,p)  \big|
\\& \lesssim 2 e^{ \frac{m g}{24} \tilde{\beta} }
e^{2 \tilde{\beta} \frac{e}{c} \| (-\Delta_0)^{-1} (\nabla_x \cdot b ) \|_{L^\infty_{t,x}} } e^{ - \frac{\tilde{\beta}}{4} \sqrt{(m c)^2 + |p|^2} } e^{ - \frac{m g}{8 c} \tilde{\beta} x_3} \| \w_{\tilde \beta, 0}   \nabla_{x,p} F_0 \|_{L^\infty (\O \times \R^3)}
\\& \ \ \ \ + \big( \frac{\delta_{i3}}{\alpha_{F} (t,x,p)} + 1 + \frac{ \| \nabla_x^2 \phi_F \|_\infty + e B_3 + m g}{(m g)^2} \big) \| e^{\tilde{\beta} |p^0|} \nabla_{x_\parallel,p} G \|_{L^\infty (\gamma_-)}
\\& \ \ \ \ + \varepsilon \Big\{ \frac{1}{ \alpha_F (t,x,p) }  
e^{ \frac{8}{m g} (4g + 2 \| \nabla_x ^2 \phi_F \|_{\infty} + e \| \p_t \p_{x_3} \phi_F (t, x_\parallel , 0) \|_{L^\infty(\p\O)} ) |\pB^0 |}
\\& \qquad \qquad + e^{ (1 + B_3 + \| \nabla_x ^2 \phi_F \|_\infty) \tB (t,x,p) } \Big( \frac{|\pB^0 (t,x,p)|}{2 c} |\tB (t,x,p)| \Big[ \frac{2}{m g}  + \frac{ 16 (\sqrt{c m} + | \pBn (t,x,p) |)}{c m^2 g} 
\\& \qquad \qquad \qquad \times \big(1 + \sqrt{(m c)^2 + |\pB (t,x,p)|^2} + 2 \pBn (t,x, p) \big) \Big]
( \| \nabla_x^2 \phi_F  \|_\infty + e B_3 + mg ) \Big) 
\\& \qquad \qquad + \big( \frac{\delta_{i3}}{\alpha_F (t,x,p)} + 1 + \frac{ \| \nabla_x^2 \phi_F \|_\infty + e B_3 + mg}{(m g)^2} \big) e^{ - \frac{\tilde \beta}{2} \big( |p^0| + \frac{m g}{2 c} x_3 \big) } e^{ \frac{3}{4} \tilde{\beta} |\pB^0 (t,x,p)| } \Big\}
\\& \qquad \qquad \qquad \times \Big\{ 24 e^{ \frac{m g}{24} \tilde{\beta} } e^{2 \tilde{\beta} \frac{e}{c} \| (-\Delta_0)^{-1} (\nabla_x \cdot b ) \|_{L^\infty_{t,x}} } \| \w_{\tilde \beta, 0} \nabla_{x,p} F_0  \|_{L^\infty (\O \times \R^3)}
\\& \qquad \qquad \qquad \qquad + 12 \big( 1 + \frac{ \| \nabla_x^2 \phi_F \|_\infty + e B_3 + mg}{(m g)^2} \big) \| e^{\tilde{\beta} |p^0|} \nabla_{x_\parallel,p} G \|_{L^\infty (\gamma_-)} \Big\},
\end{split}
\ee
and thus we conclude \eqref{est:F_x:dyn_spec} via the assumptions \eqref{condition3:F_support_spec} and \eqref{condition3:tilde_epilon_spec}.
\end{proof}

\begin{lemma} \label{lem:D3tphi_F_spec}

Suppose that \eqref{Bootstrap_spec} holds. Further, we assume 
\be \label{est:D2xphi_F_spec}
1 + B_3 + \| \nabla_x ^2 \phi_F \|_\infty 
\leq \frac{\hat{m} g}{24} \tilde{\beta},
\ee
where $\hat{m} = \min \{ m_+, m_- \}$.
Suppose there exists $\tpx > 1$ such that 
\be \notag
f_{\pm} (t, x, p) = 0
\ \text{ for any } \
\frac{4 c}{m_{\pm} g} |p| + 3 x_{3} \geq \tpx.
\ee
Set $\lambda = \frac{g}{48} \hat{m} \beta$, suppose $\varepsilon$ in the boundary condition \eqref{bdry:F_spec} satisfies that	
\be \notag
2 \varepsilon 
< \exp \Big\{ - \lambda \Big( \big( \frac{1}{\min \{ \frac{g}{4 \sqrt{2}}, \frac{c}{\sqrt{10}} \}} + 1 \big) \times \tpx + 1 \Big) \Big\}.
\ee
Consider $(t,x) \in [0, \infty) \times \bar\O$, then
\be \label{est:D.b_spec}
\begin{split}
|\nabla_x \cdot b (t, x)|  
& \lesssim C_1 \frac{ \mathbf{1}_{|x_3| \leq 1} }{ \sqrt{ x_3 } } + C_2  e^{ - \frac{\hat{m} g}{8 c} \tilde{\beta} x_3},
\end{split}
\ee
and
\be \label{est:D3tphi_F_spec} 
| \p_{x_3} \p_t \phi_{F} (t,x)| 
\lesssim C_1 + C_2 ( 1 + \frac{8 c}{\hat{m} g \tilde{\beta}} ),  
\ee
where $\hat{m} = \min\{m_+, m_- \}$, and
\begin{align}
C_{1} & = \frac{ 1 }{{\tilde{\beta}}^3} ( \frac{e_+}{\sqrt{ m^2_+ g  }} + \frac{e_-}{\sqrt{ m^2_- g  }} ),
\label{express:C_spec} \\
C_{2} & = \sum\limits_{i = \pm} \frac{e_i }{{\tilde{\beta}}^3} \Big( 2 e^{ \frac{m_i g}{24} \tilde{\beta} } 
e^{2 \tilde{\beta} \frac{e_i}{c} \| (-\Delta_0)^{-1} (\nabla_x \cdot b ) \|_{L^\infty_{t,x}} } \| \w_{\pm, \tilde \beta, 0}   \nabla_{x,p} F_{\pm, 0} \|_{L^\infty (\O \times \R^3)}
\\& \qquad \qquad \quad + \big( 1 + \frac{ \| \nabla_x^2 \phi_F \|_\infty + e B_3 + m_i g }{(m_i g)^2} \big) \| e^{\tilde{\beta} |p^0|} \nabla_{x_\parallel,p} G_i \|_{L^\infty (\gamma_-)} + \frac{1}{\sqrt{ m_i g}} \Big).
\label{express:D_spec}
\end{align}
\end{lemma}

\begin{proof}

We abuse the notation as in \eqref{abuse} in the proof. Analogous to \eqref{eq:int_v*h_x=0} and \eqref{est:D.b1}, we have
\Be \notag
\begin{split}
|\nabla_x \cdot b (t,x)| 
\leq e_+ \int_{\R^3} | v_+ \cdot \nabla_x F_+ (t,x,p) | \dd p + e_- \int_{\R^3} | v_- \cdot \nabla_x F_- (t,x,p) | \dd p.
\end{split}
\Ee 
Recall the proof of \eqref{est:F_x:dyn_spec}, we bound $\nabla_x F (t,x,p)$ by
\be \notag 
\begin{split}
& | \nabla_{x} F (t,x,p) | 
\\& \leq \mathbf{1}_{t \leq \tB (t,x,p)}
\frac{| \nabla_{x} \X  (0;t,x,p)| +| \nabla_{x} \P  (0;t,x,p)|  }{\w_{\tilde \beta} (0, \Z (0;t,x,p) )} \| \w_{\tilde \beta, 0}   \nabla_{x,p} F_0  \|_{L^\infty (\O \times \R^3)} \\& \ \ \ \ +	\mathbf{1}_{t > \tB (t,x,p)} \frac{ | \nabla_{x} \xB (t,x,p) | + |\nabla_{x} \pB (t,x,p)|}{ e^{\tilde{\beta} |\pB^0 (t,x,p)| } } \| e^{\tilde{\beta} |p^0|} \nabla_{x_\parallel,p} G \|_{L^\infty (\gamma_-)}
\\& \ \ \ \ + \mathbf{1}_{t >\tB (t,x,p) } \varepsilon | \nabla_{x} \tB (t,x,p) | \times \frac{c}{|\tpB^0 (t,x,p)|} \big| \tpBn \p_{x_3} F (t - \tB, \xB, \tpB ) \big|
\\& \ \ \ \ + \mathbf{1}_{t >\tB (t,x,p) } \varepsilon 
\frac{ | \nabla_{x} \tB (t,x,p) | }{ e^{ \frac{\tilde{\beta}}{4} |\pB^0 (t,x,p)| } } \big\{ c + e B_3 + \frac{3}{2} m g \big\}
\\& \qquad \qquad \qquad \qquad \qquad \times \sup\limits_{0 \leq t < \infty} \| e^{ \frac{\tilde{\beta}}{4} \sqrt{(m c)^2 + |p|^2} } e^{ \frac{m g}{8 c} \tilde{\beta} x_3} \nabla_{x_\parallel,p} F (t,x,p) \|_{L^\infty (\bar{\O} \times \R^3)}
\\& \ \ \ \ + \mathbf{1}_{t >\tB (t,x,p) } \varepsilon 
\frac{ | \nabla_{x} \xB (t,x,p) | + |\nabla_{x} \pB (t,x,p)| }{ e^{ \frac{\tilde{\beta}}{4} |\pB^0 (t,x,p)| } }
\\& \qquad \qquad \qquad \qquad \qquad \times \sup\limits_{0 \leq t < \infty} \| e^{ \frac{\tilde{\beta}}{4} \sqrt{(m c)^2 + |p|^2} } e^{ \frac{m g}{8 c} \tilde{\beta} x_3} \nabla_{x_\parallel,p} F (t,x,p) \|_{L^\infty (\bar{\O} \times \R^3)}.
\end{split}
\ee
From \eqref{est:tb_x:dyn_spec}, we bound $|\p_{x_i} \tBp (t,x,p)|$ by
\Be \notag
\begin{split}
& | \p_{x_i} \tBp (t,x,p) | 
\\& \leq \frac{1}{ | \vBpn (t,x,p) | }  
\Big( \delta_{i3} + \frac{|\tBp (t,x,p)|^2}{2} ( \| \nabla_x^2 \phi_F \|_\infty + e_{\pm} B_3 + m_{\pm} g ) \Big) e^{ (1 + B_3 + \| \nabla_x ^2 \phi_F  \|_\infty) |\tBp| }.
\end{split}
\Ee
Moreover, following \eqref{est:vb/vbn} we obtain
\be \notag
\frac{|v_3|}{|\vBn (t,x,p)|}
\leq 1 + \frac{ \sqrt{ 2 (m_\pm c)^2 + 2 |\pB|^2} }{\sqrt{m^2 g x_3}}.
\ee
Similar to Lemma \ref{lem:D3tphi_F}, although we use the assumption \eqref{est1:D3tphi_F_spec} in  \eqref{est1:F_x1_spec}, we remark that the condition \eqref{est:D2xphi_F_spec} is sufficient for the above estimate. We omit the rest of the proof, since it follows directly from Lemma \ref{lem:D3tphi_F} and Proposition \ref{RE:dyn_spec}.
\end{proof}

Using the results from Proposition \ref{RE:dyn_spec} and Lemma \ref{lem:D3tphi_F_spec}, we illustrate the regularity estimate for the dynamical problem as follows.

\begin{theorem}[Regularity Estimate]
\label{theo:RD_spec}

Suppose $(F_{\pm}, \phi_F)$ solves \eqref{VP_F}, \eqref{Poisson_F}, \eqref{VP_0}, \eqref{Dbc:F} and \eqref{bdry:F_spec} under the compatibility condition \eqref{CC_F0=G_spec}.
Consider 
\be \notag
F_{\pm} (t,x,p) = h_{\pm} (x,p) + f_{\pm} (t,x,p),
\ee 
where $(h_{\pm}, \rho, \Phi)$ solves \eqref{VP_h_spec}-\eqref{eqtn:Dphi_spec} and $(f_{\pm}, \varrho, \Psi)$ solves \eqref{eqtn:f_spec}-\eqref{Poisson_f_spec} respectively.
Assume all conditions in Theorem \ref{theo:RS_spec} and Theorem \ref{theo:AS_spec}.
Further, we assume that \eqref{est:beta_b_dy_spec}, \eqref{est1:D3tphi_F_spec} and
\Be \label{condition:ML_spec}
M \leq \beta e^{ - \frac{m_{\pm} g}{24} \beta}
\ \text{ and } \ 
L \leq \tilde{\beta} e^{ - \frac{m_{\pm} g}{24} \tilde{\beta} }.
\Ee
hold for some $g, \beta, \tilde \beta > 0$.
Consider $\alpha_{\pm, F} (t,x,p)$ defined in \eqref{alpha_F}, then for any $(t,x,p) \in [0, \infty) \times \bar\O \times \R^3$,
\Be \label{est_final:F_v:dyn_spec}
\begin{split} 
& \| e^{ \frac{\tilde{\beta}}{4} \sqrt{(m_{\pm} c)^2 + |p|^2} } e^{ \frac{m_{\pm} g}{8 c} \tilde{\beta} x_3} \nabla_p F_{\pm} (t,x,p) \|_{L^\infty(\O \times \R^3)} 
\\& \lesssim 2 e^{ 2 + \frac{m_{\pm} g}{24} \tilde{\beta} }
\| \w_{\pm, \tilde \beta, 0}   \nabla_{x,p} F_{\pm, 0} \|_{L^\infty (\O \times \R^3)}
+ \frac{1}{24} \tilde{\beta}
\| e^{\tilde{\beta} |p^0_{\pm}|} \nabla_{x_\parallel,p} G_{\pm} \|_{L^\infty (\gamma_-)},
\end{split}
\Ee 
and 
\Be \label{est_final:F_x:dyn_spec}
\begin{split} 
& e^{ \frac{\tilde{\beta}}{4} \sqrt{(m_{\pm} c)^2 + |p|^2} } e^{ \frac{m_{\pm} g}{8 c} \tilde{\beta} x_3} \big| \nabla_x F (t,x,p) \big|
\\& \lesssim 2 e^{ 2 + \frac{m_{\pm} g}{24} \tilde{\beta}}
\| \w_{\pm, \tilde \beta, 0} \nabla_{x,p} F_{\pm, 0} \|_{L^\infty (\O \times \R^3)}
+ ( \frac{ \mathbf{1}_{|x_3| \leq 1} }{\alpha_{\pm, F} (t,x,p)} + \frac{1}{24} \tilde{\beta} ) \| e^{\tilde{\beta} |p^0_{\pm}|} \nabla_{x_\parallel,p} G_{\pm} \|_{L^\infty (\gamma_-)}.
\end{split}
\Ee 
Moreover, 
\be \label{est_final:phi_F_spec}
\| \nabla_x^2 \phi_F (t) \|_{L^\infty (\bar \O)} + 
\| \p_t  \nabla_x  \phi_F (t) \|_{L^\infty (\bar \O)} 
\lesssim 1.
\ee
\end{theorem}

\begin{proof}

Here we omit the proof since it directly follows from Theorem \ref{theo:RD}, Proposition \ref{RE:dyn_spec} and Lemma \ref{lem:D3tphi_F_spec}.
\end{proof}

\subsubsection{Stability and Uniqueness} 
\label{Sec:SU_spec}

In this section, we prove the stability and uniqueness in Theorem \ref{theo:UD_spec} and Theorem \ref{theo:UA_spec}, respectively. The idea of the proof is similar to the proof of Proposition \ref{prop:cauchy} and Theorem \ref{theo:UD}.

\begin{theorem}[Stability Theorem]
\label{theo:UD_spec}

Suppose $(F_{1, \pm}, \nabla_x \phi_{F_1})$ and $(F_{2, \pm}, \nabla_x \phi_{F_2})$ solve \eqref{VP_F}, \eqref{Poisson_F}, \eqref{VP_0}, \eqref{Dbc:F} and \eqref{bdry:F_spec} under the compatibility condition \eqref{CC_F0=G_spec}.
Consider 
\be \notag
F_{1, \pm} (t,x,p) = h_{1, \pm} (x,p) + f_{1, \pm} (t,x,p),
\ee 
where $(h_{1, \pm}, \rho_1, \Phi_1)$ solves \eqref{VP_h_spec}-\eqref{eqtn:Dphi_spec}, and $(f_{1, \pm}, \varrho_1, \Psi_1)$ solves \eqref{eqtn:f_spec}-\eqref{Poisson_f_spec}.
We define 
\be \label{w1_dy_spec}
\w_{1, \pm} (t,x,p) = e^{ \frac{1}{12} \tilde{\beta} \big( \sqrt{(m_{\pm} c)^2 + |p|^2} + \frac{1}{c} ( e_{\pm} \phi_{F_1} (x) + m_{\pm} g x_3 ) \big) },
\ee 
and
\be \label{b1_dy_spec}
b_1 (t,x) = \int_{\R^3} \big(v_+ e_+ ( F_{1,+} - h_{1,+} ) + v_{-} e_{-} ( F_{1, -} - h_{1, -}) \big) \dd p.
\ee
Assume both $(F_{1, \pm}, \nabla_x \phi_{F_1})$ and $(F_{2, \pm}, \nabla_x \phi_{F_2})$ satisfy all assumptions in Theorem \ref{theo:RD_spec}, then
\Be \label{est:F1-2_spec}
\begin{split}
& \sum_{i = \pm} \sup_{s \in [0,t]} \| \w_{1,i} (F_{1,i} - F_{2,i} ) (s) \|_{L^{\infty} (\O \times \R^3)} 
\\& \leq \exp \Big\{ { t \Big( \frac{1}{12} \tilde{\beta} \sup_{s \in [0,t]} \| (-\Delta_0)^{-1} (\nabla_x \cdot b_1(s)) \|_{L^\infty(\O)}
+ \sum_{j = \pm} \sup_{s \in [0,t]} \| \w_{1, j} \nabla_p F_{2, j} (s) \|_{L^{\infty} (\O \times \R^3)} \Big) } \Big\}
\\& \ \ \ \ \times \sum_{i = \pm} \| \w_{1,i} (F_{1,i} - F_{2,i} ) (0) \|_{L^{\infty} (\O \times \R^3)}.
\end{split}
\Ee	
\end{theorem}

\begin{proof}

We abuse the notation as in \eqref{abuse} in the proof.
From \eqref{ODE_F_spec} and \eqref{tb}, we consider the characteristic $\Z_1 = ( \X_1, \P_1 )$ for $(F_1, \phi_{F_1})$ and the corresponding backward exit time $\tBf (t,x,p)$. 

Now consider $\w_{1, \pm} (t,x,p)$ in \eqref{w1_dy_spec}. Analogous to \eqref{VP_diff_12_dy}, we have
\Be \label{VP_diff_12_dy_spec}
\begin{split}
& \p_t ( \w_{1, \pm} (F_{1, \pm} - F_{2, \pm} )) + v_\pm \cdot \nabla_x ( \w_{1, \pm} (F_{1, \pm} - F_{1, \pm} ))
\\& \ \ \ \ + \Big( e_{\pm} \big( \frac{v_\pm}{c} \times B - \nabla_x \phi_{F_1} \big) - \nabla_x ( m_\pm g x_3) \Big) \cdot \nabla_p ( \w_{1, \pm} (F_{1, \pm} - F_{2, \pm} ))
\\& = - \frac{e_{\pm}}{c} (-\Delta_0)^{-1}  (\nabla_x \cdot b_1 ) \frac{1}{12} \tilde{\beta} \w_{1, \pm} (F_{1, \pm} - F_{2, \pm} )
\\& \ \ \ \ + e_{\pm} \nabla_x ( \phi_{F_1} - \phi_{F_2})
\cdot \nabla_p ( \w_{1, \pm} F_{2, \pm} )
 \ \ \text{in} \ \R_+ \times  \O \times \R^3,
\end{split}
\ee
where $b_1 (t,x)$ is defined in \eqref{b1_dy_spec}. Moreover,
\be \label{VP_diff_bdy_12_dy_spec}
F_1 (t,x,p) - F_2 (t,x,p) =  \varepsilon \big( F_{1} (t, x, \tilde{p}) - F_{2} (t, x, \tilde{p}) \big) \ \ \text{in} \ \R_+ \times \gamma_-,
\Ee
and
\be \label{VP_diff_initial_12_dy_spec}
F_{1, \pm} (0, x, p) - F_{2, \pm} (0, x, p) = F_{1, \pm, 0} (x,p) - F_{2, \pm, 0} (x,p) \ \ \ \ \text{in} \ \O \times \R^3.
\Ee

\smallskip

\textbf{\underline{Case 1:} $\tBf (t,x,p) < t$.}
From \eqref{VP_diff_bdy_12_dy_spec}, we have
\[
\w_{1} (F_1- F_2) (t - \tBf (t,x,p), \xBf, \pBf) 
= \varepsilon \w_{1} (F_1- F_2) (t - \tBf (t,x,p), \xBf, \tpBf). 
\]
Together with \eqref{VP_diff_12_dy_spec}, we further get
\Be \notag
\begin{split}
& \mathbf{1}_{\tBf (t,x,p) < t} \times \w_{1} (F_1- F_2) (t,x,p) - \varepsilon \w_{1} (F_1- F_2) (t - \tBf (t,x,p), \xBf, \tpBf)
\\& = \int^t_{t - \tBf (t,x,p)} - \frac{e}{c} (-\Delta_0)^{-1}  (\nabla_x \cdot b_1 ) \frac{1}{12} \tilde{\beta} \w_{1} (F_{1} - F_{2} ) \dd s 
\\& \ \ \ \ + \int^t_{t - \tBf (t,x,p)} e \nabla_x ( \phi_{F_1} - \phi_{F_2})
\cdot \nabla_p ( \w_{1} F_{2, \pm} ) \dd s 
\\& \leq \frac{e}{c} \frac{1}{12} \tilde{\beta} \sup_{s \in [0,t]} \| (-\Delta_0)^{-1} (\nabla_x \cdot b_1(s)) \|_{L^\infty(\O)}
\int^t_{t - \tBf (t,x,p)} 	\| \w_{1} (F_1  - F_2 )(s) \|_{L^{\infty} (\O \times \R^3)} \dd s 
\\& \ \ \ \ + \sup_{s \in [0,t]} \| \w_{1} \nabla_p F_2 (s) \|_{L^{\infty} (\O \times \R^3)} 
\int^t_{t - \tBf (t,x,p)} 
\| \nabla_x (\phi_{F_1} - \phi_{F_2}) (s) \|_{L^{\infty}_x(\O)}
\dd s.
\end{split}
\Ee
Following \eqref{est:Dphi_12_tb<t}-\eqref{est3:Dphi_12_tb<t} in the proof of Theorem \ref{theo:UD}, we derive
\Be \notag
\begin{split}
& \mathbf{1}_{\tBf (t,x,p) < t} \times \w_{1} (F_1- F_2) (t,x,p) - \varepsilon \w_{1} (F_1- F_2) (t - \tBf (t,x,p), \xBf, \tpBf)
\\& \lesssim_{e, m, g} \frac{1}{12} \tilde{\beta} \sup_{s \in [0,t]} \| (-\Delta_0)^{-1} (\nabla_x \cdot b_1(s)) \|_{L^\infty(\O)}
\int^t_{t - \tBf (t,x,p)} 	\| \w_{1} (F_1  - F_2 )(s) \|_{L^{\infty} (\O \times \R^3)} \dd s 
\\& \qquad + \sup_{s \in [0,t]} \| \w_{1} \nabla_p F_2 (s) \|_{L^{\infty} (\O \times \R^3)}
\times \int^t_{t - \tBf (t,x,p)} \sum_{i = \pm} \| \w_{1,i} (F_{1,i} - F_{2,i} ) (s) \|_{L^{\infty} (\O \times \R^3)} \dd s.
\end{split}
\Ee
This implies that
\Be \label{est3:Dphi_12_tb<t_spec}
\begin{split}
& \sup_{s \in [0,t]} \| \mathbf{1}_{\tBf (s,x,p) < s} \times \w_{1} (F_1- F_2) (s) \|_{L^{\infty} (\O \times \R^3)}
- \varepsilon \sup_{s \in [0,t]} \| \w_{1} (F_1- F_2) (s) \|_{L^{\infty} (\O \times \R^3)}
\\& \lesssim \frac{1}{12} \tilde{\beta} \sup_{s \in [0,t]} \| (-\Delta_0)^{-1} (\nabla_x \cdot b_1(s)) \|_{L^\infty(\O)}
\int^t_{0} \sup_{\tau \in [0, s]}	\| \w_{1} (F_1  - F_2 )(\tau) \|_{L^{\infty} (\O \times \R^3)} \dd s 
\\& \qquad + \sup_{s \in [0,t]} \| \w_{1} \nabla_p F_2 (s) \|_{L^{\infty} (\O \times \R^3)} \int^t_{0} \sum_{i = \pm} \sup_{\tau \in [0, s]} \| \w_{1,i} (F_{1,i} - F_{2,i} ) (\tau) \|_{L^{\infty} (\O \times \R^3)} \dd s.
\end{split}
\Ee

\smallskip

\textbf{\underline{Case 2:} $t \leq \tBf (t,x,p)$.}
Similar to case 1, using \eqref{VP_diff_12_dy_spec} and \eqref{VP_diff_initial_12_dy_spec}, we get
\Be \notag
\begin{split}
& \mathbf{1}_{t \leq \tBf (t,x,p)} \times \w_{1} (F_1- F_2) (t,x,p) - \w_{1} (F_{1, 0} - F_{2, 0} ) (\Z_1 (0; t,x,p))
\\& = \int^t_{0} - \frac{e}{c} (-\Delta_0)^{-1}  (\nabla_x \cdot b_1 ) \frac{1}{12} \tilde{\beta} \w_{1} (F_{1} - F_{2} ) \dd s 
+ \int^t_{0} e \nabla_x ( \phi_{F_1} - \phi_{F_2})
\cdot \nabla_p ( \w_{1} F_{2, \pm} ) \dd s 
\end{split}
\Ee
Following \eqref{est:Dphi_12_tb<t}-\eqref{est3:Dphi_12_tb<t}, we derive  
\Be \notag
\begin{split}
& \mathbf{1}_{t \leq \tBf (t,x,p)} \times \w_{1} (F_1- F_2) (t,x,p) - \w_{1} (F_{1, 0} - F_{2, 0} ) (\Z_1 (0; t,x,p))
\\& \lesssim_{e, m, g} \frac{1}{12} \tilde{\beta} \sup_{s \in [0,t]} \| (-\Delta_0)^{-1} (\nabla_x \cdot b_1(s)) \|_{L^\infty(\O)}
\int^t_{0} 	\| \w_{1} (F_1  - F_2 )(s) \|_{L^{\infty} (\O \times \R^3)} \dd s 
\\& \qquad \ \ + \sup_{s \in [0,t]} \| \w_{1} \nabla_p F_2 (s) \|_{L^{\infty} (\O \times \R^3)} \int^t_{0} \sum_{i = \pm} \| \w_{1,i} (F_{1,i} - F_{2,i} ) (s) \|_{L^{\infty} (\O \times \R^3)} \dd s.
\end{split}
\Ee
This shows that
\Be \notag
\begin{split}
& \| \mathbf{1}_{t \leq \tBf (t,x,p)} \times \w_{1} (F_1- F_2) (t) \|_{L^{\infty} (\O \times \R^3)} - \| \w_{1} (F_1- F_2) (0) \|_{L^{\infty} (\O \times \R^3)}
\\& \lesssim \frac{1}{12} \tilde{\beta} \sup_{s \in [0,t]} \| (-\Delta_0)^{-1} (\nabla_x \cdot b_1(s)) \|_{L^\infty(\O)}
\int^t_{0} 	\| \w_{1} (F_1  - F_2 )(s) \|_{L^{\infty} (\O \times \R^3)} \dd s 
\\& \qquad \ \ + \sup_{s \in [0,t]} \| \w_{1} \nabla_p F_2 (s) \|_{L^{\infty} (\O \times \R^3)} \int^t_{0} \sum_{i = \pm} \| \w_{1,i} (F_{1,i} - F_{2,i} ) (s) \|_{L^{\infty} (\O \times \R^3)} \dd s,
\end{split}
\Ee
and thus
\Be \label{est1:Dphi_12_tb>t_spec}
\begin{split}
& \sup_{s \in [0,t]} \| \mathbf{1}_{s \leq \tBf (s,x,p)} \times \w_{1} (F_1- F_2) (s) \|_{L^{\infty} (\O \times \R^3)} - \| \w_{1} (F_1- F_2) (0) \|_{L^{\infty} (\O \times \R^3)}
\\& \lesssim \frac{1}{12} \tilde{\beta} \sup_{s \in [0,t]} \| (-\Delta_0)^{-1} (\nabla_x \cdot b_1(s)) \|_{L^\infty(\O)}
\int^t_{0} \sup_{\tau \in [0, s]}	\| \w_{1} (F_1  - F_2 )(\tau) \|_{L^{\infty} (\O \times \R^3)} \dd s 
\\& \qquad \ \ + \sup_{s \in [0,t]} \| \w_{1} \nabla_p F_2 (s) \|_{L^{\infty} (\O \times \R^3)} \int^t_{0} \sum_{i = \pm} \sup_{\tau \in [0, s]} \| \w_{1,i} (F_{1,i} - F_{2,i} ) (\tau) \|_{L^{\infty} (\O \times \R^3)} \dd s.
\end{split}
\Ee
Combining \eqref{est3:Dphi_12_tb<t_spec} and \eqref{est1:Dphi_12_tb>t_spec}, then for $i = \pm$,
\Be \notag
\begin{split}
& (1 - \varepsilon) \sup_{s \in [0,t]} \| \w_{1,i} (F_{1,i}- F_{2,i}) (s) \|_{L^{\infty} (\O \times \R^3)} - \| \w_{1,i} (F_{1,i}- F_{2,i}) (0) \|_{L^{\infty} (\O \times \R^3)}
\\& \lesssim \frac{1}{12} \tilde{\beta} \sup_{s \in [0,t]} \| (-\Delta_0)^{-1} (\nabla_x \cdot b_1(s)) \|_{L^\infty(\O)}
\int^t_{0} \sup_{\tau \in [0, s]} \| \w_{1,i} (F_{1,i}- F_{2,i}) (\tau) \|_{L^{\infty} (\O \times \R^3)} \dd s 
\\& \qquad \ \ + \sup_{s \in [0,t]} \| \w_{1,i} \nabla_p F_{2, i} (s) \|_{L^{\infty} (\O \times \R^3)} \int^t_{0} \sup_{\tau \in [0, s]} \sum_{j = \pm} \| \w_{1,j} (F_{1,j} - F_{2,j} ) (\tau) \|_{L^{\infty} (\O \times \R^3)} \dd s.
\end{split}
\Ee
Summing up two cases $i = \pm$, we obtain that 
\Be \label{est:Dphi_12_spec}
\begin{split}
& \sum_{i = \pm} \sup_{s \in [0,t]} \| \w_{1,i} (F_{1,i} - F_{2,i} ) (s) \|_{L^{\infty} (\O \times \R^3)}
\\& \lesssim \sum_{i = \pm} \| \w_{1,i} (F_{1,i} - F_{2,i} ) (0) \|_{L^{\infty} (\O \times \R^3)} 
\\& \ \ \ \ + \Big( \frac{1}{12} \tilde{\beta} \sup_{s \in [0,t]} \| (-\Delta_0)^{-1} (\nabla_x \cdot b_1(s)) \|_{L^\infty(\O)}
+ \sum_{j = \pm} \sup_{s \in [0,t]} \| \w_{1, j} \nabla_p F_{2, j} (s) \|_{L^{\infty} (\O \times \R^3)} \Big)
\\& \qquad \qquad \times \int^t_{0} \sum_{j = \pm} \sup_{\tau \in [0, s]} \| \w_{1,j} (F_{1,j} - F_{2,j} ) (\tau) \|_{L^{\infty} (\O \times \R^3)} \dd s.
\end{split}
\Ee
Applying the Gronwall's inequality to \eqref{est:Dphi_12_spec}, we conclude \eqref{est:F1-2_spec}. 
\end{proof}

The uniqueness of the dynamical solution is a direct consequence of Theorem \ref{theo:UD_spec}.

\begin{theorem}[Uniqueness Theorem] 
\label{theo:UA_spec} 

Suppose $(F_{1, \pm}, \nabla_x \phi_{F_1})$ and $(F_{2, \pm}, \nabla_x \phi_{F_2})$ solve \eqref{VP_F}, \eqref{Poisson_F}, \eqref{VP_0}, \eqref{Dbc:F} and \eqref{bdry:F_spec} under the compatibility condition \eqref{CC_F0=G_spec}.
Assume both follow all assumptions in Theorem \ref{theo:AS_spec} and Theorem \ref{theo:RD_spec}.
Further, we assume that both satisfy \eqref{Uest:wh_dy_spec}, \eqref{Uest:D^-1_Db_spec} and \eqref{est_final:F_v:dyn_spec}.
Then $F_{1, \pm} = F_{2, \pm}$ a.e. in $\R_+ \times \O \times \R^3$ and $\nabla_x \phi_{F_1} = \nabla_x \phi_{F_2}$ a.e. in $\R_+ \times \O$. 
\end{theorem}

\begin{proof} 

We omit the proof, since it follows directly from Theorem \ref{theo:UA} and Theorem \ref{theo:UD_spec}.
\end{proof}

\subsubsection{Proof of the Main Theorem: Dynamical Problem} \label{sec:EX_DS_spec}

In this section, we show the existence of dynamic solutions in Theorem \ref{theo:CD_spec}. 
We follow similar steps used in the dynamic problem only under the effect of the inflow boundary conditions \eqref{VP_F}-\eqref{bdry:F}.

Since the change in boundary conditions doesn't influence the characteristic trajectory determined by the Vlasov equations \eqref{eqtn:f_spec}, we remark that Lemma \ref{lem:w/w_ell} and Lemma \ref{lem:D^-1 Db_ell} hold for the dynamic problem \eqref{eqtn:f_spec}-\eqref{Poisson_f_spec} involving the specular boundary conditions. 

\begin{lemma} 
\label{lem:wf_ell_spec}

Suppose the assumptions \eqref{Bootstrap_ell_2} and \eqref{est:beta_b_dy_ell} hold. Further, we assume that
\be \notag
\|  \w_{\pm, \beta, 0} F_{\pm, 0} \|_{L^\infty (\O \times \R^3)} 
+ \| e^{\beta \sqrt{(m_{\pm} c)^2 + |p|^2} } G_{\pm} \|_{L^\infty (\gamma_-)}
+ \| w_{\pm, \beta} h_{\pm} \|_{L^\infty  (\O \times \R^3)} < \infty.
\ee
Finally, suppose there exists $\tpx > 1$ such that for every $\ell \in \N$,
\be \label{condition4:f_support_spec} 
F^{\ell+1}_{\pm} (t, x, p) 
= 0
\ \text{ for any } \
\frac{4 c}{m_{\pm} g} |p| + 3 x_{3} \geq \tpx.
\ee
Set $\lambda = \frac{g}{48} \hat{m} \beta$, suppose $\varepsilon$ in the boundary condition \eqref{bdry:F_spec} satisfies that	
\be \label{condition4:tilde_epilon_spec}
2 \varepsilon 
< \exp \Big\{ - \lambda \Big( \big( \frac{1}{\min \{ \frac{g}{4 \sqrt{2}}, \frac{c}{\sqrt{10}} \}} + 1 \big) \times \tpx + 1 \Big) \Big\}.
\ee
Consider $(t,x,p) \in [0, \infty) \times  \bar\O \times \R^3$, then for any $\ell \in \N$,
\be \label{Uest:wF_spec}
\begin{split}
& e^{ \frac{\beta}{2} \sqrt{(m_{\pm} c)^2 + |p|^2} } e^{ \frac{m_{\pm} g}{4 c} \beta x_3} | F^{\ell+1}_{\pm} (t,x,p) |  
\\& \leq 2 e \big( \|  \w_{\pm, \beta, 0} F_{\pm, 0} \|_{L^\infty (\O \times \R^3)} + \| e^{\beta \sqrt{(m_{\pm} c)^2 + |p|^2} } G_{\pm} \|_{L^\infty (\gamma_-)} \big),
\end{split} 
\ee
and 
\be \label{Uest:wf_spec}
\begin{split}
& e^{ \frac{\beta}{2} \sqrt{(m_{\pm} c)^2 + |p|^2} } e^{ \frac{m_{\pm} g}{4 c} \beta x_3} | f^{\ell+1}_{\pm} (t,x,p) |  
\\& \leq 2 e \big( \|  \w_{\pm, \beta, 0} F_{\pm, 0} \|_{L^\infty (\O \times \R^3)} + \| e^{\beta \sqrt{(m_{\pm} c)^2 + |p|^2} } G_{\pm} \|_{L^\infty (\gamma_-)} \big) + \| w_{\pm, \beta} h_{\pm} \|_{L^\infty  (\O \times \R^3)}.
\end{split} 
\ee
\end{lemma}

\begin{proof}

For the sake of simplicity, we abuse the notation as in \eqref{abuse}.
From \eqref{def:Fell_spec}, together with \eqref{f=F-h_ell} and \eqref{decom:wf}, we have
\Be \label{decom:wf_spec}
\begin{split}
f^{\ell+1}(t,x,p)| 
& \leq | F^{\ell+1} (t,x,p) | + | h(x,p) |
\\& \leq | F^{\ell+1} (t,x,p) | + e^{ - \beta \big( \sqrt{(m c)^2 + |p|^2} + \frac{m g}{2 c} x_3 \big) } \| w_\beta h  \|_\infty.
\end{split}
\Ee
Under the initial setting $f^0_{\pm} = 0$ and $(\varrho^0,  \nabla_x \Psi^0) = (0, \mathbf{0})$, together with \eqref{Bootstrap_ell_2}, \eqref{est:beta_b_dy_ell} and Lemma \ref{lem:wf_ell}, we check that \eqref{Uest:wF_spec} and \eqref{Uest:wf_spec} hold for $\ell = 0$.

\smallskip 

Now we prove by induction. Assume a positive integer $k > 0$ and suppose that \eqref{Uest:wf_spec} holds for $0 \leq \ell \leq k$. From \eqref{eqtn:Fell_spec}, \eqref{bdry_initial:Fell_spec} and the characteristics \eqref{Z1_spec}-\eqref{ODEell_spec}, we derive
\begin{align}
& | F^{k+1} (t,x,p) |
\leq \mathbf{1}_{ t \leq t_{\mathbf{B}}^{k+1} (t,x,p) } 
| F_{0} (\mathcal{Z} ^{k +1} (0;t,x,p)) |
\label{form1:Fell_spec} \\
& \qquad + \mathbf{1}_{t > t_{\mathbf{B}}^{k+1} (t,x,p)}
| G ( \mathcal{Z}^{k+1} (t - \tB^{k +1} (t,x,p);t,x,p) ) | \label{form2:Fell_spec} \\
& \qquad + \mathbf{1}_{t > t_{\mathbf{B}}^{k+1} (t,x,p)}
\varepsilon 
\underbrace{ | F^{k} (t - \tB^{k +1} (t,x,p), \xB^{k+1} (t,x,p), \tpB^{k+1} (t,x,p)) | }_{\eqref{form3:Fell_spec}_*},
\label{form3:Fell_spec}
\end{align}
where $\tpB^{k+1} = (p_{\mathbf{B}, 1}, p_{\mathbf{B}, 2}, - p_{\mathbf{B}, 3})$.
Since \eqref{Uest:wF_spec} holds for $\ell = k$, we have
\be \label{est1:form3:Fell_spec}
\begin{split}
& e^{ \frac{\beta}{2} \sqrt{(m_{\pm} c)^2 + |\tpB^{k+1} (t,x,p)|^2} } \times \eqref{form3:Fell_spec}_*  
\\& \leq 2 e \big( \|  \w_{\pm, \beta, 0} F_{\pm, 0} \|_{L^\infty (\O \times \R^3)} + \| e^{\beta \sqrt{(m_{\pm} c)^2 + |p|^2} } G_{\pm} \|_{L^\infty (\gamma_-)} \big).
\end{split} 
\ee
Recall $\w_{\pm, \beta}^{\ell+1} (t,x,p)$ in \eqref{w^ell}, together with $|\tpB^{k+1}| = |\pB^{k+1}|$, \eqref{est:1/w_h_ell} and \eqref{Bootstrap_ell_2}, we derive
\be \label{est2:form3:Fell_spec}
\begin{split}
& e^{ \frac{\beta}{2} \sqrt{(m_{\pm} c)^2 + |\tpB^{k+1} (t,x,p)|^2} }
= \frac{\w_{\beta/2}^{k+1} (t - \tB^{k +1} (t,x,p), \xB^{k+1} (t,x,p), \pB^{k+1} (t,x,p))}{\w_{\beta/2}^{k+1} (t,x,p)} \w_{\beta/2}^{k+1} (t,x,p)
\\& \qquad \geq e^{- \frac{\beta}{2} \frac{e_{\pm}}{c} \| (-\Delta_0)^{-1} (\nabla_x \cdot b^{k} ) \|_{L^\infty_{t,x}} 
\big( \frac{ \frac{4 c}{m_{\pm} g} |p| + 4 x_{3} }{c_1} + \frac{8}{m_{\pm} g} |p| + 2 \big) } \w_{\beta/2}^{k+1} (t,x,p)
\\& \qquad \geq e^{- \frac{\beta}{2} \frac{e_{\pm}}{c} \| (-\Delta_0)^{-1} (\nabla_x \cdot b^{k} ) \|_{L^\infty_{t,x}} 
\big( \frac{ \frac{4 c}{m_{\pm} g} |p| + 4 x_{3} }{c_1} + \frac{8}{m_{\pm} g} |p| + 2 \big) } 
e^{ \frac{\beta}{2} \left( \sqrt{(m_{\pm} c)^2 + |p|^2} + \frac{m g}{2 c} g x_3 \right) }.
\end{split} 
\ee
Inputting \eqref{est2:form3:Fell_spec} into \eqref{est1:form3:Fell_spec}, we get
\be \label{est3:form3:Fell_spec}
\begin{split}
& e^{ \frac{\beta}{2} \left( \sqrt{(m_{\pm} c)^2 + |p|^2} + \frac{m g}{2 c} g x_3 \right) } \times \eqref{form3:Fell_spec}_*  
\\& \leq e^{\frac{\beta}{2} \frac{e_{\pm}}{c} \| (-\Delta_0)^{-1} (\nabla_x \cdot b^{k} ) \|_{L^\infty_{t,x}} 
\big( \frac{ \frac{4 c}{m_{\pm} g} |p| + 4 x_{3} }{c_1} + \frac{8}{m_{\pm} g} |p| + 2 \big) } 
\\& \ \ \ \ \times 2 e \big( \|  \w_{\pm, \beta, 0} F_{\pm, 0} \|_{L^\infty (\O \times \R^3)} + \| e^{\beta \sqrt{(m_{\pm} c)^2 + |p|^2} } G_{\pm} \|_{L^\infty (\gamma_-)} \big).
\end{split} 
\ee
From the assumption \eqref{condition4:f_support_spec} and \eqref{condition4:tilde_epilon_spec}, we have
\be \label{est:form3:Fell_spec}
\begin{split}
& \eqref{form3:Fell_spec}
= \mathbf{1}_{t > t_{\mathbf{B}}^{k+1} (t,x,p)}
\varepsilon \times \eqref{form3:Fell_spec}_*
\\& \leq e^{ - \frac{\beta}{2} \left( \sqrt{(m_{\pm} c)^2 + |p|^2} + \frac{m g}{2 c} g x_3 \right) } \times e \big( \|  \w_{\pm, \beta, 0} F_{\pm, 0} \|_{L^\infty (\O \times \R^3)} + \| e^{\beta \sqrt{(m_{\pm} c)^2 + |p|^2} } G_{\pm} \|_{L^\infty (\gamma_-)} \big).
\end{split} 
\ee
On the other hand, following \eqref{est:Fell1} and \eqref{est:Fell2} in the proof of Lemma \ref{lem:wf_ell_spec}, we obtain that
\be \label{est:form12:Fell_spec}
\begin{split}
& \eqref{form1:Fell_spec} + \eqref{form1:Fell_spec}
\\& \leq e^{ - \frac{\beta}{2} \left( \sqrt{(m_{\pm} c)^2 + |p|^2} + \frac{m g}{2 c} g x_3 \right) } 
\times e \big( \|  \w_{\pm, \beta, 0} F_{\pm, 0} \|_{L^\infty (\O \times \R^3)} + \| e^{\beta \sqrt{(m_{\pm} c)^2 + |p|^2} } G_{\pm} \|_{L^\infty (\gamma_-)} \big).
\end{split} 
\ee
Inputting \eqref{est:form12:Fell_spec} and \eqref{est:form3:Fell_spec} into \eqref{form1:Fell_spec}-\eqref{form3:Fell_spec}, we obtain that
\be \notag
\begin{split}
& e^{ \frac{\beta}{2} \sqrt{(m c)^2 + |p|^2} } e^{ \frac{m g}{4 c} \beta x_3} | F^{k+1} (t,x,p) |
\\& \leq 2 e \big( \|  \w_{ \beta, 0} F_{0 } \|_{L^\infty_{x,v}} + \| e^{\beta \sqrt{(m_{\pm} c)^2 + |p|^2} } G \|_{L^\infty (\gamma_-)} \big).
\end{split}
\ee 
Together with \eqref{decom:wf_spec}, we conclude \eqref{Uest:wF_spec} and \eqref{Uest:wf_spec} for $\ell = k+1$. Therefore, we complete the proof by induction.
\end{proof}

\begin{prop} 
\label{prop:DC_spec}

We assume all assumptions in Theorem \ref{theo:CS_spec} with $g, \beta, \tilde \beta>0$. 
Recall $M$ and $L$ defined in \eqref{set:M} and \eqref{set:L}, suppose that
\Be \label{choice2:g_spec}
M \leq \beta e^{ - \frac{m g}{24} \beta}
\ \text{ and } \ 
L \leq \min \{ \tilde{\beta} e^{ - \frac{m g}{24} \tilde{\beta} }, \frac{1}{1024} \beta^2 e^{ - \frac{m g}{48} \beta } \}.
\Ee
Finally, there exists $\tpx > 1$ such that the initial data in \eqref{initial:fell_spec} satisfies
\be \label{condition2:f_0_spec}
\begin{split}
F_{\pm, 0} (x, p) - h_{\pm} (x, p) = 0
\ \text{ for any } \
\frac{c}{m_{\pm} g} |p| + \frac{3}{4} x_{3} + 1 \geq \tpx.
\end{split}
\ee
Set $\lambda = \frac{g}{48} \hat{m} \beta$, suppose $\varepsilon$ in the boundary condition \eqref{bdry:f_spec} satisfies that	
\be \label{condition2:epilon_G_dy_spec}
2 \varepsilon 
< \exp \Big\{ - \lambda \Big( \big( \frac{1}{\min \{ \frac{g}{4 \sqrt{2}}, \frac{c}{\sqrt{10}} \}} + 1 \big) \times \tpx + 1 \Big) \Big\}.
\ee	
From $\nabla_x \Phi$ in \eqref{eqtn:Dphi_spec} and $\Psi^{\ell}$ in \eqref{Poisson_fell_spec}-\eqref{bdry:Psi_fell_spec}, then for any $\ell \in \N$,
\be \label{Bootstrap_ell_2_spec} 
\begin{split}
\sup_{0 \leq t < \infty} \| \nabla_x \big( \Phi + \Psi^{\ell} (t) \big) \|_{L^\infty(\O)}  
& \leq \min \left\{\frac{m_+}{e_+}, \frac{m_-}{e_-} \right\} \times \frac{g}{2},
\\ \sup_{0 \leq t < \infty}\| \nabla_x \Psi^{\ell} (t) \|_{L^\infty(\O)} & \leq \min \left\{\frac{m_+}{e_+}, \frac{m_-}{e_-} \right\} \times \frac{g}{48}.
\end{split}
\ee
From $f^{\ell+1}_{\pm}, \varrho^\ell$ in \eqref{eqtn:fell_spec}-\eqref{bdry:Psi_fell_spec} and $b^\ell$ in \eqref{def:flux_ell_spec}, then for any $\ell \in \N$,
\begin{align}
e^{ 2 \beta \frac{e_{\pm}}{c} \| (-\Delta_0)^{-1} (\nabla_x \cdot b^{\ell} ) \|_{L^\infty_{t,x}} } \leq 2, &
\label{est:e^bell_spec} \\
\sup_{0 \leq t < \infty} \| \varrho^{\ell} (t, x) \|_{L^\infty(\O)}
\leq \frac{2 M}{\beta^2} \big( e_+ e^{ - \frac{m_+ g}{4 c} \beta x_3} + e_- e^{ - \frac{m_+ g}{4 c} \beta x_3} \big), &
\label{Uest:varrhofell_spec} \\
\sup_{0 \leq t < \infty} \| e^{ \frac{\beta}{2} \sqrt{(m_{\pm} c)^2 + |p|^2} } e^{ \frac{m_{\pm} g}{4 c} \beta x_3} f^{\ell+1}_{\pm} (t,x,p) \|_{L^\infty  (\O \times \R^3)} \leq M. &
\label{Uest:wfell_spec}
\end{align}
Moreover, for any $\ell \in \N$,
\begin{align}
& \sup_{0 \leq t < \infty} e^{ \lambda t} \|  e^{\frac{\beta}{8} \sqrt{(m_{\pm} c)^2 + |p|^2} + \frac{m_{\pm} g}{16 c} \beta x_3} f^{\ell+1}_{\pm} (t,x,p) \|_{L^\infty (\O \times \R^3)}
\lesssim \beta, 
\label{Udecay:fell_spec} \\
& \sup_{0 \leq t < \infty} e^{ \lambda t} \| e^{ \frac{3 \lambda }{c}  x_3} \varrho^{\ell} (t,x) \|_{L^\infty(\O)}
\lesssim \frac{e_+ + e_-}{\beta^2}, 
\label{Udecay:varrhoell_spec} \\
& \sup_{0 \leq t < \infty} e^{ \lambda t} \|\nabla_x \Psi^{\ell} (t,x) \|_{L^\infty (\O)}
\lesssim \big( \frac{e_+ + e_-}{\beta^2} \big) \times (1 + \frac{16}{g \beta}). 
\label{Udecay:DxPsiell_spec}
\end{align} 
Finally, for every $\ell \in \N$, 
\be \label{Uest:f_supportell_spec} 
f^{\ell+1}_{\pm} (t, x, p) = 0
\ \text{ for any } \
\frac{4 c}{m_{\pm} g} |p| + 3 x_{3} \geq \tpx.
\ee
\end{prop}

\begin{proof}

Under the initial setting $f^0_{\pm} = 0$ and $(\varrho^0, \nabla_x \Psi^0) = (0, \mathbf{0})$, together with \eqref{condition2:f_0_spec} and \eqref{Uest:DPhi_spec} on $\nabla_x \Phi$, we check that \eqref{Bootstrap_ell_2_spec}-\eqref{Uest:varrhofell_spec} hold for $\ell = 0$.
Since $\nabla_x \Psi^0 = \mathbf{0}$, then \eqref{bdry:Psi_fell_spec} and \eqref{dDTE^ell_2_spec} show for any $s \in [ t - \tBp^{1} (t,x,p), t + \tFp^{1} (t,x,p)]$,
\Be \notag
\begin{split}
& \frac{d}{ds} \Big( \sqrt{(m_{\pm} c)^2 + |\P^{1}_{\pm} (s;t,x,p)|^2} + \frac{1}{c} \big( e_{\pm} \Phi (\X^{1}_{\pm} (s;t,x,p)) + m_{\pm} g \X^{1}_{\pm, 3} (s;t,x,p) \big) \Big) = 0.
\end{split}
\Ee
This, together with the initial condition \eqref{condition2:f_0_spec} and  $f^0_{\pm} = 0$, shows that 
\be \notag
f^{1}_{\pm} (t, x, p) = 0
\ \text{ for any } \
\frac{4 c}{m_{\pm} g} |p| + 3 x_{3} \geq \tpx,
\ee
and thus \eqref{Uest:f_supportell_spec} holds for $\ell = 0$.
Using \eqref{Uest:wh_spec}, \eqref{Uest:h_support_spec} and \eqref{condition2:epilon_G_dy_spec}, together with \eqref{est:e^bell_spec} and \eqref{Uest:f_supportell_spec} holding for $\ell = 0$, we apply Lemma \ref{lem:wf_ell_spec} and deduce \eqref{Uest:wfell_spec} holds for $\ell = 0$.
Since \eqref{Bootstrap_ell_2_spec}, \eqref{Uest:wfell_spec}, \eqref{Uest:f_supportell_spec} hold for $\ell = 0$, Theorem \ref{theo:AS_spec} shows that \eqref{Udecay:fell_spec}-\eqref{Udecay:DxPsiell_spec} hold for $\ell = 0$.

\smallskip

Now we prove this by induction. 
We abuse the notation as in \eqref{abuse} in the following proof.
Assume a positive integer $k > 0$ and suppose that \eqref{Bootstrap_ell_2_spec}-\eqref{Uest:f_supportell_spec} hold for $0 \leq \ell \leq k$.
Then \eqref{dDTE^ell_2_spec} shows that for any $s \in [ t - \tB^{k+1} (t,x,p), t + \tF^{k+1} (t,x,p)]$,
\Be \notag
\begin{split}
& \ \ \ \ \frac{d}{ds} \Big( \sqrt{(m c)^2 + |\P^{k+1} (s;t,x,p)|^2} + \frac{1}{c} \big( e \Phi (\X^{k+1} (s;t,x,p)) + m g \X^{k+1}_{3} (s;t,x,p) \big) \Big) 
\\& \qquad = - \frac{e}{c} \nabla_x \Psi^{k} (s, \X^{k+1} (s;t,x,p)) \cdot \V^{k+1} (s;t,x,p).
\end{split}
\Ee
This implies that for any $(t, x, p) \in \R_+ \times \O \times \R^3$,
\Be \notag
\begin{split}
& \sqrt{(m c)^2 + |p|^2} + \frac{1}{c} \big( e \Phi (x) + m g x_3 \big)
\\& \leq \mathbf{1}_{t > t_{\mathbf{B}}^{k+1} (t,x,p)} \sqrt{(m c)^2 + |\pB^{k+1} (t,x,p)|^2}
\\& \ \ \ \ + \mathbf{1}_{t \leq t_{\mathbf{B}}^{k+1} (t,x,p)} \Big( \sqrt{(m c)^2 + |\P^{k+1} (0;t,x,p)|^2} + \frac{1}{c} \big( e \Phi (\X^{k+1} (0;t,x,p)) + m g \X^{k+1}_{3} (0;t,x,p) \big) \Big) 
\\& \ \ \ \ + e \int^{t}_{\max \{0, t - t_{\mathbf{B}}^{k+1} (t,x,p) \} } \| \nabla_x \Psi^{k} (s, \X^{k+1} (s;t,x,p)) \|_{L^\infty (\O)} \dd s.
\end{split}
\Ee
Since \eqref{Udecay:DxPsiell_spec} holds for $\ell = k$ and $|\tpB^{k+1}| = |\pB^{k+1}|$, we get
\Be \label{est:energy_k_spec}
\begin{split}
& \sqrt{(m c)^2 + |p|^2} + \frac{1}{c} \big( e \Phi (x) + m g x_3 \big)
\\& \leq \mathbf{1}_{t > t_{\mathbf{B}}^{k+1} (t,x,p)} \sqrt{(m c)^2 + |\tpB^{k+1} (t,x,p)|^2}
\\& \ \ \ \ + \mathbf{1}_{t \leq t_{\mathbf{B}}^{k+1} (t,x,p)} \Big( \sqrt{(m c)^2 + |\P^{k+1} (0;t,x,p)|^2} + \frac{1}{c} \big( e \Phi (\X^{k+1} (0;t,x,p)) + m g \X^{k+1}_{3} (0;t,x,p) \big) \Big) 
\\& \ \ \ \ + \frac{e^2}{\beta^2} (1 + \frac{16}{g \beta}) \int^{t}_{\max \{0, t - t_{\mathbf{B}}^{k+1} (t,x,p) \} } e^{ - \lambda s} \dd s.
\end{split}
\Ee
If $t \leq t_{\mathbf{B}}^{k+1} (t,x,p)$, then \eqref{est:energy_k_spec}, the initial condition \eqref{condition2:f_0_spec} and $\int^{\infty}_{0} e^{ - \lambda s} \dd s \lesssim 1$ show that 
\be \notag
f^{k+1}_{\pm} (t, x, p) = 0
\ \text{ for any } \
\frac{4 c}{m_{\pm} g} |p| + 3 x_{3} \geq \tpx.
\ee
If $t > t_{\mathbf{B}}^{k+1} (t,x,p)$, we denote 
\[
(t^{(1)}, x^{(1)}, p^{(1)}) := (t - t_{\mathbf{B}}^{k+1} (t,x,p), \xB^{k+1} (t,x,p), \tpB^{k+1} (t,x,p)).
\] 
From the boundary condition \eqref{bdry:fell_spec}, we consider the characteristics when $\ell = k$. Hence, \eqref{dDTE^ell_2_spec} shows that for any $s \in [ t - \tB^{k} (t^{(1)}, x^{(1)}, p^{(1)}), t + \tF^{k} (t^{(1)}, x^{(1)}, p^{(1)})]$,
\Be \notag
\begin{split}
& \frac{d}{ds} \Big( \sqrt{(m c)^2 + |\P^{k} (s;t^{(1)}, x^{(1)}, p^{(1)})|^2} + \frac{1}{c} \big( e \Phi (\X^{k} (s;t^{(1)}, x^{(1)}, p^{(1)})) + m g \X^{k}_{3} (s;t^{(1)}, x^{(1)}, p^{(1)}) \big) \Big) 
\\& \qquad = - \frac{e}{c} \nabla_x \Psi^{k-1} (s, \X^{k} (s;t^{(1)}, x^{(1)}, p^{(1)})) \cdot \V^{k} (s;t^{(1)}, x^{(1)}, p^{(1)}).
\end{split}
\Ee
Analogous to \eqref{est:energy_k_spec}, we derive
\Be \notag
\begin{split}
& \sqrt{(m c)^2 + |p^{(1)}|^2}
\\& \leq \mathbf{1}_{t^{(1)} > t_{\mathbf{B}}^{k} (t^{(1)}, x^{(1)}, p^{(1)})} \sqrt{(m c)^2 + |\tpB^{k} (t^{(1)}, x^{(1)}, p^{(1)})|^2}
\\& \ \ \ \ + \mathbf{1}_{t^{(1)} \leq t_{\mathbf{B}}^{k} (t^{(1)}, x^{(1)}, p^{(1)})} \Big( \sqrt{(m c)^2 + |\P^{k} (0;t,x,p)|^2} 
+ \frac{1}{c} \big( e \Phi (\X^{k} (0;t^{(1)}, x^{(1)}, p^{(1)})) 
\\& \qquad \ \ + m g \X^{k+1}_{3} (0;t^{(1)}, x^{(1)}, p^{(1)}) \big) \Big) 
+ \frac{e^2}{\beta^2} (1 + \frac{16}{g \beta}) \int^{t^{(1)}}_{\max \{0, t^{(1)} - t_{\mathbf{B}}^{k} (t^{(1)}, x^{(1)}, p^{(1)}) \} } e^{ - \lambda s} \dd s.
\end{split}
\Ee
If $t^{(1)} \leq t_{\mathbf{B}}^{k} (t^{(1)}, x^{(1)}, p^{(1)})$, from \eqref{est:energy_k_spec} we derive
\Be \label{est:energy_k-1_spec}
\begin{split}
& \sqrt{(m c)^2 + |p|^2} + \frac{1}{c} \big( e \Phi (x) + m g x_3 \big)
\\& \leq \mathbf{1}_{t^{(1)} \leq t_{\mathbf{B}}^{k} (t^{(1)}, x^{(1)}, p^{(1)})} \Big( \sqrt{(m c)^2 + |\P^{k} (0;t,x,p)|^2} 
+ \frac{1}{c} \big( e \Phi (\X^{k} (0;t^{(1)}, x^{(1)}, p^{(1)})) 
\\& \qquad \ \ + m g \X^{k+1}_{3} (0;t^{(1)}, x^{(1)}, p^{(1)}) \big) \Big) 
+ \frac{e^2}{\beta^2} (1 + \frac{16}{g \beta}) \int^{t}_{0} e^{ - \lambda s} \dd s.
\end{split}
\Ee
This, together with the initial condition \eqref{condition2:f_0_spec} and $\int^{\infty}_{0} e^{ - \lambda s} \dd s \lesssim 1$ show that 
\be \notag
f^{k+1}_{\pm} (t, x, p) = 0
\ \text{ for any } \
\frac{4 c}{m_{\pm} g} |p| + 3 x_{3} \geq \tpx.
\ee
If $t^{(1)} > t_{\mathbf{B}}^{k} (t^{(1)}, x^{(1)}, p^{(1)})$, we denote 
\[
(t^{(2)}, x^{(2)}, p^{(2)}) := (t - t_{\mathbf{B}}^{k} (t^{(1)}, x^{(1)}, p^{(1)}), \xB^{k} (t^{(1)}, x^{(1)}, p^{(1)}), \tpB^{k} (t^{(1)}, x^{(1)}, p^{(1)})).
\]
Analogously, from the boundary condition \eqref{bdry:fell_spec} we consider the characteristics when $\ell = k-1$.
Following this iteration and using the initial setting $f^0_{\pm} = 0$ and $(\varrho^0, \nabla_x \Psi^0) = (0, \mathbf{0})$, we conclude that \eqref{Uest:f_supportell_spec} holds for $\ell = k+1$.

Finally, from \eqref{Uest:wh_spec}, \eqref{Uest:h_support_spec} and \eqref{condition2:epilon_G_dy_spec}, together with \eqref{est:e^bell_spec} and \eqref{Uest:f_supportell_spec} holding for $\ell = k$, we apply Lemma \ref{lem:wf_ell_spec} and deduce \eqref{Uest:wfell_spec} holds for $\ell = k+1$.
Using \eqref{Bootstrap_ell_2_spec}, \eqref{Uest:wfell_spec}, \eqref{Uest:f_supportell_spec} holding for $\ell = k$, Theorem \ref{theo:AS_spec} shows that \eqref{Udecay:fell_spec}-\eqref{Udecay:DxPsiell_spec} hold for $\ell = k+1$.
We omit the details of this part of the proof, since it follows from the proof in Theorem \ref{prop:DC}.

\smallskip

Now we have proved \eqref{Bootstrap_ell_2_spec}-\eqref{Uest:f_supportell_spec} hold for $\ell = k+1$. Therefore, we complete the proof by induction.
\end{proof}

\begin{prop} 
\label{prop:Unif_D2xDp_dy_spec}

We assume all assumptions in Theorem \ref{theo:CS_spec} and Proposition \ref{prop:DC_spec} with $g, \beta, \tilde \beta>0$. 
Then $(F^{\ell+1}_{\pm}, \nabla_x \phi_{F^\ell} )$ from the construction also satisfies the following uniform-in-$\ell$ estimates:
\be \label{Uest:DDPhi^l_dy_spec}
\sup_{0 \leq t < \infty} 
\big\{ (1 + B_3 + \| \nabla_x ^2 \phi_{F^{\ell}}  \|_\infty) + \| \p_t \p_{x_3} \phi_{F^\ell} (t, x_\parallel , 0) \|_{L^\infty(\p\O)} \big\}
\leq \frac{\hat{m} g}{24} \tilde{\beta},
\ee
and
\Be \label{Uest:h_v^l_dy_spec}
\begin{split} 
& \sup_{0 \leq t < \infty} \| e^{ \frac{\tilde{\beta}}{4} \sqrt{(m_{\pm} c)^2 + |p|^2} } e^{ \frac{m_{\pm} g}{8 c} \tilde{\beta} x_3} \nabla_p F^{\ell+1}_{\pm} (t,x,p) \|_{L^\infty(\O \times \R^3)} 
\\& \lesssim 4 e^{ \frac{m_{\pm} g}{24} \tilde{\beta} } 
\| \w_{\pm, \tilde \beta, 0}  \nabla_{x,p} F_{\pm, 0} \|_{L^\infty (\O \times \R^3)}
+ \frac{m_{\pm} g}{24} \tilde{\beta} \| e^{\tilde{\beta} |p^0_{\pm}|} \nabla_{x_\parallel,p} G_{\pm} \|_{L^\infty (\gamma_-)},
\end{split}
\Ee 
where $\hat{m} = \min \{ m_+, m_- \}$.
\end{prop}

\begin{proof}

From \eqref{Uest:f_supportell_spec} in Proposition \ref{prop:DC_spec} and \eqref{Uest:h_support_spec} in Theorem \ref{theo:CS_spec}, we obtain for any $\ell \in \N$, 
\be \notag
F^{\ell+1}_{\pm} (t, x, p) = f^{\ell+1}_{\pm} (t, x, p) = 0
\ \text{ for any } \
\frac{4 c}{m_{\pm} g} |p| + 3 x_{3} \geq \tpx.
\ee
We omit the rest of the proof, since it follows directly from Proposition \ref{prop:Unif_D2xDp_dy}, Proposition \ref{RE:dyn_spec}, Lemma \ref{lem:D3tphi_F_spec}, Theorem \ref{theo:RD_spec} and Proposition \ref{prop:DC_spec}.
\end{proof}

Using Proposition \ref{prop:DC_spec} and Proposition \ref{prop:Unif_D2xDp_dy_spec}, we show that $\{ F^{\ell+1}_{\pm} \}^{\infty}_{\ell=0}$ and $\{ \nabla_x \phi_{F^\ell} \}^{\infty}_{\ell=0}$ from the construction are both Cauchy sequences.

\begin{prop} 
\label{prop:cauchy_dy_spec}

We assume all assumptions in Theorem \ref{theo:CS_spec} and Proposition \ref{prop:DC_spec} with $g, \beta, \tilde \beta>0$.
Then for any $\ell \geq 1$, $F^{\ell}_{\pm}$ from the construction satisfies that there exists some $\bar \beta > 0$, 
\be \label{est:h_cauchy_dy_spec}
\begin{split}
& \sup_{0 \leq t < \infty} \Big( \| e^{ \frac{3 \bar{\beta}}{4} \big( \sqrt{(m_+ c)^2 + |p|^2} + \frac{1}{2 c} m_+ g x_3 \big) } (F^{\ell+1}_{+} - F^{\ell}_{+} ) \|_{L^{\infty} (\O \times \R^3)} 
\\& \qquad \qquad + \| e^{ \frac{3 \bar{\beta}}{4} \big( \sqrt{(m_- c)^2 + |p|^2} + \frac{1}{2 c} m_- g x_3 \big) } (F^{\ell+1}_{-} - F^{\ell}_{-} ) \|_{L^{\infty} (\O \times \R^3)} \Big)
\\& \leq \frac{1}{2} \sup_{0 \leq t < \infty} \Big( \| e^{ \frac{3 \bar{\beta}}{4} \big( \sqrt{(m_+ c)^2 + |p|^2} + \frac{1}{2 c} m_+ g x_3 \big) } (F^\ell_{+} - F^{\ell-1}_{+} ) \|_{L^{\infty} (\O \times \R^3)} 
\\& \qquad \qquad \qquad + \| e^{ \frac{3 \bar{\beta}}{4} \big( \sqrt{(m_- c)^2 + |p|^2} + \frac{1}{2 c} m_- g x_3 \big) } (F^\ell_{-} - F^{\ell-1}_{-} ) \|_{L^{\infty} (\O \times \R^3)}  \Big).
\end{split}
\Ee
Furthermore, $\{ F^{\ell+1}_{\pm} \}^{\infty}_{\ell=0}$, and $\{ \varrho^\ell \}^{\infty}_{\ell=0}$, $\{ \nabla_x \phi_{F^\ell} \}^{\infty}_{\ell=0}$ from the construction are Cauchy sequences in $L^{\infty} (\R_+ \times \O \times \R^3)$ and $L^{\infty} (\R_+ \times \O)$ respectively.
\end{prop}

\begin{proof}

We abuse the notation as in \eqref{abuse} in the proof.
To prove \eqref{est:h_cauchy_dy_spec}, from the construction in \eqref{def:Fell_spec}-\eqref{Poisson_Fell_spec}, we have for any $\ell \geq 1$, 
\Be \label{VP_diff^l+1_dy_spec}
\begin{split}
& \p_t (F^{\ell+1}_{\pm} - F^{\ell}_{\pm}) + v_\pm \cdot \nabla_x (F^{\ell+1}_{\pm} - F^{\ell}_{\pm}) 
\\& \ \ \ \ + \Big( e_{\pm} \big( \frac{v_\pm}{c} \times B - \nabla_x ( \Phi + \Psi^{\ell} ) \big) - \nabla_x ( m_\pm g x_3) \Big) \cdot \nabla_p (F^{\ell+1}_{\pm} - F^{\ell}_{\pm})
\\& = e_{\pm} \nabla_x ( \Psi^{\ell} - \Psi^{\ell-1})
\cdot \nabla_p F^{\ell}_{\pm}
 \ \ \text{in} \ \R_+ \times  \O \times \R^3,
\end{split}
\ee
with 
\begin{align}
& F^{\ell+1}_{\pm} (t,x,p) - F^{\ell}_{\pm} (t,x,p) 
= \varepsilon \big( F^{\ell}_{\pm} (t,x, \tilde{p}) - F^{\ell-1}_{\pm} (t,x, \tilde{p}) \big) \ \ \text{in} \ \gamma_-,
\label{VP_diff_bdy^l+1_dy_spec} \\ 
& F^{\ell+1}_{\pm} (0, x, p) - F^{\ell}_{\pm} (0, x, p) = 0 \ \ \ \ \text{in} \ \O \times \R^3,
\label{VP_diff_initial^l+1_dy_spec}
\end{align}
where $\tilde{p} = (p_1, p_2, -p_3)$.
This, together with the characteristic $(\X^{\ell+1}, \P^{\ell+1})$ for \eqref{VP_diff^l+1_dy_spec}, shows
\begin{align}
& (F^{\ell+1} - F^{\ell} ) (t,x,p) 
\notag \\
& = \int^t_{ \max\{0, t - \tB^{\ell+1} (t,x,p) \}} 
e \nabla_x ( \Psi^{\ell} - \Psi^{\ell-1})
\cdot \nabla_p F^{\ell} ( \X^{\ell+1} (s;t,x,p), \P^{\ell+1} (s;t,x,p)) \dd s
\label{diff:h^l+1_dy_1_spec} \\
& \qquad + \mathbf{1}_{t > t_{\mathbf{B}}^{\ell + 1} (t,x,p)} \varepsilon (F^{\ell} - F^{\ell - 1}) (t - t_{\mathbf{B}}^{\ell+1} (t,x,p), \xB^{\ell+1} (t,x,p), \tpB^{\ell+1} (t,x,p)).
\label{diff:h^l+1_dy_2_spec}
\end{align}

For \eqref{diff:h^l+1_dy_2_spec}, we let $\bar{\beta} = \frac{\tilde \beta}{6}$, then
\Be \label{est1:diff:h^l+1_dy_2_spec}
\begin{split}
& \eqref{diff:h^l+1_dy_2_spec}
\\& \leq \varepsilon | (F^{\ell} - F^{\ell - 1}) (t - t_{\mathbf{B}}^{\ell+1} (t,x,p), \xB^{\ell+1} (t,x,p), \tpB^{\ell+1} (t,x,p)) |
\\& \leq \frac{\varepsilon}{e^{ \frac{3 \bar{\beta}}{4} \sqrt{(m_{\pm} c)^2 + |\tpB^{\ell+1} (t,x,p)|^2} }} \sup_{0 \leq t < \infty} \| e^{ \frac{3 \bar{\beta}}{4} \big( \sqrt{(m c)^2 + |p|^2} + \frac{1}{2 c} m g x_3 \big) } (F^\ell - F^{\ell-1} ) \|_{L^{\infty} (\O \times \R^3)}.
\end{split}
\ee
Recall $\w_{\pm, \beta}^{\ell+1} (t,x,p)$ in \eqref{w^ell}, together with $|\tpB^{\ell+1}| = |\pB^{\ell+1}|$, \eqref{est:1/w_h_ell} and \eqref{est2:form3:Fell_spec}, we derive
\be \label{est2:diff:h^l+1_dy_2_spec}
\begin{split}
& e^{ - \frac{3 \bar{\beta}}{4} \sqrt{(m_{\pm} c)^2 + |\tpB^{\ell+1} (t,x,p)|^2} }
\\& = \frac{\w_{3 \bar{\beta} / 4}^{\ell+1} (t,x,p)}{\w_{3 \bar{\beta} / 4}^{k+1} (t - t_{\mathbf{B}}^{\ell+1} (t,x,p), \xB^{\ell+1} (t,x,p), \pB^{\ell+1} (t,x,p))} \frac{1}{\w_{3 \bar{\beta} / 4}^{\ell+1} (t,x,p)}
\\& \leq e^{\frac{3 \bar{\beta}}{4} \frac{e}{c} \| (-\Delta_0)^{-1} (\nabla_x \cdot b^{\ell} ) \|_{L^\infty_{t,x}} 
\big( \frac{ \frac{4 c}{m g} |p| + 4 x_{3} }{c_1} + \frac{8}{m g} |p| + 2 \big) } \frac{1}{\w_{3 \bar{\beta} / 4}^{\ell+1} (t,x,p)}.
\end{split} 
\ee
Using \eqref{Bootstrap_ell_2_spec}, \eqref{est:e^bell_spec} and \eqref{Uest:f_supportell_spec}, we get
\be \label{est3:diff:h^l+1_dy_2_spec}
\begin{split}
\varepsilon e^{ - \frac{3 \bar{\beta}}{4} \sqrt{(m c)^2 + |\tpB^{\ell+1} (t,x,p)|^2} }
& \leq \varepsilon e^{ \frac{ \frac{4 c}{m g} |p| + 4 x_{3} }{c_1} + \frac{8}{m g} |p| + 2 } e^{- \frac{3 \bar{\beta}}{4} \big( \sqrt{(m c)^2 + |p|^2} + \frac{1}{2 c} m g x_3 \big) }.
\end{split} 
\ee
Inputting \eqref{est3:diff:h^l+1_dy_2_spec} into \eqref{est1:diff:h^l+1_dy_2_spec}, together with \eqref{condition2:epilon_G_dy_spec} in Proposition \ref{prop:DC_spec}, we have
\be \label{est:diff:h^l+1_dy_2_spec}
\eqref{diff:h^l+1_dy_2_spec}
\leq \frac{1}{4} e^{- \frac{3 \bar{\beta}}{4} \big( \sqrt{(m c)^2 + |p|^2} + \frac{1}{2 c} m g x_3 \big) } \sup_{0 \leq t < \infty} \| e^{ \frac{3 \bar{\beta}}{4} \big( \sqrt{(m c)^2 + |p|^2} + \frac{1}{2 c} m g x_3 \big) } (F^\ell - F^{\ell-1} ) \|_{L^{\infty} (\O \times \R^3)}.
\ee

For \eqref{diff:h^l+1_dy_1_spec}, following \eqref{bound3:diff_h^l+1_dy}-\eqref{bound6:diff_h^l+1_dy} in Proposition \ref{prop:cauchy_dy}, we can compute
\be \label{est:diff:h^l+1_dy_1_spec}
\eqref{diff:h^l+1_dy_1_spec}
\leq \frac{1}{4} e^{- \frac{3 \bar{\beta}}{4} \big( \sqrt{(m c)^2 + |p|^2} + \frac{1}{2 c} m g x_3 \big) } \sup_{0 \leq t < \infty} \| e^{ \frac{3 \bar{\beta}}{4} \big( \sqrt{(m c)^2 + |p|^2} + \frac{1}{2 c} m g x_3 \big) } (F^\ell - F^{\ell-1} ) \|_{L^{\infty} (\O \times \R^3)}.
\ee
Finally, \eqref{est:diff:h^l+1_dy_1_spec} and \eqref{est:diff:h^l+1_dy_2_spec} show that for any $0 \leq t < \infty$,
\Be \notag
\begin{split}
& \| e^{ \frac{3 \bar{\beta}}{4} \big( \sqrt{(m_+ c)^2 + |p|^2} + \frac{1}{2 c} m_+ g x_3 \big) } (F^{\ell+1}_{+} (t, x, p) - F^{\ell}_{+} (t, x, p) ) \|_{L^{\infty} (\O \times \R^3)} 
\\& \qquad \qquad + \| e^{ \frac{3 \bar{\beta}}{4} \big( \sqrt{(m_- c)^2 + |p|^2} + \frac{1}{2 c} m_- g x_3 \big) } (F^{\ell+1}_{-} (t, x, p) - F^{\ell}_{-} (t, x, p) ) \|_{L^{\infty} (\O \times \R^3)}
\\& \leq \frac{1}{2} \Big( \| e^{ \frac{3 \bar{\beta}}{4} \big( \sqrt{(m_+ c)^2 + |p|^2} + \frac{1}{2 c} m_+ g x_3 \big) } (F^{\ell}_{+} (t, x, p) - F^{\ell-1}_{+} (t, x, p) ) \|_{L^{\infty} (\O \times \R^3)} 
\\& \qquad \qquad + \| e^{ \frac{3 \bar{\beta}}{4} \big( \sqrt{(m_- c)^2 + |p|^2} + \frac{1}{2 c} m_- g x_3 \big) } (F^{\ell}_{-} (t, x, p) - F^{\ell-1}_{-} (t, x, p) ) \|_{L^{\infty} (\O \times \R^3)}  \Big).
\end{split}
\Ee
Therefore, we conclude \eqref{est:h_cauchy_dy_spec}. 
We omit the rest of the proof, since it follows from Proposition \ref{prop:cauchy_dy}.
\end{proof}

Finally, using the Cauchy sequences of $L^\infty$-spaces in Proposition \ref{prop:cauchy_dy_spec}, we construct a weak solution of $(F_{\pm}, \phi_F)$ solving \eqref{VP_F}, \eqref{Poisson_F}, \eqref{VP_0},  \eqref{Dbc:F} and \eqref{bdry:F_spec}.

\begin{proof}[\textbf{Proof of Theorem \ref{theo:CD_spec}}]

Besides the assumptions in Theorem \ref{theo:CS_spec}, we further assume \eqref{choice:g_spec}-\eqref{condition:epilon_G_dy_spec} and thus we can apply Theorems \ref{theo:AS_spec} and \ref{theo:RD_spec} in the following proof.

\smallskip

\textbf{Step 1. Regularity: Proof of \eqref{Uest:wh_dy_spec}-\eqref{Uest:DxPsi_spec} and \eqref{Uest:f_support_spec}-\eqref{Uest:D2xD3tphi_F_spec}.}
From \eqref{est:h_cauchy_dy_spec}, the arguments of the Cauchy sequences of $L^\infty$-spaces in Proposition \ref{prop:cauchy_dy}, there exists 
\be \notag
F_{\pm} (t, x, p) \in L^\infty (\R_+ \times \bar \O \times \R^3)
\ \text{ and } \ 
\varrho(t, x), \nabla_x \phi_F (t, x) \in L^\infty (\R_+ \times \bar \O) 
\ \text{ with } \ 
\phi_F = 0 \ \ \text{on} \ \p\O,
\ee
such that as $k \to \infty$,
\be \label{weakconv_whst_dy_spec}
\begin{split}
e^{ \frac{3 \bar{\beta}}{4} ( \sqrt{(m_{\pm} c)^2 + |p|^2} + \frac{1}{2 c} m_{\pm} g x_3) } F^{k}_{\pm}
\to e^{ \frac{3 \bar{\beta}}{4} ( \sqrt{(m_{\pm} c)^2 + |p|^2} + \frac{1}{2 c} m_{\pm} g x_3 ) } F_{\pm}
& \ \text{ in } \ L^\infty (\R_+ \times \bar \O \times \R^3),
\\ e^{\beta' \frac{g}{2c} x_3} \varrho^{k} \to e^{\beta' \frac{g}{2c} x_3} \varrho 
& \ \text{ in } \ L^\infty (\R_+ \times \bar \O),
\\ \nabla_x \phi_{F^k} \to \nabla_x \phi_F
& \ \text{ in } \ L^\infty (\R_+ \times \bar \O),
\end{split}
\ee
where $\bar{\beta} = \frac{\tilde \beta}{2}$ and $\beta' =  \min\{ \frac{\bar \beta}{4}, \beta \} \times \min \{ m_{-},  m_{+} \}$. This shows that
\be \label{strong_conv_h_rho_dy_spec}
\begin{split}
F^{k}_{\pm} \to F_{\pm} 
&\ \text{ in } \ L^\infty (\R_+ \times \bar \O \times \R^3) 
\ \text{ as } \ k \to \infty,
\\ \varrho^{k} \to \varrho 
& \ \text{ in } \ L^\infty (\R_+ \times \bar \O)
\ \text{ as } \ k \to \infty.
\end{split}
\ee
Using \eqref{Bootstrap_ell_2_spec} in Proposition \ref{prop:DC_spec}, together with \eqref{eqtn:phiFell_spec} and the $L^\infty$ convergence in \eqref{weakconv_whst_dy_spec}, we prove that $\nabla_x \phi_F$ and $\nabla_x \Psi$ satisfy \eqref{Uest:Dxphi_F_spec} and \eqref{Uest:DxPsi_spec} respectively. 

Furthermore, the $L^{\infty}$ convergence $\varrho^k \to \varrho$ in \eqref{strong_conv_h_rho_dy_spec} also implies
\be \notag
| \varrho |_{C^{0,\delta}(\O)} \leq \sup_{k \in \N} |\varrho^{k+1} |_{C^{0,\delta}(\O)}.
\ee
Together with \eqref{Uest:varrhofell_spec}, \eqref{Uest:DDPhi^l_dy_spec} and \eqref{est:nabla^2phi}, we deduce that $\| \nabla_x ^2 \phi_F \|_\infty $ satisfies \eqref{est:D2xphi_F_spec}.
From Lemma \ref{lem:D3tphi_F_spec}, this implies $| \p_{x_3} \p_t \phi_{F} (t,x)| $ satisfies \eqref{est:D3tphi_F_spec}.
Therefore,  we conclude \eqref{Uest:D2xD3tphi_F_spec}.

Similarly, from \eqref{est2:form3:Fell_spec} and $F^k_{\pm} = h_{\pm} + f^k_{\pm}$ for any $k \geq 1$ in \eqref{def:Fell}, the $L^\infty$ convergence in \eqref{strong_conv_h_rho_dy_spec} implies that
\be \label{strong_conv_f_rho_dy_spec}
f^{k}_{\pm} \to f_{\pm} 
\ \text{ in } \ L^\infty (R_+ \times \bar \O \times \R^3) 
\ \text{ as } \ k \to \infty.
\ee
Using \eqref{Uest:wh_spec}, together with \eqref{Uest:wfell_spec}, we obtain \eqref{Uest:wh_dy_spec} and \eqref{Uest:f_support_spec}.

\smallskip

\textbf{Step 2. Existence.}
We omit the proof of showing $(F_{\pm}, \phi_F)$ obtained in 
\eqref{weakconv_whst_dy_spec} is the solution to \eqref{VP_F}, \eqref{Poisson_F}, \eqref{VP_0},  \eqref{Dbc:F} and \eqref{bdry:F_spec} in the sense of Definition \ref{weak_sol_dy_spec}, since it follows the same weak convergence argument of the proof of Theorem \ref{theo:CD}.

\smallskip

\textbf{Step 3. Regularity: Proof of \eqref{Uest_final:F_v:dyn_spec}-\eqref{Uest_final:F_x:dyn_spec}.}
Since $f_{\pm} = F_{\pm} - h_{\pm}$, from \eqref{est_final:hk_v_spec} and \eqref{est_final:hk_x_spec} in Theorem \ref{theo:CS_spec}, together with \eqref{choice:g_spec} and \eqref{Uest:D2xD3tphi_F_spec}, \eqref{Uest:D^-1_Db_spec} and Theorem \ref{theo:RD_spec}, we conclude $f_{\pm}$ satisfies \eqref{Uest_final:F_v:dyn_spec} and \eqref{Uest_final:F_x:dyn_spec}.

\smallskip

\textbf{Step 4. Stability and Uniqueness: Proof of \eqref{Udecay:f_spec}-\eqref{Udecay:DxPsi_spec}.}
Finally, under the assumption \eqref{choice:g_spec}, using Theorem \ref{theo:AS_spec}, together with \eqref{Uest:wfell_spec} and the $L^{\infty}$ convergence $f^k \to f$ in \eqref{strong_conv_f_rho_dy_spec}, we prove that $(f(t), \varrho(t))$ satisfies \eqref{Udecay:f_spec} and \eqref{Udecay:varrho_spec}.
Applying \eqref{Udecay:varrho_spec} in \eqref{est:nabla_phi} of Lemma \ref{lem:rho_to_phi}, we obtain \eqref{Udecay:DxPsi_spec}.
On the other hand, from \eqref{Uest:wh_dy_spec}, \eqref{Uest_final:F_v:dyn_spec} and \eqref{Uest:D^-1_Db_spec} in Theorem \ref{theo:AS_spec}, we apply Theorem \ref{theo:UA_spec}, and thus conclude the uniqueness of the solution $(F, \phi_F)$.
\end{proof}

\appendix


\section{Green's Function in \texorpdfstring{$\R^2 \times [0, \infty)$}{R2*infty}}
\label{Sec:Green} 

In this section, we study the Green's function $\mathfrak{G}$ of the following Poisson equation in the 3D half space:
\Be \label{Dphi}
\begin{split}
\Delta \phi (x) = \rho(x) \ \ & \text{in} \ x \in \O : = \R^2 \times [0, \infty), \\
\phi (x) = 0 \ \ & \text{on} \ x \in \p \O : =  \R^2 \times \{0\},
\end{split}
\Ee
such that $\phi$ solving \eqref{Dphi} takes the form of  
\Be \label{phi_rho}
\phi (x)  = \int_{\R^2 \times [0, \infty)} \mathfrak{G}(x,y) \rho (y ) \dd y.
\Ee

\begin{lemma}[\cite{GT}] \label{lemma:G} 

The Green's function for \eqref{Dphi} takes the form of  
\Be \label{GreenF}
\mathfrak{G}(x,y) =
C_1 \big( \frac{1}{|x-y|} - \frac{1}{|\tilde{x}-y|} \big), \ \ \text{for} \ x,y \in\O,
\Ee
where $\tilde x  = (x_1, x_2, - x_3)$. 
\end{lemma}

Next, we study some elliptic estimates on $\phi$.

\begin{lemma} \label{lem:rho_to_phi}

Suppose $| \rho (x) | \leq A e^{-B x_3}$ for $A, B > 0$. Then $\phi(x)$ in \eqref{phi_rho} satisfies 
\begin{equation} \label{est:nabla_phi}
|\p_{x_j}  \phi (x)|   
\leq \mathfrak{C} A \big( 1 +  \frac{1}{B} \big) \ \ \text{for  } \  x \in \R^2 \times [0, \infty).
\end{equation} 
Moreover, for any $\delta > 0$,
\Be \label{est:nabla^2phi}
\| \nabla_x^2 \phi \|_{L^\infty(\O)} 
\lesssim_\delta  \|\rho \|_{L^\infty(\O)} + \|\rho \|_{C^{0,\delta}(\O)} + \frac{A}{B}.
\Ee  
\end{lemma}

\begin{proof} 

First, we prove \eqref{est:nabla_phi}. Under direct computation,
\Be \label{eq:G_x}
\p_{x_j} \mathfrak{G} = 
\begin{cases}
& C_1 \big( \frac{- x_j + y_j}{|x-y|^3} - \frac{- x_j + y_j}{|\tilde{x}-y|^3} \big), \ \ \text{for } j = 1, 2,  \\[5pt]
& C_1 \big( \frac{- x_3 + y_3}{|x-y|^3} - \frac{ - x_3 - y_3}{|\tilde{x}-y|^3} \big), \ \ \text{for } j = 3.
\end{cases}
\Ee
Setting $|x-y| = r$ and $|\tilde{x}-y| = \tilde{r}$, we obtain for $j = 1, 2$,
\Be \label{est1:G_x}
\begin{split} 
|\p_{x_j} \mathfrak{G}|
\lesssim \big| (x_j - y_j) \big( \frac{ \tilde{r}^3 - r^3 }{ r^3 \tilde{r}^3} \big) \big|
& = \big| (x_j - y_j) \big( \frac{(\tilde{r}^2 - r^2) (\tilde{r}^2 + \tilde{r} r + r^2)}{r^3 \tilde{r}^3 (\tilde{r} + r) } \big) \big|
\\& \lesssim \big| x_3 y_3 (x_j - y_j) \big( \frac{ \tilde{r}^2 + \tilde{r} r + r^2}{r^3 \tilde{r}^3 (\tilde{r} + r) } \big) \big|
\lesssim \big| x_3 y_3 (x_j - y_j) \big( \frac{ \tilde{r} + r }{r^3 \tilde{r}^3 } \big) \big|  \\& \lesssim \frac{ y_3 }{r^3 } = \frac{ y_3 }{|x-y|^3 },
\end{split}
\Ee
where we use $r \leq \tilde{r}$ and $x_3 < |- x_3 - y_3| \leq \tilde{r}$.
For $j = 3$, using \eqref{est1:G_x} we obtain
\Be \label{est2:G_x}
\begin{split}
|\p_{x_3} \mathfrak{G}|
& \lesssim \bigg| x_3 \big( \frac{1}{ r^3} - \frac{1}{\tilde{r}^3} \big) \bigg|
+ \bigg| y_3 \big( \frac{1}{ r^3} + \frac{1}{\tilde{r}^3} \big) \bigg|
\\& \lesssim \frac{ y_3 }{r^3 } + \frac{ y_3 }{r^3 } = \frac{ 2 y_3 }{|x-y|^3 }.
\end{split}
\Ee
On the other hand, since $|x_j - y_j| \leq r$ and $|x_3 + y_3| \leq \tilde{r}$, we also get
\Be \label{est:Green_x}
|\p_{x_j} \mathfrak{G}| \lesssim \frac{1}{|x-y|^2}.
\Ee
 
From \eqref{phi_rho} and \eqref{GreenF}, together with \eqref{est1:G_x}-\eqref{est:Green_x}, then for $1 \leq j \leq 3$,
\Be \label{est1:nabla_phi}
\begin{split}
| \p_{x_j} \phi (x)| 
= | (\p_{x_j} \mathfrak{G} * \rho )(x)|
& \lesssim \int_0^\infty \int_{\R^2}  \min \big( \frac{y_3}{|x-y|^3}, \frac{1}{|x-y|^2} \big) \times A e^{- B y_3} \dd y_\parallel \dd y_3
\\& \lesssim \int_0^\infty A e^{- B y_3} 
\Big( \int_0^{1} \frac{1}{|x-y|^2} r \dd r + \int_1^{\infty} \frac{y_3}{|x-y|^3} r \dd r \Big) \dd y_3
\\& \leq \int_0^\infty A e^{- B y_3} 
\Big( \int_0^{1} \frac{ r \dd r}{r^2 + |x_3 - y_3|^2} + \int_1^{\infty} \frac{ y_3 \dd r}{r^2} \Big) \dd y_3
\\& \lesssim A  \int_0^\infty e^{- B y_3} \Big(
   \ln \big(1+ \frac{1}{2 |x_3 - y_3|^2} \big) + y_3 \Big) \dd y_3
\\& \lesssim  A \left\{ \int_{\max\{0, x_3-1\}}^{ x_3+1} + \int_{x_3+1}^\infty+ \int_{0}^{\max\{0, x_3-1\}} \right\} + A,
\end{split}
\Ee
where we have decomposed $y_3$-integral into three parts. We bound each of them:
\Be \notag
\begin{split}
\int_{\max\{0, x_3-1\}}^{ x_3+1}  \leq 1, 
\ \
\int_{x_3+1}^\infty+ \int_{0}^{\max\{0, x_3-1\}} \lesssim 1/B. 
\end{split}
\Ee
Hence, we derive that 
\Be \label{est:I_2}
| \p_{x_j} \phi (x)|  
\lesssim A + \frac{A}{B},
\Ee
and conclude \eqref{est:nabla_phi}. 

\smallskip

Second, we prove \eqref{est:nabla^2phi}. We split the proof into two cases: $x_3 \leq 1$ and $x_3 > 1$.

\smallskip

\textbf{\underline{Case 1:} $x_3 > 1$.} 
The classical result of the potential theory \cite{GT} shows, for $0 < R \leq 1/2$,
\be \label{est:Green_xixj}
\begin{split}
& \p_{x_i} \p_{x_j} \phi (x)
\\& = \int_{ \{ |x-y| > R \} \; \cap \; \O} \p_{x_i} \p_{x_j} \mathfrak{G} (x, y) \rho (y) \dd y
+ \int_{ \{ |x-y| \leq R \} \; \cap \; \O } \p_{x_i} \p_{x_j} \mathfrak{G} (x, y) \big( \rho (y) - \rho (x) \big) \dd y 
\\& \ \ \ \ + \frac{1}{3} \delta_{i j} \rho (x) - \rho (x) \int_{ \{ |x-y| \leq R \} \; \cap \; \O } \p_{x_i} \p_{x_j} \frac{1}{|\tilde{x}-y|}  \dd y.
\end{split}
\ee
Under direct computation, for $i, j = 1, 2$,
\Be \notag
\begin{split}
\p_{x_i} \p_{x_j} \mathfrak{G} 
& = C_1 \p_{x_i} \big( \frac{- x_j + y_j}{|x-y|^3} - \frac{- x_j + y_j}{|\tilde{x}-y|^3} \big)
\\& = C_1 \Big( - \delta_{i,j} (\frac{1}{|x-y|^3} - \frac{1}{|\tilde{x}-y|^3})
+ (x_j - y_j) \big( \frac{3 (x_i - y_i)}{|x-y|^5}  - \frac{3 (x_i - y_i)}{|\tilde{x}-y|^5} \big) \Big).
\end{split}
\Ee
and $i = 3, j = 1, 2$,
\Be \notag
\begin{split}
\p_{x_i} \p_{x_j} \mathfrak{G} 
& = C_1 \p_{x_i} \big( \frac{- x_j + y_j}{|x-y|^3} - \frac{- x_j + y_j}{|\tilde{x}-y|^3} \big)
\\& = C_1 (x_j - y_j) \Big( \frac{3 (x_3 - y_3)}{|x-y|^5}  - \frac{3 (x_3 + y_3)}{|\tilde{x}-y|^5} \Big).
\end{split}
\Ee

For $j = 3$,
\Be \notag
\p_{x_i} \p_{x_j} \mathfrak{G} 
= 
\begin{cases}
& C_1 \big( 3 (x_3 - y_3) \frac{x_i - y_i}{|x-y|^5} - 3 (x_3 + y_3) \frac{x_i - y_i}{|\tilde{x}-y|^5} \big), \ \ \text{for } i = 1, 2,  \\[5pt]
& C_1 \big( - (\frac{1}{|x-y|^3} - \frac{1}{|\tilde{x}-y|^3}) + 3 \frac{(x_3 - y_3)^2}{|x-y|^5} - 3 \frac{(x_3 + y_3)^2}{|\tilde{x}-y|^5} \big), \ \ \text{for } i = 3.
\end{cases}
\Ee
Setting $|x-y| = r$ and $|\tilde{x}-y| = \tilde{r}$, we can check that 
\be \notag
|x_j - y_j| \leq r \leq \tilde{r}, \ \ 
|x_3 + y_3| \leq \tilde{r}.
\ee
Thus, then for $1 \leq i, j \leq 3$,
\Be \label{est1:Green_xixj}
|\p_{x_i} \p_{x_j} \mathfrak{G} | 
\lesssim \frac{1}{|x-y|^3}.
\Ee
Inputting \eqref{est1:Green_xixj} into \eqref{est:Green_xixj}, together with $0 < R \leq 1/2$, then for $0 < \delta < 1$,
\be \label{est2:Green_xixj}
\begin{split}
& | \p_{x_i} \p_{x_j} \phi (x) |
\\& \leq \int_{ \{ |x-y| > R \} \; \cap \; \O} \frac{1}{|x-y|^3} | \rho (y) | \dd y
+ |\rho |_{C^{0,\delta}(\O)} \int_{\{ |x-y| \leq R \} \; \cap \; \O} \frac{1}{|x-y|^{3 - \delta} } \dd y
+ \frac{4}{3} \|\rho \|_{L^\infty(\O)} 
\\& \lesssim \underbrace{\int_{ \{ |x-y| > R \} \; \cap \; \O} \frac{1}{|x-y|^3} A e^{-B y_3} \dd y}_{\eqref{est2:Green_xixj}_*}
+ |\rho |_{C^{0,\delta}(\O)} R^{\delta}
+ \frac{4}{3} \|\rho \|_{L^\infty(\O)}.
\end{split}
\ee
For $\eqref{est2:Green_xixj}_*$, we have
\be \label{est2:Green_xixj_*}
\begin{split}
\eqref{est2:Green_xixj}_*
& \leq \int_{ \{ |x-y| > R \} \; \cap \; \O} \frac{2}{R^3 + |x-y|^3} A e^{-B y_3} \dd y
\\& \leq A \int_{\R^2} \frac{2}{R^3 + \big( \sqrt{(x_1 -y_1)^2 + (x_2 -y_2)^2} \big)^3 } \dd y_1 \dd y_2 \int^{\infty}_{0} e^{-B y_3} \dd y_3
\leq \frac{A}{B} R^{-1}.
\end{split}
\ee
Together with \eqref{est2:Green_xixj}, we derive 
\be \label{est3:Green_xixj}
\begin{split}
| \p_{x_i} \p_{x_j} \phi (x) |
& \lesssim \frac{A}{B} R^{-1}
+ |\rho |_{C^{0,\delta}(\O)} R^{\delta}
+ \frac{4}{3} \|\rho \|_{L^\infty(\O)}.
\end{split}
\ee
Now picking $R = 1/2$, together with \eqref{est3:Green_xixj}, we get
\be
| \p_{x_i} \p_{x_j} \phi (x) |
\lesssim_{\delta} \frac{A}{B} + \|\rho \|_{L^\infty(\O)} 
+ \|\rho \|_{C^{0,\delta}(\O)},
\ee
and thus conclude \eqref{est:nabla^2phi} in this case.

\smallskip

\textbf{\underline{Case 2:} $x_3 \leq 1$.} 
Now we introduce a smooth cutoff function $\chi = \chi (t) \in [0, 1]$, which satisfies that for any $t$, $|\chi^{\prime} (t)| \leq 1$, and
\be \label{def:chi}
\chi (t) =
\begin{cases}
1, & \text{ if } t \leq \frac{x_3}{2}, \\[5pt]
0, & \text{ if } t \geq 2.
\end{cases}
\ee
Then we write $\p_{x_i} \p_{x_j} \phi (x)$ as follows:
\begin{align}
\p_{x_i} \p_{x_j} \phi (x)
& = \int_{ \{ |x-y| > 4 \} \; \cap \; \O} \p_{x_i} \p_{x_j} \mathfrak{G} (x, y) \rho (y) \dd y
\label{est:Green_xixj_21} \\
& \ \ \ \ + \int_{ \{ |x-y| \leq 4 \} \; \cap \; \O } \p_{x_i} \p_{x_j} \mathfrak{G} (x, y) \big( \rho (y) - \chi (|x-y|) \rho (x) \big) \dd y 
\label{est:Green_xixj_22} \\
& \ \ \ \ + \int_{ \{ |x-y| \leq 4 \} \; \cap \; \O } \p_{x_i} \p_{x_j} \mathfrak{G} (x, y) \chi (|x-y|) \rho (x) \dd y.
\label{est:Green_xixj_23}
\end{align}

For \eqref{est:Green_xixj_21}, we follow \eqref{est2:Green_xixj_*}, and get
\be \label{est2:Green_xixj_21}
| \eqref{est:Green_xixj_21} | \leq \frac{A}{B} R^{-1}.
\ee

For \eqref{est:Green_xixj_22}, using \eqref{est1:Green_xixj} and \eqref{est2:Green_xixj}, we have
\be \label{est2:Green_xixj_22}
\begin{split}
& \big| \int_{ \{ |x-y| \leq 4 \} \; \cap \; \O } \p_{x_i} \p_{x_j} \mathfrak{G} (x, y) \big( \rho (y) - \chi (|x-y|) \rho (x) \big) \dd y \big|
\\& \leq \big| \int_{ \{ |x-y| \leq 4 \} \; \cap \; \O } \p_{x_i} \p_{x_j} \mathfrak{G} (x, y) \big( \rho (y) - \rho (x) \big) \dd y \big|
\\& \ \ \ \ + \big| \int_{ \{ |x-y| \leq 4 \} \; \cap \; \O } \p_{x_i} \p_{x_j} \mathfrak{G} (x, y) \big( \rho (x) - \chi (|x-y|) \rho (x) \big) \dd y \big|
\\& \lesssim |\rho |_{C^{0,\delta}(\O)} 
+ \|\rho \|_{L^\infty(\O)} \underbrace{ \big| \int_{ \{ |x-y| \leq 4 \} \; \cap \; \O } \p_{x_i} \p_{x_j} \mathfrak{G} (x, y) \big( 1 - \chi (|x-y|) \big) \dd y \big| }_{\eqref{est2:Green_xixj_22}_*}
\end{split}
\ee
Under direct computation from \eqref{eq:G_x}, we obtain for $1 \leq i, j \leq 3$, 
\be \label{est:change_x_y}
\begin{split}
\p_{x_i} \p_{x_j} \mathfrak{G} (x, y)
& = C_1 \p_{x_i} \p_{x_j} \big( \frac{1}{|x-y|} - \frac{1}{|\tilde{x}-y|} \big)
\\& = C_1 \p_{y_i} \Big( \p_{x_j} \big( - \frac{1}{|x-y|} \pm \frac{1}{|\tilde{x}-y|} \big) \Big),
\end{split}
\ee
where $\pm$ depends on whether $i, j = 3$. Using Divergence theorem on $\p_{y_i}$, we have
\begin{align}
& \int_{ \{ |x-y| \leq 4 \} \; \cap \; \O } 
\p_{y_i} \Big( \p_{x_j} \big( - \frac{1}{|x-y|} \pm \frac{1}{|\tilde{x}-y|} \big) \Big) \big( 1 - \chi (|x-y|) \big) \dd y 
\notag \\
& = \int_{ \{ |x-y| \leq 4 \} \; \cap \; \p\O } 
\p_{x_j} \big( - \frac{1}{|x-y|} \pm \frac{1}{|\tilde{x}-y|} \big) \big( 1 - \chi (|x-y|) \big) \mathbf{n}_i (y) \dd y
\label{est2:Green_xixj_22_1} \\
& \ \ \ \ + \int_{ \{ |x-y| = 4 \} \; \cap \; \O } 
\p_{x_j} \big( - \frac{1}{|x-y|} \pm \frac{1}{|\tilde{x}-y|} \big) \big( 1 - \chi (|x-y|) \big) \mathbf{n}_i (y) \dd y
\label{est2:Green_xixj_22_2} \\
& \ \ \ \ - \int_{ \{ |x-y| \leq 4 \} \; \cap \; \O } 
\Big( \p_{x_j} \big( - \frac{1}{|x-y|} \pm \frac{1}{|\tilde{x}-y|} \big) \Big) \p_{y_i} \big( 1 - \chi (|x-y|) \big) \dd y.
\label{est2:Green_xixj_22_3}
\end{align}
where $\mathbf{n}_i (y)$ represents the $i$th component of the normal vector at the boundary point $y$.
From \eqref{eq:G_x}, then for any $y \in \p\O$,
\be \notag
\big| \p_{x_j} \big( - \frac{1}{|x-y|}
\pm \frac{1}{|\tilde{x}-y|} \big) \big| \leq \frac{2 x_3}{|x-y|^3},
\ee
and thus
\be \label{est3:Green_xixj_22_1}
\begin{split}
| \eqref{est2:Green_xixj_22_1} |
& \leq \int_{ \{ |x-y| \leq 4 \} \; \cap \; \p\O } 
\frac{2 x_3}{|x-y|^3} \dd y 
\lesssim \int^{4}_{0} \frac{2 x_3}{(r + x_3)^3} r \dd r
\lesssim 1.
\end{split}
\ee
Using \eqref{est:Green_x}, we compute
\be \label{est3:Green_xixj_22_2}
\begin{split}
| \eqref{est2:Green_xixj_22_2} |
& \leq \int_{ \{ |x-y| = 4 \} \; \cap \; \O } 
\frac{1}{|x-y|^2} \dd y \lesssim 1.
\end{split}
\ee
Further, from \eqref{def:chi}, we derive 
\be \label{est3:Green_xixj_22_3}
\begin{split}
| \eqref{est2:Green_xixj_22_3} |
& \leq \int_{ \{ |x-y| \leq 4 \} \; \cap \; \O } 
\frac{1}{|x-y|^2} \big| \p_{y_i} \big( 1 - \chi (|x-y|) \big) \big| \dd y
\\& \leq \int_{ \{ |x-y| \leq 4 \} \; \cap \; \O } 
\frac{1}{|x-y|^2} \big| \p_{y_i} \chi (|x-y|) \big| \dd y
\lesssim 1.
\end{split}
\ee
Together with \eqref{est3:Green_xixj_22_1}, \eqref{est3:Green_xixj_22_2} and \eqref{est3:Green_xixj_22_3}, we obtain $| \eqref{est2:Green_xixj_22}_* | \lesssim 1$, and thus
\be \label{est3:Green_xixj_22}
| \eqref{est:Green_xixj_22} | \lesssim |\rho |_{C^{0,\delta}(\O)} 
+ \|\rho \|_{L^\infty(\O)}. 
\ee

For \eqref{est:Green_xixj_23}, similar to \eqref{est2:Green_xixj_22_1}-\eqref{est2:Green_xixj_22_3}, we can rewrite it as 
\begin{align}
& \int_{ \{ |x-y| \leq 4 \} \; \cap \; \O } \p_{x_i} \p_{x_j} \mathfrak{G} (x, y) \chi (|x-y|) \dd y
\notag \\
& = \int_{ \{ |x-y| \leq 4 \} \; \cap \; \p\O } 
\p_{x_j} \big( - \frac{1}{|x-y|} \pm \frac{1}{|\tilde{x}-y|} \big) \chi (|x-y|) \mathbf{n}_i (y) \dd y
\label{est2:Green_xixj_23_1} \\
& \ \ \ \ + \int_{ \{ |x-y| = 4 \} \; \cap \; \O } 
\p_{x_j} \big( - \frac{1}{|x-y|} \pm \frac{1}{|\tilde{x}-y|} \big) \chi (|x-y|) \mathbf{n}_i (y) \dd y
\label{est2:Green_xixj_23_2} \\
& \ \ \ \ - \int_{ \{ |x-y| \leq 4 \} \; \cap \; \O } 
\p_{x_j} \big( - \frac{1}{|x-y|} \pm \frac{1}{|\tilde{x}-y|} \big) \p_{y_i} \chi (|x-y|) \dd y.
\label{est2:Green_xixj_23_3}
\end{align}
From \eqref{est3:Green_xixj_22_1}, \eqref{est3:Green_xixj_22_2} and \eqref{est3:Green_xixj_22_3}, we deduce that
\be \notag
\big| \int_{ \{ |x-y| \leq 4 \} \; \cap \; \O } \p_{x_i} \p_{x_j} \mathfrak{G} (x, y) \chi (|x-y|) \dd y \big|
\lesssim 1,
\ee
and thus 
\be \label{est3:Green_xixj_23}
| \eqref{est:Green_xixj_23} | 
\lesssim \|\rho \|_{L^\infty(\O)}. 
\ee
Combining \eqref{est2:Green_xixj_21}, \eqref{est3:Green_xixj_22} and \eqref{est3:Green_xixj_23}, we conclude \eqref{est:nabla^2phi} in this case.
\end{proof}

\section*{Acknowledgment} CK thanks Professor Anna Mazzucato for her interest in this work. This project is partly supported by NSF-DMS 1900923, NSF-CAREER 2047681, Simons fellowship in Mathematics, and the Brain Pool fellowship funded by the Korean Ministry of Science and ICT through the National Research Foundation of Korea (NRF-2021H1D3A2A01039047).

\bibliographystyle{abbrv}

\def\cprime{$'$} \def\cprime{$'$}

\end{document}